\documentclass[10pt, leqno]{article}
\usepackage{amsmath,amsfonts,amsthm,amscd,amssymb}
\numberwithin{equation}{section}
\theoremstyle{plain}
\newtheorem{theo}{Theorem}[section]
\newtheorem{prop}[theo]{Proposition}
\newtheorem{coro}[theo]{Corollary} 
\newtheorem{lemm}[theo]{Lemma}
\theoremstyle{definition}
\newtheorem{defi}[theo]{Definition}
\newtheorem{rema}[theo]{Remark}
\newtheorem{theo-defi}[theo]{Theorem-Definition}
\newtheorem{prop-defi}[theo]{Proposition-Definition}
\newtheorem{rema-defi}[theo]{Remark-Definition}
\newtheorem{exem-defi} [theo]{Example-Definition}

\newtheorem{exem}[theo]{Example}
\newtheorem{conj}[theo]{Conjecture}

\def \al{\alpha}
\def \bet{\beta}
\def \bul{\bullet}
\def \col{\colon}
\def \Del{\Delta}

\def \del{\delta}
\def \eps{\epsilon}
\def \Gam{\Gamma}
\def \gam{\gamma}
\def \inf{\infty}

\def \kap{\kappa}
\def \Lam{\Lambda}
\def \lam{\lambda}

\def \Lo{\Longrightarrow}
\def \lo{\longrightarrow}
\def \lom{\longmapsto}
\def \mab{\mathbb}

\def \nat{\natural} 

\def \Om{\Omega}
\def \om{\omega}
\def \ol{\overline}
\def \os{\overset}
\def \parno{\par\noindent}

\def \sig{\sigma}
\def \sq{\square}
\def \sus{\subset}

\def \ul{\underline}
\def \us{\underset}

\def \vp{\varpi}

\def \vpl{\varprojlim}

\def \wh{\widehat}
\def \wt{\widetilde}
\bigskip

\newcommand{\getsfrom}{\ensuremath{
\longleftarrow\kern-.50em\lower.0ex\hbox%
{$\shortmid\,$}}}

\begin{document}
\title{Hirsch weight-filtered log crystalline complex and 
Hirsch weight-filtered log crystalline dga of a proper SNCL scheme 
in characteristic $p>0$}
\author{Yukiyoshi Nakkajima
\date{}\thanks{ 
\endgraf}}
\maketitle

\bigskip
\parno
{\bf Abstract.---}
For a $p$-adic PD-formal family $(S,{\cal I},\gam)$ of log points,  
let $S_0$ be an exact closed log subscheme of $S$ defined by ${\cal I}$. 
For a proper simple normal crossing log scheme $X$ over $S_0$, 
we construct a theory of the derived PD-Hirsch extension 
of the modification of the log crystalline complex of $X$ over 
the underlying PD-formal scheme 
$(\os{\circ}{S},{\cal I},\gam)$ of $(S,{\cal I},\gam)$.   
By using this theory, 
we construct a filtered complex $(H_{\rm zar}(X/S),P)$ 
in order to overcome obstacles arising from the incompatibility of 
the $p$-adic Steenbrink weight-filtered complex defined in 
\cite{nb} (cf.~\cite{msemi}) with the cup product of the log crystalline complex of $X$ over 
$(S,{\cal I},\gam)$.  We also construct a filtered dga $(H_{{\rm zar},{\rm TW}}(X/S),P)$ 
whose underlying filtered complex is isomorphic to $(H_{\rm zar}(X/S),P)\otimes^L_{\mab Z}{\mab Q}$. 
When $(S,{\cal I},\gam)$ is the canonical lift with the canonical PD-structure 
of the log point of a perfect field $\kap$ of characteristic $p>0$ over the 
formal spectrum of the Witt ring of $\kap$, 
we prove that $(H_{\rm zar}(X/S),P)$ and $(H_{{\rm zar},{\rm TW}}(X/S),P)$ 
are canonically isomorphic to Kim and Hain's filtered complex and 
their filtered dga in \cite{kiha}, respectively. 
\par 
As applications of this theory, we show fundamental properties of 
the weight filtration on the log crystalline cohomology sheaf of $X$ over $(S,{\cal I},\gam)$ 
when the underlying scheme $\os{\circ}{X}$ of $X$ is projective over 
the underlying scheme $\os{\circ}{S}_0$ of $S_0$. 
When $\os{\circ}{S}_0$ is connected,  
we prove that the $p$-adic variational filtered log hard Lefschetz conjecture 
with respect to the weight filtration conjectured in \cite{nb} is true 
when there exists a log fiber $X\times_Ss$ for an exact closed point $s$ of $S$ 
which is the log special fiber of a
projective strict semistable family over a complete discrete valuation ring 
of any characteristic. Using this result, 
we show the existence of the bilinear form on the primitive part of 
the log isocrystalline cohomology sheaf which is compatible with the weight filtration. 

$${\bf Contents}$$
\bigskip
\parno
\S\ref{sec:int}. Introduction
\parno
\S\ref{sec:snclv}. SNCL schemes
\parno 
\S\ref{sec:stpd}. PD-Hirsch extensions of dga's
\parno 
\S\ref{sec:ldfc}. PD-Hirsch extensions of log crystalline complexes
\parno 
\S\ref{sec:pwf}. Preweight filtrations on PD-Hirsch extensions of log crystalline complexes
\parno 
\S\ref{sec:psc}. PD-Hirsch pre-weight-filtered log crystalline complexes 
\parno 
\S\ref{sec:fcd}. Hirsch weight-filtered log crystalline dga's 
\parno 
\S\ref{sec:fc}. Contravariant functorialities of 
PD-Hirsch pre-weight-filtered log crystalline complexes 
and Hirsch weight-filtered log crystalline dga's 
\parno
\S\ref{sec:bckf}. Filtered base change theorem of PD-Hirsch pre-weight-filtered complexes 
\parno 
\S\ref{sec:mod}.  Monodromy operators of PD-Hirsch pre-weight-filtered log crystalline complexes
\parno 
\S\ref{sec:p}. Cup products of PD-Hirsch pre-weight-filtered 
log crystalline complexes
\parno 
\S\ref{sec:infhdi}. Infinitesimal deformation invariance of Hirsch weight-filtered log crystalline complexes 
and Hirsch weight-filtered log crystalline dga's
\parno 
\S\ref{sec:filbo}. The $E_2$-degeneration of the $p$-adic weight spectral sequence
\parno 
\S\ref{sec:e2}. Convergence of the weight filtration
\parno 
\S\ref{sec:st}. Strict compatibility
\parno 
\S\ref{sec:pssc}.  $p$-adic filtered Steenbrink complexes and trace morphisms 
\parno 
\S\ref{sec:cwt}.  Coincidence of the weight filtrations I
\parno 
\S\ref{sec:csc}. Semi-cosimplicial $p$-adic Steenbrink complexes and 
coincidence of the weight filtrations II
\parno 
\S\ref{sec:vlc}. $p$-adic variational filtered  log hard Lefschetz conjecture 
\parno 
\S\ref{sec:pol}. 
$p$-adic polarizations on log crystalline cohomology sheaves of projective SNCL schemes
\parno 
\S\ref{sec:ofc}. Comparison theorems between our filtered complex (resp.~
our filtered dga) with Kim-Hain's filtered complex (resp.~their filtered dga)

\bigskip
\parno

\section{Introduction}\label{sec:int} 
In this book, for a formal family of log points in characteristic $p>0$, 
we define the log crystalline analogue 
of the log de Rham complex of the Kato-Nakayama space(=the real blow up) 
$T^{\log}$ (\cite{kn}) of a log analytic family $T$ of log points. 
We use the PD-Hirsch extension to define it. 
We show that this very simple complex has an important application  
for the cup product of the log crystalline cohomology sheaf of 
a proper SNCL(=simple normal crossing log) scheme over a family of log points 
in characteristic $p>0$. 
\par 
First let us recall a classical result (\cite{sga72}, \cite{sti}, \cite{gn}). 
\par 
Let 
$\os{\circ}{\Del}:=\{\tau\in {\mab C}~\vert~\vert \tau \vert <1\}$ 
be the open unit disk over ${\mab C}$.
Let $\os{\circ}{\cal X}$ 
be an analytic (not necessarily strict) semistable family over $\os{\circ}{\Del}$. 
Let $\os{\circ}{f}\col \os{\circ}{\cal X}\lo \os{\circ}{\Del}$ be the structural morphism. 
Endow $\os{\circ}{\cal X}$ with the canonical log structure obtained by the special fiber 
of $\os{\circ}{\cal X}/\os{\circ}{\Del}$ 
and let ${\cal X}$ be the resulting log analytic space. 
(See \cite{klog1} and \cite{kn} for the definitions of log schemes and log analytic spaces.) 
In particular, we obtain $\Del$. 
Let $f\col {\cal X}\lo \Del$ be the structural morphism. 
Let $\os{\circ}{O}$ be the set of the origin of $\os{\circ}{\Del}$. 
Let $O:=(\os{\circ}{O}, {\mab N}\oplus {\mab C}^*\lo {\mab C})$ 
be the log point of $\os{\circ}{O}$, where the morphism 
${\mab N}\oplus {\mab C}^*\lo {\mab C}$ is 
given by $(n,a)\lom 0^na$ 
$(n\in {\mab N}, a\in {\mab C}^*)$, where $0^n=0\in {\mab C}$ 
for $n\not =0$ and $0^0:=1\in {\mab C}$. 
Let $X$ be the log special fiber of ${\cal X}/{\Del}$ over $O$. 
Set $\Del^*:={\Del}\setminus O$ 
and let $j_{\Del}\col \Del^*\os{\sus}{\lo} \Del$ 
be the natural inclusion.  
Let ${\mathfrak H}\lo \Del^*$ be the universal cover of $\Del^*$. 
Let us consider the following diagram whose all squares are cartesian: 
\begin{equation*}
\begin{CD}
X@>{i}>> {\cal X}@<{j}<< {\mathfrak X}@<<< \ol{\mathfrak X}\\
@VVV @V{f}VV @VVV @VVV\\
O@>{\subset}>> \Del @<{j_\Del}<< \Del^*@<<< {\mathfrak H}. 
\end{CD}
\tag{1.0.1}\label{cd:comp}
\end{equation*}
Let $\ol{j} \col \ol{\mathfrak X}\lo {\mathfrak X}\lo {\cal X}$ be the composite morphism. 
Let $R\Psi({\mab C}):=i^{-1}R\ol{j}_*({\mab C}_{\ol{\mathfrak X}})$ be 
the nearby cycle sheaf 
of the constant sheaf ${\mab C}_{\ol{\mathfrak X}}$ on $\ol{\mathfrak X}$ 
and let $\Om^{\bul}_{{\cal X}/{\mab C}}$ 
be the log de Rham complex of ${\cal X}$ over ${\mab C}$ 
(not the classical de Rham complex of $\os{\circ}{\cal X}$ over ${\mab C}$). 
Then $R\Psi({\mab C})$ is isomorphic to 
the complex $i^{-1}(\Om^{\bul}_{{\cal X}/{\mab C}}[\log \tau])$ 
(\cite{sga72}, \cite{sti}, \cite{gn}), 
where 
\begin{align*}
\Om^{\bul}_{{\cal X}/{\mab C}}[\log \tau]
:={\mab C}[\log \tau]
\otimes_{\mab C}\Om^{\bul}_{{\cal X}/{\mab C}}
\tag{1.0.2}\label{ali:hi}
\end{align*} 
is considered 
as a subdga of $\ol{j}_*(\Om^{\bul}_{\ol{\mathfrak X}/{\mab C}})$.
Furthermore, as proved in \cite{ns},  
it is not difficult to prove that 
the following natural morphism 
\begin{align*}
i^{-1}(\Om^{\bul}_{{\cal X}/{\mab C}}[\log \tau])
\lo \Om^{\bul}_{X/{\mab C}}[\log \tau]:={\mab C}[\log \tau]
\otimes_{\mab C}\Om^{\bul}_{X/{\mab C}}
\tag{1.0.3}\label{ali:his}
\end{align*} 
is a quasi-isomorphism.  
Here $\Om^{\bul}_{X/{\mab C}}$ is 
the log de Rham complex of $X$ over ${\mab C}$ 
and $\Om^{\bul}_{X/{\mab C}}[\log \tau]$ is 
the complex defined by the following ${\mab C}$-linear differential 
\begin{align*}
d((\log \tau)^n \otimes \om ):=n(\log \tau)^{n-1}\otimes d\log \tau\wedge \om+
(\log \tau)^n \otimes d\om \quad (\om \in \Om^i_{X/{\mab C}}).
\end{align*} 
The complex $\Om^{\bul}_{X/{\mab C}}[\log \tau]$ 
is, what is called, the Hirsch extension of 
$\Om^{\bul}_{X/{\mab C}}$ 
by a vector space ${\mab C}\log \tau$ over ${\mab C}$ of rank 1 
with respect to the following composite morphism 
\begin{align*}
{\mab C}\log \tau \lo
{\rm Ker}(\Om^1_{{\cal X}/{\mab C}}\lo \Om^2_{{\cal X}/{\mab C}})
\lo{\rm Ker}(\Om^1_{X/{\mab C}}\lo \Om^2_{X/{\mab C}}),
\tag{1.0.4}\label{ali:hios}
\end{align*} 
where the first morphism is defined by 
$\log \tau \lom d\log \tau$. 
The definition of the Hirsch extension of a dga(=differential graded algebra) is as follows. 
\par 
Let $A$ be a commutative ring with unit element 
and let $(B^{\bul},d)$ be a dga over $A$.  
The Hirsch extension of $B^{\bul}$ by 
a free $A$-module $M$ is, by definition, 
a natural inclusion 
\begin{align*} 
B^{\bul} \os{\sus}{\lo} {\rm Sym}_A(M)\otimes_AB^{\bul}, 
\end{align*} 
where $M$ is placed at degree $0$ 
with an $A$-linear map $\varphi \col M\lo {\rm Ker}(d\col B^1\lo B^2)$  
and ${\rm Sym}_A(M)$ is a symmetric algebra over $A$ generated by $M$;   
${\rm Sym}_A(M)\otimes_AB^{\bul}$ becomes a dga over $A$ by 
$d\col B^{\bul}\lo B^{\bul+1}$ 
and $d_M(x):=\varphi(x)$ $(x\in M)$. 
(See \cite{su} for a more general notion of the Hirsch extension.) 
\par 
The complex $\Om^{\bul}_{X/{\mab C}}[\log \tau]$ 
has a close relation with the log de Rham complex 
of the Kato-Nakayama space $T^{\log}$ of  
an fs log analytic space $T$ over ${\mab C}$.
Indeed,  
let $\eps_T \col T^{\log} \lo T$ 
be the natural morphism of topological spaces defined in \cite{kn}. 
Then $\Om^{\bul}_{O^{\log}/{\mab C},x}=({\mab C}[\log \tau]\lo {\mab C}[\log \tau]d\log \tau)$  
for any point $x\in O^{\log}={\mab S}^1$
and $\Om^{\bul}_{O^{\log}/{\mab C},x}$ is the Hirsch extension of 
$\Gam(O,\Om^{\bul}_{O/{\mab C}})=({\mab C}\lo {\mab C}d\log \tau)$ 
by ${\mab C}\log \tau$ with respect to the morphism (\ref{ali:hios}) in the case $X=O$.
In \cite{fn} Fujisawa and Nakayama have proved that 
\begin{align*} 
\Om^{\bul}_{X/{\mab C}}[\log \tau]=
R(\eps_X \pi_X)_*\pi^{-1}_X
(\Om^{\bul}_{X^{\log}/{\mab C}}), 
\tag{1.0.5}\label{ali:hokanl}
\end{align*}
where $\pi_X$ is the natural morphism 
$X^{\log}\times_{O^{\log}}{\mab R}=X^{\log}\times_{{\mab S}^1}{\mab R}\lo X^{\log}$. 
Though we give up defining the $p$-adic analogue of $X^{\log}$ and 
$X^{\log}\times_{O^{\log}}{\mab R}$, we define the $p$-adic analogue of 
$\Om^{\bul}_{X/{\mab C}}[\log \tau]$ instead. 
This is a case in which it is possible to define an appropriate sheaf of dga's
even when it is difficult to define an appropriate space. 
\par 
In this book, in order to give applications for the (integral) log crystalline cohomology sheaf 
of a proper SNCL scheme in characteristic $p>0$,  
we define the (derived) Hirsch extension  obtained 
by the PD-polynomial algebra $\Gam_A(M)$
generated by $M$ over $A$ (\cite[Appendix A]{bob}) for a certain complex
instead of the Hirsch extension of a dga obtained by ${\rm Sym}_A(M)$. 
We call this extension the {\it $($derived$)$ PD-Hirsch extension} of the complex.    
An example of a PD-Hirsch extension has first appeared in an article 
by Kim and Hain (\cite{kiha}), which will be explained later. 

\par  
To explain our motivation for studying the (derived) PD-Hirsch extension, 
let us raise a fundamental problem on 
the log crystalline cohomology sheaf of 
a proper SNCL scheme 
over a family of log points in characteristic $p>0$. 
\par 
For a log (formal) scheme $Y$, denote by $\os{\circ}{Y}$ the underlying 
(formal) scheme of $Y$. 
Let $S$ be a $p$-adic formal family of log points defined in \cite{nb}; 
locally on $S$, $S$ is isomorphic to a log $p$-adic formal scheme 
$(\os{\circ}{S}, {\mab N}\oplus {\cal O}_S^*\lo {\cal O}_S)$, 
where the morphism ${\mab N}\oplus {\cal O}_S^*\lo {\cal O}_S$ 
is defined by the morphism
$(n,a)\lom 0^na$ $(n\in {\mab N}, a\in {\cal O}_S^*)$, where $0^n=0\in {\cal O}_S$ 
for $n\not =0$ and $0^0:=1\in {\cal O}_S$. 
Let $(S,{\cal I},\gam)$ be a $p$-adic formal PD-family of log points 
($S$ is a $p$-adic formal family of log points and 
${\cal I}$ is a quasi-coherent $p$-adic PD-ideal sheaf of ${\cal O}_S$ 
with PD-structure $\gam$). 
Let $S_0$ be an exact closed log subscheme of $S$ defined by ${\cal I}$. 
Let $X/S_0$ be a proper SNCL scheme with structural morphism 
$f\col X\lo S_0\os{\sus}{\lo} S$. 
(In \S\ref{sec:snclv} below we recall the definition of the SNCL scheme.)
Let $\{\os{\circ}{X}_{\lam}\}_{\lam \in \Lam}$ be 
the set of smooth components of $X/S_0$ defined in \cite{nb}.  
(When $\os{\circ}{S}_0$ is the spectrum of a field of characteristic $p>0$, 
$\{\os{\circ}{X}_{\lam}\}_{\lam \in \Lam}$ can be taken as 
the set of the irreducible components of $\os{\circ}{X}$.) 
For a nonnegative integer $k$, let 
\begin{align*} 
\os{\circ}{X}{}^{(k)}:=
\coprod_{\{\{\lam_0,\ldots, \lam_k\}~\vert \lam_i\in \Lam, \lam_i\not=\lam_j (i\not=j)\}} 
\os{\circ}{X}_{\lam_0}\cap \cdots \cap \os{\circ}{X}_{\lam_k}
\tag{1.0.6}\label{ali:hkis}
\end{align*} 
be a scheme over $\os{\circ}{S}_0$ well-defined in \cite{nb}. 
Let $a^{(k)} \col \os{\circ}{X}{}^{(k)}\lo \os{\circ}{X}$ be the natural morphism. 
Let $F_{\os{\circ}{S}_0}\col \os{\circ}{S}_0\lo \os{\circ}{S}_0$ be the absolute Frobenius 
endomorphism of $\os{\circ}{S}_0$ and 
set $S_0^{[p]}:=S_0\times_{\os{\circ}{S}_0,F_{\os{\circ}{S}_0}}\os{\circ}{S}_0$.  
Let 
$F_{S_0/\os{\circ}{S}_0}\col  S_0\lo S_0^{[p]}$ 
be the relative Frobenius morphism of $S_0$ over $\os{\circ}{S}_0$. 
Let $S_0^{[p]}(S)$ be a log formal scheme whose underlying formal scheme 
is $\os{\circ}{S}$ and whose log structure $M_{S_0^{[p]}(S)}$ 
is a unique sub-log structure of $S$ 
such that the isomorphism 
$M_S/{\cal O}_S^*\os{\sim}{\lo} M_{S_0}/{\cal O}_{S_0}^*$ induces 
the following isomorphism 
\begin{align*} 
M_{S_0^{[p]}(S)}/{\cal O}_S^*\os{\sim}{\lo} 
{\rm Im}(F_{S_0/\os{\circ}{S}_0}^*\col 
F_{S_0/\os{\circ}{S}_0}^*(M_{S_0\times_{\os{\circ}{S}_0,F_{\os{\circ}{S}_0}}\os{\circ}{S}_0})
\lo M_{S_0})/{\cal O}_{S_0}^*. \tag{1.0.7}\label{ali:rus0avp}
\end{align*} 
(The structural morphism of $M_{S_0^{[p]}(S)}$ is the composite morphism 
$M_{S_0^{[p]}(S)}\os{\sus}{\lo} M_S\lo {\cal O}_S$.) 
We have an obvious morphism $(S,{\cal I},\gam)
\lo (S_0^{[p]}(S),{\cal I},\gam)$ of log PD-formal schemes. 
Let $(X/S)_{\rm crys}$ be the log crystalline topos of $X/(S,{\cal I},\gam)$ 
defined in \cite{klog1} and let ${\cal O}_{X/S}$ be the structure sheaf of 
$(X/S)_{\rm crys}$. 
Let $X_{\rm zar}$ be the zariski topos  of $\os{\circ}{X}$. 
Let $u_{X/S} \col (X/S)_{\rm crys}\lo X_{\rm zar}$ be the canonical projection. 
Set $f_{X/S}:=f\circ u_{X/S}$. 
Let ${\rm D}^+{\rm F}(f^{-1}({\cal O}_S))$ be the derived category of 
bounded below filtered complexes of $f^{-1}({\cal O}_S)$-modules 
and let $D^+(f^{-1}({\cal O}_S))$ be the derived category of 
bounded below complexes of $f^{-1}({\cal O}_S)$-modules.  
In \cite{nb} we have proved the following:

\begin{theo}[{\bf \cite[Existence of the zarisikian $p$-adic Steenbrink complex]{nb}}]\label{theo:wdpc}
Let $\vp_{\rm crys}^{(m)}(\os{\circ}{X}/\os{\circ}{S})$ $(m\in {\mab N})$ 
be the crystalline orientation sheaf associated to 
the set $\{\os{\circ}{X}_{\lam}\}_{\lam \in \Lam}$. That is, 
$\vp_{\rm crys}^{(m)}(\os{\circ}{X}/\os{\circ}{S})$ is 
the extension to $(\os{\circ}{X}{}^{(m)}/\os{\circ}{S})_{\rm crys}$ 
of the direct sum of  
$\os{m}{\bigwedge}{\mab Z}^E_{\os{\circ}{X}_{\lam_0}\cap \cdots \cap \os{\circ}{X}_{\lam_m}}$ 
in the Zariski topos $\os{\circ}{X}{}^{(m)}_{\rm zar}$ of $\os{\circ}{X}{}^{(m)}$ 
for the subsets $E=\{\os{\circ}{X}_{\lam_0}, \ldots, \os{\circ}{X}_{\lam_m}\}$'s 
of $\{\os{\circ}{X}_{\lam}\}_{\lam \in \Lam}$
with $\# E=m+1$.  
Then there exists a filtered complex 
\begin{align*} 
(A_{\rm zar}(X/S),P)
\in {\rm D}^+{\rm F}(f^{-1}({\cal O}_S))
\tag{1.1.1}\label{ali:raxs}
\end{align*} 
with a canonical isomorphism 
\begin{align*} 
\theta \wedge \col Ru_{X/S*}({\cal O}_{X/S})\os{\sim}{\lo} A_{\rm zar}(X/S)
\tag{1.1.2}\label{ali:ixuaoa} 
\end{align*} 
in $D^+(f^{-1}({\cal O}_S))$ 
such that 
\begin{align*} 
{\rm gr}^P_kA_{\rm zar}(X/S)\os{\sim}{\lo} \bigoplus_{j\geq \max \{-k,0\}} 
&a^{(2j+k)}_* 
(Ru^{\rm crys}_{\os{\circ}{X}{}^{(2j+k)}/\os{\circ}{S}*}
({\cal O}_{{\os{\circ}{X}{}^{(2j+k)}/\os{\circ}{S}}}) \tag{1.1.3}\label{ali:ruoovp}\\
&\otimes_{\mab Z}\vp_{\rm crys}^{(2j+k)}(\os{\circ}{X}/\os{\circ}{S}))(-j-k)[-2j-k]
\end{align*} 
in $D^+(f^{-1}({\cal O}_S))$.  
Here the Tate twist $(-j-k)$ means the Tate twist 
with respect to 
the morphism $X\lo X\times_{\os{\circ}{S}_0,F_{\os{\circ}{S}_0}}\os{\circ}{S}_0$ 
over $(S,{\cal I},\gam) \lo (S_0^{[p]}(S),{\cal I},\gam)$
induced by the absolute Frobenius endomorphism $F_X\col X\lo X$ of $X$. 
\end{theo}
As a corollary of this theorem, we obtain 
the weight filtration $P$ on $R^qf_{X/S*}({\cal O}_{X/S})$ $(q\in {\mab N})$: 
\begin{align*} 
P_{k+q}R^qf_{X/S*}({\cal O}_{X/S})&:={\rm Im}
(R^qf_{X/S*}(P_kA_{\rm zar}(X/S))\lo R^qf_{X/S*}(A_{\rm zar}(X/S)))\tag{1.1.4}\label{ali:wacp} \\
&\simeq {\rm Im}(R^qf_{X/S*}(P_kA_{\rm zar}(X/S))\lo R^qf_{X/S*}({\cal O}_{X/S})).
\end{align*} 
See \cite{msemi} and \cite{ndw} for another approach for the construction of 
$P$ on $R^qf_{X/S*}({\cal O}_{X/S})$ 
by the use of log de Rham-Witt complexes in the case where $S$ is the canonical lift of 
the log point of a perfect field $\kap$ of characteristic $p>0$ over 
the Witt ring ${\cal W}$ of $\kap$. In this case, in \cite{nb} 
we have proved that $(A_{\rm zar}(X/S),P)$ is 
canonically isomorphic to the weight-filtered complex 
$({\cal W}A^{\bul}_X,P)$ constructed in 
\cite{msemi} and \cite{ndw}. 
\par 
It is  natural to raise the following fundamental problem 
for the filtration $P$ on $R^qf_{X/S*}({\cal O}_{X/S})$: does the cup product 
\begin{align*} 
\cup  \col R^qf_{X/S*}({\cal O}_{X/S})\otimes_{{\cal O}_S}
R^{q'}f_{X/S*}({\cal O}_{X/S})
\lo R^{q+q'}f_{X/S*}({\cal O}_{X/S})
\tag{1.1.5}\label{ali:cp} 
\end{align*} 
induce the following morphism 
\begin{align*} 
\cup  \col P_kR^qf_{X/S*}({\cal O}_{X/S})
\otimes_{{\cal O}_S}P_{k'}R^{q'}f_{X/S*}({\cal O}_{X/S})
\lo P_{k+k'}R^{q+q'}f_{X/S*}({\cal O}_{X/S})~?
\tag{1.1.6}\label{ali:qqxs}
\end{align*} 
Let $(T,{\cal J},\del)$ be a $p$-adic formal PD-scheme and let 
$T_0$ be a closed subscheme of $T$ defined by ${\cal J}$. 
Let $(Y, D)$ be a proper smooth scheme with a relative SNCD(=simple normal crossing divisor) 
over $T_0$. By using the theory in \cite{nh2}, 
we can give the analogous problem for $(Y,D)/(T,{\cal J},\del)$. 
In this case, it is very easy to solve the problem affirmatively. 
Contrary to this case, the problem for $X/(S,{\cal I},\gam)$ is a serious one. 
A filtered complex $(A^{\bul},P)$ representing 
the filtered complex $(A_{\rm zar}(X/S),P)$ 
is constructed by a certain log de Rham complex.  
Because the wedge product of this de Rham complex 
induces a morphism $A^i\times A^j\lo A^{i+j+1}$,  
this morphism does not even induce the morphism (\ref{ali:cp}) 
and hence  
we do not know whether the problem above is affirmatively solved. 
If $\os{\circ}{S}$ is a $p$-adic formal ${\cal V}$-scheme in the sense of \cite{od}
($\os{\circ}{S}$ is a noetherian formal scheme over ${\rm Spf}({\cal V})$ 
with the $p$-adic topology which is topologically of finite type 
over ${\rm Spf}({\cal V})$), 
where ${\cal V}$ is a complete discrete valuation ring of mixed characteristics 
$(0,p)$ with perfect residue field, and if one ignores the torsion in (\ref{ali:qqxs}), 
the problem can be solved affirmatively by 
using theory of convergent isocrystals 
in \cite{od}, using the well-known specialization argument explained 
in \cite{ndw}, \cite{nh2} and \cite{nb} and using the 
purity of the weight 
for the crystalline cohomology of a proper smooth scheme 
(this classical fact has been proved completely in 
\cite{kme}, \cite{clpu} and \cite{ndw} with an indispensable work in \cite{ny}) and 
the weight spectral sequence of $R^qf_{X/S*}({\cal O}_{X/S})$
arising from $(A_{\rm zar}(X/S),P)$. 
However we are not satisfied with this (ad hoc) method. 
We would like to construct the filtered cup product of a weight-filtered complex 
such that the cup product of the underlying complex of this filtered complex 
is compatible with the cup product of $R^qf_{X/S*}({\cal O}_{X/S})$.  

\par 
For simplicity, assume that there exists an immersion 
$X\os{\sus}{\lo} {\cal P}$ into a log smooth  log $p$-adic formal scheme 
over $S$ in this introduction. 
(In the text we do not assume the existence of this immersion.)
Let ${\cal P}^{\rm ex}$ be the exactification of this immersion  (\cite{s3}).  
Let ${\mathfrak D}$ be the log PD-envelope of the exact closed immersion 
$X\os{\sus}{\lo} {\cal P}^{\rm ex}$ over $(S,{\cal I},\gam)$. 
Set 
$$\wt{R}u_{X/\os{\circ}{S}*}({\cal O}_{X/\os{\circ}{S}})
:={\cal O}_{\mathfrak D}\otimes_{{\cal O}_{{\cal P}^{\rm ex}}}
\Om^{\bul}_{{\cal P}^{\rm ex}/\os{\circ}{S}}
\in D^+(f^{-1}({\cal O}_S)).$$ 
One can show that 
$\wt{R}u_{X/\os{\circ}{S}*}({\cal O}_{X/\os{\circ}{S}})$ is independent of 
the choice of the immersion $X\os{\sus}{\lo} {\cal P}$. 
The complex $\wt{R}u_{X/\os{\circ}{S}*}({\cal O}_{X/\os{\circ}{S}})$ 
is not equal to $Ru_{X/\os{\circ}{S}*}({\cal O}_{X/\os{\circ}{S}})$ 
since ${\cal P}^{\rm ex}$ is not log smooth over $\os{\circ}{S}$ in general. 
In \cite{nb} we have called $\wt{R}u_{X/\os{\circ}{S}*}({\cal O}_{X/\os{\circ}{S}})$ the 
modified log crystalline complex of $X/(\os{\circ}{S},{\cal I},\gam)$. 
Though the notation $\wt{R}u_{X/\os{\circ}{S}*}
({\cal O}_{X/\os{\circ}{S}})$ is misleading (because 
this complex may depend on the morphism $X\lo S$), 
we use this notation by mimicking  the notation 
``$W\wt{\Om}_X^{\bul}$'' of the log de Rham-Witt complex 
in the case where $\os{\circ}{S}_0$ is the spectrum of a perfect field of characteristic $p>0$. 
\par 
Let $\wt{t}$ be a local section of $M_S$ 
whose image $t$ in $\ol{M}_S:=M_S/{\cal O}_S^*$ is the local generator. 
Let $U_S$ be a free ${\cal O}_S$-module of rank $1$ with a basis $u$: 
$U_S={\cal O}_Su$. 
Though we do not ask what $u$ is, we consider $u$ as 
``$\log t$'' (not ``$\log \wt{t}$'') in our mind. 
For a local section $\wt{t}{}'=a\wt{t}$ $(a\in {\cal O}_S^*)$ of $M_S$, 
we do not change $U_S$, though it may be very strange at first glance. 
Let $\Gam_{{\cal O}_S}(U_S)$ 
be the PD-polynomial algebra generated by $U_S$ over ${\cal O}_S$: 
$\Gam_{{\cal O}_S}(U_S):={\cal O}_S\langle u \rangle 
=\bigoplus_{n=0}^{\infty}{\cal O}_S u^{[n]}$. 
Because 
${\cal O}_{\mathfrak D}\otimes_{{\cal O}_{{\cal P}^{\rm ex}}}
\Om^{\bul}_{{\cal P}^{\rm ex}/\os{\circ}{S}}$ 
is a dga over $S$, 
we obtain the PD-Hirsch extension 
$$
{\cal O}_{\mathfrak D}\otimes_{{\cal O}_{{\cal P}^{\rm ex}}}
\Om^{\bul}_{{\cal P}^{\rm ex}/\os{\circ}{S}}\langle u\rangle 
:=\Gam_{{\cal O}_S}(U_S)\otimes_{{\cal O}_S}
{\cal O}_{\mathfrak D}\otimes_{{\cal O}_{{\cal P}^{\rm ex}}}
\Om^{\bul}_{{\cal P}^{\rm ex}/\os{\circ}{S}}$$
of ${\cal O}_{\mathfrak D}\otimes_{{\cal O}_{{\cal P}^{\rm ex}}}
\Om^{\bul}_{{\cal P}^{\rm ex}/\os{\circ}{S}}$ 
by using a $f^{-1}({\cal O}_S)$-linear map 
\begin{align*} 
\varphi \col \Gam_{{\cal O}_S}(U_S)\owns au^{[n]}\lom 
au^{[n-1]}d\log \wt{t}\in &
{\rm Ker}({\cal O}_{\mathfrak D}\otimes_{{\cal O}_{{\cal P}^{\rm ex}}}
\Om^1_{{\cal P}^{\rm ex}/\os{\circ}{S}}\lo 
{\cal O}_{\mathfrak D}\otimes_{{\cal O}_{{\cal P}^{\rm ex}}}\Om^2_{{\cal P}^{\rm ex}/\os{\circ}{S}})\\
& (a\in f^{-1}({\cal O}_S)),
\end{align*} 
which is independent of the choice of $\wt{t}$.   
(Strictly speaking, we should denote two  $\Gam_{{\cal O}_S}(U_S)$'s above 
by $f^{-1}(\Gam_{{\cal O}_S}(U_S))$.)
The first result needed for the construction of our desired filtered complex 
is the following: 

\begin{theo}\label{theo:dpe}
Set 
$$\wt{R}u_{X/\os{\circ}{S}*}({\cal O}_{X/\os{\circ}{S}}\langle u\rangle )
:={\cal O}_{\mathfrak D}\otimes_{{\cal O}_{{\cal P}^{\rm ex}}}
\Om^{\bul}_{{\cal P}^{\rm ex}/\os{\circ}{S}}\langle u\rangle 
\in D^+(f^{-1}({\cal O}_S)).$$ 
Then the natural morphism defined by 
$u^{[n]}\otimes \om \lom 0$ $(n\geq 1, 
\om \in {\cal O}_{\mathfrak D}\otimes_{{\cal O}_{{\cal P}^{\rm ex}}}
\Om^{\bul}_{{\cal P}^{\rm ex}/\os{\circ}{S}})$ and 
the projection $\Om^{\bul}_{{\cal P}^{\rm ex}/\os{\circ}{S}}
\lo 
\Om^{\bul}_{{\cal P}^{\rm ex}/S}$ 
induce the following canonical isomorphism$:$ 
\begin{align*} 
\wt{R}u_{X/\os{\circ}{S}*}({\cal O}_{X/\os{\circ}{S}}\langle u\rangle) 
\os{\sim}{\lo} Ru_{X/S*}({\cal O}_{X/S}). 
\tag{1.2.1}\label{ali:uysys}
\end{align*} 
\end{theo}   
The complex $\wt{R}u_{X/\os{\circ}{S}*}({\cal O}_{X/\os{\circ}{S}}\langle u\rangle )$ 
in the case where $\os{\circ}{S}_0$ is a point of characteristic $p$ 
is a log crystalline analogue of the complex 
$\Om^{\bul}_{X/{\mab C}}[\log \tau]$ in (\ref{ali:his}).  
Moreover we can define a filtered complex 
$(\wt{R}u_{X/\os{\circ}{S}*}({\cal O}_{X/\os{\circ}{S}}\langle u\rangle ),P)\in 
{\rm D}^+{\rm F}(f^{-1}({\cal O}_S))$
by counting the number of log poles of local sections of 
$\Om^{\bul}_{{\cal P}^{\rm ex}/\os{\circ}{S}}$ 
and by giving the weight $2n$ for $u^{[n]}$ $(n\in {\mab N})$. 
However we see that the graded complex 
${\rm gr}^P_0\wt{R}u_{X/\os{\circ}{S}*}({\cal O}_{X/\os{\circ}{S}}\langle u\rangle )$ 
is ``mixed''. 
Consequently we need another filtered complex. 
Because we can prove that 
${\cal P}^{\rm ex}$ is a formal SNCL scheme, we 
can define $\os{\circ}{\cal P}{}^{{\rm ex},(m)}$ $(m\in {\mab N})$ as in (\ref{ali:hkis}). 
Let ${\cal P}{}^{{\rm ex},(m)}$ be the log formal scheme whose underlying formal scheme is 
$\os{\circ}{\cal P}{}^{{\rm ex},(m)}$ and whose log structure 
is the inverse image of ${\cal P}^{\rm ex}$ 
by the natural morphisms
$\os{\circ}{\cal P}{}^{(m)}\lo \os{\circ}{\cal P}$. 
Because $\{{\cal P}{}^{{\rm ex},(m)}\}_{m\in {\mab N}}$  defines 
a semi-simplicial log formal scheme,  
we can consider the double complex 
${\cal O}_{\mathfrak D}\otimes_{{\cal O}_{{\cal P}^{{\rm ex}}}}
\Om^{\bul}_{{\cal P}^{{\rm ex},(\bul)}/\os{\circ}{S}}\langle u\rangle$ 
with appropriate signs and the associated  
single complex 
$s({\cal O}_{\mathfrak D}\otimes_{{\cal O}_{{\cal P}^{{\rm ex}}}}
\Om^{\bul}_{{\cal P}^{{\rm ex},(\bul)}/\os{\circ}{S}}\langle u\rangle)$.  
Let $P$ be the diagonal filtration (\cite{dh2}) on 
$s({\cal O}_{\mathfrak D}\otimes_{{\cal O}_{{\cal P}^{{\rm ex}}}}
\Om^{\bul}_{{\cal P}^{{\rm ex},(\bul)}/\os{\circ}{S}}\langle u\rangle)$ 
by counting the number of log poles 
$\Om^{\bul}_{{\cal P}^{{\rm ex},(m)}/\os{\circ}{S}}$ 
of local sections of 
${\cal O}_{\mathfrak D}\otimes_{{\cal O}_{{\cal P}^{{\rm ex}}}}
\Om^{\bul}_{{\cal P}^{{\rm ex},(m)}/\os{\circ}{S}}$ for each $m\in {\mab N}$, 
by giving the weight $2n$ for $u^{[n]}$ and by raising the 
weight according to the semi-simplicial degree of 
${\cal P}^{{\rm ex},(\bul)}$. 
Then we prove the following:

\begin{theo}[{\bf Existence of the PD-Hirsch weight-filtered complex}]\label{theo:dpfe}
Set 
$$(H_{\rm zar}(X/S),P):=
(s({\cal O}_{\mathfrak D}\otimes_{{\cal O}_{{\cal P}^{{\rm ex}}}}
\Om^{\bul}_{{\cal P}^{{\rm ex},(\bul)}/\os{\circ}{S}}\langle u\rangle),P)
\in {\rm D}^+{\rm F}(f^{-1}({\cal O}_S)).$$ 
The natural morphism ${\cal P}^{{\rm ex},(0)}\lo {\cal P}^{\rm ex}$ 
induces the following isomorphism
\begin{align*} 
\wt{R}u_{X/\os{\circ}{S}*}({\cal O}_{X/\os{\circ}{S}}\langle u\rangle) 
\os{\sim}{\lo} H_{\rm zar}(X/S) 
\tag{1.3.1}\label{ali:uyhsys}
\end{align*} 
in $D^+(f^{-1}({\cal O}_S))$. 
The graded complex ${\rm gr}_k^P(H_{\rm zar}(X/S))$ is canonically isomorphic to 
\begin{align*} 
&\bigoplus_{m\geq 0}
\bigoplus_{\# \ul{\lam}=m+1}
\bigoplus_{j\geq 0}
\bigoplus_{\# \ul{\mu}=k+m-2j}\tag{1.3.2}\label{ali:kmj}\\
&a_{\ul{\lam}\cup \ul{\mu}*}
Ru_{\os{\circ}{X}_{\ul{\lam}\cup \ul{\mu}}/\os{\circ}{S}*} 
({\cal O}_{\os{\circ}{X}_{\ul{\lam}\cup \ul{\mu}}/\os{\circ}{S}}
\otimes_{\mab Z}\vp_{{\rm crys},\ul{\mu}}
(\os{\circ}{X}_{\ul{\lam}\cup \ul{\mu}}/\os{\circ}{S}))(-k-m+j)[-k-2m+2j], 
\end{align*} 
where 
$\ul{\lam}:=\{\lam_0,\ldots, \lam_m\}$, $\ul{\mu}:=\{\mu_0,\ldots, \mu_{k+m-2j-1}\}$, 
$\os{\circ}{X}_{\ul{\lam}\cup \ul{\mu}}:=
\os{\circ}{X}_{\lam_0}\cap \cdots \cap \os{\circ}{X}_{\lam_m}\cap \os{\circ}{X}_{\mu_0}\cap \cdots 
\cap \os{\circ}{X}_{\mu_{k+m-2j-1}}$, 
$a_{\ul{\lam}\cup \ul{\mu}} \col \os{\circ}{X}_{\ul{\lam}\cup \ul{\mu}}\os{\subset}{\lo} 
\os{\circ}{X}$ is the natural closed immersion 
and $\vp_{{\rm crys},\ul{\mu}}
(\os{\circ}{X}_{\ul{\lam}\cup \ul{\mu}}/\os{\circ}{S})$ 
is the crystalline orientation sheaf. 
The filtered complex $(H_{\rm zar}(X/S),P)$  
and the isomorphism {\rm (\ref{ali:uyhsys})} 
are independent of the choice of the immersion 
$X\os{\sus}{\lo} {\cal P}$ over $S$. 
If $\os{\circ}{X}$ is quasi-compact, then there exists the following morphism 
\begin{equation*} 
\cup \col (H_{\rm zar}(X/S),P)
\otimes^L_{f^{-1}({\cal O}_S)}
(H_{\rm zar}(X/S),P)\lo 
(H_{\rm zar}(X/S),P)
\tag{1.3.3}\label{eqn:pxs}
\end{equation*} 
in ${\rm D}^+{\rm F}(f^{-1}({\cal O}_S))$  
fitting into the following commutative diagram 
\begin{equation*} 
\begin{CD} 
H_{\rm zar}(X/S)
\otimes^L_{f^{-1}({\cal O}_S)}H_{\rm zar}(X/S)
@>{\cup}>> H_{\rm zar}(X/S)\\
@A{\simeq}AA  @A{\simeq}AA\\
\wt{R}u_{X/S*}({\cal O}_{X/S}\langle u\rangle)\otimes^L_{f^{-1}({\cal O}_S)}
\wt{R}u_{X/S*}({\cal O}_{X/S}\langle u\rangle) @. 
\wt{R}u_{X/S*}({\cal O}_{X/S}\langle u\rangle)
\\
@V{\simeq}VV  @V{\simeq}VV\\
Ru_{X/S*}({\cal O}_{X/S})\otimes^L_{f^{-1}({\cal O}_S)}Ru_{X/S*}({\cal O}_{X/S})
@>{\cup}>>  Ru_{X/S*}({\cal O}_{X/S}). 
\end{CD}
\tag{1.3.4}\label{cd:ixtt}
\end{equation*} 
Here the lower $\cup$ in {\rm (\ref{cd:ixtt})} is the usual cup product 
of $Ru_{X/S*}({\cal O}_{X/S})$. 
\end{theo} 
We call $(H_{\rm zar}(X/S),P)$ 
the {\it PD-Hirsch weight-filtered log crystalline complex} of $X/(S,{\cal I},\gam)$. 
By using (\ref{theo:dpe}) and (\ref{theo:dpfe}),  
we obtain the following:  

\begin{coro}\label{coro:wdc}
There exists a canonical isomorphism 
\begin{align*} 
Ru_{X/S*}({\cal O}_{X/S})\os{\sim}{\lo} H_{\rm zar}(X/S)
\tag{1.4.1}\label{ali:ixuoa} 
\end{align*} 
in $D^+(f^{-1}({\cal O}_S))$.  
If $\os{\circ}{X}$ is quasi-compact, then there exists the following convergent 
spectral sequence 
\begin{align*}
E_1^{-k,q+k}&=\bigoplus_{m\geq 0}
\bigoplus_{\# \ul{\lam}=m+1}
\bigoplus_{j\geq 0}
\bigoplus_{\# \ul{\mu}=k+m-2j}\tag{1.4.2}\label{ali:ixsu} \\
&R^{q+2j-k-2m}
f_{\os{\circ}{X}_{\ul{\lam}\cup \ul{\mu}}/\os{\circ}{S}*} 
({\cal O}_{\os{\circ}{X}_{\ul{\lam}\cup \ul{\mu}}/\os{\circ}{S}}
\otimes_{\mab Z}\vp_{{\rm crys},\ul{\mu}}
(\os{\circ}{X}_{\ul{\lam}\cup \ul{\mu}}/\os{\circ}{S}))(-k-m+j)\Lo  R^qf_{X/S*}({\cal O}_{X/S}). 
\end{align*}
\end{coro} 
We should remark that 
``the filtration $P$ on $H_{\rm zar}(X/S)$''
is not ``finite'' in general. 
However, because the filtration $P$ on $H_{\rm zar}(X/S)$
is exhaustive and  bounded below if $\os{\circ}{X}$ is quasi-compact, 
we obtain the convergent weight spectral sequence (\ref{ali:ixsu}). 
If $\os{\circ}{S}$ is a $p$-adic formal ${\cal V}$-scheme, then 
we can prove that the spectral sequence {\rm (\ref{ali:ixsu})} 
modulo torsion degenerates at $E_2$. 


\par 
Let $P'$ be the induced filtration on $R^qf_{X/S*}({\cal O}_{X/S})$ by 
$(H_{\rm zar}(X/S),P)$: 
\begin{align*} 
P'_{k+q}R^qf_{X/S*}({\cal O}_{X/S})&:={\rm Im}
(R^qf_{X/S*}(P_kH_{\rm zar}(X/S))\lo R^qf_{X/S*}(H_{\rm zar}(X/S)))\tag{1.4.3}\label{ali:whacp} \\
&\simeq {\rm Im}(R^qf_{X/S*}(P_kH_{\rm zar}(X/S))\lo R^qf_{X/S*}({\cal O}_{X/S})).
\end{align*} 
As an immediate corollary of (\ref{theo:dpfe}) and (\ref{coro:wdc}),  
if $\os{\circ}{X}$ is quasi-compact, then 
we see that the cup product (\ref{ali:cp})
induces the following morphism 
\begin{align*} 
\cup  \col P'_kR^qf_{X/S*}({\cal O}_{X/S})
\otimes_{{\cal O}_S}P'_{k'}R^{q'}f_{X/S*}({\cal O}_{X/S})
\lo P'_{k+k'}R^{q+q'}f_{X/S*}({\cal O}_{X/S}),  
\tag{1.4.4}\label{ali:qppqxs}
\end{align*}  
which answers the problem (\ref{ali:qqxs}) affirmatively if one replaces 
$P$ with $P'$.

\par 
Though we think that there is no direct simple relation 
between $(H_{\rm zar}(X/S),P)$ and 
$(A_{\rm zar}(X/S),P)$, we can prove the following:

\begin{theo}[{\bf cf.~\cite{fup}}]\label{theo:crf}
There exists a filtered morphism 
\begin{align*}
\psi \col (A_{\rm zar}(X/S),P)
\lo 
(H_{\rm zar}(X/S),P)
\tag{1.5.1}\label{ali:ahuu} 
\end{align*} 
such that the underlying morphism 
$A_{\rm zar}(X/S)\lo H_{\rm zar}(X/S)$ 
is an isomorphism in $D^+(f^{-1}({\cal O}_S))$ 
fitting into the following commutative diagram
\begin{equation*} 
\begin{CD}
A_{\rm zar}(X/S)@>{\psi,~\sim}>>
H_{\rm zar}(X/S)\\
@A{\theta \wedge}A{\simeq}A @A{(\ref{ali:ixuoa})}A{\simeq}A\\
Ru_{X/S*}({\cal O}_{X/S})
@= Ru_{X/S*}({\cal O}_{X/S}). 
\end{CD}
\tag{1.5.2}\label{cd:sesti}
\end{equation*} 
\end{theo} 
The morphism $\psi$ is the log crystalline analogue of the morphism constructed by 
Fujisawa in  \cite{fup}. (\ref{theo:crf}) tells us that 
$P_kR^qf_{X/S*}({\cal O}_{X/S})\subset P'_kR^qf_{X/S*}({\cal O}_{X/S})$.
Though we do not know whether $P_k=P'_k$ on $R^qf_{X/S*}({\cal O}_{X/S})$,  
we obtain the following:

\begin{theo}[{\bf Comparison theorem of weight filtrations}]\label{theo:intro}
Let ${\cal V}$ be a complete discrete valuation ring of mixed characteristics $(0,p)$ 
with perfect residue field. 
Assume that $\os{\circ}{S}$ is a $p$-adic formal ${\cal V}$-scheme. 
Let $P$ and $P'$ be the filtrations {\rm (\ref{ali:wacp})} and {\rm (\ref{ali:whacp})} 
on $R^qf_{X/S*}({\cal O}_{X/S})$ $(q\in {\mab N})$, respectively. 
Then $P\otimes_{\cal V}K=P'\otimes_{\cal V}K$. 
\end{theo}

\par 
Next we would like to give relations  
between (\ref{theo:dpfe}) and preceding results. 
\par 
Let $s=({\rm Spec}(\kap), {\mab N}\oplus \kap^*\lo \kap)$ be the log point of $\kap$. 
Let ${\cal W}$ be the Witt ring of $\kap$. 
Let ${\cal W}(s)$ be the canonical lift of $s$ over ${\cal W}$ of $\kap$. 
Let $u$ be a variable over ${\cal W}$. 
Let $Y$ be a proper SNCL scheme  over $s$. 
Let $g\col Y\lo s\os{\sus}{\lo} {\cal W}(s)$ be the structural morphism. 
Let ${\cal W}{\Om}^{\bul}_{Y}$ and 
${\cal W}\wt{\Om}^{\bul}_{Y}$ be the log de Rham-Witt complexes of $Y/\kap$ 
constructed in \cite{msemi} (cf.~\cite{hyp}). 
Let ${\cal W}\wt{\Om}^{\bul}_{Y^{(\bul)}}$ be the semi-simplicial version of 
${\cal W}\wt{\Om}^{\bul}_{Y}$, where $Y^{(\bul)}$ is the semi-simplicial log scheme 
whose underlying semi-simplicial scheme is 
$\{\os{\circ}{Y}{}^{(k)}\}_{k\in {\mab N}}$ and whose log structure 
is the inverse image of $Y$ given by the natural morphisms
$\os{\circ}{Y}{}^{(k)}\lo \os{\circ}{Y}$'s.
Let $b^{(\bul)} \col Y^{(\bul)} \lo Y$ be the augmentation morphism. 
In \cite{kiha} Kim and Hain have constructed a complex 
${\cal W}\wt{\Om}{}^{\bul}_{Y}\langle u \rangle$ and 
a filtered complex $(s(b^{(\bul)}_*({\cal W}\wt{\Om}{}^{\bul}_{Y^{(\bul)}}\langle u \rangle)),P)$ 
(in their article $(s(b^{(\bul)}_*({\cal W}\wt{\Om}{}^{\bul}_{Y^{(\bul)}}\langle u \rangle)),P)$ 
is denoted by $(s(C(W\wt{\om}[u])),P)$).  
For simplicity of notation, let us denote 
$(s(b^{(\bul)}_*({\cal W}\wt{\Om}{}^{\bul}_{Y^{(\bul)}}\langle u \rangle)),P)$ by 
$({\cal W}\wt{\Om}{}^{\bul}_{Y^{(*)}}\langle u \rangle,P)$. 
We can prove that there exist the following canonical isomorphisms  
\begin{align*} 
{\cal W}\Om^{\bul}_Y \os{\sim}{\longleftarrow} {\cal W}\wt{\Om}{}^{\bul}_{Y}\langle u \rangle 
\os{\sim}{\lo}
{\cal W}\wt{\Om}{}^{\bul}_{Y^{(*)}}\langle u \rangle
\tag{1.6.1}\label{ali:dwy}
\end{align*} 
in $D^+(g^{-1}({\cal W}))$. 
(See (\ref{rema:pst}) and 
(\ref{rema:p}) (3) below for the problem for the construction of 
$({\cal W}\wt{\Om}{}^{\bul}_{Y^{(*)}}\langle u\rangle,P)$ 
and the problem for the proof of the existence of the isomorphisms (\ref{ali:dwy}).)
In the text we give a simpler construction of 
$({\cal W}\wt{\Om}{}^{\bul}_{X^{(*)}}\langle u \rangle,P)$ 
than Kim and Hain's construction by replacing the local admissible lift defined in \cite{msemi}
by the local immersion into a log smooth scheme over ${\rm Spf}({\cal W}\{t\})$ 
with the log structure defined by the ideal $(t) \subset {\cal W}\{t\}$. 
We show that the complex ${\cal W}\wt{\Om}{}^{\bul}_{Y^{(*)}}\langle u \rangle$ 
has a natural product $\wedge$ 
and the following diagram is commutative: 
\begin{equation*} 
\begin{CD} 
{\cal W}\wt{\Om}{}^{\bul}_{Y^{(*)}}\langle u \rangle
\otimes_{g^{-1}({\cal W})}{\cal W}\wt{\Om}{}^{\bul}_{Y^{(*)}}\langle u \rangle
@>{\wedge}>> {\cal W}\wt{\Om}{}^{\bul}_{Y^{(*)}}\langle u \rangle
\\
@A{\simeq}AA  @A{\simeq}AA\\
{\cal W}\Om^{\bul}_Y\otimes_{g^{-1}({\cal W})}{\cal W}\Om^{\bul}_Y
@>{\wedge}>>  {\cal W}\Om^{\bul}_Y. 
\end{CD}
\tag{1.6.2}\label{cd:ixwctt}
\end{equation*} 

In this book we prove the following comparison theorem: 

\begin{theo}[{\bf Comparison theorem of the filtered complexes}]\label{theo:cpiso}
Let the notations be as in {\rm (\ref{theo:dpfe})}. 
Assume that $\os{\circ}{S}={\rm Spf}({\cal W})$. 
Then there exists a canonical isomorphism 
\begin{align*} 
(H_{\rm zar}(X/S),P)
\os{\sim}{\lo} 
({\cal W}\wt{\Om}{}^{\bul}_{X^{(*)}}\langle u \rangle,P) 
\tag{1.7.1}\label{ali:icxpsu} 
\end{align*}
in ${\rm D}^+{\rm F}(f^{-1}({\cal W}))$. 
If $\os{\circ}{X}$ is quasi-compact, then the following diagram 
\begin{equation*} 
\begin{CD} 
(H_{\rm zar}(X/S),P)
\otimes^L_{f^{-1}({\cal O}_S)}(H_{\rm zar}(X/S),P)
@>{\cup}>> H_{\rm zar}(X/S)\\
@V{\simeq}VV @VV{\simeq}V \\
{\cal W}\wt{\Om}{}^{\bul}_{X^{(*)}}\langle u \rangle
\otimes_{f^{-1}({\cal W})}{\cal W}\wt{\Om}{}^{\bul}_{X^{(*)}}\langle u \rangle
@>{\wedge}>> {\cal W}\wt{\Om}{}^{\bul}_{X^{(*)}}\langle u \rangle
\end{CD}
\tag{1.7.2}\label{cd:ixwhwctt}
\end{equation*} 
is commutative. $($Note that 
${\cal W}\wt{\Om}{}^{\bul}_{X^{(*)}}\langle u \rangle$ is 
a complex of flat $f^{-1}({\cal W})$-modules.$)$
The underlying isomorphism $H_{\rm zar}(X/S)\os{\sim}{\lo} 
{\cal W}\wt{\Om}{}^{\bul}_{X^{(*)}}\langle u \rangle$ 
and the isomorphism 
$Ru_{X/S*}({\cal O}_{X/S})\os{\sim}{\lo} {\cal W}\Om^{\bul}_X$ 
in {\rm \cite{hk}} and {\rm \cite{ndw}}
give the compatibility of {\rm (\ref{cd:ixtt})} 
with {\rm (\ref{cd:ixwctt})} in the obvious sense. 
\end{theo}
\parno
Consequently we see that 
our filtered complex $(H_{\rm zar}(X/S),P)$ is 
a generalization of $({\cal W}\wt{\Om}{}^{\bul}_{Y^{(*)}}\langle u \rangle,P)$ 
to the case where the base log scheme is more general. 
\par 
On the other hand, in \cite{fup}
Fujisawa has defined a filtered complex $(K_{\mab C},W)$ for 
a proper SNCL scheme $X$ over the log point $O$. 
In fact, he has defined a mixed Hodge complex for $X/O$ 
under the assumption that the analytifications 
of the irreducible components of $\os{\circ}{X}$ are K\"{a}hler.  
Our filtered complex $(H_{\rm zar}(X/S),P)$ 
is the log crystalline analogue of a generalization of 
$(K_{\mab C},W)$. 
\par 
Next we would like to raise another problem. 
Assume that the relative dimension of 
$\os{\circ}{X} \lo \os{\circ}{S}_0$ is of pure dimension $d$. 
Let $L$ be a relatively ample line bundle on 
$\os{\circ}{X}/\os{\circ}{S}_0$. 
As in \cite[\S3]{boi},   
we obtain the Chern class 
$\eta:=c_{1,{\rm crys}}(L)$ of $L$
in $R^2f_{X/S*}({\cal O}_{X/S})$. 
(In the text we recall the construction of $\eta$.)
Stimulated by a letter of K.~Kato to me (\cite{kln}), 
I have conjectured the following in \cite{nb}: 

\begin{conj}[{\bf $p$-adic variational filtered log hard Lefschetz conjecture}]\label{conj:lhilc}
$(1)$ The following cup product 
\begin{equation*} 
\eta^i \col 
R^{d-i}f_{X/S*}
({\cal O}_{X/S})\otimes_{\mab Z}{\mab Q} 
\lo R^{d+i}f_{X/S*}({\cal O}_{X/S})
\otimes_{\mab Z}{\mab Q}
\tag{1.8.1}\label{eqn:fcvilpl} 
\end{equation*}
is an isomorphism. 
\par 
$(2)$ 
In fact, $\eta^i$ is the following isomorphism of filtered sheaves: 
\begin{equation*} 
\eta^i \col 
(R^{d-i}f_{X/S*}
({\cal O}_{X/S})\otimes_{\mab Z}{\mab Q},P) 
\os{\sim}{\lo} ((R^{d+i}f_{X/S*}
({\cal O}_{X/S})\otimes_{\mab Z}{\mab Q})(i),P). 
\tag{1.8.2}\label{eqn:filfilpl} 
\end{equation*}
Here $(i)$ means the Tate twist: 
$$P_k(R^{d+i}f_{X/S*}({\cal O}_{X/S})\otimes_{\mab Z}{\mab Q})(i))
=P_{k+2i}R^{d+i}f_{X/S*}({\cal O}_{X/S})\otimes_{\mab Z}{\mab Q}.$$ 
\end{conj}
In \cite{np} and \cite{nb} we have proved that (\ref{conj:lhilc}) (1) is true 
if there exists an exact closed point $s$ of each connected component of $S$ 
such that $X_s$ is the log special fiber of 
a projective strict semistable family over a complete discrete valuation ring 
of any characteristic. 
However we have proved nothing about (\ref{conj:lhilc}) (2) in [loc.~cit.]. 


\par 
In this book we prove the following:

\begin{theo}\label{ltheo:12}
If $\os{\circ}{S}$ is a $p$-adic formal ${\cal V}$-scheme, 
then {\rm (\ref{conj:lhilc}) (1)} implies {\rm (\ref{conj:lhilc}) (2)}. 
\end{theo}

\par 
Though the $p$-adic monodromy-weight conjecture 
in mixed characteristics in \cite{msemi} 
has not been solved yet in general
(the conjecture is solved in various special cases), 
we can prove (\ref{conj:lhilc}) (2) for the special fiber of a projective semistable family 
in mixed characteristics by combining (\ref{ltheo:12}) and  the result for (\ref{conj:lhilc}) (1):

\begin{theo}\label{theo:a}
If there exists a point $s$ of each connected component of $S$
such that $X_s$ is the log special fiber of 
a projective strict semistable family over a complete discrete valuation ring, 
then the conjecture {\rm (\ref{conj:lhilc}) (2)} is true. 
\end{theo}

\par 
We obtain the following as a corollary of (\ref{theo:a}). 

\begin{coro}[{\bf Lefschetz decomposition with weight filtrations}]\label{coro:inp}
Let the assumptions be as in {\rm (\ref{theo:a})}. 
Let $k$ and  $i$ be a nonnegative integers. 
Set
\begin{align*} 
(P_kR^{d-i}f_{X/S*}
({\cal O}_{X/S})\otimes_{\mab Z}{\mab Q})_0 
:={\rm Ker}(\eta^{i+1} \col &
P_kR^{d-i}f_{X/S*}
({\cal O}_{X/S})\otimes_{\mab Z}{\mab Q} \tag{1.11.1}\label{ali:fclpl} \\
&\lo P_k\{R^{d+i+2}f_{X/S*}({\cal O}_{X/S})
\otimes_{\mab Z}{\mab Q})(i+1)). 
\end{align*}
Moreover, assume that $\os{\circ}{S}$ is connected. 
Let $b_k^q$ be the rank of locally free sheaf 
$P_kR^qf_{X/S*}
({\cal O}_{X/S})\otimes_{\mab Z}{\mab Q}$ on $\os{\circ}{S}$. 
Then 
\begin{align*}
P_kR^qf_{X/S*}
({\cal O}_{X/S})\otimes_{\mab Z}{\mab Q}
=&(P_kR^qf_{X/S*}
({\cal O}_{X/S})\otimes_{\mab Z}{\mab Q})_0  
\oplus 
\eta (P_{k-2}R^{q-2}f_{X/S*}
({\cal O}_{X/S})\otimes_{\mab Z}{\mab Q})_0 \\
&\oplus \eta (P_{k-4}R^{q-2}f_{X/S*}
({\cal O}_{X/S})\otimes_{\mab Z}{\mab Q})_0 
\oplus \cdots.   
\end{align*} 
Consequently 
$b_k^0\leq b_{k+2}^2\leq b_{k+4}^4 \leq \cdots $ 
up to $b_*^d$ and 
$b_k^1\leq b_{k+2}^3\leq b_{k+4}^5 \leq \cdots$ 
up to $b_{*}^d$. 
In particular, 
$b_{\infty}^0\leq b_{\inf}^2\leq b_{\inf}^4 \leq \cdots $ 
up to $b_{\inf}^d$ and 
$b_{\inf}^1\leq b_{\inf}^3\leq b_{\inf}^5 \leq \cdots$ 
up to $b_{\inf}^d$. 
\end{coro}

\par 
Let the notations be as in (\ref{theo:intro}). 
Set ${\cal K}_S:={\cal O}_S\otimes_{\cal V}K$. 
For an ${\cal O}_S$-module ${\cal M}$, denote 
${\cal M}\otimes_{\cal V}K$ by ${\cal M}_K$. 
Assume that any smooth component $\os{\circ}{X}_{\lam}$ 
is projective over $\os{\circ}{S}$. 
Assume also that $\os{\circ}{X}$ is of relative pure dimension $d$ over $\os{\circ}{S}$.  
Then, by using the trace morphism constructed in \cite{od}, 
we define the following trace morphism
\begin{align*} 
{\rm Tr}_{X/S}\col R^{2d}f_{X/S*}({\cal O}_{X/S})_K \lo {\cal K}_S(-d).
\tag{1.11.2}\label{ali:ptrka}
\end{align*}
Our construction of the trace morphism (\ref{ali:ptrka}) is different from 
the log crystalline analogue of Fujisawa's trace morphism and 
it is simpler than Fujisawa's trace morphism in the case where the base field 
is the complex number field (\cite{fup}).

\begin{theo}\label{coro:intrgo}
Let the notations be as in {\rm (\ref{theo:intro})}. 
Assume that $\os{\circ}{X}_{\lam}$ for any $\lam$ is projective over $\os{\circ}{S}$ 
and that $\os{\circ}{X}$ is of pure dimension $d$ over $\os{\circ}{S}$.  
Then the following composite morphism 
\begin{align*} 
\langle ~.~\rangle:={\rm Tr}_{X/S} \circ (~\cup~)  \col R^qf_{X/S*}({\cal O}_{X/S})_K
\otimes_{{\cal K}_S}R^{2d-q}f_{X/S*}({\cal O}_{X/S})_K
\lo {\cal K}_S(-d) 
\tag{1.12.1}\label{ali:pkq}
\end{align*} 
is a perfect pairing and it is 
strictly compatible with the weight filtration, that is, 
the following isomorphism induced by the pairing $\langle~ , ~\rangle $
\begin{align*} 
R^qf_{X/S*}({\cal O}_{X/S})_K\lo {\cal H}{\it om}_{{\cal K}_S}
(R^{2d-q}f_{X/S*}({\cal O}_{X/S})_K, {\cal K}_S(-d))
\tag{1.12.2}\label{ali:pkkq}
\end{align*} 
is an isomorphism of filtered  ${\cal K}_S$-modules with respect to $P$. 
\end{theo}

\par 
As is well-known, in the complex analytic case, the (log) polarization on 
the (mixed) Hodge structure is obviously an indispensable notion for the moduli 
problem of polarized (mixed) Hodge structures 
(\cite{gr}, \cite{us}, \cite{ku}). 
To define the (log) polarization, we obviously need the (positive definite) 
bilinear forms for cohomologies. 
Especially we need the product structure for cohomologies 
which behaves well with respect to the weight filtration, 
the monodromy operator and the cohomology of an ample line bundle. 
Using (\ref{coro:intrgo}) and (\ref{theo:a}), 
we obtain the following bilinear form on the primitive part of 
the log crystalline cohomology sheaf of $X/(S,{\cal I},\gam)$
for $q\leq d$ if $\os{\circ}{X}$ is projective over $\os{\circ}{S}$ 
as in the classical case: 
\begin{align*}  
&(P_kR^{d-q}f_{X/S*}({\cal O}_{X/S})_K)_0  
\otimes_{{\cal K}_S}(P_kR^{d-q}f_{X/S*}
({\cal O}_{X/S})_K)_0  \\
&\os{{\rm id}\otimes \eta^{q}}{\lo} 
(P_kR^{d-q}f_{X/S*}
({\cal O}_{X/S})_K)_0  \otimes_{{\cal K}_S}
R^{d+q}f_{X/S*}({\cal O}_{X/S})_K(q)
\os{\cup}{\lo} {\cal K}_S(-(d-q)). 
\end{align*} 

\par  
Lastly in this introduction, we give the dga version of $(H_{\rm zar}(X/S),P)$. 
\par  
Let the notations be as in (\ref{theo:dpfe}). 
Set ${\cal K}_S:={\cal O}_S\otimes_{\mab Z}{\mab Q}$. 
Let ${\rm A}^{\geq 0}{\rm F}(f^{-1}({\cal K}_{S}))$ 
be the category of 
filtered positively-graded graded commutative dga's over $f^{-1}({\cal K}_{S})$ 
and let ${\rm Ho}({\rm A}^{\geq 0}{\rm F}(f^{-1}({\cal K}_{S})))$ 
be the localized category of  
${\rm A}^{\geq 0}{\rm F}(f^{-1}({\cal K}_{S}))$  
inverting the weakly equivalent filtered isomorphisms 
in ${\rm A}^{\geq 0}{\rm F}(f^{-1}({\cal K}_{S}))$ 
(cf.~\cite{nav}, \cite{gelma}).   
Let 
\begin{align*} 
\Phi \col {\rm Ho}({\rm A}^{\geq 0}{\rm F}(f^{-1}({\cal K}_{S})))
\lo {\rm D}^+{\rm F}(f^{-1}({\cal K}_{S}))
\end{align*} 
be an obvious forgetful functor.   
In the text we prove the following: 
\begin{theo}\label{theo:etw}
There exists an object
$$(H_{{\rm zar},{\rm TW}}(X/S),P)\in 
{\rm Ho}({\rm A}^{\geq 0}{\rm F}(f^{-1}({\cal K}_{S})))$$ 
such that 
$$\Phi((H_{{\rm zar},{\rm TW}}(X/S),P))
=(H_{\rm zar}(X/S),P)\otimes^L_{\mab Z}{\mab Q}.$$   
\end{theo}   
We call  $(H_{{\rm zar},{\rm TW}}(X/S),P)$ 
the {\it Hirsch weight-filtered log crystalline dga} of $X/(S,{\cal I},\gam)$. 
Because this book is long enough as a book, 
we do not give an application of 
$(H_{{\rm zar},{\rm TW}}(X/S),P)$ for 
the rational homotopy theory of a proper SNCL scheme 
in characteristic $p>0$ in this book unlike \cite{kiha}.

\par
\bigskip
\parno
{\bf Acknowledgment.}   
I would like to express my sincere thanks to 
T.~Fujisawa for informing me of results in his article \cite{fup} when it was a preprint.  
I am deeply grateful to the referees for giving me an advice on the organization of the introduction 
and for pointing a lot of mistypes, several unclear points in the previous version of this book 
and a quite serious gap of the proof of (\ref{theo:saih}) in it. 
Thanks to their valuable suggestions, I can improve my book.

\bigskip
\parno
{\bf Notations.} 
(1) For a log (formal) scheme $X$ in the sense of 
Fontaine-Illusie-Kato (\cite{klog1}, \cite{klog2}), 
we denote by $\os{\circ}{X}$ (resp.~$M_X=(M_X,\al_X)$)
the underlying (formal) scheme (resp.~the log structure) of $X$. 
In this book 
we consider the log structure on the Zariski site 
on $\os{\circ}{X}$.  
We denote $M_X/{\cal O}_X^*$ by $\ol{M}_X$. 
For a closed point $\os{\circ}{x}$ of $\os{\circ}{X}$, 
we endow $\os{\circ}{x}$ with the inverse image of 
the log structure of $M_X$ and we denote it by 
$x$ and call it an exact closed point of $X$. 
More generally, for a quasi-coherent ideal 
${\cal J}$ of ${\cal O}_X$, we denote 
by $X~{\rm mod}~{\cal J}$ the exact closed log subscheme of $X$ 
defined by ${\cal J}$. 
For a morphism $f\col X\lo Y$ of log schemes, 
$\os{\circ}{f}$ denotes the underlying morphism 
$\os{\circ}{X} \lo \os{\circ}{Y}$ of $f$. 
\par
(2) For a (formal) scheme $T$ 
and a commutative monoid $P$ with 
unit element $e$, we denote by 
$(T,P\oplus {\cal O}_T^*)$ a log (formal) scheme 
whose underlying (formal) scheme is $T$ and whose log structure is the 
association of a morphism 
$P \owns x \lom 0 \in {\cal O}_T$ $(x\not= e)$ 
with natural morphism $P\oplus {\cal O}_T^*\lo {\cal O}_T$ 
of sheaves of monoids in $X_{\rm zar}$. 
\par
(3) For a commutative monoid $P$ with unit element and 
for a commutative ring $A$ with unit element, 
${\rm Spec}^{\log}(A[P])$ is, by definition, the log scheme 
whose underlying scheme is ${\rm Spec}(A[P])$ 
and whose log structure is the association of the natural 
inclusion $P \os{\sus}{\lo} A[P]$. If $A$ has 
an $I$-adic topology ($I$ is an ideal of $A$), then 
$A\{P\}$ denotes $\vpl_n(A/I^n[P])$ and 
${\rm Spf}^{\log}(A\{P\})$ denotes the log scheme 
whose underlying scheme is ${\rm Spf}(\vpl_{n}(A/I^n[P]))$ 
and whose log structure is the association of the natural 
inclusion $P \os{\sus}{\lo} A\{P\}$.
\par
(4) 
For a morphism $Y \lo T$ of log 
schemes, we denote by 
${\Om}^i_{Y/T}$ ($=\om^i_{Y/T}$ in 
\cite{klog1}) $(i\in {\mab N})$ 
the sheaf of  relative logarithmic differential 
forms on $Y/T$
of degree $i$. 
For a log point $s$ whose underlying scheme is 
the spectrum of a perfect field of characteristic $p>0$ and 
for a morphism $Y\lo s$ of fine log schemes, 
${\cal W}_n{\Om}^i_Y(=W_n\om^i_{Y}$ in 
\cite{hk} and \cite{msemi}) $(n\in {\mab Z}_{\geq 1}, i\in {\mab N})$ 
denotes the log de Rham-Witt sheaf of logarithmic differential 
forms on $Y/s$ of degree $i$.
\par 
(5) (S)NCL=(simple) normal crossing log, 
(S)NCD=(simple) normal crossing divisor.
\par 
(6) For a module $M$ over a commutative ring $A$ 
with unit element and for a commutative $A$-algebra 
$B$ with unit element, $M_B$ denotes the tensor product 
$M\us{A}{\otimes}B$. 
\par 
(7) For a log scheme $T$ and 
a quasi-coherent ideal sheaf ${\cal J}$ of 
${\cal O}_T$, we denote by 
$\ul{\rm Spec}^{\log}_T({\cal O}_T/{\cal J})$ 
(resp.~
$\ul{\rm Spec}^{\log}_T({\cal O}_T/{\cal J})_{\rm red}$) 
the log scheme whose underlying scheme is 
$\ul{\rm Spec}_{\os{\circ}{T}}({\cal O}_T/{\cal J})$ 
(resp.~
$\ul{\rm Spec}_{\os{\circ}{T}}({\cal O}_T/{\cal J})_{\rm red})$
and whose log structure is the inverse image of 
the log structure of $T$. 
As already stated in (1), we denote 
$\ul{\rm Spec}^{\log}_T({\cal O}_T/{\cal J})$ 
by $T~{\rm mod}~{\cal J}$, too. 
\par 
(8) By following \cite{od}, 
for a fine log formal scheme $S$ with $p$-adic 
topology, $S_1$  and $S_0$ 
denote the log schemes whose underlying schemes are 
$\ul{\rm Spec}^{\log}_S({\cal O}_S/p{\cal O}_S)$ 
and $\ul{\rm Spec}^{\log}_S({\cal O}_S/p{\cal O}_S)_{\rm red}$, 
respectively, and whose log structures are the inverse image of the log structure of $S$.
\par 
(9) Let $(T,{\cal J},\del)$ be a fine log PD-scheme with quasi-coherent PD-ideal sheaf 
${\cal J}$ of ${\cal O}_T$ and divided power structure $\del$. 
Let $g\col Y\lo (T~{\rm mod}~{\cal J})$ be 
a morphism of fine log schemes. 
In the previous articles \cite{nh2} and \cite{nh3}, 
following \cite{klog1}, 
we have denoted the log crystalline site 
(resp.~the log crystalline topos) of $Y/(T,{\cal J},\del)$
by $(Y/T)^{\log}_{\rm crys}$ 
(resp.~$(\wt{Y/T})^{\log}_{\rm crys}$) 
of a fine log scheme $Y$
over a fine log PD-scheme $(T,{\cal J},\del)$. 
In this book, following \cite{bb} and \cite{bob}, 
we denote the log crystalline site 
(resp.~the log crystalline topos) 
of $Y/(T,{\cal J},\del)$ by ${\rm Crys}(Y/T)$ 
(resp.~$(Y/T)_{\rm crys}$). 
For a scheme $Z$, we denote the zariski site on 
$Z$ by ${\rm Zar}(Z)$ and the associated topos to 
${\rm Zar}(Z)$ by $Z_{\rm zar}$. 
\par 
(10) Let $(T,{\cal J},\del)$ be as in (9). 
In this remark we denote $T~{\rm mod}~{\cal J}$ by $T_0$. 
For a log scheme $Y$ over $T_0$, 
we denote by 
$\eps_{Y/T_0}\col Y\lo \os{\circ}{Y}$ 
(resp.~$\eps_{Y/\os{\circ}{T}_0}\col Y\lo \os{\circ}{Y}$) 
the morphism of log schemes over $T_0\lo \os{\circ}{T}_0$ 
(resp~$\os{\circ}{T}_0$) forgetting the log structure of $Y$. 
We also denote by
$\eps_{Y/T}\col Y\lo \os{\circ}{Y}$ 
(resp.~$\eps_{Y/\os{\circ}{T}_0}\col Y\lo \os{\circ}{Y}$) 
the morphism of log schemes over $T\lo \os{\circ}{T}_0$ 
(resp.~$\os{\circ}{T}_0$) forgetting the log structure of $Y$. 
\par 
(11) For a ringed topos $({\cal T},{\cal A})$, 
let $C^+({\cal A})$ 
(resp.~$D^+({\cal A})$) be 
the category of bounded below 
complexes of ${\cal A}$-modules 
(resp.~the derived category of bounded below 
complexes of ${\cal A}$-modules). 
Let 
${\rm C}^+{\rm F}({\cal A})$ 
(resp.~${\rm D}^+{\rm F}({\cal A})$) 
be the category of 
bounded below filtered complexes of ${\cal A}$-modules 
(resp.~the derived category of 
bounded below filtered complexes of 
${\cal A}$-modules).  
\par 
(12) For a short exact sequence 
$$0\lo E^{\bul} \os{f}{\lo} F^{\bul} \os{g}{\lo} G^{\bul} \lo 0$$
of bounded below complexes of objects of ${\rm C}^+({\cal A})$, 
let ${\rm MC}(f):=E^{\bul}[1]\oplus F^{\bul}$ be the mapping cone of $f$.
We fix an isomorphism 
``$E^{\bul}[1]\oplus F^{\bul}\owns (x,y) \lom g(y)\in G^{\bul}$'' 
in the derived category $D^+({\cal A})$.
\par
Let ${\rm MF}(g):=F^{\bul}\oplus G^{\bul}[-1]$ be 
the mapping fiber of $g$. 
We fix an isomorphism 
``$E^{\bul}\owns x \lom (f(x),0)\in F^{\bul}\oplus G^{\bul}[-1]$'' 
in the derived category $D^+({\cal A})$.
\par
(13) For a morphism $f\col E^{\bul}\lo F^{\bul}$ 
of complexes, we identify 
${\rm MF}(f)=E^{\bul}\oplus F^{\bul}[-1]$ and 
${\rm MC}(f)[-1]=E^{\bul}\oplus F^{\bul}[-1]$ 
with the following isomorphism 
``${\rm MF}(f) \owns (x,y)\lom (-x,y)\in   
{\rm MC}(f)[-1]$''.   
\par
(14) For a semi-cosimplicial complex $A^{\bul \bul}$ in an abelian category 
(the first $\bul$ is the semi-cosimplicial degree and the second 
$\bul$ is the complex degree), 
the boundary morphisms of $A^{\bul \bul}$ are the following: 
\begin{equation*} 
\begin{CD}
\cdots @>>> \cdots @>>> \cdots @>>> \cdots\\
@A{d}AA @A{-d}AA @A{d}AA \\
A^{0,m+1}@>>> A^{1,m+1} @>>> A^{1,m+2}@>>> \cdots \\
@A{d}AA @A{-d}AA @A{d}AA \\
A^{0m}@>>>A^{1m} @>>> A^{1,m+2}@>>> \cdots \\
@A{d}AA @A{-d}AA @A{d}AA\\
\cdots @>>> \cdots @>>> \cdots @>>> \cdots, \\
\end{CD}
\end{equation*}
where the horizontal morphisms 
are standard \v{C}ech morphisms. 
\par 
(15) For a ringed topos $({\cal T},{\cal A})$ and 
for two complexes $(A^{\bul},d^{\bul})$ and $(B^{\bul},d^{\bul})$ 
of ${\cal A}$-modules, the boundary morphism $d=\{d^n\}_{n\in {\mab Z}}$ of 
$A^{\bul}\otimes_{\cal A}B^{\bul}$ is defined as usual: 
$$d^n=\sum_{p+q=n}(d^p\otimes {\rm id}+(-1)^p{\rm id}\otimes d^q).$$
\bigskip
\par\noindent
{\bf Conventions.}
We omit the second ``log'' in the word 
``log smooth log scheme''. 
Following \cite{oc}, 
we say that a morphism $X \lo Y$ of log schemes is solid if the log 
structure of $X$ is the inverse image of that of $Y$.

\section{SNCL schemes}\label{sec:snclv} 
In this section we recall the definition of 
a 
(formal) SNCL(=simple normal crossing log) scheme 
in \cite{nb}.  
\par  
Let $S$ be a log (formal) scheme whose log structure 
is ${\mab N}\oplus {\cal O}_S^*$ with an augmentation 
${\mab N}\owns x \lom 0 \in {\cal O}_S$ $(x\not=0)$.   
We have called this log structure 
the {\it free hollow log structure of rank} 1 
on $\os{\circ}{S}$. This is a special case of 
the hollow constant log structure defined in 
\cite[Definition 4]{oc}.  
\par  
More generally, we consider a fine log (formal) scheme $S$ 
such that there exists an open covering 
$S=\bigcup_{i\in I}S_i$ such that $M_{S_i}$ 
is the free hollow log structure of rank 1. 
In \cite{nb} we have called $S$ {\it a $($formal$)$ family of log points}. 
When $(\os{\circ}{S},{\cal I},\gam)$ is 
a (formal) PD-scheme with quasi-coherent PD-ideal sheaf 
and PD-structure, 
we call $(S,{\cal I},\gam)$ a {\it $($formal$)$ PD-family of log points}. 
\par 
Let $S=\bigcup_{i\in I}S_i$ be an open covering 
of $S$ such that 
$M_{S_i}/{\cal O}^*_{S_i}\simeq {\mab N}$.   
Take a local section
$t_i\in \Gam(S_i,M_S)$ $(i\in I)$ such that 
the image of $t_i$ in $\Gam(S_i,M_S/{\cal O}^*_S)$ is a generator.   
Set $S_{ij}:=S_i\cap S_j$. 
Then there exists a unique section 
$u_{ji}\in \Gam(S_{ij},{\cal O}^*_S)$ such that 
$t_j\vert_{S_{ij}}=u_{ji}(t_i \vert_{S_{ij}})$ 
in $\Gam(S_{ij},M_S)$. 
First consider the case 
where $\os{\circ}{S}$ is a scheme (not a formal scheme). 
Consider the scheme 
${\mab A}^1_{\os{\circ}{S}_i}=
\ul{\rm Spec}_{\os{\circ}{S}_i}({\cal O}_{S_i}[\tau_i])$  
and the log scheme 
$({\mab A}^1_{\os{\circ}{S}_i}, 
({\mab N}\owns 1 \lom \tau_i\in {\cal O}_{S_i}[\tau_i])^a)$. 
Denote this log scheme by $\ol{S}_i$. 
Then, by patching $\ol{S}_i$ and $\ol{S}_j$ along 
$\ol{S}_{ij}:=\ol{S}_i\cap \ol{S}_j$ by the equation 
$\tau_j\vert_{\ol{S}_{ij}}=u_{ji}\tau_i\vert_{\ol{S}_{ij}}$,
we have a log scheme $\ol{S}=\bigcup_{i\in I}\ol{S}_i$. 
The ideal sheaves $\tau_i{\cal O}_{\ol{S}_i}$'s ($i\in I$) 
patch together and we denote by 
${\cal I}_{\ol{S}}$ 
the resulting ideal sheaf of 
${\cal O}_{\ol{S}}$. 
The isomorphism class of the log scheme $\ol{S}$ 
is independent of the choice of the system of 
generators $\tau_i$'s. 
We see that the isomorphism class
of the log scheme $\ol{S}$ and the ideal sheaf 
${\cal I}_{\ol{S}}$ 
are also independent of 
the choice of the open covering $S=\bigcup_{i\in I}S_i$. 
The natural morphism $\ol{S} \lo \os{\circ}{S}$ is log smooth 
by the criterion of the log smoothness (\cite[(3.5)]{klog1}). 
For a log scheme $Y$ over $\ol{S}$, 
we denote 
${\cal I}_{\ol{S}}\otimes_{{\cal O}_{\ol{S}}}{\cal O}_Y$
by ${\cal I}_Y$ by abuse of notation. 
If $\os{\circ}{Y}$ is flat over $\os{\circ}{\ol{S}}$, 
then ${\cal I}_Y$ can be considered as 
an ideal sheaf of ${\cal O}_Y$. 
When $\os{\circ}{S}$ is a formal scheme with ideal sheaf of definition 
${\cal J}$, we obtain the analogue of 
the object $\ol{S}$ above, which we denote by $\ol{S}$ again. 
\par 
By killing ${\cal I}_{\ol{S}}$, 
we have a natural exact closed immersion 
$S\os{\sus}{\lo} \ol{S}$ over $\os{\circ}{S}$.  
 
In \cite[(1.1.5), (1.1.6), (1.1.7)]{nb} we have proved 
the following two easy propositions: 

\begin{prop}[{\bf A special case of \cite[(1.1.5)]{nb}}]\label{prop:bar}
Let $u\col S\lo S'$ be a morphism of families of log points. 
Then $u$ induces a morphism 
$\ol{u}\col \ol{S}\lo \ol{S}{}'$ 
fitting into the following commutative diagram 
\begin{equation*} 
\begin{CD}
S@>{u}>> S'\\ 
@V{\bigcap}VV @VV{\bigcap}V \\
\ol{S} @>{\ol{u}}>> \ol{S}{}'. 
\end{CD}
\end{equation*} 
\end{prop}  

 
Let $B$ be a scheme. 
For two nonnegative integers $a$ and $d$ such that 
$a\leq d$,  
consider the following scheme 
\begin{equation*} 
\os{\circ}{\mab A}_B(a,d):=
\ul{{\rm Spec}}_B
({\cal O}_B[x_0, \ldots, x_d]/(\prod_{i=0}^ax_i)). 
\end{equation*}

\begin{defi}[{\bf \cite[(1.1.9)]{nb}}]\label{defi:zbsnc}
Let $Z$ be a scheme over $B$ 
with structural morphism $g \col Z \lo B$. 
We call $Z$ an 
{\it SNC$($=simple normal crossing$)$ scheme} over 
$B$ if $Z$ is a union of smooth schemes 
$\{Z_{\lam}\}_{\lam \in \Lam}$ over $B$ ($\Lam$ is a set) 
and if, for any point of $z \in Z$, 
there exist an open neighborhood $V$ of 
$z$ and an open neighborhood $W$ of 
$g(z)$ such that  
there exists an \'{e}tale morphism  
$V \lo \os{\circ}{\mab A}_W(a,d)$ 
such that 
\begin{equation*} 
\{Z_{\lam}\vert_{V}\}_{\lam \in \Lam} 
=\{\{x_i=0\}\}_{i=0}^a, 
\tag{2.2.1}\label{eqn:defilsnc}
\end{equation*}   
where $a$ and $d$ are nonnegative integers 
such that $a\leq d$,  
which depend on 
zariskian local neighborhoods in $Z$.  
Here $\{Z_{\lam}\vert_{V}\}_{\lam \in \Lam}$ means the set of 
$Z_{\lam}\vert_{V}$'s such that $Z_{\lam}\vert_{V}\not=\emptyset$ by abuse of notation. 
We call the set $\{Z_{\lam}\}_{\lam \in \Lam}$ 
a {\it decomposition of $Z$ by smooth components of} 
$Z$ over $B$. 
We call $Z_{\lam}$ a 
{\it smooth component} of $Z$ over $B$. 
\end{defi} 

\parno  
Set $\Del:=\{Z_{\lam}\}_{\lam \in \Lam}$.  
For an open subscheme $V$ of $Z$, 
set $\Del_{V}:=\{Z_{\lam}\vert_{V}\}_{\lam \in \Lam}$.



\par 
In \cite{nb} 
we have proved the following proposition:

\begin{prop}[{\bf \cite[(1.1.11)]{nb}}]\label{prop:cc}
Let $\Del$ and  $\Del'$ be decompositions of $Z$  
by smooth components of $Z$. 
Then, for any point $z \in Z$, 
there exists an open neighborhood $V$ of 
$z$ in $Z$ such that $\Del_V = \Del'_V$. 
\end{prop}

\par 
Let $P(\Lam)$ be a category whose objects are 
nonempty subsets of $\Lam$ and whose morphisms 
are inclusions of subsets of $\Lam$.
In \cite[(2.4)]{fup} 
this category has been denoted by ${\cal S}^+(\Lam)$. 
For a nonnegative integer $m$ and an object 
$\ul{\lam}=\{\lam_0,\cdots, \lam_m\}\in P(\Lam)$ 
$(\lam_i \not= \lam_j~{\rm if}~i\not= j, \lam_i\in \Lam)$,  
set 
\begin{equation}
Z_{\ul{\lam}} 
:=Z_{\lam_0}\cap \cdots \cap Z_{\lam_m}.  
\tag{2.3.1}\label{eqn:parlm}
\end{equation}
Set 
\begin{equation}
Z^{(m)} := \us{\# \ul{\lam}=m+1}{\coprod}Z_{\ul{\lam}}  
\tag{2.3.2}\label{eqn:kfntd}
\end{equation} 
for $m\in {\mab N}$ and 
$Z_{\phi}=Z$ ($\phi$ is the empty set) and $Z^{(-1)}=Z$. 
We also set $Z^{(m)}=\emptyset$ for $m\leq -2$.  


In \cite{nb} we have proved the following 
by using (\ref{prop:cc}):

\begin{prop}[{\bf \cite[(1.1.12)]{nb}}]\label{prop:zki}
The scheme $Z^{(m)}$ 
is independent of the choice of $\Del$.  
\end{prop}

We have the natural morphism 
$b^{(m)}\col Z^{(m)}\lo Z$.

As in \cite[(3.1.4)]{dh2} and \cite[(2.2.18)]{nh2}, 
we have an orientation sheaf $\vp^{(m)}_{\rm zar}(Z/B)$ 
$(m\in {\mab N})$ in $Z^{(m)}_{\rm zar}$ 
associated to the set $\Del$. 
If $B$ is a closed subscheme of $B'$ defined by 
a quasi-coherent nil-ideal sheaf ${\cal J}$ 
which has a PD-structure $\del$, 
then $\vp^{(m)}_{\rm zar}(Z/B)$ 
extends to an abelian sheaf 
$\vp^{(m)}_{\rm crys}(Z/B')$ in 
$(Z^{(m)}/(B',{\cal J},\del))_{\rm crys}$. 
\par 
Assume that $S$ is a scheme until (\ref{defi:ddef}) 
and that $M_S$ is the free hollow log structure of rank 1 
for the time being. 
We fix an isomorphism 
\begin{align*} 
(M_S,\al_S)\simeq ({\mab N}\oplus {\cal O}_S^*\lo {\cal O}_S) 
\end{align*} 
globally on $S$.  
Let $M_S(a,d)$ be the log structure on 
${\mab A}_{\os{\circ}{S}}(a,d)$ associated to the following morphism 
\begin{equation*} 
{\mab N}^{{\oplus}(a+1)}\owns 
(0, \ldots,0,\os{i}{1},0,\ldots, 0)\lom x_{i-1}\in 
{\cal O}_S[x_0, \ldots, x_d]/(\prod_{i=0}^ax_i).
\tag{2.4.1}  
\end{equation*} 
Let ${\mab A}_S(a,d)$ be the resulting log scheme over $S$.
The diagonal morphism ${\mab N} \lo {\mab N}^{{\oplus}(a+1)}$ 
induces a morphism ${\mab A}_S(a,d) \lo S$ of log schemes.

\begin{defi}\label{defi:lfac}  
Let $S$ be a family of log points 
(we do not assume that $M_S$ is free).  
Let $f \col X(=(\os{\circ}{X},M_X)) \lo S$ 
be a morphism of log schemes 
such that $\os{\circ}{X}$ is 
an SNC scheme over $\os{\circ}{S}$ 
with a decomposition 
$\Del:=\{\os{\circ}{X}_{\lam}\}_{\lam \in \Lam}$ 
of $\os{\circ}{X}/\os{\circ}{S}$ by its smooth components. 
We call $f$ (or $X/S$) an 
{\it SNCL$($=simple normal crossing log$)$ scheme} if, 
for any point of $x \in \os{\circ}{X}$, 
there exist an open neighborhood $\os{\circ}{V}$ of 
$x$ and an open neighborhood $\os{\circ}{W}$ of 
$\os{\circ}{f}(x)$ such that $M_W$ is 
the free hollow log structure of rank $1$ 
and such that  $f\vert_V$ factors through 
a solid and \'{e}tale morphism 
$V {\lo} {\mab A}_W(a,d)$  
such that $\Del_{\os{\circ}{V}}
=\{x_i=0\}_{i=0}^a$ in $\os{\circ}{V}$.  
(Similarly we can give the definition of 
a formal SNCL scheme over $S$ 
in the case where $\os{\circ}{S}$ is a formal scheme.) 
\end{defi}

\par 
Let $X$ be a (formal) SNCL scheme over $S$ 
with a decomposition 
$\Del:=\{\os{\circ}{X}_{\lam}\}_{\lam \in \Lam}$ 
of $\os{\circ}{X}/\os{\circ}{S}$ 
by its smooth components. 
For an element $\lam$ of $\Lam$, 
let $a_{\lam} \col \os{\circ}{X}_{\lam} 
\os{\sus}{\lo} \os{\circ}{X}$  
be the natural closed immersion. 
For an element  
$\ul{\lam}=\{\lam_0,\cdots, \lam_m\}$ 
$(\lam_i \not= \lam_j~{\rm if}~i\not= j, \lam_i\in \Lam)$ of $P(\Lam)$,  
endow $\os{\circ}{X}_{\ul{\lam}}$ with the inverse image of 
the log structure of $X$ by  
the natural morphism $\os{\circ}{X}_{\ul{\lam}}\lo \os{\circ}{X}$ 
and let $X_{\ul{\lam}}$ be the resulting log scheme. 
We set $X_{\emptyset}:=X$ for convenience of notation. 
For an integer $m\geq -1$, 
let $X^{(m)}$ be the log scheme 
whose underlying scheme 
is $\os{\circ}{X}{}^{(m)}$ and 
whose log structure is the inverse image of the log structure of $X$ 
by the natural morphism $\os{\circ}{X}{}^{(m)}\lo \os{\circ}{X}$.
For a negative integer $m\leq -2$, 
set $X^{(m)}=\emptyset$.  
Let $a_{\ul{\lam}}\col X_{\ul{\lam}} \os{\sus}{\lo} X$ 
and $a^{(m)}\col X^{(m)} \lo X$ 
be the natural exact closed immersion and the natural morphism. 
\par 
Because the structural morphism $X\lo S$ is integral, 
$X_T:=X\times_ST$ is also integral for a morphism 
$T\lo S$ of fine log schemes. Note that $(X_T)^{\circ}=
\os{\circ}{X}\times_{\os{\circ}{S}}\os{\circ}{T}$. 
We denote $(X_T)^{\circ}$ by $\os{\circ}{X}_T$.

\par 
Assume that $M_S$ is the free hollow log structure of rank 1. 
Then $\ol{S}=({\mab A}^1_{\os{\circ}{S}}, 
({\mab N}\owns 1 \lom t \in {\cal O}_S[t])^a)$.  
Set 
\begin{equation*} 
\os{\circ}{\mab A}_{\ol{S}}(a,d):=
\ul{\rm Spec}_{\os{\circ}{S}}
({\cal O}_S[x_0, \ldots, x_{d},t]/(x_0\cdots x_a-t)). 
\end{equation*} 
Then we have a natural structural morphism 
$\os{\circ}{{\mab A}}_{\ol{S}}(a,d) \lo \os{\circ}{\ol{S}}$. 
Let $\ol{M}_{\ol{S}}(a,d)$ 
be the log structure associated to 
a morphism 
${\mab N}^{a+1} \owns e_i=
(0, \ldots,0,\os{i}{1},0,\ldots, 0) \lom x_{i-1} \in 
{\cal O}_S[x_0, \ldots, x_{d},t]/(x_0\cdots x_a-t)$. 
Set 
$${\mab A}_{\ol{S}}(a,d)
:=(\ul{\rm Spec}_{\os{\circ}{S}}
({\cal O}_S[x_0, \ldots, x_{d},t]
/(x_0\cdots x_a-t)),\ol{M}_{\ol{S}}(a,d)).$$  
Then we have the following natural morphism 
\begin{equation*} 
{\mab A}_{\ol{S}}(a,d) \lo \ol{S}.  
\tag{2.5.1}
\end{equation*}

\par 
By killing ``$t$'', we have the following natural exact closed immersion 
\begin{equation*} 
{\mab A}_S(a,d) \os{\sus}{\lo} 
{\mab A}_{\ol{S}}(a,d) 
\tag{2.5.2}
\end{equation*}
of fs(=fine and saturated) log schemes 
over $S\os{\sus}{\lo} \ol{S}$ 
if the log structure of $S$ is 
the free hollow log structure of rank 1.  

In \cite{nb} we have proved the following (1) 
as a special case of [loc.~cit., (1.1.6)]
(the following (2) is a special case of (1)):

\begin{lemm}[{\bf A special case of \cite[(1.1.6)]{nb}}]\label{lemm:etl}
Let $S$ be a family of log points. Then the following hold$:$ 
\par 
$(1)$ Let $Y\lo S$ be a log smooth scheme 
which has a global chart ${\mab N}\lo P$. 
Then, Zariski locally on $Y$,  
there exists a log smooth scheme 
$\ol{Y}$ over $\ol{S}$ fitting into 
the following cartesian diagram 
\begin{equation*}
\begin{CD} 
Y @>{\sus}>> \ol{Y} \\
@VVV @VVV \\
S\times_{{\rm Spec}^{\log}({\mab Z}[{\mab N}])}
{\rm Spec}^{\log}({\mab Z}[P])
@>{\sus}>> \ol{S}\times_{{\rm Spec}^{\log}({\mab Z}[{\mab N}])}
{\rm Spec}^{\log}({\mab Z}[P]) \\ 
@VVV @VVV \\
S@>{\sus}>> \ol{S},  
\end{CD} 
\tag{2.6.1}\label{cd:xwtx} 
\end{equation*}
where the vertical morphism 
$\ol{Y} \lo \ol{S}
\times_{{\rm Spec}^{\log}({\mab Z}[{\mab N}])}
{\rm Spec}^{\log}({\mab Z}[P])$ 
is solid and \'{e}tale.  
\par 
$(2)$ Let $S$ be a family of log points 
and let $X$ be an SNCL scheme over $S$. 
Zariski locally on $X$, there exists a log scheme 
$\ol{X}$ over $\ol{S}$ fitting into 
the following cartesian diagram 
\begin{equation*}
\begin{CD} 
X @>{\sus}>> \ol{X} \\
@VVV @VVV \\
{\mab A}_S(a,d) 
@>{\sus}>> {\mab A}_{\ol{S}}(a,d) \\ 
@VVV @VVV \\
S@>{\sus}>> \ol{S},  
\end{CD} 
\tag{2.6.2}\label{cd:xwbtx} 
\end{equation*}
where the vertical morphism 
$\ol{X} \lo {\mab A}_{\ol{S}}(a,d)$ is solid and \'{e}tale.  
\end{lemm} 

\begin{defi}[{\bf A special case of \cite[(1.1.16)]{nb}}]\label{defi:lfeac}  
Let $\ol{f} \col \ol{X}\lo \ol{S}$ 
be a morphism of log schemes (on the Zariski sites). 
Set $X:=\ol{X}\times_{\ol{S}}S$.  
We call $\ol{f}$ (or $\ol{X}/\ol{S}$) 
a {\it strict semistable log scheme} over $\ol{S}$ 
if $\os{\circ}{\ol{X}}$ is a smooth scheme over $\os{\circ}{S}$, 
if $\os{\circ}{X}$ is a relative SNCD on 
$\os{\circ}{\ol{X}}/\os{\circ}{\ol{S}}$ 
(with some decomposition 
$\Del:=\{\os{\circ}{X}_{\lam}\}_{\lam \in \Lam}$ of 
$\os{\circ}{X}$ by smooth components of 
$\os{\circ}{X}$ over $\os{\circ}{S}$) 
and if, for any point of $x \in \os{\circ}{\ol{X}}$, 
there exist an open neighborhood $\os{\circ}{\ol{V}}$ of 
$x$ and an open neighborhood 
$\os{\circ}{\ol{W}}\simeq \ul{\rm Spec}_W({\cal O}_W[t])$ 
(where $W:=\ol{W}\times_{\ol{S}}S$) of 
$\os{\circ}{f}(x)$ such that 
$M_{\ol{W}}\simeq  
({\mab N}\owns 1 \lom t\in {\cal O}_W[t])^a$ and such that  
$\ol{f}\vert_{\ol{V}}$ factors through 
a solid and \'{e}tale morphism 
$\ol{V} {\lo} {\mab A}_{\ol{W}}(a,d)$ 
such that  
$\Del_{\os{\circ}{\ol{V}}}
=\{x_i=0\}_{i=0}^a$ in $\os{\circ}{\ol{V}}$.  
(Similarly we can give the definition of 
a strict semistable log formal scheme over $\ol{S}$.) 
\end{defi}


\par  
Next we recall the following(=the exactification) 
which has been proved in \cite{s3}.

\begin{prop}\label{prop:zpex}{\rm {\bf (\cite[Proposition 2.10]{s3})}}\label{prop:exad}
Let $T$ be a fine log formal ${\mab Z}_p$-scheme.   
Let ${\cal C}^{\rm ex}_{\rm hom}$ $($resp.~${\cal C})$ 
be the category of homeomorphic exact immersions 
$($resp.~the category of immersions$)$ of 
fine log formal ${\mab Z}_p$-schemes over $T$.  
Let $\iota$ be the natural functor 
${\cal C}^{\rm ex}_{\rm hom} \lo {\cal C}$. 
Then $\iota$ has a right adjoint functor 
$(~)^{\rm ex}\col {\cal C}\lo {\cal C}^{\rm ex}_{\rm hom}$. 
\end{prop}
We recall the construction of $(~)^{\rm ex}$ quickly. 
\par 
By the universality, the problem is local. 
We may assume that immersions are closed immersions. 
We may assume that 
an object 
$Y\os{\sus}{\lo} {\cal Q}$ in ${\cal C}$ has 
a global chart $(P\lo Q)$. 
Let $P^{\rm ex}$ be the inverse image of 
$Q$ by the morphism $P^{\rm gp} \lo Q^{\rm gp}$. 
Then we have a fine log formal ${\mab Z}_p$-scheme 
${\cal Q}^{{\rm prex}}:=
{\cal Q}\wh{\times}_{{\rm Spf}^{\log}({\mab Z}_p\{P\})}
{\rm Spf}^{\log}({\mab Z}_p\{P^{\rm ex}\})$ 
over ${\cal Q}$ with an exact immersion 
$Y\os{\sus}{\lo} {\cal Q}^{{\rm prex}}$. 
(``prex'' is the abbreviation of 
``pre-exactification''.) 
Let ${\cal Q}^{\rm ex}$ be the formal completion of 
${\cal Q}^{{\rm prex}}$ along $Y$. 
Since $\os{\circ}{\cal Q}{}^{\rm ex}=\os{\circ}{Y}$ 
and 
$M_{{\cal Q}^{\rm ex},x}/{\cal O}^*_{{\cal Q}^{\rm ex},x}
=M_{Y,x}/{\cal O}^*_{Y,x}$, 
we see that the morphism $Y\os{\sus}{\lo} {\cal Q}^{\rm ex}$ 
is the desired homeomorphic exact immersion over $T$.

\begin{prop}[{\bf Base change of exactifications (\cite[(1.1.27)]{nb})}]\label{prop:xpls}  
Let $T'\lo T$ be a morphism of 
fine log formal schemes. 
Let ${\cal Q}$ be a fine log formal scheme over $T$. 
Let $Y\os{\sus}{\lo} {\cal Q}$ be an immersion  
of fine log formal schemes over $T$. 
Let $Y \os{\sus}{\lo} {\cal Q}^{\rm ex}$ be 
the exactification of $Y \os{\sus}{\lo} {\cal Q}$. 
Then the immersion $Y\times_TT' \os{\sus}{\lo} {\cal Q}^{\rm ex}\times_TT'$ 
is the exactification of the immersion 
$Y\times_TT' \os{\sus}{\lo} {\cal Q}\times_TT'$ over $T'$.  
\end{prop}

We recall the following (\cite[(2.1.5)]{nh2}) describing 
the local structure of an exact closed immersion, 
which will be used in this section and later sections: 

\begin{prop}[{\bf \cite[(2.1.5)]{nh2}}]\label{prop:adla}
Let $T_0 \os{\sus}{\lo} T$ be a closed immersion of fine log schemes. 
Let $Y$ $($resp.~${\cal Q})$ be a log smooth scheme over $T_0$ 
$($resp.~$T)$, which can be considered as a log scheme over $T$.  
Let $\iota \col Y \os{\sus}{\lo} {\cal Q}$ 
be an exact closed immersion over $T$. 
Let $y$ be a point of $\os{\circ}{Y}$ and 
assume that there exists a chart $(Q \lo M_T, P \lo M_Y, Q 
\os{\rho}{\lo} P)$ of $Y \lo T_0 \os{\subset}{\lo} T$ 
on a neighborhood of $y$ such that 
$\rho$ is injective, that ${\rm Coker}(\rho^{\rm gp})$ is torsion free 
and that the natural homomorphism ${\cal O}_{Y,y} \otimes_{{\mab Z}} 
(P^{{\rm gp}}/Q^{{\rm gp}}) \lo \Om^1_{Y/T_0,y}$ is an isomorphism. 
Then, on a neighborhood of $y$, 
there exist a nonnegative integer $c$
and the following cartesian diagram 
\begin{equation}
{\small{\begin{CD}
Y @>>> {\cal Q}' @>>> {\cal Q}\\ 
@VVV  @VVV @VVV\\ 
(T_0\otimes_{{\mab Z}[Q]}{\mab Z}[P],P^a) 
@>{\sus}>> (T\otimes_{{\mab Z}[Q]}{\mab Z}[P],P^a)
@>{\sus}>>  
(T\otimes_{{\mab Z}[Q]}{\mab Z}[P],P^a)
\times_T{\mab A}^{c}_T, 
\end{CD}}}
\tag{2.10.1}\label{eqn:0txda}
\end{equation}
where the vertical morphisms are solid and \'{e}tale and 
the lower second horizontal morphism is the base change of 
the zero section $T \os{\sus}{\lo} {\mab A}^c_T$ 
and ${\cal Q}':={\cal Q}\times_{{\mab A}^c_T}T$. 
\end{prop}

\par 
Let $S_0 \os{\sus}{\lo} S$ 
be a nil-immersion of families of log points. 
Let $X/S_0$ be an SNCL scheme. 
Let 
$X \os{\sus}{\lo} {\cal P}$ be 
an immersion into a log smooth scheme over $S$. 
Let 
$X \os{\sus}{\lo} \ol{\cal P}$ be also 
an immersion into a log smooth scheme over $\ol{S}$. 
(We do not assume that there exists an immersion 
${\cal P}\os{\sus}{\lo} \ol{\cal P}$).

\begin{prop}[{\bf \cite[(1.1.40)]{nb}}]\label{prop:fsi} 
Assume that  
$X\os{\sus}{\lo} {\cal P}$ 
$($resp.~$X\os{\sus}{\lo} \ol{\cal P})$
has a global chart 
$P\lo Q$ $($resp.~$\ol{P}\lo Q)$. 
Let $P^{\rm ex}$ $($resp.~$\ol{P}{}^{\rm ex})$ 
be the inverse image of 
$Q$ by the morphism $P^{\rm gp}\lo Q^{\rm gp}$ 
$($resp.~$\ol{P}{}^{\rm gp}\lo Q^{\rm gp})$. 
Set ${\cal P}^{\rm prex}
:={\cal P}\times_{{\rm Spec}^{\log}({\mab Z}[P])}
{\rm Spec}^{\log}({\mab Z}[P^{\rm ex}])$ 
$($resp.~$\ol{\cal P}{}^{\rm prex}
:=\ol{\cal P}\times_{{\rm Spec}^{\log}({\mab Z}[\ol{P}])}
{\rm Spec}^{\log}({\mab Z}[\ol{P}{}^{\rm ex}]))$. 
Then, locally on $X$, there exists an open neighborhood 
${\cal P}^{\rm prex}{}'$ 
$($resp.~$\ol{\cal P}{}^{\rm prex}{}')$ 
of ${\cal P}^{\rm prex}$ 
$($resp.~$\ol{\cal P}{}^{\rm prex})$ 
fitting into 
the following cartesian diagram 
for some $0 \leq a \leq d \leq d':$  
\begin{equation*}
\begin{CD}
X @>{\subset}>> {\cal P}^{\rm prex}{}'\\ 
@VVV  @VVV \\
{\mab A}_{S_0}(a,d)
@>{\sus}>> {\mab A}_{S}(a,d')
\end{CD}
\tag{2.11.1}\label{eqn:xdplxda}
\end{equation*}
$($resp.~
\begin{equation*}
\begin{CD}
X @>{\subset}>> \ol{\cal P}{}^{\rm prex}{}'\\ 
@VVV  @VVV \\
{\mab A}_{S_0}(a,d)
@>{\sus}>> {\mab A}_{\ol{S}}(a,d')~),
\end{CD}
\tag{2.11.2}\label{eqn:xdpelda}
\end{equation*}
where the vertical morphisms are solid and \'{e}tale.  
\end{prop}

\begin{prop}[{\bf \cite[(1.1.41)]{nb}}]\label{prop:nexeo}
$(1)$ Let ${\cal P}^{\rm ex}$  
be the exactification of 
the immersion $X\os{\sus}{\lo}{\cal P}$.  
Then ${\cal P}^{\rm ex}$   
is a formal SNCL scheme over $S$.  
\par 
$(2)$ 
Let $\ol{\cal P}{}^{\rm ex}$ 
be the exactification of 
the immersion $X\os{\sus}{\lo}\ol{\cal P}$. 
Then $\ol{\cal P}{}^{\rm ex}$ 
is a strict semistable family over $\ol{S};$ 
$\os{\circ}{\ol{\cal P}}{}^{\rm ex}$
is a formally smooth scheme over $\os{\circ}{S}$  
such that 
$\os{\circ}{\cal D}:=\os{\circ}{\ol{\cal P}}{}^{\rm ex}
\times_{\os{\circ}{\ol{S}}}\os{\circ}{S}$ is 
a formal SNCD on 
$\os{\circ}{\ol{\cal P}}{}^{\rm ex}$ and such that 
$M_{{\ol{\cal P}}{}^{\rm ex}}=M(\os{\circ}{\cal D})$, 
where $M(\os{\circ}{\cal D})$ is the associated log structure to 
$\os{\circ}{\cal D}$ on $\os{\circ}{\ol{\cal P}}{}^{\rm ex}/\os{\circ}{S}$ 
$(${\rm \cite[p.~61]{nh2}}$)$. 
\end{prop}

\par 
For a smooth component $\os{\circ}{X}_{\lam}$ of $\os{\circ}{X}$, 
we can define a closed formal subscheme 
$\os{\circ}{\ol{\cal P}}{}^{\rm ex}_{\lam}$ of 
$\os{\circ}{\ol{\cal P}}{}^{\rm ex}$ 
(resp.~$\os{\circ}{\cal P}{}^{\rm ex}_{\lam}$ of 
$\os{\circ}{\cal P}{}^{\rm ex}$) 
such that there exists an immersion  
$\os{\circ}{X}_{\lam} \os{\sus}{\lo} 
\os{\circ}{\ol{\cal P}}{}^{\rm ex}_{\lam}$
(resp.~ $\os{\circ}{X}_{\lam} \os{\sus}{\lo} 
\os{\circ}{\cal P}{}^{\rm ex}_{\lam}$) 
which is an isomorphism as topological spaces. 
In the case where we consider the immersion 
$X\os{\sus}{\lo}\ol{\cal P}$, we set 
$\os{\circ}{\cal P}{}^{\rm ex}_{\lam}:=
\os{\circ}{\ol{\cal P}}{}^{\rm ex}_{\lam}\times_{\os{\circ}{\ol{S}}}\os{\circ}{S}$. 
For an object $\ul{\lam}=\{\lam_0,\cdots, \lam_m\}\in P(\Lam)$ 
$(m\in {\mab Z}_{\geq 0}, \lam_i \in \Lam, \lam_i \not= \lam_j~{\rm if}~i\not= j)$,  
we can define a closed formal subscheme 
$\os{\circ}{\cal P}{}^{\rm ex}_{\ul{\lam}}$ of 
$\os{\circ}{\cal P}{}^{\rm ex}$ 
corresponding to $\os{\circ}{X}_{\ul{\lam}}$.  
Because ${\cal P}^{\rm ex}$ is a formal SNCL scheme over $S$, 
we can define $\os{\circ}{\cal P}{}^{{\rm ex},(m)}$ 
for a nonnegative integer $m$. 
Indeed, 
\begin{equation}
\os{\circ}{\cal P}{}^{{\rm ex},(m)} =  
\us{\# \ul{\lam}=m+1}
{\coprod}\os{\circ}{\cal P}{}^{\rm ex}_{\ul{\lam}}.   
\tag{2.12.1}\label{eqn:kfexntd}
\end{equation} 
Endow $\os{\circ}{\cal P}{}^{\rm ex}_{\ul{\lam}}$ 
with the pull-back of the log structure of ${\cal P}^{\rm ex}$ 
and let ${\cal P}^{\rm ex}_{\ul{\lam}}$ be the resulting log formal scheme. 
Endow $\os{\circ}{\cal P}{}^{{\rm ex},(m)}$ 
with the pull-back of the log structure of ${\cal P}^{\rm ex}$ 
and let ${\cal P}{}^{{\rm ex},(m)}$ be the resulting log formal scheme. 
Then  
\begin{equation}
{\cal P}{}^{{\rm ex},(m)} =  \us{\# \ul{\lam}=m+1}
{\coprod}{\cal P}{}^{\rm ex}_{\ul{\lam}}.   
\tag{2.12.2}\label{eqn:kfxltd}
\end{equation}

\par 
Let $S$ be a family of log points. Let $M_S=(M_S,\al_S)$ be the log structure of $S$. 
In \cite{nb} we have defined a log PD-enlargement 
$((T,{\cal J},\del),z)$ of $S$ as follows (cf.~\cite{oc}): 
$(T,{\cal J},\del)$ is a fine log PD-scheme such that ${\cal J}$ is quasi-coherent and 
$z\col T_0\lo S$ is a morphism of fine log schemes, 
where $T_0:=T \mod {\cal J}$. 
When we are given a morphism $S\lo S'$ of families of log points, 
we can define a morphism of  log PD-enlargements over 
the morphism $S\lo S'$ in an obvious way. 
Endow $\os{\circ}{T}_0$ with the inverse image of the log structure of $S$. 
We denote the resulting log scheme by $S_{\os{\circ}{T}_0}$. 
It is easy to see that the natural morphism $z^*(M_S)\lo M_{T_0}$ is injective
(\cite[(1.1.4)]{nb}); 
we consider $z^*(M_S)$ is the sub log structure of $M_{T_0}$. 
Let $M$ be the sub log structure of the log structure $(M_T,\al_T)$ of $T$ such that 
the natural morphism $M_T\lo M_{T_0}$ induces an isomorphism 
$M/{\cal O}_T^*\os{\sim}{\lo} z^*(M_S)/{\cal O}_{T_0}^*$.  
Let $S(T)$ be the log scheme $(\os{\circ}{T},(M,\al_T\vert_M))$. 
Since $M/{\cal O}_T^*$ is constant, we can consider the hollowing out  
$S(T)^{\nat}$ of $S(T)$ (\cite[Remark 7]{oc}); 
the log scheme $S(T)^{\nat}$ is a family of log points. 
Set $X_{\os{\circ}{T}_0}:=X_{}\times_SS_{\os{\circ}{T}_0}=
X\times_{\os{\circ}{S}}\os{\circ}{T}_0$
(we can consider $X$ as a fine log scheme over $\os{\circ}{S}$).    
By abuse of notation, we denote by the same symbol $f$ 
the structural morphism 
$X_{\os{\circ}{T}_0} \lo S_{\os{\circ}{T}_0}$. 
\par 
Let $X'_{\os{\circ}{T}_0}=\coprod_{i\in I}X_i$ be the disjoint union 
of an affine open covering of $X_{\os{\circ}{T}_0}$ over $S_{\os{\circ}{T}_0}$ 
($X_i$ is a log open subscheme of $X_{\os{\circ}{T}_0}$). 
Assume that $f(\os{\circ}{X}_{i})$ is contained in 
an affine open subscheme of 
$\os{\circ}{T}=(S_{\os{\circ}{T}_0})^{\circ}$ 
such that the restriction of $M_{S_{\os{\circ}{T}_0}}$ 
to this open subscheme is free of rank $1$. 
Assume also that there exists a solid and \'{e}tale morphism 
$X_{i}\lo {\mab A}_{S_{\os{\circ}{T}_0}}(a,d)$. 
Then, replacing $X_{i}$ by a small log open subscheme of 
$X_{\os{\circ}{T}_0}$, 
we can assume that there exists a log smooth scheme 
$\ol{\cal P}{}'_{i}/\ol{S(T)^{\nat}}$ fitting into the following commutative diagram 
\begin{equation*}
\begin{CD} 
X_{i}@>{\subset}>> \ol{\cal P}{}'_{i}\\
@VVV @VVV \\ 
{\mab A}_{S_{\os{\circ}{T}_0}}(a,d) @>{\subset}>>
{\mab A}_{\ol{S(T)^{\nat}}}(a,d) \\
@VVV @VVV \\ 
S_{\os{\circ}{T}_0}@>{\subset}>> \ol{S(T)^{\nat}},   
\end{CD}
\tag{2.12.3}\label{cd:xnip} 
\end{equation*}
where the morphism 
$\ol{\cal P}{}'_{i}\lo {\mab A}_{\ol{S(T)^{\nat}}}(a,d)$ 
is solid and \'{e}tale ((\ref{lemm:etl})). 
Set $\ol{\cal P}{}':=\coprod_{i\in I}\ol{\cal P}{}'_{i}$. 
Set also 
\begin{equation*} 
X_{\os{\circ}{T}_0,n}:={\rm cosk}_0^{X_{\os{\circ}{T}_0}}(X'_{\os{\circ}{T}_0})_n 
\quad (0\leq m \leq N,n\in {\mab N})  
\tag{2.12.4}\label{eqn:iincl}
\end{equation*}   
and 
\begin{equation*} 
\ol{\cal P}_{n}:={\rm cosk}_0^{\ol{S(T)^{\nat}}}(\ol{\cal P}{}')_n\quad (n\in {\mab N}). 
\tag{2.12.5}\label{eqn:iigncl}
\end{equation*}  
Then we have a simplicial SNCL scheme $X_{\os{\circ}{T}_0\bul}$ 
and an immersion 
\begin{equation*}  
X_{\os{\circ}{T}_0\bul} \os{\sus}{\lo} \ol{\cal P}_{\bul} 
\tag{2.12.6}\label{eqn:eixd} 
\end{equation*} 
into a log smooth simplicial log scheme over $\ol{S(T)^{\nat}}$.  
Thus we have obtained the following:

\par 

\begin{prop}\label{prop:xbn}  
The following hold$:$ 
\par 
$(1)$ There exist a \v{C}ech diagram $X_{\os{\circ}{T}_0\bul}$ 
of $X_{\os{\circ}{T}_0}$ and an immersion 
\begin{equation*}  
\begin{CD} 
X_{\os{\circ}{T}_0\bul} 
@>{\sus}>> \ol{\cal P}_{\bul} \\
@VVV @VVV \\
S_{\os{\circ}{T}_0} @>{\subset}>> \ol{S(T)^{\nat}}
\end{CD} 
\tag{2.13.1}\label{eqn:eipxd} 
\end{equation*} 
into a log smooth simplicial log scheme over $\ol{S(T)^{\nat}}$. 
\par 
$(2)$ There exists a \v{C}ech diagram $X_{\os{\circ}{T}_0\bul}$ of $X_{\os{\circ}{T}_0}$ 
and an immersion 
\begin{equation*}  
\begin{CD} 
X_{\os{\circ}{T}_0\bul} 
@>{\sus}>> {\cal P}_{\bul} \\
@VVV @VVV \\
S_{\os{\circ}{T}_0} @>{\subset}>> S(T)^{\nat}
\end{CD} 
\tag{2.13.2}\label{eqn:eolnpxd} 
\end{equation*} 
into a log smooth simplicial log scheme over $S(T)^{\nat}$. 
\end{prop} 

\begin{coro}\label{coro:exem} 
Let the notations be as above. 
Then there exist  an 
immersion 
\begin{equation*}  
\begin{CD} 
X_{\os{\circ}{T}_0\bul} @>{\sus}>> \ol{\cal P}{}^{\rm ex}_{\bul} \\
@VVV @VVV \\
S_{\os{\circ}{T}_0} @>{\subset}>> \ol{S(T)^{\nat}}
\end{CD} 
\tag{2.14.1}\label{eqn:eipexxd} 
\end{equation*} 
into a simplicial strict semistable log formal scheme over 
$\ol{S(T)^{\nat}}$ and an immersion 
\begin{equation*}  
\begin{CD} 
X_{\os{\circ}{T}_0\bul} @>{\sus}>> 
{\cal P}^{\rm ex}_{\bul} \\
@VVV @VVV \\
S_{\os{\circ}{T}_0} @>{\subset}>> S(T)^{\nat}
\end{CD} 
\tag{2.14.2}\label{eqn:eoexpxd} 
\end{equation*} 
into a simplicial formal SNCL scheme over $S(T)^{\nat}$. 
\end{coro}
Set $X_{\ul{\lam},\os{\circ}{T}_0}:=X_{\ul{\lam}}\times_{\os{\circ}{S}}\os{\circ}{T}_0$ and 
$X_{\ul{\lam},\os{\circ}{T}_0\bul}:=X_{\ul{\lam},\os{\circ}{T}_0}
\times_{X_{\os{\circ}{T}_0}}X_{\os{\circ}{T}_0\bul}$. 
Then we have simplicial log schemes  
$\ol{\cal P}{}^{\rm ex}_{\bul \ul{\lam}}$ and ${\cal P}^{\rm ex}_{\bul \ul{\lam}}$ 
with the following exact immersions: 
\begin{align*} 
X_{\ul{\lam},\os{\circ}{T}_0\bul}\os{\sus}{\lo} {\cal P}^{\rm ex}_{\bul \ul{\lam}}
\os{\sus}{\lo} \ol{\cal P}^{\rm ex}_{\bul \ul{\lam}}.
\end{align*}
Here note that the turns of $\bul$ and $\ul{\lam}$ in 
$\ol{\cal P}{}^{\rm ex}_{\bul \ul{\lam}}$; there does not necessarily exist an immersion 
``$X_{\ul{\lam},\os{\circ}{T}_0}\os{\subset}{\lo} \ol{\cal P}{}^{\rm ex}_{\ul{\lam}}$'' into an  
ideally log smooth immersion over $\ol{S(T)^{\nat}}$: if it exists, then we can use the notation 
$\ol{\cal P}{}^{\rm ex}_{\ul{\lam} \bul}$. 
\par 

We conclude this section by recalling   
the ``mapping degree function'' defined in \cite{nb}.  
\par  
Let $v \col S \lo S'$ be a morphism of families of log points. 
Let $y$ be a point of $\os{\circ}{S}$. 
Let $h\col {\mab N}= M_{S',v(x)}/{\cal O}^*_{S',v(x)}
\lo M_{S,y}/{\cal O}^*_{S,y}={\mab N}$ 
be the induced morphism. 
Let $d\in {\mab N}$ be the image of $1\in {\mab N}$ by $h$.

\begin{defi}[{\bf A special case of \cite[(1.1.42)]{nb}}]\label{defi:ddef} 
We call $\deg(v)_{x}:=d$ 
the {\it $($mapping$)$ degree} of $v$ at $x$. 
We call $\deg(v)\col \os{\circ}{S}\lo {\mab Z}_{\geq 1}$ 
the {\it $($mapping$)$ degree function} of $v$. 
\end{defi}

 \par 
 In the case where $\os{\circ}{S}$ is a formal scheme, 
 we can obtain the analogous results to the results in this section.

\section{PD-Hirsch extensions of dga's}\label{sec:stpd} 
In this section we give the definition of the PD-Hirsch extension of a dga 
and we give fundamental properties of it. 
Any dga in the case of geometric applications in this book 
has an extra structure, which we call 
a differential graded PD-algebra structure. 
We also define the PD-Hirsch extension of it. 

\begin{defi}\label{defi:pd-hirsch}
Let $A$ be a commutative ring with unit element. 
\par 
(1) Let $(B,J,\del)$ be a PD-algebra such that $B$ is an $A$-algebra. 
Denote $\del_i(b)$ by $b^{[i]}$ $(i\in {\mab N}, b\in J)$ for simplicity of notation. 
Let $\Om^{\bul}$ be a dga over $A$ such that each $\Om^q$ $(q\in {\mab Z})$ is 
a $B$-module. Assume that we are given a $B$-linear morphism 
$\psi \col \Om^{\bul}_{B/A}\lo \Om^{\bul}$ of dga's. 
Then we say that $\Om^{\bul}$ is 
a $B$-{\it pd-dga}(={\it PD-differential graded algebra}) 
over $A$ (precisely speaking,  $(B,J,\del)$-{\it pd-dga} with respect to $\psi$)
if $d(b\om)=d\psi(b)\wedge \om +bd\om$ for $b\in B$ and $\om \in \Om^{\bul}$ 
and if   
$d(\psi(b^{[i]}))=b^{[i-1]}d\psi(b)$ for $b\in J$ $(i\geq 1)$.  
\par 
(2) Let $M$ be a projective $A$-module.  
Let $\Gam_A(M)$ be the PD-algebra over $A$ generated by $M$ 
(see \cite[Appendix A]{bob}).  
Let $\Om^{\bul}=(\Om^{\bul},d)$ be a dga over $A$. 
Let $\varphi \col  M\lo {\rm Ker}(d\col \Om^1\lo \Om^{2})$ be 
an $A$-linear morphism. 
The {\it PD-Hirsch extension} of $\Om^{\bul}$ by $M$ 
is, by definition, an injective morphism
\begin{align*} 
\Om^{\bul}=A\otimes_A\Om^{\bul}\os{\sus}{\lo} \Gam_A(M)\otimes_A\Om^{\bul}. 
\tag{3.1.1}\label{ali:gmh}
\end{align*} 
We denote $\Gam_A(M)\otimes_A\Om^{\bul}$ by 
$\Om^{\bul}\langle M\rangle$. 
We denote by $P_k \Gam_A(M)$ 
a submodule of $\Gam_A(M)$  
generated by 
$$m_{i_1}^{[i_1]}\cdots m_{i_r}^{[i_r]} \quad  
(i_1+\cdots +i_r\leq \lfloor \frac{k}{2}\rfloor)
\quad (m_{i_1}, \ldots, m_{i_r}\in M),$$
where $\lfloor?\rfloor$ is the floor function: 
$\lfloor x \rfloor:=\max\{n\in {\mab Z}~\vert~n\leq x\}$ for $x\in {\mab R}$. 
(The filtration $P$ on $\Gam_A(M)$ will be necessary for geometric applications.)  
\end{defi} 

\begin{rema}
(1) If $J=0$, then the notion of the $B$-pd-dga is equal to that of a $B$-dga over $A$.
\par 
(2) When $A$ is a ${\mab Q}$-algebra, the notion of the pd-dga
(resp.~the PD-Hirsch extension) is equivalent to 
that of the dga (resp.~the Hirsch extension).  
\end{rema}

\begin{prop}\label{prop:nmq}
Let $\Om^{\bul}$ be a dga over $A$. 
The $A$-module $\Om^{\bul}\langle M\rangle$ becomes a dga over $A$ naturally. 
If $\Om^{\bul}$ is a $B$-pd-dga, then 
$\Om^{\bul}\langle M\rangle$ becomes a $B$-pd-dga over $A$ naturally. 
\end{prop} 
\begin{proof} 
We define the degree of 
$\Gam_A(M)\otimes_A\Om^q$ $(q\in {\mab Z})$ as $q$. 
We define the algebra structure of $\Gam_A(M)\otimes_A\Om^{\bul}$ 
as the induced algebra structure by those of $\Gam_A(M)$ and $\Om^{\bul}$.  
We define the following two $A$-linear morphisms as follows: 
\begin{equation*} 
d_M(m^{[i]}):=d_{M,\varphi}(m^{[i]}):=
\begin{cases} 
0 & (i=0), \\
m^{[i-1]}\otimes \varphi (m) & (i\in {\mab Z}_{\geq 1})
\end{cases}
\tag{3.3.1}\label{ali:mdmi} 
\end{equation*}  
and 
\begin{align*} 
&d_{\Om^{\bul}\langle M\rangle}(\sum m^{[i_1]}_1\cdots m^{[i_r]}_r
\otimes \om_{i_1\cdots i_r})\tag{3.3.2}\label{ali:mdMmi} \\
&:=\sum(\sum_{j=1}^rm^{[i_1]}_1\cdots m^{[i_{j-1}]}_{j-1}  m^{[i_{j+1}]}_{j+1}
\cdots m^{[i_r]}_rd_{M}(m^{[i_j]}_j)
\wedge \om_{i_1\cdots i_r}\\
&+m^{[i_1]}_1\cdots m^{[i_r]}_r
\otimes d\om_{i_1\cdots i_r}) \quad (m_1,\ldots, m_r\in M, \om_{i_1\cdots i_r}\in \Om^q). 
\end{align*}  
\par 
We have to check that the following five relations hold: 
\par 
(i) $d_{\Om^{\bul}\langle M\rangle}^2=0$, 
\par 
(ii) 
$d_M((am)^{[i]})=a^id_M(m^{[i]})$ $(a\in A)$, 

\par
(iii) $d_M(m^{[i]}m^{[j]})=\{(i+j)!/i!j!\}d_M(m^{[i+j]})$, 
\par 
(iv) $d_M((m+n)^{[i]})=\sum_{j_1+j_2=i}d_{M}(m^{[j_1]}n^{[j_2]})$, 
\par
(v) $d_{M}((m^{[i]})^{[j]})=\dfrac{(ij)!}{(i!)^jj!}d_{M}(m^{[ij]})$. 
\par
By noting that $\varphi (m)\in {\rm Ker}(d\col  \Om^1\lo \Om^2)$, 
the first relation can be checked. 
The second relation is obvious since $\varphi$ is $A$-linear. 
The third relation is equivalent to the following equality 
\begin{align*} 
m^{[i-1]}m^{[j]}\otimes \varphi (m)+m^{[i]}m^{[j-1]}\otimes \varphi (m)=
\{(i+j)!/i!j!\}m^{[i+j-1]}\otimes \varphi (m), 
\end{align*} 
which is easily checked. 
The fourth relation is equivalent to 
the following equality 
\begin{align*} 
(m+n)^{[i-1]}\otimes \varphi (m+n)=
\sum_{j_1+j_2=i}\{
m^{[j_1-1]}n^{[j_2]}\otimes \varphi (m)+m^{[j_1]}n^{[j_2-1]}\otimes \varphi (n)\}, 
\end{align*} 
which is also easily checked. 
The fifth relation is equivalent to the following equality
\begin{align*} 
(m^{[i]})^{[j-1]}m^{[i-1]}=\dfrac{(ij)!}{(i!)^jj!}m^{[ij-1]}, 
\end{align*} 
which is also easily checked.
\end{proof}

\begin{prop}[{\bf Functoriality of PD-Hirsch extensions}]\label{prop:fc}
Let the notations be as above. 
Let $A\lo A'$ and $B\lo B'$ be morphisms of commutative rings with unit elements. 
Assume that $B'$ is an $A'$-algebra and the following diagram 
\begin{equation*} 
\begin{CD} 
A@>>> A'\\
@VVV @VVV\\
B@>>>B'
\end{CD}
\tag{3.4.1}\label{cd:aapbbp}
\end{equation*} 
is commutative. Let $\Om'{}^{\bul}$ be a dga over $B'$ and 
let $\varphi' \col M\lo {\rm Ker}(\Om'{}^1\lo \Om'{}^2)$ be 
an $A$-linear morphism. 
$($Here we consider $\Om'{}^{\bul}$ as an $A$-module.$)$
Let $(B',J',\del')$ be an analogous PD-algebra to $(B,J,\del)$. 
Let $h\col \Om^{\bul}\lo \Om'{}^{\bul}$ be a morphism of complexes 
fitting into the following commutative diagram 
\begin{equation*} 
\begin{CD} 
M@=M\\
@V{\varphi}VV @VV{\varphi'}V\\
{\rm Ker}(\Om^1\lo \Om^2) @>{h}>>{\rm Ker}(\Om'{}^1\lo \Om'{}^2). 
\end{CD} 
\tag{3.4.2}\label{cd:aammbp}
\end{equation*}
Then $h$, $\varphi$ and $\varphi'$ induce the following natural morphism 
\begin{equation*} 
\Om^{\bul}\langle M\rangle \lo \Om'{}^{\bul}\langle M\rangle.  
\tag{3.4.3}\label{cd:aaombp}
\end{equation*} 
This morphism is compatible with the compositions of 
two {\rm (\ref{cd:aapbbp})}'s and two {\rm (\ref{cd:aammbp})}'s in the obvious sense. 
\end{prop}
\begin{proof} 
We leave the proof to the reader because the proof is straightforward. 
\end{proof}

\begin{exem}\label{exem:ys} 
Let $Y\lo S$ be a morphism of log schemes. 
Let $Y\os{\sus}{\lo} {\mathfrak D}$ be a closed immersion into a log scheme over $S$ 
defined by a PD-ideal sheaf of ${\cal O}_{{\mathfrak D}}$ with PD-structure $[~]$. 
Let $g\col {\mathfrak D}\lo S$ be the structural morphism. 
Let $\Om^{\bul}$ be a sheaf of dga's on ${\mathfrak D}_{\rm zar}$ 
over $g^{-1}({\cal O}_S)$ 
such that each $\Om^q$ $(q\in {\mab Z})$ is 
an ${\cal O}_{{\mathfrak D}}$-module.  
Let $\Om^{\bul}_{[~]}$ be a quotient sheaf of $\Om^{\bul}$ divided by 
the ideal sheaf generated by $a^{[i]}-a^{[i-1]}da$'s for a local section $a$ 
of ${\rm Ker}({\cal O}_{\mathfrak D}\lo {\cal O}_Y)$. 
Let ${\cal M}$ be a locally free $g^{-1}({\cal O}_S)$-module and 
let $\varphi \col {\cal M}\lo {\rm Ker}(\Om^1_{[~]}\lo \Om^2_{[~]})$ be a 
morphism of $g^{-1}({\cal O}_S)$-modules.  
Then, as in (\ref{defi:pd-hirsch}),  we have the notion of 
an ${\cal O}_{\mathfrak D}$-pd-dga and 
the PD-Hirsch extension of $\Om^{\bul}_{[~]}$ by ${\cal M}$.  
\end{exem}

\begin{prop}\label{prop:ex}
The following hold$:$
\par 
$(1)$ Let $s$ be a fine log scheme whose underlying scheme is 
the spectrum of a perfect field $\kap$. 
Let $\star$ be a positive integer $n$ or nothing. 
Let ${\cal W}_{\star}(s)$ be the $($formal$)$ canonical lift of $s$ over 
${\rm Spf}({\cal W}_{\star})$. 
Let $Y/s$ be a log smooth scheme of Cartier type
and let $g\col Y\lo s\os{\sus}{\lo} {\cal W}_{\star}(s)$ be 
the structural morphism. 
Let ${\cal W}_{\star}\Om^{\bul}_Y$ be the log de Rham-Witt dga 
of $Y/{\cal W}_{\star}(s)$. 
$($Here note that each ${\cal W}_{\star}\Om^q_Y$ $(q\in {\mab N})$ 
is a sheaf of ${\cal W}_{\star}({\cal O}_Y)'$-modules by 
the isomorphism ${\cal W}_{\star}({\cal O}_Y)'\os{\sim}{\lo} {\cal W}_{\star}\Om^0_Y$, 
where ${\cal W}_{\star}({\cal O}_Y)'$ is the obverse Witt sheaf of $Y$ 
in the sense of {\rm \cite[(7.8)]{ndw}} and ${\cal W}_{\star}({\cal O}_Y)'$ has the 
PD-ideal sheaf $V{\cal W}_{\star}({\cal O}_Y)'.)$ 
Then ${\cal W}_{\star}\Om^{\bul}_Y$ is a sheaf of 
${\cal W}_{\star}({\cal O}_Y)'$-pd-dga's over $g^{-1}({\cal W}_{\star})$. 
\par 
$(2)$ Let the notations be as in $(1)$. 
Assume that $s$ is the log point of $\kap$. 
Let ${\cal W}_{\star}\wt{\Om}^{\bul}_Y$ be 
the log de Rham-Witt complex defined in 
{\rm \cite{msemi}} and {\rm \cite{ndw}} $($cf.~{\rm \cite{hyp})}. 
Then ${\cal W}_{\star}\wt{\Om}^{\bul}_Y$ is a sheaf of 
${\cal W}_{\star}({\cal O}_Y)'$-pd-dga's over 
$g^{-1}({\cal W}_{\star})$. 
\par 
$(3)$ Let $(T,{\cal J},\del)$ be a fine log PD-scheme. 
Let $Y\os{\sus}{\lo} {\cal Y}$ be an immersion of fine log schemes  
over $T$ and let $g\col {\cal Y}\lo T$ be the structural morphism.  
Let ${\mathfrak E}$ be the log PD-envelope over $(T,{\cal J},\del)$. Then 
${\cal O}_{\mathfrak E}\otimes_{{\cal O}_{\cal Y}}\Om^{\bul}_{{\cal Y}/U}$ 
for $U=T$ or $\os{\circ}{T}$ is 
a sheaf of ${\cal O}_{\mathfrak E}$-pd-dga's over $g^{-1}({\cal O}_T)$. 
\end{prop} 
\begin{proof} 
(1): It is well-known that 
${\cal W}_{\star}\Om^{\bul}_Y$ is a sheaf of dga's over 
$g^{-1}({\cal W}_{\star})$. 
As in the proof of \cite[(4.19)]{hk}, we can easily check 
that ${\cal W}_{\star}\Om^{\bul}_Y$ is a 
${\cal W}_{\star}({\cal O}_Y)'$-pd-dga over 
$g^{-1}({\cal W}_{\star})$ as follows. 
\par 
Let $a$ be a section of $V{\cal W}_{\star}({\cal O}_Y)'$. 
Then 
\begin{align*} 
i!d(a^{[i]})=da^i=ia^{i-1}da=i!a^{[i-1]}da.
\end{align*}  
Because ${\cal W}\Om^q_Y$ is a torsion-free sheaf 
(\cite[(6.8)]{ndw}), 
$d(a^{[i]})=a^{[i-1]}da$ in ${\cal W}\Om^1_Y$.
Because the natural projection ${\cal W}\Om^q_Y
\lo {\cal W}_{\star}\Om^q_Y$ is surjective 
(\cite[(6.4.5)]{ndw}), 
$d(a^{[i]})=a^{[i-1]}da$ in ${\cal W}_{\star}\Om^1_Y$.
\par 
(2): The proof of (2) is the same as that of (1).  
\par 
(3): By \cite[(1.3.28)]{nb}(=the log version of \cite[0 (3.1.6)]{idw}) 
\begin{align*} 
\Om^{\bul}_{{\mathfrak E}/U, [~]}
\os{\sim}{\lo} 
{\cal O}_{\mathfrak E}
\otimes_{{\cal O}_{{\cal Y}}}\Om^{\bul}_{{\cal Y}/U}. 
\tag{3.6.1}\label{ali:fwniwu}
\end{align*}  
(3) follows from this equality. 
\end{proof} 

\begin{exem}\label{exem:hipd}
(1) (\cite{kiha}) Let the notations be as in (\ref{prop:ex}) (2). 
Let $u$ be a variable over ${\cal W}_{\star}$  
and consider the free ${\cal W}_{\star}$-module  
${\cal W}_{\star}u$  over ${\cal W}_{\star}$ of rank 1. 
Here we define the degree of $u$ as $0$. 
Consider the following morphism 
\begin{align*} 
\varphi \col {\cal W_{\star}}u
\owns u\lom \theta\in 
{\rm Ker}(d\col {\cal W}_{\star}\wt{\Om}^1_Y\lo {\cal W}_{\star}\wt{\Om}^2_Y),
\tag{3.7.1}\label{ali:dtwy}
\end{align*}  
where $\theta=$``$d\log t$'' is a 1-form defined in \cite{hyp}, \cite{msemi}  and \cite{ndw}.   
Set 
\begin{align*} 
{\cal W}_{\star}\wt{\Om}^{\bul}_Y\langle u\rangle 
:={\cal W}_{\star}\wt{\Om}^{\bul}_Y\langle {\cal W}_{\star} u\rangle.
\end{align*}  
(Strictly speaking, we should denote 
${\cal W}_{\star}\wt{\Om}^{\bul}_Y\langle {\cal W}_{\star} u\rangle$ 
by 
${\cal W}_{\star}\wt{\Om}^{\bul}_Y\langle g^{-1}({\cal W}_{\star} u)\rangle$.)
Then we obtain the following PD-Hirsch extension 
\begin{align*} 
{\cal W}_{\star}\wt{\Om}^{\bul}_Y\os{\sus}{\lo} 
{\cal W}_{\star}\wt{\Om}^{\bul}_Y\langle u\rangle.  
\end{align*}
\par 
(2) Let the notations be as in (\ref{prop:ex}) (3). 
Let $r$ be a positive integer. 
Assume that $\ol{M}_T:=M_T/{\cal O}_T^*$ 
is locally isomorphic to ${\mab N}^r$ $(r\in {\mab N})$, 
where $r$ does {\it not} depend on local open log subschemes of $T$. 
(If $T$ is a family of log points of virtual dimension $r$ defined in \cite[(1.1)]{nb}, 
then $\ol{M}_T$ is locally isomorphic to ${\mab N}^r$.) 
Let us consider the following morphism 
\begin{align*} 
d\log \col g^{-1}(M_T)\owns m \lom d\log m \in \Om^1_{{\cal Y}/\os{\circ}{T}}. 
\end{align*} 
Since $d\log a=0$ $(a\in {\cal O}_T^*)$ in $\Om^1_{{\cal Y}/\os{\circ}{T}}$, 
the morphism $d\log$ induces the following morphism 
\begin{align*} 
d\log \col g^{-1}(\ol{M}_T)\owns m \lom d\log m \in \Om^1_{{\cal Y}/\os{\circ}{T}}. 
\end{align*} 
\par 
Let $t_1,\ldots, t_r$ be the local generators of $\ol{M}_T$. 
Set 
\begin{align*} 
U_T:=\bigoplus_{i=1}^r{\cal O}_Tu_i, 
\tag{3.7.2}\label{ali:utrti} 
\end{align*} 
where $u_i$'s are independent formal variables over ${\cal O}_T$.   
Here the variable $u_i$ corresponds to $t_i$. 
(We consider $u_i$ as $\log t_i$ in our mind.)  
Note that the multiplication in $M_T$ by a local section of ${\cal O}_T^*$ 
and the multiplication in $U_T$ by a local section of 
${\cal O}_T$ is essentially different. 
The local section 
$g^*(d\log t_i)\in \Om^1_{{\cal Y}/\os{\circ}{T}}$ 
is a well-defined 1-form on ${\cal Y}$.  
By abuse of notation, denote $g^*(d\log t_i)$ by $d\log t_i$. 
Since ${\rm Aut}({\mab N}^r)={\mathfrak S}_r$ (\cite[p.~47]{nh3}), the sheaf 
$U_T$ is a well-defined ${\cal O}_T$-module. 
Here we define the degree of $U_T$ as $0$. 
Consider the following morphism 
\begin{align*} 
\varphi \col U_T \owns \sum_{i=1}^ra_iu_i\lom \sum_{i=1}^ra_id\log t_i\in 
{\cal O}_{\mathfrak E}\otimes_{{\cal O}_{{\cal Y}}}\Om^1_{{\cal Y}/\os{\circ}{T}} \quad 
(a_i\in {\cal O}_T). 
\end{align*}
Then we have the following PD-Hirsch extension 
\begin{align*} 
{\cal O}_{\mathfrak E}\otimes_{{\cal O}_{{\cal Y}}}\Om^{\bul}_{{\cal Y}/\os{\circ}{T}}
\os{\sus}{\lo} 
{\cal O}_{\mathfrak E}\otimes_{{\cal O}_{{\cal Y}}}\Om^{\bul}_{{\cal Y}/\os{\circ}{T}}
\langle U_T\rangle. 
\end{align*}
(Strictly speaking, we should denote $U_T$ by $g^{-1}(U_T)$.)
The sheaf $\Gam_{{\cal O}_T}(U_T)$ is non-canonically locally isomorphic to  
${\cal O}_{{\mathfrak D}(\ol{T})}$, 
where $\ol{T}$ is a log scheme constructed in \cite[\S2]{nb} and 
${\mathfrak D}(\ol{T})$ is the log PD-envelope of the exact closed immersion 
$T\os{\sus}{\lo} \ol{T}$ over the PD scheme $(\os{\circ}{T},(0),[~])$. 
\par 
In the case of characteristic $0$, this (PD-)Hirsch extension has appeared in 
\cite{fut} and \cite{fup}. 
In this book this PD-Hirsch extension for the case $r=1$ plays an important role.
In \cite{nf} we have considered the case where $r$ is general and we give 
an important complementary result to \cite{fut}.  
\par 
\par 
(3)  The following observation will be important in this book. 
\par
Let $S$ and $S'$ be (formal) families 
of log points of virtual dimensions $r$ and $r'$, respectively (\cite[\S2]{nb}).  
Let $v\col S\lo S'$ be a morphism of (formal) families  of log points. 
This morphism induces a morphism $v^*\col \ol{M}_{S'}\lo v_*(\ol{M}_{S})$.  
Locally on $S$, this morphism  is equal to a morphism 
$v^*\col {\mab N}^{r'}\lo {\mab N}^r$. 
Set 
$e_i=(0, \ldots,0,\os{i}{1},0,\ldots, 0)\in {\mab N}^s$ for $s=r$ or $r'$  
$(1\leq i\leq r)$. 
Let $A=(a_{ij})_{1\leq i\leq r, 1\leq j\leq r'} \in M_{rr'}({\mab N})$ 
be the representing matrix of $v^*$: 
$$v^*(e_1,\ldots,e_{r'})=(e_1,\ldots,e_r)A.$$ 
Let $\wt{t}_1,\ldots,\wt{t}_r$ and $\wt{t}{}'_1,\ldots, \wt{t}{}'_{r'}$ 
be local sections of $M_S$ and $M_{S'}$ 
whose images in $\ol{M}_S\os{\sim}{\lo} {\mab N}^r$ and 
$\ol{M}_{S'}\os{\sim}{\lo}{\mab N}^{r'}$ are 
$e_1,\ldots,e_r$ and $e_1,\ldots,e_{r'}$, respectively. 
Then there exists a local section $b_j\in {\cal O}_{S}^*$ 
$(1\leq j\leq r')$ such that 
\begin{align*} 
v^*(\wt{t}{}'_j)=b_j\wt{t}{}^{a_{1j}}_1\cdots \wt{t}{}^{a_{rj}}_r. 
\tag{3.7.3}\label{ali:vta} 
\end{align*}
Let $t_i$ and $t_i'$ be the images of $\wt{t}_i$ and $\wt{t}{}'_i$ in 
$\ol{M}_S$ and $\ol{M}_{S'}$, respectively. 
Let $u_1,\ldots, u_r$ and $u'_1,\ldots,u'_{r'}$ be 
the corresponding local sections 
to $t_1,\ldots, t_r$ and $t'_1,\ldots,t'_{r'}$ of 
$U_S$ and $U_{S'}$, respectively. Then we define 
an ${\cal O}_{S'}$-linear morphism 
$v^*\col \Gam_{{\cal O}_{S'}}(U_{S'})\lo v_*(\Gam_{{\cal O}_S}(U_S))$
by the following formula: 
\begin{align*} 
v^*(u'_j)=a_{1j}u_1+\cdots+a_{rj}u_r.
\tag{3.7.4}\label{ali:vua} 
\end{align*}
and by $\Gam(v^*)$ for $v^*$ in (\ref{ali:vua}). 
Since ${\rm Aut}({\mab N}^s)={\mathfrak S}_s$, 
the morphism $v^*\col \Gam_{{\cal O}_{S'}}(U_{S'})\lo v_*(\Gam_{{\cal O}_S}(U_S))$
is independent of the choice of the local isomorphisms 
$\ol{M}_{S'}\simeq {\mab N}^{r'}$ and $\ol{M}_{S}\simeq {\mab N}^r$. 
In conclusion, 
we obtain the following well-defined morphism 
\begin{align*} 
v^*\col \Gam_{{\cal O}_{S'}}(U_{S'})\lo 
v_*(\Gam_{{\cal O}_S}(U_S))
\tag{3.7.5}\label{ali:kxvef}
\end{align*} 
of sheaves of commutative rings of unit elements on $S'$.  
This morphism satisfies the usual transitive relation 
\begin{align*} 
(v\circ v')^*=v'{}^*v^*.
\tag{3.7.6}\label{ali:kxbvef}
\end{align*} 
(In the case $r=r'=1$, $a_{11}={\rm deg}(v)$ in (\ref{defi:ddef}).)
It is easy to check that the following diagram is commutative: 
\begin{equation*}
\begin{CD}
\Gam_{{\cal O}_{S'}}(U_{S'})@>{v^*}>> v_*(\Gam_{{\cal O}_S}(U_S))\\
@V{d}VV @VV{d}V\\
\Om^1_{S'/\os{\circ}{S}{}'}\langle U_{S'}\rangle @>{v^*}>> 
v_*(\Om^1_{S/\os{\circ}{S}{}}\langle U_{S}\rangle).
\end{CD}
\tag{3.7.7}\label{cd:cucp}
\end{equation*}
Indeed, 
\begin{align*} 
v^*d(u'{}^{[l]}_j)&=v^*u'{}^{[l-1]}_jd\log t'_j =
v^*u'{}^{[l-1]}_jd\log (b_jt^{a_{1j}}_1\cdots t^{a_{rj}}_r)\\
&=v^*(u'{}^{[l-1]}_j)(a_{1j}d\log t_1+\cdots +a_{rj}d\log t_r)\\
&=(a_{1j}u_1+\cdots+a_{rj}u_r)^{[l-1]}
(a_{1j}d\log t_1+\cdots +a_{rj}d\log t_r)
\end{align*} 
and 
\begin{align*} 
dv^*(u'{}^{[l]}_j)&=d(a_{1j}u_1+\cdots+a_{rj}u_r)^{[l]}\\
&=(a_{1j}u_1+\cdots+a_{rj}u_r)^{[l-1]}
(a_{1j}d\log t_1+\cdots+a_{rj}d\log t_r). 
\end{align*} 
\par 
The remarks (2) and (3) suggest to us that, when 
we consider the (PD-)Hirsch extensions in geometric cases, 
we have only to consider local sections of $\ol{M}_S$;  
it is not necessary to consider local sections of $M_S$.  
\end{exem} 


Inspired by \cite{ey}, we give the following definition: 

\begin{defi}
Let the notations be as in (\ref{defi:pd-hirsch}) (2). 
Set 
\begin{align*} 
\Gam_A(M)^{\wedge}:=\vpl_n\Gam_{A}(M)/\Gam_{A,n}(M)
=
\vpl_n\Gam_A(M)/(x^{[n]}~\vert~x\in M)
\end{align*}
and 
\begin{align*} 
\Om^{\bul}\langle \langle M\rangle \rangle
:=\Gam_A(M)^{\wedge}\wh{\otimes}_A\Om^{\bul}:=
\vpl_n(\Gam_A(M)\otimes_A\Om^{\bul}/x^{[n]}\Om^{\bul}). 
\end{align*}
We call the following injective morphism 
\begin{align*} 
\Om^{\bul}\os{\sus}{\lo} 
\Om^{\bul}\langle \langle M\rangle \rangle
\end{align*}
the {\it completed PD-Hirsch extension} of $\Om^{\bul}$ by $(M,\varphi)$ 
(or simply by $M$). 
\end{defi}

The $A$-algebra $\Gam_A(M)^{\wedge}$ 
has a natural decreasing filtration $F$: 
\begin{align*}
F^i(\Gam_A(M)^{\wedge}):=
\{\sum a_{e_1\cdots e_k,i_1\cdots i_k}m^{[e_1]}_{i_1}\cdots m^{[e_k]}_{i_k}&~\vert~
a_{e_1\cdots e_k,i_1\cdots i_k}\in A, \\
&m_{i_1},\ldots m_{i_k}\in M, e_1+\cdots +e_k\geq i\} \quad (i\in {\mab Z}).
\end{align*} 
This filtration induces an increasing filtration 
$\Gam_A(M)^{\wedge}\wh{\otimes}_A\Om^{\bul}$ which we denote by $F$ by abuse of notation: 
\begin{align*}
F_i(\Om^{\bul}\langle \langle M\rangle \rangle):=F^{-i}(\Gam_A(M)^{\wedge})
\wh{\otimes}_A\Om^{\bul} \quad (i\in {\mab Z}).
\tag{3.8.1}\label{ali:film}
\end{align*} 
Note that $F$ is indeed a filtration on the complex 
$\Om^{\bul}\langle \langle M\rangle \rangle$. 
This filtration also induces an increasing filtration $F$ 
on  $\Om^{\bul}\langle M\rangle$.

\begin{prop}\label{prop:mom}
Let $\Om^{\bul}$ be a dga over $A$. 
The $A$-module $\Om^{\bul}\langle \langle M\rangle \rangle$ becomes a dga over $A$ naturally. 
If $\Om^{\bul}$ is a $B$-pd-dga, then 
$\Om^{\bul}\langle \langle M\rangle \rangle$ becomes a pd-dga over $A$ naturally. 
\end{prop}
\begin{proof}
Set 
\begin{align*} 
d_M(\sum_{i=0}^{\inf}m^{[i]}):=d_{M,\varphi}(\sum_{i=0}^{\inf}m^{[i]}):=
\sum_{i=1}^{\inf}m^{[i-1]}\otimes \varphi (m)
\tag{3.9.1}\label{ali:mmi} 
\end{align*}  
and 
\begin{align*} 
&d_{\Om^{\bul}\langle \langle M\rangle \rangle}
(\sum_{i_1,\ldots, i_r=0}^{\inf} m^{[i_1]}_1\cdots m^{[i_r]}_r
\otimes \om_{i_1\cdots i_r})\tag{3.9.2}\label{ali:nMmi} \\
&:=\sum_{i_1,\ldots, i_r=0}^{\inf} (\sum_{j=1}^rm^{[i_1]}_1\cdots m^{[i_{j-1}]}_{j-1}  
m^{[i_{j+1}]}_{j+1}
\cdots m^{[i_r]}_rd_{M}(m^{[i_j]}_j)
\wedge \om_{i_1\cdots i_r}\\
&+m^{[i_1]}_1\cdots m^{[i_r]}_r
\otimes d\om_{i_1\cdots i_r}) \quad (m_1,\ldots, m_r\in M, \om_{i_1\cdots i_r}\in \Om^q). 
\end{align*}  
The rest of the proof of this proposition is similar to that of (\ref{prop:nmq}). 
\end{proof}

\par 

More generally, we consider the following PD-Hirsch extension: 

\begin{prop-defi}\label{prop-defi:pdc}
Let $n$ be a positive integer. 
Let $A$ and $M$ be as in (\ref{defi:pd-hirsch}). 
Let $B_i$ $(1\leq i\leq n)$ be an $A$-algebra.  
Let $\Om^{\bul}_i$  be a dga over $A$ 
such that each $\Om^q_i$ $(q\in {\mab Z})$ is a $B_i$-module.  
Let $E_i$ be a $B_i$-module. 
Let $\varphi_i \col M\lo {\rm Ker}(d\col \Om^1_i\lo \Om^2_i)$ 
be a morphism of $A$-modules. 
Let $(\bigoplus_{i=1}^n(E_i\otimes_{B_i}\Om^{\bul}_i)[-n_i],d)$ 
$(n_i\in {\mab Z})$ be a complex of $A$-modules. 
$($Here $d$ is not necessarily the direct sum of certain differentials of 
$E_i\otimes_{B_i}\Om^{\bul}_i[-n_i].)$ 
Let $\sq$ be nothing or $\wedge$. 
Then the $A$-linear morphism 
$$d_H:=d_{H,(\varphi_1,\ldots,\varphi_r)}\col \Gam_A(M)^{\sq}\os{\sq}{\otimes}_A
(\bigoplus_{i=1}^n(E_i\otimes_{B_i}\Om^{q-n_i}_i))
\lo \Gam_A(M)^{\sq}\os{\sq}{\otimes}_A(\bigoplus_{i=1}^n(E_i\otimes_{B_i}\Om^{{q+1}-n_i}_i))$$ 
defined by the following formula 
\begin{align*} 
&d_H(\sum_{i=1}^n\sum'_{j(i)_1,\ldots, j(i)_r\in {\mab N}}m^{[j(i)_1]}_1
\cdots m^{[j(i)_r]}_r\otimes e_i\otimes \om_i)\\
&:=
\sum_{i=1}^n\sum'_{j(i)_1,\ldots, j(i)_r\in {\mab N}}\sum_{k=1}^r
m^{[j(i)_1]}_1\cdots m^{[j(i)_k-1]}_k \cdots m^{[j(i)_r]}_r\otimes e_i\otimes 
\varphi_i(m_k)\wedge \om_i\\
&+\sum_{i=1}^n\sum'_{j(i)_1,\ldots, j(i)_r\in {\mab N}}
m^{[j(i)_1]}_1\cdots m^{[j(i)_r]}_r\otimes d(e_i\otimes \om_i) \\
&\quad (m_1,\ldots m_r\in M,e_i\in E_i,\om \in \Om^{q-n_i}_i)
\end{align*} 
makes $\Gam_A(M)^{\sq}\os{\sq}{\otimes}_A(\bigoplus_{i=1}^n
(E_i\otimes_{B_i}\Om^{\bul}_i)[-n_i])$ 
a complex of $A$-modules. 
Here $\os{'}{\sum}$ means that we allow the infinite sum when $\square=\wedge{}$. 
We call the following natural injective morphism 
\begin{align*} 
\bigoplus_{i=1}^n(E_i\otimes_{B_i}\Om^{\bul}_i)[-n_i]
\os{\sus}{\lo} 
\Gam_A(M)^{\sq}\os{\sq}{\otimes}_A(\bigoplus_{i=1}^n(E_i\otimes_{B_i}\Om^{\bul}_i)[-n_i])
\end{align*} 
of $A$-modules the ({\it completed}) {\it PD-Hirsch extension} of 
$(\bigoplus_{i=1}^n(E_i\otimes_{B_i}\Om^{\bul}_i)[-n_i],d)$ 
by $(M,(\varphi_1,\ldots,\varphi_r))$. 
We denote it by $(\bigoplus_{i=1}^n(E_i\otimes_{B_i}\Om^{\bul}_i)[-n_i])\{ \langle M\rangle \}
=(\bigoplus_{i=1}^n(E_i\otimes_{B_i}\Om^{\bul}_i)[-n_i])\{ \langle M
\rangle\}_{(\varphi_1,\ldots,\varphi_r)}$, where $\{\langle M\rangle \}$ is $\langle M\rangle$  or 
$\langle \langle M\rangle\rangle$  
according to the case where $\sq$ is nothing or $\wedge$. 
\end{prop-defi} 
\begin{proof} 
Because the proof of this proposition is similar to 
that of (\ref{prop:nmq}), 
we leave the detailed proof of it to the reader. 
\end{proof}

\par 
Let $A\lo A'$ and $B\lo B'$ be morphisms of commutative rings. 
Assume that $B'$ is an $A'$-algebra and the following diagram 
\begin{equation*} 
\begin{CD} 
A@>>> A'\\
@VVV @VVV\\
B@>>>B'
\end{CD}
\end{equation*} 
is commutative. 
Let $\varphi' \col M\lo {\rm Ker}(\Om'{}^1\lo \Om'{}^2)$ be 
an analogous morphism to $\varphi$. 
Let $(B',J',\del')$ be an analogous PD-algebra to $(B,J,\del)$. 
Let $h\col \Om^{\bul}\lo \Om'{}^{\bul}$ be a morphism of complexes 
fitting into the following commutative diagram 
\begin{equation*} 
\begin{CD} 
M@=M\\
@V{\varphi}VV @VV{\varphi'}V\\
{\rm Ker}(\Om^1\lo \Om^2) @>{h}>>{\rm Ker}(\Om'{}^1\lo \Om'{}^2). 
\end{CD} 
\end{equation*}
Let $E$ (resp.~$E'$) be a $B$-module (resp.~$B'$-module). 
Let $g\col E\lo E'$ be a morphism of $B$-modules. 
(We can consider $E'$ as a $B$-module by using the morphism $B\lo B'$.)
Let $(E\otimes_B\Om^{\bul}, d)$ (resp.~$(E'\otimes_{B'}\Om'{}^{\bul}, d')$) 
be a complex of $A$-modules (resp.~a complex of $A'$-modules).  
Assume that $f:=g\otimes h\col E\otimes_B\Om^{\bul}\lo E'\otimes_{B'}\Om'{}^{\bul}$ 
is a morphism of complexes of $A$-modules. 
Let $f_{\{\langle M\rangle\}}\col 
E\otimes_{B}\Om^{\bul}\{\langle M\rangle\} \lo E'\otimes_{B'}\Om'^{\bul}
\{\langle M\rangle\}$   
be $f_{\langle M\rangle}\col 
E\otimes_{B}\Om^{\bul}\langle M\rangle \lo 
E'\otimes_{B'}\Om'^{\bul} \langle M\rangle$  
or $f_{\langle \langle M\rangle \rangle} 
\col 
E\otimes_{B}\Om^{\bul}\langle \langle M\rangle \rangle\lo 
E'\otimes_{B'}\Om'^{\bul} \langle \langle M\rangle \rangle$ 
according to the case where $\sq$ is nothing or $\wedge$. 
Consider the mapping cone ${\rm MC}(f_{\{\langle M\rangle\}})$ of $f_{\{\langle M\rangle\}}$. 
The following holds: 
\begin{align*} 
{\rm MC}(f_{\{\langle M\rangle\}})^q&= 
\{\Gam_A(M)^{\sq}\os{\sq}{\otimes}_A(E\otimes_{B}\Om^{q+1}) \}
\oplus \{\Gam_A(M)^{\sq}\os{\sq}{\otimes}_A(E'\otimes_{B'}\Om'^{q})\}\\
&=\Gam_A(M)^{\sq}\os{\sq}{\otimes}_A((E\otimes_{B}\Om^{q+1})\oplus (E'\otimes_{B'}\Om'^{q}))\\
&=\Gam_A(M)^{\sq}\os{\sq}{\otimes}_A{\rm MC}(f)^q. 
\end{align*} 
\par  
The following is the commutativity of the operation of 
the mapping cone and that of the PD-Hirsch extension modulo 
the signs of $\varphi$. 

\begin{lemm}\label{lemm:mcfgq}
The following diagram is commutative$:$
\begin{equation*} 
\begin{CD} 
{\rm MC}(f_{\{\langle M\rangle\}})^q
@=
\Gam_A(M)^{\sq}\os{\sq}{\otimes}_A{\rm MC}(f)^q\\
@V{d_{H,(\varphi,\varphi')}}VV @VV{d_{H,(-\varphi, \varphi')}}V \\
{\rm MC}(f_{\{\langle M\rangle\}})^{q+1}
@=
\Gam_A(M)^{\sq}\os{\sq}{\otimes}_A{\rm MC}(f)^{q+1}. 
\end{CD}
\tag{3.11.1}\label{cd:mqq}
\end{equation*} 
\end{lemm}
\begin{proof} 
For simplicity of notation, consider the following simple element 
$$(m^{[j]}\otimes e\otimes \om,m'^{[j']}\otimes e'\otimes \om')\in 
{\rm MC}(f_{\{\langle M\rangle\}})^q 
\quad (m,m'\in M, e\in E,\om \in \Om^{q+1},e'\in E',\om' \in \Om'{}^q).$$ 
By the boundary morphism 
\begin{align*} 
{\rm MC}(f_{\{\langle M\rangle\}})^q\lo {\rm MC}(f_{\{\langle M\rangle\}})^{q+1}, 
\end{align*}
this simple element is mapped to 
\begin{align*} 
&
(-\{m^{[j-1]}\otimes e\otimes \varphi(m)\wedge \om
+
m^{[j]}\otimes d(e\otimes \om)\},\\
&\{m'{}^{[j'-1]}\otimes e'\otimes \varphi'(m')\wedge \om'
+m'{}^{[j']}\otimes d'(e'\otimes \om')\}+m^{[j]}\otimes f(e\otimes \om))\\
&=\{(m^{[j-1]}\otimes e\otimes (-\varphi(m))\wedge \om,0)
+m^{[j]}\otimes(-d(e\otimes \om),f(e\otimes \om))\}\\
&+\{m'{}^{[j'-1]}\otimes (0,e'\otimes \varphi'(m')
\wedge \om')+m'{}^{[j']}\otimes(0,d'(e'\otimes \om'))\}\\
&=d_{H,(-\varphi,\varphi')}((m^{[j]}\otimes e\otimes \om,0))
+d_{H,(-\varphi,\varphi')}((0,m'{}^{[j']}\otimes e'\otimes \om'))\\
&=d_{H,(-\varphi,\varphi')}((m^{[j]}\otimes e\otimes \om,m'{}^{[j']}\otimes e'\otimes \om')). 
\end{align*} 
We leave the rigorous proof for a general element of  
${\rm MC}(f_{\{\langle M\rangle\}})^q$ to the serious reader. 
\end{proof}

\begin{prop}\label{prop:acy}
Let the notations be as in {\rm (\ref{prop-defi:pdc})}. 
Assume that $M$ is a direct summand of a free $A$-module of 
countable rank. 
If $(\bigoplus_{i=1}^n(E_i\otimes_{B_i}\Om^{\bul}_i)[-n_i],d)$ is acyclic,  
then $(\bigoplus_{i=1}^n(E_i\otimes_{B_i}\Om^{\bul}_i)[-n_i])\langle M\rangle$ is acyclic. 
\end{prop}
\begin{proof} 
Because the inductive limit with respect to a directed set 
preserves the exactness, we may assume that 
$M$ is a direct summand of a free $A$-module of finite rank. 
Because the problem is local on ${\rm Spec}(A)$, 
we may assume that $M$ is a free $A$-module of finite rank. 
Decompose $M=L\oplus N$ into a direct sum of free $A$-modules, 
where $L$ $(1\leq i\leq r)$ is free of rank 1. 
Then $(\bigoplus_{i=1}^n(E_i\otimes_{B_i}\Om^{\bul}_i)[-n_i])\langle M\rangle=
(\bigoplus_{i=1}^n(E_i\otimes_{B_i}\Om^{\bul}_i)[-n_i])\langle L\rangle \langle N\rangle$.  
(Here we have identified  
$\Gam_A(L)\otimes_A(\bigoplus_{i=1}^nE_i\otimes_{B_i}\Om^{q-n_i}_i)$ 
with $\bigoplus_{i=1}^n(\Gam_A(L)\otimes_AE_i)\otimes_{B_i}\Om^{q-n_i}_i$ 
for each $q\in {\mab N}$.) 
Hence, by using induction on the rank of $M$, 
it suffices to prove that $(\bigoplus_{i=1}^n(E_i\otimes_{B_i}\Om^{\bul}_i)[-n_i]\langle 
L\rangle$ is acyclic. 
Let $l$ be a basis of $L$. 
Consider the following double complex 
\begin{equation*} 
\footnotesize{
\begin{CD}  
\cdots @>>> \cdots @>>> \cdots \\
@A{{\rm id}\otimes d}AA @A{{\rm id}\otimes d}AA 
@A{d}AA \\
Al^{[2]}
\otimes_{A}\bigoplus_{i=1}^n(E_i\otimes_{B_i}\Om^{q-n_i}_i)
@>>> Al\otimes_A \bigoplus_{i=1}^n(E_i\otimes_{B_i}\Om^{q+1-n_i}_i)
@>>> \bigoplus_{i=1}^n(E_i\otimes_{B_i}\Om^{q+2-n_i}_i) \\
@A{{\rm id}\otimes d}AA @A{{\rm id}\otimes d}AA @A{d}AA \\
A l^{[2]}\otimes_{A}(\bigoplus_{i=1}^n(E_i\otimes_{B_i}\Om^{q-1-n_i}_i))@>>>  A l\otimes_{A}(\bigoplus_{i=1}^n(E_i\otimes_{B_i}\Om^{q-n_i}_i))@>>> 
\bigoplus_{i=1}^n(E_i\otimes_{B_i}\Om^{q+1-n_i}_i)
\\
@A{{\rm id}\otimes d}AA @A{{\rm id}\otimes d}AA  @A{d}AA  \\
\cdots @. \cdots @. \bigoplus_{i=1}^n(E_i\otimes_{B_i}\Om^{q-n_i}_i)\\
@. @.  @A{d}AA  \\
\cdots  
@. \cdots  @. \cdots 
,
\end{CD}} 
\tag{3.12.1}\label{cd:hiarfd}
\end{equation*} 
where the horizontal arrow is defined by 
$l^{[j]}\otimes e_i \otimes \om_i \lom l^{[j-1]}\otimes e_i\otimes \varphi_i(l)\wedge \om_i$ 
$(e_i\in E_i,\om_i\in \Om^{*}_i)$. 
Since $Al^{[j]}$ $(j\in {\mab N})$ is a free $A$-module, 
the columns of (\ref{cd:hiarfd}) is exact. 
(Here note that $E\otimes_B\Om^{\bul}\langle L\rangle$ is 
not equal to the single complex of the double complex (\ref{cd:hiarfd}) in general: 
$E\otimes_B\Om^{\bul}\langle L\rangle$ is only contained in the single complex.) 
Hence, by an elementary argument, 
$E\otimes_B\Om^{\bul}\langle L\rangle$ is acyclic (cf.~\cite[(2.7.3)]{weib}). 
\end{proof}


\begin{coro}\label{coro:hac}
Let $f \col E\otimes_{B}\Om^{\bul}\lo E'\otimes_{B'}\Om'^{\bul}$ be 
as after {\rm (\ref{prop-defi:pdc})}.  
Assume 
that $f$ is a quasi-isomorphism of $A$-modules 
and that $M$ is a direct summand of a free $A$-module of countable rank. 
Then the morphism 
$f_{\langle M\rangle} \col 
E\otimes_{B}\Om^{\bul}\langle M\rangle \lo 
E'\otimes_{B'}\Om'^{\bul}\langle M\rangle$   
is a quasi-isomorphism of $A$-modules. 
\end{coro} 
\begin{proof} 
By the assumption ${\rm MC}(f)$ is acyclic. 
By (\ref{prop:acy}), ${\rm MC}(f)\langle M\rangle$ is acyclic. 
By (\ref{lemm:mcfgq}) we see that this means that $f_{\langle M\rangle} \col 
E\otimes_{B}\Om^{\bul}\langle M\rangle \lo E'\otimes_{B'}\Om'^{\bul}\langle M\rangle$   
is a quasi-isomorphism. 
\end{proof} 

\begin{rema}\label{rema:qi}
One does not know that ${\rm MC}(f)\langle M\rangle$ is isomorphic to 
${\rm MC}(f_{\langle M\rangle})$ as a complex a priori and 
it seems non-clear to me that one can obtain (\ref{coro:hac}) easily 
without using (\ref{lemm:mcfgq}) nor (\ref{prop:acy}) 
because I do not know whether 
the associated spectral sequence by the filtration of the columns  
of (\ref{cd:hiarfd}) is regular in the sense of \cite[(5.2.10)]{weib}. 
\end{rema}

\begin{prop-defi}\label{defi:emm}
Let $M$ be a projective $A$-module. 
Let ${\cal C}$ be a category whose objects are 
$(B,E,\Om^{\bul},d,\varphi)$'s 
such that $(E\otimes_B\Om^{\bul},d)$'s are bounded below  
and whose morphisms 
are defined as after {\rm (\ref{prop-defi:pdc})}. 
Let ${\rm Ho}({\cal C})$ be the localized category of ${\cal C}$ 
inverting quasi-isomorphisms 
in ${\cal C}$ (\cite[III \S2]{gelma}, cf.~\cite[p.~29 Remark]{hard}). 
Then the functor
$$\langle M\rangle \col {\cal C}\lo 
\{{\rm complexes}~{\rm of}~A{\textrm -}{\rm modules}\}$$ 
induces the following derived functor 
$$\langle M\rangle{}^L \col {\rm Ho}({\cal C})\lo 
D^+(\{{\rm complexes}~{\rm of}~A{\textrm -}{\rm modules}\}).$$
We call the functor $\langle M\rangle{}^L$ 
the {\it derived PD-Hirsch extension by} $M$. 
\end{prop-defi}
\begin{proof}
This follows from (\ref{coro:hac}). 
\end{proof}


\begin{prop-defi}\label{defi:edm}
Let $M$ be a projective $A$-module of countable rank. 
Let ${\cal C}$ be a category whose objects are 
$(B,\Om^{\bul},d,\varphi)$'s 
such that $(\Om^{\bul},d)$'s are bounded below  
and whose morphisms 
are defined as after (\ref{prop-defi:pdc}). 
Let ${\rm Ho}({\cal C})$ be the localized category of ${\cal C}$ 
inverting weakly equivalent isomorphisms 
in ${\cal C}$. Let ${\rm DGA}(A)$ be the category of dga's over $A$. 
Let ${\rm Ho}({\rm DGA}(A))$ be the localized category of ${\rm DGA}(A)$ 
inverting weakly equivalent isomorphisms in ${\rm DGA}(A)$. 
Then the functor
$$\langle M\rangle \col {\cal C}\lo {\rm DGA}(A)$$ 
induces the following derived functor 
$$\langle M\rangle{}^L \col {\rm Ho}({\cal C})\lo
{\rm Ho}({\rm DGA}(A)).$$
We call the functor $\langle M\rangle{}^L$ 
the {\it derived PD-Hirsch extension by} $M$. 
\end{prop-defi}
\begin{proof}
This follows from (\ref{coro:hac}). 
\end{proof}

\par 
Now we apply the theory above for the derived direct images of 
morphisms of ringed topoi and the Thom-Whitney derived direct images of 
morphisms of dga's over ${\mab Q}$ in ringed topoi which have enough points. 
In the following dga's are always assumed to be positively graded. 
\par 
Let $f\col ({\cal T}',{\cal A}')\lo ({\cal T},{\cal A})$ be a morphism of ringed topoi. 
Assume that ${\cal A}'=f^*({\cal A})$. 
First let us recall the single complex functor $s$ which has been denoted by ${\bf s}$ 
in \cite[\S2]{nh3} (cf.~\cite[(5.1.9)]{dh2}). 
\par 
For a positive integer $r$, 
let $({\cal T}_{t_1 \cdots t_r}, 
{\cal A}^{t_1 \cdots t_r})_{t_1,\ldots, t_r\in {\mab N}}$ 
be the constant $r$-semi-simplicial ringed topos 
defined by $({\cal T}, {\cal A})$: 
${\cal T}_{t_1 \cdots t_r}={\cal T}$,
${\cal A}^{t_1 \cdots t_r}={\cal A}$.
Let ${\cal L}$ be an object of the category 
${\rm C}({\cal A}^{\bul \cdots \bul})$ of complexes of 
${\cal A}^{\bul \cdots \bul}$-modules ($r$-points) 
(${\cal L}$ is an $r$-semi-simplicial complexes of ${\cal A}^{\bul \cdots \bul}$-modules). 
For simplicity of notation, set $\ul{t}:=(t_1, \ldots, t_r)$ 
and $\ul{t}_j:=t_1 +\cdots +t_j$ and $\ul{t}_0:=0$. 
We also set $\ul{\bul}:=\bul  \cdots \bul$ 
($r$-points). The object ${\cal L}$ defines an 
$(r+1)$-uple complex 
${\cal L}^{\ul{\bul}\bul}=
({\cal L}^{t_1 \cdots t_r \bul})_{t_1,\ldots, t_r\in {\mab N}}$ 
of ${\cal A}$-modules whose boundary 
morphisms will be fixed 
in (\ref{eqn:bdsignss}) below. 
Let $d_{\cal L} \col {\cal L}^{\ul{t}s} \lo  {\cal L}^{\ul{t},s+1}$ 
be the boundary morphism arising from the boundary 
morphism of the complex ${\cal L}$ and let  
$\del^i_{j} \col {\cal L}^{t_1 \cdots t_j \cdots t_r s} 
\lo {\cal L}^{t_1 \cdots t_{j-1},  t_j+1, 
t_{j+1}\cdots t_r s}$ 
$(1 \leq j \leq r, 0 \leq i \leq t_j+1)$ be 
a standard coface morphism.
Consider the single complex $s({\cal L})$ 
with the following boundary morphism 
(cf.~\cite[(5.1.9.1), (5.1.9.2)]{dh3}):
\begin{equation*}
s({\cal L})^n=
\bigoplus_{t_1+\cdots +t_r+s=n}
{\cal L}^{t_1 \cdots t_rs};
\tag{3.16.1}\label{eqn:bdsignss}
\end{equation*}
\begin{align*}
d(x^{\ul{t}s}) =&  
\sum_{i= 0}^{t_1+1}(-1)^{i}
\del^i_{1}(x^{\ul{t}s})
+(-1)^{\ul{t}_1}
\sum_{i= 0}^{t_2+1}(-1)^i\del^i_{2}
(x^{\ul{t}s}) +\cdots + \\ 
{} &  (-1)^{\ul{t}_{r-1}}
\sum_{i= 0}^{t_r+1}(-1)^i
\del^i_{r}(x^{\ul{t}s})+
(-1)^{\ul{t}_r}
d_{\cal L}(x^{\ul{t}s}) \quad (x^{\ul{t}s}\in {\cal L}^{\ul{t}s}). 
\end{align*} 
\par
For a morphism $g \col {\cal L}^{\ul{\bul}\bul} \lo {\cal N}^{\ul{\bul}\bul}$ in 
${\rm C}({\cal A}^{\ul{\bul}})=
{\rm C}({\cal A}^{\bul \cdots \bul})$, we define 
$s(g) \col s({\cal L}^{\ul{\bul}\bul}) \lo s({\cal N}^{\ul{\bul}\bul})$ as the naturally 
induced morphism by $g$ (without change of signs); 
$s$ gives a functor 
\begin{equation*}
s \col {\rm C}({\cal A}^{\ul{\bul}}) \lo 
{\rm C}({\cal A}). 
\tag{3.16.2}\label{eqn:bs}
\end{equation*}
The functor {\rm (\ref{eqn:bs})} 
induces the following functors$:$
\begin{equation*}
s \col D^{\star}({\cal A}^{\ul{\bul}}) \lo D^{\star}({\cal A})
\quad (\star=\text{$+$ {\rm or} {\rm nothing}}). 
\tag{3.16.3}
\end{equation*}
We have the morphism $f_{\ul{\bul}}\col 
({\cal T}'_{\ul{\bul}},{\cal A}'{}^{\ul{\bul}})\lo ({\cal T}_{\ul{\bul}},{\cal A}^{\ul{\bul}})$ 
of ringed topoi induced by $f$ and the following diagram is commutative: 
\begin{equation*}
\begin{CD}
D^{\star}({\cal A}'{}^{\ul{\bul}})
@>{s}>> D^{\star}({\cal A}')\\
@V{f_{\ul{\bul}*}}VV @VV{f_*}V \\
D^{\star}({\cal A}{}^{\ul{\bul}})
@>{s}>> D^{\star}({\cal A}).  
\end{CD}
\end{equation*} 
\par 
Now let ${\cal M}$ be a locally free ${\cal A}$-module of countable rank. 
Assume that the topos ${\cal T}'$ has enough points. 
Let ${\cal B}$, ${\cal E}$, $\Om^{\bul}$, $({\cal E}\otimes_{\cal B}\Om^{\bul},d)$ 
and $\varphi \col  f^*({\cal M})\lo {\rm Ker}(\Om^1\lo \Om^2)$ 
be analogous objects in 
$({\cal T}',{\cal A}')$ to $B$, $E$, $\Om^{\bul}$, $(E\otimes_B\Om^{\bul},d)$  and 
$\varphi \col {\cal M} \lo {\rm Ker}(\Om^1\lo \Om^2)$, respectively. 
For example,  ${\cal B}$ is a sheaf of ${\cal A}'$-algebras. 
More generally, we need the $r$-semi-cosimplicial version 
of ${\cal B}$, ${\cal E}$, $\Om^{\bul}$, $({\cal E}\otimes_{\cal B}\Om^{\bul},d)$ 
and $\varphi \col  f^*({\cal M})\lo {\rm Ker}(\Om^1\lo \Om^2)$. 
That is, we need 
(not necessarily constant) 
$r$-semi-cosimplicial objects  
${\cal B}^{\ul{\bul}}$, ${\cal E}^{\ul{\bul}}$, $\Om^{\ul{\bul}\bul}$, 
$({\cal E}^{\ul{\bul}}\otimes_{{\cal B}^{\ul{\bul}}}\Om^{\ul{\bul} \bul},d)$ 
and $\varphi^{\ul{\bul}} \col  f^*_{\ul{\bul}}({\cal M}^{\ul{\bul}})
\lo {\rm Ker}(\Om^{\ul{\bul}1}\lo \Om^{\ul{\bul}2})$.  
Here ${\cal M}^{\ul{\bul}}$ is 
the constant $r$-semi-cosimplicial ${\cal A}^{\ul{\bul}}$-modules defined by ${\cal M}$. 
Let ${\cal C}_{({\cal T}'_{\ul{\bul}},{\cal A}'{}^{\ul{\bul}})}$ 
be the analogous category to ${\cal C}$ for 
${\cal B}^{\ul{\bul}}$'s, ${\cal E}^{\ul{\bul}}$'s, 
$\Om^{\ul{\bul} \bul}$'s, $({\cal E}^{\ul{\bul}}\otimes_{{\cal B}^{\ul{\bul}}}
\Om^{\ul{\bul} \bul},d)$'s 
and $\varphi^{\ul{\bul}} \col 
f^*_{\ul{\bul}}({\cal M}^{\ul{\bul}})\lo 
{\rm Ker}(\Om^{\ul{\bul}1}\lo \Om^{\ul{\bul}2})$'s above. 
By considering the constant case for 
${\cal C}_{({\cal T}'_{\ul{\bul}},{\cal A}'{}^{\ul{\bul}})}$, 
we obtain the category ${\cal C}_{({\cal T}',{\cal A}')}$.  
We also obtain the analogous category ${\cal C}_{({\cal T},{\cal A})}$ to 
${\cal C}_{({\cal T}',{\cal A}')}$.
By using $s$ in (\ref{eqn:bdsignss}), 
we obtain the following functors 
\begin{equation*}
s \col {\cal C}_{({\cal T}_{\ul{\bul}},{\cal A}^{\ul{\bul}})} 
\lo {\cal C}_{({\cal T},{\cal A})}
\tag{3.16.4}
\end{equation*}
and 
\begin{equation*}
s \col {\cal C}_{({\cal T}'_{\ul{\bul}},{\cal A}'{}^{\ul{\bul}})} 
\lo {\cal C}_{({\cal T}',{\cal A}')}.  
\tag{3.16.5}
\end{equation*}
We also obtain the following functors
\begin{equation*}
f\col  {\cal C}_{({\cal T}',{\cal A}')} \lo {\cal C}_{({\cal T},{\cal A})}
\tag{3.16.6}
\end{equation*}
and 
\begin{equation*}
f_{\ul{\bul}*} \col {\cal C}_{({\cal T}'_{\ul{\bul}},{\cal A}'{}^{\ul{\bul}})} 
\lo {\cal C}_{({\cal T}_{\ul{\bul}},{\cal A}^{\ul{\bul}})}. 
\tag{3.16.7}
\end{equation*} 
Obviously the following diagram is commutative: 
\begin{equation*}
\begin{CD}
{\cal C}_{({\cal T}'_{\ul{\bul}},{\cal A}'{}^{\ul{\bul}})} 
@>{s}>> {\cal C}_{({\cal T}',{\cal A}')} \\
@V{f_{\ul{\bul}*}}VV @VV{f_*}V \\
{\cal C}_{({\cal T}_{\ul{\bul}},{\cal A}^{\ul{\bul}})} 
@>{s}>> {\cal C}_{({\cal T},{\cal A})}.  
\end{CD}
\tag{3.16.8}\label{cd:uafecui} 
\end{equation*} 

\par 
For each $i$, let $I^{\ul{\bul}\bul i}$ be 
the Godement resolution of 
${\cal E}^{\ul{\bul}}\otimes_{{\cal B}^{\ul{\bul}}}\Om^{\ul{\bul}i}$ 
of ${\cal A}'$-modules. 
Then we obtain a morphism 
$\varphi^0 \col f^*_{\ul{\bul}}({\cal M}^{\ul{\bul}})
\lo {\rm Ker}(I^{\ul{\bul}01}\lo I^{\ul{\bul} 02})$ 
fitting into the following commutative diagram: 
\begin{equation*} 
\begin{CD}
{\rm Ker}(\Om^{\ul{\bul}1}\lo \Om^{\ul{\bul}2})@>>> 
{\rm Ker}(I^{\ul{\bul}01}\lo I^{\ul{\bul} 02})\\
@A{\varphi }AA @AA{\varphi^0}A \\
f^*_{\ul{\bul}}({\cal M}^{\ul{\bul}})@= f^*_{\ul{\bul}}({\cal M}^{\ul{\bul}}).
\end{CD}
\end{equation*} 
By using this diagram and the analogous diagrams repeatedly, 
we have a natural morphism 
$\varphi^{\ul{\bul} \bul} \col f^*_{\ul{\bul}}({\cal M}^{\ul{\bul}})\lo 
{\rm Ker}(I^{\ul{\bul} \bul 1}\lo I^{\ul{\bul} \bul 2})$.  
This morphism induces the morphism 
$\varphi^{\ul{\bul} \bul} \col 
{\cal M}^{\ul{\bul}}\lo 
{\rm Ker}(f_{\ul{\bul}*}(I^{\ul{\bul} \bul 1})
\lo f_{\ul{\bul}*}(I^{\ul{\bul} \bul 2}))$.  
We have the following equality: 
\begin{align*} 
s(f_{\ul{\bul}*}(I^{\ul{\bul} \bul \bul}\langle f^*_{\ul{\bul}}({\cal M})\rangle))
=s(f_{\ul{\bul}*}(I^{\ul{\bul} \bul \bul}))\langle {\cal M}\rangle
\tag{3.16.8}\label{ali:fim} 
\end{align*}
in the category ${\cal C}_{({\cal T},{\cal A})}$.  
Indeed, 
\begin{align*}
s(f_{\ul{\bul}*}(I^{\ul{\bul} \bul \bul}\langle f^*_{\ul{\bul}}({\cal M})\rangle))^q
&=s(f_{\ul{\bul}*}(\Gam_{{\cal A}'}(f^*({\cal M})) \otimes_{{\cal A}'}
I^{\ul{\bul} \bul \bul}))^q\tag{3.16.9}\label{ali:fiiim} \\
&=\Gam_{\cal A}({\cal M}) \otimes_{\cal A}s(f_{\ul{\bul}*}(I^{\ul{\bul} \bul \bul}))^q
=s(f_{\ul{\bul}*}(I^{\ul{\bul} \bul \bul}))\langle {\cal M}\rangle^q
\end{align*}
and the equality above gives the equality 
$s(f_{\ul{\bul}*}(I^{\ul{\bul} \bul \bul}\langle f^*_{\ul{\bul}}({\cal M})\rangle))
=s(f_{\ul{\bul}*}(I^{\ul{\bul} \bul \bul}))\langle {\cal M}\rangle$ 
of complexes.  
Set 
\begin{align*} 
Rf_*({\cal E}^{\ul{\bul}}\otimes_{{\cal B}^{\ul{\bul}}}
\Om^{\ul{\bul} \bul})\langle 
{\cal M}\rangle^L
&:=f_*(s(I^{\ul{\bul}\bul}\langle f^*_{\ul{\bul}}({\cal M})\rangle))
= 
s(f_{\ul{\bul}*}(I^{\ul{\bul} \bul \bul}\langle f^*_{\ul{\bul}}({\cal M})\rangle))
\tag{3.16.10}\label{ali:fdim} \\
&
=s(f_{\ul{\bul}*}(I^{\ul{\bul} \bul \bul}))\langle {\cal M}\rangle
\end{align*} 
in ${\rm Ho}({\cal C}_{({\cal T},{\cal A})})$.

\begin{defi}
We call $Rf_*({\cal E}^{\ul{\bul}}\otimes_{{\cal B}^{\ul{\bul}}}
\Om^{\ul{\bul} \bul})\langle {\cal M}\rangle^L$ 
the {\it derived PD-Hirsch extension} of 
$Rf_*({\cal E}^{\ul{\bul}}\otimes_{{\cal B}^{\ul{\bul}}}
\Om^{\ul{\bul}\bul})$ by ${\cal M}=({\cal M},\varphi^{\ul{\bul}})$. 
By (\ref{coro:hac}), $Rf_*(?)\langle {\cal M}\rangle^L$ is the following functor 
\begin{align*} 
Rf_*(?)\langle {\cal M}\rangle^L\col 
{\rm Ho}({\cal C}_{({\cal T}'_{\ul{\bul}},{\cal A}'{}^{\ul{\bul}})})
\lo {\rm Ho}({\cal C}_{({\cal T},{\cal A})}).  
\tag{3.17.1}\label{ali:mld}
\end{align*}  
\end{defi}

\begin{rema}\label{rema:fa}
(1) Consider the case $r=0$. Let $J^{\bul}$ be an $f_*$-acyclic resolution of 
${\cal E}\otimes_{\cal B}\Om^{\bul}$. 
We should note that we 
cannot define the PD-Hirsch extensions of $J^{\bul}$ nor $f_*(J^{\bul})$ in general. 
\par 
(2) Note that there exists a natural functor 
from the category ${\rm Ho}({\cal C}_{({\cal T},{\cal A})})$ to 
the derived category $D^+({\cal A})$ of bounded below complexes of ${\cal A}$-modules. 
By abuse of notation, we also denote by 
$Rf_*(?)\langle {\cal M}\rangle^L$ the composite functor of the functor (\ref{ali:mld}) 
and this natural functor. 
If $J^{\ul{\bul}\bul}$ is an $f_{\ul{\bul}*}$-acyclic resolution of 
${\cal E}^{\ul{\bul}}\otimes_{{\cal B}^{\ul{\bul}}}\Om^{\ul{\bul}\bul}\langle 
f^*_{\ul{\bul}}({\cal M})\rangle$, 
then 
$Rf_*(?)\langle {\cal M}\rangle^L=f_*(s(J^{\ul{\bul}\bul}))$ in $D^+({\cal A})$ 
since $I^{\ul{\bul} \bul}\langle f^*_{\ul{\bul}}({\cal M})\rangle$ 
is $f_{\ul{\bul}*}$-acyclic by 
(\ref{ali:fiiim}) and (\ref{prop:acy}). 
\par
(3) We also have the category 
${\rm Ho}{\rm F}({\cal C}_{({\cal T}'_{\ul{\bul}},{\cal A}'{}^{\ul{\bul}})})$ of 
``objects of 
${\rm Ho}{\rm F}({\cal C}_{({\cal T}'_{\ul{\bul}},{\cal A}'{}^{\ul{\bul}})})$ with increasing filtrations'' and 
we have the following functor 
\begin{align*} 
Rf_*(?)\langle {\cal M}\rangle^L\col 
{\rm Ho}{\rm F}({\cal C}_{({\cal T}'_{\ul{\bul}},{\cal A}'{}^{\ul{\bul}})})
\lo {\rm Ho}{\rm F}({\cal C}_{({\cal T},{\cal A})}).  
\tag{3.18.1}\label{ali:mlfd}
\end{align*}  
by using the Godement resolution and giving a filtration
on ${\cal M}$ which is compatible with 
the multiplicative structure of $\Gam_{\cal A}({\cal M})$. 
\end{rema}

\par 
Let $\Om^{\ul{\bul} \bul}$ be a dga over ${\cal A}'$
($\Om^{\ul{\bul} \bul}$ is assumed to be positively graded.). 
Assume that ${\cal A}$ is a sheaf of ${\mab Q}$-algebras. 
Let $I^{\ul{\bul} \bul \bul}$ be the Godement resolution of $\Om^{\ul{\bul} \bul}$. 
Assume that we are given a morphism 
$\varphi^{\ul{\bul}} 
\col f^*_{\ul{\bul}}({\cal M})\lo {\rm Ker}(\Om^{\ul{\bul}1}\lo \Om^{\ul{\bul}2})$. 
Then we have a natural morphism 
$\varphi^{\ul{\bul}} \col f^*_{\ul{\bul}}({\cal M})\lo 
{\rm Ker}(I^{\ul{\bul}\bul 1}\lo I^{\ul{\bul}\bul 2})$.  
Let $s_{\rm TW}$ be the  single complex functor defined in \cite[p.~24]{nav}. 
Here we set $k:={\cal A}$ in [loc.~cit.]. 
Then 
$$({\us{r~{\rm times}}
{\underbrace{s_{\rm TW}\circ \cdots \circ s_{\rm TW}}}}\circ f_{\ul{\bul}*}
(I^{\ul{\bul} \bul \bul}
\langle f^*_{\ul{\bul}}({\cal M})\rangle))$$ 
is a dga over ${\cal A}$ (\cite[p.~27]{nav}).  
For simplicity, denote this dga by 
$$s_{\rm TW}\circ f_{\ul{\bul}*}(I^{\ul{\bul} \bul \bul}\langle f^*_{\ul{\bul}}({\cal M})\rangle).$$
Let ${\mathfrak A}$ be ${\cal A}'$ or ${\cal A}$. 
Let $A^{\geq 0}({\mathfrak A})$ be the category of positively graded dga's over ${\mathfrak A}$. 
Let ${\rm Ho}(A^{\geq 0}({\mathfrak A}))$ 
be the localized category of  
$A^{\geq 0}({\mathfrak A})$  
inverting the weakly equivalent filtered isomorphisms 
in $A^{\geq 0}({\mathfrak A})$ (cf.~\cite{nav}, \cite{gelma}).   
Let ${\rm A}^{\ul{\bul}, \geq 0}({\mathfrak A})$ 
be the category of positively-graded $r$-semi-cosimplicial dga's 
over ${\mathfrak A}$. 
The dga $s_{\rm TW}\circ f_{\ul{\bul}*}(I^{\ul{\bul} \bul \bul}
\langle f^*_{\ul{\bul}}({\cal M})\rangle)$ defines an object 
of ${\rm Ho}(A^{\geq 0}({\cal A}))$. 
We denote this object by $R_{\rm TW}f_*(\Om^{\ul{\bul} \bul})
\langle {\cal M}\rangle^L$.

\begin{defi}\label{defi-o}
Let ${\rm Ho}({\rm A}^{\ul{\bul}, \geq 0}({\mathfrak A}))$ 
be the localized category of ${\rm A}^{\ul{\bul}, \geq 0}({\mathfrak A})$  
inverting the weakly equivalent filtered isomorphisms.  
Set 
\begin{align*} 
R_{\rm TW}f_*(\Om^{\ul{\bul} \bul}):=
s_{\rm TW}\circ f_{\ul{\bul}*}(I^{\ul{\bul} \bul \bul})\in {\rm Ho}(A^{\geq 0}({\cal A})).  
\end{align*}
We call the object 
$R_{\rm TW}f_*(\Om^{\ul{\bul} \bul})\langle {\cal M}\rangle^L$ 
the {\it derived PD-Hirsch extension} of 
$R_{\rm TW}f_*(\Om^{\ul{\bul} \bul})$ 
by ${\cal M}:=({\cal M},\varphi)$. 
The functor $R_{\rm TW}f_*(?)\langle {\cal M}\rangle^L$ is a functor
\begin{align*} 
R_{\rm TW}f_*(?)\langle {\cal M}\rangle^L
\col {\rm Ho}({\rm A}^{\ul{\bul}, \geq 0}({\cal A}'))\lo {\rm Ho}({\rm A}^{\geq 0}({\cal A}))
\tag{3.19.1}\label{ali:aba}
\end{align*}
by (\ref{coro:hac}). 
\end{defi} 

We also have the category 
${\rm Ho}({\rm A}^{\ul{\bul}, \geq 0}{\rm F}({\mathfrak A}))$ of 
``objects of 
${\rm Ho}({\rm A}^{\ul{\bul}, \geq 0}({\mathfrak A}))$ with increasing filtrations'' and 
we have the following functor 
\begin{align*} 
Rf_*(?)\langle {\cal M}\rangle^L\col 
{\rm Ho}({\rm A}^{\ul{\bul}, \geq 0}{\rm F}({\cal A}'))
\lo {\rm Ho}({\rm A}^{\geq 0}{\rm F}({{\cal A}})). 
\tag{3.19.2}\label{ali:mlffd}
\end{align*}  
by using the Godement resolution and giving weights 
of local sections of ${\cal M}$ which is compatible with 
the multiplicative structure of $\Gam_{\cal A}({\cal M})$.



\section{PD-Hirsch extensions of log crystalline complexes}\label{sec:ldfc}  
Let $S$ be a family of log points on which a prime number $p$ is locally nilpotent. 
Let $M_S$ be the log structure of $S$.  
Let $(T,{\cal J},\del)$ be a fine log PD-scheme of $S$ 
such that ${\cal J}$ is quasi-coherent. 
Set $T_0:=T~{\rm mod}~{\cal J}$.  
Let $((T,{\cal J},\del),z)$ be a log PD-enlargement of $S$, where 
$z\col T_0\lo S$ is the structural morphism. 
Let $S_{\os{\circ}{T}_0}$, $S(T)$ and $S(T)^{\nat}$ be as in \S\ref{sec:snclv}.  
Let $\wt{t}$ be a local section of $M_{S(T)^{\nat}}$ which gives a local generator $t$ 
of $\ol{M}_{S(T)^{\nat}}:=M_{S(T)^{\nat}}/{\cal O}_T^*$. 
Then 
\begin{align*}
d\log t:=d\log \wt{t}\in \Om^1_{S(T)^{\nat}/\os{\circ}{T}}
\end{align*} 
is independent of the choice of $\wt{t}$; $d\log t$ depends only on $t$. 
Let $X/S$ be an SNCL scheme (\cite[(1.1)]{nb}). 
Set $X_{\os{\circ}{T}_0}:=X\times_SS_{\os{\circ}{T}_0}$. 
In this section we assume that there exists an immersion 
$X_{\os{\circ}{T}_0}\os{\sus}{\lo} \ol{\cal P}$ into a log smooth scheme 
over $\ol{S(T)^{\nat}}$ and 
we prove fundamental properties of 
the PD-Hirsch extensions of pre-weight filtered log crystalline complexes obtained by 
this immersion for a flat quasi-coherent crystal on $\os{\circ}{X}_{T_0}/\os{\circ}{T}$.  
\par 
Let ${\mathfrak D}(\ol{S(T)^{\nat}})$ be the log PD-envelope of 
the immersion $S_{\os{\circ}{T}_0}\os{\sus}{\lo} \ol{S(T)^{\nat}}$ over 
$(\os{\circ}{T},{\cal J},\del)$; 
${\cal O}_{{\mathfrak D}(\ol{S(T)^{\nat}})}$ is locally isomorphic 
to the PD-polynomial algebra ${\cal O}_T\langle \tau \rangle
=\bigoplus_{n=0}^{\infty}{\cal O}_T\tau^{[n]}$ in one variable.  
\par 
For the time being, we consider a more general case than the SNCL case. 
\par 
 Let $Y$ be a log smooth scheme over $S$. 
Set $Y_{\os{\circ}{T}_0}:=Y\times_SS_{\os{\circ}{T}_0}=
Y\times_{\os{\circ}{S}}\os{\circ}{T}_0$. 
Assume that there exists an immersion 
$Y_{\os{\circ}{T}_0}\os{\sus}{\lo} \ol{\cal Q}$ into 
a log smooth scheme over $\ol{S(T)^{\nat}}$. 
Set ${\cal Q}:=\ol{\cal Q}\times_{\ol{S(T)^{\nat}}}S(T)^{\nat}$.
Let $F$ be a flat quasi-coherent crystal of  
${\cal O}_{Y_{\os{\circ}{T}_0}/\os{\circ}{T}}$-modules. 
(In the case where $F={\cal O}_{Y_{\os{\circ}{T}_0}/\os{\circ}{T}}$, 
it is enough to assume that there exists an immersion 
$Y_{\os{\circ}{T}_0}\os{\sus}{\lo} {\cal Q}$ 
into a log smooth scheme over $S(T)^{\nat}$  
(we do not need $\ol{\cal Q}$ in this case).)  
Let $\ol{\cal Q}{}^{\rm ex}$ and ${\cal Q}^{\rm ex}$ 
be the exactifications of the immersions  
$Y\os{\sus}{\lo} \ol{\cal Q}$ and 
$Y\os{\sus}{\lo} {\cal Q}$, respectively. 
Let $\ol{\mathfrak E}$ be 
the log PD-envelope of 
the immersion $Y_{\os{\circ}{T}_0}\os{\sus}{\lo} \ol{\cal Q}$ over 
$(\os{\circ}{T},{\cal J},\del)$.  
\par 
Let $(\ol{\cal F},\ol{\nabla})$ 
be the quasi-coherent ${\cal O}_{\ol{\mathfrak E}}$-module  
with integrable connection corresponding to the log crystal $F$: 
\begin{equation*} 
\ol{\nabla}\col \ol{\cal F}\lo 
\ol{\cal F}\otimes_{{\cal O}_{\ol{\cal Q}{}^{\rm ex}}}
\Om^1_{\ol{\cal Q}{}^{\rm ex}/\os{\circ}{T}}.
\tag{4.0.1}\label{eqn:olfpc}
\end{equation*}  
In fact, we have the complex
$\ol{\cal F}\otimes_{{\cal O}_{\ol{\cal Q}{}^{\rm ex}}}
\Om^{\bul}_{\ol{\cal Q}{}^{\rm ex}/\os{\circ}{T}}$.  
Set ${\cal F}:=\ol{\cal F}
\otimes_{{\cal O}_{{\mathfrak D}(\ol{S(T)^{\nat}})}}{\cal O}_{S(T)^{\nat}}$. 
Then we see that $\ol{\nabla}$ induces the following connection:  
\begin{equation*} 
\nabla\col {\cal F}\lo 
{\cal F}\otimes_{{\cal O}_{{\cal Q}^{\rm ex}}}
\Om^1_{{\cal Q}^{\rm ex}/\os{\circ}{T}}. 
\tag{4.0.2}\label{eqn:olfmpc}
\end{equation*}  
Indeed, we obtain the connection $\nabla$ because 
\begin{align*}
df^{[i]}=f^{[i-1]}df\quad  (i\in {\mab Z}_{\geq 1})
\tag{4.0.3}\label{eqn:olzpc}
\end{align*} 
for a local section 
$f$ of the PD-ideal sheaf of ${\cal O}_{{\mathfrak D}(\ol{S(T)^{\nat}})}$. 
Here we have used the formula (\ref{ali:fwniwu}). 
When $i\geq 2$, $f^{[i-1]}df=0$ in ${\Om}^1_{{\cal Q}^{\rm ex}/\os{\circ}{T}}$; 
when $i=1$, $d\tau=\tau d\log \tau=0$ in 
${\Om}^1_{{\cal Q}^{\rm ex}/\os{\circ}{T}}$ 
for a local section $\tau$ 
of ${\cal O}_{{\mathfrak D}(\ol{S(T)^{\nat}})}$ such that 
${\cal O}_{{\mathfrak D}(\ol{S(T)^{\nat}})}\simeq 
{\cal O}_T\langle \tau \rangle$.   
In fact, we have the complex 
${\cal F}\otimes_{{\cal O}_{{\cal Q}^{\rm ex}}}
\Om^{\bul}_{{\cal Q}^{\rm ex}/\os{\circ}{T}}$.
\par 
By using natural surjective morphisms  
$\Om^{\bul}_{\ol{\cal Q}{}^{\rm ex}/\os{\circ}{T}}
\lo \Om^{\bul}_{\ol{\cal Q}{}^{\rm ex}/S(T)^{\nat}}$,  
$\Om^{\bul}_{{\cal Q}^{\rm ex}/\os{\circ}{T}}\lo 
\Om^{\bul}_{{\cal Q}^{\rm ex}/S(T)^{\nat}}$, 
we also have the following log de Rham complexes  
$\ol{\cal F}\otimes_{{\cal O}_{\ol{\cal Q}{}^{\rm ex}}}
\Om^{\bul}_{\ol{\cal Q}{}^{\rm ex}/S(T)^{\nat}}$ and  
${\cal F}\otimes_{{\cal O}_{{\cal Q}^{\rm ex}}}
\Om^{\bul}_{{\cal Q}^{\rm ex}/S(T)^{\nat}}$.

\begin{rema}
In \cite[(1.3.4)]{nb} we have proved that the natural morphism 
\begin{equation*} 
{\cal O}_{{\cal Q}^{\rm ex}}
\otimes_{{\cal O}_{\cal Q}} \Om^i_{{\cal Q}/\os{\circ}{T}}\lo 
\Om^i_{{\cal Q}^{\rm ex}/\os{\circ}{T}} 
\quad (i\in {\mab N})
\tag{4.1.1}\label{eqn:yxpsnpd}
\end{equation*}
is an isomorphism. 
Consequently the natural morphism 
\begin{equation*} 
{\cal O}_{\mathfrak E}
\otimes_{{\cal O}_{\cal Q}}
\Om^i_{{\cal Q}/T}
\lo {\cal O}_{\mathfrak E}
\otimes_{{\cal O}_{{\cal Q}^{\rm ex}}}
\Om^i_{{\cal Q}^{\rm ex}/T} 
\tag{4.1.2}\label{eqn:yxpdsnpd}
\end{equation*}
is an isomorphism. 
Hence we can replace ${\cal Q}^{\rm ex}$ by ${\cal Q}$ in the argument above. 
However ${\cal Q}$ will be not an appropriate object in the SNCL case in general; 
in this SNCL case we consider the pre-weight filtered crystalline complex 
in which case $\Om^{\bul}_{{\cal Q}/\os{\circ}{T}}$ is not a good complex in general.  
For this reason, we consider only ${\cal Q}^{\rm ex}$ henceforth. 
\end{rema}

\par 
Set 
\begin{align*} 
U_{S(T)^{\nat}}:={\cal O}_Tu, 
\end{align*} 
where $u$ is a formal variable over ${\cal O}_T$ 
and 
consider the following sheaf 
\begin{align*}
\Gam_{{\cal O}_T}(U_{S(T)^{\nat}}):=
\bigoplus_{n=0}^{\infty}{\cal O}_Tu^{[n]}
\end{align*} 
of PD-polynomial algebras over ${\cal O}_T$.  
We consider the following PD-Hirsch extension 
\begin{equation*}  
{\cal F}\otimes_{{\cal O}_{{\cal Q}^{{\rm ex}}}}
\Om^{\bul}_{{\cal Q}{}^{{\rm ex}}/\os{\circ}{T}}\langle u \rangle:= 
\Gam_{{\cal O}_T}(U_{S(T)^{\nat}})\otimes_{{\cal O}_T}
{\cal F}\otimes_{{\cal O}_{{\cal Q}^{{\rm ex}}}}
\Om^{\bul}_{{\cal Q}{}^{{\rm ex}}/\os{\circ}{T}} 
\end{equation*} 
of ${\cal F}\otimes_{{\cal O}_{{\cal Q}{}^{{\rm ex}}}}
\Om^{\bul}_{{\cal Q}{}^{{\rm ex}}/\os{\circ}{T}}$ 
by the PD-algebra 
$\Gam_{{\cal O}_T}(U_{S(T)^{\nat}})$ 
over ${\cal O}_T$ with respect to the morphism 
$U_{S(T)^{\nat}}  \owns u \lom d\log t \in 
{\rm Ker}(\Om^1_{{\cal Q}{}^{{\rm ex}}/\os{\circ}{T}}\lo 
\Om^2_{{\cal Q}{}^{{\rm ex}}/\os{\circ}{T}})$ 
(cf.~\cite[p.~1260]{kiha} for the case of log de Rham-Witt complexes).  
The connection 
\begin{equation*} 
\nabla \col   {\cal F}\otimes_{{\cal O}_{{\cal Q}^{{\rm ex}}}}
\Om^j_{{\cal Q}^{{\rm ex}}/\os{\circ}{T}}\langle u \rangle
\lo  {\cal F}\otimes_{{\cal O}_{{\cal Q}{}^{{\rm ex}}}}
\Om^{j+1}_{{\cal Q}{}^{{\rm ex}}/\os{\circ}{T}}\langle u \rangle
\quad (j\in {\mab Z}_{\geq 0})
\end{equation*}   
is, by definition, an ${\cal O}_T$-linear morphism defined by the following 
\begin{equation*} 
\nabla(u^{[i]}\otimes \om)=u^{[i-1]}d\log t\wedge \om +
u^{[i]}\otimes \nabla(\om) 
\quad (~\om \in {\cal F}
\otimes_{{\cal O}_{{\cal Q}{}^{\rm ex}}}
\Om^j_{{\cal Q}{}^{{\rm ex}}/\os{\circ}{T}}~).
\tag{4.1.3}\label{eqn:bdff}
\end{equation*}
By (\ref{prop-defi:pdc}) or an easy direct calculation, $\nabla^2=0$. 
\par 
The completion $\Gam_{{\cal O}_T}(U_{S(T)^{\nat}})^{\wedge}
=\prod_{m\in {\mab N}}{\cal O}_Tu^{[m]}$ of 
$\Gam_{{\cal O}_T}(U_{S(T)^{\nat}})$ (cf.~\cite{ey})
is isomorphic to the completed PD-polynomial algebra of one variable 
over ${\cal O}_T$. 
Then we have the following natural inclusion morphism
\begin{align*} 
\Gam_{{\cal O}_T}(U_{S(T)^{\nat}})  \os{\sus}{\lo} 
\Gam_{{\cal O}_T}(U_{S(T)^{\nat}})^{\wedge}.  
\tag{4.1.4}\label{ali:uui} 
\end{align*} 
We also consider the following PD-Hirsch extension by the 
completed PD-polynomial algebra of one variable over ${\cal O}_T$ 
with respect to the morphism ${\cal O}_T u 
\owns u\lom d\log t \in {\rm Ker}(\Om^1_{{\cal Q}{}^{{\rm ex}}/\os{\circ}{T}}\lo 
\Om^2_{{\cal Q}{}^{{\rm ex}}/\os{\circ}{T}})$:  
\begin{equation*}  
{\cal F}\otimes_{{\cal O}_{{\cal Q}^{{\rm ex}}}}
\Om^{\bul}_{{\cal Q}{}^{{\rm ex}}/\os{\circ}{T}}\langle \langle u \rangle  \rangle:=
\Gam_{{\cal O}_T}\langle U_{S(T)^{\nat}}\rangle ^{\wedge} \os{\wedge}{\otimes}_{{\cal O}_T}
({\cal F}\otimes_{{\cal O}_{{\cal Q}^{{\rm ex}}}}
\Om^{\bul}_{{\cal Q}{}^{{\rm ex}}/\os{\circ}{T}}) 
\tag{4.1.5}\label{eqn:uui} 
\end{equation*} 
of ${\cal F}\otimes_{{\cal O}_{{\cal Q}{}^{{\rm ex}}}}
\Om^{\bul}_{{\cal Q}{}^{{\rm ex}}/\os{\circ}{T}}$; 
the boundary morphism 
\begin{equation*} 
\nabla \col  {\cal F}
\otimes_{{\cal O}_{{\cal Q}^{{\rm ex}}}}
\Om^j_{{\cal Q}^{{\rm ex}}/\os{\circ}{T}}\langle \langle u\rangle \rangle
\lo 
{\cal F}{\otimes}_{{\cal O}_{{\cal Q}{}^{{\rm ex}}}}
\Om^{j+1}_{{\cal Q}{}^{{\rm ex}}/\os{\circ}{T}}\langle \langle u\rangle \rangle
\quad (j\in {\mab Z}_{\geq 0})
\tag{4.1.6}\label{eqn:bbadff}
\end{equation*}   
is, by definition, an ${\cal O}_T$-linear morphism defined by the following 
\begin{equation*} 
\nabla(\sum_{i=0}^{\inf}a_iu^{[i]}\otimes \om)=
\sum_{i=1}^{\inf}a_iu^{[i-1]}d\log t\wedge \om +
\sum_{i=0}^{\inf}a_iu^{[i]}\otimes \nabla(\om) 
\quad (~a_i\in {\cal O}_T,\om \in {\cal F}
\otimes_{{\cal O}_{{\cal Q}{}^{\rm ex}}}
\Om^j_{{\cal Q}{}^{{\rm ex}}/\os{\circ}{T}}~).
\tag{4.1.7}\label{eqn:badff}
\end{equation*}
We can easily check an equality $\nabla^2=0$ for $\nabla$ in (\ref{eqn:bbadff}).  
The complex 
$\Gam_{{\cal O}_T}\langle \langle U_{S(T)^{\nat}}\rangle \rangle \wh{\otimes}_{{\cal O}_T}
{\cal F}
\otimes_{{\cal O}_{{\cal Q}^{{\rm ex}}}}
\Om^{\bul}_{{\cal Q}^{{\rm ex}}/\os{\circ}{T}}$ 
is the single complex of the following double complex: 
\begin{equation*} 
\begin{CD}  
\cdots @>>> \cdots @>>> \cdots \\
@A{{\rm id}\otimes \nabla}AA @A{{\rm id}\otimes \nabla}AA 
@A{{\rm id}\otimes \nabla}AA \\
\Gam_{{\cal O}_T,2}(U_{S(T)^{\nat}})
\otimes_{{\cal O}_T}{\cal F}
@>{(?)'d\log t\wedge}>> \Gam_{{\cal O}_T,1}(U_{S(T)^{\nat}})\otimes_{{\cal O}_T}
{\cal F}\otimes_{{\cal O}_{{\cal Q}{}^{\rm ex}}}
\Om^1_{{\cal Q}^{\rm ex}/\os{\circ}{T}}  
@>{(?)'d\log t\wedge}>> {\cal F}
\otimes_{{\cal O}_{{\cal Q}{}^{\rm ex}}}
\Om^2_{{\cal Q}^{\rm ex}/\os{\circ}{T}} \\
@. @A{{\rm id}\otimes \nabla}AA @A{{\rm id}\otimes \nabla}AA \\
@. \Gam_{{\cal O}_T,1}(U_{S(T)^{\nat}})
\otimes_{{\cal O}_T}{\cal F}@>{(?)'d\log t\wedge}>>{\cal F}
\otimes_{{\cal O}_{{\cal Q}{}^{\rm ex}}}
\Om^1_{{\cal Q}^{\rm ex}/\os{\circ}{T}} 
\\
@. @. @A{\nabla}AA  \\
@. @. {\cal F} ,
\end{CD} 
\tag{4.1.8}\label{cd:hirfd}
\end{equation*} 
where the horizontal arrow $(?)'d\log t\wedge$ is defined by 
$u^{[i]}\otimes f\otimes \om \lom u^{[i-1]}\otimes f\otimes (d\log t\wedge \om)$. 
Note that the columns and the rows of (\ref{cd:hirfd}) correspond to 
the rows and the columns of the double complexes in \cite[p.~1260]{kiha} and 
\cite[(3.19)]{ey}, respectively; we use the general principle about rows and columns 
of double complexes stated in \cite[\S2]{nh3}.  
\par 
The double complex (\ref{cd:hirfd}) is 
naturally augmented to   
${\cal F}
\otimes_{{\cal O}_{{\cal Q}{}^{\rm ex}}}
\Om^{\bul}_{{\cal Q}^{\rm ex}/S(T)^{\nat}}$ since 
$d\log t=0$ in $\Om^{\bul}_{{\cal Q}^{\rm ex}/S(T)^{\nat}}$. 
Hence we have the following morphism of complexes: 
\begin{equation*}  
{\cal F}\otimes_{{\cal O}_{{\cal Q}{}^{\rm ex}}}
\Om^{\bul}_{{\cal Q}^{\rm ex}/\os{\circ}{T}}\langle \langle u\rangle \rangle\lo 
{\cal F}
\otimes_{{\cal O}_{{\cal Q}{}^{\rm ex}}}
\Om^{\bul}_{{\cal Q}^{\rm ex}/S(T)^{\nat}}.
\tag{4.1.9}\label{ali:uaafui} 
\end{equation*}   
The morphism (\ref{ali:uui}) induces the following inclusion morphism  
\begin{align*}  
{\cal F}
\otimes_{{\cal O}_{{\cal Q}{}^{\rm ex}}}
\Om^{\bul}_{{\cal Q}^{\rm ex}/\os{\circ}{T}}\langle u \rangle \os{\sus}{\lo} 
{\cal F}
\otimes_{{\cal O}_{{\cal Q}{}^{\rm ex}}}
\Om^{\bul}_{{\cal Q}^{\rm ex}/\os{\circ}{T}}\langle \langle u\rangle \rangle. 
\tag{4.1.10}\label{ali:ufui} 
\end{align*} 
By composing (\ref{ali:ufui}) with (\ref{ali:uaafui}), we have the following morphism 
\begin{equation*} 
{\cal F}
\otimes_{{\cal O}_{{\cal Q}{}^{\rm ex}}}
\Om^{\bul}_{{\cal Q}^{\rm ex}/\os{\circ}{T}}\langle u\rangle\lo
{\cal F}
\otimes_{{\cal O}_{{\cal Q}{}^{\rm ex}}}
\Om^{\bul}_{{\cal Q}^{\rm ex}/S(T)^{\nat}}.
\tag{4.1.11}\label{ali:uafcui} 
\end{equation*}  

We would like to prove that the morphism {\rm (\ref{ali:uaafui})} is a quasi-isomorphism. 
In order to prove this, we need the following key lemma, which is a quite nontrivial work to 
formulate it. 

\begin{lemm}\label{lemm:cstite}
Let $C$ be a site. Let ${\cal B}$ be a sheaf of commutative rings with unit elments on $C$. 
Let ${\cal F}^{\bul \bul}=(\prod_{i\in {\mab N},j\in {\mab Z}}{\cal F}^{-i,j},\prod_{i\in {\mab N},j\in {\mab Z}}d^{-i,j}, 
\prod_{i\in {\mab N},j\in {\mab Z}}d'^{-i,j})$ 
be a double complex of ${\cal B}$-modules on $C$, 
where $d^{-i,j}\col {\cal F}^{-i,j}\lo {\cal F}^{-i+1,j}$ and 
$d'{}^{-i,j}\col {\cal F}^{-i,j}\lo {\cal F}^{-i,j+1}$ are morphism of ${\cal B}$-modules on $C$ 
such that $d^{-i+1j}\circ d^{-i,j}=0$, $d'^{-i,j+1}\circ d'{}^{-i,j}=0$ and 
$d'{}^{-i+1,j}\circ d^{-i,j}+d^{-i,j+1}\circ d'{}^{-i,j}=0$.  
Let ${\cal G}^{\bul}$ be a complex of ${\cal B}$-modules on $C$. 
Let $\eps^j \col {\cal F}^{0j}\lo {\cal G}^j$ $(j\in {\mab Z})$ be a morphism of ${\cal B}$-modules on $C$ 
such that $\eps^{\bul}:=\{\eps^j\}_{j\in {\mab Z}}$ induces morphisms
${\cal F}^{0\bul}\lo {\cal G}^{\bul}$ and ${\cal F}^{-1,j}\os{d^{-1,j}}{\lo} {\cal F}^{0j}\os{\eps^j}{\lo} {\cal G}^j$ 
$(j\in {\mab Z})$ of complexes of ${\cal B}$-modules.  
Let $\{V_{\mu}\}_{\mu \in M}$ be a basis of $C$. That is, for any covering family 
$\{U_{\lam}\}$ of any object $U$ of $C$ and for any $\lam$, there exists a covering family 
$\{V_{\mu}\}_{\mu \in M'}$ $(M'\subset M)$ of $U_{\lam}$.  
Assume that the following sequence 
\begin{align*} 
\cdots \lo \Gam(V_{\mu},{\cal F}^{-i,j})\lo \Gam(V_{\mu},{\cal F}^{-i+1,j})\lo \cdots \lo 
\Gam(V_{\mu},{\cal F}^{0j})\lo \Gam(V_{\mu},{\cal G}^j)\lo 0
\tag{4.2.1}\label{ali:rmafxu} 
\end{align*} 
is exact for any $\mu \in M$ and any $j\in {\mab Z}$. 
Then the natural morphism $s({\cal F}^{\bul \bul})\lo {\cal G}^{\bul}$ 
is a quasi-isomorphism. 
\end{lemm}
\begin{proof} 
(The proof of (\ref{lemm:cstite}) is not obvious 
because we consider the product of infinitely many sheaves; 
taking the product of infinitely many sheaves and taking stalks of sheaves 
does not necessarily commute as pointed out by our referee.)
We prove that the induced morphism 
${\cal H}^q(s({\cal F}^{\bul \bul}))\lo {\cal H}^q({\cal G}^{\bul})$ $(q\in {\mab N})$ 
is an isomorphism. 
This is a local problem. 
The degree $q$-part of the complex $s({\cal F}^{\bul \bul})$ is 
$\prod_{i+j=q}{\cal F}^{ij}$. 
For any object $U$ of $C$, 
\begin{align*} 
\Gam(U,\prod_{i+j=q}{\cal F}^{ij})=
\prod_{i+j=q}\Gam(U,{\cal F}^{ij}). 
\tag{4.2.2}\label{ali:rabsxu} 
\end{align*}  
By using the exact sequence (\ref{ali:rmafxu}), we see that
the following morphism 
\begin{align*} 
{\rm Ker}(\Gam(V_{\mu},\prod_{i+j=q}{\cal F}^{ij})\lo 
\Gam(V_{\mu},\prod_{i+j=q+1}{\cal F}^{ij}))\lo 
{\rm Ker}(\Gam(V_{\mu},{\cal G}^q)\lo \Gam(V_{\mu},{\cal G}^{q+1}))
\end{align*} 
induced by $\eps^{\bul}$ is surjective. 
Since $\{V_{\mu}\}_{\mu \in M}$ is a basis of $C$, this means that 
the morphism 
\begin{align*} 
{\rm Ker}(\prod_{i+j=q}{\cal F}^{ij}\lo \prod_{i+j=q+1}{\cal F}^{ij})
\lo 
{\rm Ker}({\cal G}^q\lo {\cal G}^{q+1})
\end{align*} 
is surjective. Hence 
the morphism ${\cal H}^q(s({\cal F}^{\bul \bul}))\lo {\cal H}^q({\cal G}^{\bul})$ 
is surjective. 
\par 
The rest we have to prove is the injective of the morphism 
${\cal H}^q(s({\cal F}^{\bul \bul}))\lo {\cal H}^q({\cal G}^{\bul})$.  
Let $\om'$ be a local section of ${\cal H}^q(s({\cal F}^{\bul \bul}))$. 
We may assume that there exist $V_{\mu}$ and  
a section $\om
=(\om_{ij})_{i+j=q}$ $(\om_{ij}\in \Gam(V_{\mu},{\cal F}^{ij}))$ 
of 
${\rm Ker}(\Gam(V_{\mu},\prod_{i+j=q}{\cal F}^{ij})\lo 
\Gam(V_{\mu},\prod_{i+j=q+1}{\cal F}^{ij}))$ 
whose image in ${\cal H}^q(s({\cal F}^{\bul \bul}))$ 
is equal to $\om'$. 
Assume that the image of $\om'$ in ${\cal H}^q({\cal G}^{\bul})$ is zero. 
Then we may assume that 
the image of $\om$ 
in 
$\Gam(V_{\mu},{\cal H}^q({\cal F}^{\bul}))$ is zero. 
Then there exists a covering $\{V_{\mu'}\}_{\mu'}$ of $V_{\mu}$ such that 
the image of $\om_{0q}\vert_{V_{\mu'}}$ in 
$\Gam(V_{\mu'},{\cal G}^q)$ is the image of an element  of $\Gam(V_{\mu'},{\cal G}^{q-1})$. 
Because the morphism 
$\Gam(V_{\mu'},{\cal F}^{0,q-1})\lo \Gam(V_{\mu'},{\cal G}^{q-1})$ 
is surjective by the exactness of (\ref{ali:rmafxu}), 
we can take an element $\eta_{0,q-1}$ whose image in 
$\Gam(V_{\mu'},{\cal G}^{q-1})$ is the element above. 
By the exactness of (\ref{ali:rmafxu}), we see that 
there exists an element $\eta_{-1,q}$ of $\Gam(V_{\mu'},{\cal F}^{-1,q})$ 
such that 
$\om_{0q}\vert_{V_{\mu'}}=d'{}^{0,q-1}(\eta_{0,q-1})+d^{-1,q}(\eta_{-1,q})$. 
(It is not necessary to take an object of $C$ which is smaller 
than $V_{\mu'}$ to find $\eta_{-1,q}$. This is the point of this proof.) 
Because $d^{-1,q+1}(\om_{-1,q+1})+d'{}^{0q}(\om_{0q})=0$, 
\begin{align*} 
d^{-1,q+1}(\om_{-1,q+1}\vert_{V_{\mu'}})
&=-d'{}^{0q}(\om_{0q}\vert_{V_{\mu'}})
=-d'{}^{0q}(d'{}^{0,q-1}(\eta_{0,q-1})+d^{-1,q}(\eta_{-1,q}))\\
&=d^{-1,q+1}(d'^{-1,q}(\eta_{-1,q})), 
\end{align*} 
there exists an element 
$\eta_{-2.q+1}$ of 
$\Gam(V_{\mu'},{\cal F}^{-2,q+1})$ such that 
$\om_{-1,q+1}\vert_{V_{\mu'}}=d^{-1,q}(\eta_{-1,q})+d^{-2,q+1}(\eta_{-2,q+1})$ 
by the exactness of (\ref{ali:rmafxu}) again. 
Continuing this argument repeatedly, 
we obtain an element 
$(\eta_{ij})_{i+j=q-1}$ of 
$\Gam(V_{\mu'},\{{\cal F}^{ij}\}_{i+j=q-1})$
whose coboundary is equal to $\om\vert_{V_{\mu'}}$. 
Hence the class $[\om\vert_{V_{\mu'}}]$ in 
$H^q(\Gam(V_{\mu'},s({\cal F}^{\bul \bul})))$ is zero. 
Because we have a natural composite morphism 
\begin{align*} 
{\rm Ker}(\Gam(V_{\mu},\prod_{i+j=q}{\cal F}^{ij})\lo 
\Gam(V_{\mu},\prod_{i+j=q+1}{\cal F}^{ij}))
\lo 
H^q(\Gam(V_{\mu'},s({\cal F}^{\bul \bul})))\lo \Gam(V_{\mu'},{\cal H}^q(s({\cal F}^{\bul \bul})))
\end{align*} 
and because the image of $\om\vert_{V_{\mu'}}$ is equal to $\om'\vert_{V_{\mu'}}$,  
$\om'\vert_{V_{\mu'}}=0$. Hence $\om'=0$. 
\end{proof} 

\par 
Now we can prove the following: 

\begin{theo}\label{theo:saih}  
The morphism {\rm (\ref{ali:uaafui})} is a quasi-isomorphism.  
\end{theo} 
\begin{proof} 
By (\ref{lemm:cstite}) we have only to prove that  
\begin{align*} 
0&\lo    
\Gam(W,{\cal F})
\os{d\log t \wedge}{\lo}   
\Gam(W,{\cal F}\otimes_{{\cal O}_{{\cal Q}{}^{\rm ex}}}
\Om^1_{{\cal Q}^{\rm ex}/\os{\circ}{T}})
\os{d\log t \wedge}{\lo}
\cdots 
\os{d\log t \wedge}{\lo}
\Gam(W,{\cal F}\otimes_{{\cal O}_{{\cal Q}{}^{\rm ex}}}
\Om^{j-2}_{{\cal Q}^{\rm ex}/\os{\circ}{T}}) \tag{4.3.1}\label{ali:gumt}\\
&\os{d\log t \wedge}{\lo}  
\Gam(W,{\cal F}\otimes_{{\cal O}_{{\cal Q}{}^{\rm ex}}}
\Om^{j-1}_{{\cal Q}^{\rm ex}/\os{\circ}{T}})
\os{d\log t \wedge}{\lo}  
\Gam(W,{\cal F}\otimes_{{\cal O}_{{\cal Q}{}^{\rm ex}}}
\Om^j_{{\cal Q}^{\rm ex}/S(T)^{\nat}})  \lo 0.   
\end{align*} 
is exact for a log affine open subscheme $W$ of ${\mathfrak E}$.
In \cite[(1.7.26)]{nb} we have proved that  
the following complex is exact for an integer $j$: 
\begin{align*} 
0&\lo    
{\cal F}
\os{d\log t \wedge}{\lo}   
{\cal F}\otimes_{{\cal O}_{{\cal Q}{}^{\rm ex}}}
\Om^1_{{\cal Q}^{\rm ex}/\os{\circ}{T}} 
\os{d\log t \wedge}{\lo}
\cdots 
\os{d\log t \wedge}{\lo}
{\cal F}\otimes_{{\cal O}_{{\cal Q}{}^{\rm ex}}}
\Om^{j-2}_{{\cal Q}^{\rm ex}/\os{\circ}{T}} \tag{4.3.2}\label{ali:gultpt}\\
&\os{d\log t \wedge}{\lo}  
{\cal F}\otimes_{{\cal O}_{{\cal Q}{}^{\rm ex}}}
\Om^{j-1}_{{\cal Q}^{\rm ex}/\os{\circ}{T}}  
\os{d\log t \wedge}{\lo}  {\cal F}
\otimes_{{\cal O}_{{\cal Q}{}^{\rm ex}}}
\Om^j_{{\cal Q}^{\rm ex}/S(T)^{\nat}}  \lo 0.   
\end{align*} 
Because $H^q(Z,{\cal G})=0$ $(q\geq 1)$ for an affine scheme $Z$ 
and a quasi-coherent ${\cal O}_Z$-module ${\cal G}$ (\cite[(1.3.1)]{ega3}),  because 
${\cal F}\otimes_{{\cal O}_{{\cal Q}{}^{\rm ex}}}
\Om^i_{{\cal Q}^{\rm ex}/\os{\circ}{T}}$ $(i\in {\mab N})$ is a quasi-coherent 
${\cal O}_{\mathfrak E}$-module and because the morphism $d\log t \wedge$ is 
an ${\cal O}_{\mathfrak E}$-linear morphism, 
the complex (\ref{ali:gumt}) is exact. 
\end{proof} 

\begin{rema}\label{rema:kihainc}
In \cite[Lemma 6]{kiha} Kim and Hain 
have claimed that the natural morphism 
\begin{equation*} 
{\cal W}\wt{\Om}^{\bul}_Y\langle u \rangle 
\lo 
{\cal W}\Om^{\bul}_Y
\tag{4.4.1}\label{eqn:wyuy}
\end{equation*} 
is a quasi-isomorphism for a log smooth scheme $Y$ of Cartier type 
over the log point of a perfect field of characteristic $p>0$. 
However,  in \cite[(3.36) (2)]{ey} Ertl and Yamada have pointed out that 
there is a non-small gap of the proof of Kim and Hain's claim above.  
In \cite[Lemma 6]{kiha} Kim and Hain do not either prove that 
the natural morphism 
\begin{equation*} 
{\cal W}\wt{\Om}^{\bul}_Y\langle \langle u \rangle \rangle
\lo 
{\cal W}\Om^{\bul}_Y
\tag{4.4.2}\label{eqn:wcyuy}
\end{equation*} 
is a quasi-isomorphism  because 
the operations taking the direct product of sheaves and taking stalks  
do not commute in general as pointed out by our referee; 
we do not know whether we can use \cite[(2.7.3) 2]{weib} immediately. 
In \S\ref{sec:ofc} we shall prove that the morphisms (\ref{eqn:wyuy}) and 
(\ref{eqn:wcyuy}) are isomorphisms. 
Note also that the difference between the claim that (\ref{eqn:wyuy}) 
is a quasi-isomorphism 
and the claim that (\ref{eqn:wcyuy}) is a quasi-isomorphism is big.  
\end{rema}



\begin{defi}\label{defi:wl}
We say that the crystal $F$ has {\it no poles} 
if the morphism 
$$\ol{\nabla} \col 
\ol{\cal F}\lo \ol{\cal F}\otimes_{{\cal O}_{\ol{\cal Q}{}^{\rm ex}}}\Om^1_{\ol{\cal Q}{}^{\rm ex}
/\os{\circ}{T}}$$  
factors through the natural morphism 
$\ol{\cal F}\otimes_{{\cal O}_{\ol{\cal Q}{}^{\rm ex}}}\Om^1_{\os{\circ}{\ol{\cal Q}}{}^{\rm ex}
/\os{\circ}{T}}\lo 
\ol{\cal F}\otimes_{{\cal O}_{\ol{\cal Q}{}^{\rm ex}}}\Om^1_{\ol{\cal Q}{}^{\rm ex}
/\os{\circ}{T}}$.  
\end{defi}

\begin{exem}
Let $\os{\circ}{Y}_{T_0}$ be the underlying scheme of $Y_{T_0}$. 
Let $F'$ be a flat crystal of ${\cal O}_{\os{\circ}{Y}_{T_0}/\os{\circ}{T}}$-modules. 
Let 
$$\eps \col ((Y_{\os{\circ}{T_0}}/\os{\circ}{T})_{\rm crys},{\cal O}_{Y_{\os{\circ}{T_0}}/\os{\circ}{T}})\lo 
((\os{\circ}{Y}_{T_0}/\os{\circ}{T}_0)_{\rm crys},{\cal O}_{\os{\circ}{Y}_{T_0}/\os{\circ}{T}})$$ 
be the induced morphism of ringed topoi by the morphism 
$Y_{\os{\circ}{T}_0}\lo \os{\circ}{Y}_{T_0}$ forgetting the log structures of
$Y_{\os{\circ}{T}_0}$ over $\os{\circ}{T}$. 
If $F=\eps^*(F')$, then $F$ has no poles 
if 
the underlying (formal) scheme $(\ol{\cal Q})^{\circ}$ is (formally) smooth over $\os{\circ}{T}$. 
Indeed, $F'$ corresponds to 
a quasi-coherent ${\cal O}_{\ol{\mathfrak E}}$-module 
$\ol{\cal F}$ with integrable connection 
$$\ol{\nabla}\col \ol{\cal F}\lo \ol{\cal F}\otimes_{{\cal O}_{\ol{\cal Q}{}^{\rm ex}}}
\Om^1_{(\ol{\cal Q}{}^{\rm ex})^{\circ}/\os{\circ}{T}}.$$ 
In particular, ${\cal O}_{Y_{\os{\circ}{T}_0}/\os{\circ}{T}}$ 
has no poles.  
\end{exem}

\begin{defi}[{\bf cf.~\cite[(0.4), (0.5)]{kn}}]\label{defi:tef}
If, locally on $Y_{\os{\circ}{T}_0}$, there exists a finite increasing filtration 
$\{F_i\}_{i\in {\mab Z}}$ of 
flat quasi-coherent ${\cal O}_{Y_{\os{\circ}{T}_0}/\os{\circ}{T}}$-modules on $F$ 
such that ${\rm gr}_iF:=F_i/F_{i-1}$ $(\forall i\in {\mab Z})$
is a flat quasi-coherent ${\cal O}_{Y_{\os{\circ}{T}_0}/\os{\circ}{T}}$-module 
and such that ${\rm gr}_iF$  has no poles, 
then we say that $F$ is 
{\it a flat locally nilpotent quasi-coherent 
${\cal O}_{Y_{\os{\circ}{T}_0}/\os{\circ}{T}}$-modules}.  
\end{defi} 


In the following we always assume that $F$ is 
a flat locally nilpotent quasi-coherent 
${\cal O}_{Y_{\os{\circ}{T}_0}/\os{\circ}{T}}$-modules.  
The following theorem is a first key theorem in this book.

\begin{theo}\label{theo:qii}
The morphism {\rm (\ref{ali:uafcui})} 
is a quasi-isomorphism.  
\end{theo}
\begin{proof} 
(The following proof has come to me by 
recalling theory of log de Rham-Witt complexes
(see \cite[(2.2.3) (1)]{nb}): (\ref{ali:dayn}) below is a key point for the proof.) 
This is a local problem. 
In \cite[(1.7.22.1)]{nb} we have proved that 
the following sequence 
\begin{align*} 
0  \lo {\cal O}_{\mathfrak E}
\otimes_{{\cal O}_{{{\cal Q}{}^{\rm ex}}}}
{\Om}^{\bul}_{{{\cal Q}{}^{\rm ex}}/S(T)^{\nat}}[-1] \os{d\log t\wedge }{\lo} 
{\cal O}_{\mathfrak E}
\otimes_{{\cal O}_{{{\cal Q}^{\rm ex}}}}
{\Om}^{\bul}_{{{\cal Q}^{\rm ex}}/\os{\circ}{T}}  
\lo {\cal O}_{\mathfrak E}\otimes_{{\cal O}_{{{\cal Q}^{\rm ex}}}}
{\Om}^{\bul}_{{{\cal Q}^{\rm ex}}/S(T)^{\nat}} \lo 0
\tag{4.8.1}\label{ali:gsflaxd}
\end{align*} 
is exact. 
In fact, the following sequence is exact: 
\begin{align*} 
0  \lo {\cal F}
\otimes_{{\cal O}_{{\cal Q}^{\rm ex}}}
{\Om}^{\bul}_{{\cal Q}^{\rm ex}/S(T)^{\nat}}[-1] \os{d\log t\wedge }{\lo} 
{\cal F}\otimes_{{\cal O}_{{\cal Q}^{\rm ex}}}
{\Om}^{\bul}_{{\cal Q}^{\rm ex}/\os{\circ}{T}}  
\lo {\cal F}\otimes_{{\cal O}_{{\cal Q}^{\rm ex}}}
{\Om}^{\bul}_{{\cal Q}^{\rm ex}/S(T)^{\nat}} \lo 0. 
\tag{4.8.2}\label{ali:agbxd}
\end{align*} 
Here note that the morphism 
\begin{align*} 
d\log t\wedge \col {\cal F}
\otimes_{{\cal O}_{{\cal Q}^{\rm ex}}}
\Om^{\bul}_{{\cal Q}^{\rm ex}/S(T)^{\nat}}[-1]\lo 
{\cal F}\otimes_{{\cal O}_{{{\cal Q}^{\rm ex}}}}
\Om^{\bul}_{{\cal Q}^{\rm ex}/\os{\circ}{T}}
\end{align*} 
is indeed a morphism of complexes. 
(The exact sequence in \cite[(3.6)]{hk} is mistaken because 
$\wedge d\log t$ in [loc.~cit.] is not a morphism of complexes.)
Assume that there exists a flat quasi-coherent 
${\cal O}_{Y_{\os{\circ}{T}_0}/\os{\circ}{T}}$-submodules $F'$ of $F$ 
such that ${\rm gr}F:=F/F'$ has no poles 
and such that ${\rm gr}F$ is a flat quasi-coherent 
${\cal O}_{Y_{\os{\circ}{T}_0}/\os{\circ}{T}}$-module. 
Let $({\cal F}',\nabla')$ be the integrable connection corresponding to $F'$. 
The crystals $F'$ and ${\rm gr}F:=F/F'$ give us the log de Rham complexes 
${\cal F}'\otimes_{{\cal O}_{{\cal Q}^{\rm ex}}}
{\Om}^{\bul}_{{\cal Q}^{\rm ex}/\os{\circ}{T}}$ and 
${\cal F}/{\cal F}'\otimes_{{\cal O}_{{\cal Q}^{\rm ex}}}
{\Om}^{\bul}_{{\cal Q}^{\rm ex}/\os{\circ}{T}}$, respectively. 
Set ${\rm gr}{\cal F}:={\cal F}/{\cal F}'$. 
Let $\iota \col {\cal F}'\otimes_{{\cal O}_{{\cal Q}^{\rm ex}}}
{\Om}^{\bul}_{{\cal Q}^{\rm ex}/\os{\circ}{T}} \os{\sus}{\lo} 
{\cal F}\otimes_{{\cal O}_{{\cal Q}^{\rm ex}}}
{\Om}^{\bul}_{{\cal Q}^{\rm ex}/\os{\circ}{T}}$ be the natural inclusion  
and let ${\rm MC}(\iota)$ be the mapping cone of $\iota$. 
Then we have the natural quasi-isomorphism 
${\rm MC}(\iota)\os{\sim}{\lo} {\rm gr}{\cal F}
\otimes_{{\cal O}_{{\cal Q}^{\rm ex}}}
{\Om}^{\bul}_{{\cal Q}^{\rm ex}/\os{\circ}{T}}$. 
By (\ref{coro:hac}) we have a quasi-isomorphism 
\begin{align*} 
{\rm MC}(\iota)\langle u\rangle \os{\sim}{\lo} 
{\rm gr}{\cal F}\otimes_{{\cal O}_{{\cal Q}^{\rm ex}}}
{\Om}^{\bul}_{{\cal Q}^{\rm ex}/\os{\circ}{T}}\langle u\rangle. 
\end{align*} 
Hence, by (\ref{lemm:mcfgq}), we obtain the following isomorphism
\begin{align*} 
{\rm MC}(\iota_{\langle U_{S(T)^{\nat}}\rangle})  \os{\sim}{\lo} 
{\rm gr}{\cal F}\otimes_{{\cal O}_{{\cal Q}^{\rm ex}}}
{\Om}^{\bul}_{{\cal Q}^{\rm ex}/\os{\circ}{T}}\langle u \rangle. 
\end{align*} 
This means that the upper sequence of the following commutative diagram 
is exact: 
\begin{equation*}
\begin{CD}
0@>>> {\cal F}_0\otimes_{{\cal O}_{{\cal Q}^{\rm ex}}}
{\Om}^{\bul}_{{\cal Q}^{\rm ex}/\os{\circ}{T}}\langle u\rangle@>>>
{\cal F}\otimes_{{\cal O}_{{\cal Q}^{\rm ex}}}
{\Om}^{\bul}_{{\cal Q}^{\rm ex}/\os{\circ}{T}}\langle u\rangle\\
@. @VVV @VVV \\
0@>>> {\cal F}_0\otimes_{{\cal O}_{{\cal Q}^{\rm ex}}}
{\Om}^{\bul}_{{\cal Q}^{\rm ex}/S(T)^{\nat}}@>>>
{\cal F}\otimes_{{\cal O}_{{\cal Q}^{\rm ex}}}
{\Om}^{\bul}_{{\cal Q}^{\rm ex}/S(T)^{\nat}}
\end{CD}
\tag{4.8.3}\label{cd:uafeacui} 
\end{equation*} 
\begin{equation*}
\begin{CD}
@>>>
{\rm gr}{\cal F}\otimes_{{\cal O}_{{\cal Q}^{\rm ex}}}
{\Om}^{\bul}_{{\cal Q}^{\rm ex}/\os{\circ}{T}}\langle u\rangle\lo 0\\
@. @VVV \\
@>>>
{\rm gr}{\cal F}\otimes_{{\cal O}_{{\cal Q}^{\rm ex}}}
{\Om}^{\bul}_{{\cal Q}^{\rm ex}/S(T)^{\nat}}\lo 0. 
\end{CD} 
\end{equation*}
Assume that the morphism (\ref{ali:uafcui}) is a quasi-isomorphism for 
${\rm gr}F$. Then, by (\ref{cd:uafeacui}), 
the morphism (\ref{ali:uafcui}) for $F$ is a quasi-isomorphism if and only if 
the morphism (\ref{ali:uafcui}) for $F'$ is a quasi-isomorphism. 
By making the same argument for $F'$ as above repeatedly,  
we may assume that $F$ has no poles 
since $F$ is flat and locally nilpotent. 
\par 
For simplicity of notation, set $\Om^{\bul}:=
{\cal F}\otimes_{{\cal O}_{{\cal Q}^{\rm ex}}}
{\Om}^{\bul}_{{\cal Q}^{\rm ex}/\os{\circ}{T}}$ and 
$\Om^{\bul}_{/S(T)^{\nat}}:={\cal F}\otimes_{{\cal O}_{{\cal Q}^{\rm ex}}}
{\Om}^{\bul}_{{\cal Q}^{\rm ex}/S(T)^{\nat}}$. 
We claim that the following sequence 
\begin{equation*} 
0  \lo {\cal H}^{q-1}(\Om^{\bul}_{/S(T)^{\nat}}) \os{d\log t\wedge }{\lo} 
{\cal H}^q(\Om^{\bul})  \lo {\cal H}^q(\Om^{\bul}_{/S(T)^{\nat}}) \lo 0
\tag{4.8.4}\label{ali:agxhqd}
\end{equation*} 
for $q\in {\mab N}$ obtained by (\ref{ali:gsflaxd})  is exact. 
We have only to prove  
that the following inclusion holds:   
\begin{align*} 
\nabla({\Om}^{q-1}) \cap 
\{d\log t \wedge({\Om}^{q-1})\} \subset 
d\log t\wedge(\nabla ({\Om}^{q-2})). 
\tag{4.8.5}\label{ali:dayn}
\end{align*} 
Set $\om_1:=d\log t$. 
Let us take a local basis $\{\om_1,\ldots,\om_r\}$ of 
$\Om^1_{{\cal Q}^{\rm ex}/\os{\circ}{T}}$. 
Take a local section 
$$f:=\sum_{1\leq i_1<\cdots <i_{q-1}\leq r}
f_{i_1\cdots i_{q-1}}\om_{i_1}\wedge \cdots \wedge \om_{i_{q-1}}\quad (f_{i_1\cdots i_{q-1}}\in {\cal F})$$ 
of ${\cal F}\otimes_{{\cal O}_{{\cal Q}^{\rm ex}}}
{\Om}^{q-1}_{{\cal Q}^{\rm ex}/\os{\circ}{T}}$. 
Decompose $f$ as follows: 
$$f=\sum_{2\leq i_2<\cdots <i_{q-1}\leq r}f_{1 i_2\cdots i_{q-1}}\om_1
\wedge \om_{i_2} \wedge \cdots \wedge \om_{i_{q-1}}+\sum_{2\leq i_1<\cdots <i_{q-1}\leq r}
f_{i_1\cdots i_{q-1}}\om_{i_1}\wedge \cdots \wedge \om_{i_{q-1}}.$$
Then 
\begin{align*} 
\nabla(f)&=\sum_{2\leq i_2<\cdots <i_{q-1}\leq r}\nabla(f_{1i_2\cdots i_{q-1}})
\wedge \om_1\wedge \om_{i_2}\wedge \cdots \wedge \om_{i_{q-1}}
\tag{4.8.6}\label{ali:tt} \\
&-
\sum_{2\leq i_2<\cdots <i_{q-1}\leq r}f_{1i_2\cdots i_{q-1}}\om_1
\wedge d(\om_{i_2}\wedge\cdots \wedge \om_{i_{q-1}})\\
&+\sum_{2\leq i_1<\cdots <i_{q-1}\leq r}\nabla
(f_{i_1\cdots i_{q-1}}\om_{i_1}\wedge \cdots \wedge \om_{i_{q-1}}).  
\end{align*}
Because $F$ has no poles,  
the third term on the right hand side of (\ref{ali:tt}) can be expressed as follows:    
\begin{align*}
&\sum_{2\leq i_1<\cdots <i_{q-1}\leq r}g_{i_1\cdots i_{q-1}}\tau \om_1\wedge 
\om_{i_1}\wedge \cdots \wedge \om_{i_{q-1}}
+\sum_{2\leq i_1<\cdots <i_{q}\leq r}h_{i_1\cdots i_{q}}\om_{i_1}\wedge \cdots \wedge \om_{i_{q}}
\tag{4.8.7}\label{ali:ttvt}\\
&=\sum_{2\leq i_1<\cdots <i_{q}\leq r}h_{i_1\cdots i_{q}}\om_{i_1}\wedge \cdots \wedge \om_{i_{q}},
\end{align*} 
for some $g_{i_1\cdots i_{q}}\in {\cal F}$ and $h_{i_1\cdots i_{q}}\in {\cal F}$. 
Assume that $\nabla(f)\in \om_1 \wedge ({\Om}^{q-1})$. 
Then the last term in (\ref{ali:ttvt}) vanishes. 
Hence 
\begin{align*}
\nabla(f)&=-\om_1\wedge 
(\sum_{2\leq i_2<\cdots <i_{q-1}\leq r} \nabla(f_{1i_2\cdots i_{q-1}}) 
\wedge \om_{i_2}\wedge \cdots \wedge \om_{i_{q-1}}\\
&+ 
\sum_{2\leq i_2<\cdots <i_{q-1}\leq r}f_{1i_2\cdots i_{q-1}}
\wedge d(\om_{i_2}\wedge \cdots \wedge \om_{i_{q-1}}))\\
&=-\om_1\wedge 
\nabla(\sum_{2\leq i_2<\cdots <i_{q-1}\leq r} f_{1i_2\cdots i_{q-1}}\om_{i_2}\wedge
\cdots \wedge \om_{i_{q-1}}). 
\end{align*} 
Thus we have proved that the inclusion (\ref{ali:dayn}) holds and consequently 
we see that the sequence (\ref{ali:agxhqd}) is exact. 
\par 
By the exactness of (\ref{ali:agxhqd}), the following sequence is exact: 
\begin{align*} 
0  &\lo {\cal H}^0({\Om}^{\bul}) \os{d\log t\wedge }{\lo} 
{\cal H}^1({\Om}^{\bul})
 \os{d\log t\wedge }{\lo}  \cdots \tag{4.8.8}\label{ali:qf}
\cdots  \os{d\log t\wedge }{\lo}  {\cal H}^{q-1}({\Om}^{\bul})
\os{d\log t\wedge }{\lo} {\cal H}^q({\Om}^{\bul})  \\
&\lo {\cal H}^q(\Om^{\bul}_{/S(T)^{\nat}}) \lo 0. 
\end{align*} 
By the exactness of (\ref{ali:qf}) and the lemma (\ref{lemm:lm}) below, 
this implies that the following natural morphism 
\begin{align*}
{\cal H}^q(\Om^{\bul}\langle u\rangle)
\lo {\cal H}^q(\Om^{\bul}_{/S(T)^{\nat}}) \quad (q\in {\mab N})
\tag{4.8.9}\label{ali:epnpmt}
\end{align*}  
is an isomorphism. 
\par 
We complete the proof of this theorem. 
\end{proof}

\begin{lemm}\label{lemm:lm}
Let $A$ be a commutative ring with unit element. 
Let 
$$M^{\bul \bul}=(\prod_{i\in {\mab N},j\in {\mab Z}}M^{-i,j},
\prod_{i\in {\mab N},j\in {\mab Z}}d^{-i,j}, \prod_{i\in {\mab N},j\in {\mab Z}}d'^{-i,j})$$ 
be a double complex of $A$-modules, 
where $d^{-i,j}\col M^{-i,j}\lo M^{-i+1,j}$ and 
$d'{}^{-i,j}\col M^{-i,j}\lo M^{-i,j+1}$ are morphism of $A$-modules  
such that $d^{-i+1j}\circ d^{-i,j}=0$, $d'^{-i,j+1}\circ d'{}^{-i,j}=0$ and 
$d'{}^{-i+1,j}\circ d^{-i,j}+d^{-i,j+1}\circ d'{}^{-i,j}=0$. 
Let $N^{\bul}$ be a complex of $A$-modules. 
Let $\eps^j \col M^{0j}\lo N^j$ $(j\in {\mab Z})$ be a morphism of $A$-modules  
such that $\eps^{\bul}:=\{\eps^j\}_{j\in {\mab Z}}$ induces morphisms
$M^{0\bul}\lo N^{\bul}$ and $M^{-1,j}\os{d^{-1,j}}{\lo} M^{0j}\os{\eps^j}{\lo} N^j$ 
$(j\in {\mab Z})$ of complexes of $A$-modules.  
Assume that the following sequence 
\begin{align*}
\cdots \os{d^{-i-1,j}}\lo M^{-i,j}\os{d^{-i,j}}{\lo} 
M^{-i+1,j}\os{d^{-i+1,j}}{\lo} \cdots M^{0j}\os{\eps^j}{\lo} N^j\lo 0
\tag{4.9.1}\label{ali:nmij}
\end{align*} 
is exact and that 
the following natural sequence 
\begin{align*}
H^j(M^{-i,\bul},d'{}^{-i,\bul})\os{H^j(d^{-i,j})}{\lo} 
H^j(M^{-i+1,\bul},d'{}^{-i+1,\bul})\os{H^j(d^{-i+1,j})}{\lo} H^j(M^{-i+2,\bul},d'{}^{-i+2,\bul})
\tag{4.9.2}\label{ali:nhmij}
\end{align*}
is exact. 
Denote the sub-double complex 
$(\bigoplus_{i,j\in {\mab N}}M^{-i,j},\bigoplus_{i,j\in {\mab N}}d^{-i,j}, \bigoplus_{i,j\in {\mab N}}d'^{-i,j})_{i\in {\mab N},j\in {\mab Z}}$ 
of $M^{\bul \bul}$ 
simply by $\bigoplus_{i,j\in {\mab N}}M^{-i,j}$. 
Then 
the natural morphism $s(\bigoplus_{i,j\in {\mab N}}M^{-i,j})\lo N^{\bul}$ of complexes of $A$-modules 
induces the following injective morphism 
\begin{align*} 
H^q(s(\bigoplus_{i,j\in {\mab N}}M^{-i,j}))\os{\subset}{\lo} H^q(N^{\bul}) \quad (q\in {\mab Z}). 
\tag{4.9.3}\label{ali:nmmmij}
\end{align*} 
Furthermore, assume that the following natural morphism 
$${\rm Ker}(M^{0j}\lo M^{0,j+1})\lo {\rm Ker}(N^{0j}\lo N^{0,j+1})$$  
is surjective. 
Then the natural morphism $s(\bigoplus_{i,j\in {\mab N}}M^{-i,j})\lo N^{\bul}$ of complexes 
induces the following isomorphism 
\begin{align*} 
H^q(s(\bigoplus_{i,j\in {\mab N}}M^{-i,j}))\os{\sim}{\lo} H^q(N^{\bul}) \quad (q\in {\mab Z}). 
\tag{4.9.4}\label{ali:nmsmij}
\end{align*} 
\end{lemm}
\begin{proof}
If we omit the proof of this lemma whose proof is elementary 
(I do not say that this lemma is elementary), 
then the reader has to have some time to give the proof. 
Hence we give the proof for the reader's convenience. 
\par 
Let $q$ be an integer and let $k$ be a non-negative integer. 
Let $\sum_{i=0}^k\om^{-i,q+i}$ 
$(\om^{-i,q+i}\in M^{-i,q+i})$ be an element of 
$${\rm Ker}(\sum_{-i+j=q}(d^{-i,j}+d'{}^{-i,j}) \col s(\bigoplus_{-i+j=q}M^{-i,j})\lo s(\bigoplus_{-i+j=q+1}M^{-i,j})).$$ 
Then 
\begin{align*} 
d'{}^{-k,q+k}(\om^{-k,q+k})=0.
\tag{4.9.5}\label{ali:nkqmij}
\end{align*}  
Assume that $\eps^q(\om^{0q})\in {\rm Im}(\partial^{q-1})$. 
By the exactness of (\ref{ali:nmij}), there exists an element 
$\eta^{0,q-1}\in M^{0,q-1}$ such that 
$\eps^q(\om^{0q}-d'^{0,q-1}(\eta^{0,q-1}))=0$.  
By the exactness of (\ref{ali:nmij}) again, there exists an element 
$\eta^{-1,q}\in M^{-1,q}$ such that 
$$\om^{0q}=d^{-1,q}(\eta^{-1,q})+d'^{0,q-1}(\eta^{0,q-1}).$$ 
By the exactness of (\ref{ali:nmij}) again, there exists an element 
$\eta^{-2,q+1}\in M^{-2,q+1}$ such that 
$$\om^{-1,q+1}=d^{-2,q+1}(\eta^{-2,q+1})+d'^{-1,q}(\eta^{-1,q})$$
as in the proof of (\ref{lemm:cstite}).  
Continuing this process repeatedly, 
we obtain $\eta^{0,q-1},\ldots, \eta^{-k,q+k-1}$. 
Because $\om^{-k,q+k}-d'{}^{-k,q+k-1}(\eta^{-k,q+k-1})\in {\rm Ker}(d^{-k,q+k})$ 
and because 
$$d'{}^{-k,q+k}(\om^{-k,q+k}-d'{}^{-k,q+k-1}(\eta^{-k,q+k-1}))=0$$ 
by (\ref{ali:nkqmij}), 
there exists an element $\eta^{-k-1,q+k}\in {\rm Ker}(d'{}^{-k-1,q+k})$ such that 
$$\om^{-k,q}=d^{-k-1,q}(\eta^{-k-1,q})+d'{}^{-k,q+k-1}(\eta^{-k,q+k-1})$$ 
by the exactness of (\ref{ali:nhmij}). 
Hence 
$\sum_{i=0}^k\om^{-i,q+i}$ is the coboundary of 
$\sum_{i=0}^{k+1}\eta^{-i,q+i-1}$. 
This means that the morphism (\ref{ali:nmmmij}) is injective. 
The surjectivity of the morphism (\ref{ali:nmsmij}) is obvious. 
\end{proof} 

\begin{rema}\label{rema:p}
(1) As pointed out in \cite[(2.2.3) (1)]{nb}, the formula 
\begin{align*} 
(d\wt{\om}^{i-1}_{{\cal Y}_n} \cap 
(\wt{\om}^{i-1}_{{\cal Y}_n}
\wedge (dT/T))) \subset d\wt{\om}^{i-2}_{{\cal Y}_n} \wedge dT/T
\tag{4.10.1}\label{ali:dyn}
\end{align*} 
in \cite[p.~246]{hyp}
is an obvious formula about which 
the editorial comment \cite[(6), (11A)]{hyp} is unnecessary. 
Because the boundary morphism 
${\cal H}^i(\Om_{{\cal Y}_n/{\cal W}_n(\os{\circ}{s})}) 
\lo 
{\cal H}^{i+1}(\Om_{{\cal Y}_n/{\cal W}_n(\os{\circ}{s})})$ 
is ``$p^{-n}d$'', 
the above formula is not a precisely necessary ingredient 
for the injectivity needed in [loc.~cit.].  
The boundary morphism of ${\cal H}^i(\Om_{{\cal Y}_n/{\cal W}_n(\os{\circ}{s})}) 
\lo 
{\cal H}^{i+1}(\Om_{{\cal Y}_n/{\cal W}_n(\os{\circ}{s})})$ has been misunderstood in 
[loc.~cit.] and the editorial comment for (\ref{ali:dyn}). 
\par 
\par 
(2) As pointed out in \cite[(2.2.3) (2)]{nb}, 
the proof of the injectivity of the morphism 
$\theta_n \wedge \col W_n\om^{q-1}_Y\lo W_n\wt{\om}^q_Y$ 
has not been given in \cite{msemi}, though it has been claimed 
in [loc.~cit., p.~311]. 
\par 
(3) 
I think that no general theorem, e.~g., the classical convergent theorem \cite[(5.5.1)]{weib} 
nor the complete convergence theorem \cite[(5.5.10)]{weib} implies 
(\ref{theo:qii}) because the filtration by columns 
of $\Om^{\bul}\langle u\rangle $ is {\it not} complete. 
Let us check that this claim is true by giving the following argument more precisely. 
\par 
Let $F$ be an increasing filtration 
on $\Om^{\bul}\langle U_{S(T)^{\nat}}\rangle$ 
defined analogously as {\rm (\ref{ali:film})}. 
That is, 
\begin{align*} 
F_i(\Om^{\bul}\langle u\rangle)
:=\bigoplus_{j\geq -i}\Gam_{{\cal O}_T,j}(U_{S(T)^{\nat}})\otimes_{{\cal O}_T}\Om^{\bul} \quad (i\in {\mab Z}).
\end{align*}  
For $i,j\in {\mab N}$, set 
$$d'':={\rm id}\otimes \nabla  \col 
{\cal O}_T u^{[i]}  \otimes_{{\cal O}_T}\Om^{j-i}\lo 
{\cal O}_T u^{[i]}  \otimes_{{\cal O}_T}\Om^{j+1-i}.$$  
Let  
${}''{\cal H}^q$ 
be the cohomology with respect to $d''$. 
Consider the following spectral sequence with respect to $F$ 
\begin{align*} 
E_1^{-p,p+q}={\cal H}^{q}({\rm gr}^{-p}_F(\Om^{\bul}\langle u\rangle ))
\Lo 
{\cal H}^{q}(\Om^{\bul}\langle u\rangle ). 
\tag{4.10.2}\label{ali:ppt}
\end{align*} 
This spectral sequence degenerates at $E_2$. 
Indeed, this is a local problem. 
We have the following equality: 
\begin{align*}
E_1^{-p,p+q}
&=''\!{\cal H}^{q+p}({\cal O}_Tu^{[p]}\otimes_{{\cal O}_T}{\cal F}\otimes_{{\cal O}_{{\cal Q}^{\rm ex}}}
\Om^{\bul}_{{\cal Q}^{\rm ex}/\os{\circ}{T}})
={\cal H}^{q+p}({\cal F}\otimes_{{\cal O}_{{\cal Q}^{\rm ex}}}
\Om^{\bul}_{{\cal Q}^{\rm ex}/\os{\circ}{T}}).
\end{align*}
The exact sequence (\ref{ali:qf}) tells us that 
$E_2^{0q}={\cal H}^q({\cal F}\otimes_{{\cal O}_{{\cal Q}^{\rm ex}}}{\Om}^{\bul}_{{\cal Q}^{\rm ex}/S(T)^{\nat}})$ 
and $E_2^{pq}=0$ for $p\not=0$. 
In particular, $E_{\infty}^{pq}=E_2^{pq}=0$.  
By (\ref{theo:qii}) 
\begin{align*} 
{\cal H}^{q}({\cal F}\otimes_{{\cal O}_{{\cal Q}^{\rm ex}}}
\Om^{\bul}_{{\cal Q}^{\rm ex}/\os{\circ}{T}}\langle u\rangle )
=E_2^{0q}.
\end{align*} 
This means that $F_0
{\cal H}^{q}({\cal F}\otimes_{{\cal O}_{{\cal Q}^{\rm ex}}}
\Om^{\bul}_{{\cal Q}^{\rm ex}/\os{\circ}{T}}\langle u\rangle )
=
{\cal H}^{q}({\cal F}\otimes_{{\cal O}_{{\cal Q}^{\rm ex}}}
\Om^{\bul}_{{\cal Q}^{\rm ex}/\os{\circ}{T}}\langle u\rangle )$ and 
$F_{-1}{\cal H}^{q}({\cal F}\otimes_{{\cal O}_{{\cal Q}^{\rm ex}}}
\Om^{\bul}_{{\cal Q}^{\rm ex}/\os{\circ}{T}}\langle u\rangle )=0$. 
Because the filtration $F$ on 
${\cal F}\otimes_{{\cal O}_{{\cal Q}^{\rm ex}}}
\Om^{\bul}_{{\cal Q}^{\rm ex}/\os{\circ}{T}}\langle u\rangle$ 
is {\it not} complete, the problem  whether 
the spectral sequence (\ref{ali:ppt}) is convergent is a non-trivial problem; 
(\ref{theo:qii}) tells us that (\ref{ali:ppt}) is indeed convergent. 
\end{rema}


The following is a generalization of the log crystalline sheafied version including torsions  
 (not only the log crystalline cohomological version) of \cite[(3.35) (4)]{ey}: 

\begin{coro}\label{coro:nqi}
The morphism {\rm (\ref{ali:ufui})} is a quasi-isomorphism.   
\end{coro}
\begin{proof} 
This follows from (\ref{theo:saih}) and (\ref{theo:qii}). 
\end{proof}

\begin{rema}\label{rema:b}
(1) In \cite[(3.25) (4)]{ey} Yamada and Ertl have proved  
that the induced morphism of the morphism {\rm (\ref{ali:ufui})} 
in the context of log rigid cohomologies 
is a quasi-isomorphism 
when $Y/S$ is the log special fiber of 
a strictly semistable scheme over a complete discrete valuation ring of mixed characteristics 
with perfect residue field by a different method from our method: 
the proof of the corresponding statement to (\ref{theo:qii}) 
in the context of log rigid cohomologies is different from ours. 
\par 
(2) In \cite{ey} the completed PD-Hirsch extension (\ref{eqn:uui}) 
of $\Om^{\bul}_{\ol{\cal Q}{}^{\rm ex}/\os{\circ}{T}}$
in the case of the trivial coefficient for the strict semistable family 
has a very important role. In this book we suppress to use 
(\ref{eqn:uui}) as an important ingredient. 
\end{rema}

\par
\begin{rema}\label{rema:nwc}
(1) 
Define decreasing filtrations $F$'s on 
${\cal F}\otimes_{{\cal O}_{{{\cal Q}}}}\Om^{\bul}_{{{\cal Q}}/\os{\circ}{T}}\langle u\rangle$ 
and ${\cal F}
\otimes_{{\cal O}_{{{\cal Q}}}}\Om^{\bul}_{{{\cal Q}}/\os{\circ}{T}}$ 
by the following formulas, respectively: 
\begin{align*} 
F^i({\cal F}
\otimes_{{\cal O}_{{{\cal Q}}}}\Om^{\bul}_{{{\cal Q}}/\os{\circ}{T}}\langle u\rangle):=
\sum_{j\in {\mab N}}{\cal O}_Tu^{[j]}\otimes 
{\cal F}
\otimes_{{\cal O}_{{{\cal Q}}}}\Om^{\bul \geq i-j}_{{{\cal Q}}/\os{\circ}{T}},   
\tag{4.13.1}\label{ali:fgfngr}
\end{align*} 
\begin{align*} 
F^i({\cal F}\otimes_{{\cal O}_{{{\cal Q}}}}\Om^{\bul}_{{{\cal Q}}/S(T)^{\nat}}):=
{\cal F}
\otimes_{{\cal O}_{{{\cal Q}}}}\Om^{\bul \geq i}_{{{\cal Q}}/S(T)^{\nat}}.  
\tag{4.13.2}\label{ali:fgffgr}
\end{align*}
$($Note that the former filtration $F$ is not finite.$)$
Then we can prove that the following morphism induced by {\rm (\ref{ali:uafcui})}   
\begin{align*} 
{\rm gr}_F^i({\cal F}\otimes_{{\cal O}_{{{\cal Q}}}}\Om^{\bul}_{{{\cal Q}}/\os{\circ}{T}}\langle u\rangle)
\lo 
{\rm gr}_F^i({\cal F}\otimes_{{\cal O}_{{{\cal Q}}}}\Om^{\bul}_{{{\cal Q}}/S(T)^{\nat}})
\tag{4.13.3}\label{ali:fggr}
\end{align*} 
is a quasi-isomorphism. 
This is a $p$-adic generalized analogue of \cite[(5.3)]{fut}. 
Furthermore assume that $\os{\circ}{\cal Q}$ is quasi-compact. 
As a corollary of this result, we can prove that 
there exists a positive integer $i_0$ such that,  
for all $i\geq i_0$, 
the following morphism 
\begin{align*} 
{\cal F}
\otimes_{{\cal O}_{{{\cal Q}}}}\Om^{\bul}_{{{\cal Q}}/\os{\circ}{T}}\langle u\rangle/F^i\lo 
{\cal F}\otimes_{{\cal O}_{{{\cal Q}}}}\Om^{\bul}_{{{\cal Q}}/S(T)^{\nat}}/
F^i
\tag{4.13.4}\label{ali:qi}
\end{align*} 
induces the following quasi-isomorphism: 
\begin{align*} 
{\cal F}
\otimes_{{\cal O}_{{{\cal Q}}}}\Om^{\bul}_{{{\cal Q}}/\os{\circ}{T}}\langle u\rangle/F^i\os{\sim}{\lo}  
{\cal F}\otimes_{{\cal O}_{{{\cal Q}}}}\Om^{\bul}_{{{\cal Q}}/S(T)^{\nat}}.
\tag{4.13.5}\label{ali:qfi}
\end{align*}   
However we do not 
use these results in this book because 
the filtrations $F$'s (\ref{ali:fgfngr}) and (\ref{ali:fgffgr}) 
are non-canonical in our case.
\par 
(2) 
Let the notations be as in \cite[(5.9)]{fut}. 
By using \cite[(5.3)]{fut} (cf.~(\ref{ali:qfi})), 
Fujisawa has proved that 
there exists a morphism 
\begin{align*} 
K_{\mab C}\lo \om^{\bul}_{X/*^k} \quad (k\in {\mab Z}_{\geq 1})
\tag{4.13.6}\label{eqn:kkkd}
\end{align*} 
in a derived category of bounded above ${\mab C}_X$-modules 
which induces an isomorphism 
\begin{align*}
H^q(X,K_{\mab C})\os{\sim}{\lo} H^q(X,\om^{\bul}_{X/*^k})
\tag{4.13.7}\label{eqn:kfstd}
\end{align*} 
by using the general theory of Hodge-Deligne 
(\cite[(5.9)]{fut}). However he has not proved that the morphism 
(\ref{eqn:kkkd}) itself is an isomorphism.
The obvious analogues of (\ref{theo:qii}) 
and (\ref{prop:impth}) below in the complex analytic case 
are new results. They tell us that the morphism 
(\ref{eqn:kkkd}) is indeed an isomorphism for the case $k=1$. 
Consequently we do not need any result of 
Hodge-Deligne theory for the proof of (\ref{eqn:kfstd}).  
Hence the obvious analogue of (\ref{theo:qii}) can play 
an important role in [loc.~cit., (5.9)] at least in the case $k=1$. 
We can also prove that the morphism (\ref{eqn:kkkd}) is an isomorphism 
for a general $k$. 
However we do not discuss this more in this book. 
See \cite{nf} for details.
\end{rema}

\par 
Now let $X$ be an SNCL scheme over $S_0$. 
Let $(T,{\cal J},\del)$, $T_0$ and $z\col T_0\lo S$ be 
as in the beginning of this section.  
For simplicity of notation, we denote $\os{\circ}{X}_{\os{\circ}{T}_0}
:=\os{\circ}{X}\times_{\os{\circ}{S}}\os{\circ}{T}_0$ by
$\os{\circ}{X}_{T_0}(:=(X_{T_0})^{\circ})$. 
Indeed, $\os{\circ}{X}_{T_0}$ is the underlying scheme of 
$X_{T_0}$ since the structural morphism $X\lo S$ is integral. 
Let  $\Del:=\{\os{\circ}{X}_{\lam,T_0}\}_{\lam \in \Lam}$  be a decomposition 
of $\os{\circ}{X}_{T_0}/\os{\circ}{T}_0$ by its smooth components. 
In this section we assume that there exists 
an immersion $X_{\os{\circ}{T}_0}\os{\sus}{\lo} \ol{\cal P}$ 
into a log smooth scheme over $\ol{S(T)^{\nat}}$. 
Let $\ol{\mathfrak D}$ be the log PD-envelope of 
$X_{\os{\circ}{T}_0}\os{\sus}{\lo} \ol{\cal P}$ over 
$(\os{\circ}{T},{\cal J},\del)$. 
Set ${\mathfrak D}:=
\ol{\mathfrak D}\times_{{\mathfrak D}(\ol{S(T)^{\nat}})}S(T)^{\nat}$ 
and 
${\cal P}^{\rm ex}
:=\ol{\cal P}{}^{\rm ex}\times_{\ol{S(T)^{\nat}}}S(T)^{\nat}$. 
Then ${\cal P}^{\rm ex}$ is the exactification of 
the immersion $X_{\os{\circ}{T}_0}\os{\sus}{\lo}{\cal P}$ ((\ref{prop:xpls})). 
\par  
Because $\os{\circ}{X}_{T_0}$ 
and $\os{\circ}{\cal P}{}^{\rm ex}$ 
are topologically isomorphic, 
we can define a closed subscheme 
$\os{\circ}{\cal P}{}^{\rm ex}_{\ul{\lam}}$ of 
$\os{\circ}{\cal P}{}^{\rm ex}$ 
corresponding to $\os{\circ}{X}_{\ul{\lam},T_0}$ 
for an element
$\ul{\lam}=\{\lam_0,\cdots, \lam_m\}\in P(\Lam)$ 
$(m\in {\mab Z}_{\geq 0}, \lam_i \in \Lam, \lam_i \not= \lam_j~{\rm if}~i\not= j)$ 
as explained in \S\ref{sec:snclv}, where $\os{\circ}{X}_{\lam_i,T_0}$ denotes 
a smooth component of $\os{\circ}{X}_{T_0}$ in this section.  
Endow $\os{\circ}{\cal P}{}^{\rm ex}_{\ul{\lam}}$ 
with the pull-back of the log structure of ${\cal P}^{\rm ex}$ 
and let ${\cal P}^{\rm ex}_{\ul{\lam}}$ be the resulting log scheme. 
Let $m$ be a nonnegative integer. 
Endow $\os{\circ}{\cal P}{}^{{\rm ex},(m)}$ 
with the pull-back of the log structure of 
${\cal P}^{\rm ex}$ 
and let ${\cal P}{}^{{\rm ex},(m)}$ be the resulting log scheme. 
Then, by (\ref{eqn:kfexntd}), 
\begin{equation}
{\cal P}{}^{{\rm ex},(m)} =  \us{\# \ul{\lam}=m+1}
{\coprod}{\cal P}{}^{\rm ex}_{\ul{\lam}}.   
\tag{4.13.8}\label{eqn:kfltd}
\end{equation}  
For the convenience of notation, we
set ${\cal P}{}^{{\rm ex},(-1)}:={\cal P}{}^{{\rm ex}}$. 
As usual, $\{{\cal P}{}^{{\rm ex},(m)}\}_{m\in {\mab N}}$ defines 
a semi-simplicial log scheme in a standard way.

\begin{rema}
It is easy to show that the log scheme 
${\cal P}^{\rm ex}_{\ul{\lam}}$ is not log smooth over $S(T)^{\nat}$ 
for $\ul{\lam}\not=\emptyset$ 
(cf.~\cite[p.~2947]{gkcf}). 
Let $L$ be the subsheaf defined by the preimage of 
${\rm Ker}({\cal O}_{{\cal P}^{\rm ex}}\lo {\cal O}_{{\cal P}_{\ul{\lam}}})$ 
by the structural morphism 
$M_{{\cal P}^{\rm ex}}\lo {\cal O}_{{\cal P}^{\rm ex}}$.  
Let $L_{\ul{\lam}}$ be the pull-back of $L$ 
by the closed immersion ${\cal P}_{\ul{\lam}}\lo {\cal P}^{\rm ex}$. 
Then ${\cal P}^{\rm ex}_{\ul{\lam}}$ with the ideal sheaf 
$L_{\ul{\lam}}$ of $M_{{\cal P}_{\ul{\lam}}}$ 
is ideally log smooth over $S(T)^{\nat}$ 
(cf.~\cite[Remarks 1.10]{gkcf}) since ${\cal P}^{\rm ex}$ is log smooth over $S(T)^{\nat}$. 
\end{rema}


\begin{prop}\label{prop:bpc}
Set $U=S(T)^{\nat}$ or $U=\os{\circ}{T}$. 
Let $i$ be a nonnegative integer. 
Then the following hold$:$
\par 
$(1)$ 
The canonical morphism 
\begin{equation*} 
\Om^i_{{\cal P}^{\rm ex}/U}\otimes_{{\cal O}_{{\cal P}^{\rm ex}}}
{\cal O}_{{\cal P}^{\rm ex}_{\ul{\lam}}} \lo 
\Om^i_{{\cal P}^{\rm ex}_{\ul{\lam}/U}}
\tag{4.15.1}\label{eqn:cae}
\end{equation*} 
is an isomorphism. 
\par 
$(2)$ Let $m$ be a nonnegative integer. 
The canonical morphism 
\begin{equation*} 
\Om^i_{{\cal P}^{\rm ex}/U}
\otimes_{{\cal O}_{{\cal P}^{\rm ex}}}
{\cal O}_{{\cal P}^{{\rm ex},(m)}} \lo 
\Om^i_{{\cal P}^{{\rm ex},(m)}/U}
\tag{4.15.2}\label{eqn:came}
\end{equation*} 
is an isomorphism. 
\end{prop}
\begin{proof} 
We have only to prove (1). 
Let ${\cal I}_{\ul{\lam}}$ be the defining ideal sheaf of 
the immersion ${\cal P}^{\rm ex}_{\ul{\lam}} \os{\sus}{\lo} {\cal P}^{\rm ex}$.  
Then, by \cite[Lemma 2.1.3]{nh2} (cf.~\cite[(3.6)(2)]{kn}), 
we have the following second fundamental exact sequence 
\begin{equation*} 
{\cal I}_{\ul{\lam}}/{\cal I}_{\ul{\lam}}^2
\os{d}{\lo} 
\Om^1_{{\cal P}^{\rm ex}/U}\otimes_{{\cal O}_{{\cal P}^{\rm ex}}}
{\cal O}_{{\cal P}^{\rm ex}_{\ul{\lam}}} \lo 
\Om^1_{{\cal P}^{\rm ex}_{\ul{\lam}/U}}\lo 0. 
\tag{4.15.3}\label{eqn:iopu} 
\end{equation*} 
Let $m$ be a local section of $M_{{\cal P}^{\rm ex}}$ such that 
$\al (m)\in {\cal I}_{\ul{\lam}}$, where $\al \col M_{{\cal P}^{\rm ex}}\lo {\cal O}_{{\cal P}^{\rm ex}}$
is the structural morphism. 
Then $d\al (m)=\al (m)d\log m=0$ in 
$\Om^1_{{\cal P}^{\rm ex}/U}\otimes_{{\cal O}_{{\cal P}^{\rm ex}}}
{\cal O}_{{\cal P}^{\rm ex}_{\ul{\lam}}}$. Because   
${\cal I}_{\ul{\lam}}$ has a system of generators coming from local 
sections of $M_{{\cal P}^{\rm ex}}$, the morphism $d$ in  
(\ref{eqn:iopu}) is zero. Hence we obtain (1). 
\end{proof}

\par 
Let $E$ be a flat quasi-coherent crystal of  
${\cal O}_{\os{\circ}{X}_{T_0}/\os{\circ}{T}}$-modules. 
\par 
Let 
$$\eps_{X_{\os{\circ}{T}_0}/\os{\circ}{T}}
\col 
((X_{\os{\circ}{T}_0}/\os{\circ}{T})_{\rm crys},
{\cal O}_{X_{\os{\circ}{T}_0}/\os{\circ}{T}}) 
\lo 
((\os{\circ}{X}_{T_0}/\os{\circ}{T})_{\rm crys},
{\cal O}_{\os{\circ}{X}_{T_0}/\os{\circ}{T}})$$  
be the morphism of ringed topoi induced by the morphism 
$\eps_{X_{\os{\circ}{T}_0}/\os{\circ}{T}_0}\col X_{\os{\circ}{T}_0}\lo \os{\circ}{X}_{T_0}$ 
over $\os{\circ}{T}_0$ forgetting the log structure of $X_{\os{\circ}{T}_0}$. 
\par 
Let $(\ol{\cal E},\ol{\nabla})$ 
be the quasi-coherent ${\cal O}_{\ol{\mathfrak D}}$-module  
with integrable connection corresponding to 
the log crystal $\eps_{X_{\os{\circ}{T}_0}/\os{\circ}{T}}(E)$: 
\begin{equation*} 
\ol{\nabla}\col \ol{\cal E}\lo 
\ol{\cal E}\otimes_{{\cal O}_{\ol{\cal P}{}^{\rm ex}}}
\Om^1_{\ol{\cal P}{}^{\rm ex}/\os{\circ}{T}}.
\tag{4.15.4}\label{eqn:olepc}
\end{equation*}  
Set $({\cal E},{\nabla}):=
(\ol{\cal E},\ol{\nabla})
\otimes_{{\cal O}_{{\mathfrak D}(\ol{S(T)^{\nat}})}}{\cal O}_{S(T)^{\nat}}$:  
\begin{equation*} 
\nabla\col {\cal E}\lo 
{\cal E}\otimes_{{\cal O}_{{\cal P}^{\rm ex}}}
\Om^1_{{\cal P}^{\rm ex}/\os{\circ}{T}}. 
\tag{4.15.5}\label{eqn:olompc}
\end{equation*}  
In fact, we have the log de Rham complexes  
$\ol{\cal E}\otimes_{{\cal O}_{\ol{\cal P}{}^{\rm ex}}}
\Om^{\bul}_{\ol{\cal P}{}^{\rm ex}/\os{\circ}{T}}$ and 
${\cal E}\otimes_{{\cal O}_{{\cal P}^{\rm ex}}}
\Om^{\bul}_{{\cal P}^{\rm ex}/\os{\circ}{T}}$.

\begin{lemm}\label{lemm:drl}
The family 
$\{{\cal E}\otimes_{{\cal O}_{{\cal P}^{\rm ex}}}
\Om^i_{{\cal P}^{\rm ex}_{\ul{\lam}}/\os{\circ}{T}}\}_{i\geq 0}$
forms a complex.
Consequently the family 
$\{{\cal E}\otimes_{{\cal O}_{{\cal P}^{\rm ex}}}
\Om^i_{{\cal P}^{{\rm ex},(m)}/\os{\circ}{T}}\}_{i\geq 0}$ 
for $m\in{\mab Z}_{\geq 0}$ forms a complex. 
\end{lemm}
\begin{proof} 
The problem is local. 
We may assume that there exists a solid and \'{e}tale morphism 
${\cal P}^{\rm prex}{}'\lo {\mab A}_{S(T)^{\nat}}(a,d')$ in (\ref{eqn:xdplxda}). 
Let $\lam$ be an element of $\ul{\lam}$. 
Let $x_{\lam}$ be a coordinate of the affine ring of 
${\mab A}_{S(T)^{\nat}}(a,d')$ such that ${\cal P}^{\rm ex}_{\lam}$ is 
defined by $x_{\lam}=0$.  
Then, for $e\in  {\cal E}$ and $\om \in \Om^i_{{\cal P}^{\rm ex}/\os{\circ}{T}}$, 
\begin{align*}
\nabla(x_{\lam}e\otimes \om)=
e\otimes dx_{\lam}\wedge \om+x_{\lam}\nabla(e\otimes \om)
=x_{\lam}(e\otimes d\log x_{\lam}\wedge \om+\nabla(e\otimes \om)).
\end{align*} 
Hence we obtain the complex 
${\cal E}\otimes_{{\cal O}_{{\cal P}^{\rm ex}}}
\Om^{\bul}_{{\cal P}^{\rm ex}_{\ul{\lam}}/\os{\circ}{T}}$ by 
(\ref{eqn:cae}). 
\end{proof} 

\par 
By using the natural surjective morphisms  
$\Om^{\bul}_{\ol{\cal P}{}^{\rm ex}/\os{\circ}{T}}
\lo \Om^{\bul}_{\ol{\cal P}{}^{\rm ex}/S(T)^{\nat}}$,  
$\Om^{\bul}_{{\cal P}^{\rm ex}/\os{\circ}{T}}\lo 
\Om^{\bul}_{{\cal P}^{\rm ex}/S(T)^{\nat}}$, 
$\Om^{\bul}_{{\cal P}^{\rm ex}_{\ul{\lam}}/\os{\circ}{T}}\lo 
\Om^{\bul}_{{\cal P}^{\rm ex}_{\ul{\lam}}/S(T)^{\nat}}$
and 
$\Om^{\bul}_{{\cal P}^{{\rm ex},(m)}/\os{\circ}{T}}\lo 
\Om^{\bul}_{{\cal P}^{{\rm ex},(m)}/S(T)^{\nat}}$, 
we have the log de Rham complexes  
$\ol{\cal E}\otimes_{{\cal O}_{\ol{\cal P}{}^{\rm ex}}}
\Om^{\bul}_{\ol{\cal P}{}^{\rm ex}/S(T)^{\nat}}$, 
${\cal E}\otimes_{{\cal O}_{{\cal P}^{\rm ex}}}
\Om^{\bul}_{{\cal P}^{\rm ex}/S(T)^{\nat}}$, 
${\cal E}\otimes_{{\cal O}_{{\cal P}^{\rm ex}}}
\Om^{\bul}_{{\cal P}^{\rm ex}_{\ul{\lam}}/S(T)^{\nat}}$
and 
${\cal E}\otimes_{{\cal O}_{{\cal P}^{\rm ex}}}
\Om^{\bul}_{{\cal P}^{{\rm ex},(m)}/S(T)^{\nat}}$ 
for $m\in{\mab Z}_{\geq 0}$ as in (\ref{lemm:drl}). 
\par 
In a standard way, 
the following functor 
\begin{align*} 
P(\Lam)\owns \ul{\lam}\lom {\cal P}^{\rm ex}_{\ul{\lam}}\in 
{\rm Log}(/S(T)^{\nat})
\end{align*} 
defines a cubical log scheme over $S(T)^{\nat}$ indexed by $P(\Lam)$, 
where ${\rm Log}(/S(T)^{\nat})$ is the category of fine log schemes 
over $S(T)^{\nat}$. 
For convenience of notation, we set 
${\cal P}^{\rm ex}_{\emptyset}:={\cal P}^{\rm ex}$.  
Let $b_{\ul{\lam}} \col {\cal P}^{\rm ex}_{\ul{\lam}} 
\lo {\cal P}^{\rm ex}$ 
and $b^{(m)} \col {\cal P}^{{\rm ex},(m)}\lo {\cal P}^{\rm ex}$  
be the natural morphisms. 
Set $U=S(T)^{\nat}$ or $U=\os{\circ}{T}$. 
Fix a total order on $\Lam$ once and for all. 
Then we have the \v{C}ech complex 
\begin{equation*} 
{\cal E}\otimes_{{\cal O}_{{\cal P}^{\rm ex}}}
b^{(0)}_*(\Om^i_{{\cal P}^{{\rm ex},(0)}/U})
\lo 
{\cal E}\otimes_{{\cal O}_{{\cal P}^{\rm ex}}}
b^{(1)}_*(\Om^i_{{\cal P}{}^{{\rm ex},(1)}/U}) \lo \cdots, 
\end{equation*} 
defined by the following boundary morphism. 
For an element $\ul{\lam}=\{\lam_0,\ldots,\lam_m\}$ 
$(m\in {\mab N},\lam_i <\lam_j~{\rm if}~i< j, \lam_i\in \Lam)$, 
set $\ul{\lam}_j:=\ul{\lam}\setminus \{\lam_j\}$
and let $\iota_{\ul{\lam}_j,\ul{\lam}} \col {\cal P}^{\rm ex}_{\ul{\lam}}\os{\sus}{\lo} 
{\cal P}^{\rm ex}_{\ul{\lam}_j}$
be the natural inclusion. Then the boundary morphism  
\begin{align*} 
{\cal E}\otimes_{{\cal O}_{{\cal P}^{\rm ex}}}
b^{(m-1)}_*(\Om^i_{{\cal P}{}^{{\rm ex},(m-1)}/U})
\lo 
{\cal E}\otimes_{{\cal O}_{{\cal P}^{\rm ex}}}
b^{(m)}_*(\Om^i_{{\cal P}{}^{{\rm ex},(m)}/U})
\end{align*} 
is, by definition, 
\begin{align*} 
\us{\{\ul{\lam} \vert \# \ul{\lam}=m+1\}}{\sum}
\sum_{j=0}^m(-1)^j\iota^{*}_{\ul{\lam}_j,\ul{\lam}}.
\tag{4.16.1}\label{ali:ulm}
\end{align*}  
We denote this complex by 
${\cal E}\otimes_{{\cal O}_{{\cal P}{}^{{\rm ex},(\bul)}}}
\Om^i_{{\cal P}{}^{{\rm ex},(\bul)}/U}$. 
This complex is not an example of the complex defined in \cite[(2.5)]{fup} 
because we fix a total  order on $\Lam$. 
In a standard (non-standard?) way (\cite[p.~24]{nh3}), 
we have the double complex 
${\cal E}\otimes_{{\cal O}_{{\cal P}{}^{{\rm ex},(\bul)}}}
\Om^{\bul}_{{\cal P}{}^{{\rm ex},(\bul)}/U}$ as follows:  
\begin{equation*} 
\begin{CD} 
\cdots @>>> \cdots  
@>>> \cdots \\
@A{\nabla}AA @A{-\nabla}AA  \\
{\cal E}\otimes_{{\cal O}_{{\cal P}{}^{{\rm ex},(0)}}}
\Om^1_{{\cal P}{}^{{\rm ex},(0)}/U}  @>>> 
{\cal E}\otimes_{{\cal O}_{{\cal P}{}^{{\rm ex},(1)}}}
\Om^1_{{\cal P}{}^{{\rm ex},(1)}/U} @>>> \cdots  \\
@A{\nabla}AA @A{-\nabla}AA    \\
{\cal E}\otimes_{{\cal O}_{{\cal P}{}^{{\rm ex},(0)}}}
{\cal O}_{{\cal P}{}^{{\rm ex},(0)}} @>>> 
{\cal E}\otimes_{{\cal O}_{{\cal P}{}^{{\rm ex},(1)}}}
{\cal O}_{{\cal P}{}^{{\rm ex},(1)}} 
@>>> \cdots \\
\end{CD} 
\tag{4.16.2}\label{cd:pppoex}
\end{equation*} 
and the associated single complex 
$s({\cal E}\otimes_{{\cal O}_{{\cal P}{}^{{\rm ex},(\bul)}}}
\Om^{\bul}_{{\cal P}{}^{{\rm ex},(\bul)}/U})$. 
\par 
Then the following holds: 

\begin{prop}\label{prop:sil} 
For a nonnegative integer $i$,  
the following natural morphism 
\begin{equation*} 
0\lo {\cal E}\otimes_{{\cal O}_{{\cal P}{}^{\rm ex}}}
\Om^i_{{\cal P}{}^{\rm ex}/U}
\lo {\cal E}\otimes_{{\cal O}_{{\cal P}{}^{{\rm ex},(0)}}}
\Om^i_{{\cal P}{}^{{\rm ex},(0)}/U}
\lo {\cal E}\otimes_{{\cal O}_{{\cal P}{}^{{\rm ex},(1)}}}
\Om^i_{{\cal P}{}^{{\rm ex},(1)}/U} \lo \cdots 
\tag{4.17.1}\label{eqn:pd1u}
\end{equation*} 
is exact. 
\end{prop}
\begin{proof}
Since ${\cal E}$ is a flat ${\cal O}_{\mathfrak D}$-module, 
we may assume that ${\cal E}={\cal O}_{\mathfrak D}$. 
The exactness is a local question. 
We may assume that ${\cal O}_{{\cal P}^{\rm ex}}
\simeq {\cal O}_S[x_0,\ldots,x_d][[x_{d+1},\ldots, x_{d'}]]/(x_0\cdots x_a)$ 
for some $0\leq a\leq d\leq d'$ ((\ref{prop:adla})). 
Set ${\cal P}^{\rm ex}{}':=
\ul{\rm Spec}_{{\cal P}^{\rm ex}}^{\log}
({\cal O}_{{\cal P}^{\rm ex}}/(x_{d+1},\ldots, x_{d'}))$. 
Then we have a natural retraction 
${\cal P}^{\rm ex}\lo {\cal P}^{\rm ex}{}'$ of the immersion 
${\cal P}^{\rm ex}{}'\os{\sus}{\lo} {\cal P}^{\rm ex}$. 
\begin{equation*} 
{\cal P}{}^{{\rm ex},(m)}
={\cal P}{}^{{\rm ex}{}',(m)}\times_{{\cal P}^{\rm ex}{}'}
{\cal P}^{\rm ex}.
\tag{4.17.2}\label{eqn:mpex}
\end{equation*}  
The exactness of the following sequence 
\begin{equation*} 
0\lo {\cal O}_{{\cal P}{}^{\rm ex}{}'}\lo 
{{\cal O}_{{\cal P}{}^{{\rm ex}{}',(0)}}}
\lo {{\cal O}_{{\cal P}{}^{{\rm ex}{}',(1)}}}
\lo \cdots 
\tag{4.17.3}\label{eqn:mwex}
\end{equation*} 
is well-known (we have only to replace $W$ in \cite[Lemma 11]{kiha}
 by ${\cal O}_S$ or we have only to replace ${\mab C}$ in 
\cite[(4.15)]{fut} by ${\cal O}_S$). 
Consider the tensorization 
$\otimes_{{\cal O}_S}{\cal O}_S\langle x_{d+1},\ldots,x_{d'}\rangle$ of (\ref{eqn:mwex}). 
Then this sequence is exact since 
${\cal O}_S\langle x_{d+1},\ldots, x_{d'}\rangle$ 
is a free ${\cal O}_S$-module.   
Using (\ref{eqn:mpex}) and noting 
${\cal O}_{{\cal P}^{\rm ex}{}'}\langle x_{d+1},\ldots, x_{d'}\rangle
\otimes_{{\cal O}_{{\cal P}{}^{\rm ex}}}
{{\cal O}_{{\cal P}{}^{{\rm ex},(m)}}}
=
{\cal O}_{{\cal P}^{\rm ex}{}'}\langle x_{d+1},\ldots, x_{d'}\rangle
\otimes_{{\cal O}_{{\cal P}{}^{\rm ex}}}
{{\cal O}_{{\cal P}{}^{{\rm ex}{}',(m)}}}
\otimes_{{\cal O}_{{\cal P}{}^{\rm ex}{}'}}
{{\cal O}_{{\cal P}{}^{{\rm ex}}}}
=
{\cal O}_{{\cal P}^{\rm ex}{}'}\langle x_{d+1},\ldots, x_{d'}\rangle
\otimes_{{\cal O}_{{\cal P}{}^{\rm ex}{}'}}
{{\cal O}_{{\cal P}{}^{{\rm ex}{}',(m)}}}
$, 
we see that the resulting sequence is equal to the following exact sequence 
\begin{equation*} 
0\lo {\cal O}_{\mathfrak D}
\lo {\cal O}_{\mathfrak D}
\otimes_{{\cal O}_{{\cal P}{}^{\rm ex}}}
{{\cal O}_{{\cal P}{}^{{\rm ex},(0)}}}
\lo {\cal O}_{\mathfrak D}
\otimes_{{\cal O}_{{\cal P}{}^{\rm ex}}}
{{\cal O}_{{\cal P}{}^{{\rm ex},(1)}}}
\lo \cdots.  
\end{equation*} 
Now we have only to take the tensorization 
$\otimes_{{\cal O}_{{\cal P}{}^{\rm ex}}}
\Om^i_{{\cal P}{}^{\rm ex}/U}$ 
and to use (\ref{prop:bpc}) (2). 
\end{proof}



\par
We consider the following PD-Hirsch extension 
\begin{equation*}   
{\cal E}\otimes_{{\cal O}_{{\cal P}^{{\rm ex}}}}
\Om^{\bul}_{{\cal P}{}^{{\rm ex}}/\os{\circ}{T}}\langle u\rangle 
\end{equation*} 
of 
${\cal E}\otimes_{{\cal O}_{{\cal P}{}^{{\rm ex}}}}
\Om^{\bul}_{{\cal P}{}^{{\rm ex}}/\os{\circ}{T}}$ as before 
and we consider the following PD-Hirsch extension 
\begin{equation*}   
{\cal E}\otimes_{{\cal O}_{{\cal P}{}^{{\rm ex}}_{\ul{\lam}}}} 
\Om^{\bul}_{{\cal P}{}^{{\rm ex}}_{\ul{\lam}}/\os{\circ}{T}}
\langle u \rangle 
\end{equation*} 
of 
${\cal E}\otimes_{{\cal O}_{{\cal P}{}^{{\rm ex}}_{\ul{\lam}}}}
\Om^{\bul}_{{\cal P}{}^{{\rm ex}}_{\ul{\lam}}/\os{\circ}{T}}$.  
Consequently we can consider the following PD-Hirsch extension 
\begin{equation*}   
{\cal E}\otimes_{{\cal O}_{{\cal P}{}^{{\rm ex}}}}
\Om^{\bul}_{{\cal P}{}^{{\rm ex},(m)}/\os{\circ}{T}}\langle u\rangle
\quad (m\geq -1)
\end{equation*} 
of 
${\cal E}\otimes_{{\cal O}_{{\cal P}{}^{{\rm ex}}}}
\Om^{\bul}_{{\cal P}{}^{{\rm ex},(m)}/\os{\circ}{T}}$.  

\begin{prop}\label{prop:impth}
Let 
$s({\cal E}\otimes_{{\cal O}_{{\cal P}{}^{{\rm ex},(\bul)}}}
\Om^{\bul}_{{\cal P}{}^{{\rm ex},(\bul)}/\os{\circ}{T}}\langle u \rangle)$
be the single complex of the double complex {\rm (\ref{cd:pppoex})}. 
Then we have the following commutative diagram$:$ 
\begin{equation*} 
\begin{CD}
s({\cal E}\otimes_{{\cal O}_{{\cal P}{}^{{\rm ex},(\bul)}}}
\Om^{\bul}_{{\cal P}{}^{{\rm ex},(\bul)}/\os{\circ}{T}}\langle u \rangle)
@>>>
s({\cal E}\otimes_{{\cal O}_{{\cal P}{}^{{\rm ex},(\bul)}}}
\Om^{\bul}_{{\cal P}{}^{{\rm ex},(\bul)}/S(T)^{\nat}})\\
@A{\simeq}AA @AA{\simeq}A \\ 
{\cal E}\otimes_{{\cal O}_{{\cal P}{}^{\rm ex}}}
\Om^{\bul}_{{\cal P}{}^{{\rm ex}}/\os{\circ}{T}}\langle u \rangle
@>{\sim}>> {\cal E}\otimes_{{\cal O}_{{\cal P}{}^{\rm ex}}}
\Om^{\bul}_{{\cal P}{}^{{\rm ex}}/S(T)^{\nat}}.  
\end{CD} 
\tag{4.18.1}\label{eqn:cdpe}
\end{equation*}  
\end{prop}
\begin{proof} 
This immediately follows from (\ref{prop:sil}), (\ref{theo:qii}) 
and (\ref{coro:hac}).  
\end{proof} 

\par 
We also obtain the following, which will play an important role for the construction 
of the PD-Hirsch weight-filtered complex: 

\begin{theo}\label{theo:sih} 
For an element $\ul{\lam}$ of $P(\Lam)\cup \{\emptyset\}$,   
the following natural morphism  
\begin{equation*}  
{\cal E}
\otimes_{{\cal O}_{{\cal P}{}^{\rm ex}}}
\Om^{\bul}_{{\cal P}{}^{{\rm ex}}_{\ul{\lam}}/\os{\circ}{T}}\langle u\rangle
\lo {\cal E}\otimes_{{\cal O}_{{\cal P}{}^{\rm ex}}}
\Om^{\bul}_{{\cal P}{}^{{\rm ex}}_{\ul{\lam}}/S(T)^{\nat}}
\tag{4.19.1}\label{eqn:exlr}
\end{equation*}
is a quasi-isomorphism.  
\end{theo}
\begin{proof}  
In (\ref{theo:qii}) we have already proved (\ref{theo:sih}) for the case $\ul{\lam}=\emptyset$. 
\par 
By (\ref{ali:gsflaxd}) the following sequence is exact: 
\begin{align*} 
0  \lo {\cal E}
\otimes_{{\cal O}_{{{\cal P}^{\rm ex}}}}
{\Om}^{\bul}_{{{\cal P}^{\rm ex}}/S(T)^{\nat}}[-1] \os{d\log t\wedge }{\lo} 
{\cal E}\otimes_{{\cal O}_{{{\cal P}^{\rm ex}}}}
{\Om}^{\bul}_{{{\cal P}^{\rm ex}}/\os{\circ}{T}}  
\lo {\cal E}\otimes_{{\cal O}_{{{\cal P}^{\rm ex}}}}
{\Om}^{\bul}_{{{\cal P}^{\rm ex}}/S(T)^{\nat}} \lo 0. 
\tag{4.19.2}\label{ali:agxd}
\end{align*} 
By (\ref{prop:bpc}), (\ref{lemm:drl}) and (\ref{ali:agxd}),
the following sequence is exact: 
\begin{align*} 
0  \lo {\cal E}
\otimes_{{\cal O}_{{{\cal P}^{\rm ex}}}}
{\Om}^{\bul}_{{{\cal P}^{\rm ex}_{\ul{\lam}}}/S(T)^{\nat}}[-1] \os{d\log t\wedge }{\lo} 
{\cal E}\otimes_{{\cal O}_{{{\cal P}^{\rm ex}}}}
{\Om}^{\bul}_{{{\cal P}^{\rm ex}_{\ul{\lam}}}/\os{\circ}{T}}  
\lo {\cal E}\otimes_{{\cal O}_{{{\cal P}^{\rm ex}}}}
{\Om}^{\bul}_{{{\cal P}^{\rm ex}_{\ul{\lam}}}/S(T)^{\nat}} \lo 0. 
\tag{4.19.3}\label{ali:agaxd}
\end{align*} 
Furthermore, by the same proof as that of (\ref{theo:qii}), 
the following sequence obtained by (\ref{ali:agaxd}) 
is exact: 
\begin{align*} 
0  &\lo {\cal H}^{q-1}({\cal E}
\otimes_{{\cal O}_{{{\cal P}^{\rm ex}}}}
{\Om}^{\bul}_{{{\cal P}^{\rm ex}_{\ul{\lam}}}/S(T)^{\nat}})
\os{d\log t\wedge }{\lo} 
{\cal H}^q({\cal E}\otimes_{{\cal O}_{{{\cal P}^{\rm ex}}}}
{\Om}^{\bul}_{{{\cal P}^{\rm ex}_{\ul{\lam}}}/\os{\circ}{T}}) \tag{4.19.4}\label{ali:ppxhd}\\ 
&\lo {\cal H}^q({\cal E}\otimes_{{\cal O}_{{{\cal P}^{\rm ex}}}}
{\Om}^{\bul}_{{{\cal P}^{\rm ex}_{\ul{\lam}}}/S(T)^{\nat}}) \lo 0. 
\end{align*} 
The rest of the proof is the same as that of (\ref{theo:qii}). 
\end{proof}

\section{Preweight filtrations on PD-Hirsch extensions of log crystalline complexes}\label{sec:pwf}
Let the notations be as in \S\ref{sec:ldfc}. 
In this section we recall a filtration on 
${\cal F}\otimes_{{\cal O}_{{\cal Q}^{\rm ex}}}\Om^{\bul}_{{\cal Q}^{\rm ex}/\os{\circ}{T}}$ 
defined in \cite{nb} and 
we define filtrations on 
${\cal E}\otimes_{{\cal O}_{{\cal P}^{\rm ex}}}
\Om^{\bul}_{{\cal P}{}^{\rm ex}/\os{\circ}{T}}\langle u\rangle$, 
${\cal E}\otimes_{{\cal O}_{{\cal P}^{\rm ex}}}
\Om^{\bul}_{{\cal P}^{{\rm ex}}_{\ul{\lam}}/\os{\circ}{T}}$, 
${\cal E}\otimes_{{\cal O}_{{\cal P}^{\rm ex}}}
\Om^{\bul}_{{\cal P}^{{\rm ex}}_{\ul{\lam}}/\os{\circ}{T}}\langle u\rangle$, 
${\cal E}\otimes_{{\cal O}_{{\cal P}^{\rm ex}}}
\Om^{\bul}_{{\cal P}^{{\rm ex},(m)}/\os{\circ}{T}}$
and 
${\cal E}\otimes_{{\cal O}_{{\cal P}^{\rm ex}}}
\Om^{\bul}_{{\cal P}^{{\rm ex},(m)}/\os{\circ}{T}}\langle u\rangle$. 
We calculate the graded complexes of them. 
\par 

Let $Y$ be a fine log 
(formal) scheme over a fine log (formal) scheme $T$. 
As in \cite[(4.0.2)]{nh3}, we define the {\it pre-weight filtration}
$P$ on the sheaf ${\Om}^i_{Y/\os{\circ}{T}}$ $(i\in {\mab N})$ 
of log differential forms on $Y_{\rm zar}$ as follows: 
\begin{equation*} 
P_k{\Om}^i_{Y/\os{\circ}{T}} =
\begin{cases} 
0 & (k<0), \\
{\rm Im}({\Om}^k_{Y/\os{\circ}{T}}{\otimes}_{{\cal O}_Y}
\Om^{i-k}_{\os{\circ}{Y}/\os{\circ}{T}}
\lo {\Om}^i_{Y/\os{\circ}{T}}) & (0\leq k\leq i), \\
{\Om}^i_{Y/\os{\circ}{T}} & (k > i).
\end{cases}
\tag{5.0.1}\label{eqn:pkdefpw}
\end{equation*}  
A morphism $g\col Y\lo Z$ of fine log (formal) schemes over 
$\os{\circ}{T}$ 
induces the following morphism of filtered complexes: 
\begin{equation*} 
g^*\col (\Om^{\bul}_{Z/\os{\circ}{T}},P)
\lo 
g_*((\Om^{\bul}_{Y/\os{\circ}{T}},P)).  
\tag{5.0.2}\label{eqn:lyzpp}
\end{equation*} 
More generally, for a flat ${\cal O}_Y$-module ${\cal E}$ 
and a flat ${\cal O}_Z$-module ${\cal F}$ with a morphism 
$h \col {\cal F}\lo g_*({\cal E})$ of ${\cal O}_Z$-modules, 
we have the following morphism of filtered complexes: 
\begin{equation*} 
h \col ({\cal F}\otimes_{{\cal O}_Z}\Om^{\bul}_{Z/\os{\circ}{T}},P)
\lo 
g_*(({\cal E}\otimes_{{\cal O}_Y}{\Om}^{\bul}_{Y/\os{\circ}{T}},P)).  
\tag{5.0.3}\label{eqn:lyytp}
\end{equation*}  
Especially we obtain the filtrations $P$'s on 
$\Om^{\bul}_{{\cal Q}^{{\rm ex}}/\os{\circ}{T}}$,  
$\Om^{\bul}_{{\cal P}^{{\rm ex}}_{\ul{\lam}}/\os{\circ}{T}}$
and $\Om^{\bul}_{{\cal P}^{{\rm ex},(m)}/\os{\circ}{T}}$ 
$(m\in {\mab N})$.  


\par 
In \cite{nb} we have proved the following:

\begin{prop}[{\bf \cite[(1.3.4)]{nb}}]\label{prop:injf}
The natural morphism 
\begin{equation*} 
{\cal O}_{\mathfrak E}
\otimes_{{\cal O}_{{\cal Q}^{\rm ex}}}
P_k{\Om}^i_{{\cal Q}^{\rm ex}/\os{\circ}{T}}
\lo 
{\cal O}_{\mathfrak E}
\otimes_{{\cal O}_{{\cal Q}^{\rm ex}}}
{\Om}^i_{{\cal Q}^{\rm ex}/\os{\circ}{T}} 
\quad (i,k\in {\mab Z})
\tag{5.1.1}\label{eqn:yxnpd}
\end{equation*}
is injective. 
\end{prop}

\begin{prop}\label{prop:tnmi}
The natural morphism 
\begin{equation*} 
{\cal E}\otimes_{{\cal O}_{{\cal P}^{\rm ex}}}
P_k\Om^i_{{\cal P}^{\rm ex}_{\ul{\lam}}/\os{\circ}{T}}
\lo 
{\cal E}\otimes_{{\cal O}_{{\cal P}^{\rm ex}}}
\Om^i_{{\cal P}^{{\rm ex}}_{\ul{\lam}}/\os{\circ}{T}}
\quad (i,k\in {\mab Z})
\tag{5.2.1}\label{eqn:yxamd}
\end{equation*}
is injective. 
Consequently the natural morphism 
\begin{equation*} 
{\cal E}\otimes_{{\cal O}_{{\cal P}^{\rm ex}}}
P_k{\Om}^i_{{\cal P}^{{\rm ex},(m)}/\os{\circ}{T}}
\lo 
{\cal E}\otimes_{{\cal O}_{{\cal P}^{\rm ex}}}
{\Om}^i_{{\cal P}^{{\rm ex},(m)}/\os{\circ}{T}} 
\quad (i,k\in {\mab Z})
\tag{5.2.2}\label{eqn:yxmd}
\end{equation*}
is injective. 
\end{prop}
\begin{proof} 
The same proof as that of \cite[(1.3.4)]{nb} works for this proposition. 
\end{proof}

\begin{rema}\label{rema:grk}
The proof of the convergent version of 
the injectivity of (\ref{eqn:yxmd}) is necessary in \cite{gkcf}, 
which has not been proved in [loc.~cit.]. 
\end{rema}

Because $E$ in the previous section is a flat quasi-coherent crystal of 
${\cal O}_{\os{\circ}{X}_T/\os{\circ}{T}}$-modules 
and because the image of the connection 
$\nabla\col {\cal E}\lo 
{\cal E}\otimes_{{\cal O}_{\cal P}^{\rm ex}}
\Om^1_{{\cal P}^{\rm ex}/\os{\circ}{T}}$ has no log poles,  
we can prove the following as in \cite[(1.3.17)]{nb}: 

\begin{prop}\label{prop:subc}
The complex 
${\cal E}\otimes_{{\cal O}_{{\cal P}^{\rm ex}}}
\Om^{\bul}_{{\cal P}^{\rm ex}_{\ul{\lam}}/\os{\circ}{T}}$ 
gives the subcomplex 
${\cal E}
\otimes_{{\cal O}_{{\cal P}^{\rm ex}}}
P_k\Om^{\bul}_{{\cal P}^{\rm ex}_{\ul{\lam}}/\os{\circ}{T}}$ 
$(k\in {\mab Z})$. 
\end{prop}

\parno 
We denote by $P$ the filtrations on 
${\cal E}
\otimes_{{\cal O}_{{\cal P}^{\rm ex}}}
\Om^{\bul}_{{\cal P}^{{\rm ex}}_{\ul{\lam}}/\os{\circ}{T}}$
and 
${\cal E}\otimes_{{\cal O}_{{\cal P}^{\rm ex}}}
{\Om}^{\bul}_{{\cal P}^{{\rm ex},(m)}/\os{\circ}{T}}$
induced by the injective morphisms 
(\ref{eqn:yxamd}) and (\ref{eqn:yxmd}), respectively.

In \cite{nb} we have proved the following: 

\begin{prop}[{\bf \cite[(1.3.14), (1.3.21)]{nb}}]\label{prop:grem}
Identify the points of $\os{\circ}{\mathfrak D}$ 
with those of $\os{\circ}{X}$. 
Identify also the images of the points of $\os{\circ}{X}$ 
in $\os{\circ}{\cal P}$ with the points of $\os{\circ}{X}$. 
Let $y$ be a point of $\os{\circ}{\mathfrak D}$. 
$($Let $r$ be a nonnegative integer such that 
$M_{X,y}/{\cal O}_{X,y}^*\simeq {\mab N}^r$.$)$ 
Then, for a positive integer $k$, 
the following morphism  
\begin{equation*} 
{\rm Res} \col 
P_k({\cal O}_{\mathfrak D}\otimes_{{\cal O}_{{\cal P}^{\rm ex}}}
{\Om}^{\bul}_{{\cal P}^{\rm ex}/\os{\circ}{T}}) \lo 
{\cal O}_{\mathfrak D}\otimes_{{\cal O}_{{\cal P}^{\rm ex}}}
b^{(k-1)}_*(\Om^{\bul}_{\os{\circ}{\cal P}{}^{{\rm ex},(k-1)}
/\os{\circ}{T}}
\otimes_{\mab Z}\vp^{(k-1)}_{\rm zar}(\os{\circ}{\cal P}{}^{\rm ex}/\os{\circ}{T}))[-k]
\tag{5.5.1}\label{eqn:mprarn}
\end{equation*} 
\begin{equation*} 
\sig \otimes d\log x_{\lam_0,y}\cdots d\log x_{\lam_{k-1},y}\omega  
\lom 
\sig \otimes b^*_{\lam_0\cdots \lam_{k-1}}(\omega)
\otimes({\rm orientation}~(\lam_0\cdots \lam_{k-1})) 
\end{equation*} 
$$(\sig \in {\cal O}_{\mathfrak D}, 
\om \in P_0{\Om}^{\bul}_{{\cal P}^{\rm ex}/\os{\circ}{T}})$$
{\rm (cf.~\cite[(3.1.5)]{dh2})} 
induces the following ``Poincar\'{e} residue isomorphism''  
\begin{align*} 
{\rm gr}^P_k
({\cal O}_{\mathfrak D}\otimes_{{\cal O}_{{\cal P}^{\rm ex}}}
{\Om}^{\bul}_{{\cal P}^{\rm ex}/\os{\circ}{T}}) 
& \os{\sim}{\lo} 
{\cal O}_{\mathfrak D}\otimes_{{\cal O}_{{\cal P}^{\rm ex}}}
b^{(k-1)}_*(\Om^{\bul}_{\os{\circ}{\cal P}{}^{{\rm ex},(k-1)}
/\os{\circ}{T}}
\otimes_{\mab Z}\vp^{(k-1)}_{\rm zar}
(\os{\circ}{\cal P}{}^{\rm ex}/\os{\circ}{T}))[-k]. 
\tag{5.5.2}\label{eqn:prvin} \\ 
\end{align*} 
Here $x_{\lam_l}$ $(l\in {\mab N})$ is 
the corresponding local coordinate to ${\cal P}_{\lam_l}$ 
{\rm ((\ref{eqn:xdplxda}), (\ref{eqn:xdpelda}))}. 
\par 
$(2)$ The following sequence 
\begin{equation*} 
0 \lo P_0
({\cal O}_{{\mathfrak D}}
\otimes_{{\cal O}_{{\cal P}^{\rm ex}}}
\Om^{\bul}_{{\cal P}^{\rm ex}/\os{\circ}{T}}) 
\lo c^{(0)}_*
({\cal O}_{{\mathfrak D}^{(0)}}
\otimes_{{\cal O}_{{\os{\circ}{\cal P}{}^{{\rm ex},(0)}}}}
\Om^{\bul}_{\os{\circ}{\cal P}{}^{{\rm ex},(0)}
/\os{\circ}{T}}\otimes_{\mab Z}
\vp^{(0)}_{\rm zar}(\os{\circ}{\cal P}{}^{\rm ex}
/\os{\circ}{T})) 
\tag{5.5.3}\label{eqn:cptle} 
\end{equation*} 
$$\os{\iota^{(0)*}}{\lo} {c}{}^{(1)}_{T*}
({\cal O}_{{\mathfrak D}^{(1)}}
\otimes_{{\cal O}_{\os{\circ}{\cal P}{}^{{\rm ex},(1)}}}
\Om^{\bul}_{\os{\circ}{\cal P}{}^{{\rm ex},(1)}
/\os{\circ}{T}}
\otimes_{\mab Z}
\vp^{(1)}_{\rm zar}(\os{\circ}{\cal P}{}^{\rm ex}
/\os{\circ}{T})) 
\os{\iota^{(1)*}}{\lo} \cdots $$ 
is exact, where ${\mathfrak D}^{(k)}:={\mathfrak D}\times_{{\cal P}^{\rm ex}}{\cal P}^{{\rm ex},(k)}$ 
$(k\in {\mab Z}_{\geq 0})$ and 
$c^{(k)}\col {\mathfrak D}^{(k)}\lo {\mathfrak D}$ is the natural morphism and 
$\iota^{(k)*}$ $(k\in {\mab Z}_{\geq 0})$ 
is the standard \v{C}ech morphism {\rm (\ref{ali:ulm})}. 
\end{prop}

\begin{lemm}\label{lemm:inl}
The following single complex 
\begin{align*}
&s[c^{(0)}_*
({\cal E}{\otimes}_{{\cal O}_{{\cal P}{}^{{\rm ex}}}} 
{\Om}^{\bul}_{\os{\circ}{\cal P}{}^{{\rm ex},(0)}/\os{\circ}{T}}
\otimes_{\mab Z}\vp^{(0)}_{\rm zar}
(\os{\circ}{\cal P}{}^{\rm ex}/\os{\circ}{T}))) 
\os{\iota^{(0)*}}{\lo} 
(c^{(1)}_*({\cal E}{\otimes}_{{\cal O}_{{\cal P}{}^{{\rm ex}}}}
{\Om}^{\bul}_{\os{\circ}{\cal P}{}^{{\rm ex},(1)}/\os{\circ}{T}}
\otimes_{\mab Z}\vp^{(1)}_{\rm zar}
(\os{\circ}{\cal P}{}^{\rm ex}/\os{\circ}{T})),-d) \tag{5.6.1}\label{ali:oopt}\\
&\os{\iota^{(1)*}}{\lo}  c^{(2)}_*
({\cal E}{\otimes}_{{\cal O}_{{\cal P}{}^{{\rm ex}}}} 
{\Om}^{\bul}_{\os{\circ}{\cal P}{}^{{\rm ex},(2)}/\os{\circ}{T}}
\otimes_{\mab Z}\vp^{(2)}_{\rm zar}
(\os{\circ}{\cal P}{}^{\rm ex}/\os{\circ}{T}))) 
\os{\iota^{(2)*}}{\lo} \cdots]
\end{align*}
is independent of the choice of the immersion 
$X_{\os{\circ}{T}_0}\os{\sus}{\lo} \ol{\cal P}$ over $\ol{S(T)^{\nat}}$. 
\end{lemm}
\begin{proof} 
To give the proof is a routine work. 
Indeed, let 
$X_{\os{\circ}{T}_0}\os{\sus}{\lo} \ol{\cal P}{}'$ be another immersion 
over $\ol{S(T)^{\nat}}$. Then, by considering the fiber product 
$\ol{\cal P}\times_{\ol{S(T)^{\nat}}}\ol{\cal P}{}'$, we may assume that 
there exists a morphism $\ol{\cal P}\lo \ol{\cal P}{}'$ over 
$\ol{S(T)^{\nat}}$ such that the composite morphism 
$X_{\os{\circ}{T}_0}\os{\sus}{\lo} \ol{\cal P}\lo \ol{\cal P}{}'$ 
is the given immersion $X_{\os{\circ}{T}_0}\os{\sus}{\lo} \ol{\cal P}{}'$. 
Set ${\cal P}':=\ol{\cal P}{}'\times_{\ol{S(T)^{\nat}}}S(T)^{\nat}$ and 
let ${\mathfrak D}'$ be the log PD-envelope of the immersion 
$X_{\os{\circ}{T}_0}\os{\sus}{\lo} {\cal P}{}'$ 
over $S(T)^{\nat}$ and 
set ${\mathfrak D}'{}^{(k)}:={\mathfrak D}'\times_{{\cal P}'{}^{\rm ex}}{\cal P}'{}^{{\rm ex},(k)}$ 
$(k\in {\mab N})$ and 
let $c'{}^{(k)}\col {\mathfrak D}'{}^{(k)}\lo {\mathfrak D}'$ be the natural morphism. 
Denote the complex (\ref{ali:oopt}) by $C^{\bul}$ and let 
$C'{}^{\bul}$ be the analogous complex to (\ref{ali:oopt}) 
obtained by the immersion $X_{\os{\circ}{T}_0}\os{\sus}{\lo} \ol{\cal P}{}'$. 
Let ${\cal E}'$ be an analogous sheaf to ${\cal E}$ for the immersion 
$X_{\os{\circ}{T}_0}\os{\sus}{\lo} \ol{\cal P}{}'$. 
Then we have two morphisms 
${\cal E}'{\otimes}_{{\cal O}_{{\cal P}'{}^{{\rm ex}}}} 
{\Om}^{\bul}_{\os{\circ}{\cal P}{}'^{{\rm ex},(i)}/\os{\circ}{T}}
\otimes_{\mab Z}\vp^{(i)}_{\rm zar}
(\os{\circ}{\cal P}{}'^{\rm ex}/\os{\circ}{T})
\lo 
{\cal E}{\otimes}_{{\cal O}_{{\cal P}{}^{{\rm ex}}}} 
{\Om}^{\bul}_{\os{\circ}{\cal P}{}^{{\rm ex},(i)}/\os{\circ}{T}}
\otimes_{\mab Z}\vp^{(i)}_{\rm zar}
(\os{\circ}{\cal P}{}^{\rm ex}/\os{\circ}{T})$
and 
$C'{}^{\bul}\lo C^{\bul}$ of complexes. 
By the following two spectral sequences 
$$E_1^{ij}={\cal H}^j(c^{(i)}_*({\cal E}{\otimes}_{{\cal O}_{{\cal P}{}^{{\rm ex}}}} 
{\Om}^{\bul}_{\os{\circ}{\cal P}{}^{{\rm ex},(i)}/\os{\circ}{T}}
\otimes_{\mab Z}\vp^{(i)}_{\rm zar}
(\os{\circ}{\cal P}{}^{\rm ex}/\os{\circ}{T}))) \Lo {\cal H}^{i+j}(C^{\bul})$$ 
and 
$$E_1^{ij}={\cal H}^j(
c'{}^{(i)}_*({\cal E}'{\otimes}_{{\cal O}_{{\cal P}{}'^{{\rm ex}}}} 
{\Om}^{\bul}_{\os{\circ}{\cal P}{}'^{{\rm ex},(i)}/\os{\circ}{T}}
\otimes_{\mab Z}\vp^{(i)}_{\rm zar}
(\os{\circ}{\cal P}{}'^{\rm ex}/\os{\circ}{T}))) \Lo {\cal H}^{i+j}(C{}'^{\bul}),$$
it suffices to prove that 
the morphism 
$$c'{}^{(i)}_*({\cal E}'{\otimes}_{{\cal O}_{{\cal P}{}'^{{\rm ex}}}} 
{\Om}^{\bul}_{\os{\circ}{\cal P}{}'^{{\rm ex},(i)}/\os{\circ}{T}}
\otimes_{\mab Z}\vp^{(i)}_{\rm zar}
(\os{\circ}{\cal P}{}'^{\rm ex}/\os{\circ}{T}))
\lo 
c^{(i)}_*({\cal E}{\otimes}_{{\cal O}_{{\cal P}^{{\rm ex}}}} 
{\Om}^{\bul}_{\os{\circ}{\cal P}^{{\rm ex},(i)}/\os{\circ}{T}}
\otimes_{\mab Z}\vp^{(i)}_{\rm zar}
(\os{\circ}{\cal P}{}^{\rm ex}/\os{\circ}{T}))$$ 
is a quasi-isomorphism. 
This is clear because 
\begin{align*}
c'{}^{(i)}_*({\cal E}'{\otimes}_{{\cal O}_{{\cal P}{}'^{{\rm ex}}}} 
{\Om}^{\bul}_{\os{\circ}{\cal P}{}'^{{\rm ex},(i)}/\os{\circ}{T}}
\otimes_{\mab Z}\vp^{(i)}_{\rm zar}
(\os{\circ}{\cal P}{}'^{\rm ex}/\os{\circ}{T}))&=
a^{(i)}_{T_0*}
Ru_{\os{\circ}{X}{}^{(i)}_{T_0}/\os{\circ}{T}*}
(E_{\os{\circ}{X}{}^{(i)}_{T_0}/\os{\circ}{T}}
\otimes_{\mab Z}\vp^{(i)}_{\rm crys}
(\os{\circ}{X}_{T_0}/\os{\circ}{T}))\\
&=
c^{(i)}_*({\cal E}'{\otimes}_{{\cal O}_{{\cal P}{}'^{{\rm ex}}}} 
{\Om}^{\bul}_{\os{\circ}{\cal P}{}'^{{\rm ex},(i)}/\os{\circ}{T}}
\otimes_{\mab Z}\vp^{(i)}_{\rm zar}
(\os{\circ}{\cal P}{}'^{\rm ex}/\os{\circ}{T}))
\end{align*} 
in $D^+(f^{-1}({\cal O}_T))$ 
by the crystalline Poincar\'{e} lemma. 

\end{proof} 

\begin{defi}\label{defi:sttu}
We denote the complex (\ref{ali:oopt}) by 
$$
Ru_{\os{\circ}{X}{}^{(\star)}_{T_0}/\os{\circ}{T}*}
(E_{\os{\circ}{X}{}^{(\star)}_{T_0}/\os{\circ}{T}}
\otimes_{\mab Z}\vp^{(\star)}_{\rm crys}
(\os{\circ}{X}_{T_0}/\os{\circ}{T})).$$
\end{defi}

As a corollary of (\ref{prop:grem}) (1) and (2),  
we have essentially proved the following in \cite{nb} 
by using  (\ref{lemm:inl}):  

\begin{prop}[{\bf \cite[(1.3.19), (1.3.22)]{nb}}]\label{prop:repmf}
$(1)$ 
Let $k$ be a positive integer. 
Then there exists the following isomorphism$:$ 
\begin{equation*}  
{\rm gr}^P_k
({\cal E}\otimes_{{\cal O}_{{\cal P}^{\rm ex}}}
{\Om}^{\bul}_{{\cal P}^{\rm ex}/\os{\circ}{T}})
\os{\sim}{\lo}  
a^{(k-1)}_{T_0*}
Ru_{\os{\circ}{X}{}^{(k-1)}_{T_0}/\os{\circ}{T}*}
(E_{\os{\circ}{X}{}^{(k-1)}_{T_0}/\os{\circ}{T}}
\otimes_{\mab Z}
\vp^{(k-1)}_{\rm crys}(\os{\circ}{X}_{T_0}/\os{\circ}{T}))[-k]
\tag{5.8.1}\label{eqn:eoppd}
\end{equation*}
in $D^+(f^{-1}_T({\cal O}_T))$. 
\par 
$(2)$ There exists the following isomorphism in $D^+(f^{-1}({\cal O}_T)):$
\begin{align*}
P_0({\cal E}
\otimes_{{\cal O}_{{\cal P}^{\rm ex}}}
{\Om}^{\bul}_{{\cal P}^{\rm ex}/\os{\circ}{T}})
\os{\sim}{\lo} 
Ru_{\os{\circ}{X}{}^{(\star)}_{T_0}/\os{\circ}{T}*}
(E_{\os{\circ}{X}{}^{(\star)}_{T_0}/\os{\circ}{T}}
\otimes_{\mab Z}\vp^{(\star)}_{\rm crys}
(\os{\circ}{X}_{T_0}/\os{\circ}{T})). 
\tag{5.8.2}\label{ali:pdte}
\end{align*}
The isomorphism {\rm (\ref{ali:pdte})} 
is independent of the choice of the immersion 
$X_{\os{\circ}{T}_0}\os{\sus}{\lo} \ol{\cal P}$ over $\ol{S(T)^{\nat}}$.  
In particular, 
$P_0({\cal E}\otimes_{{\cal O}_{{\cal P}^{\rm ex}}}
{\Om}^{\bul}_{{\cal P}^{\rm ex}/\os{\circ}{T}})$ 
is independent of the choice of the immersion 
$X_{\os{\circ}{T}_0}\os{\sus}{\lo} \ol{\cal P}$. 
\end{prop} 

In the case $\ul{\lam}\not=\emptyset$, 
the description of $P_0({\cal E}
\otimes_{{\cal O}_{{\cal P}^{\rm ex}}}
\Om^{\bul}_{{\cal P}^{{\rm ex}}_{\ul{\lam}}/\os{\circ}{T}})$ is 
totally different from that of 
$P_0({\cal E}\otimes_{{\cal O}_{{\cal P}^{\rm ex}}}
\Om^{\bul}_{{\cal P}^{{\rm ex}}/\os{\circ}{T}})$:

\begin{prop}\label{prop:pslt}
$(1)$ Assume that $\ul{\lam}\not=\emptyset$. 
For a nonnegative integer $k$, 
the Poincar\'{e} residue morphism 
\begin{equation*} 
{\rm Res} \col 
P_k({\cal E}
\otimes_{{\cal O}_{{\cal P}^{\rm ex}}}
\Om^{\bul}_{{\cal P}^{{\rm ex}}_{\ul{\lam}}/\os{\circ}{T}})
\lo  
\bigoplus_{\# \ul{\mu}=k}
{\cal E}
\otimes_{{\cal O}_{{\cal P}^{\rm ex}}}
b_{\ul{\lam}\cup \ul{\mu}*}
(\Om^{\bul}_{\os{\circ}{\cal P}{}^{{\rm ex}}_{\ul{\lam}
\cup \ul{\mu}}/\os{\circ}{T}}
\otimes_{\mab Z}\vp_{{\rm zar},\ul{\mu}}
(\os{\circ}{\cal P}{}^{\rm ex}/\os{\circ}{T})))[-k] 
\tag{5.9.1}\label{eqn:mprrn}
\end{equation*} 
\begin{equation*} 
\sig \otimes d\log x_{\lam_0,x}\cdots d\log m_{\lam_{k-1},x}\omega  
\lom 
\sig \otimes b^*_{\ul{\lam}\cup \ul{\mu}}(\omega)
\otimes({\rm orientation}~(\ul{\mu})) 
\end{equation*} 
$$(\sig \in {\cal E}, 
\om \in P_0\Om^{\bul}_{{\cal P}^{\rm ex}_{\ul{\mu}}
/\os{\circ}{T}})$$
{\rm (cf.~\cite[(3.1.5)]{dh2})} induces 
the following ``Poincar\'{e} residue isomorphism''  
\begin{align*} 
{\rm gr}_k^P({\cal E}
\otimes_{{\cal O}_{{\cal P}^{\rm ex}}}
\Om^{\bul}_{{\cal P}^{{\rm ex}}_{\ul{\lam}}/\os{\circ}{T}})
\os{\sim}{\lo} & 
\bigoplus_{\# \ul{\mu}=k}
{\cal E}\otimes_{{\cal O}_{{\cal P}^{\rm ex}}}
b_{\ul{\lam}\cup \ul{\mu}*}
(\Om^{\bul -k}_{
\os{\circ}{\cal P}{}^{{\rm ex}}_{\ul{\lam}
\cup \ul{\mu}}/\os{\circ}{T}}
\otimes_{\mab Z}\vp_{{\rm zar},\ul{\mu}}
(\os{\circ}{\cal P}{}^{\rm ex}/\os{\circ}{T}))[-k]
\tag{5.9.2}\label{eqn:mpprrn} 
\end{align*} 
in $D^+(f^{-1}({\cal O}_T))$. 
\end{prop}
\begin{proof}
For the case where $k$ is a positive integer,  
we can prove (\ref{prop:pslt}) in a standard way 
as in e.g., in \cite[(1.3.14)]{nb}. 
(However see the last remark in (\ref{exem:nsc}) below.) 
\par 
When $k=0$, then the isomorphism (\ref{eqn:mpprrn}) 
is nothing but an isomorphism 
\begin{align*} 
P_0({\cal E}\otimes_{{\cal O}_{{\cal P}^{\rm ex}}}
\Om^{\bul}_{{\cal P}^{{\rm ex}}_{\ul{\lam}}/\os{\circ}{T}})
\os{\sim}{\lo} &  
{\cal E}\otimes_{{\cal O}_{{\cal P}^{\rm ex}}}
b_{\ul{\lam}*}(\Om^{\bul}_{\os{\circ}{\cal P}{}^{{\rm ex}}_{\ul{\lam}}/\os{\circ}{T}}), 
\tag{5.9.3}\label{eqn:mpp0rbrn} 
\end{align*} 
which is equivalent to 
the injectivity of the following natural morphism 
\begin{align*} 
{\cal E}
\otimes_{{\cal O}_{{\cal P}^{\rm ex}}}
\Om^{\bul}_{\os{\circ}{\cal P}{}^{{\rm ex}}_{\ul{\lam}}
/\os{\circ}{T}}\lo &  
{\cal E}\otimes_{{\cal O}_{{\cal P}^{\rm ex}}}
b_{\ul{\lam}*}(\Om^{\bul}_{{\cal P}{}^{{\rm ex}}_{\ul{\lam}}/\os{\circ}{T}}). 
\tag{5.9.4}\label{eqn:mpp0rrn} 
\end{align*} 
This is a local question. 
Hence we may assume 
that $\ol{\cal P}{}^{\rm prex}=\ol{\cal P}$ and that 
there exists 
the commutative diagram (\ref{eqn:xdpelda}). 
In fact, we can assume that $\ol{\cal P}{}^{\rm prex}{}'=\ol{\cal P}$. 
Set ${\cal P}:=\ol{\cal P}\times_{\ol{S(T)^{\nat}}}S(T)^{\nat}$. 
We have only to prove that the natural morphism 
\begin{equation*} 
\Om^i_{\os{\circ}{\cal P}_{\ul{\lam}}/\os{\circ}{T}} 
\lo 
\Om^i_{{\cal P}_{\ul{\lam}}/\os{\circ}{T}} 
\quad (i\in {\mab Z})
\tag{5.9.5}\label{eqn:pps}
\end{equation*} 
is injective. 
\par 
First consider the essential case $i=1$. 
Let $x_0,\ldots x_{d'}$ be the coordinates of 
${\mab A}_{S(T)^{\nat}}(a,d')$. 
Let $x_0=x_1= \cdots =x_m=0$ $(0\leq m\leq a)$ 
be the defining equations of ${\cal P}_{\ul{\lam}}$ in 
${\cal P}$. 
Then 
\begin{equation*} 
\Om^1_{{\cal P}_{\ul{\lam}}/\os{\circ}{T}} 
=\bigoplus_{j=1}^a
{\cal O}_{{\cal P}}/(x_0,\ldots, x_m)d\log x_j
\oplus 
\bigoplus_{j=a+1}^{d'}
{\cal O}_{{\cal P}}/(x_0,\ldots, x_m)d x_j
\tag{5.9.6}\label{eqn:ylms}
\end{equation*}  
and 
\begin{equation*} 
\Om^1_{\os{\circ}{\cal P}_{\ul{\lam}}/\os{\circ}{T}} 
=\bigoplus_{j=m+1}^a
{\cal O}_{{\cal P}}/(x_0,\ldots, x_m)dx_j\oplus 
\bigoplus_{j=a+1}^{d'}
{\cal O}_{{\cal P}}/(x_0,\ldots, x_m)d x_j. 
\tag{5.9.7}\label{eqn:ylbs}
\end{equation*}  
The natural morphism 
\begin{equation*} 
{\cal O}_T[x_{m+1},\ldots, x_d]dx_j
\lo 
{\cal O}_T[x_{m+1},\ldots, x_d]d\log x_j
\end{equation*}
$(m+1\leq j\leq a)$ is injective since 
$x_j$ is a non-zero divisor in ${\cal O}_T[x_{m+1},\ldots, x_d]$. 
Because 
${\cal O}_{{\cal P}}/(x_0,\ldots,x_m)$ is \'{e}tale 
over ${\cal O}_T[x_0,\ldots,x_d]/(x_0,\ldots,x_m)
={\cal O}_T[x_{m+1},\ldots,x_d]$, 
the desired injectivity of the morphism 
(\ref{eqn:pps}) for the case $i=1$ is clear. 
\par 
Now the desired injectivity for the general case 
follows. 
\end{proof}

\begin{rema}
Note that morphism (\ref{eqn:mpp0rrn}) for the case $\ul{\lam}=\emptyset$  
\begin{align*} 
{\cal O}_{{\mathfrak D}}
\otimes_{{\cal O}_{{\cal P}^{\rm ex}}}
\Om^{\bul}_{\os{\circ}{\cal P}{}^{{\rm ex}}/\os{\circ}{T}}\lo &  
{\cal O}_{{\mathfrak D}}
\otimes_{{\cal O}_{{\cal P}^{\rm ex}}}
\Om^{\bul}_{{\cal P}{}^{{\rm ex}}/\os{\circ}{T}}. 
\tag{5.10.1}\label{eqn:mppomrrn} 
\end{align*}
is not injective because  the image of $1\otimes (x_0\cdots x_{a-1}dx_a)
\in {\cal O}_{{\mathfrak D}}
\otimes_{{\cal O}_{{\cal P}^{\rm ex}}}
\Om^1_{\os{\circ}{\cal P}{}^{{\rm ex}}/\os{\circ}{T}}$ in 
${\cal O}_{{\mathfrak D}}
\otimes_{{\cal O}_{{\cal P}^{\rm ex}}}
\Om^{\bul}_{{\cal P}{}^{{\rm ex}}/\os{\circ}{T}}$  
by the morphism (\ref{eqn:mppomrrn}) is zero. 
\end{rema}

\begin{exem}\label{exem:nsc}
In this example we consider (\ref{prop:pslt}) in the nontrivial simplest case 
to see what happens in (\ref{prop:pslt}). 
\par 
Consider the case $S_0$=$S$ and $X={\cal P}={\mab A}_S(1,1)$ 
with standard coordinates $x$ and $y$. 
Let $X_0$ and $X_1$ be exact closed log subschemes of $X$ defined by 
$x=0$ and $y=0$, respectively. Consider the case $\ul{\lam}=\{0\}$. 
Then $X_{\ul{\lam}}=X_0$ 
and 
$$\Om^{\bul}_{X_0/\os{\circ}{S}}=({\cal O}_S[y]\os{d}{\lo} {\cal O}_S[y]d\log x\oplus  
{\cal O}_S[y]d\log y\os{d}{\lo}{\cal O}_S[y]d\log x\wedge d\log y).$$
The complexes $P_k\Om^{\bul}_{X_0/\os{\circ}{S}}$ for $k=0,1,2$ are as follows: 
$$P_0\Om^{\bul}_{X_0/\os{\circ}{S}}=({\cal O}_S[y]\os{d}{\lo} {\cal O}_S[y]dy),$$
$$P_1\Om^{\bul}_{X_0/\os{\circ}{S}}
=({\cal O}_S[y]\os{d}{\lo} {\cal O}_S[y]d\log x\oplus {\cal O}_S[y]d\log y 
\os{d}{\lo} {\cal O}_S[y]d\log x\wedge dy),$$
$$P_2\Om^{\bul}_{X_0/\os{\circ}{S}}=({\cal O}_S[y]\os{d}{\lo} {\cal O}_S[y]d\log x\oplus  
{\cal O}_S[y]d\log y\os{d}{\lo}{\cal O}_S[y]d\log x\wedge d\log y).$$
The complexes ${\rm gr}^P_k\Om^{\bul}_{X_0/\os{\circ}{S}}$ for $k=0, 1,2$ are as follows: 
$${\rm gr}^P_0\Om^{\bul}_{X_0/\os{\circ}{S}}=a_{0*}(\Om^{\bul}_{\os{\circ}{X}_0/\os{\circ}{S}}),$$
\begin{align*} 
{\rm gr}^P_1\Om^{\bul}_{X_0/\os{\circ}{S}}
&=a_{01*}({\cal O}_S)[-1]\oplus 
(0\lo {\cal O}_S[y]d\log x\os{d}{\lo} {\cal O}_S[y]d\log x\wedge dy)\tag{5.11.1}\label{ali:grp} \\
&\simeq 
a_{01*}(\Om^{\bul}_{\os{\circ}{X}_{01}/\os{\circ}{S}})[-1]
\oplus (0\lo {\cal O}_S[y]\os{d}{\lo} {\cal O}_S[y] dy)\\
&=a_{01*}(\Om^{\bul}_{\os{\circ}{X}_{01}/\os{\circ}{S}})[-1]
\oplus a_{0*}(\Om^{\bul}_{\os{\circ}{X}_0/\os{\circ}{S}})[-1],
\end{align*} 
$${\rm gr}^P_2\Om^{\bul}_{X_0/\os{\circ}{S}}
=a_{01*}(\Om^{\bul}_{\os{\circ}{X}_{01}/\os{\circ}{S}})[-2].$$
Here note that the isomorphism in (\ref{ali:grp}) 
is independent of the choice of the local coordinate 
$x$ since $d\log w\in P_0\Om^1_{X_0/\os{\circ}{S}}$ 
for a local section $w$ of $({\cal O}_S[y])^*$. 
\end{exem} 

\begin{prop}\label{prop:flt}
The sheaves 
${\cal E}\otimes_{{\cal O}_{{\cal P}^{\rm ex}}}
{\Om}^i_{{\cal P}^{\rm ex}_{\ul{\lam}}/\os{\circ}{T}}$,  
${\rm gr}^P_k
({\cal E}\otimes_{{\cal O}_{{\cal P}^{\rm ex}}}
{\Om}^i_{{\cal P}^{\rm ex}_{\ul{\lam}}/\os{\circ}{T}})$ 
$(i, k\in {\mab N})$ 
and 
$P_k
({\cal E}\otimes_{{\cal O}_{{\cal P}^{\rm ex}}}
{\Om}^i_{{\cal P}^{\rm ex}_{\ul{\lam}}/\os{\circ}{T}})$ 
are locally free ${\cal O}_T$-modules.    
\end{prop}
\begin{proof}  
We may assume that ${\cal E}$ is trivial. 
When $\ul{\lam}=\emptyset$, we have already proved 
(\ref{prop:flt}) in \cite[(1.3.15)]{nb}. 
Assume that $\ul{\lam}\not=\emptyset$. 
Since the question is local on $X$, 
we may assume
that $X={\mab A}_{S}(a,d)$ and 
${\cal P}^{\rm prex}={\cal P}
= {\mab A}_{S(T)^{\nat}}(a,d')$ 
$(a\in {\mab N}, d\leq d')$ by (\ref{prop:adla})
(cf.~the proof of \cite[(1.3.5)]{nb}). 
Then it is clear that 
${\cal O}_{\mathfrak E}\otimes_{{\cal O}_{{\cal P}^{\rm ex}}}
\Om^i_{{\cal P}^{\rm ex}_{\ul{\lam}}/\os{\circ}{T}}$ 
is a locally free ${\cal O}_T$-module. 
The target of the isomorphism (\ref{eqn:mpprrn}) 
consists of locally free ${\cal O}_T$-modules.
Hence the descending induction on $k$ shows 
the locally freeness of 
$P_k({\cal O}_{\mathfrak E}\otimes_{{\cal O}_{{\cal P}^{\rm ex}}}
{\Om}^i_{{\cal P}^{\rm ex}_{\ul{\lam}}/\os{\circ}{T}})$.  
\end{proof}

\par
For an element $\ul{\lam}$ of $P(\Lam)\cup \{\emptyset\}$,    
we define a filtration $P$ on 
${\cal E}\otimes_{{\cal O}_{{\cal P}^{{\rm ex}}}}
\Om^{\bul}_{{\cal P}^{{\rm ex}}_{\ul{\lam}}/\os{\circ}{T}}\langle u\rangle$ 
for $k\in {\mab Z}$ as follows: 
\begin{equation*} 
P_k({\cal E} \otimes_{{\cal O}_{{\cal P}{}^{{\rm ex}}}}
\Om^{\bul}_{{\cal P}{}^{{\rm ex}}_{\ul{\lam}}/\os{\circ}{T}}\langle u \rangle )
:=\bigoplus_{j\geq 0}\Gam_{{\cal O}_T,j}(U_{S(T)^{\nat}})\otimes_{{\cal O}_T}P_{k-2j}({\cal E}
\otimes_{{\cal O}_{{\cal P}{}^{{\rm ex}}}}
\Om^{\bul}_{{\cal P}{}^{{\rm ex}}_{\ul{\lam}}/\os{\circ}{T}}). 
\tag{5.12.1}\label{ali:fpplp}
\end{equation*} 
Note that $P_{-1}({\cal E} \otimes_{{\cal O}_{{\cal P}{}^{{\rm ex}}}}
\Om^{\bul}_{{\cal P}{}^{{\rm ex}}_{\ul{\lam}}/\os{\circ}{T}}\langle u\rangle )=0$. 
\par 
For an integer $m\geq 0$, 
we define a pre-weight filtration $P$ on 
${\cal E} \otimes_{{\cal O}_{{\cal P}^{\rm ex}}}
\Om^{\bul}_{{\cal P}^{{\rm ex},(m)}/\os{\circ}{T}}\langle u \rangle $ 
for $k\in {\mab Z}$ as follows: 
\begin{equation*} 
P_k({\cal E} 
\otimes_{{\cal O}_{{\cal P}{}^{{\rm ex}}}}
\Om^{\bul}_{{\cal P}{}^{{\rm ex},(m)}/\os{\circ}{T}}\langle u \rangle )
:=\bigoplus_{j\geq 0}\Gam_{{\cal O}_T,j}(U_{S(T)^{\nat}})\otimes_{{\cal O}_T}
P_{k-2j}({\cal E}\otimes_{{\cal O}_{{\cal P}{}^{{\rm ex}}}}
\Om^{\bul}_{{\cal P}{}^{{\rm ex},(m)}/\os{\circ}{T}}). 
\tag{5.12.2}\label{eqn:dfhp}
\end{equation*} 
Indeed we easily see that the differential  
$\nabla$ preserves the filtrations $P$'s on 
${\cal E} 
\otimes_{{\cal O}_{{\cal P}{}^{\rm ex}}}
\Om^{\bul}_{{\cal P}^{{\rm ex},(m)}/\os{\circ}{T}}\langle u \rangle$ 
and 
${\cal E} \otimes_{{\cal O}_{{\cal P}^{\rm ex}}}
\Om^{\bul}_{{\cal P}^{{\rm ex},(m)}/\os{\circ}{T}}\langle u \rangle
$. 
Obviously, 
\begin{align*}
P_{-1}({\cal E} 
\otimes_{{\cal O}_{{\cal P}{}^{{\rm ex},(m)}}}
\Om^{\bul}_{{\cal P}{}^{{\rm ex},(m)}/\os{\circ}{T}}\langle u \rangle)
=0.
\tag{5.12.3}\label{ali:obp}
\end{align*} 
\begin{rema} 
In the case where ${\cal E}$ is trivial, 
the filtration $P$ on 
${\cal E} 
\otimes_{{\cal O}_{{\cal P}{}^{{\rm ex}}}}
\Om^{\bul}_{{\cal P}{}^{{\rm ex}}/\os{\circ}{T}}\langle u \rangle$ 
is useless for the calculation of the pre-weight filtration on 
the log crystalline cohomology 
$R^qf_{X_{\os{\circ}{T}_0/S(T)^{\nat}}}({\cal O}_{X_{\os{\circ}{T}_0/S(T)^{\nat}}})$ 
because 
the $E_1$-term of the associated spectral sequence 
of $R^qf_{X_{\os{\circ}{T}_0/S(T)^{\nat}}}({\cal O}_{X_{\os{\circ}{T}_0/S(T)^{\nat}}})$ 
to the filtration $P$ is ``mixed'' ((\ref{eqn:cptle})). 
(Note that I do not say that the filtration $P$ on ${\cal E} 
\otimes_{{\cal O}_{{\cal P}{}^{{\rm ex}}}}
\Om^{\bul}_{{\cal P}{}^{{\rm ex}}/\os{\circ}{T}}\langle u \rangle$ 
itself is useless. Indeed, 
we use this $P$ in (\ref{prop:tee}) below.) 
\end{rema}

\par 
In the case $\ul{\lam}\not=\emptyset$, 
we have the following isomorphism for any integer $k$ by (\ref{prop:pslt}) 
since the differential of $u^{[i]}$ induces the zero morphism on 
${\rm gr}_k^P({\cal E}
\otimes_{{\cal O}_{{\cal P}^{\rm ex}}}
\Om^{\bul}_{{\cal P}^{{\rm ex}}_{\ul{\lam}}/\os{\circ}{T}}\langle u \rangle )$ by (\ref{eqn:dfhp}): 
\begin{align*} 
&{\rm gr}_k^P({\cal E}
\otimes_{{\cal O}_{{\cal P}^{\rm ex}}}
\Om^{\bul}_{{\cal P}^{{\rm ex}}_{\ul{\lam}}/\os{\circ}{T}}\langle u \rangle )
\os{\sim}{\lo} \tag{5.13.1}\label{eqn:mpruzrn}\\
&\bigoplus_{j\geq 0}
\bigoplus_{\# \ul{\mu}=k-2j}
(\Gam_{{\cal O}_T,j}(U_{S(T)^{\nat}})\otimes_{{\cal O}_T}
{\cal E}
\otimes_{{\cal O}_{{\cal P}^{\rm ex}}}
b_{\ul{\lam}\cup \ul{\mu}*}
(\Om^{\bul}_{\os{\circ}{\cal P}{}^{{\rm ex}}_{\ul{\lam}
\cup \ul{\mu}}/\os{\circ}{T}}
\otimes_{\mab Z}\vp_{{\rm zar},\ul{\mu}}
(\os{\circ}{\cal P}{}^{\rm ex}/\os{\circ}{T})))[-(k-2j)].  
\end{align*} 
Hence 
\begin{align*} 
&{\rm gr}_k^P(
{\cal E}
\otimes_{{\cal O}_{{\cal P}^{\rm ex}}}
\Om^{\bul}_{{\cal P}^{{\rm ex},(m)}/\os{\circ}{T}}\langle u \rangle) 
\os{\sim}{\lo} 
\bigoplus_{\# \ul{\lam}=m+1}
{\rm gr}_k^P(
{\cal E}
\otimes_{{\cal O}_{{\cal P}^{\rm ex}}}
\Om^{\bul}_{{\cal P}^{{\rm ex}}_{\ul{\lam}}/\os{\circ}{T}}\langle u \rangle)
\tag{5.13.2}\label{eqn:mpnfruzrn}\\
\os{\sim}{\lo} 
&\bigoplus_{\# \ul{\lam}=m+1}
\bigoplus_{j\geq 0}
\bigoplus_{\# \ul{\mu}=k-2j}
(\Gam_{{\cal O}_T,j}(U_{S(T)^{\nat}})\otimes_{{\cal O}_T}{\cal E}
\otimes_{{\cal O}_{{\cal P}^{\rm ex}}}
b_{\ul{\lam}\cup \ul{\mu}*}
(\Om^{\bul -k}_{\os{\circ}{\cal P}{}^{{\rm ex}}_{\ul{\lam}
\cup \ul{\mu}}/\os{\circ}{T}}
\otimes_{\mab Z}\vp_{{\rm zar},\ul{\mu}}
(\os{\circ}{\cal P}{}^{\rm ex}/\os{\circ}{T})))[-(k-2j)]. 
\end{align*}

\parno 
Consider the following double complex 
\begin{equation*}
{\cal E}
\otimes_{{\cal O}_{{\cal P}^{\rm ex}}}
\Om^{\bul}_{{\cal P}^{{\rm ex},(\bul)}/\os{\circ}{T}}\langle u \rangle
\end{equation*}  
(see the Notation (15)) and the single complex 
$s(
{\cal E}
\otimes_{{\cal O}_{{\cal P}^{\rm ex}}}
\Om^{\bul}_{{\cal P}^{{\rm ex},(\bul)}/\os{\circ}{T}}\langle u \rangle)$.  
Let $\del(L,P)$ be the diagonal filtration 
on 
$s({\cal E}
\otimes_{{\cal O}_{{\cal P}^{\rm ex}}}
\Om^{\bul}_{{\cal P}^{{\rm ex},(\bul)}/\os{\circ}{T}}\langle u \rangle)$ 
(\cite{dh2}) 
obtained by $P$ shifted by a semi-cosimplicial degree: 
\begin{equation*} 
\del(L,P)_ks({\cal E}
\otimes_{{\cal O}_{{\cal P}^{\rm ex}}}
\Om^{\bul}_{{\cal P}^{{\rm ex},(\bul)}/\os{\circ}{T}}\langle u \rangle):=
(\cdots \lo P_{k+m}
({\cal E}
\otimes_{{\cal O}_{{\cal P}^{\rm ex}}}
\Om^{\bul}_{{\cal P}^{{\rm ex},(m)}/\os{\circ}{T}}\langle u \rangle)\lo \cdots). 
\end{equation*} 
For simplicity of notation, we denote $\del(L,P)$ simply by $P$. 
If $\os{\circ}{X}_{T_0}$ is quasi-compact, then 
\begin{align*}
P_ks({\cal E}
\otimes_{{\cal O}_{{\cal P}^{\rm ex}}}
\Om^{\bul}_{{\cal P}^{{\rm ex},(\bul)}/\os{\circ}{T}}\langle u \rangle)
=
\del(L,P)_ks({\cal E}
\otimes_{{\cal O}_{{\cal P}^{\rm ex}}}
\Om^{\bul}_{{\cal P}^{{\rm ex},(\bul)}/\os{\circ}{T}}\langle u \rangle)
=0
\end{align*} 
for $k<<0$. 
Since the horizontal boundary morphism 
\begin{align*} 
d\log t \wedge \col & {\rm gr}_{k+m}^P
({\cal E}
\otimes_{{\cal O}_{{\cal P}^{\rm ex}}}
\Om^{\bul}_{{\cal P}^{{\rm ex},(m)}/\os{\circ}{T}}\langle u \rangle) \lo \\ 
& {\rm gr}_{k+m+1}^P
({\cal E}
\otimes_{{\cal O}_{{\cal P}^{\rm ex}}}
\Om^{\bul}_{{\cal P}^{{\rm ex},(m)}/\os{\circ}{T}}\langle u \rangle) 
\quad (m\in {\mab Z}_{\geq 0})
\end{align*}  
of the double complex ${\rm gr}_k^{P}
({\cal E}
\otimes_{{\cal O}_{{\cal P}^{\rm ex}}}
\Om^{\bul}_{{\cal P}^{{\rm ex},(\bul)}/\os{\circ}{T}}\langle u \rangle)$
is zero, 
we obtain the following isomorphism 
\begin{align*} 
&{\rm gr}_k^{P}
(s({\cal E}
\otimes_{{\cal O}_{{\cal P}^{\rm ex}}}
\Om^{\bul}_{{\cal P}^{{\rm ex},(\bul)}/\os{\circ}{T}}\langle u \rangle)) 
\os{\sim}{\lo} \tag{5.13.3}\label{ali:grc}\\
&\bigoplus_{m\geq 0}
\bigoplus_{\# \ul{\lam}=m+1}
\bigoplus_{j\geq 0}
\bigoplus_{\# \ul{\mu}=k+m-2j}
(\Gam_{{\cal O}_T,j}(U_{S(T)^{\nat}})\otimes_{{\cal O}_T}{\cal E}
\otimes_{{\cal O}_{\cal P}^{\rm ex}}
b_{\ul{\lam}\cup \ul{\mu}*}
(\Om^{\bul}_{\os{\circ}{\cal P}{}^{{\rm ex}}_{\ul{\lam}
\cup \ul{\mu}}/\os{\circ}{T}} \\
& \otimes_{\mab Z}\vp_{{\rm zar},\ul{\mu}}
({\cal P}^{\rm ex}/T)))
[-k-2m+2j]. 
\end{align*} 
Note that, locally on $X_{\os{\circ}{T}_0}$, 
the right hand side of (\ref{ali:grc}) is a finite direct sum. 
If $(X_{\os{\circ}{T}_0})^{\circ}$ is quasi-compact, 
then the right hand side is a finite direct sum. 
Note also that 
$v^*(u^{[j]})={\rm deg}(v)^ju^{[j]}$ for a morphism $v\col S\lo S'$ of family of log points 
((\ref{ali:vua})). The last note will become important later in this book in the consideration of 
the ``$D$-twist'', especially, the twist of the abrelative Frobenius morphism. 
\par 

\begin{prop}\label{prop:lcz}
Locally on $X_{\os{\circ}{T}_0}$, 
$P_k(s({\cal E}
\otimes_{{\cal O}_{{\cal P}^{\rm ex}}}
\Om^{\bul}_{{\cal P}^{{\rm ex},(\bul)}/\os{\circ}{T}}\langle u \rangle))=0$ 
for $k<<0$. 
If $\os{\circ}{X}_{T_0}$ is quasi-compact, then 
$P_k
(s({\cal E}
\otimes_{{\cal O}_{{\cal P}^{\rm ex}}}
\Om^{\bul}_{{\cal P}^{{\rm ex},(\bul)}/\os{\circ}{T}}\langle u \rangle))=0$ 
for $k<<0$. That is, the filtration $P$ on 
$s({\cal E}
\otimes_{{\cal O}_{{\cal P}^{\rm ex}}}
\Om^{\bul}_{{\cal P}^{{\rm ex},(\bul)}/\os{\circ}{T}}\langle u \rangle)$
is bounded below. 
\end{prop}
\begin{proof} 
This follows from (\ref{ali:obp}). 
\end{proof}

\begin{rema}\label{rema:ex}
It is clear that the filtration $P$ on 
$s({\cal E}
\otimes_{{\cal O}_{{\cal P}^{\rm ex}}}
\Om^{\bul}_{{\cal P}^{{\rm ex},(\bul)}/\os{\circ}{T}}\langle u \rangle)$ 
is exhaustive because we allow only finite sum with respect to $u^{[j]}$ $(j\in {\mab N})$ locally. 
This observation is very important for the existence of the convergent spectral sequence 
(\ref{ali:spsarh}) below.  
This is why we cannot use 
$s({\cal E}
\otimes_{{\cal O}_{{\cal P}^{\rm ex}}}
\Om^{\bul}_{{\cal P}^{{\rm ex},(\bul)}/\os{\circ}{T}}\langle \langle u \rangle \rangle)$.  
\end{rema}

\section{PD-Hirsch pre-weight-filtered log crystalline complexes}\label{sec:psc}
Let $(T,{\cal J},\del)$, $T_0\lo S$, $Y/S$ and $Y_{\os{\circ}{T}_0}$ 
be as in \S\ref{sec:ldfc}. 
In this section we do not assume that there exists an immersion 
from $Y_{\os{\circ}{T}_0}$ into a log smooth scheme over $\ol{S(T)^{\nat}}$. 
Let $g\col Y_{\os{\circ}{T}_0} \lo S_{\os{\circ}{T}_0}\lo S(T)^{\nat}$ be the structural morphism.  
Let $\ol{F}$ be a flat quasi-coherent crystal of 
${\cal O}_{Y_{\os{\circ}{T}_0}/\os{\circ}{T}}$-modules.  
The first aim in this section is 
to recall a complex 
\begin{equation*} 
\wt{R}u_{Y_{\os{\circ}{T}_0}/\os{\circ}{T}*}(\ol{F})
\in D^+(g^{-1}({\cal O}_T)) 
\end{equation*}  
and 
to construct a new complex
\begin{equation*} 
\wt{R}u_{Y_{\os{\circ}{T}_0}/\os{\circ}{T}*}(\ol{F}\langle u \rangle)
\in D^+(g^{-1}({\cal O}_T)), 
\end{equation*}  
which we call $\wt{R}u_{Y_{\os{\circ}{T}_0}/\os{\circ}{T}*}(\ol{F}
\langle u \rangle)$ the {\it derived PD-Hirsch extension} of 
$\wt{R}u_{Y_{\os{\circ}{T}_0}/\os{\circ}{T}*}(\ol{F})$. 
Let $X/S$ be an SNCL scheme 
and let $f\col X_{\os{\circ}{T}_0}\lo S(T)^{\nat}$ be the 
structural morphism. 
Let $E$ be a flat quasi-coherent crystal of 
${\cal O}_{\os{\circ}{X}_{T_0}/\os{\circ}{T}}$-modules.  
The second aim in this section is to construct a filtered complex
\begin{equation*} 
(H_{\rm zar}(X_{\os{\circ}{T}_0}/S(T)^{\nat},E),P):=
(\wt{R}u_{X^{(\star)}_{\os{\circ}{T}_0}/\os{\circ}{T}*}
(\eps^*_{X^{(\star)}_{\os{\circ}{T}_0}/\os{\circ}{T}}(E^{(\star)})\langle u \rangle),P)
\in {\rm D}^+{\rm F}(f^{-1}_T({\cal O}_T)),
\end{equation*}  
which we call the {\it PD-Hirsch filtered crystalline complex} of $E$. 
\par  
Let $Y_{\os{\circ}{T}_0\bul}$ be the \v{C}ech diagram of 
an affine open covering of $Y_{\os{\circ}{T}_0}$. 
By abuse of notation, denote the structural morphism 
$Y_{\os{\circ}{T}_0\bul}\lo S(T)^{\nat}$ by $g_{\bul}$. 
Then, by a standard method using (\ref{lemm:etl}), 
we have a simplicial immersion 
$Y_{\os{\circ}{T}_0\bul} \os{\sus}{\lo} \ol{\cal Q}_{\bul}$ into 
a simplicial log smooth scheme over the immersion
$S_{\os{\circ}{T}_0}\os{\sus}{\lo} \ol{S(T)^{\nat}}$ if 
each member of the affine open covering is small enough.  
By taking the exactification of the immersion 
$Y_{\os{\circ}{T}_0\bul} \os{\sus}{\lo} \ol{\cal Q}_{\bul}$, 
we have the simplicial immersion 
$Y_{\os{\circ}{T}_0\bul} \os{\sus}{\lo} \ol{\cal Q}{}^{\rm ex}_{\bul}$.  
Let $\ol{\mathfrak E}_{\bul}$ be the log PD-envelope of 
this immersion over $(\os{\circ}{T},{\cal J},\del)$. 
Set ${\cal Q}^{\rm ex}_{\bul}
:=\ol{\cal Q}{}^{\rm ex}_{\bul}\times_{\ol{S(T)^{\nat}}}S(T)^{\nat}$ 
and 
${\mathfrak E}_{\bul}
:=\ol{\mathfrak E}_{\bul}\times_{\ol{S(T)^{\nat}}}S(T)^{\nat}$.   
Let $\ol{F}{}^{\bul}$ be the quasi-coherent 
${\cal O}_{Y_{\os{\circ}{T}_0\bul}/\os{\circ}{T}}$-module obtained by $\ol{F}$. 
Let $(\ol{\cal F}{}^{\bul},\ol{\nabla})$ 
be the crystal of ${\cal O}_{\ol{\mathfrak E}_{\bul}}$-module  
with integrable connection corresponding to 
the log crystal $\ol{F}$: 
\begin{equation*} 
\ol{\nabla}\col \ol{\cal F}{}^{\bul}\lo 
\ol{\cal F}{}^{\bul}\otimes_{{\cal O}_{\ol{\cal Q}{}^{\rm ex}_{\bul}}}
\Om^1_{\ol{\cal Q}{}^{\rm ex}_{\bul}/\os{\circ}{T}}.
\tag{6.0.1}\label{eqn:olbc}
\end{equation*}  
Set $({\cal F}^{\bul},{\nabla}):=
(\ol{\cal F}{}^{\bul},\ol{\nabla})
\otimes_{{\cal O}_{{\mathfrak D}(\ol{S(T)^{\nat}})}}{\cal O}_{S(T)^{\nat}}$:  
\begin{equation*} 
\nabla\col {\cal F}^{\bul}\lo 
{\cal F}^{\bul}\otimes_{{\cal O}_{{\cal Q}^{\rm ex}_{\bul}}}
\Om^1_{{\cal Q}^{\rm ex}_{\bul}/\os{\circ}{T}}. 
\tag{6.0.2}\label{eqn:olbab}
\end{equation*}  
As in (\ref{eqn:olfmpc}) we indeed have the connection $\nabla$.

\par 
For $U=S(T)^{\nat}$ or $U=\os{\circ}{T}$, 
we have the cosimplicial log de Rham complexes 
$\ol{\cal F}^{\bul}
\otimes_{{\cal O}_{\ol{\cal Q}{}^{\rm ex}_{\bul}}}
\Om^{\bul}_{\ol{\cal Q}{}^{\rm ex}_{\bul}/U}$ and 
${\cal F}^{\bul}
\otimes_{{\cal O}_{{\cal Q}^{\rm ex}_{\bul}}}
\Om^{\bul}_{{\cal Q}^{\rm ex}_{\bul}/U}$ as in \S\ref{sec:ldfc}.
\par 
Let 
\begin{equation*}  
\eps_{Y_{\os{\circ}{T}_0}/S(T)^{\nat}/\os{\circ}{T}}\col 
((Y_{\os{\circ}{T}_0}/S(T)^{\nat})_{\rm crys},
{\cal O}_{Y_{\os{\circ}{T}_0}/S(T)^{\nat}})\lo 
((Y_{\os{\circ}{T}_0}/\os{\circ}{T})_{\rm crys},
{\cal O}_{Y_{\os{\circ}{T}_0}/\os{\circ}{T}})
\tag{6.0.3}\label{eqn:tnspi}  
\end{equation*} 
be the morphism of ringed topoi forgetting the log structure of $S(T)^{\nat}$. 
Set 
$F:=\eps_{Y_{\os{\circ}{T}_0}/S(T)^{\nat}/\os{\circ}{T}}^*(\ol{F})$.  
Let 
\begin{align*}
\pi_{\rm zar}\col 
((\os{\circ}{Y}_{T_0\bul})_{\rm zar},g^{-1}_{\bul}({\cal O}_T))
\lo 
((\os{\circ}{Y}_{T_0})_{\rm zar},g^{-1}({\cal O}_T))
\tag{6.0.4}\label{eqn:ispi} 
\end{align*}
be the natural morphism of ringed topoi. 
In \cite[(1.7.27.4)]{nb} we have proved that 
there exists the following triangle 
\begin{align*} 
\lo 
R\pi_{{\rm zar}*}({\cal F}^{\bul}\otimes_{{\cal O}_{{\cal Q}^{\rm ex}_{\bul}}}
{\Om}^{\bul}_{{\cal Q}^{\rm ex}_{\bul}/\os{\circ}{T}})
\lo Ru_{Y_{\os{\circ}{T}_0}/S(T)^{\nat}*}(F) 
\os{N_{\rm zar}}{\lo}Ru_{Y_{\os{\circ}{T}_0}/S(T)^{\nat}*}(F), 
\tag{6.0.5}\label{ali:crtn}
\end{align*}  
where $N_{\rm zar} \col Ru_{Y_{\os{\circ}{T}_0}/S(T)^{\nat}*}(F)
\lo Ru_{Y_{\os{\circ}{T}_0}/S(T)^{\nat}*}(F)$ 
is the monodromy operator defined in [loc.~cit.], which will be recalled in \S\ref{sec:mod}. 
In \cite[(1.7.27)]{nb} we have proved the following:

\begin{prop-defi}[{\bf \cite[(1.7.27)]{nb}}]\label{prop:hkt}
The complex $R\pi_{{\rm zar}*}({\cal F}^{\bul}
\otimes_{{\cal O}_{{\cal Q}^{\rm ex}_{\bul}}}
{\Om}^{\bul}_{{\cal Q}^{\rm ex}_{\bul}/\os{\circ}{T}})$ 
is independent of the choice of 
an open covering of $Y_{\os{\circ}{T}_0}$ and a simplicial immersion 
$Y_{\os{\circ}{T}_0\bul} \os{\sus}{\lo} \ol{\cal Q}_{\bul}$ 
over $\ol{S(T)^{\nat}}$. 
Set 
\begin{align*} 
&\wt{R}u_{Y_{\os{\circ}{T}_0}/\os{\circ}{T}*}(\ol{F})
:=R\pi_{{\rm zar}*}
({\cal F}^{\bul}\otimes_{{\cal O}_{{\cal Q}{}^{{\rm ex}}_{\bul}}}
\Om^{\bul}_{{\cal Q}{}^{{\rm ex}}_{\bul}/\os{\circ}{T}})
\in D^+(g^{-1}({\cal O}_T)). 
\end{align*} 
In \cite{nb} we have called 
$\wt{R}u_{Y_{\os{\circ}{T}_0}/\os{\circ}{T}*}(\ol{F})$ 
the {\it modified crystalline complex} of $\ol{F}$ 
on $Y_{\os{\circ}{T}_0}/S(T)^{\nat}$. 
When $\ol{F}={\cal O}_{Y_{\os{\circ}{T}_0}/\os{\circ}{T}}$, 
we have called 
$\wt{R}u_{Y_{\os{\circ}{T}_0}/\os{\circ}{T}*}({\cal O}_{Y_{\os{\circ}{T}_0}/\os{\circ}{T}})$ 
 the {\it modified crystalline complex} of 
$Y_{\os{\circ}{T}_0}/S(T)^{\nat}/\os{\circ}{T}$.  
\end{prop-defi}

\begin{defi} 
We denote $Rg_*\wt{R}u_{Y_{\os{\circ}{T}_0}/\os{\circ}{T}*}(\ol{F})$ 
by $\wt{R}g_{Y_{\os{\circ}{T}_0}/\os{\circ}{T}*}(\ol{F})$. 
We denote ${\cal H}^q(\wt{R}g_{Y_{\os{\circ}{T}_0}/\os{\circ}{T}*}(\ol{F}))$ 
by 
$\wt{R}{}^qg_{Y_{\os{\circ}{T}_0}/\os{\circ}{T}*}(\ol{F})$. 
\end{defi} 

\begin{defi}\label{defi:mfrt}
(1) Set 
\begin{align*} 
&\wt{R}u_{Y_{\os{\circ}{T}_0}/\os{\circ}{T}*}(\ol{F}\langle u \rangle)
:=R\pi_{{\rm zar}*}
({\cal F}^{\bul}\otimes_{{\cal O}_{{\cal Q}^{{\rm ex}}_{\bul}}}
\Om^{\bul}_{{\cal Q}^{{\rm ex}}_{\bul}/\os{\circ}{T}}\langle u \rangle)
\in D^+(g^{-1}({\cal O}_T)). 
\end{align*} 
We call $\wt{R}u_{Y_{\os{\circ}{T}_0}/\os{\circ}{T}*}(\ol{F}\langle u \rangle)$ 
the {\it derived PD-Hirsch extension} 
of $\wt{R}u_{Y_{\os{\circ}{T}_0}/\os{\circ}{T}*}(\ol{F})$. 
\end{defi}

\par
The following is a starting result for investigating log crystalline cohomological sheaves; 
we see that the complex 
$\wt{R}u_{Y_{\os{\circ}{T}_0}/\os{\circ}{T}*}(\ol{F}\langle u \rangle)$ 
is indeed well-defined. 

\begin{prop}\label{prop:ncfqi}
$(1)$ There exists a canonical isomorphism 
\begin{equation*}
\wt{R}u_{Y_{\os{\circ}{T}_0}/\os{\circ}{T}*}(\ol{F}\langle u \rangle)
\os{\sim}{\lo} Ru_{Y_{\os{\circ}{T}_0}/S(T)^{\nat}*}(F). 
\tag{6.4.1}\label{eqn:exfte}
\end{equation*} 
In particular, $\wt{R}u_{Y_{\os{\circ}{T}_0}/\os{\circ}{T}*}(\ol{F}\langle u \rangle)$ 
is independent of the choice of an open covering of $Y_{\os{\circ}{T}_0}$ and 
a simplicial immersion 
$Y_{\os{\circ}{T}_0\bul} \os{\sus}{\lo} \ol{\cal Q}_{\bul}$ 
over $\ol{S(T)^{\nat}}$. 
The isomorphism {\rm (\ref{eqn:exfte})} is contravariantly functorial. 
\end{prop}
\begin{proof}
By the cohomological descent and the log Poincar\'{e} lemma, 
we have the following formula 
\begin{equation*} 
Ru_{Y_{\os{\circ}{T}_0}/S(T)^{\nat}*}(F)
=R\pi_{{\rm zar}*}({\cal F}^{\bul}\otimes_{{\cal O}_{{\cal Q}^{\rm ex}_{\bul}}}
\Om^{\bul}_{{\cal Q}^{\rm ex}_{\bul}/S(T)^{\nat}}).
\tag{6.4.2}\label{eqn:exsfpte}
\end{equation*}
By (\ref{theo:qii}) we see that the natural morphism 
\begin{equation*} 
{\cal F}^{\bul}\otimes_{{\cal O}_{{\cal Q}{}^{{\rm ex}}_{\bul}}}
\Om^{\bul}_{{\cal Q}{}^{{\rm ex}}_{\bul}/\os{\circ}{T}}\langle u \rangle
\lo 
{\cal F}^{\bul}\otimes_{{\cal O}_{{\cal Q}{}^{{\rm ex}}_{\bul}}}
\Om^{\bul}_{{\cal Q}{}^{{\rm ex}}_{\bul}/S(T)^{\nat}}
\tag{6.4.3}\label{eqn:pfbu}
\end{equation*} 
is a quasi-isomorphism. 
By taking the derived direct image $R\pi_{{\rm zar}*}$ of (\ref{eqn:pfbu}),  
we have the isomorphism (\ref{eqn:exfte}). 
\par 
Let $Y'_{\os{\circ}{T}_0\bul}$ be the \v{C}ech diagram of 
another affine open covering of $Y_{\os{\circ}{T}_0}$ 
and let $Y'_{\os{\circ}{T}_0\bul}\os{\sus}{\lo} \ol{\cal Q}{}'_{\bul}$ 
be another simplicial immersion into 
a log smooth scheme over $\ol{S(T)^{\nat}}$.  
Set $\ol{\cal Q}{}''_{\bul}
:=\ol{\cal Q}{}_{\bul}\times_{\ol{S(T)^{\nat}}}\ol{\cal Q}{}'_{\bul}$. 
Then, by considering the refinement of these affine open coverings, 
we see that there exists the \v{C}ech diagram $Y''_{\os{\circ}{T}_0\bul}$ 
of an affine open covering of $Y_{\os{\circ}{T}_0}$ fitting into 
the following commutative diagram 
\begin{equation*} 
\begin{CD} 
Y_{\os{\circ}{T}_0\bul} @>{\subset}>> \ol{\cal Q}_{\bul}\\ 
@AAA @AAA \\
Y''_{\os{\circ}{T}_0\bul} @>{\subset}>> \ol{\cal Q}{}''_{\bul}\\ 
@VVV @VVV \\
Y'_{\os{\circ}{T}_0\bul} @>{\subset}>> \ol{\cal Q}{}'_{\bul}. 
\end{CD} 
\end{equation*}
Hence we may assume that there exists the following commutative diagram: 
\begin{equation*} 
\begin{CD} 
Y_{\os{\circ}{T}_0\bul} @>{\subset}>> \ol{\cal Q}_{\bul}\\ 
@VVV @VVV \\
Y'_{\os{\circ}{T}_0\bul} @>{\subset}>> \ol{\cal Q}{}'_{\bul}. 
\end{CD} 
\tag{6.4.4}\label{cd:soq}
\end{equation*}
Let $({\cal F}'{}^{\bul},\nabla')$ be an analogous connection to 
$({\cal F}^{\bul},\nabla)$ for the immersion 
$Y'_{\os{\circ}{T}_0\bul} \os{\subset}{\lo} \ol{\cal Q}{}'_{\bul}$. 
The diagram (\ref{cd:soq}) gives us the following commutative diagram 
\begin{equation*} 
\begin{CD} 
{\cal F}^{\bul}\otimes_{{\cal O}_{{\cal Q}{}^{{\rm ex}}_{\bul}}}
\Om^{\bul}_{{\cal Q}{}^{{\rm ex}}_{\bul}/\os{\circ}{T}}\langle u \rangle 
@>{\sim}>> {\cal F}^{\bul}\otimes_{{\cal O}_{{\cal Q}{}^{{\rm ex}}_{\bul}}}
\Om^{\bul}_{{\cal Q}{}^{{\rm ex}}_{\bul}/S(T)^{\nat}}\\ 
@AAA @AAA \\
{\cal F}'{}^{\bul}\otimes_{{\cal O}_{{\cal Q}_{\bul}'{}^{{\rm ex}}}}
\Om^{\bul}_{{\cal Q}_{\bul}'{}^{{\rm ex}}/\os{\circ}{T}}\langle u \rangle 
@>{\sim}>> {\cal F}'{}^{\bul}\otimes_{{\cal O}_{{\cal Q}_{\bul}'{}^{{\rm ex}}}}
\Om^{\bul}_{{\cal Q}_{\bul}'{}^{{\rm ex}}/S(T)^{\nat}}. 
\end{CD} 
\tag{6.4.5}\label{cd:sotcq}
\end{equation*}
The commutative diagram $R\pi_{{\rm zar}*}((\ref{cd:sotcq}))$ 
tells us the independence of 
$\wt{R}u_{Y_{\os{\circ}{T}_0}/\os{\circ}{T}*}(\ol{F}\langle u \rangle)$. 
\par 
We leave the proof of the contravariant functoriality of the isomorphism 
(\ref{eqn:exfte}) to the reader. 
\end{proof}

Let us consider the complex 
\begin{align*} 
R\pi_{{\rm zar}*}({\cal F}^{\bul}\otimes_{{\cal O}_{{\cal Q}^{{\rm ex}}_{\bul}}}
\Om^{\bul}_{{\cal Q}^{{\rm ex}}_{\bul}/\os{\circ}{T}})\langle u \rangle^L:=
R\pi_{{\rm zar}*}({\cal F}^{\bul}\otimes_{{\cal O}_{{\cal Q}^{{\rm ex}}_{\bul}}}
\Om^{\bul}_{{\cal Q}^{{\rm ex}}_{\bul}/\os{\circ}{T}})\langle U_{S(T)^{\nat}} \rangle^L
\in {\rm Ho}({\cal C}_{(Y_{\rm zar},g^{-1}({\cal O}_T))}), 
\tag{6.4.6}\label{ali:sytcq}
\end{align*} 
which is a special case of the complex in (\ref{ali:mld}). 

\begin{prop}\label{prop:dpdr}
There exists a canonical isomorphism 
\begin{equation*}
\wt{R}u_{Y_{\os{\circ}{T}_0}/\os{\circ}{T}*}(\ol{F}\langle u \rangle)
\os{\sim}{\lo} R\pi_{{\rm zar}*}({\cal F}^{\bul}\otimes_{{\cal O}_{{\cal Q}^{{\rm ex}}_{\bul}}}
\Om^{\bul}_{{\cal Q}^{{\rm ex}}_{\bul}/\os{\circ}{T}})\langle u \rangle^L
\tag{6.5.1}\label{eqn:exdlfte}
\end{equation*} 
in $D^+(g^{-1}({\cal O}_T))$. Consequently the complex {\rm (\ref{ali:sytcq})} 
is well-defined. 
\end{prop} 
\begin{proof} 
Let ${\cal F}^{\bul}\otimes_{{\cal O}_{{\cal Q}^{{\rm ex}}_{\bul}}}
\Om^{\bul}_{{\cal Q}^{{\rm ex}}_{\bul}/\os{\circ}{T}}\lo I^{\bul \bul \bul}$ 
be the Godement resolution of ${\cal F}^{\bul}\otimes_{{\cal O}_{{\cal Q}^{{\rm ex}}_{\bul}}}
\Om^{\bul}_{{\cal Q}^{{\rm ex}}_{\bul}/\os{\circ}{T}}$. 
Then 
\begin{align*} 
R\pi_{{\rm zar}*}({\cal F}^{\bul}\otimes_{{\cal O}_{{\cal Q}^{{\rm ex}}_{\bul}}}
\Om^{\bul}_{{\cal Q}^{{\rm ex}}_{\bul}/\os{\circ}{T}})\langle U_{S(T)^{\nat}} \rangle^L
=U_{S(T)^{\nat}}\otimes_{{\cal O}_T}\pi_{{\rm zar}*}(s(I^{\bul \bul \bul}))=
\pi_{{\rm zar}*}(U_{S(T)^{\nat}}\otimes_{{\cal O}_T}s(I^{\bul \bul \bul})). 
\end{align*} 
Here we have used (\ref{ali:fim}). 
The complex $U_{S(T)^{\nat}}\otimes_{{\cal O}_T}I^{\bul \bul \bul}$ 
is a $\pi_{{\rm zar}*}$-acyclic resolution of 
${\cal F}^{\bul}\otimes_{{\cal O}_{{\cal Q}^{{\rm ex}}_{\bul}}}
\Om^{\bul}_{{\cal Q}^{{\rm ex}}_{\bul}/\os{\circ}{T}}$ because 
$R^q\pi_{{\rm zar}*}(U_{S(T)^{\nat}}\otimes_{{\cal O}_T}I^{\bul \bul \bul})
=U_{S(T)^{\nat}}\otimes_{{\cal O}_T}R^q\pi_{{\rm zar}*}(I^{\bul \bul \bul})=0$ $(q>0)$. 
By the definition of $\wt{R}u_{Y_{\os{\circ}{T}_0}/\os{\circ}{T}*}(\ol{F}\langle u \rangle)$ and 
this acyclicity, 
\begin{align*} 
\wt{R}u_{Y_{\os{\circ}{T}_0}/\os{\circ}{T}*}(\ol{F}\langle u \rangle)
=R\pi_{{\rm zar}*}
({\cal F}^{\bul}\otimes_{{\cal O}_{{\cal Q}^{{\rm ex}}_{\bul}}}
\Om^{\bul}_{{\cal Q}^{{\rm ex}}_{\bul}/\os{\circ}{T}}\langle u \rangle)=
\pi_{{\rm zar}*}(U_{S(T)^{\nat}}\otimes_{{\cal O}_T}s(I^{\bul \bul \bul})). 
\end{align*} 
\end{proof}

\begin{theo}[{\bf K\"{u}nneth formula}]\label{theo:aoys}
$(1)$ Let $Z/S$ be an analogous object to $Y/S$. 
Assume that $Y_{\os{\circ}{T}_0}$ and $Z_{\os{\circ}{T}_0}$ are quasi-compact. 
Let $\ol{G}$ be an analogous ${\cal O}_{Z_{{\os{\circ}{T}_0}/\os{\circ}{T}}}$-module 
to $\ol{F}$ and set $G:=\eps_{Z_{\os{\circ}{T}_0}/S(T)^{\nat}/\os{\circ}{T}}^*(\ol{G})$.  
Set $YZ_{\os{\circ}{T}_0}:=Y_{\os{\circ}{T}_0}\times_{\os{\circ}{T}_0}
Z_{\os{\circ}{T}_0}$ and 
let $p\col YZ_{\os{\circ}{T}_0}\lo Y_{\os{\circ}{T}_0}$ and 
$q\col YZ_{\os{\circ}{T}_0}\lo Z_{\os{\circ}{T}_0}$ be the projections 
and let $h\col YZ_{\os{\circ}{T}_0}\lo \os{\circ}{T}$ be the structural morphism. 
Then there exists a canonical isomorphism 
\begin{align*}
&
Lp^*(\wt{R}u_{Y_{\os{\circ}{T}_0}/\os{\circ}{T}*}(\ol{F}\langle u \rangle))
\otimes^L_{h^{-1}({\cal O}_T)}
Lq^*(\wt{R}u_{Z_{\os{\circ}{T}_0}/\os{\circ}{T}*}(\ol{G}\langle u \rangle))
\os{\sim}{\lo} \tag{6.6.1}\label{eqn:exuufte}\\
& \wt{R}u_{YZ_{\os{\circ}{T}_0}/\os{\circ}{T}*}((p^*_{\rm crys}(\ol{F})
\otimes_{{\cal O}_{YZ_{\os{\circ}{T}_0}/\os{\circ}{T}}}
q^*_{\rm crys}(\ol{G}))\langle u \rangle)
\end{align*} 
fitting into the following commutative diagram$:$
\begin{equation*}
\begin{CD} 
Lp^*(\wt{R}u_{Y_{\os{\circ}{T}_0}/\os{\circ}{T}*}(\ol{F}\langle u \rangle))
\otimes^L_{h^{-1}({\cal O}_T)}
Lq^*(\wt{R}u_{Z_{\os{\circ}{T}_0}/\os{\circ}{T}*}(\ol{G}\langle u \rangle))
@>{\sim}>> 
\\
@V{Lp^*(\ref{eqn:exfte})\otimes^LLq^*(\ref{eqn:exfte})}V{\simeq}V \\
Lp^*(Ru_{Y_{\os{\circ}{T}_0}/S(T)^{\nat}*}(F))
\otimes^L_{h^{-1}({\cal O}_T)}
Lq^*(Ru_{Z_{\os{\circ}{T}_0}/S(T)^{\nat}*}(G))
@>{\sim}>>  
\end{CD} 
\tag{6.6.2}\label{eqn:exfcte}
\end{equation*} 
\begin{equation*}
\begin{CD} 
\wt{R}u_{YZ_{\os{\circ}{T}_0}/\os{\circ}{T}*}((p^*_{\rm crys}(\ol{F})
\otimes_{{\cal O}_{YZ_{\os{\circ}{T}_0}/\os{\circ}{T}}}
q^*_{\rm crys}(\ol{G}))\langle u \rangle)\\
@V{\simeq}V{(\ref{eqn:exfte})}V \\
Ru_{YZ_{\os{\circ}{T}_0}/S(T)^{\nat}*}(p^*_{\rm crys}(F)
\otimes_{{\cal O}_{YZ_{\os{\circ}{T}_0}/\os{\circ}{T}}}
q^*_{\rm crys}(G)). 
\end{CD} 
\end{equation*} 
\par 
$(2)$ Let the notations be as in $(1)$. 
Assume that $Z_{\os{\circ}{T}_0}=Y_{\os{\circ}{T}_0}$. 
There exists the  following morphism 
\begin{align*}
\cup \col \wt{R}u_{Y_{\os{\circ}{T}_0}/\os{\circ}{T}*}
(\ol{F}\langle u \rangle)
\otimes^L_{g^{-1}({\cal O}_T)}
\wt{R}u_{Y_{\os{\circ}{T}_0}/\os{\circ}{T}*}
(\ol{G}\langle u \rangle)
\lo 
\wt{R}u_{Y_{\os{\circ}{T}_0}/\os{\circ}{T}*}
((\ol{F}\otimes_{{\cal O}_{Y_{\os{\circ}{T}_0}
/\os{\circ}{T}}}\ol{G})\langle u \rangle)
\tag{6.6.3}\label{eqn:efte}\\
\end{align*} 
fitting into the following commutative diagram$:$
\begin{equation*}
\begin{CD} 
\wt{R}u_{Y_{\os{\circ}{T}_0}/\os{\circ}{T}*}(\ol{F}\langle u \rangle)
\otimes^L_{g^{-1}({\cal O}_T)}\wt{R}u_{Y_{\os{\circ}{T}_0}/\os{\circ}{T}*}
(\ol{G}\langle u \rangle)@>{\cup}>> 
\wt{R}u_{Y_{\os{\circ}{T}_0}/\os{\circ}{T}*}((\ol{F}\otimes_{{\cal O}_{Y_{\os{\circ}{T}_0}}}
\ol{G})\langle u \rangle)\\
@V{(\ref{eqn:exfte})\otimes^L(\ref{eqn:exfte})}V{\simeq}V 
@V{\simeq}V{(\ref{eqn:exfte})}V \\
Ru_{Y_{\os{\circ}{T}_0}/\os{\circ}{T}*}(F)
\otimes^L_{g^{-1}({\cal O}_T)}
Ru_{Y_{\os{\circ}{T}_0}/\os{\circ}{T}*}(G)@>{\cup}>>  
Ru_{Y_{\os{\circ}{T}_0}/\os{\circ}{T}*}(F\otimes_{{\cal O}_{Y_{\os{\circ}{T}_0}}}G). 
\end{CD} 
\tag{6.6.4}\label{eqn:exfacbte}
\end{equation*} 
Here the lower $\cup$ is the cup product of log crystalline complexes. 
\end{theo} 
\begin{proof} 
(1): Let $W$ be $Y_{\os{\circ}{T}_0}$, $Z_{\os{\circ}{T}_0}$ or 
$YZ_{\os{\circ}{T}_0}$. Let $g_{W\bul}\col W_{\bul}\lo T$ 
and $g_W\col W\lo T$ be the structural morphism. 
Let $\pi_{W,{\rm zar}}\col (W_{\bul})_{\rm zar}\lo W_{\rm zar}$ be the 
morphism (\ref{eqn:ispi}).   
Consider the following commutative diagram: 
\begin{equation*} 
\begin{CD} 
((Y_{\os{\circ}{T}_0\bul})_{\rm zar},g^{-1}_{Y_{\os{\circ}{T}_0\bul}}({\cal O}_T))
@<{p_{\bul}}<<(((YZ_{\os{\circ}{T}_0\bul})_{\rm zar},g^{-1}_{YZ_{\os{\circ}{T}_0\bul}}({\cal O}_T)) 
@>{q_{\bul}}>>
((Z_{\os{\circ}{T}_0\bul})_{\rm zar},g^{-1}_{Z_{\os{\circ}{T}_0\bul}}({\cal O}_T)) \\ 
@V{\pi_{Y,{\rm zar}}}VV @V{\pi_{YZ,{\rm zar}}}VV @V{\pi_{Z,{\rm zar}}}VV\\
((Y_{\os{\circ}{T}_0})_{\rm zar},g^{-1}_{Y_{\os{\circ}{T}_0}}({\cal O}_T))
@<{p}<<((YZ_{\os{\circ}{T}_0})_{\rm zar},g^{-1}_{YZ_{\os{\circ}{T}_0}}({\cal O}_T)) 
@>{q}>>((Z_{\os{\circ}{T}_0})_{\rm zar},g^{-1}_{Z_{\os{\circ}{T}_0}}({\cal O}_T)). 
\end{CD} 
\tag{6.6.5}\label{cd:sotzcq}
\end{equation*}
Let ${\cal R}^{\rm ex}_{\bul}$ and ${\cal G}^{\bul}$ be 
the corresponding objects to 
${\cal Q}^{\rm ex}_{\bul}$ and ${\cal F}^{\bul}$, respectively, 
for $Z$ and $\ol{G}$ . 
Set ${\cal QR}^{\rm ex}_{\bul}:=
{\cal Q}^{\rm ex}_{\bul}\times_{\os{\circ}{T}}{\cal R}^{\rm ex}_{\bul}$ 
and 
let $h_{\bul} \col {\cal QR}^{\rm ex}_{\bul}\lo \os{\circ}{T}$ be the structural morphism. 
To prove the existence of the morphism (\ref{eqn:exuufte}), it suffices to prove that 
there exists the following morphism by using the relations 
$p \circ \pi_{YZ,{\rm zar}}=\pi_{Y,{\rm zar}}\circ p_{\bul}$ 
and $q \circ \pi_{YZ,{\rm zar}}=\pi_{Z,{\rm zar}}\circ q_{\bul}$
in a standard way: 
\begin{align*} 
&p_{\bul}^*({\cal F}^{\bul}\otimes_{{\cal O}_{{\cal Q}^{{\rm ex}}_{\bul}}}
\Om^{\bul}_{{\cal Q}^{{\rm ex}}_{\bul}/\os{\circ}{T}}\langle u \rangle)
\otimes_{h^{-1}_{\bul}({\cal O}_T)}
q_{\bul}^*({\cal G}^{\bul}\otimes_{{\cal O}_{{\cal R}^{{\rm ex}}_{\bul}}}
\Om^{\bul}_{{\cal R}^{{\rm ex}}_{\bul}/\os{\circ}{T}}\langle u \rangle)
\tag{6.6.6}\label{ali:soqrqq}\\
& \lo 
{{\cal O}_{{\cal QR}^{{\rm ex}}_{\bul}}}\otimes_{{\cal O}_{{\cal Q}^{{\rm ex}}_{\bul}}}
{\cal F}^{\bul}\otimes_{{\cal O}_{{\cal QR}^{{\rm ex}}_{\bul}}}
{{\cal O}_{{\cal QR}^{{\rm ex}}_{\bul}}}\otimes_{{\cal O}_{{\cal R}^{{\rm ex}}_{\bul}}}
{\cal G}^{\bul}
\otimes_{\otimes_{{\cal O}_{{\cal QR}^{{\rm ex}}_{\bul}}}}
\Om^{\bul}_{{\cal QR}^{{\rm ex}}_{\bul}/\os{\circ}{T}}\langle u \rangle.
\end{align*} 
Using the morphism 
$$p_{\bul}^*(\Om^i_{{\cal Q}^{{\rm ex}}_{\bul}/\os{\circ}{T}})
\otimes_{{\cal O}_{{\cal QR}^{{\rm ex}}_{\bul}}}
q_{\bul}^*(\Om^j_{{\cal R}^{{\rm ex}}_{\bul}/\os{\circ}{T}})
\lo
\Om^{i+j}_{{\cal QR}^{\rm ex}_{\bul}/\os{\circ}{T}} \quad (i,j\in {\mab N})$$  
and the multiplicative structure 
\begin{align*} 
{\cal O}_T\langle u \rangle 
\otimes_{{\cal O}_T}{\cal O}_T\langle u\rangle \owns 
u^{[i]}\otimes u^{[i']}\lom 
\begin{pmatrix}
i+i' \\
i
\end{pmatrix}u^{[i+i']}\in {\cal O}_T\langle u\rangle. 
\tag{6.6.7}\label{eqn:dups}
\end{align*} 
of $\Gam_{{\cal O}_T}(U_{S(T)^{\nat}})$, 
we have the morphism (\ref{ali:soqrqq}). 
It is clear that 
the diagram (\ref{eqn:exfcte}) with the replacement of 
the upper horizontal $\os{\sim}{\lo}$ by $\lo$ 
is commutative. 
Because we see that the lower horizontal morphism 
in (\ref{eqn:exfcte}) is an isomorphism by 
local calculations as in \cite[p.~379]{bb}, 
the morphism 
(\ref{eqn:exuufte}) is an isomorphism. 
(Note that because (\ref{eqn:dups}) is not an isomorphism of ${\cal O}_T$-modules at all, 
we cannot see that (\ref{eqn:exuufte}) is an isomorphism only by local calculations for 
the source and the target of the morphism (\ref{eqn:exuufte}).)
\par 
(2): 
Applying the derived pull-back $L\Del^*$ of the diagonal immersion 
$\Del \col Y_{\os{\circ}{T}_0}
\os{\sus}{\lo} YY_{\os{\circ}{T}_0}:=
Y_{\os{\circ}{T}_0}\times_{\os{\circ}{T}_0}Y_{\os{\circ}{T}_0}$ over $\os{\circ}{T}$
to  the morphism (\ref{eqn:exuufte}), 
we obtain the following isomorphism 
\begin{align*}
&
\wt{R}u_{Y_{\os{\circ}{T}_0}/\os{\circ}{T}*}(\ol{F}\langle u \rangle))
\otimes^L_{g^{-1}({\cal O}_T)}
\wt{R}u_{Y_{\os{\circ}{T}_0}/\os{\circ}{T}*}(\ol{G}\langle u \rangle))
\os{\sim}{\lo} \tag{6.6.8}\label{eqn:exusfte}\\
& L\Del^*\wt{R}u_{YY_{\os{\circ}{T}_0}/\os{\circ}{T}*}((p^*_{\rm crys}(\ol{F})
\otimes_{{\cal O}_{YY_{\os{\circ}{T}_0}/\os{\circ}{T}}}
q^*_{\rm crys}(\ol{G}))\langle u \rangle)\\
&= L\Del^*R\pi_{YY_{\os{\circ}{T}_0}*}
({{\cal O}_{{\cal QQ}^{{\rm ex}}_{\bul}}}\otimes_{{\cal O}_{{\cal Q}^{{\rm ex}}_{\bul}}}
{\cal F}^{\bul}\otimes_{{\cal O}_{{\cal Q}^{{\rm ex}}_{\bul}}}{\cal G}^{\bul}
\otimes_{{{\cal O}_{{\cal QQ}^{{\rm ex}}_{\bul}}}}
\Om^{\bul}_{{\cal QQ}^{{\rm ex}}_{\bul}/\os{\circ}{T}}\langle u \rangle). 
\end{align*} 
Consider the following commutative diagram of topoi: 
\begin{equation*}
\begin{CD}
Y_{\os{\circ}{T}_0,\bul,{\rm zar}}@>{\Del_{\bul}}>> YY_{\os{\circ}{T}_0,\bul,\rm zar}\\
@V{\pi_{Y_{\os{\circ}{T}_0}}}VV  @VV{\pi_{YY_{\os{\circ}{T}_0}}}V\\
Y_{\os{\circ}{T}_0,{\rm zar}}@>{\Del}>> 
YY_{\os{\circ}{T}_0,{\rm zar}}.
\end{CD}
\tag{6.6.9}\label{cd:coymp}
\end{equation*}
Using this diagram and the flatness of ${\cal F}^{\bul}$ over 
$g^{-1}({\cal O}_T)$, we obtain the following morphism as usual: 
\begin{align*}
&L\Del^*R\pi_{YY_{\os{\circ}{T}_0}*}
({{\cal O}_{{\cal QQ}^{{\rm ex}}_{\bul}}}\otimes_{{\cal O}_{{\cal Q}^{{\rm ex}}_{\bul}}}
{\cal F}^{\bul}\otimes_{{\cal O}_{{\cal Q}^{{\rm ex}}_{\bul}}}{\cal G}^{\bul}
\otimes_{{\cal O}_{{\cal QQ}^{{\rm ex}}_{\bul}}}
\Om^{\bul}_{{\cal QQ}^{{\rm ex}}_{\bul}/\os{\circ}{T}}\langle u \rangle)
\tag{6.6.10}\label{eqn:exusafte}\\
\lo 
&R\pi_{Y_{\os{\circ}{T}_0}*}L\Del^*_{\bul}
({{\cal O}_{{\cal QQ}^{{\rm ex}}_{\bul}}}\otimes_{{\cal O}_{{\cal Q}^{{\rm ex}}_{\bul}}}
{\cal F}^{\bul}\otimes_{{\cal O}_{{\cal Q}^{{\rm ex}}_{\bul}}}{\cal G}^{\bul}
\otimes_{{\cal O}_{{\cal QQ}^{{\rm ex}}_{\bul}}}
\Om^{\bul}_{{\cal QQ}^{{\rm ex}}_{\bul}/\os{\circ}{T}}\langle u \rangle)\\
&=
R\pi_{Y_{\os{\circ}{T}_0}*}({\cal F}^{\bul}\otimes_{{\cal O}_{{\cal Q}^{{\rm ex}}_{\bul}}}{\cal G}^{\bul}
\otimes_{{\cal O}_{{\cal Q}^{{\rm ex}}_{\bul}}}
\Om^{\bul}_{{\cal Q}^{{\rm ex}}_{\bul}/\os{\circ}{T}}\langle u \rangle). 
\end{align*} 
The composite morphism of the morphisms (\ref{eqn:exusfte}) and (\ref{eqn:exusafte}) 
gives us the desired morphism (\ref{eqn:efte}).
\end{proof}

\par 
Now let $X/S$ be an SNCL scheme 
and let $f\col X_{\os{\circ}{T}_0}\lo S(T)^{\nat}$ be the 
structural morphism. 
Let $X_{\os{\circ}{T}_0\bul}$ be the simplicial log scheme 
obtained by an affine open covering of $X_{\os{\circ}{T}_0}$. 
Then, by (\ref{prop:xbn}) (1), 
we have a simplicial immersion 
$X_{\os{\circ}{T}_0\bul} \os{\sus}{\lo} \ol{\cal P}_{\bul}$ into 
a simplicial log smooth scheme over $\ol{S(T)^{\nat}}$. 
Then, by taking the exactification of this simplicial immersion, 
we have the simplicial immersion 
$X_{\os{\circ}{T}_0\bul} \os{\sus}{\lo} \ol{\cal P}{}^{\rm ex}_{\bul}$.  
Let $\ol{\mathfrak D}_{\bul}$ be the log PD-envelope of 
this immersion over $(\os{\circ}{T},{\cal J},\del)$. 
Set ${\cal P}^{\rm ex}_{\bul}
:=\ol{\cal P}{}^{\rm ex}_{\bul}\times_{\ol{S(T)^{\nat}}}S(T)^{\nat}$, 
${\mathfrak D}_{\bul}
:=\ol{\mathfrak D}_{\bul}\times_{\ol{S(T)^{\nat}}}S(T)^{\nat}$. 
\par 
Let $E$ be a flat quasi-coherent crystal of 
${\cal O}_{\os{\circ}{X}_{\os{\circ}{T}_0}/\os{\circ}{T}}$-modules.  
Let $E^{\bul}$ be the crystal of 
${\cal O}_{\os{\circ}{X}_{T_0\bul}/\os{\circ}{T}}$-modules 
obtained by $E$. 
Let $(\ol{\cal E}{}^{\bul},\ol{\nabla})$ 
($\ol{\nabla} \col \ol{\cal E}{}^{\bul}\lo \ol{\cal E}{}^{\bul}
\otimes_{{\cal O}_{\ol{\cal P}{}^{\rm ex}_{\bul}}}
\Om^1_{\ol{\cal P}{}^{\rm ex}_{\bul}/\os{\circ}{T}}$) 
be the quasi-coherent 
${\cal O}_{\ol{\mathfrak D}_{\bul}}$-module 
with integrable connection corresponding to 
$\eps^*_{X_{\os{\circ}{T}_0}/\os{\circ}{T}}(E^{\bul})$. 
Set 
$({\cal E}^{\bul},\nabla)
=
(\ol{\cal E}{}^{\bul},\ol{\nabla})
\otimes_{{\cal O}_{\ol{\mathfrak D}_{\bul}}}
{\cal O}_{{\mathfrak D}_{\bul}}$.  
 For $U=S(T)^{\nat}$ or $U=\os{\circ}{T}$, 
we have the log de Rham complexes 
$\ol{\cal E}{}^{\bul}
\otimes_{{\cal O}_{\ol{\cal P}{}^{\rm ex}_{\bul}}}
\Om^{\bul}_{\ol{\cal P}{}^{\rm ex}_{\bul}/U}$, 
${\cal E}^{\bul}\otimes_{{\cal O}_{{\cal P}^{\rm ex}_{\bul}}}
\Om^{\bul}_{{\cal P}^{\rm ex}_{\bul}/U}$, 
${\cal E}^{\bul}\otimes_{{\cal O}_{{\cal P}^{\rm ex}_{\bul}}}
\Om^{\bul}_{{\cal P}^{\rm ex}_{\bul \ul{\lam}}/U}$
and 
${\cal E}^{\bul}\otimes_{{\cal O}_{{\cal P}^{\rm ex}}}
\Om^{\bul}_{{\cal P}^{{\rm ex},(m)}_{\bul}/U}$ 
for $m\in{\mab Z}_{\geq 0}$ as in the previous section.
\par 
Let 
\begin{align*}
\pi_{{\rm zar}}\col 
((\os{\circ}{X}_{T_0\bul})_{\rm zar},
f^{-1}_{\bul}({\cal O}_T))
\lo 
((\os{\circ}{X}_{T_0})_{\rm zar},f^{-1}({\cal O}_T))
\end{align*}
be the natural morphism of ringed topoi. 
\par 
Let $\eps_{X^{(m)}_{\os{\circ}{T}_0}/\os{\circ}{T}_0} \col 
X^{(m)}_{\os{\circ}{T}_0} \lo \os{\circ}{X}{}^{(m)}_{T_0}$ 
be the morphism forgetting the log structure of 
$X^{(m)}_{\os{\circ}{T}_0}/S_{\os{\circ}{T}_0}$.  
For $U=S(T)^{\nat}$ or $\os{\circ}{T}$, let 
\begin{align*}
\eps_{X^{(m)}_{\os{\circ}{T}_0}/U} 
\col ((X^{(m)}_{\os{\circ}{T}_0}/U)_{\rm crys},
{\cal O}_{X^{(m)}_{\os{\circ}{T}_0}/U})
\lo ((\os{\circ}{X}{}^{(m)}_{T_0}/\os{\circ}{T})_{\rm crys},
{\cal O}_{\os{\circ}{X}{}^{(m)}_{T_0}/\os{\circ}{T}}) 
\end{align*} 
be the morphism forgetting the log structure.     
Let $E^{(m)}$ be the restriction of 
$E$ to $\os{\circ}{X}{}^{(m)}_{T_0}$.  
\par
To define 
\begin{equation*} 
(\wt{R}u_{X^{(\star)}_{\os{\circ}{T}_0}/\os{\circ}{T}*}
(\eps^*_{X^{(\star)}_{\os{\circ}{T}_0}/\os{\circ}{T}}(E^{(\star)})\langle u \rangle),P)
\in {\rm D}^+{\rm F}(f^{-1}({\cal O}_T)), 
\end{equation*} 
we need the following lemma: 

\begin{lemm}\label{lemm:sss}
The family $\{{\cal P}^{{\rm ex},(k)}_n\}_{k,n\in {\mab N}}$ of log formal schemes 
forms a semi-simplicial and simplicial family of log formal schemes with respect to 
the upper degrees and the lower degrees, respectively. 
This semi-simplicial and simplicial log scheme 
${\cal P}^{{\rm ex},(\bul)}_{\bul}$ 
is augmented to the simplicial log scheme 
${\cal P}^{\rm ex}_{\bul}$.
\end{lemm}
\begin{proof} 
First we have to construct a morphism 
${\cal P}^{{\rm ex},(k)}_n\lo {\cal P}^{{\rm ex},(k)}_{n'}$ 
for a morphism $[n']\lo [n]$ in the standard simplicial category $\Del$. 
This construction is the same as that of \cite[(1.4.2)]{nb}. 
Next we have to construct a morphism 
${\cal P}_{\ul{\lam},n}\lo {\cal P}_{\ul{\lam}',n}$ 
for a morphism $\ul{\lam} \lo \ul{\lam}'$ in $P(\Lam)$.  
This is obvious. Hence we have a natural semi-simplicial and simplicial log scheme 
${\cal P}^{{\rm ex},(\bul)}_{\bul}$.  
Because we have a natural morphism 
${\cal P}^{{\rm ex},(0)}_n\lo {\cal P}^{{\rm ex}}_n$, 
${\cal P}^{{\rm ex},(\bul)}_{\bul}$ is augmented to 
${\cal P}^{{\rm ex}}_{\bul}$. 
\end{proof}

\par 
Let 
\begin{align*}
\pi_{{\rm zar}}\col 
((\os{\circ}{X}_{\bul,T_0})_{\rm zar},
f^{-1}_{\bul}({\cal O}_T))
\lo 
((\os{\circ}{X}_{T_0})_{\rm zar},f^{-1}({\cal O}_T))
\end{align*}
be the natural morphism of ringed topoi. 

\begin{defi}\label{defi:mrt}
(1) Set 
\begin{align*} 
&\wt{R}u_{X_{\os{\circ}{T}_0}/\os{\circ}{T}*}
(\eps^*_{X_{\os{\circ}{T}_0}/\os{\circ}{T}}(E))
:=R\pi_{{\rm zar}*}({\cal E}^{\bul}
\otimes_{{\cal O}_{{\cal P}_{\bul}^{{\rm ex}}}}
\Om^{\bul}_{{\cal P}_{\bul}^{\rm ex}/\os{\circ}{T}})\in 
D^+(f^{-1}({\cal O}_T)). 
\end{align*} 
\par 
(2) Set 
\begin{align*} 
&\wt{R}u_{X^{(\star)}_{\os{\circ}{T}_0}/\os{\circ}{T}*}
(\eps^*_{X^{(\star)}_{\os{\circ}{T}_0}/\os{\circ}{T}}(E^{(\star)}))
:=R\pi_{{\rm zar}*}(s({\cal E}^{\bul}
\otimes_{{\cal O}_{{\cal P}_{\bul}^{\rm ex}}}
\Om^{\bul}_{{\cal P}{}^{{\rm ex},(\bul)}_{\bul}/\os{\circ}{T}}))
\in D^+(f^{-1}({\cal O}_T)). 
\end{align*} 
\par
(3)   
Set 
\begin{equation*} 
Ru_{X^{(\star)}_{\os{\circ}{T}_0}/S(T)^{\nat}*}
(\eps^*_{X^{(\star)}_{\os{\circ}{T}_0}/S(T)^{\nat}}(E^{(\star)}))
:=R\pi_{{\rm zar}*}(s({\cal E}^{\bul}
\otimes_{{\cal O}_{{\cal P}_{\bul}^{\rm ex}}}
\Om^{\bul}_{{\cal P}{}^{{\rm ex},(\bul)}_{\bul}/S(T)^{\nat}}))
\in D^+(f^{-1}({\cal O}_T)). 
\end{equation*} 
\par 
(4) 
Set 
\begin{align*} 
&(\wt{R}u_{X_{\os{\circ}{T}_0}/\os{\circ}{T}*}
(\eps^*_{X_{\os{\circ}{T}_0}/\os{\circ}{T}}(E)\langle u \rangle),P)
:=R\pi_{{\rm zar}*}
(({\cal E}^{\bul}\otimes_{{\cal O}_{{\cal P}^{{\rm ex}}_{\bul}}}
\Om^{\bul}_{{\cal P}^{{\rm ex}}_{\bul}/\os{\circ}{T}}\langle u \rangle,P))
\\
&\in D^+(f^{-1}({\cal O}_T)). 
\end{align*} 
We call $\wt{R}u_{X_{\os{\circ}{T}_0}/\os{\circ}{T}*}
(\eps^*_{X_{\os{\circ}{T}_0}/\os{\circ}{T}}(E)\langle u \rangle)$ 
the {\it derived PD-Hirsch extension} of 
$\wt{R}u_{X_{\os{\circ}{T}_0}/\os{\circ}{T}*}
(\eps^*_{X_{\os{\circ}{T}_0}/\os{\circ}{T}}(E))$.  
\par 
(5) Set 
\begin{align*} 
(H_{\rm zar}(X_{\os{\circ}{T}_0}/S(T)^{\nat},E),P):=
&(\wt{R}u_{X^{(\star)}_{\os{\circ}{T}_0}/\os{\circ}{T}*}
(\eps^*_{X^{(\star)}_{\os{\circ}{T}_0}/\os{\circ}{T}}(E^{(\star)})\langle u \rangle), P)\\
&:=R\pi_{{\rm zar}*}
(s({\cal E}^{\bul}
\otimes_{{\cal O}_{{\cal P}_{\bul}^{\rm ex}}}
\Om^{\bul}_{{\cal P}{}^{{\rm ex},(\bul)}_{\bul}/\os{\circ}{T}}\langle u \rangle),P))\\
&\in{\rm D}^+{\rm F}(f^{-1}({\cal O}_T)). 
\end{align*} 
We call 
$(H_{\rm zar}(X_{\os{\circ}{T}_0}/S(T)^{\nat},E),P)$ 
the {\it PD-Hirsch filtered crystalline complex} of $E$. 
(Recall that we denote the diagonal filtration $\del(L,P)$ simply by $P$ 
in the last equality of the definition of $(H_{\rm zar}(X_{\os{\circ}{T}_0}/S(T)^{\nat},E),P)$.) 
\end{defi}

\begin{rema}\label{rema:wd}
(1) By (\ref{prop:hkt}) the complex $\wt{R}u_{X_{\os{\circ}{T}_0}/\os{\circ}{T}*}
(\eps^*_{X_{\os{\circ}{T}_0}/\os{\circ}{T}}(E))$ is well-defined. 
The complexes 
$\wt{R}u_{X^{(\star)}_{\os{\circ}{T}_0}/\os{\circ}{T}*}
(\eps^*_{X^{(\star)}_{\os{\circ}{T}_0}/\os{\circ}{T}}(E^{(\star)}))$, 
$Ru_{X^{(\star)}_{\os{\circ}{T}_0}/S(T)^{\nat}*}
(\eps^*_{X^{(\star)}_{\os{\circ}{T}_0}/S(T)^{\nat}}(E^{(\star)}))$ 
and 
$\wt{R}u_{X_{\os{\circ}{T}_0}/\os{\circ}{T}*}
(\eps^*_{X_{\os{\circ}{T}_0}/\os{\circ}{T}}(E)\langle u \rangle)$ 
will be checked to be well-defined 
((\ref{prop:naqi}), (\ref{prop:nbqi}), (\ref{prop:ncqi})). 
The filtered complex   
$(H_{\rm zar}(X_{\os{\circ}{T}_0}/S(T)^{\nat},E),P)$ 
will also be checked to be well-defined ((\ref{theo:indp})). 
\par 
(2) Strictly speaking, in the definitions of (\ref{defi:mrt}) (2), (3) and (5), 
we have to take the direct image 
${\cal P}{}^{{\rm ex},(\bul)}_{\bul}\lo {\cal P}{}^{\rm ex}_{\bul}$  
for 
$\Om^{\bul}_{{\cal P}{}^{{\rm ex},(\bul)}_{\bul}/\os{\circ}{T}}$ 
and 
$\Om^{\bul}_{{\cal P}{}^{{\rm ex},(\bul)}_{\bul}/S(T)^{\nat}}$.  
\end{rema}


Consider the case where $E$ is trivial, which will be important in this book: 

\begin{defi}\label{defi:cez} 
(1) Set 
\begin{align*} 
\wt{R}u_{X_{\os{\circ}{T}_0}/\os{\circ}{T}*}
({\cal O}_{X_{\os{\circ}{T}_0}/\os{\circ}{T}})
:=\wt{R}u_{X_{\os{\circ}{T}_0}/\os{\circ}{T}*}
(\eps^*_{X_{\os{\circ}{T}_0}/\os{\circ}{T}}({\cal O}_{\os{\circ}{X}_{T_0}/\os{\circ}{T}})). 
\end{align*} 
\par 
(2) Set 
\begin{align*} 
\wt{R}u_{X^{(\star)}_{\os{\circ}{T}_0}/\os{\circ}{T}*}
({\cal O}_{X^{(\star)}_{\os{\circ}{T}_0}/\os{\circ}{T}})
:=\wt{R}u_{X^{(\star)}_{\os{\circ}{T}_0}/\os{\circ}{T}*}
(\eps^*_{X^{(\star)}_{\os{\circ}{T}_0}/\os{\circ}{T}}
({\cal O}_{\os{\circ}{X}{}^{(\star)}_{T_0}/\os{\circ}{T}})). 
\end{align*} 
\par
(3)   
Set 
\begin{equation*} 
Ru_{X^{(\star)}_{\os{\circ}{T}_0}/S(T)^{\nat}*}
({\cal O}_{X^{(\star)}_{\os{\circ}{T}_0}/S(T)^{\nat}})
:=
Ru_{X^{(\star)}_{\os{\circ}{T}_0}/S(T)^{\nat}*}
(\eps^*_{X^{(\star)}_{\os{\circ}{T}_0}/S(T)^{\nat}}
({\cal O}_{\os{\circ}{X}{}^{(\star)}_{T_0}/\os{\circ}{T}})). 
\end{equation*} 
\par 
(4) 
Set 
\begin{align*} 
\wt{R}u_{X_{\os{\circ}{T}_0}/\os{\circ}{T}*}
({\cal O}_{X_{\os{\circ}{T}_0}/\os{\circ}{T}}\langle u \rangle)
:=\wt{R}u_{X_{\os{\circ}{T}_0}/\os{\circ}{T}*}
(\eps^*_{X_{\os{\circ}{T}_0}/\os{\circ}{T}}
({\cal O}_{\os{\circ}{X}_{T_0}/\os{\circ}{T}})\langle u \rangle). 
\end{align*} 
We call $\wt{R}u_{X_{\os{\circ}{T}_0}/\os{\circ}{T}*}
({\cal O}_{X_{\os{\circ}{T}_0}/\os{\circ}{T}}\langle u \rangle)$ 
the {\it modified Hirsch crystalline complex} of $X_{\os{\circ}{T}_0}/\os{\circ}{T}$. 
\par 
(5) Set 
\begin{align*} 
(H_{\rm zar}(X_{\os{\circ}{T}_0}/S(T)^{\nat}),P):=
&(\wt{R}u_{X^{(\star)}_{\os{\circ}{T}_0}/\os{\circ}{T}*}
({\cal O}_{X^{(\star)}_{\os{\circ}{T}_0}/\os{\circ}{T}}\langle u \rangle), P). 
\end{align*} 
We call 
$(H_{\rm zar}(X_{\os{\circ}{T}_0}/S(T)^{\nat}),P)$
the {\it Hirsch pre-weight-filtered log crystalline complex} 
of $X_{\os{\circ}{T}_0}/S(T)^{\nat}/\os{\circ}{T}$. 
\end{defi}


\begin{prop}\label{prop:naqi}
$(1)$ There exists a canonical isomorphism 
\begin{equation*}
Ru_{X_{\os{\circ}{T}_0}/S(T)^{\nat}*}
(\eps^*_{X_{\os{\circ}{T}_0}/S(T)^{\nat}}(E)) \os{\sim}{\lo} 
Ru_{X^{(\star)}_{\os{\circ}{T}_0}/S(T)^{\nat}*}
(\eps^*_{X^{(\star)}_{\os{\circ}{T}_0}/S(T)^{\nat}}(E^{(\star)})). 
\tag{6.11.1}\label{eqn:eetxte}
\end{equation*} 
In particular, 
$Ru_{X^{(\star)}_{\os{\circ}{T}_0}/S(T)^{\nat}*}
(\eps^*_{X^{(\star)}_{\os{\circ}{T}_0}/S(T)^{\nat}}(E^{(\star)}))$ is independent of 
the choice of an affine open covering of $X_{\os{\circ}{T}_0}$ 
and a simplicial immersion $X_{\os{\circ}{T}_0\bul}\os{\sus}{\lo} 
\ol{\cal P}_{\bul}$ over $\ol{S(T)^{\nat}}$. 
\par 
$(2)$ There exists a canonical isomorphism 
\begin{equation*}
\wt{R}u_{X_{\os{\circ}{T}_0}/\os{\circ}{T}*}
(\eps^*_{X_{\os{\circ}{T}_0}/\os{\circ}{T}} 
(E)) \os{\sim}{\lo} 
\wt{R}u_{X^{(\star)}_{\os{\circ}{T}_0}/\os{\circ}{T}*}
(\eps^*_{X^{(\star)}_{\os{\circ}{T}_0}/\os{\circ}{T}} 
(E^{(\star)})). 
\tag{6.11.2}\label{eqn:e}
\end{equation*} 
In particular, 
$\wt{R}u_{X^{(\star)}_{\os{\circ}{T}_0}/\os{\circ}{T}*}
(\eps^*_{X^{(\star)}_{\os{\circ}{T}_0}/\os{\circ}{T}} 
(E^{(\star)}))$ is independent of the choice of an affine open covering of $X_{\os{\circ}{T}_0}$ 
and a simplicial immersion 
$X_{\os{\circ}{T}_0\bul}\os{\sus}{\lo} \ol{\cal P}_{\bul}$ over $\ol{S(T)^{\nat}}$.  
\end{prop}
\begin{proof}
(1), (2): 
Let $X'_{\os{\circ}{T}_0\bul}$ be the \v{C}ech diagram of 
another affine open covering of $X_{\os{\circ}{T}_0}$ 
and let $X'_{\os{\circ}{T}_0\bul}\os{\sus}{\lo} \ol{\cal P}{}'_{\bul}$ 
be another simplicial immersion into 
a log smooth scheme over $\ol{S(T)^{\nat}}$.  
As in (\ref{cd:soq}), 
by considering the refinement of two affine open coverings, 
we may assume that there exists the following commutative diagram: 
\begin{equation*} 
\begin{CD} 
X_{\os{\circ}{T}_0\bul} @>{\subset}>> \ol{\cal P}_{\bul}\\ 
@VVV @VVV \\
X'_{\os{\circ}{T}_0\bul} @>{\subset}>> \ol{\cal P}{}'_{\bul}
\end{CD} 
\tag{6.11.3}\label{cd:sop}
\end{equation*}
and we may assume that 
the morphism $X_{\os{\circ}{T}_0\bul}\lo X'_{\os{\circ}{T}_0\bul}$ 
induces a morphism 
$X_{\ul{\lam},\os{\circ}{T}_0\bul}\lo X'_{\ul{\lam},\os{\circ}{T}_0\bul}$. 
Hence the morphism 
$\ol{\cal P}_{\bul}\lo \ol{\cal P}{}'_{\bul}$ induces a morphism 
${\cal P}_{\bul \ul{\lam}}\lo {\cal P}{}'_{\bul \ul{\lam}}$. 
Set $U:=S(T)^{\nat}$ or $\os{\circ}{T}$. 
Then we obtain the following commutative diagram 
\begin{equation*} 
\begin{CD} 
{\cal E}^{\bul}
\otimes_{{\cal O}_{{\cal P}{}^{\rm ex}_{\bul}}}
\Om^{\bul}_{{\cal P}{}^{\rm ex}_{\bul}/U} @>>> 
s({\cal E}^{\bul}
\otimes_{{\cal O}_{{\cal P}{}^{{\rm ex},(\bul)}_{\bul}}}
\Om^{\bul}_{{\cal P}{}^{{\rm ex},(\bul)}_{\bul}/U}) \\ 
@AAA @AAA \\
{\cal E}'{}^{\bul}
\otimes_{{\cal O}_{{\cal P}_{\bul}'{}^{\rm ex}}}
\Om^{\bul}_{{\cal P}_{\bul}'{}^{\rm ex}/U} @>>> 
s({\cal E}{}'^{\bul}
\otimes_{{\cal O}_{{\cal P}_{\bul}'{}^{{\rm ex},(\bul)}}}
\Om^{\bul}_{{\cal P}_{\bul}'{}^{{\rm ex},(\bul)}/U}). 
\end{CD} 
\tag{6.11.4}\label{cd:sopp}
\end{equation*}
By (\ref{prop:sil}) we see that the horizontal morphisms 
in  (\ref{cd:sopp}) are quasi-isomorphisms.
Now we have only to apply $R\pi_{{\rm zar}*}$ to (\ref{cd:sopp}). 
\end{proof}

\par 
The following proposition tells us that 
the complex 
$\wt{R}u_{X_{\os{\circ}{T}_0}/\os{\circ}{T}*}
(\eps^*_{X_{\os{\circ}{T}_0}/\os{\circ}{T}}(E)\langle u \rangle)$ 
is also independent of 
the choice of an open covering of $X$ 
and a simplicial immersion 
$X_{\os{\circ}{T}_0\bul}\os{\sus}{\lo} \ol{\cal P}_{\bul}$ over $\ol{S(T)^{\nat}}$.

\begin{prop}\label{prop:nbqi}
There exists a canonical isomorphism 
\begin{equation*}
\wt{R}u_{X_{\os{\circ}{T}_0}/\os{\circ}{T}*}
(\eps^*_{X_{\os{\circ}{T}_0}/\os{\circ}{T}}(E)\langle u \rangle)
\os{\sim}{\lo} 
Ru_{X_{\os{\circ}{T}_0}/S(T)^{\nat}*}(\eps^*_{X_{\os{\circ}{T}_0}/S(T)^{\nat}}(E)). 
\tag{6.12.1}\label{eqn:exte}
\end{equation*} 
\end{prop}
\begin{proof} 
This is a special case of (\ref{prop:ncfqi}). 
\end{proof}

\begin{prop}\label{prop:ncqi}
There exists a canonical isomorphism 
\begin{equation*}
H_{\rm zar}(X_{\os{\circ}{T}_0}/S(T)^{\nat},E)=
\wt{R}u_{X^{(\star)}_{T_0}/\os{\circ}{T}*}
(\eps^*_{X^{(\star)}_{\os{\circ}{T}_0}/\os{\circ}{T}}(E^{(\star)})\langle u \rangle) 
\os{\sim}{\lo} 
Ru_{X^{(\star)}_{\os{\circ}{T}_0}/S(T)^{\nat}*}
(\eps^*_{X^{(\star)}_{\os{\circ}{T}_0}/S(T)^{\nat}}(E^{(\star)})).
\tag{6.13.1}\label{eqn:eaxte}
\end{equation*}
In particular, 
$H_{\rm zar}(X_{\os{\circ}{T}_0}/S(T)^{\nat},E)$
is independent of the choice of an affine open covering of $X_{\os{\circ}{T}_0}$ 
and a simplicial immersion 
$X_{\os{\circ}{T}_0\bul}\os{\sus}{\lo} \ol{\cal P}_{\bul}$ over $\ol{S(T)^{\nat}}$.  
\end{prop}
\begin{proof}
By the definition (\ref{defi:mrt}) (3), 
we have the following formula 
\begin{equation*} 
Ru_{X^{(\star)}_{\os{\circ}{T}_0}/S(T)^{\nat}*}
(\eps^*_{X^{(\star)}_{\os{\circ}{T}_0}/S(T)^{\nat}}(E^{(\star)}))
=R\pi_{{\rm zar}*}(s({\cal E}^{\bul}
\otimes_{{\cal O}_{{\cal P}^{{\rm ex},(\bul)}_{\bul}}}
\Om^{\bul}_{{\cal P}^{{\rm ex},{(\bul)}}_{\bul}/S(T)^{\nat}})).
\tag{6.13.2}\label{eqn:exspte}
\end{equation*}
Consider the following natural morphism:  
\begin{equation*} 
s({\cal E}^{\bul}
\otimes_{{\cal O}_{{\cal P}{}^{{\rm ex},(\bul)}_{\bul}}}
\Om^{\bul}_{{\cal P}{}^{{\rm ex},(\bul)}_{\bul}/\os{\circ}{T}}\langle u \rangle)
\lo 
s({\cal E}^{\bul}
\otimes_{{\cal O}_{{\cal P}{}^{{\rm ex},(\bul)}_{\bul}}}
\Om^{\bul}_{{\cal P}{}^{{\rm ex},(\bul)}_{\bul}/S(T)^{\nat}}) 
\tag{6.13.3}\label{eqn:pbu}
\end{equation*}
By (\ref{theo:sih}) we see that this morphism is a quasi-isomorphism. 
Hence we obtain the isomorphism (\ref{eqn:eaxte}). 
We can prove the independence of 
$H_{\rm zar}(X_{\os{\circ}{T}_0}/S(T)^{\nat},E)$ 
as in the proof of (\ref{prop:naqi}). 
\end{proof}

\begin{prop}\label{prop:whf} 
The following diagram 
\begin{equation*} 
\begin{CD}
\wt{R}u_{X^{(\star)}_{\os{\circ}{T}_0}/\os{\circ}{T}*}
(\eps^*_{X^{(\star)}_{\os{\circ}{T}_0}/\os{\circ}{T}} 
(E^{(\star)})\langle u \rangle)
@>{(\ref{eqn:eaxte}),\sim}>>
Ru_{X^{(\star)}_{\os{\circ}{T}_0}/S(T)^{\nat}*}
(\eps^*_{X^{(\star)}_{\os{\circ}{T}_0}/S(T)^{\nat}}(E^{(\star)}))
\\
@AAA @A{(\ref{eqn:eetxte})}A{\simeq}A \\ 
\wt{R}u_{X_{\os{\circ}{T}_0}/\os{\circ}{T}*}(\eps^*_{X_{\os{\circ}{T}_0}/\os{\circ}{T}}
(E)\langle u \rangle)
@>{(\ref{eqn:exte}),\sim}>>
Ru_{X_{\os{\circ}{T}_0}/S(T)^{\nat}*}(\eps^*_{X_{\os{\circ}{T}_0}/S(T)^{\nat}} 
(E)).  
\end{CD} 
\tag{6.14.1}\label{eqn:cp}
\end{equation*}  
is commutative. 
\end{prop}
\begin{proof}
The commutativity follows from the local descriptions of 
the four morphisms in (\ref{eqn:cp}). 
\end{proof}

\begin{coro}\label{coro:lb}
The following hold$:$
\par 
$(1)$ There exists the following isomorphism
\begin{align*} 
Ru_{X_{\os{\circ}{T}_0}/S(T)^{\nat}*}
(\eps^*_{X_{\os{\circ}{T}_0}/S(T)^{\nat}}(E)) \os{\sim}{\lo} 
\wt{R}u_{X^{(\star)}_{\os{\circ}{T}_0}/\os{\circ}{T}*}
(\eps^*_{X^{(\star)}_{\os{\circ}{T}_0}/\os{\circ}{T}}(E^{(\star)})\langle u\rangle)=
H_{\rm zar}(X_{\os{\circ}{T}_0}/S(T)^{\nat},E). 
\tag{6.15.1}\label{ali:eeqie}
\end{align*}  
\par 
$(2)$ The canonical morphism 
\begin{align*} 
\wt{R}u_{X_{\os{\circ}{T}_0}/\os{\circ}{T}*}(\eps^*_{X_{\os{\circ}{T}_0}/\os{\circ}{T}} 
(E)\langle u \rangle)\lo 
\wt{R}u_{X^{(\star)}_{\os{\circ}{T}_0}/\os{\circ}{T}*}
(\eps^*_{X^{(\star)}_{\os{\circ}{T}_0}/\os{\circ}{T}} 
(E^{(\star)})\langle u \rangle)
\tag{6.15.2}\label{ali:eetie}
\end{align*} 
is an isomorphism. 
\end{coro} 
\begin{proof} 
This corollary follows from (\ref{prop:whf}). 
\end{proof}


\begin{theo}\label{theo:indp}
The filtered complex
$(H_{\rm zar}(X_{\os{\circ}{T}_0}/S(T)^{\nat},E),P)$
is independent of the choice of an affine open covering of $X_{\os{\circ}{T}_0}$ 
and a simplicial immersion 
$X_{\os{\circ}{T}_0\bul}\os{\sus}{\lo} \ol{\cal P}_{\bul}$ over $\ol{S(T)^{\nat}}$.  
\end{theo}
\begin{proof} 
Let $X'_{\os{\circ}{T}_0\bul}$ be the \v{C}ech diagram of 
another affine open covering of $X_{\os{\circ}{T}_0}$ 
and let $X'_{\os{\circ}{T}_0\bul}\os{\sus}{\lo} \ol{\cal P}{}'_{\bul}$ 
be another simplicial immersion into 
a log smooth scheme over $\ol{S(T)^{\nat}}$.  
We may assume that there exists the commutative diagram (\ref{cd:sop}) 
and that the morphism $X_{\os{\circ}{T}_0\bul}\lo X'_{\os{\circ}{T}_0\bul}$ 
induces a morphism 
$X_{\ul{\lam},\os{\circ}{T}_0\bul}\lo X'_{\ul{\lam},\os{\circ}{T}_0\bul}$. 
By (\ref{prop:ncqi}),  $H_{\rm zar}(X_{\os{\circ}{T}_0}/S(T)^{\nat},E)$  
is independent of the choice of an affine open covering of $X_{\os{\circ}{T}_0}$ 
and a simplicial immersion 
$X_{\os{\circ}{T}_0\bul}\os{\sus}{\lo} \ol{\cal P}_{\bul}$ over $\ol{S(T)^{\nat}}$.  
Hence, by using descending induction, we have only to prove that 
${\rm gr}^P_k(H_{\rm zar}(X_{\os{\circ}{T}_0}/S(T)^{\nat},E))$ 
is independent of the choice of an affine open covering of $X_{\os{\circ}{T}_0}$ 
and a simplicial immersion 
$X_{\os{\circ}{T}_0\bul}\os{\sus}{\lo} \ol{\cal P}_{\bul}$ over $\ol{S(T)^{\nat}}$.
Let 
\begin{align*} 
\pi_{\ul{\lam},{\rm crys}}\col ((\os{\circ}{X}_{T_0,\ul{\lam},\bul}/\os{\circ}{T})_{\rm crys},
{\cal O}_{\os{\circ}{X}_{T_0,\ul{\lam},\bul}/\os{\circ}{T}})
\lo 
(\os{\circ}{X}_{\ul{\lam},T_0}/\os{\circ}{T})_{\rm crys},
{\cal O}_{\os{\circ}{X}_{\ul{\lam},T_0}/\os{\circ}{T}})
\end{align*} 
be the natural morphism of ringed topoi. 
Let $E_{\ul{\lam}}$ and $E^{\bul}_{\ul{\lam}}$ be the inverse images of $E$ and $E^{\bul}$ to 
$(\os{\circ}{X}_{\ul{\lam},T_0}/\os{\circ}{T})_{\rm crys}$ and 
$(\os{\circ}{X}_{\ul{\lam},T_0,\bul}/\os{\circ}{T})_{\rm crys}$, respectively. 
By using (\ref{ali:grc}), \cite[(1.3.4)]{nh2}  and the cohomological descent, 
the graded complex  
${\rm gr}^P_k(H_{\rm zar}(X_{\os{\circ}{T}_0}/S(T)^{\nat},E))$
can be calculated as follows: 
\begin{align*} 
&{\rm gr}^P_k(H_{\rm zar}(X_{\os{\circ}{T}_0}/S(T)^{\nat},E))= 
{\rm gr}^P_k(\wt{R}u_{X^{(\star)}_{\os{\circ}{T}_0}/\os{\circ}{T}*}
(\eps^*_{X^{(\star)}_{\os{\circ}{T}_0}/\os{\circ}{T}}(E^{(\star)})\langle u \rangle) )
\tag{6.16.1}\label{ali:vspgrc}\\
&={\rm gr}^P_k
R\pi_{{\rm zar}*}
(s({\cal E}^{\bul}
\otimes_{{\cal O}_{{\cal P}^{\rm ex}}}
\Om^{\bul}_{{\cal P}^{{\rm ex},(\bul)}_{\bul}/\os{\circ}{T}}\langle u \rangle))\\
&=
R\pi_{{\rm zar}*}{\rm gr}_k^P
(s({\cal E}^{\bul}
\otimes_{{\cal O}_{{\cal P}^{\rm ex}}}
\Om^{\bul}_{{\cal P}{}^{{\rm ex},(\bul)}_{\bul}/\os{\circ}{T}}\langle u \rangle)) 
\\
&\os{\sim}{\lo} \bigoplus_{m\geq 0}
\bigoplus_{\sharp \ul{\lam}=m+1}
\bigoplus_{j\geq 0}
\bigoplus_{\sharp \ul{\mu}=k+m-2j}\\
& R\pi_{{\rm zar}*}({\cal E}^{\bul}
\otimes_{{\cal O}_{{\cal P}^{\rm ex}_{\bul}}}
b_{\bul \ul{\lam}\cup \ul{\mu}*}
(\Om^{\bul}_{\os{\circ}{\cal P}{}^{{\rm ex}}_{\bul \ul{\lam}
\cup \ul{\mu}}/\os{\circ}{T}}\otimes_{\mab Z}\vp_{{\rm zar},\ul{\mu}}
(\os{\circ}{\cal P}{}^{\rm ex}_{\bul}/\os{\circ}{T})))[-k-2m+2j] \\
&=\bigoplus_{m\geq 0}
\bigoplus_{\sharp \ul{\lam}=m+1}
\bigoplus_{j\geq 0}
\bigoplus_{\sharp \ul{\mu}=k+m-2j}\\
&R\pi_{{\rm zar}*}
a_{\ul{\lam}\cup \ul{\mu},{T}_0*}
Ru_{\os{\circ}{X}{}_{\ul{\lam}\cup \ul{\mu},{T}_0,\bul}/\os{\circ}{T}_0*}
(E^{\bul}_{\ul{\lam}\cup \ul{\mu}}\otimes_{\mab Z}\vp_{{\rm crys},\ul{\mu}}
(\os{\circ}{X}{}_{\ul{\mu},{T}_0,\bul}/\os{\circ}{T}))[-k-2m+2j] \\
&=\bigoplus_{m\geq 0}
\bigoplus_{\sharp \ul{\lam}=m+1}
\bigoplus_{j\geq 0}
\bigoplus_{\sharp \ul{\mu}=k+m-2j}\\
&
a_{\ul{\lam}\cup \ul{\mu},{T}_0*}Ru_{\os{\circ}{X}{}_{\ul{\lam}\cup \ul{\mu},{T}_0}/\os{\circ}{T}_0*}
R\pi_{\ul{\lam}\cup \ul{\mu},{\rm crys}*}
(E^{\bul}_{\ul{\lam}\cup \ul{\mu}}\otimes_{\mab Z}\vp_{{\rm crys},\ul{\mu}}
(\os{\circ}{X}{}_{\ul{\mu},{T}_0,\bul}/\os{\circ}{T}))[-k-2m+2j] \\
&=\bigoplus_{m\geq 0}
\bigoplus_{\sharp \ul{\lam}=m+1}
\bigoplus_{j\geq 0}
\bigoplus_{\sharp \ul{\mu}=k+m-2j}\\
&a_{\ul{\lam}\cup \ul{\mu},{T}_0*}
Ru_{\os{\circ}{X}{}_{\ul{\lam}\cup \ul{\mu},{T}_0}/\os{\circ}{T}_0*}
(E_{\ul{\lam}\cup \ul{\mu}}\otimes_{\mab Z}\vp_{{\rm crys},\ul{\mu}}
(\os{\circ}{X}{}_{\ul{\mu},{T}_0}/\os{\circ}{T}))[-k-2m+2j]. 
\end{align*} 
Moreover we obtain the following commutative diagram 
\begin{equation*} 
\begin{CD} 
R\pi_{{\rm zar}*}{\rm gr}_k^P
(s({\cal E}^{\bul}
\otimes_{{\cal O}_{{\cal P}^{\rm ex}}}
\Om^{\bul}_{{\cal P}{}^{{\rm ex},(\bul)}_{\bul}/\os{\circ}{T}}\langle u \rangle)) = \\ 
@AAA \\
R\pi_{{\rm zar}*}{\rm gr}_k^P
(s({\cal E}'{}^{\bul}
\otimes_{{\cal O}_{{\cal P}'{}^{\rm ex}}}
\Om^{\bul}_{{\cal P}_{\bul}'{}^{{\rm ex},(\bul)}/\os{\circ}{T}}\langle u \rangle)) 
= \\
\displaystyle{\bigoplus_{m\geq 0}
\bigoplus_{\sharp \ul{\lam}=m+1}
\bigoplus_{j\geq 0}
\bigoplus_{\sharp \ul{\mu}=k+m-2j}}
a_{\ul{\lam}\cup \ul{\mu},{T}_0*}Ru_{\os{\circ}{X}{}_{\ul{\lam}\cup \ul{\mu},{T}_0}/\os{\circ}{T}_0*}
(E_{\ul{\lam}\cup \ul{\mu}}\otimes_{\mab Z}\vp_{{\rm crys},\ul{\mu}}
(\os{\circ}{X}{}_{\ul{\mu},{T}_0}/\os{\circ}{T}))[-k-2m+2j]\\
@| \\ 
\displaystyle{\bigoplus_{m\geq 0}
\bigoplus_{\sharp \ul{\lam}=m+1}
\bigoplus_{j\geq 0}
\bigoplus_{\sharp \ul{\mu}=k+m-2j}}
a_{\ul{\lam}\cup \ul{\mu},{T}_0*}Ru_{\os{\circ}{X}{}_{\ul{\lam}\cup \ul{\mu},{T}_0}/\os{\circ}{T}_0*}
(E_{\ul{\lam}\cup \ul{\mu}}\otimes_{\mab Z}\vp_{{\rm crys},\ul{\mu}}
(\os{\circ}{X}{}_{\ul{\mu},{T}_0}/\os{\circ}{T}))[-k-2m+2j]. 
\end{CD}
\tag{6.16.2}\label{eqn:xdhs}
\end{equation*} 
Hence we obtain the desired independence of 
${\rm gr}^P_k(H_{\rm zar}(X_{\os{\circ}{T}_0}/S(T)^{\nat},E))$. 
\end{proof}

\begin{rema}\label{rema:an}
Because $X_{\os{\circ}{T}_0,n}$ $(n\in {\mab N})$ is an  open subscheme of an affine scheme, 
$P_{l}s({\cal E}^{\bul}
\otimes_{{\cal O}_{{\cal P}^{\rm ex}}}
\Om^{\bul}_{{\cal P}^{{\rm ex},(\bul)}_n/\os{\circ}{T}}\langle u \rangle)=0$ 
for $l<<0$.  
\end{rema} 

\begin{coro}\label{coro:ts}
There exists a spectral sequence whose $E_1$-terms are as follows$:$  
\begin{align*} 
&E_1^{-k,q+k}=
\bigoplus_{m\geq 0}
\bigoplus_{\# \ul{\lam}=m+1}
\bigoplus_{j\geq 0}
\bigoplus_{\# \ul{\mu}=k+m-2j}
R^{q+2j-k-2m}
f_{\os{\circ}{X}_{\ul{\lam}\cup \ul{\mu},T_0}/\os{\circ}{T}*} 
(E\vert_{\os{\circ}{X}_{\ul{\lam}\cup \ul{\mu},T_0}/\os{\circ}{T}}
\tag{6.18.1}\label{ali:spsalarh}\\
&\otimes_{\mab Z}\vp_{{\rm crys},\ul{\mu}}
(\os{\circ}{X}_{\ul{\lam}\cup \ul{\mu},T_0}/\os{\circ}{T})). 
\end{align*} 
If $\os{\circ}{X}_{T_0}$ is quasi-compact, then 
there exists the following convergent spectral sequence$:$ 
\begin{align*} 
&E_1^{-k,q+k}=
\bigoplus_{m\geq 0}
\bigoplus_{\# \ul{\lam}=m+1}
\bigoplus_{j\geq 0}
\bigoplus_{\# \ul{\mu}=k+m-2j}
R^{q+2j-k-2m}
f_{\os{\circ}{X}_{\ul{\lam}\cup \ul{\mu},T_0}/\os{\circ}{T}*} 
(E\vert_{\os{\circ}{X}_{\ul{\lam}\cup \ul{\mu},T_0}/\os{\circ}{T}}
\tag{6.18.2}\label{ali:spsarh}\\
&\otimes_{\mab Z}\vp_{{\rm crys},\ul{\mu}}
(\os{\circ}{X}_{\ul{\lam}\cup \ul{\mu},T_0}/\os{\circ}{T}))
\Lo R^qf_{X_{\os{\circ}{T}_0}/S(T)^{\nat}*}
(\eps^*_{X_{\os{\circ}{T}_0}/S(T)^{\nat}}(E))\quad (q\in {\mab N}). 
\end{align*} 
\end{coro} 
\begin{proof} 
The existence of the spectral sequence follows from (\ref{ali:vspgrc}). 
The existence of the convergent spectral sequence follows from (\ref{prop:lcz}) 
(this implies that (\ref{ali:spsarh}) is regular), 
(\ref{rema:ex}) and the cohomological version of 
the \cite[Classical Convergence Theorem 5.5.1]{weib} and (\ref{ali:eeqie}). 
\end{proof}

\begin{defi}
We denote by $P_H$ the induced filtration on 
$R^hf_{X_{\os{\circ}{T}_0}/S(T)^{\nat}*}
(\eps^*_{X_{\os{\circ}{T}_0}/S(T)^{\nat}}(E))$ 
defined by (\ref{ali:spsarh}). 
\end{defi}

We use the following in \S\ref{sec:p} below. 

\begin{prop}\label{prop:baf} 
If $\os{\circ}{X}_{T_0}$ is quasi-compact, then 
$(H_{\rm zar}(X_{\os{\circ}{T}_0}/S(T)^{\nat},E),P)$ is bounded above. 
\end{prop} 
\begin{proof} 
First we note that $H_{\rm zar}(X_{\os{\circ}{T}_0}/S(T)^{\nat},E)$ is bounded above. 
Indeed,  $H_{\rm zar}(X_{\os{\circ}{T}_0}/S(T)^{\nat},E)$ is isomorphic to 
$Ru_{X_{\os{\circ}{T}_0}/S(T)^{\nat}*}
(\eps^*_{X_{\os{\circ}{T}_0}/S(T)^{\nat}}(E))$. 
This complex 
is bounded above by the log Poincar\'{e} lemma
because  $\os{\circ}{X}_{T_0}$ is quasi-compact, 
\par 
By (\ref{rema:an}),  
$P_{l}H_{\rm zar}(X_{\os{\circ}{T}_0}/S(T)^{\nat},E)=0$ 
for $l<<0$ because $\os{\circ}{X}_{T_0}$ is quasi-compact. 
By (\ref{ali:vspgrc}), 
\begin{align*} 
&{\rm gr}^P_k(H_{\rm zar}(X_{\os{\circ}{T}_0}/S(T)^{\nat},E))\\
&=\bigoplus_{m\geq 0}
\bigoplus_{\# \ul{\lam}=m+1}
\bigoplus_{j\geq 0}
\bigoplus_{\# \ul{\mu}=k+m-2j}
Ru_{\os{\circ}{X}{}_{\ul{\lam}\cup \ul{\mu},{T}_0}/\os{\circ}{T}_0*}
(E_{\ul{\lam}\cup \ul{\mu}}\otimes_{\mab Z}\vp_{{\rm crys},\ul{\mu}}
(\os{\circ}{X}{}_{\ul{\mu},{T}_0}/\os{\circ}{T}))[-k-2m+2j]. 
\end{align*} 
Because $X_{\os{\circ}{T}_0}$ is quasi-compact, there exists a positive integer $q$ 
such that 
$$R^qu_{\os{\circ}{X}{}_{\ul{\lam}\cup \ul{\mu},{T}_0}/\os{\circ}{T}_0*}
(E_{\ul{\lam}\cup \ul{\mu}}\otimes_{\mab Z}\vp_{{\rm crys},\ul{\mu}}
(\os{\circ}{X}{}_{\ul{\mu},{T}_0}/\os{\circ}{T}))=0$$ 
for any $\ul{\lam}$ and $\ul{\mu}$ by the Poincar\'{e} lemma. 
We complete the proof of (\ref{prop:baf}). 
\end{proof} 

The last aim in this section is to prove necessary results for the comparison 
of $(H_{\rm zar}(X_{\os{\circ}{T}_0}/S(T)^{\nat},E),P)$ 
with Kim-Hain's filtered complex in the case where $S$ is the log point $s$ of 
a perfect field of characteristic $p>0$ and $T={\cal W}_n(s)$. 
\par 
Recall that $E_{\ul{\lam}}$ is the inverse image of $E$ to 
$(\os{\circ}{X}_{\ul{\lam},T_0}/\os{\circ}{T})_{\rm crys}$. 
Let $f_{\ul{\lam}} \col X_{\ul{\lam},\os{\circ}{T}_0} \lo S(T)^{\nat}$ be the structural morphism.  
For $U=S(T)^{\nat}$ or $\os{\circ}{T}$, 
let 
\begin{align*}
\eps_{X_{\ul{\lam},\os{\circ}{T}_0}/U} 
\col ((X_{\ul{\lam},\os{\circ}{T}_0}/U)_{\rm crys},
{\cal O}_{X_{\ul{\lam},\os{\circ}{T}_0}/U})
\lo ((\os{\circ}{X}_{\ul{\lam},\os{\circ}{T}_0}/\os{\circ}{T})_{\rm crys},
{\cal O}_{\os{\circ}{X}_{\ul{\lam},T_0}/\os{\circ}{T}}), 
\end{align*} 
be the morphism forgetting the log structure.

\begin{prop}\label{prop:nlgr} 
For a nonempty set $\ul{\lam}$, 
set 
\begin{align*} 
(\wt{R}u_{X_{\ul{\lam},\os{\circ}{T}_0}/\os{\circ}{T}*}
(\eps^*_{X_{\ul{\lam},\os{\circ}{T}_0}/\os{\circ}{T}}
(E_{\ul{\lam}})), P)
:=R\pi_{\ul{\lam},{\rm zar}*}((s({\cal E}^{\bul}
\otimes_{{\cal O}_{{\cal P}{}^{\rm ex}_{\bul}}}
\Om^{\bul}_{{\cal P}{}^{{\rm ex}}_{\bul\ul{\lam}}/\os{\circ}{T}}),P))
\end{align*} 
in ${\rm D}^+{\rm F}(f_{\ul{\lam}}^{-1}({\cal O}_T))$. 
Let $a_{\ul{\lam},\ul{\lam}\cup \ul{\mu},\os{\circ}{T}_0} \col 
\os{\circ}{X}_{\ul{\lam}\cup \ul{\mu},T_0} \lo 
\os{\circ}{X}_{\ul{\lam},T_0}$ be the natural closed immersion.  
For an integer $k$, 
there exists a natural isomorphism 
\begin{align*}
& {\rm gr}_k^P\wt{R}u_{X_{\ul{\lam},\os{\circ}{T}_0}/\os{\circ}{T}*}
(\eps^*_{X_{\ul{\lam},\os{\circ}{T}_0}/\os{\circ}{T}}(E_{\ul{\lam}}))
\os{\sim}{\lo} \tag{6.21.1}\label{eqn:ele}\\
& 
\bigoplus_{\# \ul{\mu}=k} 
a_{\ul{\lam},\ul{\lam}\cup \ul{\mu},T_0*}
Ru_{\os{\circ}{X}_{\ul{\lam}\cup \ul{\mu},T}
/\os{\circ}{T}}
(E_{\ul{\lam}\cup \ul{\mu}}\otimes_{\mab Z}
\vp_{{\rm crys},\ul{\mu}}(\os{\circ}{X}_T/\os{\circ}{T}))[-k]. 
\end{align*} 
\end{prop}
\begin{proof} 
As in (\ref{ali:vspgrc}), we obtain the isomorphism 
(\ref{eqn:ele}) by using (\ref{eqn:mpprrn}).  
\end{proof}

\begin{coro}\label{coro:il}
The filtered complex 
$(\wt{R}u_{X_{\ul{\lam},\os{\circ}{T}_0}/\os{\circ}{T}*}
(\eps^*_{X_{\ul{\lam},\os{\circ}{T}_0}/\os{\circ}{T}}(E_{\ul{\lam}})),P)$  
is independent of the choice of an open covering of $X$ 
and an immersion 
$X_{\os{\circ}{T}_0,\bul}\os{\sus}{\lo} \ol{\cal P}_{\bul}$ over $\ol{S(T)^{\nat}}$. 
\end{coro}
\begin{proof} 
Note that $P_{-1}\wt{R}u_{X_{\ul{\lam},\os{\circ}{T}_0}/\os{\circ}{T}*}
(\eps^*_{X_{\ul{\lam},\os{\circ}{T}_0}/\os{\circ}{T}}(E_{\ul{\lam}}))=0$. 
By using (\ref{eqn:ele}), ascending induction on $k$ 
shows (\ref{coro:il}). 
\end{proof}  

\begin{defi}\label{defi:ndd}
Set 
\begin{align*} 
(\wt{R}u_{X^{(m)}_{\os{\circ}{T}_0}/\os{\circ}{T}*}
(\eps^*_{X^{(m)}_{\os{\circ}{T}_0}/\os{\circ}{T}}(E))\tag{6.23.1}\label{ali:gsmled}
:=\bigoplus_{\# \ul{\lam}=m+1}
(\wt{R}u_{X_{\ul{\lam},\os{\circ}{T}_0}/\os{\circ}{T}*}
(\eps^*_{X_{\ul{\lam},\os{\circ}{T}_0}/\os{\circ}{T}}
(E_{\ul{\lam}})), P). 
\end{align*} 
\end{defi}


\begin{prop}\label{prop:lx}  
Assume that  $\ul{\lam}\not=\emptyset$.  
Let 
\begin{equation*} 
\pi_{\ul{\lam},{\rm zar}} \col 
((X_{\ul{\lam},\os{\circ}{T}_0\bul})_{\rm zar},f^{-1}_{\bul}({\cal O}_T)) \lo 
((X_{\ul{\lam},\os{\circ}{T}_0})_{\rm zar},f^{-1}({\cal O}_T)) 
\tag{6.24.1}\label{eqn:lbd} 
\end{equation*} 
be the natural morphism of ringed topoi.  
Set 
\begin{align*} 
Ru_{X_{\ul{\lam},\os{\circ}{T}_0}/S(T)^{\nat}*}
(\eps^*_{X_{\ul{\lam},\os{\circ}{T}_0}/S(T)^{\nat}}(E_{\ul{\lam}}))&= 
R\pi_{\ul{\lam},{\rm zar}}(s({\cal E}^{\bul}\otimes_{{\cal O}_{{\cal P}^{\rm ex}_{\bul}}}
{\Om}^{\bul}_{{\cal P}^{\rm ex}_{\bul}/S(T)^{\nat}}\otimes_{{\cal O}_{{\cal P}^{\rm ex},\bul}}
{\cal O}_{{\cal P}^{\rm ex}_{\bul \ul{\lam}}}))\tag{6.24.2}\label{eqn:lbid} \\
&= R\pi_{\ul{\lam},{\rm zar}}(s({\cal E}^{\bul}\otimes_{{\cal O}_{{\cal P}^{\rm ex}_{\bul}}}
{\Om}^{\bul}_{{\cal P}^{\rm ex}_{\bul\ul{\lam}}/S(T)^{\nat}}))
\end{align*} 
in ${\rm D}^+{\rm F}(f_{\ul{\lam}}^{-1}({\cal O}_T))$. 
Then 
$Ru_{X_{\ul{\lam},\os{\circ}{T}_0}/S(T)^{\nat}*}
(\eps^*_{X_{\ul{\lam},\os{\circ}{T}_0}/S(T)^{\nat}}(E_{\ul{\lam}}))$
is independent of the choice of an affine open covering of $X_{\os{\circ}{T}_0}$ 
and an immersion 
$X_{\os{\circ}{T}_0,\bul}\os{\sus}{\lo} \ol{\cal P}_{\bul}$ over $\ol{S(T)^{\nat}}$. 
\end{prop} 
\begin{proof} 
We have the following exact sequence:  
\begin{align*} 
0\lo 
s({\cal E}^{\bul}\otimes_{{\cal O}_{{\cal P}^{\rm ex}_{\bul}}}
{\Om}^{\bul}_{{\cal P}^{\rm ex}_{\bul}/S(T)^{\nat}})[-1]
\os{d\log t\wedge }{\lo} 
s({\cal E}^{\bul}\otimes_{{\cal O}_{{\cal P}^{\rm ex}_{\bul}}}
{\Om}^{\bul}_{{\cal P}^{\rm ex}_{\bul}/\os{\circ}{T}})
\lo 
s({\cal E}^{\bul}\otimes_{{\cal O}_{{\cal P}^{\rm ex}_{\bul}}}
{\Om}^{\bul}_{{\cal P}^{\rm ex}_{\bul}/S(T)^{\nat}})
\lo 0. 
\end{align*}
Since this is locally split, the following sequence is also locally split: 
\begin{align*} 
&0\lo 
s({\cal E}^{\bul}\otimes_{{\cal O}_{{\cal P}^{\rm ex}_{\bul}}}
{\Om}^{\bul}_{{\cal P}^{\rm ex}_{\bul}/S(T)^{\nat}}
\otimes_{{\cal O}_{{\cal P}^{\rm ex}_{\bul}}}
{\cal O}_{{\cal P}^{\rm ex}_{\bul \ul{\lam}}})[-1]
\os{d\log t\wedge }{\lo} 
s({\cal E}^{\bul}\otimes_{{\cal O}_{{\cal P}^{\rm ex}_{\bul}}}
{\Om}^{\bul}_{{\cal P}^{\rm ex}_{\bul}/\os{\circ}{T}}
\otimes_{{\cal O}_{{\cal P}^{\rm ex}_{\bul}}}
{\cal O}_{{\cal P}^{\rm ex}_{\bul \ul{\lam}}})\\
&\lo 
s({\cal E}^{\bul}\otimes_{{\cal O}_{{\cal P}^{\rm ex}_{\bul}}}
{\Om}^{\bul}_{{\cal P}^{\rm ex}_{\bul}/S(T)^{\nat}}
\otimes_{{\cal O}_{{\cal P}^{\rm ex}_{\bul}}}
{\cal O}_{{\cal P}^{\rm ex}_{\bul \ul{\lam}}})
\lo 0. 
\end{align*}  
Hence we have the following triangle: 
\begin{align*} 
&R\pi_{\ul{\lam}*}(s({\cal E}^{\bul}\otimes_{{\cal O}_{{\cal P}^{\rm ex}_{\bul}}}
{\Om}^{\bul}_{{\cal P}^{\rm ex}_{\bul}/S(T)^{\nat}}
\otimes_{{\cal O}_{{\cal P}^{\rm ex}_{\bul}}}
{\cal O}_{{\cal P}^{\rm ex}_{\bul \ul{\lam}}}))[-1]
\os{d\log t\wedge }{\lo} 
R\pi_{\ul{\lam}*}(s({\cal E}^{\bul}\otimes_{{\cal O}_{{\cal P}^{\rm ex}_{\bul}}}
{\Om}^{\bul}_{{\cal P}^{\rm ex}_{\bul}/\os{\circ}{T}}
\otimes_{{\cal O}_{{\cal P}^{\rm ex}_{\bul}}}
{\cal O}_{{\cal P}^{\rm ex}_{\bul \ul{\lam}}}))\\
&\lo 
R\pi_{\ul{\lam}*}(s({\cal E}^{\bul}\otimes_{{\cal O}_{{\cal P}^{\rm ex}_{\bul}}}
{\Om}^{\bul}_{{\cal P}^{\rm ex}_{\bul}/S(T)^{\nat}}
\otimes_{{\cal O}_{{\cal P}^{\rm ex}_{\bul}}}
{\cal O}_{{\cal P}^{\rm ex}_{\bul \ul{\lam}}}))\os{+1}{\lo}. 
\end{align*}  
By (\ref{ali:ppxhd}), for each $n\in {\mab N}$, 
the following sequence 
\begin{align*} 
&0\lo R^{i-1}\pi_{\ul{\lam}*}(s({\cal E}^n\otimes_{{\cal O}_{{\cal P}^{\rm ex}_n}}
{\Om}^{\bul}_{{\cal P}^{\rm ex}_n/S(T)^{\nat}}
\otimes_{{\cal O}_{{\cal P}^{\rm ex}_n}}
{\cal O}_{{\cal P}^{\rm ex}_{n \ul{\lam}}}))
\os{d\log t\wedge }{\lo} 
R^i\pi_{\ul{\lam}*}(s({\cal E}^{\bul}\otimes_{{\cal O}_{{\cal P}^{\rm ex}_n}}
{\Om}^{\bul}_{{\cal P}^{\rm ex}_n/\os{\circ}{T}}
\otimes_{{\cal O}_{{\cal P}^{\rm ex}_n}}
{\cal O}_{{\cal P}^{\rm ex}_{n \ul{\lam}}}))\\
&\lo 
R^i\pi_{\ul{\lam}*}(s({\cal E}^{\bul}\otimes_{{\cal O}_{{\cal P}^{\rm ex}_n}}
{\Om}^{\bul}_{{\cal P}^{\rm ex}_n/S(T)^{\nat}}
\otimes_{{\cal O}_{{\cal P}^{\rm ex}_n}}
{\cal O}_{{\cal P}^{\rm ex}_{n\ul{\lam}}}))\lo 0
\end{align*}  
is exact. 
Hence, by using the standard spectral sequence, we see that 
the morphism 
\begin{align*} 
R^i\pi_{\ul{\lam}*}(s({\cal E}^{\bul}\otimes_{{\cal O}_{{\cal P}^{\rm ex}_{\bul}}}
{\Om}^{\bul}_{{\cal P}^{\rm ex}_{\bul}/\os{\circ}{T}}
\otimes_{{\cal O}_{{\cal P}^{\rm ex}_{\bul}}}
{\cal O}_{{\cal P}^{\rm ex}_{\bul \ul{\lam}}}))
\lo 
R^i\pi_{\ul{\lam}*}(s({\cal E}^{\bul}\otimes_{{\cal O}_{{\cal P}^{\rm ex}_{\bul}}}
{\Om}^{\bul}_{{\cal P}^{\rm ex}_{\bul}/S(T)^{\nat}}
\otimes_{{\cal O}_{{\cal P}^{\rm ex}_{\bul}}}
{\cal O}_{{\cal P}^{\rm ex}_{\bul \ul{\lam}}}))
\end{align*}
is surjective. Consequently we have the following exact sequence: 
\begin{align*} 
& 0\lo R\pi^{i-1}_{\ul{\lam}*}(s({\cal E}^{\bul}\otimes_{{\cal O}_{{\cal P}^{\rm ex}_{\bul}}}
{\Om}^{\bul}_{{\cal P}^{\rm ex}_{\bul}/S(T)^{\nat}}
\otimes_{{\cal O}_{{\cal P}^{\rm ex}_{\bul}}}
{\cal O}_{{\cal P}^{\rm ex}_{\bul \ul{\lam}}}))
\os{d\log t\wedge }{\lo} 
\wt{R}{}^iu_{X_{\ul{\lam},\os{\circ}{T}_0}/\os{\circ}{T}*}
(\eps^*_{X_{\ul{\lam},\os{\circ}{T}_0}/\os{\circ}{T}}(E_{\ul{\lam}}))
\tag{6.24.3}\label{eqn:lbiod} \\
&\lo 
R^i\pi_{\ul{\lam}*}(s({\cal E}^{\bul}\otimes_{{\cal O}_{{\cal P}^{\rm ex}_{\bul}}}
{\Om}^{\bul}_{{\cal P}^{\rm ex}_{\bul}/S(T)^{\nat}}
\otimes_{{\cal O}_{{\cal P}^{\rm ex}_{\bul}}}
{\cal O}_{{\cal P}^{\rm ex}_{\bul \ul{\lam}}}))\lo 0. 
\end{align*}  
Induction on $i$ tells us that 
$R^i\pi_{\ul{\lam}*}(s({\cal E}^{\bul}\otimes_{{\cal O}_{{\cal P}^{\rm ex}_{\bul}}}
{\Om}^{\bul}_{{\cal P}^{\rm ex}_{\bul}/S(T)^{\nat}}
\otimes_{{\cal O}_{{\cal P}^{\rm ex}_{\bul}}}
{\cal O}_{{\cal P}^{\rm ex}_{\bul \ul{\lam}}}))$ 
is independent of the choice of an affine open covering of $X_{\os{\circ}{T}_0}$ 
and an immersion 
$X_{\os{\circ}{T}_0,\bul}\os{\sus}{\lo} \ol{\cal P}_{\bul}$ over $\ol{S(T)^{\nat}}$. 
Because the rest of the proof is a routine work, we leave the rest to the reader. 
\end{proof}

Set 
\begin{align*} 
&Ru_{X^{(m)}_{\os{\circ}{T}_0}/S(T)^{\nat}*}
(\eps^*_{X^{(m)}_{\os{\circ}{T}_0}/S(T)^{\nat}}(E))
:=\bigoplus_{\# \ul{\lam}=m+1}
Ru_{X_{\ul{\lam},\os{\circ}{T}_0}/S(T)^{\nat}*}
(\eps^*_{X_{\ul{\lam},\os{\circ}{T}_0}/S(T)^{\nat}}(E_{\ul{\lam}})). 
\tag{6.24.4}\label{ali:gsfled}
\end{align*} 
Then 
\begin{align*} 
&Ru_{X^{(*)}_{\os{\circ}{T}_0}/S(T)^{\nat}*}
(\eps^*_{X^{(*)}_{\os{\circ}{T}_0}/S(T)^{\nat}}(E))
=s(a^{(\bul)}_*Ru_{X^{(\bul)}_{\os{\circ}{T}_0}/S(T)^{\nat}*}
(\eps^*_{X^{(\bul)}_{\os{\circ}{T}_0}/S(T)^{\nat}}(E))). 
\tag{6.24.5}\label{ali:gsflmed}
\end{align*} 


\section{Hirsch weight-filtered log crystalline dga's}\label{sec:fcd}
Let the notations be as in the previous section. 
In this section we assume that $\os{\circ}{T}$ is a $p$-adic formal scheme 
and we ignore $p$-torsions of ${\cal O}_T$-modules. 
Let $g\col Y\lo T$ and $f\col X_{\os{\circ}{T}_0}\lo S(T)^{\nat}$ be the structural morphism 
as in the previous sections. 
Let $A^{\geq 0}(g^{-1}({\cal O}_T)\otimes_{\mab Z}{\mab Q})$ 
be the category of positively graded dga's 
over $g^{-1}({\cal O}_T)\otimes_{\mab Z}{\mab Q}$. 
Let ${\rm A}^{\geq 0}{\rm F}(f^{-1}({\cal O}_T)\otimes_{\mab Z}{\mab Q})$ 
be the category of positively graded dga's 
over $f^{-1}({\cal O}_T)\otimes_{\mab Z}{\mab Q}$ with increasing filtrations. 
As in the classical case \cite[III Proposition (2.1.10)]{bb}, 
the topos $(Y/T)_{\rm crys}$ has enough points. 
In this section, by using the Thom-Whitney derived functor constructed in \cite{nav},  
we construct an object 
$$\wt{R}_{\rm TW}u_{Y_{\os{\circ}{T}_0}/\os{\circ}{T}*}
({\cal O}_{Y_{\os{\circ}{T}_0}/\os{\circ}{T}})_{\mab Q}\in  
{\rm Ho}(A^{\geq 0}(g^{-1}({\cal O}_T)\otimes_{\mab Z}{\mab Q})).$$  
Moreover, by using the derived PD-Hirsch extension 
$\langle U_{S(T)^{\nat}}\rangle{}^L$,  
we construct objects  
\begin{align*} 
\wt{R}_{\rm TW}u_{Y_{\os{\circ}{T}_0}/\os{\circ}{T}*}
({\cal O}_{X_{\os{\circ}{T}_0}/\os{\circ}{T}})_{\mab Q}
\langle u\rangle^L :=
\wt{R}_{\rm TW}u_{Y_{\os{\circ}{T}_0}/\os{\circ}{T}*}
({\cal O}_{X_{\os{\circ}{T}_0}/\os{\circ}{T}})_{\mab Q}
\langle U_{S(T)^{\nat}}\rangle^L  
\in {\rm Ho}(A^{\geq 0}(g^{-1}({\cal O}_T)\otimes_{\mab Z}{\mab Q}))
\end{align*}
and 
\begin{align*} 
(H_{{\rm zar},{\rm TW}}(X_{\os{\circ}{T}_0}/S(T)^{\nat})_{\mab Q},P)
\in {\rm Ho}(A^{\geq 0}F(f^{-1}({\cal O}_T)\otimes_{\mab Z}{\mab Q})). 
\end{align*}
\par
Though the proof of the following proof is easy, the result itself is interesting. 

\begin{prop}\label{prop:indcrc}
Let $(U,{\cal K},\eps)$ be a fine log $p$-adic formal PD-scheme. 
Set $U_0:=\ul{\rm Spec}^{\log}_U(U/{\cal K})$. 
Let $Z$ be a fine log scheme over $U_0$ 
and let $h\col Z\lo U$ be the structural morphism.  
Assume that there exists an immersion $Z\os{\sus}{\lo} {\cal Z}$ 
into a log smooth log scheme over $U$. 
Let $A^{\geq 0}({\cal O}_{Z/U})$ 
be the category of positively graded pd-dga's over ${\cal O}_{Z/U}$. 
Then there exists a canonical object 
\begin{align*} 
\Om^{\bul}_{(Z/U)_{\rm crys}}\in {\rm Ho}(A^{\geq 0}({\cal O}_{Z/U}))
\end{align*}
whose image in $D^+({\cal O}_{Z/U})$ 
is isomorphic to ${\cal O}_{Z/U}$. 
\end{prop}
\begin{proof} 
Let ${\mathfrak F}$ be the log PD-envelope of the immersion 
$Z\os{\sus}{\lo} {\cal Z}$ over $(U,{\cal K},\eps)$. 
Let
\begin{align*} 
L_{\rm crys} \col 
\{{\cal O}_{{\mathfrak F}}{\textrm -}{\rm modules}\}
\lo \{{\cal O}_{Z/U}{\textrm -}{\rm modules}\}
\end{align*}
be the log linearization functor for ${\cal O}_{{\mathfrak F}}$-modules.
Then we obtain a sheaf of cosimplicial dga 's 
\begin{align*} 
(L_{\rm crys}({\cal O}_{{\mathfrak F}}\otimes_{{\cal O}_{{\cal Z}}}
\Om^{\bul}_{{\cal Z}/U}),L(d))=
(L_{\rm crys}(\Om^{\bul}_{{\mathfrak F}/U,[~]}),L(d))
\end{align*}  
over $h^{-1}_{Z/U}({\cal O}_U)$. 
Here we have used the equality 
${\cal O}_{{\mathfrak F}}\otimes_{{\cal O}_{{\cal Z}}}
\Om^{\bul}_{{\cal Z}/U}=\Om^{\bul}_{{\mathfrak F}/U,[~]}$, 
which has been proved in \cite[(1.3.28)]{nb} which is the log version of 
\cite[0 (3.1.4), (3.1.6)]{idw}. 

Set 
\begin{align*} 
\Om^{\bul}_{(Z/U)_{\rm crys}}:=
(L_{\rm crys}(\Om^{\bul}_{{\mathfrak F}/U,[~]}),L(d)). 
\end{align*}
We claim that $\Om^{\bul}_{(Z/U)_{\rm crys}}$ is independent of the choice of 
an immersion $Z\os{\sus}{\lo} {\cal Z}$. 
\par 
Let $Z\os{\sus}{\lo} {\cal Z}'$ be 
another immersion into a log smooth scheme over $U$. 
By considering the product ${\cal Z}\times_U{\cal Z}'$, 
we may assume that we have a morphism ${\cal Z}\lo {\cal Z}'$ such that 
the composite morphism $Z\os{\sus}{\lo}{\cal Z}\lo {\cal Z}'$ is the given immersion. 
Let ${\mathfrak F}'$ be the log PD-envelope of this composite immersion over 
$(U,{\cal K},\eps)$. Then we have a natural morphism ${\mathfrak F}\lo {\mathfrak F}'$
and hence we have a morphism 
$L_{\rm crys}(\Om^{\bul}_{{\mathfrak F}'/U,[~]})\lo 
L_{\rm crys}(\Om^{\bul}_{{\mathfrak F}/U,[~]})$ fitting into 
the following commutative diagram 
\begin{equation*} 
\begin{CD}
{\cal O}_{Z/U}@>{\sim}>>L_{\rm crys}(\Om^{\bul}_{{\mathfrak F}/U,[~]})\\
@|@AAA\\
{\cal O}_{Z/U}@>{\sim}>>L_{\rm crys}(\Om^{\bul}_{{\mathfrak F}'/U,[~]}). 
\end{CD}
\end{equation*}
Here we have used the log Poincar\'{e} lemma (\cite[(2.2.7)]{nh2}). 
Now we can complete the proof of (\ref{prop:indcrc}). 
\end{proof}

\begin{theo}\label{theo:q}
Let the notations be as in {\rm (\ref{prop:indcrc})} without 
assuming the immersion $Z\os{\sus}{\lo}{\cal Z}$. 
Assume that $(U,{\cal K},\eps)$ is $p$-adic. 
Then there exists a canonical object 
\begin{align*}
R_{\rm TW}u_{Z/U*}({\cal O}_{Z/U})_{\mab Q}\in {\rm Ho}(A^{\geq 0}(h^{-1}({\cal K}_U)))
\end{align*} 
whose image in $D^+(h^{-1}({\cal K}_U))$ 
is equal to $Ru_{Z/U*}({\cal O}_{Z/U})\otimes^L_{\mab Z}{\mab Q}$. 
\end{theo} 
\begin{proof} 
Let $Z_{\bul}$ be a simplicial scheme obtained by an open covering of $Z$ 
such that there exists an immersion 
$Z_{\bul}\os{\sus}{\lo} {\cal Z}_{\bul}$ 
into a simplicial log smooth scheme over $U$ as usual. 
Let ${\mathfrak F}_{\bul}$ be the log PD-envelope of this immersion over 
$(U,{\cal K},\eps)$. 
Let $v_{\bul}\col (Z_{\bul}/U)_{\rm crys}\lo U_{\rm zar}$ be the morphism of topoi 
obtained by $v$. 
Then we obtain a sheaf of dga's  
\begin{align*} 
{\cal O}_{{\mathfrak F}_{\bul}}\otimes_{{\cal O}_{{\cal Z}_{\bul}}}
\Om^{\bul}_{{\cal Z}_{\bul}/U}\otimes_{\mab Z}{\mab Q}=
\Om^{\bul}_{{\mathfrak F}_{\bul}/U,[~]}\otimes_{\mab Z}{\mab Q}
\end{align*}  
over $v^{-1}_{\bul}({\cal K}_U)$. Set 
\begin{align*} 
R_{\rm TW}u_{Z/U*}({\cal O}_{Z/U})_{\mab Q}:=
R_{\rm TW}\pi_{{\rm zar}*}
(\Om^{\bul}_{{\mathfrak F}_{\bul}/U,[~]}\otimes_{\mab Z}{\mab Q}). 
\end{align*}
We claim that $R_{\rm TW}u_{Z/U*}({\cal O}_{Z/U})_{\mab Q}$ 
is independent of the choices of 
an open covering of $Z$ 
and a simplicial immersion $Z_{\bul}\os{\sus}{\lo} {\cal Z}_{\bul}$ 
into a log smooth simplicial log formal scheme as usual. 
Let $Z'_{\bul}\os{\sus}{\lo} {\cal Z}'_{\bul}$ be another simplicial immersion. 
Let ${\mathfrak F}'_{\bul}$ be the log PD-envelope of this immersion over 
$(U,{\cal K},\eps)$. 
Then, by taking a refinement of two open coverings of $Z$, 
we may assume that there exists the following commutative diagram 
\begin{equation*}
\begin{CD}
Z_{\bul}@>{\sus}>>{\cal Z}_{\bul}\\
@VVV @VVV\\
Z'_{\bul}@>{\sus}>>{\cal Z}'_{\bul}
\end{CD}
\end{equation*} 
over $U$. 
By \cite[(4.4)]{nav} the following diagram is commutative 
\begin{equation*}
\begin{CD}
R_{\rm TW}\pi_{{\rm zar}*}
(\Om^{\bul}_{{\mathfrak F}_{\bul}/U,[~]}\otimes_{\mab Z}{\mab Q})
@>{\sim}>>R\pi_{{\rm zar}*}
(\Om^{\bul}_{{\mathfrak F}_{\bul}/U,[~]}\otimes_{\mab Z}{\mab Q})\\
@AAA @AAA\\
R_{\rm TW}\pi_{{\rm zar}*}
(\Om^{\bul}_{{\mathfrak F}'_{\bul}/U,[~]}\otimes_{\mab Z}{\mab Q})
@>{\sim}>>R\pi_{{\rm zar}*}
(\Om^{\bul}_{{\mathfrak F}'_{\bul}/U,[~]}\otimes_{\mab Z}{\mab Q})
\end{CD}
\end{equation*} 
over $U$. If $I$ is a $\pi_{{\rm zar}*}$-acyclic sheaf on $Z_{\bul}$, then 
$I\otimes_{\mab Z}{\mab Q}$ is also a $\pi_{{\rm zar}*}$-acyclic sheaf on $Z_{\bul}$. 
Hence 
$R\pi_{{\rm zar}*}
(\Om^{\bul}_{{\mathfrak F}_{\bul}/U,[~]}\otimes_{\mab Z}{\mab Q})
=
R\pi_{{\rm zar}*}
(\Om^{\bul}_{{\mathfrak F}_{\bul}/U,[~]})\otimes^L_{\mab Z}{\mab Q}
=
R\pi_{{\rm zar}*}
(\Om^{\bul}_{{\mathfrak F}'_{\bul}/U,[~]})\otimes^L_{\mab Z}{\mab Q}
=R\pi_{{\rm zar}*}
(\Om^{\bul}_{{\mathfrak F}'_{\bul}/U,[~]}\otimes_{\mab Z}{\mab Q})$. 
Consequently we obtain the following commutative diagram 
\begin{equation*}
\begin{CD}
R_{\rm TW}\pi_{{\rm zar}*}
(\Om^{\bul}_{{\mathfrak F}_{\bul}/U,[~]}\otimes_{\mab Z}{\mab Q})
@>{\sim}>>R\pi_{{\rm zar}*}
(\Om^{\bul}_{{\mathfrak F}_{\bul}/U,[~]})\otimes^L_{\mab Z}{\mab Q}\\
@AAA @|\\
R_{\rm TW}\pi_{{\rm zar}*}
(\Om^{\bul}_{{\mathfrak F}'_{\bul}/U,[~]}\otimes_{\mab Z}{\mab Q})
@>{\sim}>>R\pi_{{\rm zar}*}
(\Om^{\bul}_{{\mathfrak F}'_{\bul}/U,[~]})\otimes^L_{\mab Z}{\mab Q}.  
\end{CD}
\end{equation*} 
This diagram tells us that (\ref{theo:q}) holds. 
\end{proof} 
\par 
The following complex 
\begin{align*}
{\cal O}_{{\mathfrak E}_{\bul}}\otimes_{{\cal O}_{{\cal Q}^{{\rm ex}}_{\bul}}}
\Om^{\bul}_{{\cal Q}^{{\rm ex}}_{\bul}/\os{\circ}{T}}
\langle u \rangle =\Om^{\bul}_{{\mathfrak E}_{\bul}/\os{\circ}{T},[~]}\langle u\rangle:=
\Om^{\bul}_{{\mathfrak E}_{\bul}/\os{\circ}{T},[~]}\langle u \rangle
\end{align*}  
is, in fact, a sheaf of dga's over $g^{-1}_{\bul}({\cal O}_T)$ in the topos $Y_{\bul,{\rm zar}}$. 
Then we have the following sheaf of dga's over 
$g^{-1}({\cal O}_T)$ in the topos $Y_{\rm zar}$ by (\ref{ali:aba}):
\begin{align*}  
\wt{R}_{\rm TW}u_{Y_{\os{\circ}{T}_0}/\os{\circ}{T}*}
({\cal O}_{Y_{\os{\circ}{T}_0}/\os{\circ}{T}}\langle u \rangle)_{\mab Q}
:=
R_{\rm TW}\pi_{{\rm zar}*}
(\Om^{\bul}_{{\mathfrak E}_{\bul}/\os{\circ}{T},[~]}\langle u \rangle
\otimes_{\mab Z}{\mab Q}). 
\end{align*}

\begin{defi}\label{defi:pdtw}
We call $\wt{R}_{\rm TW}u_{Y_{\os{\circ}{T}_0}/\os{\circ}{T}*}
({\cal O}_{Y_{\os{\circ}{T}_0}/\os{\circ}{T}}\langle u \rangle)_{\mab Q}$ 
the {\it PD-Hirsch extensions} of 
$\wt{R}_{\rm TW}u_{Y_{\os{\circ}{T}_0}/\os{\circ}{T}*}
({\cal O}_{Y_{\os{\circ}{T}_0}/\os{\circ}{T}})_{\mab Q}$. 
\end{defi}

\begin{prop}\label{prop:ijn}
There exists the following canonical isomorphisms 
\begin{align*}
\wt{R}_{\rm TW}u_{Y_{\os{\circ}{T}_0}/\os{\circ}{T}*}
({\cal O}_{Y_{\os{\circ}{T}_0}/\os{\circ}{T}}\langle u\rangle)_{\mab Q}
\tag{7.4.1}\label{eqn:exnfte}
&\os{\sim}{\lo} R_{\rm TW}u_{Y_{\os{\circ}{T}_0}/S(T)^{\nat}*}
({\cal O}_{Y_{\os{\circ}{T}_0}/S(T)^{\nat}*})_{\mab Q}
\end{align*} 
in ${\rm Ho}(A^{\geq 0}(g^{-1}({\cal O}_T)_{\os{\circ}{T}_0})\otimes_{\mab Z}{\mab Q})$. 
In particular, 
$\wt{R}_{\rm TW}u_{Y_{\os{\circ}{T}_0}/\os{\circ}{T}*}
({\cal O}_{Y_{\os{\circ}{T}_0}/\os{\circ}{T}}\langle u \rangle)_{\mab Q}
$ 
is independent of the choice of an affine open covering of $Y_{\os{\circ}{T}_0}$ 
and a simplicial immersion 
$Y_{\os{\circ}{T}_0\bul} \os{\sus}{\lo} \ol{\cal Q}_{\bul}$ 
over $\ol{S(T)^{\nat}}$. 
\end{prop}
\begin{proof} 
By (\ref{eqn:exfte}) we obtain (\ref{prop:ijn}). 
We leave the detail of the proof to the reader. 
\end{proof} 

\par 
Now we consider the case where $Y_{\os{\circ}{T}_0}=X_{\os{\circ}{T}_0}$. 
For a subset $\ul{\lam}$ of $\Lam$, set 
\begin{align*} 
&(\wt{R}_{\rm TW}u_{X_{\ul{\lam},\os{\circ}{T}_0}/\os{\circ}{T}*}
({\cal O}_{X_{\ul{\lam},\os{\circ}{T}_0}/\os{\circ}{T}}\langle u \rangle), P)\\
&:=s_{\rm TW}\pi_{\ul{\lam},T{\rm zar}*}
((G({\cal O}_{{\mathfrak D}_{\bul}}\otimes_{{\cal O}_{{\cal P}^{{\rm ex}}_{\bul}}}
\otimes_{{\cal O}_{{\cal P}{}^{\rm ex}_{\bul}}}
\Om^{\bul}_{{\cal P}{}^{{\rm ex}}_{\bul\ul{\lam}}/\os{\circ}{T}}
\otimes_{\mab Z}{\mab Q})\langle U_{S(T)^{\nat}} \rangle),G(P))
\end{align*} 
in ${\rm Ho}
(A^{\geq 0}(f^{-1}({\cal O}_T))\otimes_{\mab Z}{\mab Q})$. 
Here $G$ means the Godement resolution. 
For a nonnegative integer $m$, we obtain 
the filtered dga 
\begin{align*} 
&(\wt{R}_{\rm TW}u_{X^{(m)}_{\os{\circ}{T}_0}/\os{\circ}{T}*}
({\cal O}_{X^{(m)}_{\os{\circ}{T}_0}/\os{\circ}{T}}\langle u \rangle),P)\\
&:=\bigoplus_{\# \ul{\lam}=m+1}
s_{\rm TW}\pi_{\ul{\lam},T{\rm zar}*}
((G({\cal O}_{{\mathfrak D}_{\bul}}\otimes_{{\cal O}_{{\cal P}^{{\rm ex}}_{\bul}}}
\otimes_{{\cal O}_{{\cal P}{}^{\rm ex}_{\bul}}}
\Om^{\bul}_{{\cal P}{}^{{\rm ex}}_{\bul\ul{\lam}}/\os{\circ}{T}}
\otimes_{\mab Z}{\mab Q})\langle U_{S(T)^{\nat}} \rangle),G(P))
\end{align*} 
in ${\rm Ho}
(A^{\geq 0}(f^{-1}({\cal O}_T)_{\os{\circ}{T}_0})\otimes_{\mab Z}{\mab Q})$. 
The semi-cosimplicial filtered dga 
\begin{align*} 
\{\bigoplus_{\# \ul{\lam}=m+1}
s_{\rm TW}\pi_{\ul{\lam},T{\rm zar}*}(
((G({\cal O}_{{\mathfrak D}_{\bul}}\otimes_{{\cal O}_{{\cal P}^{{\rm ex}}_{\bul}}}
\otimes_{{\cal O}_{{\cal P}{}^{\rm ex}_{\bul}}}
\Om^{\bul}_{{\cal P}{}^{{\rm ex}}_{\bul\ul{\lam}}/\os{\circ}{T}}
\otimes_{\mab Z}{\mab Q})\langle U_{S(T)^{\nat}}\rangle),G(P)))
\}_{m\in {\mab N}}
\end{align*} 
defines the following filtered dga: 
\begin{align*} 
&(H_{{\rm zar},{\rm TW}}(X_{\os{\circ}{T}_0}/S(T)^{\nat})_{\mab Q},P)\\
&:=s_{\rm TW}\{\bigoplus_{\# \ul{\lam}=m+1}
s_{\rm TW}\pi_{\ul{\lam},T{\rm zar}*}(
((G({\cal O}_{{\mathfrak D}_{\bul}}\otimes_{{\cal O}_{{\cal P}^{{\rm ex}}_{\bul}}}
\otimes_{{\cal O}_{{\cal P}{}^{\rm ex}_{\bul}}}
\Om^{\bul}_{{\cal P}{}^{{\rm ex}}_{\bul\ul{\lam}}/\os{\circ}{T}}
\otimes_{\mab Z}{\mab Q})\langle U_{S(T)^{\nat}} \rangle),G(P)))
\}_{m\in {\mab N}}
. 
\end{align*} 

\begin{defi}
We call 
\begin{align*} 
(H_{{\rm zar},{\rm TW}}(X_{\os{\circ}{T}_0}/S(T)^{\nat})_{\mab Q},P)
\end{align*} 
the {\it PD-Hirsch extension of the semi-cosimplicial zariskian 
filtered crystalline dga} of $X_{\os{\circ}{T}_0}/S(T)^{\nat}/\os{\circ}{T}$. 
\end{defi} 

\begin{prop}\label{prop:ncz}
Let 
\begin{align*} 
\Phi \col {\rm Ho}
(A^{\geq 0}(f^{-1}({\cal O}_T))\otimes_{\mab Z}{\mab Q})
\lo {\rm D}^+{\rm F}(f^{-1}({\cal O}_T)\otimes_{\mab Z}{\mab Q})
\end{align*} 
be the forgetful functor. 
Then 
\begin{align*} 
\Phi((H_{{\rm zar},{\rm TW}}(X_{\os{\circ}{T}_0}/S(T)^{\nat})_{\mab Q},P))
=(H_{\rm zar}(X_{\os{\circ}{T}_0}/S(T)^{\nat})_{\mab Q},P). 
\tag{7.6.1}\label{ali:twz}
\end{align*} 
\end{prop}
\begin{proof} 
This follows from the general theory (\cite[(6.3), (6.14)]{nav}). 
\end{proof}

\begin{rema}\label{rema:sroh}
In the following we mainly give the proofs of results only for $(H_{\rm zar},P)$, 
not for $(H_{{\rm zar},{\rm TW}},P)$, though we also give the statements for 
not only $(H_{\rm zar},P)$ but also $(H_{{\rm zar},{\rm TW}},P)$; 
we leave the proofs of 
the statements for $(H_{{\rm zar},{\rm TW}},P)$ to the reader because 
we have only to give additional minor arguments to the proofs of 
results for $(H_{\rm zar},P)$. 
\end{rema}

\section{Contravariant functorialities of 
PD-Hirsch pre-weight-filtered log crystalline complexes 
and Hirsch weight-filtered log crystalline dga's}\label{sec:fc}
Let the notations be as in the previous section. 
In this section we show the contravariant functoriality of 
$\wt{R}u_{Y_{\os{\circ}{T}_0}/S(T)^{\nat}*}
(\eps^*_{Y_{\os{\circ}{T}_0}/\os{\circ}{T}}(F)\langle u \rangle)$.  
We also show the contravariant functoriality of 
$(H_{\rm zar}(X_{\os{\circ}{T}_0}/S(T)^{\nat},E),P)$
with respect to a certain morphism which includes the 
abrelative Frobenius morphism defined in \cite{nb}. 
As a consequence, we obtain the contravariant functoriality of 
the spectral sequence (\ref{ali:spsarh}). 
In particular,  we can define a generalization of the Tate twist 
which appears in (\ref{ali:spsarh}). 
\par 
In the following we need the (not so standard) argument in \cite[(1.5)]{nb} 
because we consider non-nil immersions. 
\par 
Let the notations be as in the previous section. 
Let $S'$ and $(T',{\cal J}',\del')$ be the analogues of $S$ and $(T,{\cal J},\del)$, respectively. 
Let $S\lo S'$ be a morphism of log schemes and 
let $(T,{\cal J},\del)\lo (T',{\cal J}',\del')$ be 
a morphism of log PD-enlargements of $S$ and $S'$, respectively. 
Let $\star$ be nothing or $\prime$. 
Set $T^{\star}_0:=T^{\star}~{\rm mod}~{\cal J}^{\star}$. 
Let $Y$ and $Y'$ be log smooth schemes over $S$ and $S'$, respectively. 
Set $Y'_{\os{\circ}{T}{}'_0}:=Y'\times_{\os{\circ}{S}{}'}\os{\circ}{T}{}'_0$.  
Let 
\begin{equation*} 
\begin{CD} 
Y_{\os{\circ}{T}_0} @>{g}>> Y'_{\os{\circ}{T}{}'_0} \\ 
@VVV @VVV \\ 
S_{\os{\circ}{T}_0}  @>>> S'_{\os{\circ}{T}{}'_0}  \\ 
@V{\bigcap}VV @VV{\bigcap}V \\ 
S(T)^{\nat} @>{v}>> S'(T')^{\nat}
\end{CD}
\tag{8.0.1}\label{eqn:xdxbxss}
\end{equation*} 
be a commutative diagram of log smooth schemes 
over $S_{\os{\circ}{T}_0}$ and $S'_{\os{\circ}{T}{}'_0}$.  
Then there exist the \v{C}ech diagrams $Y_{\os{\circ}{T}_0\bul}$ 
and $Y'_{\os{\circ}{T}{}'_0\bul}$
of affine open coverings of 
$Y_{\os{\circ}{T}_0}$ and $Y'_{\os{\circ}{T}{}'_0}$, respectively and 
simplicial immersions 
$Y_{\os{\circ}{T}_0\bul} \os{\sus}{\lo} \ol{\cal Q}{}''_{\bul}$ 
and 
$Y'_{\os{\circ}{T}{}'_0\bul} \os{\sus}{\lo} \ol{\cal Q}{}'_{\bul}$ 
into log smooth schemes over  
$\ol{S(T)^{\nat}}$ and $\ol{S'(T')^{\nat}}$, respectively, 
such that $g$ induces a morphism $g_{\bul} \col Y_{\os{\circ}{T}_0\bul}
\lo Y'_{\os{\circ}{T}{}'_0\bul}$ of simplicial log schemes. 
Set $\ol{\cal Q}_{\bul}:= \ol{\cal Q}{}''_{\bul}
\times_{\ol{S(T)^{\nat}}}(\ol{\cal Q}{}'_{\bul}\times_{\ol{S'(T)'{}^{\nat}}}\ol{S(T)^{\nat}})$. 
Then the second projection $\ol{g}_{\bul}\col \ol{\cal Q}_{\bul} \lo \ol{\cal Q}{}'_{\bul}$ 
and the natural immersion $Y_{\os{\circ}{T}_0\bul}\os{\sus}{\lo} \ol{\cal Q}_{\bul}$ 
fits into the following commutative diagram 
\begin{equation*} 
\begin{CD} 
Y_{\os{\circ}{T}_0\bul}
@>{\sus}>> \ol{\cal Q}_{\bul}\\
@V{g_{\bul}}VV 
@VV{\ol{g}_{\bul}}V \\ 
Y'_{\os{\circ}{T}{}'_0\bul} @>{\sus}>> \ol{\cal Q}{}'_{\bul}
\end{CD} 
\tag{8.0.2}\label{cd:xpnxp} 
\end{equation*} 
over 
\begin{equation*} 
\begin{CD} 
S_{\os{\circ}{T}_0} @>{\subset}>> \ol{S(T)^{\nat}}\\ 
@VVV @VVV \\ 
S'_{\os{\circ}{T}{}'_0} @>{\subset}>> \ol{S'(T')^{\nat}}. 
\end{CD}
\end{equation*} 
Let $\ol{\mathfrak E}_{\bul}$ and $\ol{\mathfrak E}{}'_{\bul}$ be the log PD-envelopes of 
the immersions $Y_{\os{\circ}{T}_0\bul}\os{\sus}{\lo} \ol{\cal Q}_{\bul}$ 
and 
$Y_{\os{\circ}{T}{}'_0\bul}\os{\sus}{\lo} \ol{\cal Q}{}'_{\bul}$, respectively. 
Let $\ol{F}$ and $\ol{F}{}'$ be flat quasi-coherent crystals of 
${\cal O}_{Y_{\os{\circ}{T}_0}/\os{\circ}{T}}$-modules  
and ${\cal O}_{\os{\circ}{Y}{}'_{T'_0}/\os{\circ}{T}{}'}$-modules, 
respectively.     
Let $g^*_{{\rm crys}}(\ol{F}{}')\lo \ol{F}$
be a morphism of ${\cal O}_{Y_{\os{\circ}{T}_0}/\os{\circ}{T}}$-modules. 
Let $\star$ be nothing or $\prime$. 
Let $\ol{F}{}^{\star \bul}$ be the quasi-coherent 
${\cal O}_{Y^{\star}_{\os{\circ}{T}{}^{\star}_0\bul}/\os{\circ}{T}{}^{\star}}$-modules 
obtained by $\ol{F}{}^{\star}$. 
Let $(\ol{\cal F}{}^{\star \bul},\ol{\nabla}{}^{\star})$ 
be a quasi-coherent ${\cal O}_{\ol{\mathfrak E}{}^{\star}_{\bul}}$-module  
with integrable connection corresponding to 
the log crystal $\ol{F}{}^{\star}$: 
\begin{equation*} 
\ol{\nabla}{}^{\star}\col \ol{\cal F}{}^{\star\bul}\lo 
\ol{\cal F}{}^{\star \bul}\otimes_{{\cal O}_{\ol{\cal Q}{}^{{\star}{\rm ex}}_{\bul}}}
\Om^1_{\ol{\cal Q}{}^{{\star}{\rm ex}}_{\bul}/\os{\circ}{T}}.
\tag{8.0.3}\label{eqn:olbca}
\end{equation*}  
Set $({\cal F}^{\star \bul},\nabla^{\star}):=
(\ol{\cal F}{}^{\star \bul},\ol{\nabla}{}^{\star})
\otimes_{{\cal O}_{\ol{\mathfrak D}(\ol{S^{\star}(T^{\star})^{\nat}})}}
{\cal O}_{S^{\star}(T^{\star})^{\nat}}$:  
\begin{equation*} 
\nabla^{\star} \col {\cal F}^{\star \bul}\lo 
{\cal F}^{\star \bul}\otimes_{{\cal O}_{{\cal Q}^{\star {\rm ex}}_{\bul}}}
\Om^1_{{\cal Q}^{\star {\rm ex}}_{\bul}/\os{\circ}{T}{}^{\star}}. 
\tag{8.0.4}\label{eqn:olmbc}
\end{equation*}  
The morphisms $g^*_{{\rm crys}}(F')\lo F$,  
$\ol{g}\col \ol{\cal Q}_{\bul}\lo \ol{\cal Q}{}'_{\bul}$,  
(\ref{cd:xpnxp}) and (\ref{ali:kxvef}) induce the following morphism 
\begin{align*} 
\ol{g}{}^* \col 
{\cal F}'{}^{\bul}\otimes_{{\cal O}_{{\cal Q}'{}^{{\rm ex}}_{\bul}}}
\Om^{\bul}_{{\cal Q}'{}^{{\rm ex}}_{\bul}/\os{\circ}{T}{}'}\langle u' \rangle 
:=
{\cal F}'{}^{\bul}\otimes_{{\cal O}_{{\cal Q}'{}^{{\rm ex}}_{\bul}}}
\Om^{\bul}_{{\cal Q}'{}^{{\rm ex}}_{\bul}/\os{\circ}{T}{}'}
\langle U_{S'(T')^{\nat}} \rangle 
\lo
g_{\bul *}({\cal F}{}^{\bul}\otimes_{{\cal O}_{{\cal Q}{}^{{\rm ex}}_{\bul}}}
\Om^{\bul}_{{\cal Q}{}^{{\rm ex}}_{\bul}/\os{\circ}{T}})\langle u \rangle. 
\tag{8.0.5}\label{ali:ext}
\end{align*} 
Hence we obtain the following morphism  
\begin{align*} 
g^*\col \wt{R}u_{Y'_{\os{\circ}{T}{}'_0}/\os{\circ}{T}{}'*}(\ol{F}{}'\langle u'\rangle) 
\lo 
Rg_*\wt{R}u_{Y_{\os{\circ}{T}_0}/\os{\circ}{T}*}(\ol{F}\langle u \rangle).
\tag{8.0.6}\label{ali:exrt}
\end{align*} 

\begin{prop}\label{prop:indu}
$(1)$ The morphism {\rm (\ref{ali:exrt})} is independent of the choices of 
affine open coverings of $Y_{\os{\circ}{T}_0}$ and $Y_{\os{\circ}{T}{}'_0}$, 
and simplicial immersions 
$Y_{\os{\circ}{T}_0\bul} \os{\sus}{\lo} \ol{\cal Q}_{\bul}$ 
and 
$Y_{\os{\circ}{T}{}'_0\bul} \os{\sus}{\lo} \ol{\cal Q}{}'_{\bul}$ 
into log smooth schemes over  
$\ol{S(T)^{\nat}}$ and $\ol{S'(T')^{\nat}}$ and the morphism 
$\ol{g}_{\bul}\col \ol{\cal Q}_{\bul}\lo \ol{\cal Q}{}'_{\bul}$. 
\par 
$(2)$ Let $S'\lo S''$, $T'\lo T''$,  
$h\col Y'_{\os{\circ}{T}{}'_0}\lo Y''_{\os{\circ}{T}{}''_0}$ and 
$h^*_{{\rm crys}}(\ol{F}{}'')\lo \ol{F}{}'$
be analogous morphisms to $S\lo S'$, $T\lo T'$, 
$g\col Y_{\os{\circ}{T}{}_0}\lo Y'_{\os{\circ}{T}{}'_0}$
and $g^*_{{\rm crys}}(\ol{F}{}')\lo \ol{F}$, respectively.  
Then 
$$(h\circ g)^*=Rh_*(g^*)\circ h^*\col 
\wt{R}u_{Y''_{\os{\circ}{T}{}''_0}/\os{\circ}{T}{}''*}(\ol{F}{}''\langle u''\rangle) 
\lo 
R(h\circ g)_*\wt{R}u_{Y_{\os{\circ}{T}_0}/\os{\circ}{T}*}(\ol{F}\langle u \rangle).$$ 
and 
${\rm id}^*_{Y_{\os{\circ}{T}{}_0}}=
{\rm id}_{\wt{R}u_{Y_{\os{\circ}{T}_0}/\os{\circ}{T}*}(\ol{F}\langle u \rangle)}$. 
\end{prop}
\begin{proof} 
(1), (2): The proof needs only a routine work once 
one constructs the morphism (\ref{ali:exrt}) 
and know the well-definedness of 
$\wt{R}u_{Y_{\os{\circ}{T}_0}/\os{\circ}{T}*}(\ol{F}\langle u \rangle)$ ((\ref{theo:indp})). 
We leave the detailed proof to the reader. 
\end{proof}

\par  
Let $X$ and $Y$ be SNCL schemes over $S$ and $S'$, respectively. 
Let 
\begin{equation*} 
\begin{CD} 
X_{\os{\circ}{T}_0} 
@>{g}>> Y_{\os{\circ}{T}{}'_0}\\
@VVV @VVV \\ 
S_{\os{\circ}{T}_0} @>>> S'_{\os{\circ}{T}{}'_0} \\ 
@V{\bigcap}VV @VV{\bigcap}V \\ 
S(T)^{\nat} @>{v}>> S'(T')^{\nat}
\end{CD}
\tag{8.1.1}\label{eqn:xdxduss}
\end{equation*} 
be a commutative diagram of 
SNCL schemes 
over $S_{\os{\circ}{T}_0}$ and $S'_{\os{\circ}{T}{}'_0}$ 
\par 
Let $X_{\os{\circ}{T}_0,\bul} \os{\sus}{\lo} \ol{\cal P}{}'_{\bul}$
and 
$Y_{\os{\circ}{T}{}'_0\bul} \os{\sus}{\lo} \ol{\cal Q}_{\bul}$ 
be immersions into log smooth simplicial schemes over 
$\ol{S(T)^{\nat}}$ and $\ol{S'(T')^{\nat}}$, respectively. 
Indeed, these immersions exist by (\ref{eqn:eipxd}). 
Set 
$$\ol{\cal P}_{\bul}:=
\ol{\cal P}{}'_{\bul}\times_{\ol{S(T)^{\nat}}}
(\ol{\cal Q}_{\bul}\times_{\ol{S'(T')^{\nat}}}\ol{S(T)^{\nat}})
=\ol{\cal P}{}'_{\bul}\times_{\ol{S'(T')^{\nat}}}
\ol{\cal Q}_{\bul}.$$ 
Let 
$\ol{g}_{\bul}\col \ol{\cal P}_{\bul} 
\lo \ol{\cal Q}_{\bul}$
be the second projection. 
Then we have the following commutative diagram 
\begin{equation*} 
\begin{CD} 
X_{\os{\circ}{T}_0\bul}
@>{\sus}>> \ol{\cal P}_{\bul}\\
@V{g_{\bul}}VV 
@VV{\ol{g}_{\bul}}V \\ 
Y_{\os{\circ}{T}{}'_0,\bul} @>{\sus}>> \ol{\cal Q}_{\bul}
\end{CD} 
\tag{8.1.2}\label{cd:xpnxap} 
\end{equation*} 
over 
\begin{equation*} 
\begin{CD} 
S_{\os{\circ}{T}_0} @>{\subset}>> \ol{S(T)^{\nat}}\\ 
@VVV @VVV \\ 
S'_{\os{\circ}{T}{}'_0} @>{\subset}>> \ol{S'(T')^{\nat}}. 
\end{CD}
\end{equation*} 
Let $\ol{\mathfrak D}_{\bul}$ and $\ol{\mathfrak E}_{\bul}$ be the log PD-envelopes 
of  the immersions 
$X_{\os{\circ}{T}_0\bul}\os{\sus}{\lo}\ol{\cal P}_{\bul}$ over 
$(\os{\circ}{T},{\cal J},\del)$ and 
$Y_{\os{\circ}{T}{}'_0,\bul}
\os{\sus}{\lo}\ol{\cal Q}_{\bul}$ over 
$(\os{\circ}{T}{}',{\cal J}',\del')$, respectively.     
Set ${\mathfrak D}_{\bul}:=
\ol{\mathfrak D}_{\bul}\times_{{\mathfrak D}(\ol{S(T)^{\nat}})}S(T)^{\nat}$
and 
${\mathfrak E}_{\bul}:=
\ol{\mathfrak E}_{\bul}\times_{{\mathfrak D}(\ol{S'(T')^{\nat}})}S'(T')^{\nat}$. 
By (\ref{cd:xpnxap}) 
we have the following natural morphism  
\begin{equation*} 
\ol{g}{}^{\rm PD}_{\bul}\col 
\ol{\mathfrak D}_{\bul}\lo \ol{\mathfrak E}_{\bul}.   
\tag{8.1.3}\label{eqn:gpdf} 
\end{equation*}  
Hence we have the following natural morphism 
\begin{equation*} 
g^{\rm PD}_{\bul}\col 
{\mathfrak D}_{\bul}\lo {\mathfrak E}_{\bul}.   
\tag{8.1.4}\label{eqn:gpndf} 
\end{equation*}

\par 
Let $E$ and  $F$ be flat quasi-coherent crystals of 
${\cal O}_{\os{\circ}{X}_{T_0}/\os{\circ}{T}}$-modules  
and ${\cal O}_{\os{\circ}{Y}_{T_0'}/\os{\circ}{T}{}'}$-modules, respectively.     
Let 
\begin{align*} 
\os{\circ}{g}{}^*_{{\rm crys}}(F)
\lo E
\tag{8.1.5}\label{ali:gnfe} 
\end{align*} 
be a morphism of 
${\cal O}_{\os{\circ}{X}_{T_0}/\os{\circ}{T}}$-modules. 
Assume that the following condition holds:
\smallskip
\parno 
$(8.1.6)$: for any smooth component 
$\os{\circ}{X}_{\lam}$ of $\os{\circ}{X}_{T_0}$ over $\os{\circ}{T}_0$, 
there exists a unique smooth component 
$\os{\circ}{Y}_{\mu}$ of 
$\os{\circ}{Y}_{T'_0}$ over $\os{\circ}{T}{}'_0$ such that $g$ 
induces a morphism 
$\os{\circ}{X}_{\lam} \lo \os{\circ}{Y}_{\mu}$. 
\smallskip
\par 
The following theorem plays an important role in this book. 

\begin{theo}[{\bf Contravariant functoriality}]\label{theo:ccm}
Let $f'\col Y_{\os{\circ}{T}{}'_0}\lo S'(T')^{\nat}$ be the structural morphism. 
Then the following hold$:$
\par
$(1)$ The morphism $g\col X_{\os{\circ}{T}_0}\lo Y_{\os{\circ}{T}{}'_0}$ 
induces the following well-defined morphism 
\begin{align*} 
g^*\col (H_{\rm zar}(Y_{\os{\circ}{T}{}'_0}/S'(T')^{\nat},F),P)
\lo
Rg_*((H_{\rm zar}(X_{\os{\circ}{T}_0}/S(T)^{\nat},E),P)) 
\tag{8.2.1}\label{ali:ccm}
\end{align*} 
in ${\rm D}^+{\rm F}(f'{}^{-1}({\cal O}_{T'}))$. 
\par 
$(2)$ Let $Z/S''$, $(T'',{\cal J}',\del'')$ 
be similar objects to $Y/S'$ 
and $(T',{\cal J}',\del')$, 
respectively.  
Let $h\col Y_{\os{\circ}{T}{}'_0}\lo Z_{\os{\circ}{T}{}''_0}$ be a similar morphism to 
$g\col X_{\os{\circ}{T}_0}\lo Y_{\os{\circ}{T}{}'_0}$. 
Let $\os{\circ}{h}{}^*_{\rm crys}(G)\lo F$ be a similar morphism to 
$\os{\circ}{g}{}^*_{\rm crys}(F)\lo E$. 
Let $f''\col Z_{\os{\circ}{T}{}''_0}\lo S''(T'')^{\nat}$ 
be the structural morphism. 
Then  
\begin{align*} 
(h\circ g)^*=Rh_*(g^*)\circ h^*\col &
(H_{\rm zar}(Z_{\os{\circ}{T}{}''_0}/S''(T'')^{\nat},G),P)\lo 
R(h\circ g)_*((H_{\rm zar}(X_{\os{\circ}{T}_0}/S(T)^{\nat},E),P))
\tag{8.2.2}\label{ali:hgm}
\end{align*} 
in ${\rm D}^+{\rm F}(f''{}^{-1}({\cal O}_{T''}))$ and 
${\rm id}^*_{X_{\os{\circ}{T}_0}}=
{\rm id}_{(H_{\rm zar}(X_{\os{\circ}{T}_0}/S(T)^{\nat},E),P)}$. 
\par 
$(3)$ 
The isomorphisms {\rm (\ref{eqn:eaxte})} and {\rm (\ref{eqn:eetxte})} 
are functorial, that is, there exists a morphism 
\begin{align*} 
g^* \col Ru_{Y^{(\star)}_{\os{\circ}{T}{}'_0}/S'(T')^{\nat}*}
(\eps^*_{Y^{(\star)}_{\os{\circ}{T}{}'_0}/S'(T')^{\nat}}(F^{(\star)}))
\lo 
Rg_*Ru_{X^{(\star)}_{\os{\circ}{T}_0}/S(T)^{\nat}*}
(\eps^*_{X^{(\star)}_{\os{\circ}{T}_0}/S(T)^{\nat}}(E^{(\star)}))
\tag{8.2.3}\label{ali:eexte}
\end{align*} 
fitting into the following commutative diagram$:$ 
\begin{equation*}
\begin{CD} 
H_{\rm zar}(Y_{\os{\circ}{T}{}'_0}/S'(T')^{\nat},F,u')
@>{g^*}>>
Rg_*(H_{\rm zar}(X_{\os{\circ}{T}_0}/S(T)^{\nat},E))\\
@V{\simeq}VV @VV{\simeq}V \\
Ru_{Y^{(\star)}_{\os{\circ}{T}{}'_0}/S'(T')^{\nat}*}
(\eps^*_{Y^{(\star)}_{\os{\circ}{T}{}'_0}/S'(T')^{\nat}}(F^{(\star)}))
@>{g^*}>>
Rg_*Ru_{X^{(\star)}_{\os{\circ}{T}_0}/S(T)^{\nat}*}
(\eps^*_{X^{(\star)}_{\os{\circ}{T}_0}/S(T)^{\nat}}(E^{(\star)}))\\
@A{\simeq}AA @AA{\simeq}A\\
Ru_{Y_{\os{\circ}{T}{}'_0}/S'(T')^{\nat}*}
(\eps^*_{Y_{\os{\circ}{T}_0}/S'(T')^{\nat}}(F)) @>{g^*}>>
Rg_*Ru_{X_{\os{\circ}{T}_0}/S(T)^{\nat}*}
(\eps^*_{X_{\os{\circ}{T}_0}/S(T)^{\nat}}(E)). 
\end{CD} 
\tag{8.2.4}\label{cd:eetxte}
\end{equation*} 
The morphism {\rm (\ref{ali:eexte})} satisfies 
the similar transitive relation to that in ${\rm (2)}$. 
\par 
$(4)$: Assume that $\os{\circ}{T}$ is a $p$-adic formal scheme. 
The morphism $g\col X_{\os{\circ}{T}_0}\lo Y_{\os{\circ}{T}{}'_0}$ 
induces the following morphism 
\begin{align*} 
g^*\col (H_{{\rm zar},{\rm TW}}(Y_{\os{\circ}{T}{}'_0}/S'(T')^{\nat}),P)
\lo
R_{\rm TW}g_*((H_{{\rm zar},{\rm TW}}(X_{\os{\circ}{T}_0}/S(T)^{\nat}),P)) 
\tag{8.2.5}\label{ali:ccpm}
\end{align*} 
in ${\rm Ho}({\rm A}^{\geq 0}{\rm F}(f'{}^{-1}({\cal O}_{T'})\otimes_{\mab Z}{\mab Q}))$. 
This satisfies the similar transitive relation to that in $(2)$. 
\end{theo}
\begin{proof} 
(1), (2), (3): The proof is the same as those of \cite[(1.5.2)]{nb} and [loc.~cit., (1.5.8)], 
in which we prove the analogous theorem to (\ref{theo:ccm}) for the $p$-adic filtered 
Steenbrink complex. 
\par 
(4): We have only to pay attention to the multiplicative structures of 
the crystalline complexes. 
\end{proof} 

In the rest of this section, in addition to (8.1.6), 
assume that the following condition holds:
\medskip 
\parno
$(8.2.6)$: 
there exist positive integers $e_{\lam}$'s  
$(\lam \in \Lam)$ such that 
there exist local equations $x_{\lam}=0$ and 
$y_{\lam}=0$ of 
$\os{\circ}{X}_{\lam}$ and $\os{\circ}{Y}_{\lam}$, 
respectively,  
such that $g^*(y_{\lam})=x^{e_{\lam}}_{\lam}$.  

In \cite[(1.7.5)]{nb} we have proved the following: 

\begin{prop}[{\bf \cite[(1.7.5)]{nb}}]\label{prop:xxle} 
Let the assumptions and the notations be as above. 
Let $\lam$ be an element of $\Lam$ such that 
$x \in \os{\circ}{X}_{\lam}$. 
Then $\deg(v)_{f(x)}=e_{\lam}$.  
In particular, $e_{\lam}$'s are independent of $\lam$. 
\end{prop}

\begin{rema}\label{rema:imr}
In \cite{nb} we need additional assumptions  
for the morphism $g\col X_{\os{\circ}{T}_0}\lo Y_{\os{\circ}{T}{}'_0}$ 
satisfying the conditions (8.1.6) and (8.2.6) 
(e.g.,~$p\nmid \deg (v^{\nat})$) to discuss the 
contravariant functoriality of 
the filtered complex $(A_{\rm zar}(X_{\os{\circ}{T}_0}/S(T)^{\nat},E),P)$, 
which will be recalled in \S\ref{sec:pssc} below. 
We also need these additional assumptions  
for the following spectral sequence obtained by this filtered complex if one considers the compatibility of 
the actions $g^*$ on $R^qf_{X_{\os{\circ}{T}_0}/S(T)^{\nat}*}
(\eps^*_{X_{\os{\circ}{T}_0}/S(T)^{\nat}}(E))$ with the action 
$\os{\circ}{g}{}^*$ on the following $E_1$-terms into account: 
\begin{align*} 
E_1^{-k,q+k}&=\bigoplus_{j\geq \max \{-(k+m),0\}} 
R^{q-2j-k}f_{\os{\circ}{X}{}^{(2j+k)}_{T_0}
/\os{\circ}{T}*}
(E_{\os{\circ}{X}{}^{(2j+k)}_{T_0}
/\os{\circ}{T}}
\otimes_{\mab Z} \tag{8.4.1}\label{ali:ms}\\
&\phantom{R^{q-2j-k-m}f_{(\os{\circ}{X}^{(k)}, 
Z\vert_{\os{\circ}{X}^{(2j+k)}})/S*} 
({\cal O}}
  \vp^{(2j+k)}_{\rm crys}(
\os{\circ}{X}_{T_0}/\os{\circ}{T}))(-j-k,v) \\
&\Lo 
R^qf_{X_{\os{\circ}{T}_0}/S(T)^{\nat}*}
(\eps^*_{X_{\os{\circ}{T}_0}/S(T)^{\nat}}(E))  
\quad (q\in {\mab Z}).  
\end{align*}
Here $(-j-k,v)$ will be explained soon. 
\par 
In (\ref{ali:ccm}) we do not need the additional assumptions. 
However the $E_1$-terms of the spectral sequence (\ref{ali:sparh}) below 
have much more direct summands than those of (\ref{ali:ms}). 
\end{rema}

\begin{defi}
Let $w \col {\cal E} \lo {\cal F}$ 
be a morphism of  
$f^{-1}({\cal O}_T)$-modules (resp.~${\cal O}_T$-modules). 
The $D$-{\it twist}(:=degree twist) of $w$ by $k$  with respect to $v$ 
$$w(-k) \col {\cal E}(-k;v)\lo {\cal F}(-k;v)$$  
is, by definition, the morphism  
$\deg(v)^kw \col {\cal E} \lo {\cal F}$.  
This definition is well-defined for the derived category 
$D^+(f^{-1}({\cal O}_T))$ 
(resp.~$D^+({\cal O}_T)$). 
\end{defi}

Let $v\col S\lo S'$ be a morphism of families of log points 
and let $((T,{\cal J},\del),z)\lo ((T',{\cal J}',\del'),z')$ be a morphism of 
log PD-enlargements over $v$. 
Let $v^{\nat}\col S(T)^{\nat}\lo S'(T')^{\nat}$ be the induced morphism 
by $v$ and the morphism $T\lo T'$. 
Let $\ol{t}{}'$ be a local section of $M_{S'(T')^{\nat}}$ such that 
the image $t'$ of $\ol{t}{}'$ in $\ol{M}_{S'(T')^{\nat}}=M_{S'(T')^{\nat}}/{\cal O}_{T'}^*$ 
is the local generator. 
By the definition of $\deg(v^{\nat})_x$ ((\ref{defi:ddef})), 
we have 
\begin{equation*} 
v^{\nat*}(d\log \ol{t}{}')= \deg(v^{\nat})d\log \wt{t}. 
\tag{8.5.1}\label{eqn:uta}  
\end{equation*}  
Let us recall the ${\cal O}_{S'}$-linear morphism (\ref{ali:kxvef}) 
for the case above: 
\begin{equation*} 
v^*\col \Gam_{{\cal O}_{S'}}(U_{S'})\lo v_*(\Gam_{{\cal O}_S}(U_S)). 
\tag{8.5.2}\label{eqn:utva}  
\end{equation*} 

\begin{exem}\label{exem:abfd}
(1) In the case where $S$ is of characteristic $p>0$, 
set $S^{[p]}:=S\times_{\os{\circ}{S},\os{\circ}{F}_S}\os{\circ}{S}$, 
where $F_S\col S\lo S$ is the absolute Frobenius endomorphism of $S$. 
The endomorphism $F_S$ induces the natural morphism 
$F_{S/\os{\circ}{S}}\col S\lo S^{[p]}$. 
In \cite[(1.5.14)]{nb} we have called $F_{S/\os{\circ}{S}}$ the {\it abrelative Frobenius morphism} 
of $S$ (see also \cite{oc}).
Let $F_{T,T'}\col ((T,{\cal J},\del),z)\lo ((T',{\cal J}',\del'),z')$ be 
a morphism of log PD enlargements over 
$F_{S/\os{\circ}{S}}\col S\lo S':=S^{[p]}$.  
We can define the {\it abrelative Frobenius action} 
$F^{\nat}_{T,T'}\col S(T)^{\nat}\lo S^{[p]}(T')^{\nat}$ with respect to  $F_{T,T'}$.  
Here note that ${\rm deg}(F^{\nat}_{T,T'})=p$. 
\par  
Set $X^{[p]}:=X\times_SS^{[p]}=X\times_{\os{\circ}{S},\os{\circ}{F}_S}\os{\circ}{S}$. 
Let $T'_0$ be the exact closed log subscheme of $T'$ defined by 
the PD-ideal sheaf ${\cal J}'$. 
Then we have the abrelative Frobenius morphism 
\begin{align*} 
F^{\rm abr}_{X_{S_{\os{\circ}{T}_0}}/S(T)^{\nat},S^{[p]}(T')^{\nat}} 
\col X_{S_{\os{\circ}{T}_0}}\lo X^{[p]}_{S^{[p]}_{\os{\circ}{T}{}'_0}}
\tag{8.6.1}\label{ali:arf}
\end{align*} 
over $S(T)^{\nat}\lo S^{[p]}(T')^{\nat}$. 
This abrelative Frobenius morphism satisfies the conditions (8.1.6) and (8.2.6). 
\par 
When $(T',{\cal J}',\del')=(T,{\cal J},\del)$ and $z'$ is the composite morphism 
$F_{S/\os{\circ}{S}}\circ z$, then we denote $F^{\nat}_{T,T'}$ by 
$F^{\nat}_{T}$. In this case, we also denote the morphism 
(\ref{ali:arf}) by 
\begin{align*} 
F^{\rm abr}_{X_{S_{\os{\circ}{T}_0}}/S(T)^{\nat}} 
\col X_{S_{\os{\circ}{T}_0}}\lo X^{[p]}_{S^{[p]}_{\os{\circ}{T}{}_0}}. 
\tag{8.6.2}\label{ali:asrf}
\end{align*} 
We define the {\it abrelative Frobenius action} $F^{\nat,{\rm abr}*}_{T}$ with respect to  
$F^{\rm abr}_{T}$ on $u^{[i]}$ by the formula 
$$F^{\nat,{\rm abr}*}_{T}(u^{[i]})=p^iu^{[i]}.$$ 
Obviously the morphism $F^{\nat,{\rm abr}*}_{T}$ is independent of the choice of $u$. 
\par 
(2) Let the notations be as in (1). 
If there exists a PD-endomorphism
$F_T\col ((T,{\cal J},\del),z)\lo ((T,{\cal J},\del),z)$ 
over $F_S\col S \lo S$, 
then we have the endomorphism 
$S(T)^{\nat}\lo S(T)^{\nat}$ of $S(T)^{\nat}$ and 
we have the following absolute Frobenius endomorphism 
\begin{align*} 
F^{\rm abs}_{X_{S_{\os{\circ}{T}_0}}/S(T)^{\nat}} 
\col X_{S_{\os{\circ}{T}_0}}\lo X_{S_{\os{\circ}{T}{}_0}}
\tag{8.6.3}\label{ali:arsf}
\end{align*} 
This absolute Frobenius endomorphism over $S(T)^{\nat}\lo S(T)^{\nat}$ 
satisfies the conditions (8.1.6) and (8.2.6). 
We can define the {\it absolute Frobenius action} $F^{\nat{\rm abs}*}_{T}$ with respect to  
$F_{T}$ on $u^{[i]}$ by  the formula 
$F^{\nat,{\rm abs}*}_{T}(u^{[i]})=p^iu^{[i]}$. 
\end{exem} 

\par 
Let the notations and the assumptions be as in the previous section. 
By the same proof as that of \cite[(1.5.9)]{nb}, we see that 
the following isomorphism appearing in (\ref{ali:vspgrc}) 
is the following isomorphism 
\begin{align*} 
&{\rm gr}_k^{P}
(s({\cal E}^{\bul}
\otimes_{{\cal O}_{{\cal P}^{\rm ex}_{\bul}}}
\Om^{\bul}_{{\cal P}^{{\rm ex},(\bul)}_{\bul}/\os{\circ}{T}}\langle u \rangle)) 
\os{\sim}{\lo} \tag{8.6.4}\label{ali:vsgrc}\\
&\bigoplus_{m\geq 0}
\bigoplus_{\# \ul{\lam}=m+1}
\bigoplus_{j\geq 0}
\bigoplus_{\# \ul{\mu}=k+m-2j}
({\cal E}^{\bul} 
\otimes_{{\cal O}_{{\cal P}_{\bul}^{\rm ex}}}
b_{\ul{\lam}\cup \ul{\mu}*}
(\Om^{\bul}_{
\os{\circ}{\cal P}{}^{{\rm ex}}_{\bul,\ul{\lam}
\cup \ul{\mu}}/\os{\circ}{T}} \\
& \otimes_{\mab Z}\vp_{{\rm zar},\ul{\mu}}
({\cal P}^{\rm ex}_{\bul}/\os{\circ}{T})))(-(k+m-j);v)[-k-2m+2j]. 
\end{align*} 


\begin{prop}\label{prop:ngr}
For a nonempty set $\ul{\lam}$, let 
$E_{\os{\circ}{X}_{\ul{\lam},T_0}/\os{\circ}{T}}$ 
be the restriction of $E$ to 
$(\os{\circ}{X}_{\ul{\lam},T_0}/\os{\circ}{T})_{\rm crys}$. 
Let $k$ be an integer. Then the following hold$:$ 
\par 
$(1)$ There exists the following natural isomorphism 
\begin{align*}
& {\rm gr}_k^PH_{\rm zar}(X_{\os{\circ}{T}_0}/S(T)^{\nat},E)
\os{\sim}{\lo} \tag{8.7.1}\label{eqn:ee}\\
& \bigoplus_{m\geq 0}
\bigoplus_{\# \ul{\lam}=m+1}
\bigoplus_{j\geq 0}
\bigoplus_{\# \ul{\mu}=k+m-2j}
a_{\ul{\lam}\cup \ul{\mu},\os{\circ}{T}_0*}
Ru_{\os{\circ}{X}_{\ul{\lam}\cup \ul{\mu},T}
/\os{\circ}{T}}
(E_{\os{\circ}{X}_{\ul{\lam}\cup \ul{\mu},T}
/\os{\circ}{T}}
\otimes_{\mab Z}
\vp_{{\rm crys},\ul{\mu}}(\os{\circ}{X}_T/\os{\circ}{T})) \\
& [-k-2m+2j](-(k+m-j);v). 
\end{align*} 
\par 
$(2)$ 
There exists the following natural isomorphism 
\begin{align*}
& {\rm gr}_k^P\wt{R}u_{X^{(\star)}_{\os{\circ}{T}_0}/\os{\circ}{T}*}
(\eps^*_{X^{(\star)}_{\os{\circ}{T}_0}/\os{\circ}{T}}
(E))
\os{\sim}{\lo} \tag{8.7.2}\label{eqn:ege}\\
& \bigoplus_{m\geq 0}
\bigoplus_{\# \ul{\lam}=m+1}
\bigoplus_{\# \ul{\mu}=k+m}
a_{\ul{\lam}\cup \ul{\mu},\os{\circ}{T}_0*}
Ru_{\os{\circ}{X}_{\ul{\lam}\cup \ul{\mu},T}
/\os{\circ}{T}}
(E_{\os{\circ}{X}_{\ul{\lam}\cup \ul{\mu},T}
/\os{\circ}{T}}
\otimes_{\mab Z}
\vp_{{\rm crys},\ul{\mu}}(\os{\circ}{X}_T/\os{\circ}{T})) \\
& [-k-2m](-k-m;v). 
\end{align*} 
\end{prop}
\begin{proof} 
(1): By \cite[(1.3.4.5)]{nh2} and (\ref{ali:vsgrc}) 
we have the following formula:  
\begin{align*} 
&{\rm gr}_k^PH_{\rm zar}(X_{\os{\circ}{T}_0}/S(T)^{\nat},E)
={\rm gr}_k^P\wt{R}u_{X^{(\star)}_{\os{\circ}{T}_0}/\os{\circ}{T}*}
(\eps^*_{X^{(\star)}_{\os{\circ}{T}_0}/\os{\circ}{T}}
(E)\langle u \rangle)
\tag{8.7.3}\label{ali:grtec}\\
&={\rm gr}_k^P
R\pi_{{\rm zar}*}
(s({\cal E}^{\bul}
\otimes_{{\cal O}_{{\cal P}^{{\rm ex},(\bul)}_{\bul}}}
\Om^{\bul}_{{\cal P}{}^{{\rm ex},(\bul)}_{\bul}/\os{\circ}{T}}\langle u \rangle) \\
&= R\pi_{{\rm zar}*}
({\rm gr}_k^{P}s({\cal E}^{\bul}
\otimes_{{\cal O}_{{\cal P}^{{\rm ex},(\bul)}_{\bul}}}
\Om^{\bul}_{{\cal P}^{{\rm ex},(\bul)}_{\bul}/\os{\circ}{T}}\langle u \rangle))) \\
&\os{\sim}{\lo} R\pi_{{\rm zar}*}
(
\bigoplus_{m\geq 0}
\bigoplus_{\# \ul{\lam}=m+1}
\bigoplus_{j\geq 0}
\bigoplus_{\# \ul{\mu}=k+m-2j}
({\cal E}^{\bul}
\otimes_{{\cal O}_{{\cal P}^{\rm ex}_{\bul}}}
b_{\ul{\lam}\cup \ul{\mu}*}
(\Om^{\bul}_{
\os{\circ}{\cal P}{}^{{\rm ex}}_{\bul, \ul{\lam}
\cup \ul{\mu}}/\os{\circ}{T}}
\otimes_{\mab Z}\\
&\quad \quad \quad \quad 
\quad \quad \quad \quad
\vp_{{\rm zar},\ul{\mu}}
({\cal P}^{\rm ex}_{\bul}/S)))(-(k+m-j;v)))[-k-2m+2j] 
\\
&=\bigoplus_{m\geq 0}
\bigoplus_{\# \ul{\lam}=m+1}
\bigoplus_{j\geq 0}
\bigoplus_{\# \ul{\mu}=k+m-2j}
a_{\ul{\lam}\cup \ul{\mu},\os{\circ}{T}_0*}
Ru_{\os{\circ}{X}_{\ul{\lam}\cup \ul{\mu},T}/\os{\circ}{T}}
(E_{\os{\circ}{X}_{\ul{\lam}\cup \ul{\mu},T}/\os{\circ}{T}}
\otimes_{\mab Z}
\vp_{{\rm crys},\ul{\mu}}(\os{\circ}{X}_T/\os{\circ}{T}))\\
&[-k-2m+2j](-(k+m-j);v). 
\end{align*}
\par 
(2): By omitting 
$U_{S(T)^{\nat}} \otimes_{{\cal O}_T}$ 
in the proof of (1), we obtain (\ref{eqn:ege}). 
\end{proof}


\begin{coro}\label{coro:wtsp}
$(1)$ Assume that $\os{\circ}{X}_{T_0}$ is quasi-compact.  
Then there exists the following convergent spectral sequence$:$  
\begin{align*} 
&E_1^{-k,q+k}=
\bigoplus_{m\geq 0}
\bigoplus_{\# \ul{\lam}=m+1}
\bigoplus_{j\geq 0}
\bigoplus_{\# \ul{\mu}=k+m-2j}
R^{q+2j-k-2m}
f_{\os{\circ}{X}_{\ul{\lam}\cup \ul{\mu},T_0}/\os{\circ}{T}*} 
(E\vert_{\os{\circ}{X}_{\ul{\lam}\cup \ul{\mu},T_0}/\os{\circ}{T}}
\tag{8.8.1}\label{ali:sparh}\\
&\otimes_{\mab Z}\vp_{{\rm crys},\ul{\mu}}
(\os{\circ}{X}_{\ul{\lam}\cup \ul{\mu},T_0}/T))
(-(k+m-j),v) 
\Lo R^qf_{X_{\os{\circ}{T}_0}/S(T)^{\nat}*}
(\eps^*_{X_{\os{\circ}{T}_0}/S(T)^{\nat}}(E)). 
\end{align*} 
\par 
$(2)$ Assume that $\os{\circ}{X}_{T_0}$ is quasi-compact.  
Then there exists the following convergent spectral sequence$:$  
\begin{align*} 
&E_1^{-k,q+k}=
\bigoplus_{m\geq 0}
\bigoplus_{\# \ul{\lam}=m+1}
\bigoplus_{\# \ul{\mu}=k+m}
R^{q-k-2m}
f_{\os{\circ}{X}_{\ul{\lam}\cup \ul{\mu},T_0}/T*} 
(E\vert_{\os{\circ}{X}_{\ul{\lam}\cup \ul{\mu},T_0}/T}
\tag{8.8.2}\label{ali:spexh}\\
&\otimes_{\mab Z}\vp_{{\rm crys},\ul{\mu}}
(\os{\circ}{X}_{\ul{\lam}\cup \ul{\mu},T_0}/T))
(-k-m,v) 
\Lo \wt{R}{}^qf_{X_{\os{\circ}{T}_0}/\os{\circ}{T}*}
(\eps^*_{X_{\os{\circ}{T}_0}/\os{\circ}{T}}(E)). 
\end{align*}  
\end{coro}
\begin{proof} 
(1): (1) immediately follows from (\ref{coro:ts}) (2) and (\ref{prop:ngr}) (1). 
\par 
(2): (2) immediately follows from (\ref{eqn:e}) and (\ref{prop:ngr}) (2). 
\end{proof} 

\begin{defi}\label{defi:nif}
We call the induced filtration on $R^hf_{X_{\os{\circ}{T}_0}/S(T)^{\nat}*}
(\eps^*_{X_{\os{\circ}{T}_0}/S(T)^{\nat}}(E))$ 
by the spectral sequence (\ref{ali:sparh}) the {\it PD-Hirsch filtration} on 
$R^hf_{X_{\os{\circ}{T}_0}/S(T)^{\nat}*}
(\eps^*_{X_{\os{\circ}{T}_0}/S(T)^{\nat}}(E))$. 
When $E={\cal O}_{\os{\circ}{X}_{T_0}/\os{\circ}{T}}$, 
we call the following spectral sequence 
\begin{align*} 
&E_1^{-k,q+k}=
\bigoplus_{m\geq 0}
\bigoplus_{\# \ul{\lam}=m+1}
\bigoplus_{j\geq 0}
\bigoplus_{\# \ul{\mu}=k+m-2j}
R^{q+2j-k-2m}
f_{\os{\circ}{X}_{\ul{\lam}\cup \ul{\mu},T_0}/\os{\circ}{T}*} 
({\cal O}_{\os{\circ}{X}_{\ul{\lam}\cup \ul{\mu},T_0}/\os{\circ}{T}}
\tag{8.9.1}\label{ali:spaorh}\\
&\otimes_{\mab Z}\vp_{{\rm crys},\ul{\mu}}
(\os{\circ}{X}_{\ul{\lam}\cup \ul{\mu},T_0}/\os{\circ}{T}))
(-(k+m-j),v) 
\Lo R^qf_{X_{\os{\circ}{T}_0}/S(T)^{\nat}*}
({\cal O}_{X_{\os{\circ}{T}_0}/S(T)^{\nat}}). 
\end{align*} 
the {\it PD-Hirsch pre-weight spectral sequence} of  
$R^hf_{X_{\os{\circ}{T}_0}/S(T)^{\nat}*}
({\cal O}_{X_{\os{\circ}{T}_0}/S(T)^{\nat}})$.
\end{defi} 

\begin{rema}\label{rema:pdhkh}
The $E_1$-term in \cite[p.~1271]{kiha} is mistaken 
because $t$ can be greater than $n$ in [loc.~cit.]; in (\ref{ali:spaorh}), 
the $m$ in (\ref{ali:spaorh}) can be greater than $k$. 
Furthermore, the complex $G_i$ in [loc.~cit.] has not been made explicit and 
$G_i$ depends also on $t$ in [loc.~cit.]; $G_i$ does not depend only on $i$; 
$G_i$ should be denoted by $G_{i,n,t}$; in (\ref{ali:spaorh}), 
$G_{i,t}$ corresponds to 
\begin{align*} 
\bigoplus_{\# \ul{\lam}=m+1}
\bigoplus_{\# \ul{\mu}=k+m-2j}
R^{q+2j-k-2m}
f_{\os{\circ}{X}_{\ul{\lam}\cup \ul{\mu},T_0}/\os{\circ}{T}*} 
&({\cal O}_{\os{\circ}{X}_{\ul{\lam}\cup \ul{\mu},T_0}/\os{\circ}{T}}\\
&\otimes_{\mab Z}\vp_{{\rm crys},\ul{\mu}}
(\os{\circ}{X}_{\ul{\lam}\cup \ul{\mu},T_0}/\os{\circ}{T}))
(-(k+m-j),v),
\end{align*} 
which depends on $j$, $k$  and $m$.  
\end{rema} 

\par
Next we describe the boundary morphism 
between the $E_1$-terms of (\ref{ali:sparh}).
Fix a total order on $\Lam$ once and for all. 
Consider $\ul{\mu}\in P(\Lam)$ such that 
$\# \ul{\mu}=(k-1)+(m+1)-2j=k+m-2j$.  
Then 
$E^{-(k-1),q+k}_1$ contains the following higher direct image
$$R^{q+2j-k-2m}
f_{\os{\circ}{X}_{\ul{\lam}\cup \ul{\mu},T_0}/T*} 
(E\vert_{\os{\circ}{X}_{\ul{\lam}\cup \ul{\mu},T_0}/T}
\otimes_{\mab Z}\vp_{{\rm crys},\ul{\mu}}
(\os{\circ}{X}_{\ul{\lam}\cup \ul{\mu},T_0}/T))
(-(k+m-j);v)$$ 
with $\# \ul{\lam}=m+2$ as a direct factor. 
Describe $\ul{\lam}:=\{\lam_0,\ldots, \lam_{m+1}\}$ 
$(\lam_0< \cdots <\lam_{m+1})$
and set $\ul{\lam}_l:=\ul{\lam}\setminus \{\lam_l\}$ $(0 \leq l \leq m+1)$. 
Here $~~\widehat{}~~$ means the elimination. 
Let $\iota_{\ul{\lam}_l,\ul{\lam},\os{\circ}{T}_0}\col 
\os{\circ}{X}_{\ul{\lam}\cup \ul{\mu},\os{\circ}{T}_0}
\os{\sus}{\lo} 
\os{\circ}{X}_{\ul{\lam}_l\cup \ul{\mu},\os{\circ}{T}_0}$ be the natural 
closed immersion. 
If $\ul{\mu}\owns \lam_l$, then the morphism 
$\iota_{\ul{\lam}_l,\ul{\lam},\os{\circ}{T}_0}$ is the identity of 
$\os{\circ}{X}_{\ul{\lam}\cup \ul{\mu},\os{\circ}{T}_0}$.  
Let 
\begin{align*} 
\rho_{\ul{\lam},\ul{\lam}_l} \col &
R^{q+2j-k-2m}f_{\os{\circ}{X}_{\ul{\lam}_l\cup \ul{\mu},T_0}/T*} 
(E\vert_{\os{\circ}{X}_{\ul{\lam}_l\cup \ul{\mu},T_0}/T} 
\otimes_{\mab Z}\vp_{{\rm crys},\ul{\mu}}
(\os{\circ}{X}_{\ul{\lam}_l\cup \ul{\mu},T_0}/T))(-(k+m-j))
 \\
&\lo 
R^{q+2j-k-2m}f_{\os{\circ}{X}_{\ul{\lam}\cup \ul{\mu},T_0}/T*} 
(E\vert_{\os{\circ}{X}_{\ul{\lam}\cup \ul{\mu},T_0}/T} 
\otimes_{\mab Z}\vp_{{\rm crys},\ul{\mu}}
(\os{\circ}{X}_{\ul{\lam}\cup \ul{\mu},T_0}/T))
(-(k+m-j))
\end{align*} 
be the induced morphism from $\iota_{\ul{\lam}_l,\ul{\lam},\os{\circ}{T}_0}$. 
\par 
Describe $\ul{\mu}=\{\mu_0,\ldots,\mu_{k+m-2j}\}$ $(\mu_0<\cdots<\mu_{k+m-2j})$ 
and set $\ul{\mu}_{l'}:=\ul{\mu} \setminus \{\mu_{l'}\}$ 
$(0 \leq l' \leq k+m-2j)$. 
The $E_1$-term $E^{-(k-1),q+k}_1$ also 
contains the following cohomological sheaf 
$$R^{q+2j-k-2m+2}
f_{\os{\circ}{X}_{\ul{\lam}\cup \ul{\mu}_{l'},T_0}/T*} 
(E\vert_{\os{\circ}{X}_{\ul{\lam}\cup \ul{\mu}_{l'},T_0}/T}
\otimes_{\mab Z}\vp_{{\rm crys},\ul{\mu}_{l'}}
(\os{\circ}{X}_{\ul{\lam}\cup \ul{\mu},T_0}/T))
(-(k+m-j);v)$$ 
with $\# \ul{\lam}=m+1$ as a direct factor. 
 Let $\iota(\ul{\mu}_{l'})\col \os{\circ}{X}_{\ul{\lam}\cup \ul{\mu}_{l'},\os{\circ}{T}_0}
\os{\sus}{\lo} \os{\circ}{X}_{\ul{\lam}\cup \ul{\mu},\os{\circ}{T}_0}$ be the natural 
closed immersion. If $\ul{\lam}\owns \mu_{l'}$, 
then the morphism $\iota(\ul{\mu}_{l'})$ is the identity of 
$\os{\circ}{X}_{\ul{\lam}\cup \ul{\mu}_{l'},T_0}$.  
If $\ul{\lam}\not\owns \mu_{l'}$, 
then $\os{\circ}{X}_{\ul{\lam}\cup \ul{\mu},T_0}$ is a smooth divisor 
on $X_{\ul{\lam}\cup \ul{\mu}_{l'},T_0}/T_0$.
Let 
\begin{align*} 
&(-1)^{l'}G_{\ul{\mu}_{l'},\ul{\mu}} \col 
R^{q+2j-k-2m}f_{\os{\circ}{X}_{\ul{\lam}\cup \ul{\mu},T_0}/T*} 
(E\vert_{\os{\circ}{X}_{\ul{\lam}\cup \ul{\mu},T_0}/T} 
\otimes_{\mab Z}\vp_{{\rm crys},\ul{\mu}}
(\os{\circ}{X}_{\ul{\lam}\cup \ul{\mu},T_0}/T))
(-(k+m-j);v) \\
&\lo 
R^{q+2j-k-2m}f_{\os{\circ}{X}_{\ul{\lam}_l\cup \ul{\mu}_{l'},T_0}/T*} 
(E\vert_{\os{\circ}{X}_{\ul{\lam}_l\cup \ul{\mu}_{l'},T_0}/T} 
\otimes_{\mab Z}\vp_{{\rm crys},\ul{\mu}_{l'}}
(\os{\circ}{X}_{\ul{\lam}_l\cup \ul{\mu}_{l'},T_0}/T))
(-(k+m-j);v)
\end{align*} 
be the Gysin morphism defined in \cite[(2.8.4.5)]{nh2} in the case 
$\ul{\lam}\not\owns \mu_{l'}$. 
In the case $\ul{\lam}\owns \mu_{l'}$, we denote 
$(-1)^{l'}{\rm id}_{R^{q+2j-k-2m}f_{\os{\circ}{X}_{\ul{\lam}\cup \ul{\mu},T_0}/T*} 
(E\vert_{\os{\circ}{X}_{\ul{\lam}\cup \ul{\mu},T_0}/T} 
\otimes_{\mab Z}\vp_{{\rm crys},\ul{\mu}}
(\os{\circ}{X}_{\ul{\lam}\cup \ul{\mu},T_0}/T))
(-(k+m-j);v)}$ by 
$(-1)^{l'}G_{\ul{\mu}_{l'},\ul{\mu}}$ by abuse of notation. 
\par 
We obtain the following as in \cite[(1.5.21)]{nb}:

\begin{prop}\label{prop:bddes} 
Let 
$d_1^{-k,q+k} \col E_1^{-k,q+k}\lo 
E_1^{-k+1,q+k}$ be the boundary morphism. 
Then $d_1^{-k,q+k}$ is described by the following diagram$:$
\begin{equation*} 
\begin{split} 
{} & R^{q+2j-k-2m}f_{\os{\circ}{X}_{\ul{\lam}_l\cup \ul{\mu},T_0}/T*} 
(E\vert_{\os{\circ}{X}_{\ul{\lam}_l\cup \ul{\mu},T_0}/T} \\ 
{} & \phantom{R^{q-2\ul{t}_r-k}
f_{(D^{(\ul{t}_r+k)}_{(\ul{t}+e_j)}, 
Z\vert_{D^{(\ul{t}_r+k)}_{(\ul{t}+e_j)}})/S*}
({\cal O}}
\otimes_{\mab Z}\vp_{{\rm crys},\ul{\mu}}
(\os{\circ}{X}_{\ul{\lam}_l\cup \ul{\mu},T_0}/T))
(-(k+m-j);v)
\end{split} 
\tag{8.11.1}\label{cd:gsd}
\end{equation*}  
$$\text{\scriptsize
{${(-1)^{m+l}\rho_{\ul{\lam},\ul{\lam}_l}}$}}~\downarrow~$$
\begin{equation*} 
\begin{split} 
{} & R^{q+2j-k-2m}f_{\os{\circ}{X}_{\ul{\lam}\cup \ul{\mu},T_0}/T*} 
(E\vert_{\os{\circ}{X}_{\ul{\lam}\cup \ul{\mu},T_0}/T} \\ 
{} & \phantom{R^{q-2\ul{t}_r-k}f_{(D^{(\ul{t}_r+k)}_{\ul{t}}, 
Z\vert_{D^{(\ul{t}_r+k)}_{\ul{t}}})/S*}
({\cal O}} 
\otimes_{\mab Z}\vp_{{\rm crys},\ul{\mu}}
(\os{\circ}{X}_{\ul{\lam}\cup \ul{\mu},T_0}/T))
(-(k+m-j);v)
\end{split} 
\end{equation*}  
$$\text{\scriptsize
{${-(-1)^{m+l'}G_{\ul{\mu}_{l'},\ul{\mu}}}$}}~\downarrow~$$
\begin{equation*} 
\begin{split} 
{} & R^{q+2j-k-2m+2}
f_{\os{\circ}{X}_{\ul{\lam}\cup \ul{\mu}_{l'},T_0}/T*} 
(E\vert_{\os{\circ}{X}_{\ul{\lam}\cup \ul{\mu}_{l'},T_0}/T} \\
{} & \phantom{R^{q-2\ul{t}_r-k+2}f_{(D^{(\ul{t}_r+k-1)}_{\ul{t}},
Z\vert_{D^{(\ul{t}_r+k-1)}_{\ul{t}}})/S*}
({\cal O}}
\otimes_{\mab Z}\vp_{{\rm crys},\ul{\mu}_{l'}}
(\os{\circ}{X}_{\ul{\lam}\cup \ul{\mu}_{l'},T_0}/T))
(-(k+m-j)+1;v).
\end{split} 
\end{equation*}   
\end{prop}



\section{Filtered base change theorem of PD-Hirsch pre-weight-filtered complexes}\label{sec:bckf}
In this section we prove the filtered base change theorem 
of $(H_{\rm zar},P)$. 
\par
Let the notations be as in the previous section. 
Assume that $\os{\circ}{X}_{T_0}$ is quasi-compact. 
Let $f \col X_{\os{\circ}{T}_0} \lo S(T)^{\nat}$ 
be the structural morphism.  

\begin{prop}\label{prop:bdccd}  
Assume that $\os{\circ}{f} \col \os{\circ}{X}_{T_0}\lo \os{\circ}{T}$ 
is quasi-compact and quasi-separated. 
Then $Rf_*((H_{\rm zar}(X_{\os{\circ}{T}_0}/S(T)^{\nat},E),P))$ 
is isomorphic to a bounded filtered complex of 
${\cal O}_T$-modules. 
\end{prop}
\begin{proof}
By  \cite[7.6 Theorem]{bob},  
$Rf_{\os{\circ}{X}{}^{(l)}_{m,T_0}/\os{\circ}{T}*}
(E_{\os{\circ}{X}{}^{(l)}_{m,T_0}/\os{\circ}{T}})$ 
$(0\leq m \leq N, l\in {\mab N})$ is bounded. 
Hence 
$Rf_*((H_{\rm zar}(X_{\os{\circ}{T}_0}/S(T)^{\nat},E),P))$ 
is bounded by the spectral sequence (\ref{ali:spaorh}). 
\end{proof}

\begin{theo}[{\bf Log base change theorem of 
$(H_{\rm zar},P)$}]\label{theo:bccange} 
Let the assumptions be as in {\rm (\ref{prop:bdccd})}.  
Let $(T',{\cal J}',\del')$ be another log PD-enlargement over $S$. 
Assume that  ${\cal J}'$ is quasi-coherent. 
Set $T'_0:=\ul{\rm Spec}^{\log}_{T'}({\cal O}_{T'}/{\cal J}')$. 
Let $v\col (S(T')^{\nat},{\cal J}',\del') \lo (S(T)^{\nat},{\cal J},\del)$ be 
a morphism of fine log PD-schemes. 
Let 
$f' \col X_{\os{\circ}{T}{}'_0}=X\times_{S}S_{\os{\circ}{T}{}'_0} \lo S(T')^{\nat}$ 
be the base change morphism of $f$  
by the morphism $S(T')^{\nat}\lo S(T)^{\nat}$.  
Let $q \col X_{\os{\circ}{T}{}'_0} \lo X_{\os{\circ}{T}_0}$ 
be the induced morphism by $u$. 
Then there exists 
the following canonical filtered isomorphism
\begin{equation*}
Lu^*Rf_*((H_{\rm zar}(X_{\os{\circ}{T}_0}/S(T)^{\nat},E),P)) 
\os{\sim}{\lo} Rf'_*((H_{\rm zar}(X_{\os{\circ}{T}{}'_0}/S(T')^{\nat},
\os{\circ}{q}{}^{*}_{\rm crys}(E)),P))
\tag{9.2.1}\label{eqn:blucpw}
\end{equation*}
in ${\rm DF}(f'{}^{-1}({\cal O}_{T'}))$. 
\end{theo}
\begin{proof}  
Let the notations be as in \S\ref{sec:psc}. 
Set 
$\ol{\cal P}_{\bul,\ol{S(T')^{\nat}}}
:=\ol{\cal P}_{\bul}\times_{\ol{S(T)^{\nat}}}\ol{S(T')^{\nat}}$. 
Let   
$\ol{\mathfrak D}{}'_{\bul}$ 
be the log PD-envelope of the immersion 
$X_{\os{\circ}{T}{}'_0}\os{\sus}{\lo} \ol{\cal P}_{\bul,\ol{S(T')^{\nat}}}$ 
over $(\os{\circ}{T}{}',{\cal J}',\del')$. 
Then we have the natural morphisms 
$\ol{\cal P}_{\bul,\ol{S(T')^{\nat}}}\lo \ol{\cal P}_{\bul}$ 
and 
$\ol{\mathfrak D}{}'_{\bul}  \lo \ol{\mathfrak D}_{\bul}$.  
We also have the identity morphism 
${\rm id}\col \os{\circ}{q}{}^{*}_{\rm crys}(E)
\lo \os{\circ}{q}{}^{*}_{\rm crys}(E)$. 
Obviously the morphism 
$q\col X_{\os{\circ}{T}{}'_0}\lo X_{\os{\circ}{T}{}_0}$ satisfies 
the conditions (8.1.6).  
Hence we have the following natural morphism  
\begin{equation*} 
(H_{\rm zar}(X_{\os{\circ}{T}_0}/S(T)^{\nat},E),P) 
\lo 
Rq_*((H_{\rm zar}(X_{\os{\circ}{T}{}'_0}/S(T')^{\nat},E),P))
\tag{9.2.2}\label{eqn:bcxa}
\end{equation*} 
by (\ref{theo:ccm}).
By applying $Rf_*$ to (\ref{eqn:bcxa}) and using 
the adjoint property of $L$ and $R$ (\cite[(1.2.2)]{nh2}), 
we have the natural morphism (\ref{eqn:blucpw}). 
Here we have used the boundedness in 
(\ref{prop:bdccd}) for the well-definedness of $Lu^*$. 
\par
The rest of the proof is the same as that of \cite[(1.6.2)]{nb}. 
\end{proof} 

\par 
Let $\os{\circ}{Y}$ be a smooth scheme over $\os{\circ}{T}$. 
Endow $\os{\circ}{Y}$ with the inverse image of $M_{S_{\os{\circ}{T}}}$ 
and let $Y$ be the resulting log scheme. 
Let ${\cal Y}$ be a log smooth scheme defined in (\ref{coro:connfil}) below. 
Let $D_{{\cal Y}/S(T)^{\nat}}(1)$ be the log PD-envelope of the immersion 
${\cal Y}\os{\sus}{\lo} {\cal Y}\times_{S(T)^{\nat}}{\cal Y}$ over $(S(T)^{\nat},{\cal J},\del)$. 
As in \cite[V]{bb} and \cite[\S7]{bob}, we have the following two corollaries 
(cf.~\cite[(2.10.5), (2.10.7)]{nh2}) by 
using (\ref{ali:eeqie}), (\ref{theo:bccange}) and 
a fact that 
$p_i \col \os{\circ}{D}_{{\cal Y}/S(T)}(1)\lo \os{\circ}{\cal Y}$ 
$(i=1,2)$ 
is flat (\cite[(6.5)]{klog1}): 

\begin{coro}\label{coro:connfil}
Let $g\col X_{\os{\circ}{T}_0} \lo Y$ be an SNCL scheme. 
Assume that $Y$ has a log smooth lift ${\cal Y}$ over $S(T)^{\nat}$. 
Let $q$ be an integer.
Let $g\col X_{\os{\circ}{T}_0}\lo {\cal Y}$ 
be the structural morphism. 
Then there exists  a quasi-nilpotent integrable connection 
\begin{align*}
&P_kR^qg_{X_{\os{\circ}{T}_0}/{\cal Y}*}
(\eps^*_{X_{\os{\circ}{T}_0}/S(T)^{\nat}}(E))
\os{\nabla_k}{\lo} \tag{9.3.1}\\
& P_kR^qg_{X_{\os{\circ}{T}_0}/{\cal Y}*}
(\eps^*_{X_{\os{\circ}{T}_0}/S(T)^{\nat}}(E))
{\otimes}_{{\cal O}_{\cal Y}}{\Om}_{{\cal Y}/S(T)^{\nat}}^1
\end{align*}
making the following diagram commutative 
for any two nonnegative integers $k\leq l:$
\begin{equation*}
\begin{CD}
P_kR^qg_{X_{\os{\circ}{T}_0}/{\cal Y}*}
(\eps^*_{X_{\os{\circ}{T}_0}/S(T)^{\nat}}(E))
@>{\nabla_k}>> \\
@V{\bigcap}VV  \\
P_lR^qg_{X_{\os{\circ}{T}_0}/{\cal Y}*}
(\eps^*_{X_{\os{\circ}{T}_0}/S(T)^{\nat}}(E))
@>{\nabla_l}>>
\end{CD}
\tag{9.3.2}
\end{equation*}
\begin{equation*}
\begin{CD}
P_kR^qg_{X_{\os{\circ}{T}_0}/{\cal Y}*}
(\eps^*_{X_{\os{\circ}{T}_0}/S(T)^{\nat}}(E^{\bul \leq N}))
{\otimes}_{{\cal O}_{\cal Y}}{\Om}_{{\cal Y}/S(T)^{\nat}}^1\\
@V{\bigcap}VV \\ 
P_lR^qg_{X_{\os{\circ}{T}_0}/{\cal Y}*}
(\eps^*_{X_{\os{\circ}{T}_0}/S(T)^{\nat}}(E^{\bul \leq N})) 
{\otimes}_{{\cal O}_{\cal Y}}{\Om}_{{\cal Y}/S(T)^{\nat}}^1.
\end{CD} 
\end{equation*}
\end{coro}
\begin{proof} 
This follows from (\ref{theo:bccange}) as in \cite[\S7]{bob}.
\end{proof} 

\begin{coro}\label{coro:fctd}
Let the notations and the assumptions be as in $(\ref{prop:bdccd})$. 
Then 
$$Rf_*
(P_kH_{\rm zar}(X_{\os{\circ}{T}_0}/S(T)^{\nat},E)) \quad (k \in{\mab N})$$
has finite tor-dimension. 
Moreover, if $\os{\circ}{T}$ is noetherian and 
if $\os{\circ}{f}$ is proper,
then $Rf_*(P_kH_{\rm zar}(X_{\os{\circ}{T}_0}/S(T)^{\nat},E))$ 
is a perfect complex of ${\cal O}_T$-modules.
\end{coro}

Using \cite[(2.10.10)]{nh2}, 
we have the following corollary 
(cf.~\cite[(2.10.11)]{nh2}): 

\begin{coro}\label{coro:filpcerf}
Let the notations 
and the assumptions be as in $(\ref{coro:fctd})$.
Then the filtered complex 
$Rf_*((H_{\rm zar}(X_{\os{\circ}{T}_0}/S(T)^{\nat},E),P))$ 
is a {\it filtered perfect}
complex of ${\cal O}_T$-modules, that is, 
locally on $T_{\rm zar}$, filteredly quasi-isomorphic to 
a filtered strictly perfect complex {\rm (\cite[(2.10.8)]{nh2})}.
\end{coro}
\begin{proof}
(\ref{coro:filpcerf}) 
immediately follows 
from (\ref{coro:fctd}) and 
\cite[(2.10.10)]{nh2}.
\end{proof}

\section{Monodromy operators of PD-Hirsch pre-weight-filtered log crystalline complexes}\label{sec:mod}
In this section we first recall the monodromy operators defined 
in \cite{nb}, which is an immediate generalization of 
those in \cite{hyp} and \cite{hk}.
\par 
Let $S$, $((T,{\cal J},\del),z)$ and $T_0$ be as in \S\ref{sec:ldfc}.  
Let $Y$ be a log smooth scheme over $S$. 
Let $Y'_{\os{\circ}{T}_0}$ be the disjoint union of the member of 
an affine open covering of $Y_{\os{\circ}{T}_0}$. 
By abuse of notation, we also denote by $g$ 
the composite morphism 
$g \col Y_{\os{\circ}{T}_0} \lo S_{\os{\circ}{T}_0} \lo S(T)^{\nat}$.  
Let $Y_{\os{\circ}{T}_0\bul}$ be the \v{C}ech diagram of 
$Y'_{\os{\circ}{T}_0}$ over $Y_{\os{\circ}{T}_0}$: 
$Y_{\os{\circ}{T}_0,n}={\rm cosk}_0^{Y_{\os{\circ}{T}_0}}(Y'_{\os{\circ}{T}_0})_n$ $(n\in {\mab N})$. 
Let $g_{\bul}\col  Y_{\os{\circ}{T}_0\bul}\lo S(T)^{\nat}$ 
be the structural morphism. 
For $U=S(T)^{\nat}$ or $\os{\circ}{T}$, 
let 
$$\eps_{Y_{\os{\circ}{T}_0\bul}/U}\col 
((Y_{\os{\circ}{T}_0\bul}/U)_{\rm crys},
{\cal O}_{Y_{\os{\circ}{T}_0\bul}/U})\lo 
((\os{\circ}{Y}_{T_0,\bul}/\os{\circ}{T})_{\rm crys},
{\cal O}_{\os{\circ}{Y}_{T_0,\bul}/\os{\circ}{T}})$$  
be the morphism forgetting the log structures of $Y_{\os{\circ}{T}_0\bul}$ and $U$. 
Let 
$Y_{\os{\circ}{T}_0\bul} \os{\sus}{\lo} \ol{\cal Q}_{\bul}$ 
be an immersion into a log smooth scheme over $\ol{S(T)^{\nat}}$. 
Set ${\cal Q}_{\bul}
:=\ol{\cal Q}_{\bul}\times_{\ol{S(T)^{\nat}}}S(T)^{\nat}$.  
Let $\ol{\mathfrak E}_{\bul}$ 
be the log PD-envelope of 
the immersion $Y_{\os{\circ}{T}_0\bul} \os{\sus}{\lo} \ol{{\cal Q}}_{\bul}$ over 
$(\os{\circ}{T},{\cal J},\del)$. 
Set ${\mathfrak E}_{\bul}
=\ol{\mathfrak E}_{\bul}
\times_{{\mathfrak D}(\ol{S(T)^{\nat}})}S(T)^{\nat}$.  
Let $\theta_{{\cal Q}_{\bul}}\in 
{\Om}^1_{{{\cal Q}}_{\bul}/\os{\circ}{T}}$ 
be the pull-back of $\theta =d\log t\in  
{\Om}^1_{S(T)^{\nat}/\os{\circ}{T}}$ by the structural morphism 
${\cal Q}_{\bul}\lo S(T)^{\nat}$.  
Let $\ol{F}$ be a flat quasi-coherent crystal of 
${\cal O}_{Y_{\os{\circ}{T}_0}/\os{\circ}{T}}$-modules. 
Let $\ol{F}{}^{\bul}$ be the crystal of 
${\cal O}_{Y_{\os{\circ}{T}_0\bul}/\os{\circ}{T}}$-modules obtained by $\ol{F}$. 
Let $(\ol{\cal F}{}^{\bul},\ol{\nabla})$ be the quasi-coherent 
${\cal O}_{\ol{\mathfrak E}_{\bul}}$-module 
with integrable connection corresponding to 
$\ol{F}{}^{\bul}$. 
Set 
$({\cal F}^{\bul},\nabla)=
(\ol{\cal F}{}^{\bul},\ol{\nabla})
\otimes_{{\cal O}_{\ol{\mathfrak E}_{\bul}}}
{\cal O}_{{\mathfrak E}_{\bul}}$.   
Let 
\begin{equation*} 
\nabla \col {\cal F}^{\bul} 
\lo {\cal F}^{\bul}\otimes_{{\cal O}_{{{\cal Q}}_{\bul}}}
{\Om}^1_{{\cal Q}_{\bul}/\os{\circ}{T}} 
\tag{10.0.1}\label{eqn:nidfopltd}
\end{equation*}  
be the induced connection by $\ol{\nabla}$. 
Since ${\Om}^1_{{\cal Q}_{\bul}/S(T)^{\nat}}$ 
is a quotient of 
${\Om}^1_{{\cal Q}_{\bul}/\os{\circ}{T}}$, 
we also have the induced connection 
\begin{equation*} 
\nabla_{/S(T)^{\nat}} 
\col {\cal F}^{\bul} \lo {\cal F}^{\bul}
\otimes_{{\cal O}_{{{\cal Q}}_{\bul}}}
{\Om}^1_{{\cal Q}_{\bul}/S(T)^{\nat}}  
\tag{10.0.2}\label{eqn:nidfqtopltd}
\end{equation*}   
by $\nabla$.  
The object 
$({\cal F}^{\bul},\nabla_{/S(T)^{\nat}})$ 
corresponds to the log crystal 
$F:=\eps_{Y_{\os{\circ}{T}_0\bul}/S(T)^{\nat}}^*
(\ol{F}{}^{\bul})$ of ${\cal O}_{Y_{\os{\circ}{T}_0\bul}/S(T)^{\nat}}$-modules. 

\par
In \cite{nb} we have proved the following whose proof is not difficult: 

\begin{prop}[{\bf \cite[(1.7.22)]{nb}}]\label{prop:mce}
The following sequence 
\begin{align*} 
0 & \lo {\cal F}^{\bul}
\otimes_{{\cal O}_{{{\cal Q}}_{\bul}}}
{\Om}^{\bul}_{{{\cal Q}}_{\bul}/S(T)^{\nat}}[-1] 
\os{\theta_{{\cal Q}_{{\bul}} \wedge}}{\lo} 
{\cal F}^{\bul}
\otimes_{{\cal O}_{{{\cal Q}}_{\bul}}}
{\Om}^{\bul}_{{{\cal Q}}_{\bul}/\os{\circ}{T}} 
\lo {\cal F}^{\bul}
\otimes_{{\cal O}_{{{\cal Q}}_{\bul}}}
{\Om}^{\bul}_{{{\cal Q}}_{\bul}/S(T)^{\nat}} \lo 0
\tag{10.1.1}\label{eqn:gsflxd}\\ 
\end{align*} 
is exact. 
\end{prop}

\parno 
Let 
\begin{equation*} 
{\cal F}^{\bul}
\otimes_{{\cal O}_{{{\cal Q}}_{\bul}}}
{\Om}^{\bul}_{{{\cal Q}}_{\bul}/S(T)^{\nat}} 
\lo 
{\cal F}^{\bul}
\otimes_{{\cal O}_{{{\cal Q}}_{\bul}}}
{\Om}^{\bul}_{{{\cal Q}}_{\bul}/S(T)^{\nat}} 
\tag{10.1.2}\label{eqn:gyglxd}
\end{equation*} 
be the boundary morphism of (\ref{eqn:gsflxd}) 
in the derived category 
$D^+(f^{-1}_T({\cal O}_T))$. 
(We make the convention on the sign of 
the boundary morphism as in \cite[p.~12 (4)]{nh2}.)  
By using the formula 
$Ru_{Y_{\os{\circ}{T}_0\bul}/S(T)^{\nat}*}(F^{\bul})
={\cal F}^{\bul}
\otimes_{{\cal O}_{{{\cal Q}}_{\bul}}}
{\Om}^{\bul}_{{{\cal Q}}_{\bul}/S(T)^{\nat}}$  
(\cite[(6.4)]{klog1}, (cf.~\cite[(2.2.7)]{nh2})), 
we have the following morphism  
\begin{equation*} 
Ru_{Y_{\os{\circ}{T}_0\bul}/S(T)^{\nat}*}(F^{\bul})
\lo 
Ru_{Y_{\os{\circ}{T}_0\bul}/S(T)^{\nat}*}(F^{\bul}).   
\tag{10.1.3}\label{eqn:uyons}
\end{equation*}  
Let 
\begin{equation*} 
\pi_{{\rm crys}} \col 
((Y_{\os{\circ}{T}_0\bul}/S(T)^{\nat})_{\rm crys},{\cal O}_{Y_{\os{\circ}{T}_0\bul}/S(T)^{\nat}})
\lo 
((Y_{\os{\circ}{T}_0}/S(T)^{\nat})_{\rm crys},{\cal O}_{Y_{\os{\circ}{T}_0}/S(T)^{\nat}})
\tag{10.1.4}\label{eqn:tzcar}
\end{equation*} 
and 
\begin{equation*} 
\pi_{{\rm zar}} \col 
((\os{\circ}{Y}_{T_0\bul})_{\rm zar},
g^{-1}_{\bul}({\cal O}_T)) \lo 
((\os{\circ}Y_{T_0})_{\rm zar},g^{-1}({\cal O}_T)) 
\tag{10.1.5}\label{eqn:tzar}
\end{equation*} 
be the natural morphisms of ringed topoi. 
Applying $R\pi_{{\rm zar}*}$ to (\ref{eqn:uyons}) 
and using the formula 
$\pi_{{\rm zar}*}\circ u_{Y_{\os{\circ}{T}_0\bul}/S(T)^{\nat}}
=u_{Y_{\os{\circ}{T}_0}/S(T)^{\nat}}\circ \pi_{{\rm crys}}$ 
and using the cohomological descent,  
we have the following morphism   
\begin{equation*}
N_{\rm zar} \col Ru_{Y_{\os{\circ}{T}_0}/S(T)^{\nat}*}(F) 
\lo 
Ru_{Y_{\os{\circ}{T}_0}/S(T)^{\nat}*}(F).
\tag{10.1.6}\label{eqn:nzgslyne}
\end{equation*}

In \cite{nb} we have proved the following whose proof is not difficult: 

\begin{prop}[{\bf \cite[(1.7.26), (1.7.30)]{nb}}] 
The morphism {\rm (\ref{eqn:nzgslyne})} 
is independent of the choice of the disjoint union 
of the member of an affine open covering of 
$Y$ and an simplicial immersion 
$Y_{\bul} \os{\sus}{\lo} 
\ol{{\cal Q}}_{\bul}$ over $\ol{S(T)^{\nat}}$. 
\end{prop}

\par 
Let 
$v \col (S(T)^{\nat},{\cal J},\del) \lo (S(T)^{\nat},{\cal J},\del)$ be 
an endomorphism of $(S(T)^{\nat},{\cal J},\del)$.  
Let 
\begin{equation*} 
\begin{CD} 
Y_{\os{\circ}{T}_0} @>{g}>> Y_{\os{\circ}{T}_0} \\ 
@VVV @VVV \\ 
S_{\os{\circ}{T}_0} @>{v_0}>> S_{\os{\circ}{T}_0} \\ 
@V{\bigcap}VV @VV{\bigcap}V \\ 
S(T)^{\nat} @>{v}>> S(T)^{\nat}
\end{CD}
\tag{10.2.1}\label{eqn:ydxduss}
\end{equation*} 
be a commutative diagram of log schemes.  
The morphism (\ref{eqn:nzgslyne}) 
is nothing but a morphism 
\begin{equation*}
N_{\rm zar} \col Ru_{Y_{\os{\circ}{T}_0}/S(T)^{\nat}*}(F) 
\lo 
Ru_{Y_{\os{\circ}{T}_0}/S(T)^{\nat}*}(F)(-1;v)
\tag{10.2.2}\label{eqn:mlcepglynl}
\end{equation*}
since $v^*(\theta_{{\cal Q}_{\bul}})=\deg(v)\theta_{{\cal Q}_{\bul}}$.
In \cite{nb} we have called the morphism (\ref{eqn:mlcepglynl}) the 
{\it  zariskian monodromy operator} 
of $Y_{\os{\circ}{T}_0}/S(T)^{\nat}$. 

By (\ref{eqn:exfte}) we have the following diagram 
\begin{equation*} 
\begin{CD}
Ru_{Y_{\os{\circ}{T}_0}/S(T)^{\nat}*}(F) 
@>{N_{\rm zar}}>>
Ru_{Y_{\os{\circ}{T}_0}/S(T)^{\nat}*}(F)\\
@A{\simeq}AA @AA{\simeq}A \\ 
\wt{R}u_{Y_{\os{\circ}{T}_0}/\os{\circ}{T}*}(F\langle u \rangle) 
@. \wt{R}u_{Y_{\os{\circ}{T}_0}/\os{\circ}{T}*}(F\langle u \rangle).  
\end{CD} 
\tag{10.2.3}\label{ali:eyfnt}
\end{equation*}
Consider the morphism 
\begin{equation*} 
d/du: {\cal O}_T\langle u \rangle \owns u^{[i]} \lom u^{[i-1]}
\in {\cal O}_T\langle u \rangle. 
\tag{10.2.4}\label{ali:eyufnt}
\end{equation*}
Because 
\begin{align*} 
d/du=d(au)/du\cdot d/d(au)=ad/d(au)
\end{align*}  
for a local section $a\in {\cal O}_T^*$, 
the morphism (\ref{ali:eyufnt}) is independent of the choice of $u$. 
Hence we have the following morphism 
\begin{equation*}
D:=``d/du{\textrm '}{\textrm '} \col \Gam_{{\cal O}_T}(U_{S(T)^{\nat}})\lo 
\Gam_{{\cal O}_T}(U_{S(T)^{\nat}}). 
\tag{10.2.5}\label{eqn:bufnt}
\end{equation*}
This morphism induces the following morphism 
\begin{equation*} 
D: \wt{R}u_{Y_{\os{\circ}{T}_0}/S(T)^{\nat}*}(F\langle u \rangle) 
\lo \wt{R}u_{Y_{\os{\circ}{T}_0}/S(T)^{\nat}*}(F\langle u \rangle). 
\tag{10.2.6}\label{ali:eufant}
\end{equation*}
We also have the following similar morphism 
\begin{equation*}
D \col \Gam_{{\cal O}_T}(U_{S(T)^{\nat}})^{\wedge}\lo 
\Gam_{{\cal O}_T}(U_{S(T)^{\nat}})^{\wedge}. 
\tag{10.2.7}\label{eqn:bmt}
\end{equation*}

\begin{theo}\label{theo:monf}
Assume that $F$ is a flat locally nilpotent quasi-coherent 
${\cal O}_{Y_{\os{\circ}{T}_0}/\os{\circ}{T}}$-modules.  
Then the following diagram is commutative: 
\begin{equation*} 
\begin{CD}
Ru_{Y_{\os{\circ}{T}_0}/S(T)^{\nat}*}(F) 
@>{N_{\rm zar}}>> 
Ru_{Y_{\os{\circ}{T}_0}/S(T)^{\nat}*}(F)\\
@A{\simeq}AA @AA{\simeq}A \\ 
\wt{R}u_{Y_{\os{\circ}{T}_0}/\os{\circ}{T}*}(F\langle u \rangle) 
@>{-D}>> \wt{R}u_{Y_{\os{\circ}{T}_0}/\os{\circ}{T}*}(F\langle u \rangle).  
\end{CD} 
\tag{10.3.1}\label{cd:efnt}
\end{equation*}
\end{theo}
\begin{proof} 
We may assume that there exists an immersion 
$Y_{\os{\circ}{T}_0}\os{\sus}{\lo} \ol{\cal Q}$ into a log smooth scheme over $\ol{S(T)^{\nat}}$. 
Let the notations be as in the beginning of \S\ref{sec:ldfc}. 
Set ${\cal Q}:=\ol{\cal Q}\times_{\ol{S(T)^{\nat}}}S(T)^{\nat}$.
Because the following diagram 
\begin{equation*} 
\begin{CD}
{\cal F}\otimes_{{\cal O}_{{\cal Q}^{{\rm ex}}}}
\Om^{\bul}_{{\cal Q}{}^{{\rm ex}}/\os{\circ}{T}}\langle u\rangle 
@>{-D}>> 
{\cal F}\otimes_{{\cal O}_{{\cal Q}^{{\rm ex}}}}
\Om^{\bul}_{{\cal Q}{}^{{\rm ex}}/\os{\circ}{T}}\langle u\rangle \\
@V{\simeq}VV @VV{\simeq}V \\ 
{\cal F}\otimes_{{\cal O}_{{\cal Q}^{{\rm ex}}}}
\Om^{\bul}_{{\cal Q}{}^{{\rm ex}}/\os{\circ}{T}}\langle \langle u\rangle \rangle 
@>{-D}>> {\cal F}\otimes_{{\cal O}_{{\cal Q}^{{\rm ex}}}}
\Om^{\bul}_{{\cal Q}{}^{{\rm ex}}/\os{\circ}{T}}\langle \langle u\rangle \rangle.  
\end{CD} 
\tag{10.3.2}\label{ali:fnct}
\end{equation*}
is commutative, 
we have only to prove that the following diagram is commutative: 
\begin{equation*} 
\begin{CD}
{\cal F}\otimes_{{\cal O}_{{\cal Q}^{{\rm ex}}}}
\Om^{\bul}_{{\cal Q}{}^{{\rm ex}}/S(T)^{\nat}} 
@>{N_{\rm zar}}>>  
{\cal F}\otimes_{{\cal O}_{{\cal Q}^{{\rm ex}}}}
\Om^{\bul}_{{\cal Q}{}^{{\rm ex}}/S(T)^{\nat}} \\
@A{\simeq}AA @AA{\simeq}A \\ 
{\cal F}\otimes_{{\cal O}_{{\cal Q}^{{\rm ex}}}}
\Om^{\bul}_{{\cal Q}{}^{{\rm ex}}/\os{\circ}{T}}\langle \langle u\rangle \rangle 
@>{-D}>> {\cal F}\otimes_{{\cal O}_{{\cal Q}^{{\rm ex}}}}
\Om^{\bul}_{{\cal Q}{}^{{\rm ex}}/\os{\circ}{T}}\langle \langle u\rangle \rangle.  
\end{CD} 
\tag{10.3.3}\label{ali:fnt}
\end{equation*}
(Here note that to take the completion is important in the following argument.) 
The problem is local. 
To prove this commutativity, take a local section 
$\om_0$ in 
$${\rm Ker}(\nabla_{/S(T)^{\nat}} \col 
{\cal F}\otimes_{{\cal O}_{{\cal Q}^{{\rm ex}}}}
\Om^i_{{\cal Q}{}^{{\rm ex}}/S(T)^{\nat}}
\lo 
{\cal F}\otimes_{{\cal O}_{{\cal Q}^{{\rm ex}}}}
\Om^{i+1}_{{\cal Q}{}^{{\rm ex}}/S(T)^{\nat}})
\quad  (i\in {\mab Z}_{\geq 0}).$$  
Let $\wt{\om}_0$ be a lift of $\om_0$ in 
${\cal F}\otimes_{{\cal O}_{{\cal Q}^{{\rm ex}}}}
\Om^i_{{\cal Q}{}^{{\rm ex}}/\os{\circ}{T}}$. 
Then $N_{\rm zar}(\om_0)=\eta$, where 
$\eta$ is a local section of 
${\cal F}\otimes_{{\cal O}_{{\cal Q}^{{\rm ex}}}}
\Om^i_{{\cal Q}{}^{{\rm ex}}/S(T)^{\nat}}$ 
such that 
\begin{equation*} 
d\log t\wedge \eta=\nabla(\wt{\om}_0).
\tag{10.3.4}\label{ali:frt}
\end{equation*}  
We would like to construct a local section 
$\sum_{j=0}^{\infty}u^{[j]}\wt{\om}_j$ of 
$${\rm Ker}
({\cal O}_T\langle \langle u \rangle \rangle \otimes_{{\cal O}_T}
{\cal F}\otimes_{{\cal O}_{{\cal Q}^{{\rm ex}}}}
\Om^i_{{\cal Q}{}^{{\rm ex}}/\os{\circ}{T}}
\lo 
{\cal O}_T\langle \langle u \rangle \rangle \otimes_{{\cal O}_T}
{\cal F}\otimes_{{\cal O}_{{\cal Q}^{{\rm ex}}}}
\Om^{i+1}_{{\cal Q}{}^{{\rm ex}}/\os{\circ}{T}}).$$ 
Consider the following equation 
\begin{align*} 
0=\nabla(\sum_{j\geq 0}u^{[j]}\wt{\om}_j)=\nabla(\wt{\om}_0)+
\sum_{j\geq 1}(d\log t\wedge u^{[j-1]}\wt{\om}_j+u^{[j]}\nabla(\wt{\om}_j)).
\end{align*}  
This is equivalent to the following equation
\begin{equation*} 
d\log t\wedge \wt{\om}_{j+1}+\nabla(\wt{\om}_j)=0 \quad (j\geq 0). 
\tag{10.3.5}\label{ali:fjwt}
\end{equation*}  
We would like to find $\wt{\om}_j$ $(j\geq 1)$ by using induction on $j$. 
Indeed, assume that $\wt{\om}_i$ $(i\leq j)$ exists. 
Then $\nabla(d\log t\wedge \wt{\om}_{j}+\nabla(\wt{\om}_{j-1}))=0$. 
Hence $d\log t\wedge \nabla(\wt{\om}_{j})=0$ since $\nabla^2=0$. 
This means the existence of $\wt{\om}_{j+1}$ in (\ref{ali:fjwt}) by (\ref{ali:gultpt}). 
Let $\om_1$ be the image of $\wt{\om}_1$ in  
${\cal F}\otimes_{{\cal O}_{{\cal Q}^{{\rm ex}}}}
\Om^i_{{\cal Q}{}^{{\rm ex}}/S(T)^{\nat}}$.  
Then, by the relations (\ref{ali:frt}) and (\ref{ali:fjwt}), 
we have $\om_1=-\eta$ because 
$d\log t \wedge \col {\cal F}\otimes_{{\cal O}_{{\cal Q}^{{\rm ex}}}}
\Om^{i-1}_{{\cal Q}{}^{{\rm ex}}/S(T)^{\nat}}
\lo {\cal F}\otimes_{{\cal O}_{{\cal Q}^{{\rm ex}}}}
\Om^{i}_{{\cal Q}{}^{\rm ex}/\os{\circ}{T}}$ 
is injective. 
On the other hand, 
$(-d/du)(\sum_{j\geq 0}u^{[j]}\wt{\om}_j)=
-\sum_{j\geq 1}u^{[j-1]}\wt{\om}_j$ whose image in 
${\cal F}\otimes_{{\cal O}_{{\cal Q}^{{\rm ex}}}}
\Om^i_{{\cal Q}{}^{{\rm ex}}/S(T)^{\nat}}$ is $-\om_1=\eta$. 
This tells us that the diagram (\ref{ali:fnt}) is commutative. 
\end{proof} 

\begin{rema}
In the proof of \cite[Lemma 7]{kiha} the existence of the cocycle 
$a_0+a_1u^{[1]}+\cdots \in R\Gam(W\wt{\om})[u]$ is unclear. 
\end{rema}

Let $X/S$ be an SNCL scheme. 
Let $E$ be flat quasi-coherent crystals of 
${\cal O}_{\os{\circ}{X}_{T_0}/\os{\circ}{T}}$-modules. 
Then $\eps^*_{X_{\os{\circ}{T}_0}/\os{\circ}{T}}(E)$ is a flat locally nilpotent quasi-coherent 
${\cal O}_{X_{\os{\circ}{T}_0}/\os{\circ}{T}}$-modules.

\begin{coro}\label{coro:ctt}
The following diagram is commutative: 
\begin{equation*} 
\begin{CD}
Ru_{X_{\os{\circ}{T}_0}/S(T)^{\nat}*}
(\eps^*_{X_{\os{\circ}{T}_0}/S(T)^{\nat}}(E))  
@>{N_{\rm zar}}>> 
Ru_{X_{\os{\circ}{T}_0}/S(T)^{\nat}*}
(\eps^*_{X_{\os{\circ}{T}_0}/S(T)^{\nat}}(E))   \\
@A{\simeq}AA @AA{\simeq}A \\ 
\wt{R}u_{X_{\os{\circ}{T}_0}/\os{\circ}{T}*}
(\eps^*_{X_{\os{\circ}{T}_0}/\os{\circ}{T}}(E)\langle u \rangle) 
@>{-D}>> \wt{R}u_{X_{\os{\circ}{T}_0}/\os{\circ}{T}*}(
\eps^*_{X_{\os{\circ}{T}_0}/\os{\circ}{T}}(E)\langle u \rangle)\\
@V{\simeq}VV @VV{\simeq}V \\ 
\wt{R}u_{X^{(\star)}_{\os{\circ}{T}_0}/\os{\circ}{T}*}
(\eps^*_{X^{(\star)}_{\os{\circ}{T}_0}/\os{\circ}{T}}(E^{(\star)})\langle u \rangle) 
@>{-D}>> \wt{R}u_{X^{(\star)}_{\os{\circ}{T}_0}/\os{\circ}{T}*}(
\eps^*_{X^{(\star)}_{\os{\circ}{T}_0}/\os{\circ}{T}}(E^{\star})\langle u \rangle).  
\end{CD} 
\tag{10.5.1}\label{cd:efnbtt}
\end{equation*}
\end{coro}
\begin{proof} 
The commutativity of the upper square is a special case of (\ref{cd:efnt}). 
The commutativity of the lower square is obvious. 
\end{proof}

\begin{defi} 
We call the morphisms 
\begin{align*} 
-D \col 
\wt{R}u_{Y_{\os{\circ}{T}_0}/\os{\circ}{T}*}(F\langle u \rangle) 
\lo \wt{R}u_{Y_{\os{\circ}{T}_0}/\os{\circ}{T}*}(F\langle u \rangle)
\end{align*}   
and 
\begin{align*} 
-D\col \wt{R}u_{X^{(\star)}_{\os{\circ}{T}_0}/\os{\circ}{T}*}
(\eps^*_{X^{(\star)}_{\os{\circ}{T}_0}/\os{\circ}{T}}(E^{(\star)})\langle u \rangle) 
\lo \wt{R}u_{X^{(\star)}_{\os{\circ}{T}_0}/\os{\circ}{T}*}(
\eps^*_{X^{(\star)}_{\os{\circ}{T}_0}/\os{\circ}{T}}(E^{\star})\langle u \rangle)
\end{align*} 
the {\it monodromy operators} of   
$\wt{R}u_{Y_{\os{\circ}{T}_0}/\os{\circ}{T}*}(F\langle u \rangle)$
and 
$\wt{R}u_{X^{(\star)}_{\os{\circ}{T}_0}/\os{\circ}{T}*}
(\eps^*_{X^{(\star)}_{\os{\circ}{T}_0}/\os{\circ}{T}}(E^{(\star)})\langle u \rangle)$,
respectively and we denote them by $N_{\rm zar}$.
\end{defi}

\begin{prop}\label{prop:eun}
The monodromy operators 
$$N_{\rm zar}\col 
\wt{R}u_{Y_{\os{\circ}{T}_0}/\os{\circ}{T}*}
(F\langle u \rangle) 
\lo \wt{R}u_{Y_{\os{\circ}{T}_0}/\os{\circ}{T}*}(F\langle u \rangle)$$ 
and 
$$N_{\rm zar}\col 
\wt{R}u_{X^{(\star)}_{\os{\circ}{T}_0}/\os{\circ}{T}*}
(\eps^*_{X^{(\star)}_{\os{\circ}{T}_0}/\os{\circ}{T}}(E^{(\star)})\langle u \rangle)
\lo 
\wt{R}u_{X^{(\star)}_{\os{\circ}{T}_0}/\os{\circ}{T}*}
(\eps^*_{X^{(\star)}_{\os{\circ}{T}_0}/\os{\circ}{T}}(E^{(\star)})\langle u \rangle)$$  
induce the following morphisms, respectively$:$
\begin{align*} 
N_{\rm zar}\col 
P_k\wt{R}u_{Y_{\os{\circ}{T}_0}/\os{\circ}{T}*}(F\langle u \rangle) 
\lo P_{k-2}\wt{R}u_{Y_{\os{\circ}{T}_0}/\os{\circ}{T}*}(F\langle u \rangle), 
\tag{10.7.1}\label{ali:npr}
\end{align*} 
\begin{align*} 
N_{\rm zar}\col 
P_k\wt{R}u_{X^{(\star)}_{\os{\circ}{T}_0}/\os{\circ}{T}*}
(\eps^*_{X^{(\star)}_{\os{\circ}{T}_0}/\os{\circ}{T}}(E^{(\star)})\langle u \rangle)
\lo P_{k-2}\wt{R}u_{X^{(\star)}_{\os{\circ}{T}_0}/\os{\circ}{T}*}
(\eps^*_{X^{(\star)}_{\os{\circ}{T}_0}/\os{\circ}{T}}(E^{(\star)})\langle u \rangle).  
\tag{10.7.2}\label{ali:nbpr}
\end{align*}  
\end{prop} 
\begin{proof} 
These are obvious by the definitions of $N_{\rm zar}$ and the local definition of 
$P$ in (\ref{eqn:dfhp}). 
\end{proof}

\section{Cup products of PD-Hirsch pre-weight-filtered 
log crystalline complexes}\label{sec:p}
In this section we define the cup product of  
$(H_{\rm zar}(X_{\os{\circ}{T}_0}/S(T)^{\nat},E),P)$ 
and give several fundamental properties of it. 
\par 
Let the notations be as in the previous section. 
First let us recall the cup product of log crystalline complexes.
\par 
Let $\ol{F}{}'$ be another flat quasi-coherent crystal of 
${\cal O}_{Y_{\os{\circ}{T}_0}/\os{\circ}{T}}$-modules.  
Let 
${\cal F}^{\bul}{}'$ be the analogous sheaf to 
${\cal F}^{\bul}$ for $\ol{F}{}'$.
The wedge product 
$$\wedge \col 
\Om^{\bul}_{{\cal Q}^{\rm ex}_{\bul}/S(T)^{\nat}}
\otimes_{{\cal O}_T}\Om^{\bul}_{{\cal Q}^{\rm ex}_{\bul}/S(T)^{\nat}}
\lo \Om^{\bul}_{{\cal Q}^{\rm ex}_{\bul}/S(T)^{\nat}}$$  
induces the following morphism 
\begin{equation*} 
({\cal F}^{\bul}
\otimes_{{\cal O}_{{\cal Q}^{\rm ex}_{\bul}}}
\Om^{\bul}_{{\cal Q}^{\rm ex}_{\bul}/S(T)^{\nat}})
\otimes_{{\cal O}_T}({\cal F}'{}^{\bul}
\otimes_{{\cal O}_{{\cal Q}^{\rm ex}_{\bul}}}
\Om^{\bul}_{{\cal Q}^{\rm ex}_{\bul}/S(T)^{\nat}}) \lo 
({\cal F}^{\bul}\otimes_{{\cal O}_{{\mathfrak E}_{\bul}}}{\cal F}'{}^{\bul})
\otimes_{{\cal O}_{{\cal Q}^{\rm ex}_{\bul}}}
\Om^{\bul}_{{\cal Q}^{\rm ex}_{\bul}/S(T)^{\nat}}. 
\tag{11.0.1}\label{ali:tuwte} 
\end{equation*}
Assume that $Ru_{Y_{\os{\circ}{T}_0}/S(T)^{\nat}*}
(\eps^*_{Y_{\os{\circ}{T}_0}/S(T)^{\nat}/\os{\circ}{T}}(\ol{F}))$ is bounded above. 
Then this morphism induces the following cup product  
\begin{align*} 
\cup  \col &Ru_{Y_{\os{\circ}{T}_0}/S(T)^{\nat}*}
(\eps^*_{Y_{\os{\circ}{T}_0}/S(T)^{\nat}/\os{\circ}{T}}(\ol{F}))
\otimes^L_{{\cal O}_T}Ru_{Y_{\os{\circ}{T}_0}/S(T)^{\nat}*}
(\eps^*_{Y_{\os{\circ}{T}_0}/S(T)^{\nat}/\os{\circ}{T}}(\ol{F}{}'))
\tag{11.0.2}\label{ali:tte} \\
&\lo Ru_{Y_{\os{\circ}{T}_0}/S(T)^{\nat}*}(
(\eps^*_{Y_{\os{\circ}{T}_0}/S(T)^{\nat}/\os{\circ}{T}}
(\ol{F}\otimes_{{\cal O}_{Y_{\os{\circ}{T}_0}/\os{\circ}{T}}}\ol{F}{}')).
\end{align*} 
Similarly the wedge product 
$\Om^{\bul}_{{\cal Q}^{\rm ex}_{\bul}/\os{\circ}{T}}
\otimes_{{\cal O}_T}\Om^{\bul}_{{\cal Q}^{\rm ex}_{\bul}/\os{\circ}{T}}
\lo \Om^{\bul}_{{\cal Q}^{\rm ex}_{\bul}/\os{\circ}{T}}$ 
induces the following morphism 
\begin{equation*} 
({\cal F}^{\bul}\otimes_{{\cal O}_{{\cal Q}^{\rm ex}_{\bul}}}
\Om^{\bul}_{{\cal Q}^{\rm ex}_{\bul}/\os{\circ}{T}})
\otimes_{{\cal O}_T}({\cal F}'{}^{\bul}
\otimes_{{\cal O}_{{\cal Q}^{\rm ex}_{\bul}}}
\Om^{\bul}_{{\cal Q}^{\rm ex}_{\bul}/\os{\circ}{T}}) \lo 
({\cal F}^{\bul}\otimes_{{\cal O}_{{\mathfrak E}_{\bul}}}{\cal F}'{}^{\bul})
\otimes_{{\cal O}_{{\cal Q}^{\rm ex}_{\bul}}}
\Om^{\bul}_{{\cal Q}^{\rm ex}_{\bul}/\os{\circ}{T}}
\tag{11.0.3}\label{eqn:ttoe} 
\end{equation*}
Assume that $\wt{R}u_{Y_{\os{\circ}{T}_0}/\os{\circ}{T}*}(\ol{F})$ is bounded above. 
Then this morphism induces the following cup product  
\begin{align*} 
\cup  \col &
\wt{R}u_{Y_{\os{\circ}{T}_0}/\os{\circ}{T}*}(\ol{F})
\otimes^L_{{\cal O}_T}\wt{R}u_{Y_{\os{\circ}{T}_0}/\os{\circ}{T}*}(\ol{F}{}')
\lo \wt{R}u_{Y_{\os{\circ}{T}_0}/\os{\circ}{T}*}(
\ol{F}\otimes_{{\cal O}_{Y_{\os{\circ}{T}_0}/\os{\circ}{T}}}\ol{F}{}').
\tag{11.0.4}\label{ali:ttoe} 
\end{align*} 
By considering the morphism (\ref{eqn:dups}), 
we have the following morphism 
\begin{align*} 
\Gam_{{\cal O}_T}(U_{S(T)^{\nat}}) 
\otimes_{{\cal O}_T}\Gam_{{\cal O}_T}(U_{S(T)^{\nat}})  
\lo 
\Gam_{{\cal O}_T}(U_{S(T)^{\nat}}). 
\tag{11.0.5}\label{eqn:dupus}
\end{align*} 
The morphisms (\ref{eqn:ttoe}) and (\ref{eqn:dups}) 
induce the following morphism 
\begin{align*} 
\cup \col 
\wt{R}u_{Y_{\os{\circ}{T}_0}/\os{\circ}{T}*}(\ol{F}\langle u \rangle)
\otimes^L_{{\cal O}_T}
\wt{R}u_{Y_{\os{\circ}{T}_0}/\os{\circ}{T}*}(\ol{F}{}'\langle u \rangle)
\lo \wt{R}u_{Y_{\os{\circ}{T}_0}/\os{\circ}{T}*}
((\ol{F}\otimes_{{\cal O}_{Y_{\os{\circ}{T}_0}/\os{\circ}{T}}}\ol{F}{}')
\langle u \rangle).
\tag{11.0.6}\label{ali:ttube} 
\end{align*} 
Here note that 
$\wt{R}u_{Y_{\os{\circ}{T}_0}/\os{\circ}{T}*}(\ol{F}\langle u \rangle)$ 
is bounded above by (\ref{cd:hirfd}) and hence $\otimes^L_{{\cal O}_T}$ in 
(\ref{ali:ttube}) is well-defined.  

\begin{prop}\label{prop:cmc} 
\begin{align*} 
&N_{\rm zar}\circ (~\cup ~)=
N_{\rm zar}\cup {\rm id}+{\rm id}\cup N_{\rm zar} 
\tag{11.1.1}\label{cd:efntt}\\
\end{align*} 
on 
$\wt{R}u_{Y_{\os{\circ}{T}_0}/\os{\circ}{T}*}(\ol{F}\langle u \rangle) 
\otimes^L_{g^{-1}({\cal O}_T)}
\wt{R}u_{Y_{\os{\circ}{T}_0}/\os{\circ}{T}*}(\ol{F}{}'\langle u \rangle)$. 
\end{prop}
\begin{proof} 
The problem is local. 
By (\ref{theo:monf}) it suffices to prove that 
\begin{equation*} 
-d/du(\om\wedge \om')=(-d/du(\om))\wedge \om'+\om \wedge (-d/du(\om')) 
\tag{11.1.2}\label{eqn:dud}
\end{equation*} 
for local sections 
$\om \in s({\cal O}_T\langle u \rangle
\otimes_{{\cal O}_T}{\cal F}^{\bul}
\otimes_{{\cal O}_{{\cal Q}^{\rm ex}_{\bul}}}
\Om^{\bul}_{{\cal Q}^{\rm ex}_{\bul}/\os{\circ}{T}})$ 
and 
$\om' \in s({\cal O}_T\langle u \rangle
\otimes_{{\cal O}_T}{\cal F}'{}^{\bul}
\otimes_{{\cal O}_{{\cal Q}^{\rm ex}_{\bul}}}
\Om^{\bul}_{{\cal Q}^{\rm ex}_{\bul}/\os{\circ}{T}})$. 
This immediately follows from the following elementary formula 
$$\dfrac{(i+i'-1)!}{(i-1)!i'!}+\dfrac{(i+i'-1)!}{i!(i-1)'!}
=\dfrac{(i+i')!}{i!i'!}\quad (i,i'\in {\mab Z}_{\geq 1}).$$ 
\end{proof}

\begin{prop}\label{prop:cmpp}
Assume that $\wt{R}u_{Y_{\os{\circ}{T}_0}/\os{\circ}{T}*}(\ol{F})$ and 
$\wt{R}u_{Y_{\os{\circ}{T}_0}/\os{\circ}{T}*}(\ol{F}{}')$ are bounded above. 
Let $\ol{F}{}''$ be an analogous sheaf to $\ol{F}{}'$. 
Then 
\begin{align*} 
(?\cup ?)\cup ?=?\cup (?\cup ?) 
\tag{11.2.1}\label{eqn:dutmbd}
\end{align*} 
on 
$\wt{R}u_{Y_{\os{\circ}{T}_0}/\os{\circ}{T}*}(\ol{F}\langle u \rangle) 
\otimes^L_{g^{-1}({\cal O}_T)}
\wt{R}u_{Y_{\os{\circ}{T}_0}/\os{\circ}{T}*}(\ol{F}{}'\langle u \rangle)
\otimes^L_{g^{-1}({\cal O}_T)}
\wt{R}u_{Y_{\os{\circ}{T}_0}/\os{\circ}{T}*}(\ol{F}{}''\langle u \rangle)$. 
\end{prop}
\begin{proof} 
The problem is local. Let ${\cal F}''{}^{\bul}$ be an analogous sheaf to ${\cal F}'^{\bul}$. 
It suffices to prove that 
\begin{equation*} 
(\om\wedge \om')\wedge \om''=\om\wedge (\om'\wedge \om'')
\tag{11.2.2}\label{eqn:dutd}
\end{equation*} 
for $\om \in {\cal O}_T\langle u \rangle
\otimes_{{\cal O}_T}{\cal F}^{\bul}
\otimes_{{\cal O}_{{\cal Q}{}^{\rm ex}_{\bul}}}
\Om^{\bul}_{{\cal Q}{}^{\rm ex}_{\bul}/\os{\circ}{T}}$, 
$\om' \in {\cal O}_T\langle u \rangle
\otimes_{{\cal O}_T}{\cal F}'{}^{\bul}
\otimes_{{\cal O}_{{\cal Q}{}^{\rm ex}_{\bul}}}
\Om^{\bul}_{{\cal Q}{}^{\rm ex}_{\bul}/\os{\circ}{T}}$
and 
$\om'' \in {\cal O}_T\langle u \rangle
\otimes_{{\cal O}_T}{\cal F}''{}^{\bul}
\otimes_{{\cal O}_{{\cal Q}{}^{\rm ex}_{\bul}}}
\Om^{\bul}_{{\cal Q}{}^{\rm ex}_{\bul}/\os{\circ}{T}}$. 
This is obvious. 
\end{proof}

\par 
Let us define the cup product of 
$(H_{\rm zar}(X_{\os{\circ}{T}_0}/S(T)^{\nat},E),P)$ as in 
the cup product on the cohomology of singular cochains in topology: 
the cup product using the Alexander-Whitney map.  
Fujisawa's explanation \cite[(2.9), (2.10), (2.11)]{fup} is convenient for us. 
\par 
Let $E'$ be an analogous crystal of 
${\cal O}_{\os{\circ}{X}/\os{\circ}{S}}$-modules to $E$.  
First note that $(H_{\rm zar}(X_{\os{\circ}{T}_0}/S(T)^{\nat},E),P)$ is bounded above 
since $\os{\circ}{X}_{T_0}$ is quasi-compact ((\ref{prop:baf})). 
Next take an element $\ul{\lam}=\{\lam_0,\ldots,\lam_{l+l'}\}\in P(\Lam)$ 
$(\lam_0 < \cdots <\lam_{l+l'})$.   
Then we have exact closed immersions 
$\iota: {\cal P}_{(\lam_0,\ldots, \lam_{l+l'})} 
\os{\sus}{\lo} {\cal P}_{(\lam_0,\ldots, \lam_l)}$ 
and 
$\iota' \col {\cal P}_{(\lam_0,\ldots, \lam_{l+l'})} 
\os{\sus}{\lo} {\cal P}_{(\lam_l,\ldots, \lam_{l+l'})}$.
\par 
The morphism 
\begin{equation*} 
\Om^i_{{\cal P}^{\rm ex}_{(\lam_0,\ldots, \lam_l)\bul}/\os{\circ}{T}}
\otimes_{{\cal O}_T}
\Om^{i'}_{{\cal P}^{\rm ex}_{(\lam_l,\ldots, \lam_{l+l'})\bul}/\os{\circ}{T}}
\owns \om\otimes \om'\lom
(-1)^{il'}\iota^*(\om)\wedge \iota'{}^*(\om')\in 
\Om^{i+i'}_{{\cal P}^{\rm ex}_{(\lam_0,\ldots, \lam_{l+l'})\bul}/\os{\circ}{T}}
\tag{11.2.3}\label{eqn:dpcpps}
\end{equation*} 
induces the following natural morphism 
\begin{equation*} 
({\cal E}^{\bul}\otimes_{{\cal O}_{{\cal P}^{\rm ex}_{\bul}}}
\Om^i_{{\cal P}^{{\rm ex},{(l)}}_{\bul}/S(T)^{\nat}})
\otimes_{{\cal O}_T}
({\cal E}'{}^{\bul}
\otimes_{{\cal O}_{{\cal P}^{\rm ex}_{\bul}}}
\Om^{i'}_{{\cal P}^{{\rm ex},{(l')}}_{\bul}/S(T)^{\nat}}) 
\lo 
({\cal E}^{\bul}\otimes_{{\cal O}_{{\mathfrak D}_{\bul}}}{\cal E}{}'^{\bul})
\otimes_{{\cal O}_{{\cal P}^{\rm ex}_{\bul}}}
\Om^{i+i'}_{{\cal P}_{\bul}^{{\rm ex},{(l+l')}}/S(T)^{\nat}}. 
\tag{11.2.4}\label{eqn:dcps}
\end{equation*} 
Here the sign $(-1)^{il'}$ arises naturally from our convention on the boundary morphisms  
in (\ref{cd:pppoex}) in order that 
the morphism (\ref{eqn:dps}) below is a morphism of complexes. 
(We leave the reader to check that 
the morphism (\ref{eqn:dps}) is indeed a morphism of complexes, 
which is slightly tedious.)
This induces the following morphism by \cite[(2.10)]{fup}
\begin{equation*} 
s({\cal E}^{\bul}
\otimes_{{\cal O}_{{\cal P}_{\bul}^{\rm ex}}}
\Om^{\bul}_{{\cal P}_{\bul}^{{\rm ex},{(\bul)}}/S(T)^{\nat}})
\otimes_{{\cal O}_T}
s({\cal E}'{}^{\bul}
\otimes_{{\cal O}_{{\cal P}_{\bul}^{\rm ex}}}
\Om^{\bul}_{{\cal P}_{\bul}^{{\rm ex},{(\bul)}}/S(T)^{\nat}})
\lo 
s({\cal E}^{\bul}\otimes_{{\cal O}_{{\mathfrak D}_{\bul}}}{\cal E}{}'^{\bul}
\otimes_{{\cal O}_{{\cal P}^{\rm ex}_{\bul}}}
\Om^{\bul}_{{\cal P}_{\bul}^{{\rm ex},{(\bul)}}/S(T)^{\nat}}). 
\tag{11.2.5}\label{eqn:dps}
\end{equation*} 
The morphisms (\ref{eqn:dps}) and (\ref{eqn:dupus}) induce the following morphism 
\begin{align*} 
&s({\cal E}^{\bul}
\otimes_{{\cal O}_{{\cal P}^{{\rm ex},(\bul)}_{\bul}}}
\Om^{\bul}_{{\cal P}^{{\rm ex},(\bul)}/\os{\circ}{T}}\langle u \rangle)
\otimes 
s({\cal E}{}'^{\bul}
\otimes_{{\cal O}_{{\cal P}^{{\rm ex},(\bul)}_{\bul}}}
\Om^{\bul}_{{\cal P}^{{\rm ex},(\bul)}_{\bul}/\os{\circ}{T}}\langle u\rangle) 
\tag{11.2.6}\label{ali:peb}\\
&\lo s({\cal E}^{\bul}\otimes_{{\cal O}_{{\mathfrak D}_{\bul}}}{\cal E}{}'^{\bul}
\otimes_{{\cal O}_{{\cal P}^{{\rm ex},(\bul)}_{\bul}}}
\Om^{\bul}_{{\cal P}^{{\rm ex},(\bul)}/\os{\circ}{T}}\langle u \rangle). 
\end{align*} 
By the definitions (\ref{eqn:dfhp}) and (\ref{defi:mrt}) (5) (cf.~\cite[(2.12)]{fup}), 
this morphism induces the following filtered morphism 
\begin{align*} 
&(s({\cal E}^{\bul}
\otimes_{{\cal O}_{{\cal P}^{{\rm ex},(\bul)}_{\bul}}}
\Om^{\bul}_{{\cal P}^{{\rm ex},(\bul)}/\os{\circ}{T}}\langle u \rangle),P)
\otimes 
(s({\cal E}{}'^{\bul}
\otimes_{{\cal O}_{{\cal P}^{{\rm ex},(\bul)}_{\bul}}}
\Om^{\bul}_{{\cal P}^{{\rm ex},(\bul)}_{\bul}/\os{\circ}{T}}\langle u \rangle),P)
\tag{11.2.7}\label{ali:peab}\\
&\lo (s({\cal E}^{\bul}\otimes_{{\cal O}_{{\mathfrak D}_{\bul}}}{\cal E}{}'^{\bul}
\otimes_{{\cal O}_{{\cal P}^{{\rm ex},(\bul)}_{\bul}}}
\Om^{\bul}_{{\cal P}^{{\rm ex},(\bul)}/\os{\circ}{T}}\langle u \rangle),P). 
\end{align*} 
The morphism (\ref{ali:peab}) induces the following morphism 
\begin{align*} 
& (H_{\rm zar}(X_{\os{\circ}{T}_0}/S(T)^{\nat},E),P)
\otimes^L_{f^{-1}({\cal O}_T)}
(H_{\rm zar}(X_{\os{\circ}{T}_0}/S(T)^{\nat},E'),P)
& \tag{11.2.8}\label{ali:te} \\
& \lo (H_{\rm zar}(X_{\os{\circ}{T}_0}/S(T)^{\nat},
E\otimes_{{\cal O}_{\os{\circ}{X}_{T_0}/T}}E'),P).
\end{align*} 

\begin{defi}\label{defi:chp}
We call the morphism (\ref{ali:te}) 
the {\it cup product} of $(H_{\rm zar},P)$ 
and we denote it by $~\cup ~$. 
\end{defi} 

\begin{prop}\label{prop:npp}
The following diagram is commutative$:$   
\begin{equation*} 
\begin{CD} 
Ru_{X_{\os{\circ}{T}_0}/S(T)^{\nat}*}
(\eps_{X_{\os{\circ}{T}_0}/S(T)^{\nat}}^*(E))
\otimes^L_{f^{-1}({\cal O}_T)}
Ru_{X_{\os{\circ}{T}_0}/S(T)^{\nat}*}
(\eps_{X_{\os{\circ}{T}_0}/S(T)^{\nat}}^*(E'))
@>{\cup }>>  \\
@A{\simeq}AA  \\
\wt{R}u_{X_{\os{\circ}{T}_0}/\os{\circ}{T}*}(
\eps_{X_{\os{\circ}{T}_0}/\os{\circ}{T}}^*(E)  \langle u \rangle)
\otimes^L_{f^{-1}({\cal O}_T)}
\wt{R}u_{X_{\os{\circ}{T}_0}/\os{\circ}{T}*}(
\eps_{X_{\os{\circ}{T}_0}/\os{\circ}{T}}^*(E')\langle u \rangle)
@>{\cup }>>  \\
@V{\simeq}VV \\
H_{\rm zar}(X_{\os{\circ}{T}_0}/S(T)^{\nat},E)
\otimes^L_{f^{-1}({\cal O}_T)}
H_{\rm zar}(X_{\os{\circ}{T}_0}/S(T)^{\nat},E')
@>{\cup }>> 
\end{CD}
\tag{11.4.1}\label{cd:xtt}
\end{equation*}
\begin{equation*} 
\begin{CD} 
Ru_{X_{\os{\circ}{T}_0}/S(T)^{\nat}*}
(\eps_{X_{\os{\circ}{T}_0}/S(T)^{\nat}}^*
(E\otimes_{{\cal O}_{\os{\circ}{X}_{T_0}/\os{\circ}{T}}}E'))\\
@AA{\simeq}A \\
\wt{R}u_{X_{\os{\circ}{T}_0}/\os{\circ}{T}*}(
\eps_{X_{\os{\circ}{T}_0}/\os{\circ}{T}}^*
(E\otimes_{{\cal O}_{\os{\circ}{X}_{T_0}/\os{\circ}{T}}}E')\langle u \rangle)\\
@VV{\simeq}V \\
H_{\rm zar}(X_{\os{\circ}{T}_0}/S(T)^{\nat},
E\otimes_{{\cal O}_{\os{\circ}{X}_{T_0}/\os{\circ}{T}}}E').
\end{CD}
\end{equation*}
\end{prop}
\begin{proof} 
This is obvious by the constructions of the morphisms (\ref{eqn:efte}) and 
(\ref{ali:te}). 
\end{proof}

\begin{prop}\label{prop:cmstc}
Assume that $\os{\circ}{X}$ is quasi-compact. 
Then 
\begin{align*} 
&N_{\rm zar}\circ (~\cup ~)=N_{\rm zar}\cup {\rm id}+{\rm id}\cup N_{\rm zar}\tag{11.5.1}
\label{cd:efidtt}\\
\end{align*} 
on 
$H_{\rm zar}(X_{\os{\circ}{T}_0}/S(T)^{\nat},E)
\otimes^L_{f^{-1}({\cal O}_T)}
H_{\rm zar}(X_{\os{\circ}{T}_0}/S(T)^{\nat},E')$. 
\end{prop}
\begin{proof} 
The proof of this proposition is similar to that of (\ref{prop:cmc}). 
\end{proof}


\begin{prop}\label{prop:eee} 
Assume that $\os{\circ}{X}_{T_0}$ is quasi-compact. 
Let $E''$ be an analogous sheaf to $E$. 
Then 
\begin{align*} 
(?\cup ?)\cup ?=?\cup (?\cup ?) 
\tag{11.6.1}\label{eqn:dutmd}
\end{align*} 
on 
$H_{\rm zar}(X_{\os{\circ}{T}_0}/S(T)^{\nat},E) 
\otimes^L_{f^{-1}({\cal O}_T)} 
H_{\rm zar}(X_{\os{\circ}{T}_0}/S(T)^{\nat},E')
\otimes^L_{f^{-1}({\cal O}_T)} 
H_{\rm zar}(X_{\os{\circ}{T}_0}/S(T)^{\nat},E'')$. 
\end{prop}
\begin{proof} 
The proof of this proposition is the same as that of (\ref{prop:cmpp}). 
\end{proof} 

\par 
Let $q$ and $q'$ be integers. 
Then the cup product 
\begin{align*} 
\cup \col &
R^qf_{X_{\os{\circ}{T}_0}/S(T)^{\nat}*}(\eps_{X_{\os{\circ}{T}_0}/S(T)^{\nat}}^*
(E))\otimes_{{\cal O}_T}
R^{q'}f_{X_{\os{\circ}{T}_0}/S(T)^{\nat}*}(\eps_{X_{\os{\circ}{T}_0}/S(T)^{\nat}}^*(E'))
\tag{11.6.2}\label{ali:oop}\\
& \lo 
R^{q+q'}f_{X_{\os{\circ}{T}_0}/S(T)^{\nat}*}(\eps_{X_{\os{\circ}{T}_0}
/S(T)^{\nat}}^*(E\otimes_{{\cal O}_{\os{\circ}{X}_{T_0}/\os{\circ}{T}}}E'))
\end{align*}  
is compatible with $P$'s. 
That is, the morphism above induces the following morphism 
\begin{align*} 
\cup  \col &
P_kR^qf_{X_{\os{\circ}{T}_0}/S(T)^{\nat}*}(\eps_{X_{\os{\circ}{T}_0}/S(T)^{\nat}}^*
(E))\otimes_{{\cal O}_T}
P_{k'}R^{q'}f_{X_{\os{\circ}{T}_0}/S(T)^{\nat}*}(\eps_{X_{\os{\circ}{T}_0}/S(T)^{\nat}}^*(E'))
\tag{11.6.3}\label{ali:otp}\\
& \lo 
P_{k+k'}R^{q+q'}f_{X_{\os{\circ}{T}_0}/S(T)^{\nat}*}(\eps_{X_{\os{\circ}{T}_0}/S(T)^{\nat}}^*
(E\otimes_{{\cal O}_{\os{\circ}{X}_{T_0}/\os{\circ}{T}}}E')). 
\end{align*}

\section{Infinitesimal deformation invariance of Hirsch weight-filtered log crystalline complexes 
and Hirsch weight-filtered log crystalline dga's}\label{sec:infhdi}  
In this section we prove the infinitesimal deformation invariance 
of the pull-back of a morphism of SNCL schemes on filtered Hirsch extensions. 
As in \cite{boi}, to prove the invariance, we use Dwork's trick for enlarging 
the radius of convergence of log $F$-isocrystals by the use of the relative Frobenius.  
Precisely speaking, in our case as in \cite{nb}, 
we use the base change by the iteration of the {\it abrelative} 
Frobenius morphism (not (usual) relative Frobenius morphism) of the base scheme. 
\par 
For a filtered complex $(K,P)$ in the derived category of filtered complexes 
(\cite{nh2}), denote $(K,P)\otimes_{\mab Z}^L{\mab Q}$ by  $(K,P)_{\mab Q}$ 
for simplicity of notation.

\par 
The following is a main result in this section.   

\begin{theo}[{\bf Infinitesimal deformation invariance of the pull-back of a morphism on 
Hirsch weight-filtered log crystalline complexes}]
\label{theo:definv}
Let $\star$ be nothing  or $\prime$. 
Let $n$ be a positive integer.  
Let $S^{\star}$ be a family of log points. 
Assume that $S^{\star}$ is of characteristic $p>0$. 
Let $F_{S^{\star}}\col S^{\star}\lo S^{\star}$ be the absolute Frobenius endomorphism. 
Set $S^{{\star}[p^n]}:=
S^{\star}\times_{\os{\circ}{S}{}^{\star},\os{\circ}{F}{}^n_{S^{\star}}}\os{\circ}{S}{}^{\star}$. 
Let $(T^{\star},{\cal J}^{\star},\del^{\star})$ be 
a log $p$-adic formal PD-thickening of $S^{\star}$. 
Set $T^{\star}_0:=T^{\star}~{\rm mod}~{\cal J}^{\star}$. 
Let $f^{\star} \col X^{\star}_{} \lo S^{\star}$ be an SNCL scheme over $S^{\star}$. 
Assume that $\os{\circ}{X}{}^{\star}_{T_0}$ is quasi-compact. 
Let $\iota^{\star} \col T^{\star}_0(0) \os{\subset}{\lo} T^{\star}_0$  
be an exact closed nilpotent immersion.  
Set $S^{\star}_{\os{\circ}{T}{}^{\star}_{0}}
:=S^{\star}\times_{\os{\circ}{S}{}^{\star}}\os{\circ}{T}{}^{\star}_0$ 
and 
$S^{\star}_{\os{\circ}{T}{}^{\star}_0(0)}:=S^{\star}
\times_{\os{\circ}{S}{}^{\star}}\os{\circ}{T}{}^{\star}_0(0)$ 
and $X^{\star}_{\os{\circ}{T}_0}:=
X^{\star}_{}\times_{S^{\star}}S^{\star}_{\os{\circ}{T}{}^{\star}_0}$,  
$X^{\star}_{\os{\circ}{T}{}^{\star}_0}(0)
:=X^{\star}\times_{S^{\star}}S^{\star}_{\os{\circ}{T}{}^{\star}_0(0)}$. 
Set $X^{\star}{}^{[p^n]}_{}:=
X^{\star}_{}\times_{S^{\star}}S^{{\star}[p^n]}$ 
and $X^{\star}{}^{[p^n]}_{\! \! \!\os{\circ}{T}_0}:=
X^{\star}{}^{[p^n]}\times_{S^{{\star}[p^n]}}S^{{\star}[p^n]}_{\os{\circ}{T}_0}$. 
$($Note that the underlying scheme of 
$X^{\star}{}^{[p^n]}_{\! \! \!\os{\circ}{T}_0}$ is equal to that of 
$X^{\star}_{\os{\circ}{T}_0}$ and that we have the log scheme 
$S^{{\star}[p^n]}_{\os{\circ}{T}_0}$ by using the composite morphism 
$T_0\lo S^{\star}\lo S^{\star [p^n]}$, 
where 
the morphism $S^{\star}\lo S^{\star[p^n]}$ is the composite of 
the abrelative Frobenius morphisms of $S^{[p^m]\star}$ $(0\leq m\leq n-1)$ 
{\rm ((\ref{exem:abfd})))}.  
Let $n$ be a positive integer such that the pull-back morphism 
$F^{n*}_{T^{\star}_0}\col 
{\cal O}_{T^{\star}_0}\lo {\cal O}_{T^{\star}_0}$ 
kills ${\rm Ker}({\cal O}_{T^{\star}_0}\lo {\cal O}_{T^{\star}_0(0)})$.
Here $S^{\star[p^n]}(T^{\star})$ is defined by 
the composite morphism $T_0\lo S^{\star}\lo S^{\star[p^n]}$. 
Let 
\begin{align*} 
g_0 \col X_{\os{\circ}{T}_0}(0) \lo X'_{\os{\circ}{T}{}'_0}(0)
\tag{12.1.1}\label{eqn:ldehtvn}
\end{align*}   
be a morphism of log schemes over $S(T)^{\nat}\lo S'(T')^{\nat}$ 
satisfying the conditions {\rm (8.1.6)}  and {\rm (8.2.6)} 
for 
$X_{\os{\circ}{T}_0}(0)$ and $X'_{\os{\circ}{T}{}'_0}(0)$. 
Set $X^{\star}_{T^{\star}_0}
:=X^{\star}_{}\times_{S^{\star}}T^{\star}_0$. 
Then the following hold$:$ 
\par 
$(1)$ There exist a canonical filtered morphism
\begin{align*}
g^*_{0} &: 
(H_{\rm zar}(X'_{\os{\circ}{T}{}'_0}/S'(T')^{\nat}),P)_{\mab Q} 
\lo 
Rg_{0*}((H_{\rm zar}(X_{\os{\circ}{T}_0}/S(T)^{\nat}),P)_{\mab Q}). 
\tag{12.1.2}\label{eqn:ldefinvn}
\end{align*}  
and a canonical morphism 
\begin{align*} 
g^*_0\col 
Ru_{X'{}^{(\star)}_{\os{\circ}{T}{}'_0}/S'(T')^{\nat}*}
({\cal O}_{X'{}^{(\star)}_{\os{\circ}{T}{}'_0}/S'(T')^{\nat}})_{\mab Q} 
\lo Rg_{0*}(Ru_{X^{(\star)}_{\os{\circ}{T}_0}/S(T)^{\nat}*}
({\cal O}_{X^{(\star)}_{\os{\circ}{T}_0}/S(T)^{\nat}})_{\mab Q}) 
\tag{12.1.3}\label{ali:ldtun}
\end{align*} 
fitting into the following commutative diagram
\begin{equation*} 
\begin{CD}
H_{\rm zar}(X'_{\os{\circ}{T}{}'_0}/S'(T')^{\nat})_{\mab Q} 
@>{g^*_{0}}>>
Rg_{0*}((H_{\rm zar}(X_{\os{\circ}{T}_0}/S(T)^{\nat}))_{\mab Q})\\
@V{\simeq}VV @VV{\simeq}V \\
Ru_{X'{}^{(\star)}_{\os{\circ}{T}{}'_0}/S'(T')^{\nat}*}
({\cal O}_{X'^{(\star)}_{\os{\circ}{T}{}'_0}/S'(T')^{\nat}})_{\mab Q} 
@>{g^*_{0}}>>
Rg_{0*}Ru_{X^{(\star)}_{\os{\circ}{T}_0}/S(T)^{\nat}*}
({\cal O}_{X^{(\star)}_{\os{\circ}{T}{}_0}/S(T)^{\nat}})_{\mab Q}\\ 
@A{\simeq}AA @AA{\simeq}A \\
Ru_{X'_{\os{\circ}{T}{}'_0}/S'(T')^{\nat}*}
({\cal O}_{X'_{\os{\circ}{T}{}'_0}/S'(T')^{\nat}})_{\mab Q} 
@>{g^*_{0}}>>
Rg_{0*}Ru_{X_{\os{\circ}{T}{}_0}/S(T)^{\nat}*}
({\cal O}_{X_{\os{\circ}{T}{}_0}/S(T)^{\nat}})_{\mab Q}. 
\end{CD}
\tag{12.1.4}\label{cd:ldtvn}
\end{equation*} 
Here the last horizontal morphism $g_0^*$ is the morphism constructed in 
{\rm \cite[(5.3.1)]{nb}}. 
\par 
$(2)$ Let $S''$, $(T'',{\cal J}'',\del'')$ and $\iota'' \col T''_0(0)\os{\subset}{\lo} T''_0$ 
be analogous objects to 
$S'$, $(T',{\cal J}',\del')$ and $\iota' \col T'_0(0)\os{\subset}{\lo} T'_0$, respectively.  
Let $g'_{0}\col X'_{\os{\circ}{T}{}'_0}(0)\lo 
X''_{\os{\circ}{T}{}''_0}(0)$ be 
a similar morphism to $g_{0}$. 
Then 
\begin{align*} 
(g'_{0}\circ g_{0})^*
=Rg'_{0*}(g^*_{0})\circ g'{}^*_{\! \!0}\col 
Ru_{X''_{\os{\circ}{T}{}''_0}/S''(T'')^{\nat}*}
({\cal O}_{X''_{\os{\circ}{T}{}''_0}/S''(T'')^{\nat}})_{\mab Q} 
\lo 
R(g'_{0}\circ g_{0})_*Ru_{X_{\os{\circ}{T}{}_0}/S(T)^{\nat}*}
({\cal O}_{X_{\os{\circ}{T}{}_0}/S(T)^{\nat}})_{\mab Q}. 
\tag{12.1.5}\label{eqn:ldfilnvn}
\end{align*}
\par 
$(3)$ 
\begin{align*}
{\rm id}^*_{X_{T^{\star}_0}(0)}=
{\rm id}_{(H_{\rm zar}(X_{\os{\circ}{T}_0}/S(T)^{\nat})_{\mab Q},P)}.
\tag{12.1.6}\label{eqn:ldeoxnvn}
\end{align*}
\par 
$(4)$ If $g_{0}$ has a lift $g_{} \col X_{\os{\circ}{T}_0}
\lo X'_{\os{\circ}{T}{}'_0}$ over 
$S_{\os{\circ}{T}_0}\lo S'_{\os{\circ}{T}{}'_0}$ satisfying the conditions $(8.1.6)$ and $(8.2.6)$, 
then $g^*_{0}$ in $(\ref{eqn:ldefinvn})$ is equal to the induced morphism by $g^*_{}$ 
in $(\ref{ali:ccm})$ 
in $E^{}={\cal O}_{X_{\os{\circ}{T}{}_0}/S(T)^{\nat}}$
and 
$F^{}={\cal O}_{Y_{\os{\circ}{T}{}'_0}/S'(T')^{\nat}}$.
\end{theo}
\begin{proof}
(Though the proof of this theorem is the same as that of \cite[(5.3.1)]{nb}, 
we give it for the completeness of this book.)
\par 
(1): Because the exact closed immersion 
$\os{\circ}{\iota}\col \os{\circ}{T}{}^{\star}_0(0) \os{\sus}{\lo} 
\os{\circ}{T}{}^{\star}_0$ is nilpotent and because 
$F^{n*}_{T^{\star}_0}\col 
{\cal O}_{T^{\star}_0}\lo {\cal O}_{T^{\star}_0}$ 
kills ${\rm Ker}({\cal O}_{T^{\star}_0}\lo {\cal O}_{T^{\star}_0}(0))$, 
we see that there exists a morphism 
$\os{\circ}{\rho}{}^{(n)\star}\col \os{\circ}{T}{}^{\star}_0\lo \os{\circ}{T}{}^{\star}_0(0)$ 
such that $\os{\circ}{\rho}{}^{(n)\star} \circ \os{\circ}{\iota}=\os{\circ}{F}{}^n_{T^{\star}_0(0)}$ 
and $\os{\circ}{\iota}\circ \os{\circ}{\rho}{}^{(n)\star}
=\os{\circ}{F}{}^{n}_{T^{\star}_0}$. 
Hence 
\begin{align*} 
X^{\star}{}^{[p^n]}_{\! \! \! \os{\circ}{T}{}^{\star}_0}
&=X^{\star}_{}\times_{\os{\circ}{S}{}^{\star},\os{\circ}{F}{}^n_{S^{\star}}}
\os{\circ}{S}{}^{\star}\times_{\os{\circ}{S}{}^{\star}}\os{\circ}{T}{}^{\star}_0
=X^{\star}_{\os{\circ}{T}{}^{\star}_0}
\times_{\os{\circ}{T}{}^{\star}_0,\os{\circ}{F}{}^n_{T^{\star}_0}}
\os{\circ}{T}{}^{\star}_0
=X^{\star}_{\os{\circ}{T}{}^{\star}_0}
\times_{\os{\circ}{T}{}^{\star}_0,\os{\circ}{\iota}\circ \os{\circ}{\rho}{}^{(n)\star}}
\os{\circ}{T}{}^{\star}_0\tag{12.1.7}\label{eqn:ldet0vn}\\
&=
X^{\star}_{\os{\circ}{T}{}^{\star}_0}(0)
\times_{\os{\circ}{T}{}^{\star}_0(0),\os{\circ}{\rho}{}^{(n)\star}}
\os{\circ}{T}{}^{\star}_0. 
\end{align*}  
\par 
Consider the filtered complex 
$(H_{\rm zar}(X^{\star[p^n]}_{\os{\circ}{T}{}^{\star}_0}/
S^{\star[p^n]}(T^{\star})^{\nat}),P)_{\mab Q}$.   
Because we are given the morphism 
$g_{0} \col X_{\os{\circ}{T}_0}(0)
\lo X'_{\os{\circ}{T}{}'_0}(0)$, 
we have the following base change morphism 
\begin{align*} 
g^{[p^n]}_{0}
\col 
X^{[p^n]}_{\os{\circ}{T}_0}=
X_{\os{\circ}{T}_0}(0)
\times_{\os{\circ}{T}_0(0),\os{\circ}{\rho}{}^{(n)}}\os{\circ}{T}_0
\lo X'_{\os{\circ}{T}{}'_0}(0)
\times_{\os{\circ}{T}{}'_0(0),\os{\circ}{\rho}{}^{(n)'}}\os{\circ}{T}{}'_0
=X'{}^{[p^n]}_{\!\!\!\os{\circ}{T}{}'_0}. 
\end{align*} 
Because $g_{0}$ satisfies the condition (8.1.6), 
the morphism $g^{[p^n]}_{0}$ also satisfies it. 
Hence we have the pull-back morphism by the functoriality of 
$(H_{\rm zar},P)$ ((\ref{ali:ccm})): 
\begin{align*}  
(H_{\rm zar}(X'{}^{[p^n]}_{\os{\circ}{T}{}'_0}/S'{}^{[p^n]}(T')^{\nat}),P)_{\mab Q}
\lo 
Rg^{[p^n]}_{0*}
((H_{\rm zar}(X^{[p^n]}_{\os{\circ}{T}{}_0}/S^{[p^n]}(T)^{\nat}),P)_{\mab Q}).  
\tag{12.1.8}\label{ali:yxlpn}
\end{align*}  
The abrelative Frobenius morphism 
$$F^{{\rm ar},n}_{X^{\star}_{\os{\circ}{T}{}^{\star}_0}/S^{\star}(T^{\star})^{\nat}} 
\col X^{\star}_{\os{\circ}{T}{}^{\star}_0}
\lo X^{\star [p^n]}_{\os{\circ}{T}{}^{\star}_0}$$ 
over the morphism $S^{\star}(T^{\star})^{\nat}\lo S^{\star[p^n]}(T^{\star})^{\nat}$
induces an isomorphism
\begin{align*} 
F^{{\rm ar},n*}_{X^{\star}_{\os{\circ}{T}{}^{\star}_0}/S^{\star}(T^{\star})^{\nat}}
\col &
(H_{\rm zar}(X^{\star[p^n]}_{\os{\circ}{T}{}^{\star}_0}/S^{\star[p^n]}
(T^{\star})^{\nat}), P)_{\mab Q}
\os{\sim}{\lo} (H_{\rm zar}(X^{\star}_{\os{\circ}{T}{}^{\star}_0}
/S^{\star}(T^{\star})^{\nat}), P)_{\mab Q}
\end{align*} 
by (\ref{ali:sparh}) because the base change over $\os{\circ}{T}{}^{\star}$ of 
the abrelative Frobenius morphism induces an isomorphism of  
the classical isocrystalline complex of a smooth scheme  
(\cite[(2.24)]{hk}, cf.~\cite[(1.3)]{boi}).
Here, as in \cite[(2.24)]{hk}, we identity 
$(X^{\star}{}^{[p^n]}_{\! \! \!\os{\circ}{T}{}^{\star}_0})_{\rm zar}$ 
with $(X^{\star}_{\os{\circ}{T}{}^{\star}_0})_{\rm zar}$ 
via the canonical equivalence and we have used the assumption that 
$g_{0}$ satisfies the condition (8.2.6). 
Consequently we obtain the following diagram: 
\begin{equation*} 
\begin{CD} 
(H_{\rm zar}(X'{}^{[p^n]}_{\os{\circ}{T}{}'_0}/S^{'[p^n]}(T')^{\nat}), P)_{\mab Q}
@>{g^{[p^n]*}_{0}}>> 
Rg^{[p^n]}_{0*}((H_{\rm zar}(X^{[p^n]}_{\os{\circ}{T}_0}/S^{[p^n]}(T)^{\nat}), P)_{\mab Q})
)\\
@V{\simeq}V{F^{{\rm ar},n*}_{X'{}_{\os{\circ}{T}{}'_0}/S'(T')^{\nat}}}V 
@V{\simeq}V{F^{{\rm ar},n*}_{X_{\os{\circ}{T}_0}/S(T)^{\nat}}}V \\
(H_{\rm zar}(X'_{\os{\circ}{T}{}'_0}/S'(T')^{\nat}), P)_{\mab Q}@.
(H_{\rm zar}(X_{\os{\circ}{T}_0}/S(T)^{\nat}), P)_{\mab Q}. 
\end{CD} 
\tag{12.1.9}\label{cd:xy}
\end{equation*}
This diagram gives us the morphism $g^*_{0}$ in (\ref{eqn:ldefinvn}): 
$$g^*_{0}:=F^{{\rm ar},n*}_{X_{\os{\circ}{T}_0}/S(T)^{\nat}}
\circ 
g^{[p^n]*}_{0} \circ (F^{{\rm ar},n*}_{X'{}_{\os{\circ}{T}{}'_0}/S'(T')^{\nat}})^{-1}.$$
This morphism is independent of the choice of $n$ by 
the transitive relation of the pull-back morphism ((\ref{ali:ccm})). 
Similarly we have the middle morphism $g^*_0$ 
(and even the lower morphism $g^*_0$) in (\ref{cd:ldtvn}). 
By the contravariant functoriality 
(\ref{cd:eetxte}) of the isomorphisms (\ref{eqn:eaxte}) and (\ref{eqn:eetxte}), 
we obtain the commutative diagram 
(\ref{cd:ldtvn}). 
\par 
(2): By the transitive relation of the pull-back morphism again, 
$g^*_{0}$ and $g'{}^*_{\!\!0}$ 
are compatible with the composition of 
$g_{0}$ and $g'_{0}$. 
\par 
(3): The formula (\ref{eqn:ldeoxnvn}) immediately follows from 
the definition of $g^*_{0}$. 
\par 
(4): Assume that $g_{0}$ has the lift $g_{}$ in the statement of this theorem. 
Then we have the relation $g^*_{0}=g^*_{{}}$ by 
the following commutative diagram, the transitivity (\ref{ali:ccm}) and 
the diagram (\ref{cd:xy}): 
\begin{equation*} 
\begin{CD} 
X_{\os{\circ}{T}_0}@>{g_{}}>>X'_{\os{\circ}{T}{}'_0}\\
@V{F^{{\rm ar},n}_{X_{\os{\circ}{T}_0}/S(T)^{\nat}}}VV 
@VV{F^{{\rm ar},n}_{X'{}_{\os{\circ}{T}{}'_0}/S'(T')^{\nat}}}V \\
X^{[p^n]}_{\os{\circ}{T}_0}@>{g_{0}^{[p^n]}}>>X'{}^{[p^n]}_{\os{\circ}{T}{}'_0}. 
\end{CD} 
\end{equation*}
\end{proof}

\begin{coro}[{\bf Infinitesimal deformation invariance of  
Hirsch weight-filtered log crystalline complexes}]\label{coro:finvcae}
If $S'=S$, $T'=T$ and $X_{\os{\circ}{T}_0}(0)=X'_{\os{\circ}{T}_0}(0)$, then 
\begin{align*} 
(H_{\rm zar}(X'_{\os{\circ}{T}_0}/S(T)^{\nat}), P)_{\mab Q}=
(H_{\rm zar}(X_{\os{\circ}{T}_0}/S(T)^{\nat}), P)_{\mab Q}. 
\tag{12.2.1}\label{ali:xdnz}
\end{align*} 
\end{coro}

\begin{coro}[{\bf Infinitesimal deformation invariance of log isocrystalline 
cohomologies with weight filtrations}]\label{coro:finvliae}
Let the notations be as in {\rm (\ref{coro:finvcae})}. 
Let $P$ be the induced filtrations on 
$R^qf_{X'_{\os{\circ}{T}_0}/S(T)^{\nat}*}
({\cal O}_{X'_{\os{\circ}{T}_0}/S(T)^{\nat}})_{\mab Q}$ and 
$R^qf_{X_{\os{\circ}{T}_0}/S(T)^{\nat}*}
({\cal O}_{X_{\os{\circ}{T}_0}/S(T)^{\nat}})_{\mab Q}$ by 
$(H_{\rm zar}(X'_{\os{\circ}{T}_0}/\os{\circ}{T}), P)_{\mab Q}$ 
and 
$(H_{\rm zar}(X_{\os{\circ}{T}_0}/S(T)^{\nat}), P)_{\mab Q}$, 
respectively.
Then 
\begin{align*} 
(R^qf_{X'_{\os{\circ}{T}_0}/S(T)^{\nat}*}
({\cal O}_{X'_{\os{\circ}{T}_0}/S(T)^{\nat}})_{\mab Q},P)&=
(R^qf_{X_{\os{\circ}{T}_0}/S(T)^{\nat}*}
({\cal O}_{X_{\os{\circ}{T}_0}/S(T)^{\nat}})_{\mab Q},P).
\tag{12.3.1}\label{ali:wfoa}\\
\end{align*}  
Moreover, if $T$ is restrictively hollow $(${\rm \cite[(1.1.38)]{nb}}$)$, that is, 
if $S(T)$ is hollow, 
then 
\begin{align*} 
(R^qf_{X'_{T_0}/T*}({\cal O}_{X'_{T_0}/T})_{\mab Q},P)
&=
(R^qf_{X_{T_0}/T*}
({\cal O}_{X_{T_0}/T})_{\mab Q},P).
\tag{12.3.2}\label{ali:wfota}\\
\end{align*}  
\end{coro} 

\begin{rema}\label{rema:ani}
The analogues of (\ref{theo:definv}) and (\ref{coro:finvcae}) hold 
for $(H_{{\rm zar},{\rm TW}},P)$ hold. 
\end{rema}

\section{The $E_2$-degeneration of the $p$-adic weight spectral sequence}\label{sec:filbo}
Let ${\cal V}$ be a complete discrete valuation ring with perfect residue field $\kap$ 
of mixed characteristics. 
In this section we assume that the underlying scheme $\os{\circ}{S}$ of 
the family of log points  $S$ is a $p$-adic ${\cal V}$-scheme in the sense 
of \cite{od}. Let $X/S$ be a proper SNCL scheme.   Let 
$(T,z)$ be a flat log $p$-adic enlargement (see e.g.,~\cite{oc} (or \cite{nb}) for this notion); 
$z$ is a morphism $T_1\lo S$, 
where $T_1$ is an exact log subscheme of $T$ defined by $p{\cal O}_T$. 
Endow $p{\cal O}_T$ with the canonical PD-structure.  
In this section we prove the $E_2$-degeneration of 
the $p$-adic weight spectral sequence of 
$X_{\os{\circ}{T}_1}/S(T)^{\nat}$ modulo torsion obtained by 
$(H_{\rm zar}(X_{\os{\circ}{T}_0}/S(T)^{\nat}),P)$ by using 
the infinitesimal deformation invariance of isocrystalline cohomologies 
with weight filtrations in the previous section ((\ref{ali:wfoa})).

\begin{theo}[{\bf $E_2$-degeneration I}]\label{theo:e2dam}  
Let $s$ be the log point of a perfect field of characteristic $p>0$. 
The spectral sequence $(\ref{ali:sparh})$ for the case $S=s$ 
and $E={\cal O}_{\os{\circ}{X}_{T_0}/\os{\circ}{T}}$ 
degenerates at $E_2$.  
\end{theo} 
\begin{proof} 
As in the proof of \cite[(5.4.1)]{nb}, we may assume that 
$T={\cal W}(s)$.  In this case we have the absolute Frobenius endomorphism 
$F_{{\cal W}(s)}\col {\cal W}(s)\lo {\cal W}(s)$. 
If $\os{\circ}{s}$ is the spectrum of a finite field, 
then the $E_1$-term 
$E^{-k,q+k}_1$of (\ref{ali:sparh}) is of pure weight of $q+k$ 
by \cite[Corollary 1. 2)]{kme}, \cite[(1.2)]{clpu} and \cite[(2.2) (4)]{ndw}.  
However see  \cite[(6.11)]{ny} for the gap of the proof 
the weak-Lefschetz conjecture 
for a hypersurface of a large degree  in \cite{bwl}; 
I have filled the gap in \cite[(6.10)]{ny}. 
\par 
The rest of the proof is the same as that of \cite[(5.4.1)]{nb}
\end{proof}

\begin{theo}[{\bf $E_2$-degeneration II}]\label{theo:e2dgfam} 
Let $T$ be a log $p$-adic enlargement of $S/{\cal V}$ 
with structural morphism $T_1\lo S$. 
The spectral sequence $(\ref{ali:sparh})$ modulo torsion 
for the case $E={\cal O}_{\os{\circ}{X}_{T_0}/\os{\circ}{T}}$ 
degenerates at $E_2$.  
\end{theo}
\begin{proof} 
Because the problem is local on $\os{\circ}{T}$, we may assume that 
$\os{\circ}{T}$ is quasi-compact.  
By virtue of (\ref{theo:e2dam}) and (\ref{coro:finvcae}), 
the proof of this theorem is the same as that of \cite[(5.4.3)]{nb}.
\end{proof}

\section{Convergence of the weight filtration}\label{sec:e2}
In this section we prove the log convergence of 
the filtration on the log crystalline cohomology modulo torsions 
induced by the spectral sequence (\ref{ali:sparh}) if $X/S$ is proper. 
\par 
Roughly speaking, we obtain all the results in this section 
by using the log base change theorem of 
$(H_{\rm zar},P)$ ((\ref{theo:bccange})) and by 
replacing $(H_{{\rm zar},{\mab Q}},P)$ defined in this book 
with $(A_{{\rm zar},{\mab Q}},P)$ in \cite{nb}. 
Here we have to note that the base change morphism 
satisfy the conditions (8.1.6) and (8.2.6) as in the proof of (\ref{theo:bccange}). 
For this reason, we omit or sketch the proofs of almost all the results in this section. 
We use fundamental notions and results in \cite[(5.2)]{nb}. 
\par 
Let ${\cal V}$ be a complete discrete valuation ring of mixed characteristics 
$(0,p)$ with perfect residue field. Let $K$ be the fraction field of ${\cal V}$. 
Set $B=({\rm Spf}({\cal V}),{\cal V}^*)$.   
Let $S$ be a $p$-adic formal family of log points over $B$ 
such that $\os{\circ}{S}$ is a ${\cal V}/p$-scheme.  
Let $X/S$ be a proper SNCL scheme.  
\par  
Let $n$ be a nonnegative integer.  
Let $T$ be an object of the category ${\cal E}^{\sq}_{p,n}:={\rm Enl}^{\sq}_p(S^{[p^n]}/{\cal V})$ 
of (solid) $p$-adic enlargements of $S^{[p^n]}/{\cal V}$ 
($\sq=$sld or nothing) (\cite[(5.1.3)]{nb}).   
Then the hollowing out $S^{[p^n]}(T)^{\nat}$ of $S^{[p^n]}(T)$ is a formal family of log points. 
Let $T_i \lo S^{[p^n]}$ $(i=0,1)$ be the structural morphism. 
Here $T_1:=\ul{\rm Spec}^{\log}_T({\cal O}_T/p)$ and $T_0:=(T_1)_{\rm red}$. 
Set $X^{[p^n]}_{\os{\circ}{T}_i}:=X\times_SS^{[p^n]}_{\os{\circ}{T}_i} 
=X^{[p^n]}\times_{\os{\circ}{S}}\os{\circ}{T}_i$, 
where $X^{[p^n]}:=X\times_SS^{[p^n]}$. 
Let 
$f^{[p^n]}_{\os{\circ}{T}} \col X^{[p^n]}_{\os{\circ}{T}_i}\lo S^{[p^n]}(T)^{\nat}$ 
be the structural morphism. 
(Note that this notation is different from the notation 
$f$ in \S\ref{sec:psc} since we add the symbol $\os{\circ}{T}$ to $f^{[p^n]}$ as a subscript.) 
\par 
Assume that we are given a flat coherent crystal 
$E_n=\{E_n(\os{\circ}{T})\}_{T\in {\cal E}^{\sq}_{p,n}}$ of 
${\cal O}_{\{\os{\circ}{X}{}^{[p^n]}_{T_1}
/\os{\circ}{T}\}_{T\in {\cal E}^{\sq}_{p,n}}}$-modules.  
Let $T$ be an object of ${\cal E}^{\sq}_{p,n}$.
Let $z\col T_1\lo S^{[p^n]}$ be the structural morphism. 
Because $S^{[p^n]}(T)^{\nat}$  is a $p$-adic formal family of log points,  
we have a base change  $X^{[p^n]}_{\os{\circ}{T}_1}/S^{[p^n]}_{\os{\circ}{T}_1}$ 
of a proper SNCL scheme 
with an exact PD-closed immersion 
$S^{[p^n]}_{\os{\circ}{T}_1} \os{\sus}{\lo} S^{[p^n]}(T)^{\nat}$. 
Then we obtain the following $p$-adic iso-zariskian filtered complex
\begin{align*} 
(H_{\rm zar}(X^{[p^n]}_{\os{\circ}{T}_1}/
S^{[p^n]}(T)^{\nat},E^{}(\os{\circ}{T})),P)_{\mab Q}
\in {\rm D}^+{\rm F}(
f^{-1}_{\os{\circ}{T}}({\cal K}_T))
\tag{14.0.1}\label{eqn:auxtds}
\end{align*} 
for each $T\in {\cal E}^{\sq}_{p,n}$.

\begin{prop}[{\bf cf.~\cite[(5.2.2)]{nb}}]\label{prop:tptt} 
Let ${\mathfrak g}\col T'\lo T$ be a morphism in ${\cal E}^{\sq}_{p,n}$.  
Then the induced morphism 
$g_{} \col X^{[p^n]}_{\os{\circ}{T}{}'_1}\lo X^{[p^n]}_{\os{\circ}{T}_1}$ 
by ${\mathfrak g}$
gives us the following natural morphism 
\begin{align*} 
g^*_{}\col &
(H_{\rm zar}(X^{[p^n]}_{\os{\circ}{T}_1}/
S^{[p^n]}(T)^{\nat},E^{}(\os{\circ}{T})),P)
\tag{14.1.1}\label{eqn:auqeds}\\
& \lo
Rg_{*}
((H_{\rm zar}(X^{[p^n]}_{\os{\circ}{T}{}'_1}/S^{[p^n]}(T')^{\nat},
E^{}(\os{\circ}{T}{}')),P))
\end{align*}
of filtered complexes in 
${\rm D}^+{\rm F}(f_{\os{\circ}{T}}^{-1}({\cal K}_T))$ 
fitting into the following commutative diagram$:$ 
\begin{equation*} 
\begin{CD} 
H_{\rm zar}(X^{[p^n]}_{\os{\circ}{T}_1}/
S^{[p^n]}(T)^{\nat},E^{}(\os{\circ}{T}))
@>{g^*_{}}>> \\
@A{\simeq}AA \\
Ru_{X^{[p^n]}_{\os{\circ}{T}_1}/
S^{[p^n]}(T)^{\nat}*}(\eps_{X^{[p^n]}_{\os{\circ}{T}_1}/
S^{[p^n]}(T)^{\nat}}(E^{}(\os{\circ}{T})))\otimes_{\mab Z}{\mab Q}
@>{g^*_{}}>>
\end{CD} 
\end{equation*} 
\begin{equation*} 
\begin{CD} 
Rg_{*}
((H_{\rm zar}(X^{[p^n]}_{\os{\circ}{T}{}'_1}/S^{[p^n]}(T')^{\nat},
E^{}(\os{\circ}{T}{}')),P))\\ 
@AA{\simeq}A \\
Rg_{*}(Ru_{X^{[p^n]}_{\os{\circ}{T}{}'_1}/
S^{[p^n]}(T')^{\nat}*}(\eps^*_{X^{[p^n]}_{\os{\circ}{T}{}'_1}/
S^{[p^n]}(T')^{\nat}}(E^{}(\os{\circ}{T}{}'))\otimes_{\mab Z}{\mab Q}).
\end{CD} 
\end{equation*} 
For a similar morphism ${\mathfrak h}\col T''\lo T'$ to ${\mathfrak g}$ and 
a similar morphism 
$h_{}\col 
X_{\os{\circ}{T}{}''_1}\lo X_{\os{\circ}{T}{}'_1}$ to $g_{}$, 
the following relation 
\begin{align*} 
(h_{}\circ g_{})^*=
Rh_{*}(g^*_{})
\circ h^*_{}
\end{align*} 
holds.  
\begin{equation*} 
{\rm id}_{X^{[p^n]}_{\os{\circ}{T}_1}}^*={\rm id} 
_{(H_{\rm zar}(X^{[p^n]}_{\os{\circ}{T}_1}/
S(T)^{\nat},E^{}(\os{\circ}{T})),P)}.  
\end{equation*} 
\end{prop} 
\begin{proof} 
This immediately follows from the contravariant functoriality (\ref{theo:ccm}). 
\end{proof} 

\parno 
The morphism (\ref{eqn:auqeds}) induces the following morphism 
\begin{align*} 
{\mathfrak g}^*\col &
R
f^{[p^n]}_{\os{\circ}{T}*}((H_{\rm zar}(X^{[p^n]}_{\os{\circ}{T}_1}
/S^{[p^n]}(T)^{\nat},E^{}(\os{\circ}{T})),P))
\tag{14.1.2}\label{eqn:auaeds}\\
&\lo
R{\mathfrak g}_*R
f^{[p^n]}_{\os{\circ}{T}{}'*}
((H_{\rm zar}(X^{[p^n]}_{\os{\circ}{T}{}'_1}/S^{[p^n]}(T')^{\nat},
E^{}(\os{\circ}{T}{}')),P))
\end{align*}
of filtered complexes in ${\rm D}^+{\rm F}(f^{-1}({\cal K}_T))$.

\par  
The following is a key lemma for (\ref{theo:pwfec}) below.

\begin{lemm}[{\bf cf.~\cite[(5.2.3)]{nb}}]\label{lemm:pnlcfi}
Let ${\rm IsocF}^{\sq}_p(S/{\cal V})$ be the category of filtered log $p$-adically convergent isocrystals on 
${\rm Enl}^{\sq}_p(S/{\cal V})$. 
Assume that $M_S$ is split. Let $k$ be a nonnegative integer or $\infty$. 
Then there exists an object
\begin{align*} 
(R^q f^{[p^n]}_*
(P_kH_{\rm zar}(X^{[p^n]}_{}/K,E^{})),P)
\tag{14.2.1}\label{ali:pkalkf} 
\end{align*} 
of ${\rm IsocF}^{\sq}_p(\os{\circ}{S}/{\cal V})$ 
such that 
\begin{align*} 
&(R^qf^{[p^n]}_*
(P_kH_{\rm zar}(X^{[p^n]}_{}/K,E^{})),P)_{\os{\circ}{T}}=
\tag{14.2.2}\label{ali:pkakf} \\
&(R^qf^{[p^n]}_{\os{\circ}{T}*}
(P_kH_{\rm zar}(X^{[p^n]}_{\os{\circ}{T}_1}/S^{[p^n]}({\os{\circ}{T})},
E^{}(\os{\circ}{T}))),P)
\end{align*} 
for any object $\os{\circ}{T}$ in ${\rm Enl}^{\sq}_p(\os{\circ}{S}/{\cal V})$. 
In particular, there exists an object
\begin{align*} 
(R^q
f^{[p^n]}_{*}
(\eps^*_{X^{[p^n]}_{}/K}(E^{}_K)),P)
\tag{14.2.3}\label{ali:pklkf} 
\end{align*} 
of ${\rm IsocF}_p(\os{\circ}{S}/{\cal V})$ 
such that 
\begin{align*} 
(R^q
f^{[p^n]}_{*}
(\eps^*_{X^{[p^n]}_{}/K}(E^{}_K)),P)_{\os{\circ}{T}}=
(R^qf^{[p^n]}_{X^{[p^n]}_{\os{\circ}{T}_1}/S^{[p^n]}(\os{\circ}{T})*}
(\eps^*_{X^{[p^n]}_{\os{\circ}{T}_1}/S^{[p^n]}(\os{\circ}{T})}
(E^{}(\os{\circ}{T})))_{\mab Q},P)
\tag{14.2.4}\label{ali:pkkf} 
\end{align*} 
for any object $\os{\circ}{T}$ in ${\rm Enl}_p(\os{\circ}{S}/{\cal V})$. 
\end{lemm} 
\begin{proof} 
The proof is the same as that of \cite[(5.2.3)]{nb}. 
\end{proof}

\par 
Let $A$ be a commutative ring with unit element. 
Recall that we have said that a filtered $A$-module $(M,P)$ is filteredly flat 
if $M$ and $M/P_kM$ $(\forall k\in {\mab Z})$ are flat $A$-modules 
(\cite[(1.1.14)]{nh2}). As a corollary of (\ref{lemm:pnlcfi}), we obtain the following:

\begin{coro}[{\bf cf.~\cite[(5.2.4)]{nb}}]\label{coro:flft}
For a hollow log $p$-adic enlargement $T$ of $S^{[p^n]}/{\cal V}$,    
the filtered sheaf
\begin{align*}  
(R^qf^{[p^n]}_{\os{\circ}{T}*}
(P_kH_{\rm zar}(X^{[p^n]}_{\os{\circ}{T}_1}/S^{[p^n]}(T)^{\nat},
E^{}(\os{\circ}{T}))),P) 
\tag{14.3.1}\label{ali:pce}
\end{align*} 
is a filteredly flat ${\cal K}_T$-modules. 
In particular, the filtered sheaf 
\begin{align*}  
(R^qf^{[p^n]}_{\os{\circ}{T}*}
(\eps^*_{X^{[p^n]}_{\os{\circ}{T}_1}
/S^{[p^n]}(T)^{\nat}}E^{}(\os{\circ}{T})),P)
\tag{14.3.2}\label{ali:pcte}
\end{align*} 
is a filteredly flat ${\cal K}_T$-module. 
\end{coro}

\begin{theo}[{\bf Log $p$-adic convergence of the weight filtration (cf.~\cite[(5.2.3)]{nb})}]
\label{theo:pwfec} 
Let $k$, $q$ be two nonnegative integers. 
Then there exists a unique object
\begin{align*} 
(R^qf^{[p^n]}_{*}
(P_kH_{\rm zar}(X^{[p^n]}_{}/K,E^{}))^{\sq},P)
\tag{14.4.1}\label{ali:pcpkxee}
\end{align*} 
of ${\rm IsocF}^{\sq}_p(S^{[p^n]}/{\cal V})$ 
such that  
\begin{align*} 
&(R^qf^{[p^n]}_{*}
(P_kH_{\rm zar}(X^{[p^n]}_{}/K,E^{}))^{\sq},P)_{T}
\tag{14.4.2}\label{ali:pkaskf} \\
&=
(R^qf^{[p^n]}_{\os{\circ}{T}*}
(P_kH_{\rm zar}
(X^{[p^n]}_{\os{\circ}{T}_1}/S^{[p^n]}(T)^{\nat},E^{}(\os{\circ}{T}))),P)
\end{align*} 
for any object $T$ of ${\rm Enl}^{\sq}_p(S^{[p^n]}/{\cal V})$.  
In particular, there exists a unique object
\begin{align*} 
(R^qf^{[p^n]}_{*}
(\eps^*_{X^{[p^n]}_{}/K}(E^{}_K))^{\nat,\sq},P)
\tag{14.4.3}\label{ali:pcee}
\end{align*}  
of ${\rm IsocF}^{\sq}_p(S^{[p^n]}/{\cal V})$ 
such that  
\begin{align*} 
(R^q f^{[p^n]}_{*}
(\eps^*_{X_{}/K}(E^{}_K))^{\nat,\sq},P)_T=
(R^qf^{[p^n]}_{X^{[p^n]}_{\os{\circ}{T}_1}/S^{[p^n]}(T)^{\nat}*}
(\eps^*_{X^{[p^n]}_{T_1}/S^{[p^n]}(T)^{\nat}}
(E^{}(\os{\circ}{T}))),P)
\tag{14.4.4}\label{ali:apce}
\end{align*} 
for any object $T$ of ${\rm Enl}^{\sq}_p(S^{[p^n]}/{\cal V})$. 
\end{theo} 
\begin{proof} 
By using (\ref{lemm:pnlcfi}), the proof is the same as that of \cite[(5.2.3)]{nb}. 
\end{proof}


\parno 
The following is a filtered version of \cite[(3.5)]{od}. 
To consider the category ${\rm IsocF}^{\rm sld}_p(S^{[p^n]}/{\cal V})$ 
(not ${\rm IsocF}_p(S^{[p^n]}/{\cal V})$) is important. 

\begin{lemm}[{\bf cf.~\cite[(5.2.6)]{nb}}]\label{lemm:nmr}
Let ${\cal V}'$ be a finite extension of complete discrete valuation ring of ${\cal V}$. 
Let $h\col S'\lo S$ be a morphism of $p$-adic formal families of log points over 
${\rm Spec}({\cal V}')\lo {\rm Spec}({\cal V})$.  
Let $g\col Y_{}\lo X_{}$ be 
a morphism of $N$-truncated simplicial base changes 
of SNCL schemes over $S'\lo S$. 
Let $f'\col Y_{}\lo S'$ be the structural morphism. 
Then there exists a natural morphism 
\begin{align*} 
g^*\col &
g^*((R^qf^{[p^n]}_{*}
(P_kH_{\rm zar}(X^{[p^n]}_{}/K,E^{}))^{\rm sld},P))\\
&\lo (R^qf'{}^{[p^n]}_{*}
(P_kH_{\rm zar}
(Y^{[p^n]}_{}/K',\os{\circ}{g}{}^*(E^{})))^{\rm sld},P). 
\end{align*}
in ${\rm IsocF}^{\rm sld}_p((S')^{[p^n]}/{\cal V}')$.  
If $Y_{}=X_{}\times_SS'
$, 
then this morphism is an isomorphism. 
\end{lemm}

\begin{rema}[{\bf cf.~\cite[(5.2.7)]{nb}}]\label{rema:ud}
As in \cite[(3.6)]{od}, 
$(R^qf^{[p^n]}_{*}
(P_kH_{\rm zar}(X^{[p^n]}_{}/K,E^{})),P)$ 
in ${\rm IsocF}^{\sq}_p(S^{[p^n]}/{\cal V})$ descends to 
the object 
$(R^qf^{[p^n]}_{*}
(P_kH_{\rm zar}(X^{[p^n]}_{}/K_0,E^{})),P)$
of ${\rm IsocF}^{\sq}_p(S^{[p^n]}/{\cal W})$.  
\end{rema}


\begin{theo}[{\bf Log convergence of the weight filtration (cf.~\cite[(5.2.8)]{nb})}]\label{theo:pwfaec}
Let the notations and the assumptions be as above.  
Consider the morphisms $X^{[p^m]}\lo S^{[p^m]}$ 
over ${\rm Spf}({\cal V})$ $(m=0,1)$ 
as a morphism $X^{[p^m]}\lo S^{[p^m]}$ 
over ${\rm Spf}({\cal W})$, respectively. 
Set ${\cal E}^{\sq}_{\star,{\cal W}}:={\rm Enl}^{\sq}_{\star}(S/{\cal W})$.  
Let 
$F^{\rm ar}_{X/S/{\cal W}} \col X\lo X^{[p]}$ 
be the abrelative Frobenius morphism 
over the morphism $S\lo S^{[p]}$ over ${\rm Spf}({\cal W})$.  
Let $W_{X/S^{[p]}/{\cal W}}
\col X^{[p]}\lo X$ 
be the projection. 
Let $E:=\{E_n\}_{n=0}^{\inf}$ be a sequence of flat coherent 
$\{{\cal O}_{\os{\circ}{X}{}_{\os{\circ}{T}_1}
/\os{\circ}{T}}\}_{T\in {\cal E}^{\sq}_{p,{\cal W}}}$-modules 
with a morphism 
\begin{align*} 
\Phi_n\col F_{\os{\circ}{X}{}}^*  (E_{n+1})\lo E_n
\tag{14.7.1}\label{ali:pchkee} 
\end{align*} 
of $\{{\cal O}_{\os{\circ}{X}{}_{\os{\circ}{T}_1}
/\os{\circ}{T}}\}_{T\in {\cal E}^{\sq}_{p,{\cal W}}}$-modules.  
Let $\os{\circ}{W}{}^{(l)}_{T}
\col (\os{\circ}{X}{}^{[p]}_{T_1})^{(l)}\lo \os{\circ}{X}{}^{(l)}_{T_1}$ 
$(l\in {\mab N}, 0\leq m\leq N)$ 
be also the projection over $\os{\circ}{T}$. 
Let 
$F^{\inf}{\textrm -}{\rm IsosF}^{\sq}(S/{\cal V})$ be the category of 
$F^{\inf}$-isospans on $S/{\cal V}$ {\rm (\cite[(5.1.14)]{nb})}. 
Assume that for any $l\geq 0$ and $n\geq 0$, 
the morphism 
\begin{align*} 
R^qf_{(\os{\circ}{X}{}^{[p]}_{T_1})^{(l)}/\os{\circ}{T}*}
(\os{\circ}{W}{}^{(l)*}_{T,{\rm crys}}
(E_{n+1}(\os{\circ}{T})_{(\os{\circ}{X}{}^{[p]}_{T_1})^{(l)}/\os{\circ}{T}}))
_{\mab Q}\lo  
R^qf_{\os{\circ}{X}{}^{(l)}_{T_1}/\os{\circ}{T}*}
(E_n(\os{\circ}{T})_{\os{\circ}{X}{}^{(l)}_{T_1}/\os{\circ}{T}})_{\mab Q}
\end{align*}
induced by $\Phi_n$ is an isomorphism for any object $T$ of ${\cal E}^{\sq}_{p,{\cal W}}$.  
Then there exists an object
\begin{align*} 
\{((R^q f_*(P_kH_{\rm zar}(X/K,E_n))^{\sq},P),\Phi_n)\}_{n=0}^{\inf}
\tag{14.7.2}\label{ali:pcxpkee} 
\end{align*}  
of $F^{\inf}{\textrm -}{\rm IsosF}^{\sq}(S/{\cal V})$ such that 
\begin{align*} 
(R^q f_*(P_kH_{\rm zar}(X/K,E_n))^{\sq}_T,P)=
(R^q f_*(P_kH_{\rm zar}(X_{T_1}/S(T)^{\nat},E_n(\os{\circ}{T}))),P)
\tag{14.7.3}\label{ali:pcpnle}
\end{align*} 
for any object $T$ of ${\rm Enl}^{\sq}_p(S/{\cal V})$. 
In particular, there exists an object
\begin{align*} 
\{(R^qf_{*}
(\eps^*_{X/K}(E_{n,K}))^{\nat,\sq},P),\Phi_n\}_{n=0}^{\inf}
\tag{14.7.4}\label{ali:pcxakee} 
\end{align*}  
of $F^{\infty}{\textrm -}{\rm IsosF}^{\sq}(S/{\cal V})$ such that 
\begin{align*} 
\{(R^qf_*(\eps^*_{X/K}(E_{n,K}))^{\nat,\sq}_T,P)\}_{n=0}^{\inf}=
(R^qf_{X_{\os{\circ}{T}_1}/S(T)^{\nat}*}(\eps^*_{X_{\os{\circ}{T}_1}/S(T)^{\nat}}(E_n(\os{\circ}{T})))_{\mab Q},P) 
\tag{14.7.5}\label{ali:pcxnle}
\end{align*} 
for any object $T$ of ${\rm Enl}^{\sq}_p(S/{\cal V})$. 
\end{theo}
\begin{proof} 
By using (\ref{theo:pwfec}), the proof is the same as that of \cite[(5.2.8)]{nb}. 
\end{proof}



\begin{coro}\label{coro:fenlt}
Assume that $E_n=E_0$ for any $n\in {\mab N}$. Set $E:=E_0$. 
For any object $T$ of ${\rm Enl}^{\sq}(S/{\cal V})$,    
$$(R^qf_{T*}(P_kH_{\rm zar}(X_{\os{\circ}{T}_1}/S(T)^{\nat},E(\os{\circ}{T}))),P)$$ 
is a filteredly flat ${\cal K}_T$-module. 
In particular, the filtered sheaf 
\begin{align*}  
(R^qf_{T*}(\eps^*_{X_{\os{\circ}{T}_1}/S(T)^{\nat}}(E(\os{\circ}{T})))_{\mab Q},P) 
\tag{14.8.1}\label{ali:bpce}
\end{align*} 
is a filteredly flat ${\cal K}_T$-module. 
\end{coro}

\begin{exem}\label{exam:ofl} 
Let the notations be as before (\ref{theo:pwfaec}) (1). 
Then there exists an object
$$(R^qf_*(P_kH_{\rm zar}(X/K))^{\sq},P)$$  
of $F{\textrm -}{\rm IsocF}^{\sq}(S/{\cal V})$
such that 
\begin{align*} 
(R^qf_*(P_kH_{\rm zar}(X/K))^{\sq},P)_T
=(R^qf_*(P_kH_{\rm zar}(X_{\os{\circ}{T}_1}/S(T)^{\nat}),P)\tag{14.9.1}\label{ali:pcoe}
\end{align*} 
for any object $T$ of ${\rm Enl}^{\sq}_p(S/{\cal V})$.   
Indeed, the assumption in (\ref{theo:pwfaec})  
is satisfied by the base change of \cite[(1.3)]{boi} 
(cf.~the proof of \cite[(3.7)]{od}).
In particular,  there exists an object 
\begin{align*} 
(R^qf_*({\cal O}_{X/K})^{\nat,\sq},P)
\tag{14.9.2}\label{ali:pcxnee} 
\end{align*}  
of $F{\textrm -}{\rm IsocF}^{\sq}(S/{\cal V})$ such that 
\begin{align*} 
(R^qf_*({\cal O}_{X/K})^{{\nat},{\sq}},P)_T=
(R^qf_{X_{\os{\circ}{T}_1}/S(T)^{\nat}*}
({\cal O}_{X_{\os{\circ}{T}_1}/S(T)^{\nat}})_{\mab Q},P) 
\tag{14.9.3}\label{ali:pctoe}
\end{align*} 
for any object $T$ of ${\rm Enl}_p^{\sq}(S/{\cal V})$. 
\end{exem}

\bigskip
\parno
{\bf (2) Functoriality}
\bigskip
\begin{prop}\label{prop:fcuccv}
$(1)$ 
Let $g$ be as in {\rm (\ref{theo:ccm})}. 
Let the notations and the assumption be as in 
{\rm (\ref{coro:fenlt})}.  
Then the log $p$-adically convergent isocrystal 
$P_kR^qf_{X/K*}
(\eps^*_{X/K}(E_{K}))$ 
is contravariantly functorial.  
\end{prop} 
\begin{proof} 
This is obvious. 
\end{proof}


\begin{prop}\label{prop:spcnvuc}
Let the notations and the assumption be as in {\rm (\ref{coro:fenlt})}. 
Let ${\cal V}'/{\cal V}$ be a finite extension.   
Let $S'\lo S$ be a morphism of log $p$-adic formal families of log points 
over ${\rm Spf}({\cal V}')\lo {\rm Spf}({\cal V})$. 
Set $K':={\rm Frac}({\cal V}')$. 
Let $T'$ and $T$ be log $(p$-adic$)$ enlargements of $S'$ and $S$, respectively. 
Let $T'\lo T$ be a morphism of log $(p$-adic$)$ enlargements over 
$S'\lo S$. Let $u\col S'(T')^{\nat}\lo S(T)^{\nat}$ be the induced morphism.  
Let $Y$ be a log scheme over $S'$ 
which is similar to $X$ over $S$. 
Let $F$ be a similar $F$-isocrystal of 
$\{{\cal O}_{\os{\circ}{Y}{}_{\os{\circ}{T}{}'_1}
/\os{\circ}{T}{}'}\}_{T'\in {\cal E}^{\sq}_{p,{\cal V}}}$-modules to $E$. 
Let $k$ and $q$ be nonnegative integers. 
Let 
$R^qf_{\os{\circ}{X}{}^{(k)}/K*}
(E_{\os{\circ}{X}{}^{(k)}/K} 
\otimes_{\mab Z}\varpi^{(k)}(\os{\circ}{X}/K))$ be an object of 
$F{\textrm -}{\rm Isoc}^{\sq}_{\star}(S/{\cal V})$  
such that
$$R^qf_{\os{\circ}{X}{}^{(k)}/K*}(E_{\os{\circ}{X}{}^{(k)}/K} 
\otimes_{\mab Z}\varpi^{(k)}(\os{\circ}{X}/K))_T= 
R^qf_{\os{\circ}{X}{}^{(k)}_{T_1}
/\os{\circ}{T}}
(E(T)_{\os{\circ}{X}{}^{(k)}_{T_1}
/\os{\circ}{T}} \otimes_{\mab Z}\varpi^{(k)}_{\rm crys}(\os{\circ}{X}_{T_1}/\os{\circ}{T}))
_{\mab Q}$$  
for any object $T$ of ${\rm Enl}^{\sq}_p(S/{\cal V})$. 
Then 
there exists the following spectral sequence 
in $F{\textrm -}{\rm Isoc}^{\sq}(S/{\cal V}):$
\begin{equation*} 
\begin{split} 
{} & 
E_1^{-k,q+k}=
\bigoplus_{m\geq 0}
\bigoplus_{\# \ul{\lam}=m+1}
\bigoplus_{j\geq 0}
\bigoplus_{\# \ul{\mu}=k+m-2j}
R^{q+2j-k-2m}
f_{\os{\circ}{X}_{\ul{\lam}\cup \ul{\mu}}/K*} 
(E\vert_{\os{\circ}{X}_{\ul{\lam}\cup \ul{\mu}}/K}\\
{} & \phantom{R^{q-2j-k-m}f_{(\os{\circ}{X}^{(k)}, 
Z\vert_{\os{\circ}{X}^{(2j+k)}})/K*} 
({\cal O}}\otimes_{\mab Z}\varpi^{(2j+k+m)}(\os{\circ}{X}/K))(-(k+m-j)) \\
&\Lo 
R^qf_{X/K*}
(\eps^*_{X/K}(E_K))
\quad (q\in {\mab Z}). 
\end{split} 
\tag{14.11.1}\label{eqn:getcpsp}
\end{equation*}  
\end{prop}

\begin{defi}
We call (\ref{eqn:getcpsp}) 
the {\it Poincar\'{e} spectral sequence} of 
$$R^qf_{X/K*}(\eps^*_{X/K}(E_K))$$ 
in ${\rm Isoc}_p^{\sq}(S/{\cal V})$ and $F{\textrm -}{\rm Isoc}^{\sq}(S/{\cal V})$, 
respectively.
\end{defi}


\bigskip
\parno
{\bf (3) Monodromy and the cup product of a line bundle}
\bigskip
\parno
Let the notations and the assumptions be as in {\rm (\ref{coro:fenlt})}. 

\begin{prop}\label{prop:convmon} 
$(1)$ There exists the monodromy operator 
\begin{align*} 
N_{\rm zar} \col & (R^qf_{X/K*}(\eps^*_{X/K}(E_K)),P)
\lo (R^qf_{X/K*}(\eps^*_{X/K}(E_K)),P\langle -2 \rangle)(-1) 
\end{align*} 
in ${\rm IsocF}_p^{\sq}(S/{\cal V})$. 
\par 
$(2)$ 
There exists the monodromy operator  
\begin{align*} 
N_{\rm zar} \col & 
(R^qf_{X/S*}
(\eps^*_{X/K}(E_K)),P)
\lo (R^qf_{X/S*}(\eps^*_{X/K}(E_K)),P\langle -2 \rangle)(-1) 
\end{align*} 
in $F{\textrm -}{\rm IsocF}^{\sq}(S/{\cal V})$.
\end{prop}

\parno 
The log convergent version of (\ref{conj:lhilc}) is the following:

\begin{conj}[{\bf Filtered log convergent log hard Lefschetz conjecture}]
\label{conj:clhvlc} 
The following cup product 
\begin{equation*} 
\eta^i \col R^{d-i}f_*({\cal O}_{X/K})^{\nat,\sq}\lo R^{d+i}f_*({\cal O}_{X/K})^{\nat,\sq}(i) \quad (i\in {\mab N})
\tag{14.14.1}\label{eqn:fcicpl} 
\end{equation*}
is an isomorphism in $F{\textrm -}{\rm Isoc}^{\sq}(S/{\cal V})$.   
In fact, $\eta^i$ is the following isomorphism of filtered sheaves: 
\begin{equation*} 
\eta^i \col 
(R^{d-i}f_*({\cal O}_{X/K})^{\nat,\sq},P) 
\os{\sim}{\lo} (R^{d+i}f_*({\cal O}_{X/K})^{\nat,\sq}(i),P). 
\tag{14.14.2}\label{eqn:filqiopl} 
\end{equation*}
Here $P_k(R^{d+i}f_*({\cal O}_{X/K})^{\nat,\sq}(i)):=P_{k+2i}R^{d+i}f_*({\cal O}_{X/K})^{\nat,\sq}$. 
\end{conj}

\begin{theo}[{\bf Filtered log Berthelot-Ogus isomorphism}]\label{theo:bofis}
Let the notations and the assumptions be as in {\rm (\ref{theo:pwfaec})}.  
Let $T$ be an object of ${\rm Enl}^{\sq}(S/{\cal V})$. 
Let $T_0\lo S$ be the structural morphism.  
Let $f_0 \col X_{\os{\circ}{T}_0}\lo T_0$ 
be the structural morphism.  
If there exists an SNCL lift 
$f_1 \col (X)_1 \lo T_1$ of $f_0$, 
then there exists the following canonical 
filtered isomorphism
\begin{equation*} 
(R^qf_*({\cal O}_{X/K})^{\nat}_T,P)
\os{\sim}{\lo}
(R^qf_{X_{\os{\circ}{T}_1}/S(T)^{\nat}*} 
({\cal O}_{X_{\os{\circ}{T}_1}/S(T)^{\nat}}),P).
\tag{14.15.1}\label{eqn:xntp}
\end{equation*} 
\end{theo}
\begin{proof} 
This follows from the proof of \cite[(3.8)]{od} and that of (\ref{theo:pwfaec}). 
\end{proof}

\section{Strict compatibility}\label{sec:st}
In this section we prove the strict compatibility of the pull-back of a morphism 
with respect to the weight filtration. 
Because the proofs of the results are the same as those of \cite[(5.4.6)]{nb} and 
\cite[(5.4.7)]{nb}, we omit the proofs.  

\begin{theo}[{\bf Strict compatibility I}]\label{theo:stpfgb}
Let the notations be as in {\rm (\ref{theo:e2dam})} and the proof of it. 
Let $Y/s'$ be an analogous object to $X/s$. 
Let $f' \col Y\lo s'$ be the structural morphism. 
Let $h\col s\lo s'$ be a morphism of log schemes. 
Let $g\col 
X_{\os{\circ}{T}_0}\lo Y_{\os{\circ}{T}{}'_0}$ 
be the morphism in {\rm (\ref{eqn:xdxduss})} for the case $S=s$ and $S'=s'$ 
satisfying the condition {\rm (8.1.6)}. 
Let ${\cal W}'$ be the Witt ring of $\Gam(s',{\cal O}_{s'})$ 
and set $K'_0:={\rm Frac}({\cal W}')$. 
Assume that $\os{\circ}{s}\lo \os{\circ}{s}{}'$ is finite. 
Let $q$ be a nonnegaitve integer. 
Let us endow 
$H^q_{{\rm crys}}(Y/{\cal W}(s'))\otimes_{{\cal W}'}K_0$  
with the induced filtration by $P$ on 
$H^q_{{\rm crys}}(Y/{\cal W}(s'))_{K'_0}$.   
Then the induced morphism 
\begin{equation*}
g^* \col H^q_{{\rm crys}}(Y/{\cal W}(s'))\otimes_{{\cal W}'}K_0
\lo 
H^q_{{\rm crys}}(X/{\cal W}(s))\otimes_{\cal W}K_0
\tag{15.1.1}\label{eqn:gbwstn}
\end{equation*} 
is strictly compatible with the weight filtration. 
\end{theo}

\begin{theo}[{\bf Strict compatibility II}]\label{theo:stpbgb}
Let the notations and the assumption be as in {\rm (\ref{theo:e2dgfam})} 
and {\rm (\ref{eqn:xdxduss})}. 
Let $Y/S'$ and $T'$ be analogous objects to $X/S$ 
and $T$, respectively. 
Let $g$ be the morphism in {\rm (\ref{eqn:xdxduss})} satisfying the condition {\rm (8.1.6)}. 
Let $q$ be a nonnegative integer. 
Then the induced morphism 
\begin{equation*}
g^* \col v^*(R^qf'_{Y_{\os{\circ}{T}{}'_1}/S'(T')^{\nat}*}
({\cal O}_{Y_{\os{\circ}{T}{}'_1}/S'(T')^{\nat}}))_{\mab Q}
\lo R^qf_{X_{\os{\circ}{T}_1}/S(T)^{\nat}*}
({\cal O}_{X_{\os{\circ}{T}{}_1}/S(T)^{\nat}})_{\mab Q}
\tag{15.2.1}\label{eqn:gvbstn}
\end{equation*}  
is strictly compatible with the weight filtration.
Consequently the induced morphism 
\begin{equation*}
g^* \col v^*(R^qf'_*({\cal O}_{Y/K'})^{\nat,\sq})
\lo R^qf_*({\cal O}_{X/K})^{\nat,\sq}
\tag{15.2.2}\label{eqn:gbstn}
\end{equation*} 
in $F{\textrm -}{\rm Isoc}^{\sq}(S/{\cal V})$ 
is strictly compatible with the weight filtration.
\end{theo}

\section{$p$-adic filtered Steenbrink complexes and trace morphisms}\label{sec:pssc}
Let the notations be as in \S\ref{sec:e2}. 
In this section we recall the filtered complex 
$$(A_{\rm zar}(X_{\os{\circ}{T}_0}/S(T)^{\nat},E),P)$$ 
defined in \cite{nb} and we define the following trace morphism 
\begin{equation*}  
R^{2d}f_{X_{\os{\circ}{T}_1}/S(T)^{\nat}*}
({\cal O}_{X_{\os{\circ}{T}_1}/S(T)^{\nat}})_{\mab Q}
\lo {\cal K}_{\os{\circ}{T}}(-d) 
\end{equation*}
when $\os{\circ}{X}_{T_0}$ is projective over $\os{\circ}{T}_0$ and 
the relative dimension of $\os{\circ}{X}_{T_0}/\os{\circ}{T}_0$ is of pure dimension $d$,  
where $T$ is an object of ${\rm Enl}_p(S/{\cal V})$.   
\par 
In \cite{fup}, to define the trace morphism, Fujisawa has used 
the comparison between the Steenbrink (bi)filtered complex and 
the (bi)filtered complex obtained by Hirsch extension
(=an analogue of $(H_{\rm zar}(X_{\os{\circ}{T}_0}/S(T)^{\nat}),P)$ with 
Hodge filtration). 
We do not use the comparison between 
$(A_{\rm zar}(X_{\os{\circ}{T}_0}/S(T)^{\nat}),P)$ and 
$(H_{\rm zar}(X_{\os{\circ}{T}_0}/S(T)^{\nat}),P)$
to define our trace morphism; our proof of the existence of the trace morphism 
is much simpler than that in [loc.~cit]; our method is also 
applicable in the complex analytic case. 
\par

\par 
Let the notations be as in \S\ref{sec:mod}. 
Set  
\begin{align*} 
A_{\rm zar}({\cal P}^{\rm ex}_{\bul}/S(T)^{\nat},{\cal E}^{\bul})^{ij}
& :=({\cal E}^{\bul}\otimes_{{\cal O}_{{\cal P}^{\rm ex}}}
{\Om}^{i+j+1}_{{\cal P}^{\rm ex}_{\bul}/\os{\circ}{T}})/P_j 
\tag{16.0.1}\label{cd:accef} \\
& :=({\cal E}^{\bul}
\otimes_{{\cal O}_{{\cal P}^{\rm ex}_{\bul}}}
{\Om}^{i+j+1}_{{\cal P}^{\rm ex}_{\bul}/\os{\circ}{T}})/
P_j({\cal E}^{\bul}
\otimes_{{\cal O}_{{\cal P}^{\rm ex}_{\bul}}}
{\Om}^{i+j+1}_{{\cal P}^{\rm ex}_{\bul}/\os{\circ}{T}})  
\quad (i,j \in {\mab N}). 
\end{align*}   
The sheaf 
$A_{\rm zar}({\cal P}^{\rm ex}_{\bul}/S(T)^{\nat},{\cal E})^{ij}$ 
has a quotient filtration $P$ obtained by the 
filtration $P$ on 
${\cal E}^{\bul}\otimes_{{\cal O}_{{\cal P}^{\rm ex}_{\bul}}}
{\Om}^{i+j+1}_{{\cal P}^{\rm ex}_{\bul}/\os{\circ}{T}}$. 
We consider the following boundary morphisms of double complexes: 
\begin{equation*}
\begin{CD}
A_{\rm zar}({\cal P}^{\rm ex}_{\bul}/S(T)^{\nat},
{\cal E}^{\bul})^{i,j+1}  @.  \\ 
@A{d\log t\wedge}AA  @. \\
A_{\rm zar}({\cal P}^{\rm ex}_{\bul}/S(T)^{\nat},
{\cal E}^{\bul})^{ij}
@>{-\nabla}>> 
A_{\rm zar}({\cal P}^{\rm ex}_{\bul}/S(T)^{\nat},
{\cal E}^{\bul})^{i+1,j}\\
\end{CD}
\tag{16.0.2}\label{cd:lccbd} 
\end{equation*}  
as in \cite[(1.4.0.9)]{nb}. 
Then we have the double complex 
$A_{\rm zar}({\cal P}^{\rm ex}_{\bul}/S(T)^{\nat},
{\cal E}^{\bul})^{\bul \bul}$. 
Let 
$A_{\rm zar}({\cal P}^{\rm ex}_{\bul}/S(T)^{\nat},
{\cal E}^{\bul})$ 
be the single complex of  
$A_{\rm zar}({\cal P}^{\rm ex}_{\bul}/S(T)^{\nat},
{\cal E}^{\bul})^{\bul \bul}$: 
\begin{align*} 
A_{\rm zar}({\cal P}^{\rm ex}_{\bul}/S(T)^{\nat},
{\cal E}^{\bul})^i
=\bigoplus_{j\geq 0}
({\cal E}^{\bul}
\otimes_{{\cal O}_{{\cal P}^{\rm ex}_{\bul}}}
{\Om}^{i+1}_{{\cal P}^{\rm ex}_{\bul}/\os{\circ}{T}})/
P_j({\cal E}^{\bul}
\otimes_{{\cal O}_{{\cal P}^{\rm ex}_{\bul}}}
{\Om}^{i+1}_{{\cal P}^{\rm ex}_{\bul}/\os{\circ}{T}}) 
\end{align*} 
with boundary morphism $-\nabla+d\log t\wedge$. 
The double complex 
$A_{\rm zar}({\cal P}^{\rm ex}_{\bul}/S(T)^{\nat},
{\cal E}^{\bul})^{\bul \bul}$ 
has a filtration $P=\{P_k\}_{k \in {\mab Z}}$ 
defined by the following formula: 
\begin{equation*} 
P_kA_{\rm zar}({\cal P}^{\rm ex}_{\bul}/S(T)^{\nat},
{\cal E}^{\bul})
:=(\cdots 
P_{2j+k+1}A_{\rm zar}({\cal P}^{\rm ex}_{\bul}/S(T)^{\nat},
{\cal E}^{\bul})^i
\cdots)\in {\rm C}^+{\rm F}(f^{-1}_T({\cal O}_T)).    
\tag{16.0.3}\label{eqn:lpcad}
\end{equation*} 
Let $(A_{\rm zar}({\cal P}^{\rm ex}_{\bul}/S(T)^{\nat},
{\cal E}^{\bul}),P)$ 
be the filtered single complex of the filtered double complex 
$(A_{\rm zar}({\cal P}^{\rm ex}_{\bul}/S(T)^{\nat},
{\cal E}^{\bul})^{\bul \bul},P)$.

\par 
In \cite[(1.4.4)]{nb} we have proved the following as a special case:

\begin{theo}[{\bf \cite[(1.4.4)]{nb}}]\label{theo:indcr} 
The filtered complex 
$R\pi_{{\rm zar}*}
((A_{\rm zar}({\cal P}^{\rm ex}_{\bul}/S(T)^{\nat},{\cal E}^{\bul}),P))$ 
is independent of the choice of the disjoint union 
of an affine simplicial open covering of $X$ and 
a simplicial immersion 
$X_{\bul} \os{\sus}{\lo} \ol{\cal P}_{\bul}$ over $\ol{S(T)^{\nat}}$.  
\end{theo}

Set 
\begin{equation*} 
(A_{\rm zar}(X_{\os{\circ}{T}_0}/S(T)^{\nat},E),P)
:=R\pi_{{\rm zar}*}
((A_{\rm zar}({\cal P}^{\rm ex}_{\bul}/S(T)^{\nat},{\cal E}^{\bul}),P))
\tag{16.1.1}\label{eqn:axtt}
\end{equation*} 
and we have called the filtered complex 
$(A_{\rm zar}(X_{\os{\circ}{T}_0}/S(T)^{\nat},E),P)
\in {\rm D}^+{\rm F}(f^{-1}_T({\cal O}_T))$ 
the $p$-adic zariskian filtered Steenbrink complex of 
$E^{}$ on $X_{\os{\circ}{T}_0}/(S(T)^{\nat},{\cal J},\del)$. 
In the proof of \cite[(1.4.7)]{nb} we have obtained the following: 

\begin{prop}[{\bf \cite[(1.4.7)]{nb}}]\label{prop:graxnp}
There exists the following canonical isomorphism$:$ 
\begin{align*} 
{\rm gr}^P_kA_{\rm zar}
(X_{\os{\circ}{T}_0}/S(T)^{\nat},E))
\os{\sim}{\lo} \bigoplus_{j\geq \max \{-k,0\}} 
&a^{(2j+k)}_{T_0*} 
(Ru_{\os{\circ}{X}{}^{(2j+k)}_{T_0}/\os{\circ}{T}*}
(E_{\os{\circ}{X}{}^{(2j+k)}_{T_0}
/\os{\circ}{T}} \tag{16.2.1}\label{ali:ruovp}\\
&\otimes_{\mab Z}\vp_{\rm crys}^{(2j+k)}
(\os{\circ}{X}_{T_0}/\os{\circ}{T})))[-2j-k]
\end{align*}
in $D^+(f^{-1}({\cal O}_T))$. 
Here $\vp_{\rm crys}^{(\star)}
(\os{\circ}{X}_{T_0}/\os{\circ}{T})$ is the crystalline orientation sheaf 
associated to 
the set $\{\os{\circ}{X}_{\lam}\}_{\lam \in \Lam}$. 
\end{prop}

In \cite[(1.4.3)]{nb}  we have also proved the following as a special case: 

\begin{prop}[{\bf \cite[(1.4.3)]{nb}}]\label{prop:tefc}
There exists the following canonical isomorphism 
\begin{equation*} 
\theta \wedge :=\theta_{X_{\os{\circ}{T}_0}/S(T)^{\nat}} \wedge \col 
Ru_{X_{\os{\circ}{T}_0}/S(T)^{\nat}*}
(\eps^*_{X_{\os{\circ}{T}_0}/S(T)^{\nat}}(E))\os{\sim}{\lo} 
A_{\rm zar}(X_{\os{\circ}{T}_0}/S(T)^{\nat},E)
\tag{16.3.1}\label{eqn:uz} 
\end{equation*} 
in $D^+(f^{-1}({\cal O}_T))$. 
\end{prop}

\par 
We recall the morphism (\ref{eqn:uz}). This morphism is, by definition, 
the induced morphism by the following morphism: 
\begin{align*} 
d\log t\wedge  \col {\cal E}^{\bul}
\otimes_{{\cal O}_{{\cal P}^{\rm ex}_{\bul}}}
{\Om}^{i}_{{\cal P}^{\rm ex}_{\bul}/S(T)^{\nat}}\lo 
({\cal E}^{\bul}
\otimes_{{\cal O}_{{\cal P}^{\rm ex}_{\bul}}}
{\Om}^{i+1}_{{\cal P}^{\rm ex}_{\bul}/\os{\circ}{T}})
/
P_j({\cal E}^{\bul}
\otimes_{{\cal O}_{{\cal P}^{\rm ex}_{\bul}}}
{\Om}^{i+1}_{{\cal P}^{\rm ex}_{\bul}/\os{\circ}{T}})  \quad (i\in {\mab N}). 
\end{align*} 
Since $d\log t$ is a one form, this morphism indeed 
induces the morphism of complexes:
\begin{align*} 
{\cal E}^{\bul}
\otimes_{{\cal O}_{{\cal P}^{\rm ex}_{\bul}}}
{\Om}^{\bul}_{{\cal P}^{\rm ex}_{\bul}/S(T)^{\nat}}
\lo A_{\rm zar}({\cal P}^{\rm ex}_{\bul}/S(T)^{\nat},
{\cal E}^{\bul}). 
\end{align*}
Hence, by the log Poincar\'{e} lemma, we obtain the morphism (\ref{eqn:uz}).  
In \cite[(1.4.3)]{nb} we have also proved that the morphism (\ref{eqn:uz})
is independent of the choices of the affine open covering of $X_{\os{\circ}{T}_0}$ and 
the simplicial immersion $X_{\os{\circ}{T}_0\bul}\os{\sus}{\lo}\ol{\cal P}$ over 
$S(T)^{\nat}$. 
\par
By (\ref{ali:ruovp}) and (\ref{eqn:uz}), 
we obtain the following spectral sequence 
\begin{align*} 
E_1^{-k,q+k}=
\bigoplus_{j\geq \max \{-k,0\}} &
R^{q-2j-k}f_{\os{\circ}{X}{}^{(2j+k)}_{T_0}/\os{\circ}{T}*}
(E\vert_{\os{\circ}{X}{}^{(2j+k)}_{T_0}/\os{\circ}{T}}
\otimes_{\mab Z}
\vp^{(2j+k)}_{\rm crys}(\os{\circ}{X}_T/\os{\circ}{T})) 
\tag{16.3.2}\label{eqn:espsp} \\
&  \Lo 
R^qf_{X_{\os{\circ}{T}_0}/S(T)^{\nat}*}(\eps^*_{X_{\os{\circ}{T}_0}/S(T)^{\nat}}
(E)) 
\quad (q\in {\mab Z}). 
\end{align*}   
More precisely, if one consider the action of the abrelative Frobenius morphism of 
$S(T)^{\nat}$ on this spectral sequence as in \cite{nb}, then this sequence 
becomes 
\begin{align*} 
E_1^{-k,q+k}=
\bigoplus_{j\geq \max \{-k,0\}} &
R^{q-2j-k}f_{\os{\circ}{X}{}^{(2j+k)}_{T_0}/\os{\circ}{T}*}
(E\vert_{\os{\circ}{X}{}^{(2j+k)}_{T_0}/\os{\circ}{T}}
\otimes_{\mab Z}
\vp^{(2j+k)}_{\rm crys}(\os{\circ}{X}_T/\os{\circ}{T}))(-j-k) 
\tag{16.3.3}\label{eqn:esafsp} \\
&  \Lo 
R^qf_{X_{\os{\circ}{T}_0}/S(T)^{\nat}*}
(\eps^*_{X_{\os{\circ}{T}_0}/S(T)^{\nat}}(E)) \quad (q\in {\mab Z}) 
\end{align*}   
when $\os{\circ}{S}$ is a $p$-adic flat formal scheme over ${\rm Spf}({\mab Z}_p)$. 

\par 
Let $\{\os{\circ}{X}_{\lam}\}_{\lam \in \Lam}$ be 
the set of the smooth components 
of $\os{\circ}{X}_{T_0}$: 
$\os{\circ}{X}_{T_0}=\bigcup_{\lam \in \Lam}\os{\circ}{X}_{\lam}$; 
$\os{\circ}{X}_{\lam}$ is smooth over 
$\os{\circ}{S}_{\os{\circ}{T}_0}=\os{\circ}{T}_0$. 
Let $\os{\circ}{\cal P}{}^{\rm ex}_{\lam}$ 
be the smooth closed subscheme of 
$\os{\circ}{\cal P}{}^{\rm ex}$ over $\os{\circ}{T}$ 
which is topologically isomorphic to $\os{\circ}{X}_{\lam}$. 
Fix a total order on $\Lam$ once and for all. 
For different $\lam_0<\lam_1<\ldots< \lam_{k+1}$ and 
a nonnegative integer  $0\leq j\leq k+1$, 
set $\ul{\lam}:=\{\lam_0\cdots \lam_{k+1}\}$ 
and 
$\ul{\lam}_j:=\ul{\lam}\setminus \{\lam_j\}$,  
set $\os{\circ}{X}_{\ul{\lam}}
:=\os{\circ}{X}_{\lam_0}\cap \cdots \cap \os{\circ}{X}_{\lam_{k+1}}$ 
and $\os{\circ}{\cal P}{}^{\rm ex}_{\ul{\lam}}
:=\os{\circ}{\cal P}{}^{\rm ex}_{\lam_0}\cap 
\cdots \cap \os{\circ}{\cal P}{}^{\rm ex}_{\lam_{k+1}}$. 
Let $\os{\circ}{\mathfrak D}_{\ul{\lam}}$ be 
the PD-envelope of the immersion 
$\os{\circ}{X}_{\ul{\lam}} \os{\sus}{\lo} 
\os{\circ}{\cal P}{}^{\rm ex}_{\ul{\lam}}$ 
over $(\os{\circ}{T},{\cal J},\del)$ and let 
$c_{\ul{\lam}} \col 
\os{\circ}{\mathfrak D}_{\ul{\lam}} \lo \os{\circ}{\mathfrak D}$ 
be the natural morphism of schemes. 
Let  
\begin{equation*} 
\iota_{\ul{\lam}_j,\ul{\lam}}\col \os{\circ}{X}_{\ul{\lam}} 
\os{\sus}{\lo} \os{\circ}{X}_{\ul{\lam}_j}
\end{equation*} 
be the natural closed immersion.  
Let 
\begin{align*} 
&\iota^{(k)*} \col   
c^{(k)}_*
(c^{(k)*}({\cal E})
\otimes_{{\cal O}_{\os{\circ}{\cal P}{}^{{\rm ex},(k)}}}
\Om^{\bul}_{\os{\circ}{\cal P}{}^{{\rm ex},(k)}/\os{\circ}{T}}
\otimes_{\mab Z}
\vp^{(k)}_{\rm zar}
(\os{\circ}{\cal P}{}^{\rm ex}/\os{\circ}{T})) 
\tag{16.3.4}\label{eqn:odpaps}\\
& \lo 
c^{(k+1)}_*
(c^{(k+1)*}({\cal E})
\otimes_{{\cal O}_{\os{\circ}{\cal P}{}^{{\rm ex},(k+1)}}}
\Om^{\bul}_{\os{\circ}{\cal P}{}^{{\rm ex},(k+1)}
/\os{\circ}{T}}
\otimes_{\mab Z}\vp^{(k+1)}_{\rm zar}
(\os{\circ}{\cal P}{}^{\rm ex}/\os{\circ}{T}))  
\end{align*} 
be the natural \v{C}ech morphism, which is the summation of 
the following morphism with respect to 
$\lam_j$ and $0\leq j \leq k+1$ 
(cf.~\cite[(5.0.5)]{ndw}): 
\begin{align*} 
& 
c_{\ul{\lam}_j*}
(c^*_{\ul{\lam}_j}({\cal E}) 
\otimes_{
{\cal O}_{{\cal P}{}^{\rm ex}_{\ul{\lam}_j}}}
\Om^{\bul}_{{\cal P}{}^{\rm ex}_{\ul{\lam}_j}/\os{\circ}{T}}
\otimes_{\mab Z}
\vp_{{\rm zar}{\ul{\lam}_j}}
(\os{\circ}{\cal P}{}^{\rm ex}/\os{\circ}{T}))  
\lo c_{\ul{\lam}*}
(c^*_{\ul{\lam}}({\cal E})\otimes_{
{\cal O}_{{\cal P}{}^{\rm ex}_{\ul{\lam}}}}
\Om^{\bul}_{{\cal P}^{\rm ex}_{\ul{\lam}}/\os{\circ}{T}}
\otimes_{\mab Z}
\vp_{{\rm zar}{\ul{\lam}}}
(\os{\circ}{\cal P}{}^{\rm ex}/\os{\circ}{T})))\tag{16.3.5}\label{eqn:odkps}\\
\end{align*}  
\begin{equation*} 
{\om_{\ul{\lam}_j}}\otimes(\ul{\lam}_j)
\lom  
{(-1)^j\iota^{*}_{\ul{\lam}_j,\ul{\lam}}
(\om_{\ul{\lam}_j})}
\otimes(\ul{\lam}).  
\end{equation*}

\begin{lemm}[{\bf cf.~\cite[4.12]{msemi}, \cite[(10.1.16)]{ndw}, \cite[(1.4.2)]{nb}}]\label{lemm:ti} 
Let $k$ be a positive integer. 
Let 
\begin{align*} 
\iota^{(k-1)*}_n \col  & 
c^{(k-1)}_{n*}
(c^{(k-1)*}_n({\cal E}^{n})
\otimes_{{\cal O}_{{\cal P}^{\rm ex}_n}}
\Om^{\bul}_{\os{\circ}{\cal P}{}^{{\rm ex},(k-1)}_n/\os{\circ}{T}}
\otimes_{\mab Z}\vp^{(k-1)}_{\rm zar}(\os{\circ}{\cal P}{}^{\rm ex}_n/\os{\circ}{T})) 
\lo 
\tag{16.4.1}\label{eqn:odips} \\
& 
c^{(k)}_{n*}
(c^{(k)*}_n({\cal E}^{n})
\otimes_{{\cal O}_{{\cal P}^{\rm ex}_n}}
\Om^{\bul}_{\os{\circ}{\cal P}{}^{{\rm ex},(k)}_n
/\os{\circ}{T}}
\otimes_{\mab Z}
\vp^{(k)}_{\rm zar}(\os{\circ}{\cal P}{}^{\rm ex}_n/\os{\circ}{T}))
\end{align*}
be the morphism {\rm (\ref{eqn:odpaps})}.  
Then the following diagram is commutative$:$
\begin{equation*} 
\begin{CD} 
{\rm gr}^P_{k+1}
({\cal E}^{n}
\otimes_{{\cal O}_{{\cal P}^{\rm ex}_n}}
{\Om}^{i+1}_{{\cal P}^{\rm ex}_n/\os{\circ}{T}})
@>{\simeq}>>  
\\
@A{d\log t\wedge}AA  \\
{\rm gr}^P_k
({\cal E}^{n}
\otimes_{{\cal O}_{{\cal P}^{\rm ex}_n}}
{\Om}^{i}_{{\cal P}^{\rm ex}_n/\os{\circ}{T}})
@>{\simeq}>> 
\end{CD}
\tag{16.4.2}\label{eqn:xdxgras}
\end{equation*} 
\begin{equation*} 
\begin{CD} 
(c^{(k)}_{n*}
(c^{(k)*}_n
({\cal E}^{n})\otimes_{
{\cal O}_{\os{\circ}{\cal P}{}^{{\rm ex},(k)}_n}}
\Om^{i-k}_{\os{\circ}{\cal P}{}^{{\rm ex},(k)}_n
/\os{\circ}{T}} \otimes_{\mab Z}
\vp^{(k)}_{\rm zar}
(\os{\circ}{\cal P}{}^{\rm ex}_n/\os{\circ}{T}))[-k-1]) 
\\
@AA{\iota^{(k-1)*}_n}A \\
(c^{(k-1)}_{n*}
(c^{(k-1)*}_n({\cal E}^{n})\otimes_{
{\cal O}_{\os{\circ}{\cal P}{}^{{\rm ex},(k-1)}_n}}
\Om^{i-k}_{\os{\circ}{\cal P}{}^{{\rm ex},(k-1)}_n
/\os{\circ}{T}} \otimes_{\mab Z}
\vp^{(k-1)}_{\rm zar}
(\os{\circ}{\cal P}{}^{\rm ex}_n/\os{\circ}{T}))[-k]). 
\end{CD}
\end{equation*} 
Here the upper horizontal isomorphism and the lower horizontal isomorphism are 
the morphisms {\rm (\ref{eqn:prvin})} for $k+1$ and $k$, respectively. 
\end{lemm}

\par 
The morphism $\iota_{\ul{\lam}_j,\ul{\lam}}\col 
\os{\circ}{X}_{\ul{\lam}} 
\os{\sus}{\lo} \os{\circ}{X}_{\ul{\lam}_{j}}$ 
induces the morphism  
\begin{equation}
(-1)^{j}
\iota_{\ul{\lam}_{j},\ul{\lam}{\rm crys}}^{*}
\col 
\iota_{\ul{\lam}_{j},\ul{\lam}_{\rm crys}}^{*}(E_{\ul{\lam}_{j}}
\otimes_{\mab Z}\vp_{\ul{\lam}_{j}{\rm crys}}
(\os{\circ}{X}_{T_0}/\os{\circ}{T})) 
\lo E_{\ul{\lam}}\otimes_{\mab Z}\vp_{\ul{\lam}{\rm crys}}
(\os{\circ}{X}_{T_0}/\os{\circ}{T}) 
\tag{16.4.3}\label{eqn:defcbd}
\end{equation}
as in \cite[(2.11.1.2)]{nh2}.
Here 
$$\vp_{\ul{\lam}{\rm crys}}
(\os{\circ}{X}_{T_0}/\os{\circ}{T}) 
\quad \text{and} \quad  \vp_{\ul{\lam}_{j}{\rm crys}}
(\os{\circ}{X}_{T_0}/\os{\circ}{T})$$  
are the crystalline orientation sheaves of
$\os{\circ}{X}_{\ul{\lam}}$ and 
$\os{\circ}{X}_{\ul{\lam}_j}$ 
in 
$(\os{\circ}{X}_{\ul{\lam}}
/\os{\circ}{T})_{\rm crys}$ and $(\os{\circ}{X}_{\ul{\lam}_{j}}
/\os{\circ}{T})_{\rm crys}$, 
respectively,  
defined similarly in \cite[p.~81, (2.8)]{nh2}.
Set 
\begin{align*}
& \rho:=\bigoplus_{j\geq \max \{-k,0\}}
\sum_{\{\ul{\lam}:=\{\lam_{0},\ldots,\lam_{2j+k}\}
~\vert~\lam_l<\lam_k~(l<k)\}}
\sum_{\bet=0}^{2j+k}(-1)^{\bet}
\iota_{\ul{\lam}_j,\ul{\lam}{\rm crys}}^{*} 
\col 
\tag{16.4.4}\label{eqn:rhogsn}\\
&\bigoplus_{j\geq \max \{-k,0\}} 
R^{q-2j-k}f_{\os{\circ}{X}{}^{(2j+k)}_{T_0}/\os{\circ}{T}*}
(E^m_{\os{\circ}{X}{}^{(2j+k)}_{T_0}
/\os{\circ}{T}}\otimes_{\mab Z}\vp^{(2j+k)}_{\rm crys}
(\os{\circ}{X}_{T_0}/\os{\circ}{T}))  \\ 
&(-j-k,v)\lo \\ 
& \bigoplus_{j\geq \max \{-k,0\}} 
R^{q-2j-k}f_{\os{\circ}{X}{}^{(2j+k+1)}_{T_0}/\os{\circ}{T}}
(E_{\os{\circ}{X}{}^{(2j+k+1)}_{T_0}
/\os{\circ}{T}}
\otimes_{\mab Z}\vp^{(2j+k+1)}_{\rm crys}
(\os{\circ}{X}_{T_0}/\os{\circ}{T}))\\
& (-j-k,v).
\end{align*}
\par
For a nonnegative integer $q$ and an integer $l$, 
let
\begin{equation*}
(-1)^jG^{\ul{\lam}_j}_{\ul{\lam}}
\col 
R^qf_{\os{\circ}{X}_{\ul{\lam}_j}/\os{\circ}{T}*}(E_{\ul{\lam}_j}
\otimes_{\mab Z} 
\vp_{\rm crys}(\os{\circ}{X}_{\ul{\lam}_j}/\os{\circ}{T}))(-l) 
\tag{16.4.5}\label{eqn:egs}
\end{equation*}
\begin{equation*}
\lo 
R^{q+2}f_{\os{\circ}{X}_{\ul{\lam}_j}
/\os{\circ}{T}*}(E_{\ul{\lam}_j}\otimes_{\mab Z} \vp_{\rm crys}
(\os{\circ}{X}_{\ul{\lam}_j}
/\os{\circ}{T}))(-(l-1))
\end{equation*}
be the obvious sheafied version of the Gysin morphism 
defined in \cite[(2.8.4.5)]{nh2}. 
Set
\begin{align*}
& G:=\bigoplus_{j\geq \max \{-k,0\}} \sum_{\{\ul{\lam}:=\{\lam_{0},\ldots,\lam_{2j+k}\}
~\vert~\lam_l<\lam_k~(l<k)\}}
\sum_{\bet=0}^{2j+k}(-1)^{\bet}
G_{\ul{\lam}}^{\ul{\lam}_{\bet}} \col 
\tag{16.4.6}\label{eqn:togbsn}\\
&\bigoplus_{j\geq \max \{-k,0\}} 
R^{q-2j-k}
f_{\os{\circ}{X}{}^{(2j+k)}_{T_0}
/\os{\circ}{T}*}(E_{\os{\circ}{X}{}^{(2j+k)}_{T_0}/\os{\circ}{T}}
\otimes_{\mab Z} 
\vp^{(2j+k)}_{\rm crys}(\os{\circ}{X}_{T_0}/\os{\circ}{T}))(-j-k)\\
&\lo \\
&\bigoplus_{j\geq \max \{-k+1,0\}} 
R^{q-2j-k+2}
f_{\os{\circ}{X}{}^{(2j+k-1)}_{T_0}
/\os{\circ}{T}*}(E_{\os{\circ}{X}{}^{(2j+k-1)}_{T_0}/\os{\circ}{T}}
\otimes_{\mab Z} 
\vp^{(2j+k-1)}_{\rm crys}(\os{\circ}{X}_{T_0}/\os{\circ}{T}))\\
&(-j-k+1).
\end{align*}

The following is a special case of \cite[(1.5.21)]{nb}: 

\begin{prop}[{\bf \cite[(1.5.21)]{nb}}]\label{prop:deccbd} 
The boundary morphism between the $E_1$-terms of 
the spectral sequence {\rm (\ref{eqn:esafsp})} 
is given by the following$:$ 
\begin{equation*} 
\bigoplus_{j\geq \max \{-k,0\}} 
R^{q-2j-k}
f_{\os{\circ}{X}{}^{(2j+k)}_{T_0}
/\os{\circ}{T}*}
(E_{\os{\circ}{X}{}^{(2j+k)}_{T_0}
/\os{\circ}{T}} 
\otimes_{\mab Z}\vp^{(2j+k)}_{\rm crys}
(\os{\circ}{X}_{T_0}/\os{\circ}{T}))(-j-k)
\end{equation*}  
$$\text{
{${G+\rho}$}}
~\downarrow \quad \quad \quad \quad \quad \quad \quad 
\quad \quad \quad \quad \quad$$
\begin{equation*} 
\bigoplus_{j\geq \max \{-k+1,0\}} 
R^{q-2j-k+2}
f_{\os{\circ}{X}{}^{(2j+k-1)}_{T_0}
/\os{\circ}{T}*}
(E_{\os{\circ}{X}{}^{(2j+k-1)}_{T_0}
/\os{\circ}{T}} \otimes_{\mab Z}\vp^{(2j+k-1)}_{\rm crys}
(\os{\circ}{X}_{T_0}/\os{\circ}{T}))(-j-k+1). 
\end{equation*}   
\end{prop}

\begin{prop}\label{prop:pad}
Assume that $\os{\circ}{X}/\os{\circ}{S}$ 
is pure of relative dimension $d$. 
Then the following hold$:$
\par 
$(1)$ If $k\not=0$, then $E_1^{-k,2d+k} =0$. 
\par 
$(2)$ If $k> 1$, then $E_1^{-k,2d-1+k} =0$. 
\par
$(3)$ If $k\leq 0$, then $E_1^{-k,2d+1+k} =0$.
\end{prop}
\begin{proof} 
(1): The relative dimension of 
$\os{\circ}{X}{}^{(2j+k)}_{T_0}/\os{\circ}{T}_0$ is equal to $d-2j-k$.  
Consider the case $q=2d$ in (\ref{eqn:esafsp}) 
and the inequality $j\geq \max \{-k,0\}$. 
We have an easy equivalence: $2d-2j-k\leq 2(d-2j-k)
\Leftrightarrow j\leq -2^{-1}k$. 
If $k< 0$, then $j\geq -k$. There is no $(j,k)$ satisfying the three inequalities. 
Hence $E_1^{-k,2d+k}=0$ for $k<0$. 
If $k> 0$, then there is no $(j,k)$ again. Hence $E_1^{-k,2d+k}=0$ if $k\not=0$.
\par 
(2): Consider the case $q=2d-1$ in (\ref{eqn:esafsp}) and the inequality $j\geq \max \{-k,0\}$. 
We have an easy equivalence: 
$2d-1-2j-k\leq 2(d-2j-k)\Leftrightarrow 2j+k\leq 1$. 
If $k>1$, then there is no $(j,k)$ satisfying the inequalities.  
Hence $E_1^{-k,2d-1+k}=0$. 
\par 
(3): Consider the case $q=2d+1$ in (\ref{eqn:esafsp}) and the inequality $j\geq \max \{-k,0\}$. 
We have an easy equivalence: 
$2d+1-2j-k\leq 2(d-2j-k)\Leftrightarrow 2j+k\leq -1$. 
Assume that $k\leq 0$.  Then $j+k\geq 0$. 
There is no $(j,k)$ satisfying the inequalities. 
Hence $E_1^{-k,2d+1+k}=0$. 
\end{proof}

\begin{coro}\label{coro:eco} 
$(1)$ 
\begin{equation*} 
E_{\infty}^{0,2d}=E_2^{0,2d}=
{\rm Coker}(E_1^{-1,2d}\lo E_1^{0,2d}).
\tag{16.7.1}\label{eqn:eis2d}
\end{equation*}  
\par 
$(2)$ 
\begin{align*}
R^{2d}f_{X_{\os{\circ}{T}_0}/S(T)^{\nat}*}
(\eps^*_{X_{\os{\circ}{T}_0}/S(T)^{\nat}}
(E)) =E_2^{0,2d}.
\tag{16.7.2}\label{eqn:eise2d}
\end{align*} 
\end{coro} 
\begin{proof} 
These follow from (\ref{prop:pad}). 
\end{proof}

\par 
Now let the notations be as in \S\ref{sec:e2}. 
Assume that $X_{\os{\circ}{T}_1}/S(T)^{\nat}$ is of pure dimension $d\geq 0$. 
Next we define a trace morphism 
\begin{equation*}  
R^{2d}f_{X_{\os{\circ}{T}_1}/S(T)^{\nat}*}
({\cal O}_{X_{\os{\circ}{T}_1}/S(T)^{\nat}})_{\mab Q}
\lo {\cal K}_{\os{\circ}{T}}(-d). 
\tag{16.7.3}\label{eqn:radx} 
\end{equation*} 
Consider the following spectral sequence
$:$
\begin{align*} 
E_1^{-k,q+k}=
\bigoplus_{j\geq \max \{-k,0\}} &
R^{q-2j-k}f_{\os{\circ}{X}{}^{(2j+k)}_{T_1}/\os{\circ}{T}*}
({\cal O}_{{\os{\circ}{X}{}^{(2j+k)}_{T_1}/\os{\circ}{T}}}
\otimes_{\mab Z}
\vp^{(2j+k)}_{\rm crys}(\os{\circ}{X}_T/\os{\circ}{T}))(-j-k)
\tag{16.7.4}\label{ali:eopsp}\\
& \Lo R^qf_{X_{\os{\circ}{T}_1}/S(T)^{\nat}*}({\cal O}_{X_{\os{\circ}{T}_1}/S(T)^{\nat}}). 
\end{align*} 
for the case $E={\cal O}_{\os{\circ}{X}/\os{\circ}{T}}$ in 
{\rm (\ref{eqn:espsp})} 
and consider $E_1^{0,2d}$:   
\begin{align*} 
E_1^{0,2d}=&R^{2d}
f_{\os{\circ}{X}{}^{(0)}_{T_1}/T*}
({\cal O}_{\os{\circ}{X}{}^{(0)}_{T_1}/\os{\circ}{T}}
\otimes_{\mab Z}\vp^{(0)}_{\rm crys}
(\os{\circ}{X}_{T_1}/T)).  
\tag{16.7.5}\label{ali:scpdrh}
\end{align*} 
We have the trace morphism 
\begin{align*} 
{\rm Tr}_{\os{\circ}{X}{}^{(0)}_{T_1}/\os{\circ}{T}}
\col R^{2d}f_{\os{\circ}{X}{}^{(0)}_{T_1}/T*} 
({\cal O}_{\os{\circ}{X}{}^{(0)}_{T_1}/\os{\circ}{T}}
\otimes_{\mab Z}\vp_{{\rm crys}}
(\os{\circ}{X}{}^{(0)}_{T_1}/\os{\circ}{T}))_{\mab Q}
\lo {\cal K}_{\os{\circ}{T}}(-d). 
\tag{16.7.6}\label{ali:rdsm}
\end{align*}
by the proof of \cite[(3.12)]{od}. 
This morphism is surjective since 
the trace morphism 
\begin{align*} 
{\rm Tr}_{\os{\circ}{X}_{\lam,T_1}/\os{\circ}{T}}
\col R^{2d}f_{\os{\circ}{X}_{\lam,T_1}/T*} 
({\cal O}_{\os{\circ}{X}_{{\lam},T_1}/\os{\circ}{T}}
\otimes_{\mab Z}\vp_{{\rm crys}}
(\os{\circ}{X}_{{\lam},T_1}/\os{\circ}{T}))_{\mab Q}
\lo {\cal K}_{\os{\circ}{T}}(-d)
\end{align*}
is surjective.

\par 
We prove that ${\rm Tr}_{\os{\circ}{X}{}^{(0)}_{T_1}/\os{\circ}{T}}$ 
induces a morphism $(E_2^{0,2d})_{\mab Q}\lo {\cal K}_{\os{\circ}{T}}(-d)$, 
which we denote by   
\begin{equation*} 
{\rm Tr}_{X_{\os{\circ}{T}_1}/S(T)^{\nat}}
\col (E_2^{0,2d})_{\mab Q}\lo {\cal K}_{\os{\circ}{T}}(-d)
\tag{16.7.7}\label{eqn:chtr}
\end{equation*} 
again. To prove this, let us recall the following theorem of Berthelot (\cite[p.~567]{bb}):

\begin{prop}[{\bf \cite[p.~567]{bb}}]\label{prop:bab}
Let $\kap$ be a perfect field of characteristic $p>0$. 
Let ${\cal W}_n$ be the Witt ring of $\kap$ of length $n>0$.  
Let $Y$ be a proper smooth scheme over $\kap$ 
of pure dimension $d$ with structural morphism 
$g_Y \col Y\lo \kap$. 
Let $E$ be a smooth closed subscheme of $Y$ 
with pure codimension $e$ with structural morphism 
$g_E \col E\lo \kap$. 
Let $\iota \col E \os{\sus}{\lo} Y$ be the closed immersion. 
Then the following diagram is commutative$:$ 
\begin{equation*} 
\begin{CD} 
R\Gam(E/{\cal W}_n)\otimes^L_{{\cal W}_n} 
R\Gam(Y/{\cal W}_n)[-2e]@>{\cup   \iota^*}>> 
R\Gam(E/{\cal W}_n)[-2e]  @>{\rm Tr}_{g_E}>> {\cal W}_n[-2d]\\ 
@V{G_{E/Y}\otimes {\rm id}}VV @VV{G_{E/Y}}V  @|\\ 
R\Gam(Y/{\cal W}_n)\otimes^L_{{\cal W}_n} 
R\Gam(Y/{\cal W}_n)@>{\cup}>> 
R\Gam(Y/{\cal W}_n)@>{\rm Tr}_{g_Y}>> {\cal W}_n[-2d].  
\end{CD}
\tag{16.8.1}\label{cd:eytr}
\end{equation*} 
\end{prop} 

\begin{coro}\label{coro:trw}
By abuse of notation, let 
$${\rm Tr}_{g_E}\col H^{2(d-e)}_{\rm crys}(E/{\cal W}_n)
\lo  {\cal W}_n$$ 
$$({\rm resp}.~{\rm Tr}_{g_Y}\col H^{2d}_{\rm crys}(Y/{\cal W}_n)\lo  {\cal W}_n)$$  
be also the trace morphism.  
Let $x$ be an element of $H^{2(d-e)}_{\rm crys}(Y/{\cal W})$. 
Let $c_e(E) \in H^{2e}_{\rm crys}(Y/{\cal W})$ be 
the cycle class of $E$.
Then ${\rm Tr}_{g_E}(\iota^*(x))={\rm Tr}_{g_Y}(c_e(E)\cdot x)$. 
\end{coro}
\begin{proof} 
By (\ref{cd:eytr}) we have a formula
${\rm Tr}_{g_E}(y\cdot \iota^*(y'))={\rm Tr}_{g_Y}(G_{E/Y}(y)\cdot y')$ 
for $y\in H^r_{\rm crys}(Y/W)$ and $y\in H^{2e-r}_{\rm crys}(Y/W)$ 
$(r\in {\mab Z})$. We have only to set $y=1\in H^0_{\rm crys}(Y/W)$ 
and $y'=x$. 
\end{proof} 

By the existence of the trace morphism 
as a convergent crystal for a projective smooth morphism due to Ogus \cite[(3.12)]{od},  
(\ref{prop:bab}) and (\ref{coro:trw}) are generalized for the relative case:

\begin{theo}\label{theo:brab}   
Let $Y$ be a projective smooth scheme over $\os{\circ}{T}_1$ 
of pure relative dimension $d$ with structural morphism 
$g_Y \col Y\lo \os{\circ}{T}$. 
Let $E$ be a relative smooth closed subscheme of $Y$ 
with pure relative codimension $e$ with structural morphism 
$g_E \col E\lo \os{\circ}{T}_1$. 
Let $\iota \col E \os{\sus}{\lo} Y$ be the closed immersion. 
Let $r$ be a nonnegative integer. 
Then the following diagram is commutative$:$ 
\begin{equation*} 
\begin{CD} 
R^rg_{E/\os{\circ}{T}}({\cal O}_{E/\os{\circ}{T}})\otimes_K 
R^{2(d-e)-r}g_{Y/\os{\circ}{T}}
({\cal O}_{Y/\os{\circ}{T}})_K@>{\cup  \iota^*}>> 
R^{2(d-e)}g_{E/\os{\circ}{T}}({\cal O}_{E/\os{\circ}{T}})_K@>{\rm Tr}_{g_{E/\os{\circ}{T}}}>> 
{\cal K}_{T}\\ 
@V{G_{E/Y}\otimes {\rm id}}VV @VV{G_{E/Y}}V  @|\\ 
R^{r+2e}g_{Y/\os{\circ}{T}}({\cal O}_{Y/\os{\circ}{T}})_K\otimes_K
R^{2(d-e)-r}g_{Y/\os{\circ}{T}}({\cal O}_{Y/\os{\circ}{T}})_K@>{\cup}>> 
R^{2d}g_{Y/\os{\circ}{T}}({\cal O}_{E/\os{\circ}{T}})_K 
@>{\rm Tr}_{g_{Y/\os{\circ}{T}}}>> {\cal K}_{T}.  
\end{CD}
\tag{16.10.1}\label{cd:eytar}
\end{equation*} 
\end{theo} 

\begin{coro}\label{coro:rrw}
Let $x$ be an element of $R^{2(d-e)}g_{Y/\os{\circ}{T}}(Y/\os{\circ}{T})_K$. 
Let $c_e(E) \in R^{2e}g_{Y/\os{\circ}{T}}(Y/\os{\circ}{T})_K$ be 
the cycle class of $E$.
Then ${\rm Tr}_{g_{E/\os{\circ}{T}}}(\iota^*(x))
={\rm Tr}_{g_{Y/\os{\circ}{T}}}(c_e(E)\cdot x)$. 
\end{coro}

\begin{prop}\label{prop:trx} 
The trace morphism {\rm (\ref{ali:rdsm})} 
induces a morphism $(E_2^{0,2d})_{\mab Q}\lo {\cal K}_{T}(-d)$, 
which we denote by   
\begin{equation*} 
{\rm Tr}^A_{X_{\os{\circ}{T}_1}/S(T)^{\nat}}\col 
(E_2^{0,2d})_{\mab Q}\lo {\cal K}_{T}(-d).
\tag{16.12.1}\label{eqn:cstr}
\end{equation*} 
\end{prop} 
\begin{proof} 
For $\lam_0<\lam_1$, 
let 
\begin{align*} 
G_{\lam_i} \col &
R^{2d-2}f_{\os{\circ}{X}_{\{\lam_0,\lam_1\},T_1}/\os{\circ}{T}*}
({\cal O}_{\os{\circ}{X}_{\{\lam_0,\lam_1\},T_1}/\os{\circ}{T}}
\otimes_{\mab Z}\vp_{\rm crys}
(\os{\circ}{X}_{\{\lam_0,\lam_1\},T_1}/\os{\circ}{T}))_{\mab Q}(-1)\\
&\lo 
R^{2d}
f_{\os{\circ}{X}{}^{(0)}_{T_1}/T*}
({\cal O}_{\os{\circ}{X}{}^{(0)}_{\lam_i,T_1}/\os{\circ}{T}}
\otimes_{\mab Z}\vp^{(0)}_{\rm crys}
(\os{\circ}{X}_{\lam_i,T_1}/T))_{\mab Q}
\end{align*}
be the crystalline Gysin morphism for the closed immersion 
$X_{\{\lam_0,\lam_1\}}\os{\sus}{\lo} X_{\lam_i}$. 
By \cite[(1.5.21)]{nb} (cf.~\cite[$(10.1.2;\star)$]{ndw}) 
the boundary morphism 
\begin{align*} 
E_1^{-1,2d}=&R^{2d}
f_{\os{\circ}{X}{}^{(1)}_{T_1}/T*}
({\cal O}_{\os{\circ}{X}{}^{(0)}_{T_1}/\os{\circ}{T}}
\otimes_{\mab Z}\vp^{(0)}_{\rm crys}
(\os{\circ}{X}_{T_1}/T))_{\mab Q}(-1)  \\
&\lo E_1^{0,2d}=R^{2d}
f_{\os{\circ}{X}{}^{(0)}_{T_1}/T*}
({\cal O}_{\os{\circ}{X}{}^{(0)}_{T_1}/\os{\circ}{T}}
\otimes_{\mab Z}\vp^{(0)}_{\rm crys}
(\os{\circ}{X}_{T_1}/T))_{\mab Q}=(E_1^{0,2d})_{\mab Q}
\end{align*} 
is given by the following \v{C}ech-Gysin morphism
\begin{align*}
G:=\sum_{\lam_{0}<\lam_{1}}
(G_{\lam_0}-G_{\lam_1}) \col &
R^{2d-2}
f_{\os{\circ}{X}{}^{(1)}_{T_1}
/\os{\circ}{T}*}
({\cal O}_{\os{\circ}{X}{}^{(1)}_{T_1}/\os{\circ}{T}}
\otimes_{\mab Z} 
\vp^{(1)}_{\rm crys}(\os{\circ}{X}_{T_1}/\os{\circ}{T}))_{\mab Q}(-1)
\tag{16.12.2}\label{eqn:togsn}
\\
&\lo R^{2d}
f_{\os{\circ}{X}{}^{(0)}_{T_1}
/\os{\circ}{T}*}({\cal O}_{\os{\circ}{X}{}^{(0)}_{T_1}
/\os{\circ}{T}}\otimes_{\mab Z} 
\vp^{(0)}_{\rm crys}(\os{\circ}{X}_{T_1}/\os{\circ}{T}))_{\mab Q}.
\end{align*}
Hence (\ref{prop:trx}) follows from the commutative square (\ref{cd:eytar}).  
\end{proof} 

\begin{defi}\label{defi:ffc}
We call ${\rm Tr}^A_{X_{\os{\circ}{T}_1}/S(T)^{\nat}}$ 
the {\it trace morphism} of $X_{\os{\circ}{T}_1}/S(T)^{\nat}$ 
with respect to $A_{\rm zar}(X_{\os{\circ}{T}_1}/S(T)^{\nat})_{\mab Q}$. 
\end{defi}

Henceforth, let the notations and the assumptions be as in \S\ref{sec:e2}. 

\begin{prop}\label{prop:cv}
The trace morphism ${\rm Tr}^A_{X_{\os{\circ}{T}_1}/S(T)^{\nat}}$ extends to 
the following morphism 
\begin{align*} 
{\rm Tr}^A_{X/K} \col E_2^{0,2d}(X/K) \lo {\cal O}_{X/K}(-d)
\end{align*} 
on ${\rm Enl}^{\sq}(S/{\cal V})$. 
Here $E_2^{0,2d}(X/K)$ is a log convergent $F$-isocrystal 
on ${\rm Enl}^{\sq}(S/{\cal V})$ obtained by $E_2^{0,2d}$ 
in {\rm (\ref{ali:eopsp})}. 
\end{prop} 
\begin{proof} 
Because the Gysin morphism extends to a morphism of convergent isocrystals 
(\cite[(3.13)]{od}), 
we obtain (\ref{prop:cv}) by (\ref{prop:bddes}). 
\end{proof}

\begin{prop}\label{prop:xt}
If each fiber $X_t$ of $X_{\os{\circ}{T}_1}/S(T)^{\nat}$ for any exact closed point 
$t\in S(T)^{\nat}$ is geometrically connected, then 
${\rm Tr}^A_{X_{\os{\circ}{T}_1}/S(T)^{\nat}}$ is an isomorphism. 
\end{prop}
\begin{proof} 
By the description of the boundary morphism $E_1^{ij}\lo E_1^{i+1,j}$ 
(\cite[(1.5.21)]{nb}), 
$(E_2^{0,2d})(d)_{\mab Q}$ is the dual of $E^{00}_2$, which is equal to ${\cal K}_T$. 
Hence $(E_2^{0,2d})_{\mab Q}\simeq {\cal K}_T(-d)$. 
Since ${\rm Tr}^A_{X_{\os{\circ}{T}_1}/S(T)^{\nat}}$ is a surjective morphism of 
the free module ${\cal K}_T$-modules of the same rank, 
${\rm Tr}^A_{X_{\os{\circ}{T}_1}/S(T)^{\nat}}$ is an isomorphism. 
\end{proof}

\begin{coro}\label{coro:pad}
Let us denote the composite morphism of the cup product and 
${\rm Tr}^A_{X/K}$ by $\langle ?.?\rangle:$ 
\begin{align*} 
\langle ?.?\rangle \col &R^{q}f_*({\cal O}_{X/K})^{\nat,\sq}
\otimes_{{\cal O}_{X/K}}R^{2d-q}f_*({\cal O}_{X/K})^{\nat,\sq}
\os{\cup}{\lo}  R^{2d}f_*({\cal O}_{X/K})^{\nat,\sq}\os{{\rm Tr}^A_{X/K}}
\lo {\cal O}_{X/K}(-d) \tag{16.16.1}\label{ali:fxss}. 
\end{align*} 
Then $\langle ?.?\rangle$ induces the following isomorphism 
\begin{align*} 
R^{q}f_*({\cal O}_{X/K})^{\nat,\sq}
\os{\sim}{\lo} {\cal H}{\it om}_{{\cal K}_T}(R^{2d-q}f_*({\cal O}_{X/K})^{\nat,\sq},{\cal O}_{X/K}(-d)). 
\tag{16.16.2}\label{ali:fxsts}
\end{align*} 
\end{coro} 
\begin{proof}
Let $T$ be an object of ${\rm Enl}_p^{\sq}(S/{\cal V})$. 
By the theory of convergent isocrystals, we may assume that 
$\os{\circ}{T}$ is a point. In fact, we may assume that 
$\os{\circ}{T}$ is the formal spectrum of the Witt ring of 
$\Gam(\os{\circ}{T}_{1,{\rm red}},{\cal O}_{\os{\circ}{T}_{1,{\rm red}}})$. 
Let 
\begin{align*} 
{\rm Tr}^T_{X_{\os{\circ}{T}_1}/S(T)^{\nat}}\col 
R^{2d}f_{X_{\os{\circ}{T}_1}/S(T)^{\nat}*}({\cal O}_{X_{\os{\circ}{T}_1}/S(T)^{\nat}})
\lo {\cal K}_T(-d)
\end{align*} 
be Tsuji's trace morphism. 
Tsuji's Poincar\'{e} duality gives us an isomorphism 
\begin{align*} 
R^{q}f_{X_{\os{\circ}{T}_1}/S(T)^{\nat}*}
({\cal O}_{X_{\os{\circ}{T}_1}/S(T)^{\nat}})_{\mab Q}
\os{\sim}{\lo} {\cal H}{\it om}_{{\cal K}_T}(
R^{2d-q}f_{X_{\os{\circ}{T}_1}/S(T)^{\nat}*}({\cal O}_{X_{\os{\circ}{T}_1}/S(T)^{\nat}})_{\mab Q},{\cal K}_T(-d)). 
\tag{16.16.3}\label{ali:fmxss}
\end{align*} 
First assume that $\os{\circ}{X}_{\os{\circ}{T}_1}$ is geometrically connected. 
Then, consider the following composite ${\cal K}_T$-linear isomorphism 
\begin{align*} 
{\cal K}_T(-d)\os{({\rm Tr}^A_{X_{\os{\circ}{T}_1}/S(T)^{\nat}})^{-1}}{\lo} 
(E_2^{0,2d})_{\mab Q}=
R^{2d}f_{X_{\os{\circ}{T}_1}/S(T)^{\nat}*}({\cal O}_{X_{\os{\circ}{T}_1}/S(T)^{\nat}})_{\mab Q}
\os{{\rm Tr}^T_{X_{\os{\circ}{T}_1}/S(T)^{\nat}}}{\lo} {\cal K}_T(-d). 
\end{align*} 
This is obtained by a global section of ${\cal K}_T^*$. 
Hence the morphisms (\ref{ali:fxsts}) and (\ref{ali:fmxss}) are the same up to 
${\cal K}_T^*$. 
If $\os{\circ}{X}_{\os{\circ}{T}_1}$ is not geometrically connected, then 
we have only to take an extension 
$\Gam(\os{\circ}{T}_{1,{\rm red}},{\cal O}_{\os{\circ}{T}_{1,{\rm red}}})$. 
\end{proof}



\section{Coincidence of the weight filtrations I}\label{sec:cwt}
Let the notations be as in \S\ref{sec:psc} and the beginning of the previous section. 
In this section we construct a filtered morphism 
\begin{align*}
\psi \col (A_{\rm zar}(X_{\os{\circ}{T}_0}/S(T)^{\nat},E),P)
\lo 
(H_{\rm zar}(X_{\os{\circ}{T}_0}/S(T)^{\nat},E),P)
\tag{17.0.1}\label{ali:ahupu} 
\end{align*} 
such that the underlying morphism 
$A_{\rm zar}(X_{\os{\circ}{T}_0}/S(T)^{\nat},E)\lo 
H_{\rm zar}(X_{\os{\circ}{T}_0}/S(T)^{\nat},E)$ 
is an isomorphism. 
This is a log crystalline analogue of the filtered morphism constructed in \cite{fup}.  
To construct the morphism (\ref{ali:ahupu}), 
we have to solve problems about signs arising from the signs of 
the boundary morphisms of $A_{\rm zar}(X_{\os{\circ}{T}_0}/S(T)^{\nat},E)$ 
and $H_{\rm zar}(X_{\os{\circ}{T}_0}/S(T)^{\nat},E)$, 
which most mathematicians do not want to think of them as possible.  
It is convenient to another $(A_{\rm zar}(X_{\os{\circ}{T}_0}/S(T)^{\nat},E),P)$ 
whose boundary morphism
is different from that of $(A_{\rm zar}(X_{\os{\circ}{T}_0}/S(T)^{\nat},E),P)$. 
Let us define it. 
\par 
In this section we also consider the following boundary morphisms 
of double complexes: 
\begin{equation*}
\begin{CD}
A_{\rm zar}({\cal P}^{\rm ex}_{\bul}/S(T)^{\nat},
{\cal E}^{\bul})^{i,j+1}  @.  \\ 
@A{d\log t \wedge}AA  @. \\ 
A_{\rm zar}({\cal P}^{\rm ex}_{\bul}/S(T)^{\nat},{\cal E}^{\bul})^{ij}
@>{\nabla}>> 
A_{\rm zar}({\cal P}^{\rm ex}_{\bul}/S(T)^{\nat},
{\cal E}^{\bul})^{i+1,j},\\
\end{CD}
\tag{17.0.2}\label{cd:lcacbd} 
\end{equation*}  
which are different from (\ref{cd:lccbd}). 
This convention of signs is the same as 
the original Steenbrink complex in \cite{sti}.  
Using these boundary morphisms, we have the double complex 
$A'_{\rm zar}({\cal P}^{\rm ex}_{\bul}/S(T)^{\nat},
{\cal E}^{\bul})^{\bul \bul}:=(A_{\rm zar}({\cal P}^{\rm ex}_{\bul}/S(T)^{\nat},
{\cal E}^{\bul})^{ij})_{i,j\in {\mab Z}}$. 
Let 
$A'_{\rm zar}({\cal P}^{\rm ex}_{\bul}/S(T)^{\nat},{\cal E}^{\bul})$ 
be the single complex of  
$A'_{\rm zar}({\cal P}^{\rm ex}_{\bul}/S(T)^{\nat},
{\cal E}^{\bul})^{\bul \bul}$. 
\par 
The double complex 
$A'_{\rm zar}({\cal P}^{\rm ex}_{\bul}/S(T)^{\nat},{\cal E}^{\bul})^{\bul \bul}$ 
has a filtration $P=\{P_k\}_{k \in {\mab Z}}$ 
defined by the following formula: 
\begin{equation*} 
P_kA'_{\rm zar}({\cal P}^{\rm ex}_{\bul}/S(T)^{\nat},
{\cal E}^{\bul})^{\bul \bul}
:=(\cdots P_{2j+k+1}A'_{\rm zar}({\cal P}^{\rm ex}_{\bul}/S(T)^{\nat},
{\cal E}^{\bul})^{ij}\cdots). 
\tag{17.0.3}\label{eqn:lpcpad}
\end{equation*} 
Let $(A'_{\rm zar}({\cal P}^{\rm ex}_{\bul}/S(T)^{\nat},
{\cal E}^{\bul}),P)$ 
be the filtered single complex of the filtered double complex 
$(A'_{\rm zar}({\cal P}^{\rm ex}_{\bul}/S(T)^{\nat},
{\cal E}^{\bul})^{\bul \bul},P)$.  
Then the isomorphism 
$$A'_{\rm zar}({\cal P}^{\rm ex}_{\bul}/S(T)^{\nat},
{\cal E}^{\bul})^{ij}\owns \om \lom 
(-1)^{i}\om \in A_{\rm zar}({\cal P}^{\rm ex}_{\bul}/S(T)^{\nat},{\cal E}^{\bul})^{ij}$$  
induces the following filtered isomorphism  
$$(A'_{\rm zar}({\cal P}^{\rm ex}_{\bul}/S(T)^{\nat},
{\cal E}^{\bul})^{\bul \bul},P)\os{\sim}{\lo} 
(A_{\rm zar}({\cal P}^{\rm ex}_{\bul}/S(T)^{\nat},
{\cal E}^{\bul})^{\bul \bul},P).$$ 
Set 
\begin{equation*} 
(A'_{\rm zar}(X_{\os{\circ}{T}_0}/S(T)^{\nat},E),P)
:=R\pi_{{\rm zar}*}
((A'_{\rm zar}({\cal P}^{\rm ex}_{\bul}/S(T)^{\nat},{\cal E}^{\bul}),P)).   
\end{equation*} 
Obviously we obtain the following isomorphism: 
\begin{align*} 
(A'_{\rm zar}(X_{\os{\circ}{T}_0}/S(T)^{\nat},E),P)\os{\sim}{\lo} 
(A_{\rm zar}(X_{\os{\circ}{T}_0}/S(T)^{\nat},E),P).  
\end{align*}

\begin{prop}\label{prop:indscr} 
The morphism 
\begin{align*} 
(-1)^id\log t\wedge \col {\cal E}^{\bul}
\otimes_{{\cal O}_{{\cal P}^{\rm ex}_{\bul}}}
{\Om}^{i}_{{\cal P}^{\rm ex}_{\bul}/S(T)^{\nat}}\lo 
({\cal E}^{\bul}
\otimes_{{\cal O}_{{\cal P}^{\rm ex}_{\bul}}}
{\Om}^{i+1}_{{\cal P}^{\rm ex}_{\bul}/\os{\circ}{T}})
/
P_j({\cal E}^{\bul}
\otimes_{{\cal O}_{{\cal P}^{\rm ex}_{\bul}}}
{\Om}^{i+1}_{{\cal P}^{\rm ex}_{\bul}/\os{\circ}{T}})  \quad (i\in {\mab N}). 
\end{align*} 
induces the following isomorphism
\begin{equation*} 
(-1)^{\bul}\theta \wedge :=(-1)^{\bul}\theta_{X_{\os{\circ}{T}_0}/S(T)^{\nat}} 
\wedge \col Ru_{X_{\os{\circ}{T}_0}/S(T)^{\nat}*}
(\eps^*_{X_{\os{\circ}{T}_0}/S(T)^{\nat}}(E))\os{\sim}{\lo} 
A'_{\rm zar}(X_{\os{\circ}{T}_0}/S(T)^{\nat},E).
\tag{17.1.1}\label{eqn:uaz} 
\end{equation*} 
Here $(-1)^{\bul}$ is only a notation and there is no other meaning.  
\end{prop}
\begin{proof}
The proof of this proposition is the same as that of \cite[(1.4.3)]{nb}.
\end{proof}

\par
We also have to change signs of the boundary morphisms of 
$H_{\rm zar}(X_{\os{\circ}{T}_0}/S(T)^{\nat},E)$ as follows.
\par 
Recall that the boundary morphism of 
$H_{\rm zar}(X_{\os{\circ}{T}_0}/S(T)^{\nat},E)$ 
is induced by the following boundary morphism 
\begin{equation*} 
\small{
\begin{CD} 
{\cal O}_Tu^{[k]}\otimes_{{\cal O}_T}{\cal E}\otimes_{{\cal O}_{{\cal P}{}^{{\rm ex},(m-1)}}}
\Om^{q+1}_{{\cal P}{}^{{\rm ex},(m-1)}/\os{\circ}{T}}\\
@A{(-1)^{m-1}\nabla}AA  \\
{\cal O}_Tu^{[k]}\otimes_{{\cal O}_T}{\cal E}
\otimes_{{\cal O}_{{\cal P}{}^{{\rm ex},(m-1)}}}
\Om^q_{{\cal P}{}^{{\rm ex},(m-1)}/\os{\circ}{T}} 
@>{\us{\lam_0<\cdots<\lam_{m-1}}{\sum}\sum_{j=0}^m
(-1)^j\iota^{*}_{\ul{\lam}_j,\ul{\lam}}}>> 
{\cal O}_Tu^{[k]}\otimes_{{\cal O}_T}{\cal E}\otimes_{{\cal O}_{{\cal P}{}^{{\rm ex},(m)}}}
\Om^q_{{\cal P}{}^{{\rm ex},(m)}/\os{\circ}{T}}\\
@V{(-1)^{m-1}((?)'d\log t\wedge)}VV    \\
{\cal O}_Tu^{[k-1]}\otimes_{{\cal O}_T}
{\cal E}\otimes_{{\cal O}_{{\cal P}{}^{{\rm ex},(m-1)}}}
\Om^{q+1}_{{\cal P}{}^{{\rm ex},(m-1)}/\os{\circ}{T}} \\
\end{CD} 
\tag{17.1.2}\label{cd:poex}}
\end{equation*}  
for $\ul{\lam}=\{\lam_0,\ldots,\lam_{m-1}\}\in P(\Lam)$. 
Let 
$H'_{\rm zar}(X_{\os{\circ}{T}_0}/S(T)^{\nat},E)$ be the analogous complex 
to $H_{\rm zar}(X_{\os{\circ}{T}_0}/S(T)^{\nat},E)$ 
with the following boundary morphism 
\begin{equation*} 
\tiny{
\begin{CD} 
{\cal O}_Tu^{[k]}\otimes_{{\cal O}_T}{\cal E}\otimes_{{\cal O}_{{\cal P}{}^{{\rm ex},(m-1)}}}
\Om^{q+1}_{{\cal P}{}^{{\rm ex},(m-1)}/\os{\circ}{T}}\\
@A{\nabla}AA  \\
{\cal O}_Tu^{[k]}\otimes_{{\cal O}_T}{\cal E}\otimes_{{\cal O}_{{\cal P}{}^{{\rm ex},(m-1)}}}
\Om^q_{{\cal P}{}^{{\rm ex},(m-1)}/\os{\circ}{T}} 
@>{(-1)^{q+m}\us{\lam_0<\cdots<\lam_{m-1}}{\sum}\sum_{j=0}^m
(-1)^j\iota^{\ul{\lam}_j*}_{\ul{\lam}}}>> 
{\cal O}_Tu^{[k]}\otimes_{{\cal O}_T}{\cal E}\otimes_{{\cal O}_{{\cal P}{}^{{\rm ex},(m)}}}
\Om^q_{{\cal P}{}^{{\rm ex},(m)}/\os{\circ}{T}}\\
@V{((?)'d\log t\wedge)}VV    \\
{\cal O}_Tu^{[k-1]}\otimes_{{\cal O}_T}{\cal E}\otimes_{{\cal O}_{{\cal P}{}^{{\rm ex},(m-1)}}}
\Om^{q+1}_{{\cal P}{}^{{\rm ex},(m-1)}/\os{\circ}{T}}. \\
\end{CD} 
\tag{17.1.3}\label{cd:poaex}}
\end{equation*}

\parno  
Let 
\begin{align*} 
\eps \col {\mab N}\times {\mab N}\lo {\mab Z}/2
\end{align*} 
be a map satisfying the following equations
\begin{align*} 
\eps(m-1,q+1)=\eps(m-1,q)+m-1
\end{align*} 
and 
\begin{align*} 
\eps(m,q)=\eps(m-1,q)+q+m. 
\end{align*} 
Indeed, this function exists because
\begin{align*} 
\eps(m+1,q+1)=\eps(m,q+1)+m+1+q+1=\eps(m,q)+m+m+q=\eps(m,q)+q
\end{align*} 
and 
\begin{align*} 
\eps(m+1,q+1)=\eps(m+1,q)+m+1=\eps(m,q)+m+1+q+m+1=\eps(m,q)+q. 
\end{align*} 
Then the morphism
\begin{align*}
{\cal O}_Tu^{[k]}\otimes_{{\cal O}_T}
{\cal E}\otimes_{{\cal O}_{{\cal P}{}^{{\rm ex},(m-1)}}}
\Om^q_{{\cal P}{}^{{\rm ex},(m-1)}/\os{\circ}{T}} 
\owns \om \lom (-1)^{\eps(m-1,q)}\om
\in 
{\cal O}_Tu^{[k]}\otimes_{{\cal O}_T}
{\cal E}\otimes_{{\cal O}_{{\cal P}{}^{{\rm ex},(m-1)}}}
\Om^q_{{\cal P}{}^{{\rm ex},(m-1)}/\os{\circ}{T}} 
\end{align*} 
induces a filtered isomorphism 
\begin{align*} 
(H'_{\rm zar}(X_{\os{\circ}{T}_0}/S(T)^{\nat},E),P)
\os{\sim}{\lo} 
H_{\rm zar}(X_{\os{\circ}{T}_0}/S(T)^{\nat},E),P). 
\tag{17.1.4}\label{ali:xxst}
\end{align*} 
To fix the idea, set $\eps(0,0)=0$. 
Then $\eps(0,q)=1$ for odd $q$ and $\eps(0,q)=0$ for even $q$. 
The morphism 
\begin{align*}  
{\cal O}_Tu^{[k]}\otimes_{{\cal O}_T}
{\cal E}\otimes_{{\cal O}_{{\cal P}{}^{{\rm ex}}}}
\Om^q_{{\cal P}{}^{{\rm ex}}/\os{\circ}{T}} \owns \om 
\lom (-1)^q\om  \in 
{\cal O}_Tu^{[k]}\otimes_{{\cal O}_T}
{\cal E}\otimes_{{\cal O}_{{\cal P}{}^{{\rm ex},(0)}}}
\Om^q_{{\cal P}{}^{{\rm ex},(0)}/\os{\circ}{T}} 
\end{align*} 
induces the following isomorphism 
\begin{align*}
\wt{R}u_{X_{\os{\circ}{T}_0}/\os{\circ}{T}*}
(\eps^*_{X_{\os{\circ}{T}_0}/\os{\circ}{T}}(E)\langle u\rangle)
\os{\sim}{\lo} 
H'_{\rm zar}(X_{\os{\circ}{T}_0}/S(T)^{\nat},E)
\tag{17.1.5}\label{ali:xxopst}
\end{align*} 
by (\ref{ali:xxst}) and (\ref{ali:eetie}). 

The following is a crystalline version of \cite[(5.25), (5.29)]{fup} modulo 
the difference of signs: 
\begin{theo}[{\bf cf.~\cite[(5.25), (5.29)]{fup}}]\label{theo:ha}
There exists the following filtered morphism 
\begin{align*} 
\psi' \col (A'_{\rm zar}(X_{\os{\circ}{T}_0}/S(T)^{\nat},E),P)\lo 
(H'_{\rm zar}(X_{\os{\circ}{T}_0}/S(T)^{\nat},E),P)
\tag{17.2.1}\label{ali:azh}
\end{align*} 
such that the underlying morphism 
\begin{align*}
\psi' \col A'_{\rm zar}(X_{\os{\circ}{T}_0}/S(T)^{\nat},E)
\lo H'_{\rm zar}(X_{\os{\circ}{T}_0}/S(T)^{\nat},E)
\tag{17.2.2}\label{ali:azeh}
\end{align*}
is an isomorphism fitting into the following commutative diagram
\begin{equation*} 
\begin{CD}
A'_{\rm zar}(X_{\os{\circ}{T}_0}/S(T)^{\nat},E)@>{\psi',\sim}>>
H'_{\rm zar}(X_{\os{\circ}{T}_0}/S(T)^{\nat},E)\\
@A{(-1)^{\bul}\theta \wedge}AA @AA{\simeq}A\\
Ru_{X_{\os{\circ}{T}_0}/S(T)^{\nat}*}
(\eps^*_{X_{\os{\circ}{T}_0}/S(T)^{\nat}}(E))
@. \wt{R}u_{X_{\os{\circ}{T}_0}/\os{\circ}{T}*}
(\eps^*_{X_{\os{\circ}{T}_0}/\os{\circ}{T}}(E)\langle u \rangle)\\
@| @VV{\simeq}V\\
Ru_{X_{\os{\circ}{T}_0}/S(T)^{\nat}*}
(\eps^*_{X_{\os{\circ}{T}_0}/S(T)^{\nat}}(E))
@=Ru_{X_{\os{\circ}{T}_0}/S(T)^{\nat}*}
(\eps^*_{X_{\os{\circ}{T}_0}/S(T)^{\nat}}(E)). 
\end{CD}
\tag{17.2.3}\label{cd:ssati}
\end{equation*} 
\end{theo}  
\begin{proof} 
Fix a total order on $\Lam$ once and for all as before 
and fix an isomorphism from orientation sheaves with ${\mab Z}$. 
Note that 
\begin{align*} 
A'_{\rm zar}({\cal P}^{\rm ex}_{\bul}/S(T)^{\nat},{\cal E}^{\bul})
=(\cdots \lo \bigoplus_{j=0}^{q}({\cal E}^{\bul}
\otimes_{{\cal O}_{{\cal P}^{\rm ex}_{\bul}}}
{\Om}^{q+1}_{{\cal P}^{\rm ex}_{\bul}/\os{\circ}{T}})/
P_j({\cal E}^{\bul}
\otimes_{{\cal O}_{{\cal P}^{\rm ex}_{\bul}}}
{\Om}^{q+1}_{{\cal P}^{\rm ex}_{\bul}/\os{\circ}{T}})\lo \cdots)
\end{align*} 
with boundary morphism 
$\nabla+d\log t\wedge $. 
Let 
\begin{align*} 
{\rm Res}'_{{\cal P}_{\bul \ul{\lam}}} \col 
{\cal E}^{\bul}
\otimes_{{\cal O}_{{\cal P}^{\rm ex}_{\bul}}}
{\Om}^{\bul}_{{\cal P}^{\rm ex}_{\bul}/\os{\circ}{T}}
\lo {\cal E}^{\bul}\otimes_{{\cal O}_{{\cal P}^{\rm ex}}}
\Om^{\bul-(m+1)}_{{\cal P}^{\rm ex}_{\bul \ul{\lam}}/\os{\circ}{T}}
\end{align*} 
be the morphism defined by the following: 
\begin{align*}  
\sig \otimes \omega d\log x_{\lam_{i_0},y}\cdots d\log x_{\lam_{i_{m}},y} 
\lom 
\sig \otimes b^*_{\lam_{i_0}\cdots \lam_{i_{m}}}(\omega)
\tag{17.2.4}\label{ali:pbis}
\end{align*} 
$$(\sig \in {\cal E}_{\bul}, 
\om \in P_0{\Om}^{\bul}_{{\cal P}^{\rm ex}/\os{\circ}{T}}, i_0<\cdots <i_m)$$
{\rm (cf.~\cite[(3.1.5)]{dh2})}. 
Here $y$ is a point of $\os{\circ}{\mathfrak D}_{\bul}$,  
$r$ is a nonnegative integer such that 
$M_{X,y}/{\cal O}_{X,y}^*\simeq {\mab N}^r$,  
$m$ is a nonnegative integer less than $r$, 
$\ul{\lam}=\{\lam_{i_0},\ldots, \lam_{i_{m}}\}$ 
and $\{x_{\lam_0,y},\ldots, x_{\lam_{r-1},y}\}$ is a basis of 
$M_{X,y}/{\cal O}_{X,y}^*$. 
Note that the morphism (\ref{ali:pbis}) is different from the morphism (\ref{eqn:mprrn}) 
about the places of logarithmic forms in the source of the morphism. 
We define the following morphism 
\begin{align*} 
\psi'_{\rm loc} \col \bigoplus_{j,q}({\cal E}^{\bul}
\otimes_{{\cal O}_{{\cal P}^{\rm ex}_{\bul}}}
{\Om}^{q+1}_{{\cal P}^{\rm ex}_{\bul}/\os{\circ}{T}})/
P_j({\cal E}^{\bul}
\otimes_{{\cal O}_{{\cal P}^{\rm ex}_{\bul}}}
{\Om}^{q+1}_{{\cal P}^{\rm ex}_{\bul}/\os{\circ}{T}})
\lo  
\us{q,\ul{\lam}}{\bigoplus}
{\cal O}_T\langle u\rangle \otimes_{{\cal O}_T}
{\cal E}^{\bul}\otimes_{{\cal O}_{{\cal P}^{\rm ex}}}
\Om^{q+1-\# \ul{\lam}}_{{\cal P}^{\rm ex}_{\bul \ul{\lam}/\os{\circ}{T}}}
\end{align*} 
by 
\begin{align*} 
\psi'_{\rm loc} :=\sum_{j\geq 0}\sum_{\{\ul{\lam}\,\vert\,\# \ul{\lam}\,\geq\,j+1\}}
u^{[\# \ul{\lam}-(j+1))]}\otimes {\rm Res}'_{{\cal P}_{\bul \ul{\lam}}} 
=\sum_{\ul{\lam}}\sum_{j= 0}^{\# \ul{\lam}-1}
u^{[\# \ul{\lam}-(j+1))]}\otimes {\rm Res}'_{{\cal P}_{\bul \ul{\lam}}} 
\tag{17.2.5}\label{ali:pis}
\end{align*} 
(cf.~\cite[p.~172]{fup}). 
Note that the restriction of $\psi'_{\rm loc}$ to 
$P_j({\cal E}^{\bul}\otimes_{{\cal O}_{{\cal P}^{\rm ex}_{\bul}}}
{\Om}^{q+1}_{{\cal P}^{\rm ex}_{\bul}/\os{\circ}{T}})$ indeed vanishes. 
It is clear that  
\begin{align*} 
{\rm Res}'_{{\cal P}_{\bul \ul{\lam}}}\circ \nabla
=\nabla \circ {\rm Res}'_{{\cal P}_{\bul \ul{\lam}}}.
\end{align*} 
Let $\lam_0<\cdots <\lam_m$ be the different elements of $\ul{\lam}$. 
Set $\ul{\lam}_i:=\ul{\lam}\setminus \{\lam_i\}$. 
We claim that 
\begin{align*} 
{\rm Res}'_{{\cal P}_{\bul \ul{\lam}}}\circ (d\log t\wedge )=
(d\log t\wedge)\circ {\rm Res}'_{{\cal P}_{\bul \ul{\lam}}}+
\sum_{i=0}^m(-1)^{i+q-m}{\rm Res}'_{{\cal P}_{\bul \ul{\lam}_i}}. 
\end{align*} 
Indeed, express 
$\om \in {\cal E}^{\bul}
\otimes_{{\cal O}_{{\cal P}^{\rm ex}_{\bul}}}
{\Om}^{q+1}_{{\cal P}^{\rm ex}_{\bul}/\os{\circ}{T}}$ 
by the following form:  
$$\om=\eta d\log x_{\lam_0}\wedge \cdots \wedge d\log x_{\lam_m}+
\sum_{i=0}^m\eta_id\log x_{\lam_0}\wedge \cdots \wedge d\log x_{\lam_{i-1}}\wedge 
d\log x_{\lam_{i+1}}\wedge \cdots \wedge d\log x_{\lam_m}
+\eta'',$$
where $\eta$, $\eta_i$ have no log differential forms 
$d\log x_{\lam_0}, \ldots, d\log x_{\lam_m}$ and $\eta''$ 
has no log differential forms $d\log x_{\lam_0}\wedge \cdots \wedge d\log x_{\lam_{i-1}}$ 
for $0\leq i\leq m$. 
Then 
\begin{align*} 
{\rm Res}'_{{\cal P}_{\bul \ul{\lam}}}(d\log t\wedge \om)&=
{\rm Res}'_{{\cal P}_{\bul \ul{\lam}}}(d\log t\wedge \eta\wedge d\log x_{\lam_0}
\wedge \cdots \wedge d\log x_{\lam_m})\tag{17.2.6}\label{ali:dlx}\\
&+\sum_{i=0}^m{\rm Res}'_{{\cal P}_{\bul \ul{\lam}}}(
d\log t\wedge \eta_i\wedge d\log x_{\lam_0}\wedge \cdots \wedge d\log x_{\lam_{i-1}}\wedge 
d\log x_{\lam_{i+1}}\wedge \cdots \wedge d\log x_{\lam_m})\\
&=d\log t\wedge \eta+\sum_{i=0}^m(-1)^{(q+1-m)+i-1}\eta_i\\
&=d\log t\wedge {\rm Res}'_{{\cal P}_{\bul \ul{\lam}}}(\om)+
\sum_{i=0}^m(-1)^{i+q-m}{\rm Res}'_{{\cal P}_{\bul \ul{\lam}_i}}(\om). 
\end{align*} 
Let 
$$p_{\ul{\mu}}\col 
\us{q,\ul{\lam}}{\bigoplus}
{\cal O}_T\langle u\rangle \otimes_{{\cal O}_T}
{\cal E}^{\bul}\otimes_{{\cal O}_{{\cal P}^{\rm ex}}}
\Om^{q+1-\# \ul{\lam}}_{{\cal P}^{\rm ex}_{\bul \ul{\lam}/\os{\circ}{T}}}
\lo 
{\cal O}_T\langle u\rangle \otimes_{{\cal O}_T}
{\cal E}^{\bul}\otimes_{{\cal O}_{{\cal P}^{\rm ex}}}
\Om^{q+1-\# \ul{\mu}}_{{\cal P}^{\rm ex}_{\bul,\ul{\mu}/\os{\circ}{T}}}$$
be the projection for $\ul{\mu}$. 
Since 
$$d\log t \wedge \col {\cal E}^{\bul}
\otimes_{{\cal O}_{{\cal P}^{\rm ex}_{\bul}}}
{\Om}^{q+1}_{{\cal P}^{\rm ex}_{\bul}/\os{\circ}{T}}/
P_j({\cal E}^{\bul}
\otimes_{{\cal O}_{{\cal P}^{\rm ex}_{\bul}}}
{\Om}^{q+1}_{{\cal P}^{\rm ex}_{\bul}/\os{\circ}{T}})\lo 
{\cal E}^{\bul}
\otimes_{{\cal O}_{{\cal P}^{\rm ex}_{\bul}}}
{\Om}^{q+1}_{{\cal P}^{\rm ex}_{\bul}/\os{\circ}{T}}/
P_{j+1}({\cal E}^{\bul}
\otimes_{{\cal O}_{{\cal P}^{\rm ex}_{\bul}}}
{\Om}^{q+1}_{{\cal P}^{\rm ex}_{\bul}/\os{\circ}{T}}),$$ 
we obtain the following formula by (\ref{ali:dlx}) for $j$ and $\ul{\mu}$ such that 
$\# \ul{\mu}\geq j+1$: 
\begin{align*} 
p_{\ul{\mu}} \circ \psi'_{\rm loc} \circ(\nabla+d\log t\wedge)&=
p_{\ul{\mu}} \circ \psi'_{\rm loc}\circ \nabla+
p_{\ul{\mu}} \circ \psi'_{\rm loc}\circ d\log t\wedge 
\tag{17.2.7}\label{ali:dtw}
\\
&=u^{[\# \ul{\mu}-(j+1)]}\otimes 
{\rm Res}'_{{\cal P}_{\bul,\ul{\mu}}} \circ \nabla 
+u^{[\# \ul{\mu}-(j+2)]}\otimes  
(d\log t\wedge) \circ {\rm Res}'_{{\cal P}_{\bul,\ul{\mu}}}\\
& +u^{[\# \ul{\mu}-(j+2))]}\otimes
\sum_{i=0}^m(-1)^{i+q+m}\iota^{*}_{\ul{\mu}_i,\ul{\mu}}{\rm Res}'_{{\cal P}_{\bul,\ul{\mu}_i}}\\
&=u^{[\# \ul{\mu}-(j+1)]}\otimes 
{\rm Res}'_{{\cal P}_{\bul,\ul{\mu}}} \circ \nabla 
+u^{[\# \ul{\mu}-(j+2)]}\otimes  
(d\log t\wedge) \circ {\rm Res}'_{{\cal P}_{\bul,\ul{\mu}}}\\
& +u^{[\# \ul{\mu}-1-(j+1)]}\otimes
\sum_{i=0}^m(-1)^{i+q+m}\iota^{*}_{\ul{\mu}_i,\ul{\mu}}{\rm Res}'_{{\cal P}_{\bul,\ul{\mu}_i}}\\
&= p_{\ul{\mu}} \circ 
(\nabla+(?)'d\log t\wedge+\sum_{\ul{\lam}}\sum_{i=0}^m(-1)^{i+q+m}\iota^{*}_{\ul{\lam}_i,\ul{\lam}})\circ 
\psi'_{\rm loc}. 
\end{align*}
Here we define $u^{[-1]}$ as $0$. 
Let $\psi' \col  A'_{\rm zar}(X_{\os{\circ}{T}_0}/S(T)^{\nat},E)\lo 
H'_{\rm zar}(X_{\os{\circ}{T}_0}/S(T)^{\nat},E)$ be the induced morphism by $\psi'_{\rm loc}$.  
By (\ref{ali:dtw}) we see that  $\psi'$ is a morphism of complexes. 
\par
Next we check that $\psi'$ is a filtered morphism. 
The subsheaf 
$P_{2j+k+1}({\cal E}^{\bul}
\otimes_{{\cal O}_{{\cal P}^{\rm ex}_{\bul}}}
{\Om}^{q+1}_{{\cal P}^{\rm ex}_{\bul}/\os{\circ}{T}})$ of 
${\cal E}^{\bul}
\otimes_{{\cal O}_{{\cal P}^{\rm ex}_{\bul}}}
{\Om}^{q+1}_{{\cal P}^{\rm ex}_{\bul}/\os{\circ}{T}}$ 
is mapped to 
\begin{align*} 
&\sum_{\ul{\lam}}\sum_{j=0}^{\# \ul{\lam}-1}
P_{2(\# \ul{\lam}-(j+1))+((2j+k+1-\# \ul{\lam})-(\# \ul{\lam}-1))}
({\cal O}_T\langle u\rangle \otimes_{{\cal O}_T}
{\cal E}^{\bul}\otimes_{{\cal O}_{{\cal P}^{\rm ex}}}
\Om^{\bul}_{{\cal P}^{\rm ex}_{\ul{\lam}/\os{\circ}{T}}})\\
&=P_k\sum_{\ul{\lam}}\sum_{j=0}^{\# \ul{\lam}-1}
({\cal O}_T\langle u\rangle \otimes_{{\cal O}_T}
{\cal E}^{\bul}\otimes_{{\cal O}_{{\cal P}^{\rm ex}}}
\Om^{\bul}_{{\cal P}^{\rm ex}_{\ul{\lam}/\os{\circ}{T}}}).
\end{align*}
Hence $\psi'_{\rm loc}$ is a filtered morphism. 
Consequently $\psi'$ is a filtered morphism. 
\par 
As usual, we can see that $\psi'$ is independent of the choice of 
an affine open covering of 
$X_{\os{\circ}{T}_0}$ and the simplicial immersion 
$X_{\os{\circ}{T}_0,\bul}\os{\sus}{\lo}\ol{\cal P}_{\bul}$ over $S(T)^{\nat}$. 
Indeed, assume that we are given another similar immersion 
$X'_{\os{\circ}{T}_0,\bul}\os{\sus}{\lo}\ol{\cal P}{}'_{\bul}$. 
Then we can prove the independence as in the proof of (\ref{theo:indp}). 
\par
Let $\om$ be a local section of 
${\cal E}^{\bul}
\otimes_{{\cal O}_{{\cal P}^{\rm ex}_{\bul}}}
{\Om}^q_{{\cal P}^{\rm ex}_{\bul}/\os{\circ}{T}}$. 
Then 
\begin{align*} 
\psi'_{\rm loc}((-1)^{q}d\log t\wedge \om)=
(-1)^{q}\sum_{\ul{\lam}}u^{[0]}\otimes{\rm Res}'_{\bul \ul{\lam}}(\om)=
(-1)^{\eps(0,q)}\sum_{\ul{\lam}}{\rm Res}'_{\bul \ul{\lam}}(\om). 
\end{align*} 
Hence the diagram (\ref{cd:ssati}) is commutative. 
By (\ref{eqn:uaz}), (\ref{ali:xxopst}) and (\ref{eqn:exte})
we see that  (\ref{ali:azeh}) is an isomorphism. 
\end{proof} 

\begin{coro}\label{coro:nci}
There exists the following filtered morphism 
\begin{align*} 
\psi \col (A_{\rm zar}(X_{\os{\circ}{T}_0}/S(T)^{\nat},E),P)\lo 
(H_{\rm zar}(X_{\os{\circ}{T}_0}/S(T)^{\nat},E),P)
\tag{17.3.1}\label{ali:azfh}
\end{align*} 
such that the underlying morphism 
\begin{align*}
\psi \col A_{\rm zar}(X_{\os{\circ}{T}_0}/S(T)^{\nat},E)
\lo H_{\rm zar}(X_{\os{\circ}{T}_0}/S(T)^{\nat},E)
\tag{17.3.2}\label{ali:azfeh}
\end{align*}
is an isomorphism fitting into the following commutative diagram
\begin{equation*} 
\begin{CD}
A_{\rm zar}(X_{\os{\circ}{T}_0}/S(T)^{\nat},E)@>{\psi,\sim}>>
H_{\rm zar}(X_{\os{\circ}{T}_0}/S(T)^{\nat},E)\\
@A{\theta \wedge}A{\simeq}A @AA{\simeq}A\\
Ru_{X_{\os{\circ}{T}_0}/S(T)^{\nat}*}
(\eps^*_{X_{\os{\circ}{T}_0}/S(T)^{\nat}}(E))
@. \wt{R}u_{X_{\os{\circ}{T}_0}/\os{\circ}{T}*}
(\eps^*_{X_{\os{\circ}{T}_0}/\os{\circ}{T}}(E)\langle u \rangle)\\
@| @VV{\simeq}V\\
Ru_{X_{\os{\circ}{T}_0}/S(T)^{\nat}*}
(\eps^*_{X_{\os{\circ}{T}_0}/S(T)^{\nat}}(E))
@=Ru_{X_{\os{\circ}{T}_0}/S(T)^{\nat}*}
(\eps^*_{X_{\os{\circ}{T}_0}/S(T)^{\nat}}(E)). 
\end{CD}
\tag{17.3.3}\label{cd:ssti}
\end{equation*} 
\end{coro}

\begin{defi}\label{defi:crf}
We call the filtered morphism {\rm (\ref{ali:azfh})} is the {\it crystalline Fujisawa morphism}. 
\end{defi} 


Let the notations be as in \S\ref{sec:filbo}. 
Then we can prove the following: 

\begin{theo}[{\bf Comparison theorem of weight filtrations}]\label{theo:text}
Let $P'$ be the induced filtration on 
$R^qf_{X_{\os{\circ}{T}_0}/S(T)^{\nat}*}({\cal O}_{X_{\os{\circ}{T}_0}/S(T)^{\nat}})$ $(q\in {\mab N})$ by the filtered complex
$(H_{\rm zar}(X_{\os{\circ}{T}_0}/S(T)^{\nat}),P)$. 
Then $P'\otimes_{\cal V}K=P\otimes_{\cal V}K$. 
\end{theo}
\begin{proof} 
(cf.~\cite[(5.4.6), (5.4.7)]{nb})
Because $P\otimes_{\cal V}K$ and $P'\otimes_{\cal V}K$ 
extend to a convergent $F$-isocrystal on $\os{\circ}{T}/{\cal V}$, 
we may assume that $\os{\circ}{T}$ is the formal spectrum of a finite extension of 
${\cal V}$. 
By the deformation invariance of $P\otimes_{\cal V}K$ and $P'\otimes_{\cal V}K$, 
we may assume that $\os{\circ}{T}$ is the formal spectrum of the Witt ring of a 
finite extension of $\kap$. By a standard technique, we may assume that 
$\os{\circ}{T}_1$ is a scheme of finite type over a finite field. 
By the convergence of 
$P\otimes_{\cal V}K$ and $P'\otimes_{\cal V}K$ again, 
we may assume that  $\os{\circ}{T}$ is the formal spectrum of the Witt ring of a 
finite field. 
In this case the weight spectral sequences for 
$(H_{\rm zar}(X_{\os{\circ}{T}_0}/S(T)^{\nat}),P)$ and 
$(A_{\rm zar}(X_{\os{\circ}{T}_0}/S(T)^{\nat}),P)$ and the purity of the 
weight proved in \cite[(6.11) or Appendix]{ny} and \cite{ndw} 
(cf.~\cite{kme}, \cite{clpu})) tell us that 
$P\otimes_{\mab Z}{\mab Q}=P'\otimes_{\mab Z}{\mab Q}$. 
\end{proof} 

\begin{coro}\label{coro:pd}
The composite morphism {\rm (\ref{ali:fxss})}  
induces the following morphism 
\begin{align*} 
&P_kR^{q}f_{X_{\os{\circ}{T}_0}/S(T)^{\nat}*}
({\cal O}_{X_{\os{\circ}{T}_0}/S(T)^{\nat}})_{\mab Q}
\otimes_{{\cal K}_T}
P_{k'}R^{2d-q}f_{X_{\os{\circ}{T}_0}/S(T)^{\nat}*}({\cal O}_{X_{\os{\circ}{T}_0}/S(T)^{\nat}})_{\mab Q}
\lo \tag{17.6.1}\label{ali:fxpss}\\
& \lo P_{k+k'}{\cal K}_T(-d), 
\end{align*} 
which gives the following isomorphism 
\begin{align*} 
P_kR^qf_{X_{\os{\circ}{T}_0}/S(T)^{\nat}*}({\cal O}_{X_{\os{\circ}{T}_0}/S(T)^{\nat}})_K
\os{\sim}{\lo} P_k{\cal H}{\it om}_{{\cal K}_S}
(R^{2d-q}f_{X_{\os{\circ}{T}_0}/S(T)^{\nat}*}({\cal O}_{X_{\os{\circ}{T}_0}/S(T)^{\nat}})_K, {\cal K}_T(-d))
\tag{17.6.2}\label{ali:pkkcq}
\end{align*} 
for any $k\in {\mab Z}$. 
\end{coro} 
\begin{proof}
By (\ref{ali:otp}) 
we have the following morphism 
\begin{align*} 
\cup   \col &
P_kR^hf_{X_{\os{\circ}{T}_0}/S(T)^{\nat}*}({\cal O}_{X_{\os{\circ}{T}_0}/S(T)^{\nat}})
\otimes_{{\cal O}_T}
P_{k'}R^{2d-h}f_{X_{\os{\circ}{T}_0}/S(T)^{\nat}*}({\cal O}_{X_{\os{\circ}{T}_0}/S(T)^{\nat}})
\tag{17.6.3}\label{ali:odtp}\\
& \lo 
P_{k+k'}R^{2d}f_{X_{\os{\circ}{T}_0}/S(T)^{\nat}*}({\cal O}_{X_{\os{\circ}{T}_0}/S(T)^{\nat}}). 
\end{align*}  
Because $R^{2d}f_{X_{\os{\circ}{T}_0}/S(T)^{\nat}*}({\cal O}_{X_{\os{\circ}{T}_0}/S(T)^{\nat}})$ 
and ${\cal K}_T(-d)$ are pure of weight $2d$, 
the trace morphism 
${\rm Tr}^A_{X_{\os{\circ}{T}_0}/S(T)^{\nat}}$ is 
a filtered morphism. Hence we have the morphism (\ref{ali:fxpss}). 
\par 
By (\ref{coro:pad}) and by the same proof as that of (\ref{theo:text}), 
we obtain the isomorphism (\ref{ali:pkkcq}). 
\end{proof}

\section{Semi-cosimplicial $p$-adic Steenbrink complexes and 
coincidence of the weight filtrations II}\label{sec:csc}
In this section we define the semi-cosimplicial version of 
$(A_{\rm zar}(X_{\os{\circ}{T}_0}/S(T)^{\nat},E),P)$ 
to give another proof of the comparison theorem of 
weight filtrations (\ref{theo:text}). 
This method is sketched in \cite[p.~1270]{kiha} 
for the case of log de Rham-Witt complexes for the trivial coefficient 
when the base field is finite.  
\par 
Let the notations be as in \S\ref{sec:pssc}. 
Take an index subset $\ul{\lam}=\{\lam_0,\ldots,\lam_{m-1}\}$ ($\lam_i< \lam_j$ for $i<j$). 
Set  
\begin{align*} 
A_{\rm zar}({\cal P}^{\rm ex}_{\ul{\lam}}/S(T)^{\nat},{\cal E})^{ij}
& :=({\cal E}\otimes_{{\cal O}_{{\cal P}^{\rm ex}}}
\Om^{i+j+1}_{{\cal P}^{\rm ex}_{\ul{\lam}}/\os{\circ}{T}})/P_j 
\tag{18.0.1}\label{ali:accef} \\
& :=({\cal E}
\otimes_{{\cal O}_{{\cal P}^{\rm ex}}}
{\Om}^{i+j+1}_{{\cal P}^{\rm ex}_{\ul{\lam}}/\os{\circ}{T}})/
P_j({\cal E}
\otimes_{{\cal O}_{{\cal P}^{\rm ex}}}
{\Om}^{i+j+1}_{{\cal P}^{\rm ex}_{\ul{\lam}}/\os{\circ}{T}})  
\quad (i,j \in {\mab N}). 
\end{align*}   
As in \S\ref{sec:pssc}, we have the filtered double complex 
$(A_{\rm zar}({\cal P}^{\rm ex}_{\ul{\lam}}/S(T)^{\nat},{\cal E})^{\bul \bul},P)$ and  
the filtered single complex 
$(A_{\rm zar}({\cal P}^{\rm ex}_{\ul{\lam}}/S(T)^{\nat},{\cal E}),P)
:=s((A_{\rm zar}({\cal P}^{\rm ex}_{\ul{\lam}}/S(T)^{\nat},{\cal E})^{\bul \bul},P))$. 
Let 
\begin{equation*} 
\pi_{\ul{\lam},{\rm zar}} 
\col (X_{\ul{\lam},\os{\circ}{T}_0\bul})_{\rm zar}
\lo (X_{\ul{\lam},\os{\circ}{T}_0})_{\rm zar}
\end{equation*} 
be the natural morphism of ringed topoi.  
Let $f_{\ul{\lam}}\col X_{\ul{\lam},\os{\circ}{T}_0}\lo \os{\circ}{T}$ 
and $f_{\ul{\lam} \bul}\col X_{\ul{\lam},\os{\circ}{T}_0\bul}\lo \os{\circ}{T}$
be the structural morphisms. 
The first aim in this section is to prove that 
$$R\pi_{\ul{\lam},{\rm zar}*} 
((A_{\rm zar}({\cal P}^{\rm ex}_{\ul{\lam}}/S(T)^{\nat},{\cal E}),P))$$ 
is well-defined. 
Let us recall the morphism 
\begin{equation*} 
\pi_{\ul{\lam},{\rm zar}} \col 
((X_{\ul{\lam},\os{\circ}{T}_0\bul})_{\rm zar},f^{-1}_{\bul}({\cal O}_T)) \lo 
((X_{\ul{\lam},\os{\circ}{T}_0})_{\rm zar},
f^{-1}_{\ul{\lam}}({\cal O}_T)) 
\end{equation*} 
of ringed topoi ((\ref{eqn:lbd})) and the complex  (\ref{eqn:lbid}); 
\begin{align*} 
Ru_{X_{\ul{\lam},\os{\circ}{T}_0}/S(T)^{\nat}*}
(\eps^*_{X_{\ul{\lam},\os{\circ}{T}_0}/S(T)^{\nat}}(E_{\ul{\lam}}))
&= R\pi_{\ul{\lam},{\rm zar}}(s({\cal E}^{\bul}\otimes_{{\cal O}_{{\cal P}^{\rm ex}_{\bul}}}
{\Om}^{\bul}_{{\cal P}^{\rm ex}_{\bul\ul{\lam}}/S(T)^{\nat}})).
\end{align*}

\begin{prop}\label{prop:tefnc} 
$(1)$ There exists the following canonical isomorphism$: $
\begin{equation*} 
\theta \wedge:=\theta_{X_{\ul{\lam},\os{\circ}{T}_0}/S(T)^{\nat}}
\wedge \col 
Ru_{X_{\ul{\lam},\os{\circ}{T}_0}/S(T)^{\nat}*}
(\eps^*_{X_{\ul{\lam},\os{\circ}{T}_0}/S(T)^{\nat}}(E_{\ul{\lam}}))
\os{\sim}{\lo} 
R\pi_{\ul{\lam},{\rm zar}*}(A_{\rm zar}({\cal P}^{\rm ex}_{\ul{\lam}}/S(T)^{\nat},{\cal E})).
\tag{18.1.1}\label{eqn:ulul} 
\end{equation*}  
\par
$(2)$ There exists the following canonical isomorphism$: $
\begin{align*} 
&{\rm gr}_k^PR\pi_{\ul{\lam},{\rm zar}*}(A_{\rm zar}({\cal P}^{\rm ex}_{\ul{\lam}}/S(T)^{\nat},{\cal E}))
\os{\sim}{\lo} 
\bigoplus_{j\geq \max \{-k,0\}} 
\bigoplus_{\# \ul{\mu}=2j+k+1}
 \tag{18.1.2}\label{ali:rllvp}\\
&a_{\ul{\lam}\cup \ul{\mu},\os{\circ}{T}_0*} 
(Ru_{\os{\circ}{X}_{\ul{\lam}\cup \ul{\mu},T_0}/\os{\circ}{T}*}
(E_{\ul{\lam}\cup \ul{\mu},T_0}/\os{\circ}{T}
\otimes_{\mab Z}\vp_{\rm crys}^{(2j+k)}(\os{\circ}{X}_{\ul{\mu},T_0}/\os{\circ}{T})))[-2j-k]. 
\end{align*} 
\end{prop}
\begin{proof}
(1): The following proof is the crystalline version of the proof of 
\cite[Proposition 3.2]{gkcf}. 
The proof is slightly different from that of 
(\ref{prop:tefc}) by the property of ``$P_0$'' for $X_{\ul{\lam}}$.  
\par 
Since ${\cal E}^{\bul}_{\ul{\lam}}:=
{\cal E}\otimes_{{\cal O}_{{\cal P}^{\rm ex}_{\bul}}}
{\cal O}_{{\cal P}^{\rm ex}_{\bul \ul{\lam}}}$ is a flat
${\cal O}_{{\mathfrak D}_{\bul \ul{\lam}}}$-module, 
it suffices to prove that the natural morphism 
\begin{align*} 
d\log t \wedge  & 
\col 
{\cal O}_{{\mathfrak D}_{\bul \ul{\lam}}}
\otimes_{{\cal O}_{{\cal P}^{\rm ex}_{\bul \ul{\lam}}}}
{\Om}^i_{{\cal P}^{\rm ex}_{\bul \ul{\lam}}/S}
\lo \\
& \{({\cal O}_{{\mathfrak D}_{\bul \ul{\lam}}}
\otimes_{{\cal O}_{{\cal P}^{\rm ex}_{\bul \ul{\lam}}}}
{\Om}^{i+1}_{{\cal P}^{\rm ex}_{\bul \ul{\lam}}/\os{\circ}{T}}
/P_0   \os{(-1)^id\log t \wedge}{\lo} 
{\cal O}_{{\mathfrak D}_{\bul \ul{\lam}}}
\otimes_{{\cal O}_{{\cal P}^{\rm ex}_{\ul{\bul,\lam}}}}
{\Om}^{i+2}_{{\cal P}^{\rm ex}_{\bul \ul{\lam}}/\os{\circ}{T}}/P_1 \\ 
& \os{(-1)^id\log t \wedge}{\lo}\cdots )\}
\end{align*} 
is a quasi-isomorphism.  
\par 
As in \cite[3.15]{msemi} 
(cf.~\cite[(6.28) (9), (6.29) (1)]{ndw}), 
it suffices to prove that the sequence 
\begin{align*} 
& 0 \lo 
{\rm gr}_0^P
({\cal O}_{{\mathfrak D}_{n,\ul{\lam}}}
{\otimes}_{{\cal O}_{{\cal P}^{\rm ex}_{n,\ul{\lam}}}} 
{{\Om}}^{\bul}_{{\cal P}^{\rm ex}_{n,\ul{\lam}}/\os{\circ}{T}}) 
\os{\theta_{{\cal P}^{\rm ex}_{n,\ul{\lam}}}\wedge}{\lo}
{\rm gr}_1^P
({\cal O}_{{\mathfrak D}_{n,\ul{\lam}}}
{\otimes}_{{\cal O}_{{\cal P}^{\rm ex}_{n,\ul{\lam}}}}
{{\Om}}^{\bul}_{{\cal P}^{\rm ex}_{n,\ul{\lam}}/\os{\circ}{T}})[1] 
\\
& \os{\theta_{{\cal P}^{\rm ex}_{n,\ul{\lam}}}\wedge}{\lo}
{\rm gr}_2^P
({\cal O}_{{\mathfrak D}_{n,\ul{\lam}}}
{\otimes}_{{\cal O}_{{\cal P}_{n,\ul{\lam}}^{\rm ex}}} 
{{\Om}}^{\bul}_{{\cal P}^{\rm ex}_{n,\ul{\lam}}/\os{\circ}{T}})[2]  
\os{\theta_{{\cal P}^{\rm ex}_{n,\ul{\lam}}}\wedge}{\lo} \cdots 
\end{align*} 
is exact. 
By (\ref{eqn:mpprrn}) we have only to prove that 
the following sequence 
\begin{align*} 
0 & \lo {\cal O}_{{\mathfrak D}_{n,\ul{\lam}}}
{\otimes}_{{\cal O}_{{\cal P}{}^{\rm ex}_{n,\ul{\lam}}}} 
{\Om}^i_{{\os{\circ}{\cal P}}{}^{\rm ex}_{n,\ul{\lam}}/\os{\circ}{T}} \lo 
\bigoplus_{\# \ul{\mu}=1}
({\cal O}_{{\mathfrak D}_{n,\ul{\lam}\cup \ul{\mu}}}
{\otimes}_{{\cal O}_{{\cal P}{}^{{\rm ex}}_{n,\ul{\lam}\cup \ul{\mu}}}} 
b_{\ul{\lam}\cup \ul{\mu}*}
\Om^i_{\os{\circ}{\cal P}{}^{{\rm ex}}_{n,\ul{\lam}\cup \ul{\mu}}/\os{\circ}{T}}
\otimes_{\mab Z}\vp_{{\rm zar},\ul{\mu}}({\cal P}^{\rm ex}_n/S)) 
\tag{18.1.3}\label{eqn:zzp}\\
&   
\lo \bigoplus_{\# \ul{\mu}=2}
({\cal O}_{{\mathfrak D}_{n,\ul{\lam}\cup \ul{\mu}}}
{\otimes}_{{\cal O}_{{\cal P}{}^{{\rm ex}}_{n,\ul{\lam}\cup \ul{\mu}}}}
b_{\ul{\lam}\cup \ul{\mu}*}
\Om^i_{\os{\circ}{\cal P}{}^{{\rm ex}}_{n,\ul{\lam}\cup \ul{\mu}}/\os{\circ}{T}}
\otimes_{\mab Z}\vp_{{\rm zar},\ul{\mu}}({\cal P}^{\rm ex}_n/S)) 
\lo \cdots 
\end{align*}
is exact for each nonnegative integer $i$.   
Following the idea in the proof of \cite[Proposition 3.2]{gkcf}, 
set $S_j:=\{\ul{\nu}~\vert  \# (\ul{\nu} \setminus \ul{\lam})=j\}$  
$(j\in {\mab Z}_{\geq 0})$. 
Consider the following double complex 
\begin{equation*} 
\Om^{jk}_i:=\bigoplus_{\# \ul{\nu}=j+k,\ul{\nu}\in S_j}
{\cal O}_{{\mathfrak D}_{n,\ul{\lam}\cup \ul{\nu}}}
{\otimes}_{{\cal O}_{{\cal P}{}^{{\rm ex}}_{n,\ul{\lam}\cup \ul{\nu}}}}
b_{\ul{\lam}\cup \ul{\nu}*}
\Om^i_{\os{\circ}{\cal P}{}^{{\rm ex}}_{n,\ul{\lam}\cup \ul{\nu}}/\os{\circ}{T}}
\otimes_{\mab Z}\vp_{{\rm zar},\ul{\nu}}({\cal P}^{\rm ex}_n/\os{\circ}{T}) 
\quad (j,k\in {\mab Z}_{\geq 0})
\end{equation*} 
such that the horizontal boundary morphisms 
and the vertical boundary morphisms are given by the \v{C}ech 
morphism of the pull-backs of the restriction morphisms. 
Then the complex (\ref{eqn:zzp}) is equal to $\Om^{\bul \bul}_i$.  
Then $\Om^{j\bul}_i$ is exact since the boundary morphisms 
of $\Om^{j\bul}_i$ is given identities with signs. 
Hence the complex (\ref{eqn:zzp}) is exact. 
\par 
In a routine way 
we can prove that the isomorphism (\ref{eqn:ulul}) is independent of 
the choice of the disjoint union 
of an affine simplicial open covering of $X$ and 
a simplicial immersion 
$X_{\bul} \os{\sus}{\lo} \ol{\cal P}_{\bul}$ over $\ol{S(T)^{\nat}}$.  
\par 
(2): As in the proof of \cite[(1.4.4)]{nb}, we obtain (\ref{ali:rllvp}) by using (\ref{eqn:mpprrn}). 
\end{proof}

\begin{coro}\label{coro:lid}
The filtered complex 
$R\pi_{\ul{\lam},{\rm zar}*}
((A_{\rm zar}({\cal P}^{\rm ex}_{\bul \ul{\lam}}/S(T)^{\nat},{\cal E}^{\bul}),P))$ 
is independent of the choice of the disjoint union 
of an affine simplicial open covering of $X$ and 
a simplicial immersion 
$X_{\bul} \os{\sus}{\lo} \ol{\cal P}_{\bul}$ over $\ol{S(T)^{\nat}}$.  
\end{coro}
\begin{proof} 
This follows from (\ref{prop:tefnc}). 
\end{proof}

\begin{defi}
Set 
\begin{equation*} 
(A_{\rm zar}(X_{\ul{\lam},\os{\circ}{T}_0}/S(T)^{\nat},E_{\ul{\lam}}),P)
:=R\pi_{\ul{\lam},{\rm zar}*}
((A_{\rm zar}({\cal P}^{\rm ex}_{\bul \ul{\lam}}/S(T)^{\nat},{\cal E}^{\bul}),P)).  
\end{equation*} 
and we call 
$(A_{\rm zar}(X_{\ul{\lam},\os{\circ}{T}_0}/S(T)^{\nat},E_{\ul{\lam}})),P)
\in {\rm D}^+{\rm F}(f^{-1}({\cal O}_T))$ 
the {\it $p$-adic zariskian filtered Steenbrink complex} of 
$E_{\ul{\lam}}$ on $X_{\ul{\lam},\os{\circ}{T}_0}/(S(T)^{\nat},{\cal J},\del)$. 
\end{defi}

We state the following for our memory: 
\begin{prop}\label{prop:grlp}
There exists the following canonical isomorphism$:$ 
\begin{align*} 
{\rm gr}^P_kA_{\rm zar}(X_{\ul{\lam},\os{\circ}{T}_0}/S(T)^{\nat},E_{\ul{\lam}})
\os{\sim}{\lo} 
&\bigoplus_{j\geq \max \{-k,0\}} 
\bigoplus_{\# \ul{\mu}=2j+k+1}
a_{\ul{\lam}\cup \ul{\mu},\os{\circ}{T}_0*} 
(Ru_{\os{\circ}{X}_{\ul{\lam}\cup \ul{\mu},T_0}/\os{\circ}{T}*}
(E_{\ul{\lam}\cup \ul{\mu},T_0}/\os{\circ}{T} \tag{18.4.1}\label{ali:rlvp}\\
&\otimes_{\mab Z}\vp_{\rm crys}^{(2j+k)}
(\os{\circ}{X}_{\ul{\mu},T_0}/\os{\circ}{T})))[-2j-k]. 
\end{align*}
\end{prop}



\begin{defi} 
Set 
\begin{align*} 
(A_{\rm zar}(X^{(\star)}_{\os{\circ}{T}}/S(T)^{\nat},E^{(\star)}),P) 
:=(s(\cdots \lo a^{(m)}_*(\bigoplus_{\# \ul{\lam}=m+1}
A_{\rm zar}(X_{\ul{\lam},\os{\circ}{T}_0}/S(T)^{\nat},E_{\ul{\lam}}))\lo \cdots),\del(L,P)).
\end{align*}  
We call $(A_{\rm zar}(X^{(\star)}_{\os{\circ}{T}}/S(T)^{\nat},E^{(\star)}),P)$ 
the {\it semi-cosimplicial $p$-adic Steenbrink complex} of $E$.  
\end{defi}

We have the following natural filtered morphism 
\begin{align*} 
(A_{\rm zar}(X_{\os{\circ}{T}_0}/S(T)^{\nat},E),P)
\lo
(A_{\rm zar}(X^{(\star)}_{\os{\circ}{T}}/S(T)^{\nat},E^{(\star)}),P).
\end{align*}


\begin{prop}\label{prop:kd}
The following diagram is commutative$:$ 
\begin{equation*} 
\begin{CD} 
Ru_{X^{(\star)}_{\os{\circ}{T}_0}/S(T)^{\nat}*}(\eps^*_{X^{(\star)}_{\os{\circ}{T}_0}/S(T)^{\nat}}
(E))
@>{\theta \wedge, \sim}>> 
A_{\rm zar}(X^{(\star)}_{\os{\circ}{T}_0}/S(T)^{\nat},E^{(\star)}) \\
@A{\simeq}AA @AAA \\
Ru_{X_{\os{\circ}{T}_0}/S(T)^{\nat}*}(\eps^*_{X_{\os{\circ}{T}_0}/S(T)^{\nat}}
(E))
@>{\theta \wedge, \sim}>>  
A_{\rm zar}(X_{\os{\circ}{T}_0}/S(T)^{\nat},E).
\end{CD} 
\tag{18.6.1}\label{cd:abb} 
\end{equation*} 
Here the right vertical morphism is an underlying morphism of 
the filtered morphism 
\begin{equation*} 
(A_{\rm zar}(X_{\os{\circ}{T}_0}/S(T)^{\nat},E),P)
\lo 
(A_{\rm zar}(X^{(\star)}_{\os{\circ}{T}_0}/S(T)^{\nat},E^{(\star)}),P).
\tag{18.6.2}\label{cd:apfb} 
\end{equation*} 
\end{prop}
\begin{proof} 
Obviously the diagram (\ref{cd:abb}) is commutative.  
(\ref{prop:kd}) follows from (\ref{eqn:eetxte}), (\ref{prop:tefc}) and 
(\ref{prop:tefnc}). 
\end{proof} 

\begin{coro}\label{coro:tme}  
The morphism 
\begin{equation*} 
A_{\rm zar}(X_{\os{\circ}{T}_0}/S(T)^{\nat},E)
\lo A_{\rm zar}(X^{(\star)}_{\os{\circ}{T}_0}/S(T)^{\nat},E^{(\star)})
\tag{18.7.1}\label{cd:aatfb} 
\end{equation*} 
is a quasi-isomorphism. 
\end{coro}
\begin{proof} 
This follows from (\ref{prop:kd}). 
\end{proof}

\begin{prop}[{\bf cf.~\cite[p.~1270]{kiha}}]\label{prop:bn}
There exists a filtered morphism 
\begin{align*} 
(H_{\rm zar}(X_{\os{\circ}{T}_0}/S(T)^{\nat},E),P)
\lo
(A_{\rm zar}(X^{(\star)}_{\os{\circ}{T}}/S(T)^{\nat},E^{(\star)}),P)
\tag{18.8.1}\label{ali:haawcc}
\end{align*} 
fitting into the following commutative diagram$:$
\begin{equation*} 
\begin{CD}
\wt{R}u_{X^{(\star)}_{\os{\circ}{T}_0}/\os{\circ}{T}*}
(\eps^*_{X^{(\star)}_{\os{\circ}{T}_0}/\os{\circ}{T}}(E^{(\star)})\langle u \rangle)
@=H_{\rm zar}(X_{\os{\circ}{T}_0}/S(T)^{\nat},E)\\
@V{\simeq}VV @VVV\\
Ru_{X^{(\star)}_{\os{\circ}{T}_0}/S(T)^{\nat}*}
(\eps^*_{X^{(\star)}_{\os{\circ}{T}_0}/S(T)^{\nat}}(E^{(\star)}))
@>{\sim}>>A_{\rm zar}(X^{(\star)}_{\os{\circ}{T}}/S(T)^{\nat},E^{(\star)})\\
@A{\simeq}AA @AA{\simeq}A\\
Ru_{X_{\os{\circ}{T}_0}/S(T)^{\nat}*}
(\eps^*_{X_{\os{\circ}{T}_0}/S(T)^{\nat}}(E))
@>{\sim}>>A_{\rm zar}(X_{\os{\circ}{T}}/S(T)^{\nat},E). 
\end{CD}
\tag{18.8.2}\label{ali:maawcc}
\end{equation*}
Consequently the morphism 
\begin{align*} 
H_{\rm zar}(X_{\os{\circ}{T}_0}/S(T)^{\nat},E)\lo
A_{\rm zar}(X^{(\star)}_{\os{\circ}{T}}/S(T)^{\nat},E^{(\star)})
\tag{18.8.3}\label{ali:haancc}
\end{align*} 
is an isomorphism. 
\end{prop}
\begin{proof} 
By (\ref{eqn:cp}) and (\ref{cd:abb}) 
we have the following commutative diagram 
\begin{equation*} 
\begin{CD}
\wt{R}u_{X^{(\star)}_{\os{\circ}{T}_0}/\os{\circ}{T}*}
(\eps^*_{X^{(\star)}_{\os{\circ}{T}_0}/\os{\circ}{T}} 
(E^{(\star)})\langle u \rangle)
@>{\sim}>> 
Ru_{X^{(\star)}_{\os{\circ}{T}_0}/S(T)^{\nat}*}
(\eps^*_{X^{(\star)}_{T_0}/S(T)^{\nat}}(E^{(\star)}))
\\
@A{\simeq}AA @A{\simeq}AA  \\ 
\wt{R}u_{X_{\os{\circ}{T}_0}/\os{\circ}{T}*}
(\eps^*_{X_{\os{\circ}{T}_0}/\os{\circ}{T}}(E)\langle u \rangle)
@>{\sim}>> Ru_{X_{\os{\circ}{T}_0}/S(T)^{\nat}*}
(\eps^*_{X_{\os{\circ}{T}_0}/S(T)^{\nat}}(E)) 
\end{CD} 
\tag{18.8.4}\label{eqn:kcep}
\end{equation*}  
\begin{equation*} 
\begin{CD}
@>{\theta \wedge, \sim}>> 
A_{\rm zar}(X^{(\star)}_{\os{\circ}{T}_0}/S(T)^{\nat},E^{(\star)})\\
@. @A{\simeq}AA\\ 
@>{\theta \wedge, \sim}>>
A_{\rm zar}(X_{\os{\circ}{T}_0}/S(T)^{\nat},E).  
\end{CD} 
\end{equation*}  
It is almost trivial to check that the upper composite morphism in 
(\ref{eqn:kcep}) induces the following filtered morphism 
\begin{equation*} 
(\wt{R}u_{X^{(\star)}_{\os{\circ}{T}_0}/\os{\circ}{T}*}
(\eps^*_{X^{(\star)}_{\os{\circ}{T}_0}/\os{\circ}{T}}(E^{(\star)})\langle u \rangle), P)
\lo 
(A_{\rm zar}(X^{(\star)}_{\os{\circ}{T}_0}/S(T)^{\nat},E),P). 
\tag{18.8.5}\label{eqn:wut} 
\end{equation*}  
Indeed, since $d\log t \in P_1\Om^{\bul}_{{\cal P}^{{\rm ex},(\bul)}_{\bul}/\os{\circ}{T}}$, 
the following composite morphism 
\begin{align*} 
{\cal E}^{\bul}
\otimes_{{\cal O}_{{\cal P}^{{\rm ex},(\bul)}_{\bul}}}
\Om^{\bul}_{{\cal P}^{{\rm ex},(\bul)}_{\bul}/\os{\circ}{T}}\langle u \rangle
& \lo 
{\cal E}^{\bul}
\otimes_{{\cal O}_{{\cal P}^{{\rm ex},(\bul)}_{\bul}}}
\Om^{\bul}_{{\cal P}^{{\rm ex},(\bul)}_{\bul}/S(T)^{\nat}} \\
&\os{d\log t \wedge}{\lo} 
{\cal E}^{\bul}
\otimes_{{\cal O}_{{\cal P}^{{\rm ex},(\bul)}_{\bul}}}
\Om^{\bul}_{{\cal P}^{{\rm ex},(\bul)}_{\bul}/\os{\circ}{T}}/P_0
\end{align*} 
induces the following morphism 
\begin{align*} 
P_k({\cal E}^{\bul}
\otimes_{{\cal O}_{{\cal P}^{{\rm ex},(\bul)}_{\bul}}}
\Om^i_{{\cal P}^{{\rm ex},(\bul)}_{\bul}/\os{\circ}{T}}\langle u \rangle)
&\lo  
(P_{k+1}+P_0)({\cal E}^{\bul}
\otimes_{{\cal O}_{{\cal P}^{{\rm ex},(\bul)}_{\bul}}}
\Om^i_{{\cal P}^{{\rm ex},(\bul)}_{\bul}/\os{\circ}{T}})/P_0
\end{align*} 
for each $i\in {\mab N}$. 
Hence we obtain the following morphism of complexes: 
\begin{align*} 
P_k({\cal E}^{\bul}
\otimes_{{\cal O}_{{\cal P}^{{\rm ex},(\bul)}_{\bul}}}
\Om^{\bul}_{{\cal P}^{{\rm ex},(\bul)}_{\bul}/\os{\circ}{T}}
\langle u \rangle)
&\lo  (\cdots \lo (P_{2j+k+1}+P_j)({\cal E}^{\bul}
\otimes_{{\cal O}_{{\cal P}^{{\rm ex},(\bul)}_{\bul}}}
\Om^{\bul}_{{\cal P}^{{\rm ex},(\bul)}_{\bul}/\os{\circ}{T}})/P_j\lo \cdots). 
\end{align*} 
Hence we obtain (\ref{prop:bn}). 
\end{proof} 

\begin{coro}
There exists the following diagram 
\begin{align*} 
(H_{\rm zar}(X_{\os{\circ}{T}_0}/S(T)^{\nat},E),P)
\lo
(A_{\rm zar}(X^{(\star)}_{\os{\circ}{T}}/S(T)^{\nat},E^{(\star)}),P)\longleftarrow 
(A_{\rm zar}(X_{\os{\circ}{T}_0}/S(T)^{\nat},E),P)
\tag{18.9.1}\label{ali:haabwcc}
\end{align*}  
which induces the following isomorphisms 
\begin{align*} 
H_{\rm zar}(X_{\os{\circ}{T}_0}/S(T)^{\nat},E)
\os{\sim}{\lo}
A_{\rm zar}(X^{(\star)}_{\os{\circ}{T}}/S(T)^{\nat},E^{(\star)})\os{\sim}{\longleftarrow} 
A_{\rm zar}(X_{\os{\circ}{T}_0}/S(T)^{\nat},E). 
\tag{18.9.2}\label{ali:hmnbwcc}
\end{align*}  
\end{coro} 

For an integer $k$, 
we can calculate the complex 
${\rm gr}^{P}_k
A_{\rm zar}(X^{(\star)}_{\os{\circ}{T}_0}/S(T)^{\nat},E)$ 
as follows$:$ 
\par 
\begin{align*} 
& {\rm gr}^{P}_k
A_{\rm zar}(X^{(\star)}_{\os{\circ}{T}_0}/S(T)^{\nat},E)
={\rm gr}^{\del(L,P)}_k s(\cdots \lo \bigoplus_{\# \ul{\lam}=m+1}
A_{\rm zar}(X_{\ul{\lam},\os{\circ}{T}_0}/S(T)^{\nat},E_{\ul{\lam}})\lo \cdots)\\
&=\bigoplus_{m\geq 0}
\bigoplus_{\# \ul{\lam}=m+1}
\bigoplus_{j\geq \max \{-(k+m),0\}} 
\bigoplus_{\# \ul{\mu}=2j+k+m+1}
a_{\ul{\lam}\cup \ul{\mu}*}
Ru_{\os{\circ}{X}_{\ul{\lam}\cup \ul{\mu},T_0}/\os{\circ}{T}}
(E_{\os{\circ}{X}_{\ul{\lam}\cup \ul{\mu},T_0}/\os{\circ}{T}}
\otimes_{\mab Z}
\vp_{{\rm crys},\ul{\mu}}(\os{\circ}{X}_{T_0}/\os{\circ}{T}
))\\
&[-k-2m-2j](-j-k-m). 
\end{align*} 
Hence we obtain the following spectral sequence$:$ 
\begin{align*} 
&E_1^{-k,q+k}=\bigoplus_{m\geq 0}
\bigoplus_{\# \ul{\lam}=m+1}
\bigoplus_{j\geq \max \{-(k+m),0\}} 
\bigoplus_{\# \ul{\mu}=2j+k+m+1}
R^{q-2j-k-2m}f_{\os{\circ}{X}_{\ul{\lam}\cup \ul{\mu},T_0}/\os{\circ}{T}*} 
\tag{18.9.3}\label{eqn:espp}\\
& (E_{\os{\circ}{X}_{\ul{\lam}\cup \ul{\mu},T_0}/\os{\circ}{T}}
\otimes_{\mab Z}
\vp^{(2j+k+m)}_{\ul{\mu},{\rm crys}}
(\os{\circ}{X}_{T_0}/\os{\circ}{T}))(-j-k-m;v)
\Lo 
R^qf_{X_{\os{\circ}{T}_0}/S(T)^{\nat}*}
(\eps^*_{X_{\os{\circ}{T}_0}/S(T)^{\nat}}(E)) 
\quad (q\in {\mab Z}).
\end{align*}

\begin{theo}[{\bf $E_2${\rm -}degeneration}]\label{theo:e2d}
Let the notations be as in {\rm \S\ref{sec:e2}}. 
Then the spectral sequence {\rm (\ref{eqn:espp})} 
degenerates at $E_2$ modulo torsion. 
\end{theo} 
\begin{proof} 
As in the proofs in \S\ref{sec:filbo}, 
the proof of this theorem is the same as that of \cite[(5.4.3)]{nb}.
\end{proof} 

Now we can give another proof of (\ref{theo:text}) without using 
the crystalline Fujisawa's morphism (\ref{defi:crf}) in the previous section. 

\begin{theo}[{\bf Comparison theorem of weight filtrations}]\label{theo:mrtp}
Let the notations be as in {\rm \S\ref{sec:e2}}. 
Let ${}_1P$, ${}_2P$, ${}_3P$ be the induced filtrations on 
$R^qf_{X_{\os{\circ}{T}_0}/S(T)^{\nat}*}({\cal O}_{X_{\os{\circ}{T}_0}/S(T)^{\nat}})$ 
by the filtered complexes 
$(A_{\rm zar}(X_{\os{\circ}{T}_0}/S(T)^{\nat}),P)$, 
$(A_{\rm zar}(X^{(\star)}_{\os{\circ}{T}_0}/S(T)^{\nat}),P)$ 
and $(H_{\rm zar}(X_{\os{\circ}{T}_0}/S(T)^{\nat}),P)$, respectively. 
Let ${}_iP_K$ be the induced filtration on 
$R^qf_{X_{\os{\circ}{T}_0}/S(T)^{\nat}*}
({\cal O}_{X_{\os{\circ}{T}_0}/S(T)^{\nat}})\otimes_{\cal V}K$ by ${}_iP$. 
Then \begin{align*} 
{}_1P_K={}_2P_K={}_3P_K.
\end{align*} 
\end{theo}
\begin{proof}
The filtration ${}_1P_K$ 
extends to a convergent $F$-isocrystal 
on $\os{\circ}{T}/{\cal V}$ locally on $\os{\circ}{T}$. 
Similarly we can prove that ${}_2P_K$ and ${}_3P_K$ also 
extend to a convergent $F$-isocrystal 
on $\os{\circ}{T}/{\cal V}$ locally on $\os{\circ}{T}$. 
By using these results, a fundamental result in \cite[(3.17)]{od}, 
we may assume that $S$ is the log point of a perfect field of characteristic $p>0$. 
By the standard technique, we may assume that $\os{\circ}{S}$ is an affine  
smooth scheme over a finite field ${\mab F}_q$ and then we may assume that 
$\os{\circ}{S}={\rm Spec}({\mab F}_q)$. In this case 
${}_1P_K={}_2P_K={}_3P_K$ by the purity of the Frobenius weight 
for the crystalline cohomology of a proper smooth scheme over ${\mab F}_q$.  
\end{proof} 

\begin{rema}
Let the notations be as in \cite[(5.29)]{fup}. 
By using arguments in this section (but without using 
the morphism in \cite[(5.26)]{fup}), 
we can prove that there exists an isomorphism 
\begin{align*} 
H^q(Y,A)\os{\sim}{\lo} H^q(Y,K)
\end{align*} 
of mixed Hodge structures. 
This is a much simpler proof than his proof in [loc.~cit.]. 
\end{rema}

We conclude this section by establishing 
a relation between the crystalline Fujisawa's morphism 
and the morphisms in  (\ref{ali:haabwcc}). 
The following proposition is quite surprising to me 
because I first thought that 
the following diagram is not commutative at all. 
It also tells us why we do not need 
the crystalline Fujisawa's morphism to prove (\ref{theo:text})
in this section. 

\begin{prop}
The following diagram is commutative: 
\begin{equation*}
\begin{CD} 
A_{\rm zar}(X_{\os{\circ}{T}_0}/S(T)^{\nat},E)
@>{\psi}>>
H_{\rm zar}(X_{\os{\circ}{T}_0}/S(T)^{\nat},E)\\
@V{\simeq}VV @VV{\simeq}V \\
A_{\rm zar}(X^{(\star)}_{\os{\circ}{T}_0}/S(T)^{\nat},E^{(\star)})
@=A_{\rm zar}(X^{(\star)}_{\os{\circ}{T}_0}/S(T)^{\nat},E^{(\star)}).
\end{CD}
\tag{18.13.1}\label{cd:ppb}
\end{equation*} 
\end{prop}
\begin{proof}
This immediately follows from the commutativity of 
the diagrams (\ref{cd:ssti}) 
and (\ref{ali:maawcc}). 
\end{proof}

In fact, we can also obtain the following: 

\begin{theo}\label{theo:sfjw}
There exists a filtered morphism 
\begin{equation*} 
\psi^{(\star)} \col (A_{\rm zar}(X^{(\star)}_{\os{\circ}{T}_0}/S(T)^{\nat},E),P)
\lo (H_{\rm zar}(X_{\os{\circ}{T}_0}/S(T)^{\nat},E),P)
\tag{18.14.1}\label{ali:msidwcc}
\end{equation*} 
such that 
the composite morphism 
\begin{align*} 
(A_{\rm zar}(X_{\os{\circ}{T}_0}/S(T)^{\nat},E),P)
\lo (A_{\rm zar}(X^{(\star)}_{\os{\circ}{T}_0}/S(T)^{\nat},E),P)
\os{\psi^{(\star)}}{\lo} (H_{\rm zar}(X_{\os{\circ}{T}_0}/S(T)^{\nat},E),P)
\tag{18.14.2}\label{ali:mswcc}
\end{align*} 
is equal to $\psi$. 
It fits into the following commutative diagram$:$
\begin{equation*} 
\begin{CD}
\wt{R}u_{X^{(\star)}_{\os{\circ}{T}_0}/\os{\circ}{T}*}
(\eps^*_{X^{(\star)}_{\os{\circ}{T}_0}/\os{\circ}{T}}(E^{(\star)})\langle u \rangle)
@=H_{\rm zar}(X_{\os{\circ}{T}_0}/S(T)^{\nat},E)\\
@V{\simeq}VV @AAA\\
Ru_{X^{(\star)}_{\os{\circ}{T}_0}/S(T)^{\nat}*}
(\eps^*_{X^{(\star)}_{\os{\circ}{T}_0}/S(T)^{\nat}}(E^{(\star)}))
@>{\sim}>>A_{\rm zar}(X^{(\star)}_{\os{\circ}{T}}/S(T)^{\nat},E^{(\star)})\\
@A{\simeq}AA @AA{\simeq}A\\
Ru_{X_{\os{\circ}{T}_0}/S(T)^{\nat}*}
(\eps^*_{X_{\os{\circ}{T}_0}/S(T)^{\nat}}(E))
@>{\sim}>>A_{\rm zar}(X_{\os{\circ}{T}}/S(T)^{\nat},E). 
\end{CD}
\tag{18.14.3}\label{ali:mamewcc}
\end{equation*}
\end{theo} 
\begin{proof} 
We obtain the filtered complex $(A'_{\rm zar}(X^{(\star)}_{\os{\circ}{T}_0}/S(T)^{\nat},E),P)$ 
as in the previous section. 
Then we obtain the following morphism 
$$(A'_{\rm zar}(X^{(\star)}_{\os{\circ}{T}_0}/S(T)^{\nat},E),P)\lo 
(H'_{\rm zar}(X^{(\star)}_{\os{\circ}{T}_0}/S(T)^{\nat},E),P)$$
by the following morphism 
\begin{align*} 
\psi^{(\star),{\rm loc}}{}' \col 
\bigoplus_{\ul{\lam}}
\bigoplus_{j,q}({\cal E}^{\bul}
\otimes_{{\cal O}_{{\cal P}^{\rm ex}_{\bul}}}
{\Om}^{q+1}_{{\cal P}^{\rm ex}_{\bul\ul{\lam}}/\os{\circ}{T}})/
P_j({\cal E}^{\bul}
\otimes_{{\cal O}_{{\cal P}^{\rm ex}_{\bul}}}
{\Om}^{q+1}_{{\cal P}^{\rm ex}_{\bul\ul{\lam}}/\os{\circ}{T}})
\lo  
\us{q,\ul{\lam}}{\bigoplus}
{\cal O}_T\langle u\rangle \otimes_{{\cal O}_T}
{\cal E}^{\bul}\otimes_{{\cal O}_{{\cal P}^{\rm ex}}}
\Om^{q+1-\# \ul{\lam}}_{{\cal P}^{\rm ex}_{\bul\ul{\lam}/\os{\circ}{T}}}
\end{align*} 
defined by 
\begin{align*} 
\psi^{(\star),{\rm loc}}{}'  :=
\sum_{\ul{\lam}}
\sum_{j\geq 0}\sum_{\{\ul{\mu}\,\vert\,\# \ul{\mu}\,\geq\,j+1\}}
u^{[\# \ul{\mu}-(j+1)]}\otimes {\rm Res}'{}^{\ul{\lam}}_{{\cal P}_{\bul  \ul{\mu}}}
\tag{18.14.4}\label{ali:pies}
\end{align*} 
(cf.~(\ref{ali:pis})) as in the previous section. 
Here ${\rm Res}'{}^{\ul{\lam}}_{{\cal P}_{\bul  \ul{\mu}}}$ is the following residue morphism
$${\cal E}^{\bul}
\otimes_{{\cal O}_{{\cal P}^{\rm ex}_{\bul}}}
{\Om}^{\bul}_{{\cal P}^{\rm ex}_{\bul\ul{\lam}}/\os{\circ}{T}}/
P_j({\cal E}^{\bul}
\otimes_{{\cal O}_{{\cal P}^{\rm ex}_{\bul}}}
{\Om}^{\bul}_{{\cal P}^{\rm ex}_{\bul\ul{\lam}}/\os{\circ}{T}})
\lo 
{\cal O}_T\langle u\rangle \otimes_{{\cal O}_T}
{\cal E}^{\bul}\otimes_{{\cal O}_{{\cal P}^{\rm ex}}}
\Om^{\bul-\# \ul{\mu}}_{{\cal P}^{\rm ex}_{\bul\ul{\lam}\cup \ul{\mu}/\os{\circ}{T}}}$$
with respect to $\ul{\mu}$. 
We leave the rest of the proof to the reader because 
the proof is the same as that of (\ref{theo:ha}). 
\end{proof}

\section{$p$-adic variational filtered log hard Lefschetz conjecture}\label{sec:vlc}
In this section we prove the $p$-adic variational filtered log hard Lefschetz conjecture 
(\ref{conj:lhilc}) with respect to the weight filtration 
when there exists a fiber obtained by a projective semistable family 
over a complete discrete valuation ring in any characteristic. 
\par  
First let us recall the first log crystalline Chern class defined in \cite[(1.8.16)]{nb}, 
which is the log version of the first log crystalline Chern class defined in \cite[(3.1)]{boi}. 
\par 
Let $(T,{\cal J},\del)$ be a fine log PD-scheme on which $p$ is locally nilpotent. 
Let $T_0$ be the exact closed log subscheme defined by ${\cal J}$. 
Let $Y$ be a fine simplicial log scheme over $T_0$. 
Let $M_{Y/T}$ be a sheaf on ${\rm Crys}(Y/T)$ such that,  
for an object $T'$ of ${\rm Crys}(Y/T)$,  
$\Gam(T',M_{Y/T}):=\Gam(T',M_{T'})$. 
Let $i\col Y_{\rm zar}\lo (Y/T)_{\rm crys}$ be the canonical morphism 
defined by $\Gam(U,i^*E):=E((U,U,0))$, where $U$ is an open log subscheme of $Y$ and   
let ${\cal J}_{Y/T}$ be a sheaf on ${\rm Crys}(Y/T)$ defined by 
$\Gam((U',T',\del'),{\cal J}_{Y/T}):=\Gam(T',{\rm Ker}({\cal O}_{T'}\lo {\cal O}_{U'}))$.  
Then we have the following exact sequence 
\begin{align*}
0\lo 1+{\cal J}_{Y/T}\lo M_{Y/T}\lo i_*(M_Y)\lo 0. 
\end{align*} 
(This is the log version of the exact sequence in \cite[(3.1.2)]{boi}.)
Since $M_Y$ and $M_{Y/T}$ are sheaves of integral monoids, 
the following sequence is exact: 
\begin{align*}
0\lo 1+{\cal J}_{Y/T}\lo M^{\rm gp}_{Y/T}\lo i_*(M^{\rm gp}_Y)\lo 0. 
\tag{19.0.1}\label{ali:mygyp}
\end{align*} 
Hence we have the following composite morphism 
\begin{align*} 
c_{1,{\rm crys}}\col M_Y^{\rm gp}\lo 
Ru_{Y/T*}(1+{\cal J}_{Y/T})[1]\os{\log}{\lo} 
Ru_{Y/T*}({\cal J}_{Y/T})[1]\lo Ru_{Y/T*}({\cal O}_{Y/T})[1]. 
\tag{19.0.2}\label{ali:mygp}
\end{align*} 
In \cite{nb} we have called the morphism (\ref{ali:mygp}) 
the first log crystalline Chern class map of $Y/T$.

\par 
Let $Y_{\bul}$ be the \v{C}ech diagram of an open covering of $Y$ and 
let $Y_{\bul}\os{\sus}{\lo} {\cal Y}_{\bul}$ be an immersion 
into a log smooth simplicial log scheme over $T$. 
Let ${\mathfrak E}_{\bul}$ 
be the log PD-envelope of this immersion over $(T,{\cal J},\del)$. 
Consider the following complex 
\begin{align*} 
M^{\rm gp}_{{\mathfrak E}_{\bul}}\os{d\log}{\lo} 
{\cal O}_{{\mathfrak E}_{\bul}}
\otimes_{{\cal O}_{{\cal Y}_{\bul}}}
\Om^1_{{\cal Y}_{\bul}/T}\os{d}{\lo}
{\cal O}_{{\mathfrak E}_{\bul}}
\otimes_{{\cal O}_{{\cal Y}_{\bul}}}
\Om^2_{{\cal Y}_{\bul}/T} \os{d}{\lo}\cdots. 
\end{align*} 
Denote this complex by 
${\cal O}_{{\mathfrak E}_{\bul}}
\otimes_{{\cal O}_{{\mathfrak E}_{\bul}}}
\Om^{\times}_{{\cal Y}_{\bul} /T}$. 
Set 
$${\cal K}^{\times}_{{\mathfrak E}_{\bul}/T}:={\rm Ker}
({\cal O}_{{\mathfrak E}_{\bul}}
\otimes_{{\cal O}_{{\cal Y}_{\bul}}}
\Om^{\times}_{{\cal Y}_{\bul}/T}\lo 
M^{\rm gp}_{Y_{\bul}}).$$ 
Set also 
$${\cal J}^{\bul}_{{\mathfrak E}_{\bul}}:={\rm Ker}
({\cal O}_{{\mathfrak E}_{\bul}}
\otimes_{{\cal O}_{{\cal Y}_{\bul}}}
\Om^{\bul}_{{\cal Y}_{\bul}/T}\lo 
{\cal O}_{Y_{\bul}}).$$ 
By using (\ref{ali:mygyp}), we obtain the following logarithm 
$$\log \col {\cal K}^{\times}_{{\mathfrak E}_{\bul}/T}\lo 
{\cal J}^{\bul}_{{\mathfrak E}_{\bul}}$$ 
as in \cite[(3.2)]{boi}. 
Because the following sequence 
\begin{align*} 
0\lo {\cal K}^{\times}_{{\mathfrak E}_{\bul}/T}\lo 
{\cal O}_{{\mathfrak E}_{\bul}}
\otimes_{{\cal O}_{{\cal Y}_{\bul}}}
\Om^{\times}_{{\cal Y}_{\bul}/T}\lo 
M^{\rm gp}_{Y_{\bul}}\lo 0
\end{align*} 
is exact, 
we obtain the following morphism
\begin{align*} 
M^{\rm gp}_Y=R\pi_{{\rm zar}*}(M^{\rm gp}_{Y_{\bul}})
\lo 
R\pi_{{\rm zar}*}({\cal K}^{\times}_{{\mathfrak E}_{\bul}/T})[1].
\end{align*} 
Hence we obtain the following composite morphism 
\begin{align*} 
M^{\rm gp}_{Y}\lo &
R\pi_{{\rm zar}*}({\cal K}^{\times}_{{\mathfrak E}_{\bul}/T})[1]
\os{\log}{\lo} R\pi_{{\rm zar}*}({\cal J}^{\bul}_{{\mathfrak E}_{\bul}})[1]
\lo R\pi_{{\rm zar}*}({\cal O}_{{\mathfrak E}_{\bul}}
\otimes_{{\cal O}_{{\cal Y}_{\bul}}}
\Om^{\bul}_{{\cal Y}_{\bul}/T})[1]\tag{19.0.3}\label{ali:zac}\\
&=Ru_{Y/T*}({\cal O}_{Y/T})[1].
\end{align*} 

In \cite[(1.8.17)]{nb} we have proved the following: 

\begin{prop}[{\bf \cite[(1.8.17)]{nb}}]\label{prop:ofc}
The first log crystalline Chern class map
$c_{1,{\rm crys}}$ in {\rm (\ref{ali:mygp})} is equal to 
the morphism {\rm (\ref{ali:zac})}. 
\end{prop}

\par 
Let the notations be as in \S\ref{sec:psc}.   
Next we define the first log crystalline Chern class map in 
$\wt{R}u_{Y/\os{\circ}{T}*}
({\cal O}_{Y/\os{\circ}{T}})[1]$ by replacing $T$ with $\os{\circ}{T}$.  
Namely, 
consider the following complex: 
\begin{align*} 
{\cal O}_{{\mathfrak E}_{\bul}}
\otimes_{{\cal O}_{{\cal Q}_{\bul}}}
\Om^{\times}_{{\cal Q}_{\bul}/\os{\circ}{T}}:=
(M^{\rm gp}_{{\mathfrak E}_{\bul}}
\os{d\log}{\lo} {\cal O}_{{\mathfrak E}_{\bul}}
\otimes_{{\cal O}_{{\cal Q}_{\bul}}}
\Om^1_{{\cal Q}_{\bul}/\os{\circ}{T}}
\os{d}{\lo}
{\cal O}_{{\mathfrak E}_{\bul}}
\otimes_{{\cal O}_{{\cal Q}_{\bul}}}
\Om^2_{{\cal Q}_{\bul}/\os{\circ}{T}}  \os{d}{\lo}\cdots).  
\end{align*}  
Set 
\begin{align*} 
{\cal K}^{\times}_{{\mathfrak E}_{\bul}/\os{\circ}{T}}
:={\rm Ker}({\cal O}_{{\mathfrak E}_{\bul}}
\otimes_{{\cal O}_{{\cal Q}_{\bul}}}
\Om^{\times}_{{\cal Q}_{\bul}/\os{\circ}{T}} 
\lo M^{\rm gp}_{Y_{\bul}})
\end{align*} 
and 
$${\cal J}^{\bul}_{{\mathfrak E}_{\bul}} 
:={\rm Ker}({\cal O}_{{\mathfrak E}_{\bul}}
\otimes_{{\cal O}_{{\cal Q}_{\bul}}}
\Om^{\bul}_{{\cal Q}_{\bul}/\os{\circ}{T}}\lo {\cal O}_{Y_{\bul}}).$$ 
Then we have the following logarithm morphism 
\begin{align*} 
\log \col {\cal K}^{\times}_{{\mathfrak E}_{\bul}/\os{\circ}{T}} 
\lo {\cal J}^{\bul}_{{\mathfrak E}_{\bul}}.  
\end{align*} 
and the following morphism as above: 
\begin{align*} 
M^{\rm gp}_{Y}\lo 
R\pi_{{\rm zar}*}({\cal K}^{\times}_{{\mathfrak E}_{\bul}/\os{\circ}{T}})[1].
\end{align*} 
Hence we obtain the following composite morphism 
\begin{align*} 
M^{\rm gp}_{Y}&\lo 
R\pi_{{\rm zar}*}({\cal K}^{\times}_{{\mathfrak E}_{\bul}/\os{\circ}{T}})[1]
\os{\log}{\lo} 
R\pi_{{\rm zar}*}({\cal J}^{\bul}_{{\mathfrak E}_{\bul}})[1]
\tag{19.1.1}\label{ali:upuwc}\\
&\lo R\pi_{{\rm zar}*}({\cal O}_{{\mathfrak E}_{\bul}}
\otimes_{{\cal O}_{{\cal Q}_{\bul}}}
\Om^{\bul}_{{\cal Q}_{\bul}/\os{\circ}{T}})[1]=
\wt{R}u_{Y/\os{\circ}{T}}({\cal O}_{Y/\os{\circ}{T}})[1]. 
\end{align*} 
We denote this morphism by $\wt{c}_{1,{\rm crys}}$. 
\par 
Finally we define the ``Hirsch extension'' of $\wt{c}_{1,{\rm crys}}$.
\par 
Let 
$\langle{\cal J}_{{\mathfrak E}_{\bul}},U_{S(T)^{\nat}}\rangle$ 
be the PD-ideal of 
$\Gam_{{\cal O}_{{\mathfrak E}_{\bul}}}
({\cal O}_{{\mathfrak E}_{\bul}}\otimes_{{\cal O}_T}U_{S(T)^{\nat}})$ 
generated by 
${\cal J}_{{\mathfrak E}_{\bul}}$ and 
$\Gam_{{\cal O}_{{\mathfrak E}_{\bul}},i}
({\cal O}_{{\mathfrak E}_{\bul}}\otimes_{{\cal O}_T}U_{S(T)^{\nat}})$
$(i\in {\mab Z}_{\geq 1})$.  
This is indeed a PD-ideal sheaf of 
$\Gam_{{\cal O}_{{\mathfrak E}_{\bul}}}
({\cal O}_{{\mathfrak E}_{\bul}}\otimes_{{\cal O}_T}U_{S(T)^{\nat}})$. 
Note that any local section of 
$\langle{\cal J}_{{\mathfrak E}_{\bul}}, U_{S(T)^{\nat}}\rangle$
is locally nilpotent. 
Consider the subgroup 
$1+\langle{\cal J}_{{\mathfrak E}_{\bul}},U_{S(T)^{\nat}}\rangle$ in 
$\Gam_{{\cal O}_{{\mathfrak E}_{\bul}}}
({\cal O}_{{\mathfrak E}_{\bul}}\otimes_{{\cal O}_T}U_{S(T)^{\nat}})^*$.  
Let 
$\al \col M_{{\mathfrak E}_{\bul}}\lo {\cal O}_{{\mathfrak E}_{\bul}}$ 
be the strucutral morphism. 
Identify  ${\cal O}_{{\mathfrak E}_{\bul}}^*$ with the image of it by $\al$. 
Set  
\begin{align*} 
M_{{\mathfrak E}_{\bul}}\langle u \rangle 
:=M_{{\mathfrak E}_{\bul}}\oplus_{1+{\cal J}_{{\mathfrak E}_{\bul}}}
(1+\langle{\cal J}_{{\mathfrak E}_{\bul}}, U_{S(T)^{\nat}}\rangle)
\end{align*} 
and $M^{\rm gp}_{{\mathfrak E}_{\bul}}\langle u \rangle 
:=(M_{{\mathfrak E}_{\bul}}\langle u \rangle)^{\rm gp}$. 
We obtain the following natural morphism 
\begin{align*} 
M^{\rm gp}_{{\mathfrak E}_{\bul}}\langle u \rangle 
\lo {\cal O}_{{\mathfrak E}_{\bul}}
\otimes_{{\cal O}_{{\cal Q}_{\bul}}}
\Om^1_{{\cal Q}_{\bul}/\os{\circ}{T}}\langle u \rangle 
\end{align*} 
by using the morphism 
$d\log \col M^{\rm gp}_{{\mathfrak E}_{\bul}}\lo \Om^1_{{\cal Q}_{\bul}/\os{\circ}{T}}$ 
and the local section $d\log t$. 
Consider the following complex: 
\begin{align*} 
M^{\rm gp}_{{\mathfrak E}_{\bul}}\langle u \rangle 
\lo {\cal O}_{{\mathfrak E}_{\bul}}
\otimes_{{\cal O}_{{\cal Q}_{\bul}}}
\Om^1_{{\cal Q}_{\bul}/\os{\circ}{T}}\langle u \rangle 
\lo
{\cal O}_{{\mathfrak E}_{\bul}}
\otimes_{{\cal O}_{{\cal Q}_{\bul}}}
\Om^2_{{\cal Q}_{\bul}/\os{\circ}{T}}\langle u \rangle \lo\cdots.  
\end{align*} 
Denote this complex by ${\cal O}_{{\mathfrak E}_{\bul}}
\otimes_{{\cal O}_{{\cal Q}_{\bul}}}
\Om^{\times}_{{\cal Q}_{\bul}/\os{\circ}{T}}\langle u \rangle$. 
Set 
\begin{align*} 
{\cal K}^{\times}_{{\mathfrak E}_{\bul}/\os{\circ}{T}}\langle u \rangle 
:={\rm Ker}({\cal O}_{{\mathfrak E}_{\bul}}
\otimes_{{\cal O}_{{\cal Q}_{\bul}}}
\Om^{\times}_{{\cal Q}_{\bul}/\os{\circ}{T}}\langle u \rangle
\lo M^{\rm gp}_{Y_{\bul}}). 
\end{align*} 
Set also 
$${\cal J}^{\bul}_{{\mathfrak E}_{\bul}}\langle u \rangle 
:={\rm Ker}({\cal O}_{{\mathfrak E}_{\bul}}
\otimes_{{\cal O}_{{\cal Q}_{\bul}}}
\Om^{\bul}_{{\cal Q}_{\bul}/\os{\circ}{T}}\langle u \rangle\lo {\cal O}_{Y_{\bul}}).$$ 
Then we have the following logarithm morphism 
\begin{align*} 
\log \col {\cal K}^{\times}_{{\mathfrak E}_{\bul}/\os{\circ}{T}}\langle u \rangle 
\lo {\cal J}^{\bul}_{{\mathfrak E}_{\bul}}\langle u \rangle.  
\end{align*} 
Because the following exact sequence 
\begin{align*} 
0\lo {\cal K}^{\times}_{{\mathfrak E}_{\bul}/\os{\circ}{T}}\langle u \rangle 
\lo {\cal O}_{{\mathfrak E}_{\bul}}
\otimes_{{\cal O}_{{\cal Q}_{\bul}}}
\Om^{\times}_{{\cal Q}_{\bul}/\os{\circ}{T}}\langle u \rangle \lo 
M^{\rm gp}_{Y_{\bul}}\lo 0
\end{align*} 
is exact, 
we obtain the following morphism
\begin{align*} 
M^{\rm gp}_{Y}\lo 
R\pi_{{\rm zar}*}({\cal K}^{\times}_{{\mathfrak E}_{\bul}/\os{\circ}{T}}\langle u \rangle )[1].
\end{align*} 
Hence we obtain the following composite morphism 
\begin{align*} 
M^{\rm gp}_{Y}&\lo 
R\pi_{{\rm zar}*}({\cal K}^{\times}_{{\mathfrak E}_{\bul}/\os{\circ}{T}}\langle u \rangle )[1]
\os{\log}{\lo} 
R\pi_{{\rm zar}*}({\cal J}^{\bul}_{{\mathfrak E}_{\bul}}\langle u \rangle)[1]
\tag{19.1.2}\label{ali:upc}\\
&\lo R\pi_{{\rm zar}*}({\cal O}_{{\mathfrak E}_{\bul}}
\otimes_{{\cal O}_{{\cal Q}_{\bul}}}
\Om^{\bul}_{{\cal Q}_{\bul}/\os{\circ}{T}}
\langle u \rangle)[1]=\wt{R}u_{Y/\os{\circ}{T}}({\cal O}_{Y/\os{\circ}{T}}\langle u \rangle)[1]. 
\end{align*} 

\begin{prop-defi}\label{prop-defi:ch}
The morphism (\ref{ali:upc}) is independent of the choices of 
an open covering of $s$ and 
the simplicial immersion $Y_{\bul}\os{\sus}{\lo}{\cal Q}_{\bul}$ 
into a log smooth scheme over $S(T)^{\nat}$. 
We call the morphism {\rm (\ref{ali:upc})} the {\it Hirsch extension} of 
$\wt{c}_{1,{\rm crys}}$. 
We denote it by  $\wt{c}_{1,{\rm crys}}\langle u \rangle$. 
\end{prop-defi}
\begin{proof}
To show the independence of the choices is a routine work, 
we leave the proof to the reader. 
\end{proof}

\begin{prop}\label{prop:hc}
The following diagram 
\begin{equation*} 
\begin{CD}
M_{Y}^{\rm gp}@>{\wt{c}_{1,{\rm crys}}\langle u \rangle}>>
\wt{R}u_{Y/\os{\circ}{T}}
({\cal O}_{Y/\os{\circ}{T}}\langle u \rangle)[1]\\
@| @VVV\\
M_{Y}^{\rm gp}@>{\wt{c}_{1,{\rm crys}}}>>
\wt{R}u_{Y/\os{\circ}{T}}
({\cal O}_{Y/\os{\circ}{T}})[1]\\
@| @VVV\\
M_{Y}^{\rm gp}@>{c_{1,{\rm crys}}}>>
Ru_{Y/S(T)^{\nat}}({\cal O}_{Y/S(T)^{\nat}})[1]
\end{CD}
\tag{19.3.1}\label{ali:mnh}
\end{equation*}
is commutative. 
\end{prop}
\begin{proof}
This follows from the following commutative diagram of exact sequences 
and the following two commutative diagrams: 
\begin{equation*} 
\begin{CD}
0@>>> {\cal K}^{\times}_{{\mathfrak E}_{\bul}/\os{\circ}{T}}\langle u \rangle 
@>>> {\cal O}_{{\mathfrak E}_{\bul}}
\otimes_{{\cal O}_{{\cal Q}_{\bul}}}
\Om^{\times}_{{\cal Q}_{\bul}/\os{\circ}{T}}\langle u \rangle @>>>
M^{\rm gp}_{Y_{\bul}}@>>> 0\\
@. @VVV @VVV @|\\
0@>>> {\cal K}^{\times}_{{\mathfrak E}_{\bul}/\os{\circ}{T}}
@>>> {\cal O}_{{\mathfrak E}_{\bul}}
\otimes_{{\cal O}_{{\cal Q}_{\bul}}}
\Om^{\times}_{{\cal Q}_{\bul}/\os{\circ}{T}} @>>>
M^{\rm gp}_{Y_{\bul}}@>>> 0\\
@. @VVV @VVV @|\\
0@>>> {\cal K}^{\times}_{{\mathfrak E}_{\bul}/S(T)^{\nat}}@>>>
{\cal O}_{{\mathfrak E}_{\bul}}
\otimes_{{\cal O}_{{\cal Q}_{\bul}}}
\Om^{\times}_{{\cal Q}_{\bul}/S(T)^{\nat}}@>>>
M^{\rm gp}_{Y_{\bul}}@>>> 0, 
\end{CD}
\end{equation*}

\begin{equation*} 
\begin{CD}
{\cal K}^{\times}_{{\mathfrak E}_{\bul}/\os{\circ}{T}}\langle u \rangle 
@>{\log}>> {\cal J}^{\bul}_{{\mathfrak E}_{\bul}}\langle u \rangle \\
@VVV @VVV \\
{\cal K}^{\times}_{{\mathfrak E}_{\bul}/\os{\circ}{T}}
@>{\log}>> {\cal J}^{\bul}_{{\mathfrak E}_{\bul}}  \\
@VVV @| \\
{\cal K}^{\times}_{{\mathfrak E}_{\bul}/S(T)^{\nat}} 
@>{\log}>> {\cal J}^{\bul}_{{\mathfrak E}_{\bul}} 
\end{CD}
\end{equation*}
and 
\begin{equation*} 
\begin{CD}
{\cal J}^{\bul}_{{\mathfrak E}_{\bul}}\langle u \rangle 
@>{\subset}>> {\cal O}_{{\mathfrak E}_{\bul}}
\otimes_{{\cal O}_{{\cal Q}_{\bul}}}
\Om^{\bul}_{{\cal Q}_{\bul}/\os{\circ}{T}}\langle u \rangle \\
@VVV @VVV \\
{\cal J}^{\bul}_{{\mathfrak E}_{\bul}} 
@>{\subset}>> {\cal O}_{{\mathfrak E}_{\bul}}
\otimes_{{\cal O}_{{\cal Q}_{\bul}}}
\Om^{\bul}_{{\cal Q}_{\bul}/\os{\circ}{T}} \\
@| @VVV \\
{\cal J}^{\bul}_{{\mathfrak E}_{\bul}} 
@>{\subset}>> {\cal O}_{{\mathfrak E}_{\bul}}
\otimes_{{\cal O}_{{\cal Q}_{\bul}}}
\Om^{\bul}_{{\cal Q}_{\bul}/S(T)^{\nat}}. 
\end{CD}
\end{equation*}
\end{proof}

\begin{prop}\label{prop:nl}
Let $Y/S$ and $T/S$ be as in {\rm \S\ref{sec:psc}}.
The following composite morphism 
\begin{align*} 
{\cal O}_{Y}^* \subset M_Y^{\rm gp} \os{c_{1,{\rm crys}}}{\lo}
Ru_{Y/S(T)^{\nat}*}({\cal O}_{Y/S(T)^{\nat}})[1] 
\os{N_{\rm zar}[1]}{\lo} 
Ru_{Y/S(T)^{\nat}*}({\cal O}_{Y/S(T)^{\nat}})[1] 
\tag{19.4.1}\label{ali:oyt}
\end{align*} 
is zero. 
Consequently $N_{\rm zar}(\lam)=0$ in 
$R^2g_{Y/S(T)^{\nat}*}({\cal O}_{Y/S(T)^{\nat}})$ 
for a local section of $\lam \in R^1g_{Y/S(T)^{\nat}*}({\cal O}_{Y}^*)$.
\end{prop}
\begin{proof} 
Let the notations be as in {\rm \S\ref{sec:psc}}.
Consider the following complex 
\begin{align*} 
{\cal O}_{{\mathfrak E}_{\bul}}\otimes_{{\cal O}_{{\cal Q}_{\bul}}}
\Om^{\times}_{{\cal Q}_{\bul}/S(T)^{\nat}}:=
{\cal O}_{{\mathfrak E}_{\bul}}^*\os{d\log}{\lo} 
{\cal O}_{{\mathfrak E}_{\bul}}\otimes_{{\cal O}_{{\cal Q}_{\bul}}}
\Om^1_{{\cal Q}_{\bul}/S(T)^{\nat}}\os{d}{\lo} 
{\cal O}_{{\mathfrak E}_{\bul}}\otimes_{{\cal O}_{{\cal Q}_{\bul}}}
\Om^2_{{\cal Q}_{\bul}/S(T)^{\nat}}\os{d}{\lo}\cdots.   
\end{align*} 
Set ${\cal O}_{{\mathfrak E}_{\bul}}\otimes_{{\cal O}_{{\cal Q}_{\bul}}}
{\cal K}^{\times}_{{\cal Q}_{\bul}/S(T)^{\nat}}
:={\rm Ker}({\cal O}_{{\mathfrak E}_{\bul}}\otimes_{{\cal O}_{{\cal Q}_{\bul}}}
\Om^{\times}_{{\cal Q}_{\bul}/S(T)^{\nat}}\lo 
{\cal O}^*_{Y_{\os{\circ}{T}_0\bul}})$. 
Then we have the following exact sequence: 
\begin{align*} 
0\lo {\cal O}_{{\mathfrak E}_{\bul}}\otimes_{{\cal O}_{{\cal Q}_{\bul}}}
{\cal K}^{\times}_{{\cal Q}_{\bul}/S(T)^{\nat}}
\lo {\cal O}_{{\mathfrak E}_{\bul}}\otimes_{{\cal O}_{{\cal Q}_{\bul}}}
\Om^{\times}_{{\cal Q}_{\bul}/S(T)^{\nat}}\lo {\cal O}^*_{Y_{\os{\circ}{T}_0\bul}}\lo 0.
\end{align*} 
Let $c_{1,{\rm dR}}\col 
{\cal O}^*_{Y_{\os{\circ}{T}_0\bul}}\lo 
{\cal O}_{{\mathfrak E}_{\bul}}\otimes_{{\cal O}_{{\cal Q}_{\bul}}}
\Om^{\bul}_{{\cal Q}_{\bul}/S(T)^{\nat}}[1]$ be 
the following composite morphism 
\begin{align*} 
{\cal O}^*_{Y_{\os{\circ}{T}_0\bul}}
\lo 
{\cal O}_{{\mathfrak E}_{\bul}}\otimes_{{\cal O}_{{\cal Q}_{\bul}}}
{\cal K}^{\times}_{{\cal Q}_{\bul}/S(T)^{\nat}}[1]\os{\log}{\lo}  
{\cal O}_{{\mathfrak E}_{\bul}}\otimes_{{\cal O}_{{\cal Q}_{\bul}}}
\Om^{\bul}_{{\cal Q}_{\bul}/S(T)^{\nat}}[1]. 
\tag{19.4.2}\label{ali:oykt}
\end{align*} 
By the log version of the argument of \cite[(3.3)]{boi}, 
the following diagram is commutative: 
\begin{equation*} 
\begin{CD}
{\cal O}_{Y}^*
@>{c_{1,{\rm crys}}}>>
Ru_{Y/S(T)^{\nat}*}
({\cal O}_{Y/S(T)^{\nat}})[1]\\
@| @|\\
{\cal O}_{Y}^*
@>{c_{1,{\rm dR}}}>>
R\pi_{{\rm zar}*}
({\cal O}_{{\mathfrak E}_{\bul}}\otimes_{{\cal O}_{{\cal Q}_{\bul}}}
\Om^{\bul}_{{\cal Q}_{\bul}/S(T)^{\nat}})[1]. 
\end{CD}
\end{equation*} 
Consider also the following complex 
\begin{align*} 
{\cal O}_{{\mathfrak E}_{\bul}}\otimes_{{\cal O}_{{\cal Q}_{\bul}}}
\Om^{\times}_{{\cal Q}_{\bul}/\os{\circ}{T}}:=
{\cal O}_{{\mathfrak E}_{\bul}}^*\os{d\log}{\lo} 
{\cal O}_{{\mathfrak E}_{\bul}}\otimes_{{\cal O}_{{\cal Q}_{\bul}}}
\Om^1_{{\cal Q}/\os{\circ}{T}}\os{d}{\lo} 
{\cal O}_{{\mathfrak E}_{\bul}}\otimes_{{\cal O}_{{\cal Q}_{\bul}}}
\Om^2_{{\cal Q}_{\bul}/\os{\circ}{T}}\os{d}{\lo}\cdots.   
\end{align*} 
Let $c'_{1,{\rm dR}}\col 
{\cal O}^*_{Y_{\os{\circ}{T}_0\bul}}
\lo {\cal O}_{{\mathfrak E}_{\bul}}\otimes_{{\cal O}_{{\cal Q}_{\bul}}}
\Om^{\bul}_{{\cal Q}_{\bul}/\os{\circ}{T}}[1]$ be 
the following composite morphism 
\begin{align*} 
{\cal O}^*_{Y_{\os{\circ}{T}_0\bul}}
\lo {\cal O}_{{\mathfrak E}_{\bul}}\otimes_{{\cal O}_{{\cal Q}_{\bul}}}
{\cal K}^{\times}_{{\cal Q}_{\bul}/\os{\circ}{T}}[1]\lo 
{\cal O}_{{\mathfrak E}_{\bul}}\otimes_{{\cal O}_{{\cal Q}_{\bul}}}
\Om^{\bul}_{{\cal Q}_{\bul}/\os{\circ}{T}}[1]. 
\end{align*} 
Obviously we have the following commutative diagram 
\begin{equation*} 
\begin{CD}
{\cal O}^*_{Y_{\os{\circ}{T}_0\bul}}
@>{c_{1,{\rm dR}}}>>
{\cal O}_{{\mathfrak E}_{\bul}}\otimes_{{\cal O}_{{\cal Q}_{\bul}}}
\Om^{\bul}_{{\cal Q}/S(T)^{\nat}}[1]@=
{\cal O}_{{\mathfrak E}_{\bul}}\otimes_{{\cal O}_{{\cal Q}_{\bul}}}
\Om^{\bul}_{{\cal Q}/S(T)^{\nat}}[1]\\
@| @AAA @AAA\\
{\cal O}^*_{Y_{\os{\circ}{T}_0\bul}}
@>{c'_{1,{\rm dR}}}>>
{\cal O}_{{\mathfrak E}_{\bul}}\otimes_{{\cal O}_{{\cal Q}_{\bul}}}
\Om^{\bul}_{{\cal Q}_{\bul}/\os{\circ}{T}}[1]
@>{\subset}>>
{\cal O}_{{\mathfrak E}_{\bul}}\otimes_{{\cal O}_{{\cal Q}_{\bul}}}
\Om^{\bul}_{{\cal Q}_{\bul}/\os{\circ}{T}}\langle u \rangle [1]. 
\end{CD}
\end{equation*} 
Hence we have the following commutative diagram 
\begin{equation*} 
\begin{CD}
{\cal O}^*_{Y_{\os{\circ}{T}_0\bul}}
@>{c_{1,{\rm dR}}}>>
{\cal O}_{{\mathfrak E}_{\bul}}\otimes_{{\cal O}_{{\cal Q}_{\bul}}}
\Om^{\bul}_{{\cal Q}_{\bul}/S(T)^{\nat}}[1]\\
@| @AAA \\
{\cal O}^*_{Y_{\os{\circ}{T}_0\bul}}
@>>> 
{\cal O}_{{\mathfrak E}_{\bul}}\otimes_{{\cal O}_{{\cal Q}_{\bul}}}
P_0\Om^{\bul}_{{\cal Q}_{\bul}/\os{\circ}{T}}\langle u \rangle [1]. 
\end{CD}
\end{equation*} 
By this commutative diagram 
we obtain the following commutative diagram
\begin{equation*} 
\begin{CD}
{\cal O}_{Y}^*
@>{c_{1,{\rm crys}}}>>
Ru_{Y/S(T)^{\nat}*}({\cal O}_{Y/S(T)^{\nat}})[1]\\
@| @AAA \\
{\cal O}_{Y}^*
@>>> 
P_0\wt{R}u_{Y/\os{\circ}{T}*}({\cal O}_{Y/\os{\circ}{T}}
\langle u \rangle)[1]. 
\end{CD}
\tag{19.4.3}\label{cd:pou} 
\end{equation*} 
By (\ref{cd:efnt}) and (\ref{ali:npr}) we obtain the following commutative diagram: 
\begin{equation*} 
\begin{CD}
Ru_{Y/S(T)^{\nat}*}({\cal O}_{{\cal Q}/S(T)^{\nat}})[1] 
@>{N_{\rm zar}[1]}>> 
Ru_{Y/S(T)^{\nat}*}({\cal O}_{{\cal Q}/S(T)^{\nat}})[1] \\
@AAA @AAA \\
P_0\wt{R}u_{Y/\os{\circ}{T}*}({\cal O}_{{\cal Q}/\os{\circ}{T}}
\langle u \rangle)[1]
@>{N_{\rm zar}[1]}>> 
P_{-2}\wt{R}u_{Y/\os{\circ}{T}*}({\cal O}_{{\cal Q}/\os{\circ}{T}}
\langle u \rangle)[1]=0. 
\end{CD}
\tag{19.4.4}\label{cd:poyu} 
\end{equation*} 
By (\ref{cd:pou}) and (\ref{cd:poyu}) we see that 
the composite morphism (\ref{ali:oyt}) is zero.  
\end{proof}

\par 
Let $X/S$ be an SNCL scheme. 
By (\ref{prop-defi:ch}) 
we have the following morphism 
\begin{align*} 
c_{1,{\rm crys}}\langle u \rangle \col 
M_{X_{\os{\circ}{T}_0}}^{\rm gp} \lo 
\wt{R}u_{X_{\os{\circ}{T}_0}/\os{\circ}{T}}
({\cal O}_{X_{\os{\circ}{T}_0}/\os{\circ}{T}}\langle u \rangle)[1].
\end{align*} 
Hence the natural morphism
\begin{align*}
\wt{R}u_{X_{\os{\circ}{T}_0}/\os{\circ}{T}}
({\cal O}_{X_{\os{\circ}{T}_0}/\os{\circ}{T}}\langle u \rangle)
\lo 
Ru_{X^{(\star)}_{\os{\circ}{T}_0}/\os{\circ}{T}}
({\cal O}_{X^{(\star)}_{\os{\circ}{T}_0}/\os{\circ}{T}}\langle u \rangle)
=H_{\rm zar}(X_{\os{\circ}{T}}/S(T)^{\nat})
\end{align*} 
induces the following morphism 
\begin{align*}  
c_{1,{\rm crys}}\langle u \rangle \col 
M_{X_{\os{\circ}{T}_0}}^{\rm gp} \lo 
H_{\rm zar}(X_{\os{\circ}{T}}/S(T)^{\nat})[1].
\end{align*} 
It is obvious that the following diagram is commutative: 
\begin{equation*} 
\begin{CD}
M_{X_{\os{\circ}{T}_0}}^{\rm gp}
@= M_{X_{\os{\circ}{T}_0}}^{\rm gp}@=M_{X_{\os{\circ}{T}_0}}^{\rm gp} \\
@V{c_{1,{\rm crys}}}VV @V{c_{1,{\rm crys}}\langle u \rangle}VV 
@V{c_{1,{\rm crys}}\langle u \rangle}VV\\
Ru_{X_{\os{\circ}{T}_0}/S(T)^{\nat}}({\cal O}_{X_{\os{\circ}{T}_0}/S(T)^{\nat}})[1]
@<{\simeq}<<\wt{R}u_{X_{\os{\circ}{T}_0}/\os{\circ}{T}}
({\cal O}_{X_{\os{\circ}{T}_0}/\os{\circ}{T}}\langle u \rangle)[1]
@>{\simeq}>> H_{\rm zar}(X_{\os{\circ}{T}_0}/S(T)^{\nat})[1]. 
\end{CD}
\end{equation*}

\begin{prop}\label{prop:cf}
Let $\lam$ be a a line bundle on $X_{\os{\circ}{T}_0}$.
Then the morphism 
\begin{align*} 
c_{1,{\rm crys}}\langle u \rangle (\lam) \cup \col 
H_{\rm zar}(X_{\os{\circ}{T}_0}/S(T)^{\nat},E)\lo 
H_{\rm zar}(X_{\os{\circ}{T}_0}/S(T)^{\nat},E)
\end{align*}
is the underlying morphism 
of the following filtered morphism 
\begin{align*} 
c_{1,{\rm crys}}\langle u \rangle (\lam) \cup \col 
(H_{\rm zar}(X_{\os{\circ}{T}_0}/S(T)^{\nat},E),P)
\lo (H_{\rm zar}(X_{\os{\circ}{T}_0}/S(T)^{\nat},E),P). 
\tag{19.5.1}\label{ali:n}
\end{align*}
\end{prop}
\begin{proof}
By the proof of (\ref{prop:nl}) 
we see that $c_{1,{\rm crys}}\langle u \rangle$ 
factors through 
$P_0H_{\rm zar}(X_{\os{\circ}{T}_0}/S(T)^{\nat})$. 
Hence $c_{1,{\rm crys}}\langle u \rangle (\lam)\cup $ is a filtered morphism 
$(H_{\rm zar}(X_{\os{\circ}{T}_0}/S(T)^{\nat}),P)
\lo (H_{\rm zar}(X_{\os{\circ}{T}_0}/S(T)^{\nat}),P)$ by (\ref{ali:te}). 
\end{proof} 

\par 
Next we construct a morphism 
\begin{align*} 
{\cal O}_X^*[-1]\otimes^L_{\mab Z}(A_{\rm zar}(X_{\os{\circ}{T}_0}/S(T)^{\nat},E),P)
\lo (A_{\rm zar}(X_{\os{\circ}{T}_0}/S(T)^{\nat}),P)
\tag{19.5.2}\label{ali:azox}
\end{align*} 
Here we endow ${\cal O}_X^*[-1]$ with the trivial filtration. 
To define this morphism, we prove the following: 

\begin{prop}\label{prop:tee}
Assume that $\os{\circ}{X}$ is quasi-compact. 
Endow $P_0
\wt{R}u_{X_{\os{\circ}{T}_0}/\os{\circ}{T}}
({\cal O}_{X_{\os{\circ}{T}_0}/\os{\circ}{T}})$ with the trivial filtration. 
Then there exists the following canonical morphism 
\begin{align*} 
P_0\wt{R}u_{X_{\os{\circ}{T}_0}/\os{\circ}{T}}
({\cal O}_{X_{\os{\circ}{T}_0}/\os{\circ}{T}})\otimes^L_{f^{-1}({\cal O}_T)}
(A_{\rm zar}(X_{\os{\circ}{T}_0}/S(T)^{\nat},E),P)
\lo (A_{\rm zar}(X_{\os{\circ}{T}_0}/S(T)^{\nat},E),P). 
\tag{19.6.1}\label{ali:azopx}
\end{align*} 
\end{prop}
\begin{proof}  
Let $X_{\os{\circ}{T}_0\bul}\os{\sus}{\lo} {\cal P}_{\bul}$ be the simplicial immersion 
in \S\ref{sec:psc}. 
Then we have the following wedge product
\begin{align*} 
\wedge \col 
P_0({\cal O}_{{\mathfrak D}_{\bul}}
\otimes_{{\cal O}_{{\cal P}^{\rm ex}_{\bul}}}
{\Om}^{q'}_{{\cal P}^{\rm ex}_{\bul}/\os{\circ}{T}})
\otimes_{f^{-1}_{\bul}({\cal O}_T)}
({\cal E}^{\bul}
\otimes_{{\cal O}_{{\cal P}^{\rm ex}_{\bul}}}
{\Om}^{q+1}_{{\cal P}^{\rm ex}_{\bul}/\os{\circ}{T}})
\lo 
{\cal E}^{\bul}
\otimes_{{\cal O}_{{\cal P}^{\rm ex}_{\bul}}}
{\Om}^{q'+q+1}_{{\cal P}^{\rm ex}_{\bul}/\os{\circ}{T}}. 
\end{align*} 
This morphism induces the following morphism 
\begin{align*} 
\wedge \col &
P_0({\cal O}_{{\mathfrak D}_{\bul}}
\otimes_{{\cal O}_{{\cal P}^{\rm ex}_{\bul}}}
{\Om}^{q'}_{{\cal P}^{\rm ex}_{\bul}/\os{\circ}{T}})\otimes_{f^{-1}_{\bul}({\cal O}_T)}
({\cal E}^{\bul}
\otimes_{{\cal O}_{{\cal P}^{\rm ex}_{\bul}}}
{\Om}^{q+1}_{{\cal P}^{\rm ex}_{\bul}/\os{\circ}{T}}/P_j)\\
&\lo {\cal E}^{\bul}
\otimes_{{\cal O}_{{\cal P}^{\rm ex}_{\bul}}}
{\Om}^{q'+q+1}_{{\cal P}^{\rm ex}_{\bul}/\os{\circ}{T}}/P_j \quad (j\in {\mab N})
\end{align*} 
and 
\begin{align*} 
&\wedge \col  P_0({\cal O}_{{\mathfrak D}_{\bul}}
\otimes_{{\cal O}_{{\cal P}^{\rm ex}_{\bul}}}
{\Om}^{q'}_{{\cal P}^{\rm ex}_{\bul}/\os{\circ}{T}})
\otimes_{f^{-1}_{\bul}({\cal O}_T)}
P_{2j+k+1}({\cal E}^{\bul}
\otimes_{{\cal O}_{{\cal P}^{\rm ex}_{\bul}}}
{\Om}^{q+1}_{{\cal P}^{\rm ex}_{\bul}/\os{\circ}{T}})
\\
&\lo 
P_{2j+k+1}({\cal E}^{\bul}
\otimes_{{\cal O}_{{\cal P}^{\rm ex}_{\bul}}}
{\Om}^{q'+q+1}_{{\cal P}^{\rm ex}_{\bul}/\os{\circ}{T}}) \quad (j,k\in {\mab N}). 
\end{align*} 
Hence we have the following morphism 
\begin{align*} 
&P_0
\wt{R}u_{X_{\os{\circ}{T}_0}/\os{\circ}{T}}
({\cal O}_{X_{\os{\circ}{T}_0}/\os{\circ}{T}})\otimes^L_{f^{-1}({\cal O}_T)}
(A_{\rm zar}(X_{\os{\circ}{T}_0}/S(T)^{\nat},E),P)\lo \\
&
R\pi_{{\rm zar}*}(P_0({\cal O}_{{\mathfrak D}_{\bul}}
\otimes_{{\cal O}_{{\cal P}^{\rm ex}_{\bul}}}
{\Om}^{\bul}_{{\cal P}^{\rm ex}_{\bul}/\os{\circ}{T}}))
\otimes_{f^{-1}({\cal O}_T)}
R\pi_{{\rm zar}*}((A_{\rm zar}({\cal P}_{\bul}/S(T)^{\nat},{\cal E}^{\bul}),P))
\lo\\
&
R\pi_{{\rm zar}*}(P_0({\cal O}_{{\mathfrak D}_{\bul}}
\otimes_{{\cal O}_{{\cal P}^{\rm ex}_{\bul}}}
{\Om}^{\bul}_{{\cal P}^{\rm ex}_{\bul}/\os{\circ}{T}})\otimes_{f^{-1}_{\bul}({\cal O}_T)}
(A_{\rm zar}({\cal P}_{\bul}/S(T)^{\nat},{\cal E}^{\bul}),P))
\os{\wedge}{\lo} \\
&
R\pi_{{\rm zar}*}((A_{\rm zar}({\cal P}_{\bul}/S(T)^{\nat},{\cal E}^{\bul}),P))=
(A_{\rm zar}(X_{\os{\circ}{T}_0}/S(T)^{\nat},E),P). 
\end{align*} 
\end{proof} 

\begin{prop}\label{prop:n}
There exists the canonical morphism {\rm (\ref{ali:azox})}. 
\end{prop}
\begin{proof}
By the proof of (\ref{prop:nl}), 
the morphism 
$\wt{c}_{1,{\rm crys}}\col {\cal O}_X^*[-1]\lo 
\wt{R}u_{X_{\os{\circ}{T}_0}/\os{\circ}{T}}
({\cal O}_{X_{\os{\circ}{T}_0}/\os{\circ}{T}})$ 
factors through the following morphism 
\begin{align*} 
\wt{c}_{1,{\rm crys}}\col {\cal O}_X^*[-1]\lo 
P_0\wt{R}u_{X_{\os{\circ}{T}_0}/\os{\circ}{T}}
({\cal O}_{X_{\os{\circ}{T}_0}/\os{\circ}{T}}). 
\end{align*} 
Hence we have the following composite morphism by (\ref{prop:tee}): 
\begin{align*}  
&{\cal O}_X^*[-1]\otimes^L_{\mab Z}
(A_{\rm zar}(X_{\os{\circ}{T}_0}/S(T)^{\nat},E),P) \\
& \lo P_0\wt{R}u_{X_{\os{\circ}{T}_0}/\os{\circ}{T}}
({\cal O}_{X_{\os{\circ}{T}_0}/\os{\circ}{T}})\otimes^L_{\mab Z}
(A_{\rm zar}(X_{\os{\circ}{T}_0}/S(T)^{\nat},E),P)\\
&\lo 
P_0\wt{R}u_{X_{\os{\circ}{T}_0}/\os{\circ}{T}}
({\cal O}_{X_{\os{\circ}{T}_0}/\os{\circ}{T}})\otimes^L_{f^{-1}({\cal O}_T)}
(A_{\rm zar}(X_{\os{\circ}{T}_0}/S(T)^{\nat},E),P)
\\
&\lo (A_{\rm zar}(X_{\os{\circ}{T}_0}/S(T)^{\nat},E),P).
\end{align*}
\end{proof}

\begin{prop}\label{prop:chc}
The induced morphism 
\begin{align*} 
&\us{j\geq \max \{-k,0\}}{\bigoplus} 
{\cal O}_X^*[-1]\otimes^L_{\mab Z}
a^{(2j+k)}_{T_0*} 
(Ru_{\os{\circ}{X}{}^{(2j+k)}_{T_0}/\os{\circ}{T}*}
(E_{\os{\circ}{X}{}^{(2j+k)}_{T_0}
/\os{\circ}{T}} 
\otimes_{\mab Z}\vp_{\rm crys}^{(2j+k)}
(\os{\circ}{X}_{T_0}/\os{\circ}{T})))[-2j-k]
\lo \\
&\us{j\geq \max \{-k,0\}}{\bigoplus}  
a^{(2j+k)}_{T_0*} 
(Ru_{\os{\circ}{X}{}^{(2j+k)}_{T_0}/\os{\circ}{T}*}
(E_{\os{\circ}{X}{}^{(2j+k)}_{T_0}
/\os{\circ}{T}} 
\otimes_{\mab Z}\vp_{\rm crys}^{(2j+k)}(\os{\circ}{X}_{T_0}/\os{\circ}{T})))[-2j-k]
\end{align*} 
by the morphism {\rm (\ref{ali:azox})} using the identification of 
$${\rm gr}^P_k(A_{\rm zar}(X_{\os{\circ}{T}_0}/S(T)^{\nat},E))$$
with 
$$\us{j\geq \max \{-k,0\}}{\bigoplus}  
a^{(2j+k)}_{T_0*} 
(Ru_{\os{\circ}{X}{}^{(2j+k)}_{T_0}/\os{\circ}{T}*}
(E_{\os{\circ}{X}{}^{(2j+k)}_{T_0}
/\os{\circ}{T}} 
\otimes_{\mab Z}\vp_{\rm crys}^{(2j+k)}
(\os{\circ}{X}_{T_0}/\os{\circ}{T})))[-2j-k]$$
induces the first crystalline Chern class map of $\os{\circ}{X}{}^{(2j+k)}_{T_0}$ 
on each direct factor.  
\end{prop}
\begin{proof} 
This is obvious by the construction of 
the morphism (\ref{ali:azox}). 
\end{proof}

\begin{defi}[{\bf cf.~\cite[(9.8)]{np}}]\label{defi:etclog}
Let $\lam :=c_{1,\star}({\cal L})\in 
R^2f_{X_{\os{\circ}{T}_0}/S(T)^{\nat}}({\cal O}_{X_{\os{\circ}{T}_0}/S(T)^{\nat}})$ 
be the log cohomology 
class of an invertible sheaf ${\cal L}$ on  $\os{\circ}{X}$.
We say that $\lam$ is compatible with the spectral sequence (\ref{eqn:espsp})
if the induced morphism of the left cup product of $\lam$ on the 
$E_1$-terms of (\ref{eqn:espsp}) is equal 
to the induced morphism of the restriction of ${\cal L}$ 
to various ${X}{}^{(k)}$'s $(k\in {\mab Z}_{>0})$.
\end{defi}

The following has an application for 
the weight spectral sequence (\ref{eqn:esafsp}) for the trivial coefficient, 
which has not been proved in \cite{np} (but which has been expected)
(see \cite{np} and \cite{a} for the application of this theorem):

\begin{coro}
The cohomology class $\lam$ is compatible with the spectral sequence 
{\rm (\ref{eqn:espsp})}. 
\end{coro}

\begin{prop}\label{prop:nr}
The following diagram 
\begin{equation*}
\begin{CD}
{\cal O}_X^*[-1]\otimes^L_{\mab Z}(A_{\rm zar}(X_{\os{\circ}{T}_0}/S(T)^{\nat}),P)@>>>(A_{\rm zar}(X_{\os{\circ}{T}_0}/S(T)^{\nat}),P)\\
@V{c_{1,{\rm crys}}\vert_{{\cal O}_X^*[-1]}\otimes^L\psi}VV  @VV{\psi}V\\
(H_{\rm zar}(X_{\os{\circ}{T}_0}/S(T)^{\nat}),P)
\otimes^L_{f^{-1}({\cal O}_T)}
(H_{\rm zar}(X_{\os{\circ}{T}_0}/S(T)^{\nat}),P)@>{\cup}>> 
(H_{\rm zar}(X_{\os{\circ}{T}_0}/\os{\circ}{S}),P)
\end{CD}
\tag{19.11.1}\label{cd:halp}
\end{equation*}
is commutative.  
\end{prop}
\begin{proof} 
Because the morphism $c_{1,{\rm crys}}\vert_{{\cal O}_X^*[-1]}$ factors through 
$R\pi_{{\rm zar}*}(P_0({\cal O}_{{\mathfrak D}_{\bul}}
\otimes_{{\cal O}_{{\cal P}^{\rm ex}_{\bul}}}
{\Om}^{\bul}_{{\cal P}^{\rm ex}_{\bul}/\os{\circ}{T}})$, 
there is no local residue of $c_{1,{\rm crys}}(\lam)$ for a line bundle $\lam$ 
on $\os{\circ}{X}_{T_0}$. 
By the local description of  $\psi^{\rm loc}$ (\ref{ali:pis}), 
we obtain the commutativity of (\ref{cd:halp}). 
\end{proof}

Let $L$ be a relatively ample line bundle on 
$\os{\circ}{X}_{T_0}/\os{\circ}{T}_0$; $L$ defines 
the cohomology class in $H^1(Y,{\cal O}_Y^*)$, which we denote by $L$ again. 
Consider the following composite morphism 
\begin{align*} 
{\cal O}_Y^*\os{\subset}{\lo} M_Y^{\rm gp}\os{c_{1,{\rm crys}}}{\lo}  
Ru_{Y/T*}({\cal O}_{Y/S(T)^{\nat}})[1], 
\tag{19.11.2}\label{ali:oyogp}
\end{align*} 
which we denote by $c_{1,{\rm crys}}$ again. 
Then we obtain the log cohomology class 
$\eta=c_{1,{\rm crys}}(L)$ of $L$
in $R^2f_{X_{\os{\circ}{T}_0}/S(T)^{\nat}*}
({\cal O}_{X_{\os{\circ}{T}_0}/S(T)^{\nat}})$. 

\par 
Let the notations be as in \S\ref{sec:e2}. 
In particular, $\os{\circ}{T}$ is assumed to be a $p$-adic formal ${\cal V}$-scheme. 
Assume that the relative dimension of 
$\os{\circ}{X}_{\os{\circ}{T}_1} \lo T_1$ is of pure dimension $d$. 
In \cite{nb} we have conjectured the following: 

\begin{conj}[{\bf $p$-adic variational filtered log hard Lefschetz conjecture}]\label{conj:lhpilc}
$(1)$ The following cup product 
\begin{equation*} 
\eta^i \col 
R^{d-i}f_{X_{\os{\circ}{T}_1}/S(T)^{\nat}*}
({\cal O}_{X_{\os{\circ}{T}_1}/S(T)^{\nat}})
\otimes_{\mab Z}{\mab Q} 
\lo (R^{d+i}f_{X_{\os{\circ}{T}_1}/S(T)^{\nat}*}
({\cal O}_{X_{\os{\circ}{T}_1}/S(T)^{\nat}})
\otimes_{\mab Z}{\mab Q})(i)
\tag{19.12.1}\label{eqn:filllpl} 
\end{equation*}
is an isomorphism. 
\par 
$(2)$ 
In fact, $\eta^i$ is the following isomorphism of filtered sheaves: 
\begin{equation*} 
\eta^i \col 
(R^{d-i}f_{X_{\os{\circ}{T}_1}/S(T)^{\nat}*}
({\cal O}_{X_{\os{\circ}{T}_1}/S(T)^{\nat}})\otimes_{\mab Z}{\mab Q},P) 
\os{\sim}{\lo} ((R^{d+i}f_{X_{\os{\circ}{T}_1}/S(T)^{\nat}*}
({\cal O}_{X_{\os{\circ}{T}_1}/S(T)^{\nat}})\otimes_{\mab Z}{\mab Q})(i),P). 
\tag{19.12.2}\label{eqn:filfill} 
\end{equation*}
\end{conj}

In the rest of this section 
we would like to prove that {\rm (\ref{conj:lhpilc}) (1)} implies {\rm (\ref{conj:lhpilc}) (2)}. 

\begin{theo}\label{theo:fdd}
Assume that $\os{\circ}{T}$ is the formal spectrum of 
the Witt ring of a perfect field $\kap$.  
Then {\rm (\ref{conj:lhpilc}) (1)} in this case implies {\rm (\ref{conj:lhpilc}) (2)} 
in this case. 
\end{theo}
\begin{proof} 
Let ${\cal A}_0$ be a smooth ${\mab F}_p$-algebra contained in $\kap$ 
such that $X$ has a proper SNCL model ${\cal X}$ over 
${\cal S}_0:=({\rm Spec}({\cal A}_0),({\mab N}\oplus {\cal A}^*_0\lo {\cal A}_0))$ 
(cf.~the argument before \cite[(3.1)]{ndw}) 
and $L$ comes from a line bundle ${\cal L}$ on ${\cal X}$: 
$X={\cal X}\times_{{\cal S}_0}s$ and $L={\cal L}\times_{{\cal S}_0}s$. 
Let ${\cal A}$ be a $p$-adically formally smooth algebra over ${\mab Z}_p$ 
which is a lift of ${\cal A}_0$. Endow ${\rm Spf}({\cal A})$ 
with a hollow log structure ${\mab N}\oplus {\cal A}^*\lo {\cal A}$ and 
let ${\cal S}$ be the resulting log formal scheme over 
${\rm Spf}({\mab Z}_p)=({\rm Spf}({\mab Z}_p), {\mab Z}_p^*)$. 
The log formal scheme ${\cal S}$ has a PD-ideal $p{\cal O}_{{\cal S}}$, 
which defines an exact closed immersion ${\cal S}_0 \os{\subset}{\lo} {\cal S}$.  
Let $\{{\cal P}^q_k\}_{k=0}^{\infty}$
be the weight filtration on 
$R^qf_{{\cal X}/{\cal A}*}({\cal O}_{{\cal X}/{\cal A}})_{\mab Q}$. 
We claim that 
\begin{align*} 
{\cal P}^{d+i}_k\cap {\rm Im}(c_{1,{\rm crys}}({\cal L})^i\cup)=
c_{1,{\rm crys}}({\cal L})^i\cup({\cal P}^{d-i}_{k-2i}).  
\tag{19.13.1}\label{ali:bkih}
\end{align*}
Indeed, by (\ref{prop:cf}), 
$c_{1,{\rm crys}}({\cal L})\in 
P_2R^2f_{{\cal X}/{\cal A}*}({\cal O}_{{\cal X}/{\cal A}})_{\mab Q}$. 
Hence, by (\ref{ali:otp}), the right hand side on 
(\ref{ali:bkih}) is contained in the left hand side of (\ref{ali:bkih}). 
By (\ref{exam:ofl}),  
${\cal P}^q_k$  and ${\cal P}^q_{\infty}$ 
extend to objects ${\cal P}^{q,{\rm conv}}_k$ and 
${\cal P}^{q,{\rm conv}}_{\infty}$ 
of $F{\textrm -}{\rm Isoc}({\rm Spec}({\cal A}_0)/{\mab Z}_p)$, respectively. 
Because the category $F{\textrm -}{\rm Isoc}({\rm Spec}({\cal A}_0)/{\mab Z}_p)$
is an abelian category (\cite[p.~795]{od}), 
both hand sides on (\ref{ali:bkih}) extend to objects of 
$F{\textrm -}{\rm Isoc}({\rm Spec}({\cal A}_0)/{\mab Z}_p)$. 
We may assume that ${\rm Spec}({\cal A}_0)$ is connected. 
By \cite[(4.1)]{od}, for a closed point $\os{\circ}{s}$ of ${\rm Spec}({\cal A}_0)$, 
the pull-back functor $\os{\circ}{s}{}^*\col {\rm Isoc}({\rm Spec}({\cal A}_0)/{\mab Z}_p)
\lo {\rm Isoc}(\os{\circ}{s}/{\mab Z}_p)$ is faithful. 
Hence it suffices to prove that 
\begin{align*} 
\os{\circ}{s}{}^*({\cal P}^{d+i,{\rm conv}}_k)\cap 
{\rm Im}(\os{\circ}{s}{}^*(c_{1,{\rm crys}}({\cal L})^i\cup))
=\os{\circ}{s}{}^*c_{1,{\rm crys}}({\cal L})^i\cup({\cal P}^{d-i,{\rm conv}}_{k-2i}). 
\tag{19.13.2}\label{ali:bshi}
\end{align*}
Consider the $p$-adic enlargement 
${\cal W}(\os{\circ}{s})$ 
of ${\rm Spec}({\cal A}_0)/{\mab Z}_p$ . 
Because ${\rm Spf}({\cal A})$ is formally smooth over ${\rm Spf}({\mab Z}_p)$, 
there exists a morphism  ${\cal W}(\os{\circ}{s})\lo {\rm Spf}({\cal A})$ 
fitting into the following commutative diagram
\begin{equation*} 
\begin{CD} 
\os{\circ}{s}@>{\subset}>> {\cal W}(\os{\circ}{s}) \\
@VVV @VVV \\
{\rm Spec}({\cal A}_0) @>>> {\rm Spf}({\cal A}) \\
@VVV @VVV \\
{\rm Spec}({\mab F}_p) @>>> {\rm Spf}({\mab Z}_p).  
\end{CD} 
\end{equation*}  
Then the log structure of ${\cal W}(s)$ is 
the inverse image of the log structure of 
$({\rm Spf}({\cal A}),({\mab N}\oplus {\cal A}^*\lo {\cal A}))$. 
By the proof of \cite[(5.1.30)]{nb}, we see that 
\begin{align*} 
(\os{\circ}{s}{}^*({\cal P}{}^{\rm conv}_{\infty}))_{{\cal W}(\os{\circ}{s})}=
R^qf_{{\cal X}_{s}/{\cal W}(s)*}
({\cal O}_{{\cal X}_{s}/{\cal W}(s)})_{\mab Q}
\tag{19.13.3}\label{ali:bsiznh}
\end{align*}
and 
\begin{align*} 
(\os{\circ}{s}{}^*({\cal P}{}^{\rm conv}_k))_{{\cal W}(\os{\circ}{s})}=
P_kR^qf_{{\cal Z}_{s}/{\cal W}(s)*}
({\cal O}_{{\cal Z}_{s}/{\cal W}(s)})_{\mab Q}.
\tag{19.13.4}\label{ali:bnush}
\end{align*}
Since the residue field $k(\os{\circ}{s})$ is a finite field, 
the equality of (\ref{ali:bshi}) at the value ${\cal W}(\os{\circ}{s})$ holds by 
the existence of the weight spectral sequence (\ref{eqn:espsp})
and by the purity of the weight as in the proof of \cite[(2.18.2)]{nh2} 
and the action of the pull-back of the absolute Frobenius endomorphism
of ${\cal X}$ on the line bundle is the multiplication by $p$. 
This shows the equality (\ref{ali:bshi}) 
because ${\cal W}(\os{\circ}{s})$ is a formally smooth lift 
of $\os{\circ}{s}$ over ${\rm Spf}({\mab Z}_p)$. 
\end{proof}

\begin{theo}\label{theo:llh}
{\rm (\ref{conj:lhpilc}) (1)} implies {\rm (\ref{conj:lhpilc}) (2)}. 
\end{theo}
\begin{proof} 
By using (\ref{theo:fdd}), 
one can obtain (\ref{theo:llh}) as in (\ref{theo:stpbgb}) 
whose proof is the same as that of \cite[(5.4.7)]{nb}. 
\end{proof}

\begin{coro}\label{coro:nilll}
Let the notations be as in {\rm (\ref{conj:lhpilc})} if 
there exists an exact closed point $t$ 
of each connected component of $S(T)^{\nat}$ such that 
the fiber ${\cal X}_t$ of ${\cal X}$ over $t$ is the log special fiber of 
a projective strict semistable family over a complete discrete valuation ring of 
equal characteristic or mixed characteristics. 
\end{coro}
\begin{proof}
In \cite[(5.5.19)]{nb} (resp.~\cite[(5.5.16)]{nb}) we have proved that the conjecture 
(\ref{conj:lhpilc}) (1) is true if 
there exists an exact closed point $t$ 
of each connected component of $S(T)^{\nat}$ such that 
the fiber ${\cal X}_t$ of ${\cal X}$ over $t$ is the log special fiber of 
a projective strict semistable family over a complete discrete valuation ring of 
mixed characteristics (resp.~equal characteristic). 
Now (\ref{coro:nilll}) follows from (\ref{theo:llh}). 
\end{proof} 

The following is essentially the same as 
(\ref{coro:inp}):

\begin{coro}\label{coro:wekl}
Let the assumptions be as in {\rm (\ref{coro:nilll})}. 
Let $k$ be a nonnegative integer and let $i$ be a nonnegative integer. 
Set
\begin{align*} 
(P_kR^{d-i}f_{X_{\os{\circ}{T}_1}/S(T)^{\nat}*}
({\cal O}_{X_{\os{\circ}{T}_1}/S(T)^{\nat}})
\otimes_{\mab Z}{\mab Q})_0 
&:={\rm Ker}(\eta^{i+1} \col 
P_kR^{d-i}f_{X_{\os{\circ}{T}_1}/S(T)^{\nat}*}
({\cal O}_{X_{\os{\circ}{T}_1}/S(T)^{\nat}})
\otimes_{\mab Z}{\mab Q} 
 \tag{19.16.1}\label{ali:fcltpl} \\
&\lo P_k\{R^{d+i+2}f_{X_{\os{\circ}{T}_1}/S(T)^{\nat}*}
({\cal O}_{X_{\os{\circ}{T}_1}/S(T)^{\nat}})
\otimes_{\mab Z}{\mab Q} 
)(i+2)\}). 
\end{align*}
Then 
\begin{align*}
P_kR^qf_{X_{\os{\circ}{T}_1}/S(T)^{\nat}*}
({\cal O}_{X_{\os{\circ}{T}_1}/S(T)^{\nat}})
\otimes_{\mab Z}{\mab Q}
=&(P_kR^qf_{X_{\os{\circ}{T}_1}/S(T)^{\nat}*}
({\cal O}_{X_{\os{\circ}{T}_1}/S(T)^{\nat}})
\otimes_{\mab Z}{\mab Q})_0  \\
&\oplus 
\eta (P_{k-2}R^{q-2}f_{X_{\os{\circ}{T}_1}/S(T)^{\nat}*}
({\cal O}_{X_{\os{\circ}{T}_1}/S(T)^{\nat}})
\otimes_{\mab Z}{\mab Q})_0 \\
&\oplus \eta (P_{k-4}R^{q-4}f_{X_{\os{\circ}{T}_1}/S(T)^{\nat}*}
({\cal O}_{X_{\os{\circ}{T}_1}/S(T)^{\nat}})
\otimes_{\mab Z}{\mab Q})_0 
\oplus \cdots.   
\end{align*} 
Moreover, assume that $\os{\circ}{T}$ is connected. 
Let $b_k^q$ be the rank of locally free sheaf 
$$P_kR^qf_{X_{\os{\circ}{T}_1}/S(T)^{\nat}*}
({\cal O}_{X_{\os{\circ}{T}_1}/S(T)^{\nat}})
\otimes_{\mab Z}{\mab Q}$$ on $\os{\circ}{T}.$
Then 
$b_k^0\leq b_{k+2}^2\leq b_{k+4}^4 \leq \cdots $ 
up to $b_*^d$ and 
$b_k^1\leq b_{k+2}^3\leq b_{k+4}^5 \leq \cdots$ 
up to $b_{*}^d$. 
In particular, 
$b_{\infty}^0\leq b_{\inf}^2\leq b_{\inf}^4 \leq \cdots $ 
up to $b_{\inf}^d$ and 
$b_{\inf}^1\leq b_{\inf}^3\leq b_{\inf}^5 \leq \cdots$ 
up to $b_{\inf}^d$. 
\end{coro}
\begin{proof} 
The argument is well-known and easy. 
Indeed, (\ref{coro:wekl}) immediately 
follows from the following obvious commutative diagram 
\begin{equation*} 
\begin{CD}
R^{d-i-2}f_{X_{\os{\circ}{T}_1}/S(T)^{\nat}*}
({\cal O}_{X_{\os{\circ}{T}_1}/S(T)^{\nat}})
\otimes_{\mab Z}{\mab Q} @>{\eta^{i+2},\sim}>>R^{d+i+2}f_{X_{\os{\circ}{T}_1}/S(T)^{\nat}*}
({\cal O}_{X_{\os{\circ}{T}_1}/S(T)^{\nat}})
\otimes_{\mab Z}{\mab Q} \\
@V{\eta}VV @AA{\eta}A\\
R^{d-i}f_{X_{\os{\circ}{T}_1}/S(T)^{\nat}*}
({\cal O}_{X_{\os{\circ}{T}_1}/S(T)^{\nat}})
\otimes_{\mab Z}{\mab Q}  @>{\eta^i,\sim}>>R^{d+i}f_{X_{\os{\circ}{T}_1}/S(T)^{\nat}*}
({\cal O}_{X_{\os{\circ}{T}_1}/S(T)^{\nat}})
\otimes_{\mab Z}{\mab Q} 
\end{CD}
\end{equation*}
and the strict compatibility of $\eta$ with respect to the weight filtration ((\ref{eqn:filfill})).
\end{proof}

\section{$p$-adic polarizations on log crystalline cohomology sheaves of projective SNCL schemes}\label{sec:pol}
In this section we give results for the $p$-adic variational monodromy-conjecture 
and the $p$-adic variational filtered log hard Lefschetz conjecture for 
log crystalline cohomologcal sheaves of SNCL schemes in certain cases 
if the standard conjecture due to Grothendieck is true. 
See also \cite{a} and see \cite{sap} for the $l$-adic case.  
Let the notations be as in \S\ref{sec:e2}. 
Consider the following cup product 
\begin{align*} 
( ?,? )   \col &
R^{q}f_{X_{\os{\circ}{T}_1}/S(T)^{\nat}*}({\cal O}_{X_{\os{\circ}{T}_1}/S(T)^{\nat}})\otimes_{{\cal O}_T}
R^{q'}f_{X_{\os{\circ}{T}_1}/S(T)^{\nat}*}({\cal O}_{X_{\os{\circ}{T}_1}/S(T)^{\nat}})
\tag{20.0.1}\label{ali:obop}\\
& \lo 
R^{q+q'}f_{X_{\os{\circ}{T}_1}/S(T)^{\nat}*}({\cal O}_{X_{\os{\circ}{T}_1}/S(T)^{\nat}})
\quad (q,q'\in {\mab N}).
\end{align*}  

\begin{prop}\label{prop:epp} 
Let the notations be as before {\rm (\ref{prop:nl})}. 
For simplicity of notation, let us denote $N_{\rm zar}$ simply by $N$. 
Then the following hold$:$
\par 
$(1)$ 
Let $L$ be a line bundle on $X_{\os{\circ}{T}_1}$. Set 
$\lam=c_{1,{\rm crys}}(L)$ and 
$$l:=\lam \cup \col 
R^qf_{X_{\os{\circ}{T}_1}/S(T)^{\nat}*}({\cal O}_{X_{\os{\circ}{T}_1}/S(T)^{\nat}})\lo 
R^{q+2}f_{X_{\os{\circ}{T}_1}/S(T)^{\nat}*}({\cal O}_{X_{\os{\circ}{T}_1}/S(T)^{\nat}}).$$ 
Then  
$(l(x),y) =(x, l(y)) =l((x,y))$. 
\par 
$(2)$ 
$N \circ  l =l \circ N$. 
\end{prop}
\begin{proof}
(1): (1) follows from (\ref{eqn:dutmd}).  
\par
(2): By (\ref{cd:efntt}), 
$(N \lam )\cup x+ l (Nx)=N(l(x))$ for a local section $x$ of 
$R^q f_{X_{\os{\circ}{T}_1}/S(T)^{\nat}*}({\cal O}_{X_{\os{\circ}{T}_1}/S(T)^{\nat}})$. 
By (\ref{prop:nl}), $N\lam=0$. Hence 
$l (Nx)=N(l(x))$. This means that $l\circ N=N \circ l $. 
\end{proof}

Now assume that 
$\os{\circ}{X}/\os{\circ}{S}$ is projective and of pure relative dimension $d$. 
Let the notations be as in \S\ref{sec:e2}. 
Then we have the following pairing: 
\begin{align*} 
\langle ?,? \rangle   \col &
R^qf_{X_{\os{\circ}{T}_1}/S(T)^{\nat}*}({\cal O}_{X_{\os{\circ}{T}_1}/S(T)^{\nat}})_{\mab Q}\otimes_{{\cal K}_T}
R^{2d-q}f_{X_{\os{\circ}{T}_1}/S(T)^{\nat}*}({\cal O}_{X_{\os{\circ}{T}_1}/S(T)^{\nat}})_{\mab Q}
\tag{20.1.1}\label{ali:ocop}\\
&\lo 
R^{2d}f_{X_{\os{\circ}{T}_1}/S(T)^{\nat}*}({\cal O}_{X_{\os{\circ}{T}_1}/S(T)^{\nat}})_{\mab Q}
\os{{\rm Tr}^A_{X_{\os{\circ}{T}_1}/S(T)^{\nat}}}{\lo} {\cal K}_T(-d). 
\end{align*}  
Here ${\rm Tr}^A_{X_{\os{\circ}{T}_1}/S(T)^{\nat}}$ 
is the trace morphism defined in (\ref{defi:ffc}).

\begin{prop}\label{prop:tr}
The following formulas hold$:$
\begin{align*} 
\langle N(x),y\rangle +\langle x,N(y)\rangle =0
\tag{20.2.1}\label{ali:nddd}
\end{align*}  
for local sections 
$x \in 
R^qf_{X_{\os{\circ}{T}_1}/S(T)^{\nat}*}({\cal O}_{X_{\os{\circ}{T}_1}/S(T)^{\nat}})_{\mab Q}$
and 
$y\in 
R^{2d-q}f_{X_{\os{\circ}{T}_1}/S(T)^{\nat}*}({\cal O}_{X_{\os{\circ}{T}_1}/S(T)^{\nat}})_{\mab Q}$ 
and 
\begin{align*} 
\langle l(x),y\rangle =\langle x,l(y)\rangle
\tag{20.2.2}\label{ali:l}
\end{align*}  
for local sections 
$x \in 
R^{q-2}f_{X_{\os{\circ}{T}_1}/S(T)^{\nat}*}
({\cal O}_{X_{\os{\circ}{T}_1}/S(T)^{\nat}})_{\mab Q}$
and 
$y\in R^{2d-q}f_{X_{\os{\circ}{T}_1}/S(T)^{\nat}*}
({\cal O}_{X_{\os{\circ}{T}_1}/S(T)^{\nat}})_{\mab Q}$. 
\end{prop}
\begin{proof} 
By (\ref{cd:efntt}),  
$(N(x),y) +(x,N(y)) =N((x,y))$. 
By (\ref{eqn:eise2d}), 
\begin{align*}
R^{2d}f_{X_{\os{\circ}{T}_1}/S(T)^{\nat}*}({\cal O}_{X_{\os{\circ}{T}_1}/S(T)^{\nat}})
=E^{0,2d}_2. 
\end{align*} 
By (\ref{prop:pad}) (1), $E^{-k,2d+k}_1=0$ for $k\not=0$. 
Because $N$ induces a morphism 
$N\col E^{0,2d}_2\lo E^{2,2d-2}_2$, this is the zero morphism. 
Hence (\ref{ali:nddd}) follows. 
The formula (\ref{ali:l}) follows from (\ref{prop:epp}) (1).  
\end{proof}

\par 
In the following let the notations be as in (\ref{conj:lhpilc}). 
Let us recall the following: 

\begin{theo}[{\bf Hard Lefschetz theorem for crystalline cohomologies}]\label{theo:hlc}
Let $Y$ be a projective smooth scheme over $\os{\circ}{T}_1$ 
of pure dimension $e$. Let $L$ be a relative ample line bundle on $Y/\os{\circ}{T}_1$.  
Let $\eta=c_{1,{\rm crys}}(L)$ be the crystalline Chern class of $L$
in $R^2f_{Y/\os{\circ}{T}*}({\cal O}_{Y/\os{\circ}{T}})$ {\rm (\cite[\S3]{boi})}. 
Then the following cup product 
\begin{equation*} 
\eta^i \col 
R^{e-i}f_{Y/\os{\circ}{T}*}({\cal O}_{Y/\os{\circ}{T}})_{\mab Q} 
\lo R^{e+i}f_{Y/\os{\circ}{T}*}({\cal O}_{Y/\os{\circ}{T}})_{\mab Q}(i)
\tag{20.3.1}\label{eqn:fcvcilpl} 
\end{equation*}
is an isomorphism. 
\end{theo}
\begin{proof} 
As in the proof of (\ref{theo:fdd}),  
this follows from the usual hard Lefschetz theorem for crystalline cohomologies  
of a projective smooth scheme over a finite field in \cite{kme}. 
(However there is a gap in the proof of weak Lefschetz theorem due to Berthelot  
as already mentioned in the proof of (\ref{theo:e2dam}).)
\end{proof}

\par 
Assume that $\os{\circ}{X}{}^{(0)}_{T_1}$ is projective over $\os{\circ}{T}_1$. 
Then $\os{\circ}{X}{}^{(k)}_{T_1}$ $(k\in {\mab N})$ 
is also projective over $\os{\circ}{T}_1$. 
Following \cite[(2.8)]{gn}, 
set    
\begin{equation*} 
H^{-k,q,j}:=
\begin{cases} 
R^{q+d-2j-k}
f_{\os{\circ}{X}{}^{(2j+k)}_{\os{\circ}{T}_1}/S(T)^{\nat}*}
({\cal O}_{\os{\circ}{X}{}^{(2j+k)}_{T_1}/\os{\circ}{T}})\otimes_{\mab Z}{\mab Q}(-j-k) & (j\geq \max\{0,-k\}), \\
0 & {\rm otherwise}.
\end{cases}
\end{equation*}  
(This is a picking out of the direct factor of the $E_1$-term of 
the weight spectral sequence (\ref{eqn:esafsp}) 
twisted by $d$ for the case of  the trivial coefficient.) 
We define the following morphisms
$$d'\col H^{-k,q,j}\lo H^{-k+1,q+1,j+1},$$ 
$$d''\col H^{-k,q,j}\lo H^{-k+1,q+1,j},$$ 
$$\nu \col H^{-k,q,j}\lo H^{-k+2,q,j+1}(-1)$$ 
by the following formula: 
$$d'(x)=\rho(x),\quad d''(x)=G(x),\quad \nu(x):={\rm id}$$
for a local section $x$ of 
$R^{q+d-2j-k}
f_{\os{\circ}{X}{}^{(2j+k)}_{\os{\circ}{T}_1}/S(T)^{\nat}*}
({\cal O}_{\os{\circ}{X}{}^{(2j+k)}_{T_1}/\os{\circ}{T}})
\otimes_{\mab Z}{\mab Q}(-j-k)$. 
Here the morphism $d''$ is different from 
the morphism $d''$ in \cite[p.~145]{gn} about signs. 
Because one can easily check that 
\begin{equation*} 
H^{k+2,q,j+k+1}=
\begin{cases} 
0& (j=0), \\
H^{-k,q,j} & (j\geq 1),
\end{cases} 
\end{equation*}  
\begin{equation*} 
{\rm Ker}(\nu^{k+1}\col H^{-k,q,j}\lo H^{k+2,q,j+k+1}(-k-1))
:=
\begin{cases} 
H^{-k,q,0}& (j=0), \\
0 & (j\geq 1).
\end{cases}
\tag{20.3.2}\label{eqn:filpl} 
\end{equation*}

Let $E^{\bul \bul}_1$ be the $E_1$-term of the spectral sequence (\ref{eqn:esafsp}) 
for the case of  the trivial coefficient. 
Set 
\begin{align*} 
H^{-k,q}&:=\bigoplus_{j\geq \max\{0,-k\}}H^{-k,q,j}
=E_1^{-k,q+d+k}\otimes_{\mab Z}{\mab Q} \quad (k,q\in {\mab Z}).  
\end{align*}

Set $H:=\bigoplus_{k,q\in {\mab Z}}H^{-k,q}$, 
$l_1:=\nu \col H^{-k,q}\lo H^{-k+2,q}(-1)$ and $l_2:=l$.
By the construction of $\nu$, $l_1^k= {\rm id}\col H^{-k,q}\lo H^{k,q}$ 
is an isomorphism for $k\in {\mab N}$. 
By (\ref{theo:hlc}) $l_2^q\col H^{k,-q}\lo H^{k,q}$ is an isomorphism for $q\in {\mab N}$.. 
Obviously $l_1\circ l_2=l_2\circ l_1$. 
Hence $(H,l_1,l_2)$ satisfies the axiom of the bigraded Lefschetz module in 
\cite[(4.1)]{gn} (see also the notion of the (bi)graded polarized Hodge structure in 
\cite[(4.1.1), (4.2.1)]{sam}). 
Set 
\begin{align*} 
H^{-k,-q}_0:=H^{-k,-q}\cap {\rm Ker}(l_1^{k+1})\cap {\rm Ker}(l_2^{q+1}). 
\end{align*} 
Then we have the following Lefschetz decomposition 
\begin{align*} 
H^{kq}=\bigoplus_{r,s\in {\mab N}}l_1^rl_2^sH^{k-2r,q-2s}_0. 
\end{align*}
By (\ref{eqn:filpl}) 
\begin{align*}
H^{-k,-q}_0:=H^{-k,-q,0}\cap {\rm Ker}(l_2^{q+1})
&={\rm Ker}(l_2^{q+1}\col R^{d-k-q}
f_{\os{\circ}{X}{}^{(k)}_{\os{\circ}{T}_0}/S(T)^{\nat}*}
({\cal O}_{\os{\circ}{X}{}^{(k)}_{T_0}/\os{\circ}{T}})\otimes_{\mab Z}{\mab Q}(-k)\\
&\lo 
R^{d-k+q+2}
f_{\os{\circ}{X}{}^{(k)}_{\os{\circ}{T}_0}/S(T)^{\nat}*}
({\cal O}_{\os{\circ}{X}{}^{(k)}_{T_0}/\os{\circ}{T}})\otimes_{\mab Z}{\mab Q}(-k+q+1)\\
&=(R^{d-k-q}
f_{\os{\circ}{X}{}^{(k)}_{\os{\circ}{T}_0}/S(T)^{\nat}*}
({\cal O}_{\os{\circ}{X}{}^{(k)}_{T_0}/\os{\circ}{T}})\otimes_{\mab Z}{\mab Q})_0(-k). 
\end{align*} 
Here ${}_0$ means the primitive part.  
Let 
\begin{align*}
\langle ?,?\rangle \col H\otimes_{{\cal K}_T}H\lo {\cal K}_T
\end{align*}
be the morphism obtained by 
the Poincar\'{e} duality morphism for crystalline cohomologies (\cite[(3.12)]{od}).

Set 
\begin{equation*}
d:=G+\rho 
\col H^{-k,q}\lo H^{-k+1,q}
\tag{20.3.3}\label{eqn:grhobd}
\end{equation*}
as in (\ref{prop:deccbd}). 
Then $d^2=0$. 
We also have the following relation 
$$dl_i=l_id \quad(i=1,2),$$
\par 
Set 
\begin{align*} 
\eps(a):=(-1)^{\frac{a(a-1)}{2}} \quad (a\in {\mab Z}). 
\end{align*}

For local sections $x$ of 
$$H^{-k,q,j}=R^{q+d-2j-k}
f_{\os{\circ}{X}{}^{(2j+k)}_{\os{\circ}{T}_0}/S(T)^{\nat}*}
({\cal O}_{\os{\circ}{X}{}^{(2j+k)}_{T_0}/\os{\circ}{T}})(-j-k)\otimes_{\mab Z}{\mab Q}$$ 
and $y$ of 
$$H^{k,-q,j+k}
=R^{-q+d-2j-k}
f_{\os{\circ}{X}{}^{(2j+k)}_{\os{\circ}{T}_0}/S(T)^{\nat}*}
({\cal O}_{\os{\circ}{X}{}^{(2j+k)}_{T_0}/\os{\circ}{T}})(-j)\otimes_{\mab Z}{\mab Q},$$
set 
$$\psi(x,y):=\eps(-k-q-d)\langle x,y\rangle:=
\eps(-k-q-d)\langle x,y\rangle_{(j,k)} \in {\cal K}_T(-d),$$
where $\langle x,y\rangle:=\langle ?,?\rangle_{(j,k)}$ 
means the Poincare duality morphism for 
$\os{\circ}{X}{}^{(2j+k)}_{T_0}/\os{\circ}{T}$: 
\begin{align*}
&R^{q+d-2j-k}
f_{\os{\circ}{X}{}^{(2j+k)}_{\os{\circ}{T}_0}/S(T)^{\nat}*}
({\cal O}_{\os{\circ}{X}{}^{(2j+k)}_{T_0}/\os{\circ}{T}})(-j-k)\otimes_{\mab Z}{\mab Q}
\otimes_{{\cal K}_T}\\
&R^{-q+d-2j-k}
f_{\os{\circ}{X}{}^{(2j+k)}_{\os{\circ}{T}_0}/S(T)^{\nat}*}
({\cal O}_{\os{\circ}{X}{}^{(2j+k)}_{T_0}/\os{\circ}{T}})(-j)\otimes_{\mab Z}{\mab Q}
\os{\cup}{\lo} \\
&R^{2(d-2j-k)}
f_{\os{\circ}{X}{}^{(2j+k)}_{\os{\circ}{T}_0}/S(T)^{\nat}*}
({\cal O}_{\os{\circ}{X}{}^{(2j+k)}_{T_0}/\os{\circ}{T}})(-2j-k)\otimes_{\mab Z}{\mab Q}
\os{\rm Tr(-2j-k)}{\lo} {\cal K}_T(-d). 
\end{align*}

The following is a $p$-adic analogue of \cite[(3.5)]{gn}: 

\begin{prop}[{\bf cf.~\cite[(3.5)]{gn}}]\label{prop:an}
The following relations hold$:$
\par 
$(1)$ $\psi(y,x)=(-1)^d\psi(x,y)$.
\par 
$(2)$ $\psi(l_1y,x)+\psi(x,l_1y)=0$.
\par 
$(3)$ $\psi(l_2y,x)+\psi(x,l_2y)=0$.
\par 
$(4)$ $\psi(d'y,x)+\psi(x,d''y)=0$.
\par 
$(5)$ $\psi(d''y,x)+\psi(x,dy)=0$.
\end{prop}
\begin{proof} 
The proof is the same as that of \cite[(3.5)]{gn} as follows. 
\par 
(1): (1) follows from 
$\langle x,y\rangle=(-1)^{(d-k)^2-q^2}\langle y,x\rangle
=(-1)^{d+k+q}\langle y,x\rangle$ and 
\begin{align*}
\dfrac{(-k-q-d)^2-(-k-q-d)+(k+q-d)^2-(k+q-d)}{2}
&=k^2+q^2+d^2+2kq+d\\
&=k+q\mod 2.
\end{align*} 
\par 
(2): For local sections $x$ of $H^{-(k-2),q,j}$ 
and $y$ of $H^{k,-q,j+k}$, 
$\psi(l_1y,x)=\eps(-k-q-d)\langle x,y\rangle$ and 
$\psi(x,l_1y)=\eps(-(k-2)-q-d)\langle x,y\rangle$. 
It is easy to see that $\eps(a+2)\eps(a)=-1$. 
\par 
(3): We can easily prove that (3) holds as in (2). 
\par 
(4): For local sections $x$ of 
$H^{-(k-1),q-1,j}$
and $y$ of 
$H^{k,-q,j+k}$
$\psi(d'x,y)=\eps(-k-q-d)\langle d'x,y\rangle$ 
$\psi(x,d''y)=\eps(-(k-1)-(q-1)-d)\langle x,d''y\rangle=
\eps(-k-q+2-d)\langle x,d''y\rangle=-\eps(-k-q-d)\langle x,d''y\rangle$.  
\par 
(5): (5) follows from (1) and (4). 
\end{proof}


\begin{theo}\label{theo:pos}
Assume that $\os{\circ}{T}$ is a $p$-adic formal ${\cal V}$-scheme. 
For an exact closed point $t$ of $S(T)^{\nat}$, let $K_0(t)$ be the fraction field 
of the Witt ring ${\cal W}(\kap_t)$, where $\kap_t$ is the residue field of $t$. 
Assume that there exists an exact closed point $t$ of 
each connected component of $S(T)^{\nat}$ such that 
for any $\os{\circ}{X}{}^{(k)}$ $(k\in {\mab N})$ and for any $i\in {\mab N}$,  
the following Chern class morphism  
\begin{align*} 
{\rm CH}^i(\os{\circ}{X}{}^{(k)}_t)\otimes_{\mab Z}K_0(t)
\lo 
H^{2i}_{\rm crys}(\os{\circ}{X}{}^{(k)}_t/{\cal W}(\kap_t))
\otimes_{{\cal W}(\kap_t)}K_0(t)
\end{align*} 
is surjective 
and 
$H^{2i+1}_{\rm crys}(\os{\circ}{X}{}^{(k)}_t/{\cal W}(\kap_t))
\otimes_{{\cal W}(\kap_t)}K_0(t)=0$ 
and the Hodge standard conjecture by Grothendieck is true for 
$\os{\circ}{X}{}^{(k)}_t$. Then the variational $p$-adic monodromy-weight conjecture and 
the variational $p$-adic filtered log hard Lefschetz conjecture are true. 
\end{theo}
\begin{proof} 
By the assumptions we can define a argument about positivity due to 
Steenbrink-Satio (\cite{sam}) for 
$H^{2i}_{\rm crys}(\os{\circ}{X}{}^{(k)}_t/{\cal W}(\kap_t))
\otimes_{{\cal W}(\kap_t)}K_0(t)$ as in the case where the base field is ${\mab C}$ 
(cf.~\cite{iti}, \cite{rx}). 
Hence the $p$-adic monodromy-weight conjecture and the $p$-adic log hard Lefschetz conjecture are true 
for $X_t$ (cf.~[loc. cit.]). 
Consequently (\ref{theo:pos}) follows as in the proof of (\ref{theo:fdd}). 
\end{proof}


\section{Comparison theorems between our filtered complex (resp.~
our filtered dga) with Kim-Hain's filtered complex (resp.~their filtered dga)}\label{sec:ofc}
In this section we give the right construction of 
$(s(C(W\wt{\om}[u])),P)$ in \cite{kiha}. 
The method of our construction is quite different from the construction 
of $(s(C(W\wt{\om}[u])),P)$ in [loc.~cit.] because we use results in \cite{nb} and 
the previous sections and because 
the local log smooth scheme for our construction of 
$(s(C(W\wt{\om}[u])),P)$ is not necessarily a log smooth lift of a given log smooth scheme. 
Our method does not need the admissible lift defined in \cite{msemi} nor 
the local calculation of the log blow up of the product of two admissible lifts in 
\cite{hyp} (and \cite{msemi}), and  simplifies arguments about the construction and 
the properties of several log de Rham-Witt complexes. 
The main results in this section are comparison theorems between $(s(C(W\wt{\om}[u])),P)$ 
and $(s_{\rm TW}(C(W\wt{\om}[u])\otimes_{\mab Z}{\mab Q}),P)$
with our filtered complex $(H_{\rm zar},P)$ 
and our filtered dga $(H_{{\rm zar},{\rm TW}},P)$ in previous sections, 
respectively, in the case where 
the base log scheme is the log point of a perfect field of characteristic $p>0$. 
As a result we see that $(H_{\rm zar},P)$ and $(H_{{\rm zar},{\rm TW}},P)$
are generalizations of $(s(C(W\wt{\om}[u])),P)$ and 
$(s_{\rm TW}(C(W\wt{\om}[u])\otimes_{\mab Z}{\mab Q}),P)$, respectively, 
to the case where the base log scheme is more general.   
\par 
Let the notations be as in \S\ref{sec:ldfc}. 
In this section we assume that $S$ is the log point $s$ of a perfect field 
$\kap$ of characteristic $p>0$. Let ${\cal W}_n$ $(n\geq 1)$ be the Witt ring of 
$\kap$ and let ${\cal W}_n(s)$ be the canonical lift of $s$ over ${\cal W}_n$. 
Let $Y$ be a log smooth scheme of Cartier type over $s$. 
Let ${\cal W}_n\Om^{\bul}_Y$ and ${\cal W}_n\wt{\Om}^{\bul}_Y$ be the 
log de Rham-Witt complex of $Y/s$ defined in \cite{hyp}, \cite{msemi} and \cite{ndw}. 
Let $\theta_n=$``$[d\log t]$''$\in {\cal W}_n\wt{\Om}^1_Y$ be 
the canonical closed one form defined in 
\cite{msemi}. 
Let ${\cal W}_n(Y)$ be the canonical lift of $Y$ over ${\cal W}_n$. 
Let $F_{{\cal W}_n(Y)}\col {\cal W}_n(Y)\lo {\cal W}_n(Y)$ be the Frobenius endomorphism of 
${\cal W}_n(Y)$.
The sheaf ${\cal W}_n\wt{\Om}^i_Y$ 
$(i\in {\mab N})$ defined in \cite{msemi} and \cite{nb} 
is a ${\cal W}_n({\cal O}_Y)$-module, 
where ${\cal W}_n({\cal O}_Y)$ is the ``reverse'' Witt sheaf of $Y$ 
in the sense of \cite{ndw}, that is, 
${\cal W}_n({\cal O}_Y)={\cal H}^0(\Om^{\bul}_{{\cal Y}_n/{\cal W}_n(s)})$ locally, 
where ${\cal Y}_n$ is a local log smooth lift of $Y$ over ${\cal W}_n(s)$. 
Let ${\cal W}_n({\cal O}_Y)'$ be 
the ``obverse'' Witt sheaf in the sense of \cite[\S7]{ndw}, 
that is,  ${\cal W}_n({\cal O}_Y)'$ is the structure sheaf of ${\cal W}_n(Y)$.  
Let 
\begin{align*}
C^{-1}\col {\cal W}_n({\cal O}_Y)'\os{\sim}{\lo} 
{\cal W}_n({\cal O}_Y)
\end{align*} 
be the inverse Cartier isomorphism in \cite[(7.5)]{ndw}. 
Let $F_n$ be the evaluation of the crystal $F$ of ${\cal O}_{Y/{\cal W}_n(s)}$-module 
at the object ${\cal W}_n(Y)$ of 
the crystalline site of $Y$ over $({\cal W}_n(s), p{\cal W}_n,[~])$. 
By using this isomorphism, 
$F_n$ becomes a ${\cal W}_n({\cal O}_Y)$-module and each 
$F_n\otimes_{{\cal W}_n({\cal O}_Y)}{\cal W}_n\wt{\Om}^i_Y$ 
is a ${\cal W}_n({\cal O}_Y)$-module. 
Consequently $F_n\otimes_{{\cal W}_n({\cal O}_Y)}{\cal W}_n\wt{\Om}^i_Y$ 
is a ${\cal W}_n({\cal O}_Y)'$-module and we obtain 
the complex $F_n\otimes_{{\cal W}_n({\cal O}_Y)}{\cal W}_n\wt{\Om}^{\bul}_Y$ 
of ${\cal W}_n$-modules. 
Let $F_n\otimes_{{\cal W}_n({\cal O}_Y)}
{\cal W}_n\wt{\Om}^{\bul}_Y\langle u\rangle$ be the Hirsch extension 
of $F_n\otimes_{{\cal W}_n({\cal O}_Y)}
{\cal W}_n\wt{\Om}^{\bul}_Y$ with respect to 
the $g^{-1}({\cal W}_n)$-linear morphism 
$\varphi \col {\cal W}_nu\owns u \lom \theta_n \in 
{\rm Ker}({\cal W}_n\wt{\Om}^1_Y\lo {\cal W}_n\wt{\Om}^2_Y)$. 
Here, to define the $g^{-1}({\cal W}_n)$-linearity, we use the following composite morphism 
\begin{align*} 
{\cal W}_n\lo {\cal W}_n({\cal O}_Y)'\os{C^{-1},\sim}{\lo} {\cal W}_n({\cal O}_Y). 
\end{align*}

\par 
Let us recall the following: 
\begin{prop}[{\bf \cite[(2.2.7)]{nb}}]\label{prop:t} 
Let $R\col {\cal W}_{n+1}\wt{\Om}^1_{Y}\lo 
{\cal W}_n\wt{\Om}^{\bul}_{Y}$ be the projection defined in 
{\rm \cite[(2.2.6)]{nb}}. 
Then $R(\theta_{n+1})=\theta_n$. 
\end{prop}

\par 
In the following we always assume that 
F is a flat locally nilpotent quasi-coherent ${\cal O}_{Y/{\cal W}_n(s)}$-module. 

The following is a generalization of Kim-Hain's result 
in \cite[Lemma 6]{kiha} (however see (\ref{rema:kh}) below): 

\begin{theo}\label{theo:khr}
The natural morphism 
\begin{align*} 
F_n\otimes_{{\cal W}_n({\cal O}_Y)}{\cal W}_n\wt{\Om}^{\bul}_Y\langle u \rangle 
\lo 
F_n\otimes_{{\cal W}_n({\cal O}_Y)}{\cal W}_n\Om^{\bul}_Y
\end{align*} 
is a quasi-isomorphism. 
\end{theo} 
\begin{proof} 
As in 
\cite[(1.4.3)]{hyp} 
we have the following exact sequence
\begin{align*} 
0  \lo F_n\otimes_{{\cal W}_n({\cal O}_Y)}{\cal W}_n\Om^{\bul}_Y[-1] 
\os{\theta_n \wedge }{\lo}  
F_n\otimes_{{\cal W}_n({\cal O}_Y)}{\cal W}_n\wt{\Om}^{\bul}_Y 
\lo F_n\otimes_{{\cal W}_n({\cal O}_Y)}{\cal W}_n\Om^{\bul}_Y \lo 0. 
\tag{21.2.1}\label{ali:agwxd}
\end{align*} 
We can also derive this exact sequence immediately from 
(\ref{ali:agxhqd}) by using the local definitions of 
${\cal W}_n\wt{\Om}^q_Y$ and ${\cal W}_n\Om^q_Y$ $(q\in {\mab N})$. 
(In the following argument we do not use the injectivity of 
$\theta_n \wedge$ in [loc.~cit.]; instead we use the injectivity in a more general situation, 
that is, we use the injectivity of $d\log t$ in (\ref{ali:agxhqd}).)
We claim that the following sequence 
\begin{equation*} 
\small{
0  \lo {\cal H}^{q-1}(
F_n\otimes_{{\cal W}_n({\cal O}_Y)}{\cal W}_n\Om^{\bul}_Y) 
\os{\theta_n \wedge }{\lo} 
{\cal H}^q(F_n\otimes_{{\cal W}_n({\cal O}_Y)}{\cal W}_n\wt{\Om}{}^{\bul}_Y)  
\lo {\cal H}^q(F_n\otimes_{{\cal W}_n({\cal O}_Y)}{\cal W}_n\Om^{\bul}_Y) \lo 0
\quad (q\in {\mab N})} 
\tag{21.2.2}\label{ali:agxahqd}
\end{equation*} 
obtained by (\ref{ali:agwxd}) is exact. 
Indeed, this is a local problem. We may assume that there exists an 
immersion $Y\os{\sus}{\lo} \ol{\cal Q}$ into a log smooth scheme over ${\cal W}_n(s)$. 
Let $\ol{\mathfrak E}$ be the log PD-envelope of this immersion over 
$({\cal W}_n(\os{\circ}{s}),p{\cal W}_n,[~])$. 
Let $(\ol{\cal F},\ol{\nabla})$ be the corresponding ${\cal O}_{\ol{\mathfrak E}}$-module 
with integrable connection. 
Set $({\cal F},\nabla):=
(\ol{\cal F},\ol{\nabla})\otimes_{{\mathfrak D}(\ol{{\cal W}_n(s)})}{\cal W}_n$ 
and ${\cal Q}:=\ol{\cal Q}\times_{\ol{{\cal W}_n(s)}}{\cal W}_n(s)$. 
By the log Poincar\'{e} lemma and \cite[(2.1.12.1), (2.2.14.1)]{nb}, 
we have the following equalities:  
\begin{align*} 
{\cal F}
\otimes_{{\cal O}_{\cal Q}}
{\Om}^{\bul}_{{\cal Q}/{\cal W}_n(s)}
=Ru_{Y/{\cal W}_n(s)*}(\eps^*_{Y/{\cal W}_n(s)/{\cal W}_n}(F)) 
=F_n\otimes_{{\cal W}_n({\cal O}_Y)}{\cal W}_n\Om^{\bul}_Y 
\end{align*} 
and 
\begin{align*} 
{\cal F}
\otimes_{{\cal O}_{\cal Q}}
{\Om}^{\bul}_{{\cal Q}/{\cal W}_n(\os{\circ}{s})}
=\wt{R}u_{Y/{\cal W}_n*}(F) 
=F_n\otimes_{{\cal W}_n({\cal O}_Y)}{\cal W}_n\wt{\Om}^{\bul}_Y. 
\end{align*} 
Hence the sequence (\ref{ali:agxahqd}) is equal to the exact sequence 
(\ref{ali:agxhqd}) in the special case. 
Consequently the sequence (\ref{ali:agxahqd}) is exact. 
The rest of the proof is the same as that of (\ref{theo:qii}). 
\end{proof} 

\begin{rema}\label{rema:kh}
As pointed out in \cite[(3.36) (2)]{ey}, 
there is a gap in the proof of \cite[Lemma 6]{kiha}.  
(\ref{theo:khr}) fills this gap. See also (\ref{rema:kihainc}) and (\ref{rema:b}) (1).  
\end{rema}

Set 
\begin{align*} 
{\cal W}_n\langle \langle u\rangle \rangle:=
\prod_{m\in {\mab N}}{\cal W}_nu^{[m]}. 
\end{align*} 
and 
\begin{equation*}  
F_n\otimes_{{\cal W}_n({\cal O}_Y)}{\cal W}_n\wt{\Om}^{\bul}_Y
\langle \langle u \rangle \rangle:=
{\cal W}_n\langle \langle u \rangle \rangle \otimes_{{\cal W}_n}
F_n\otimes_{{\cal W}_n({\cal O}_Y)}{\cal W}_n\wt{\Om}^{\bul}_Y, 
\end{equation*} 
where the differential of 
$F_n\otimes_{{\cal W}_n({\cal O}_Y)}{\cal W}_n\wt{\Om}^{\bul}_Y
\langle \langle u \rangle \rangle$ 
is defined as in (\ref{eqn:badff}).

\begin{theo}\label{theo:kchr}
The natural morphism 
\begin{align*} 
F_n\otimes_{{\cal W}_n({\cal O}_Y)}{\cal W}_n\wt{\Om}^{\bul}_Y\langle \langle u \rangle \rangle 
\lo 
F_n\otimes_{{\cal W}_n({\cal O}_Y)}{\cal W}_n\Om^{\bul}_Y
\end{align*} 
is a quasi-isomorphism. 
\end{theo} 
\begin{proof} 
By (\ref{ali:agwxd}) we see that the following sequence 
\begin{align*} 
0&\lo F_n
\os{\theta_n\wedge}{\lo}   
F_n\otimes_{{\cal W}_n({\cal O}_Y)}{\cal W}_n\wt{\Om}^1_Y
\os{d\log t \wedge}{\lo}
\cdots 
\os{\theta_n\wedge}{\lo}
F_n\otimes_{{\cal W}_n({\cal O}_Y)}{\cal W}_n\wt{\Om}^{j-2}_Y
\tag{21.4.1}\label{ali:gulwtpt}\\
&\os{d\log t \wedge}{\lo}  
F_n\otimes_{{\cal W}_n({\cal O}_Y)}{\cal W}_n\wt{\Om}^{j-1}_Y
\os{\theta_n\wedge}{\lo}  F_n\otimes_{{\cal W}_n({\cal O}_Y)}{\cal W}_n\Om^j_Y   \lo 0.   
\end{align*} 
is exact for $j\in {\mab N}$. 
Because ${\cal W}_n\Om^i_Y$ and ${\cal W}_n\wt{\Om}^i_Y$ $(i\in {\mab N})$ 
are quasi-coherent ${\cal W}_n({\cal O}_Y)$-modules by the local definitions of 
${\cal W}_n\Om^i_Y$ and ${\cal W}_n\wt{\Om}^i_Y$, the following sequence 
\begin{align*} 
0&\lo \Gam(V,F_n)
\os{\theta_n\wedge}{\lo}   
\Gam(V,F_n\otimes_{{\cal W}_n({\cal O}_Y)}{\cal W}_n\wt{\Om}^1_Y)
\os{\theta_n\wedge}{\lo}
\cdots 
\os{\theta_n\wedge}{\lo}
\Gam(V,F_n\otimes_{{\cal W}_n({\cal O}_Y)}{\cal W}_n\wt{\Om}^{j-2}_Y)
\tag{21.4.2}\label{ali:gulkwtpt}\\
&\os{\theta_n\wedge}{\lo}  
\Gam(V,F_n\otimes_{{\cal W}_n({\cal O}_Y)}{\cal W}_n\wt{\Om}^{j-1}_Y)
\os{\theta_n\wedge}{\lo}  
\Gam(V,F_n\otimes_{{\cal W}_n({\cal O}_Y)}{\cal W}_n\Om^j_Y)   \lo 0 
\end{align*}  
are exact for any log open affine subscheme of $Y$. 
The rest of the proof is the same as that of (\ref{theo:saih}). 
\end{proof}

\begin{coro}\label{coro:n}
The natural morphism 
\begin{align*} 
F_n\otimes_{{\cal W}_n({\cal O}_Y)}{\cal W}_n\wt{\Om}^{\bul}_Y
\langle u \rangle \os{\subset}{\lo} 
F_n\otimes_{{\cal W}_n({\cal O}_Y)}{\cal W}_n
\wt{\Om}^{\bul}_Y\langle \langle u \rangle \rangle
\tag{21.5.1}\label{ali:uuweni} 
\end{align*} 
is a quasi-isomorphism. 
\end{coro}

\begin{theo}\label{theo:cp}
There exists a canonical isomorphism 
\begin{align*} 
\wt{R}u_{Y/{\cal W}_n*}(F\langle u \rangle) 
\os{\sim}{\lo} 
F_n\otimes_{{\cal W}_n({\cal O}_Y)}{\cal W}_n
\wt{\Om}^{\bul}_Y\langle u \rangle. 
\tag{21.6.1}\label{ali:ruyw}
\end{align*} 
This isomorphism is compatible with the pull-back of a morphism
$Y\lo Y'$ over $s\lo s'$, where $Y'/s'$ is an analogous object to
$Y/s;$ it is also compatible with the projections 
$\wt{R}u_{Y/{\cal W}_n*}(F\langle u \rangle) \lo 
\wt{R}u_{Y/{\cal W}_{n-1}*}(F\langle u \rangle)$ 
and  
$F_n\otimes_{{\cal W}_n({\cal O}_Y)}{\cal W}_n
\wt{\Om}^{\bul}_Y\langle u\rangle
\lo 
F_{n-1}\otimes_{{\cal W}_{n-1}({\cal O}_Y)}{\cal W}_{n-1}
\wt{\Om}^{\bul}_Y\langle u\rangle$. 
\end{theo}
\begin{proof} 
We leave the detailed proof to the reader because the proof of this theorem 
is the same as that of \cite[(2.2.13)]{nb} if one uses 
(\ref{coro:hac}). 
\end{proof}

Let the notations be as in \S\ref{sec:psc}. 
We define the filtered complex
$(s(C(W\wt{\om}[u])),P)$ defined in \cite[Lemma 12]{kiha} 
from our point of view as follows. 
To define it, we develop theory of log de Rham-Witt complexes 
${\cal W}_n\wt{\Om}{}^{\bul}_{X_{\ul{\lam}}}$ and 
${\cal W}_n\Om^{\bul}_{X_{\ul{\lam}}}$ for an element $\ul{\lam}$ of $P(\Lam)$. 
(In this section we assume that $\ul{\lam}\not=\emptyset$.)
Here note that though $X_{\ul{\lam}}$ is not log smooth over $s$, 
it is ideally log smooth over $s$. 
\par 
First let us recall the filtered complex 
$$(\wt{R}u_{X_{\ul{\lam}}/{\cal W}_n(\os{\circ}{s})*}
(\eps^*_{X_{\ul{\lam}}/{\cal W}_n(\os{\circ}{s})}
({\cal O}_{\os{\circ}{X}_{\ul{\lam}}/{\cal W}_n(\os{\circ}{s})})),P)$$
(a special case of (\ref{prop:nlgr})).  
Set 
\begin{align*} 
{\cal W}_n\wt{\Om}{}^i_{X_{\ul{\lam}}}:=
{\cal H}^i(\wt{R}u_{X_{\ul{\lam}}/{\cal W}_n(\os{\circ}{s})*}
(\eps^*_{X_{\ul{\lam}}/{\cal W}_n(\os{\circ}{s})}
({\cal O}_{\os{\circ}{X}_{\ul{\lam}}/{\cal W}_n(\os{\circ}{s})})))
\quad (i\in {\mab N})
\tag{21.6.2}\label{ali:wnoi}
\end{align*} 
and 
\begin{align*} 
P_k{\cal W}_n\wt{\Om}{}^i_{X_{\ul{\lam}}}:=
{\cal H}^i(P_k\wt{R}u_{X_{\ul{\lam}}/{\cal W}_n(\os{\circ}{s})*}
(\eps^*_{X_{\ul{\lam}}/{\cal W}_n(\os{\circ}{s})}
({\cal O}_{\os{\circ}{X}_{\ul{\lam}}/{\cal W}_n(\os{\circ}{s})})))
\quad (i\in {\mab N},k\in {\mab Z})
\tag{21.6.3}\label{ali:wnwoi}
\end{align*} 
as abelian sheaves on $\os{\circ}{X}$. 
Set $P_{\infty}{\cal W}_n\wt{\Om}{}^i_{X_{\ul{\lam}}}:=
{\cal W}_n\wt{\Om}{}^i_{X_{\ul{\lam}}}$. 
(I have not yet claimed that 
$P_k{\cal W}_n\wt{\Om}{}^i_{X_{\ul{\lam}}}$ is a subsheaf of 
${\cal W}_n\wt{\Om}{}^i_{X_{\ul{\lam}}}$.) 

\begin{prop}\label{prop:cpp}
There exists a well-defined product structure 
\begin{align*} 
\wedge \col {\cal W}_n\wt{\Om}{}^i_{X_{\ul{\lam}}}\times 
{\cal W}_n\wt{\Om}{}^j_{X_{\ul{\lam}}}
\lo {\cal W}_n\wt{\Om}{}^{i+j}_{X_{\ul{\lam}}}
\tag{21.6.1}\label{ali:pppxw}
\end{align*} 
which induces a morphism 
\begin{align*} 
\wedge \col P_k{\cal W}_n\wt{\Om}{}^i_{X_{\ul{\lam}}}\times 
P_{k'}{\cal W}_n\wt{\Om}{}^j_{X_{\ul{\lam}}}
\lo P_{k+k'}{\cal W}_n\wt{\Om}{}^{i+j}_{X_{\ul{\lam}}}. 
\tag{21.6.2}\label{ali:pppw}
\end{align*} 
\end{prop} 
\begin{proof}
Consider a local immersion $X\os{\sus}{\lo} {\cal P}$ into 
a log smooth scheme over ${\cal W}_n(s)$. 
Let ${\mathfrak D}$ be the log PD-envelope of this immersion over 
$({\cal W}_n(s),p{\cal W}_n,[~])$. 
Then, by the definition (\ref{ali:wnoi}),  
\begin{align*}
{\cal W}_n\wt{\Om}{}^i_{X_{\ul{\lam}}} 
:={\cal H}^i({\cal O}_{{\mathfrak D}}\otimes_{{\cal O}_{{\cal P}{}^{\rm ex}}}
\Om^{\bul}_{{\cal P}{}^{{\rm ex}}_{\ul{\lam}}/{\cal W}_n(\os{\circ}{s})}). 
\end{align*} 
Let $\om_1 \in {\cal O}_{{\mathfrak D}}
\otimes_{{\cal O}_{{\cal P}{}^{\rm ex}}}
\Om^i_{{\cal P}{}^{{\rm ex}}_{\ul{\lam}}/{\cal W}_n(\os{\circ}{s})}$ 
and 
$\om_2 \in {\cal O}_{{\mathfrak D}}
\otimes_{{\cal O}_{{\cal P}{}^{\rm ex}}}
\Om^j_{{\cal P}{}^{{\rm ex}}_{\ul{\lam}}/{\cal W}_n(\os{\circ}{s})}$ 
be cocycles.  
Let $[\om_1]\in {\cal H}^i({\cal O}_{{\mathfrak D}}
\otimes_{{\cal O}_{{\cal P}{}^{\rm ex}}}
\Om^{\bul}_{{\cal P}{}^{{\rm ex}}_{\ul{\lam}}/{\cal W}_n(\os{\circ}{s})})$ and 
$[\om_2]\in {\cal H}^j({\cal O}_{{\mathfrak D}}
\otimes_{{\cal O}_{{\cal P}{}^{\rm ex}}}
\Om^{\bul}_{{\cal P}{}^{{\rm ex}}_{\ul{\lam}}/{\cal W}_n(\os{\circ}{s})})$ 
be the cohomology classes of $\om_1$ and $\om_2$, respectively.
Set 
\begin{align*}
[\om_1]\wedge [\om_2]=[\om_1 \wedge \om_2].
\tag{21.6.3}\label{ali:oom12}
\end{align*}  
If one of $[\om_1]$ and $[\om_2]$ is a coboundary, then 
$[\om_1 \wedge \om_2]$ is coboundary since $[\om_1]$ and $[\om_2]$ are cocycles. 
Hence the product (\ref{ali:oom12}) is well-defined. 
This product structure is independent of the choice of the immersion 
$X\os{\sus}{\lo} {\cal P}$. Indeed, let $X\os{\sus}{\lo} {\cal P}'$ be another immersion 
into a log smooth scheme over ${\cal W}_n(s)$. 
Then, by considering the product ${\cal P}\times_{{\cal W}_n(s)}{\cal P}'$, 
we may assume that there exists the following commutative diagram 
\begin{equation*} 
\begin{CD}
X@>>>{\cal P}\\
@| @VVV\\
X@>>>{\cal P}'
\end{CD}
\end{equation*}
over ${\cal W}_n(s)$. 
Let ${\mathfrak D}'$ be the log PD-envelope of the immersion 
$X\os{\sus}{\lo} {\cal P}'$ over $({\cal W}_n(s),p{\cal W}_n,[~])$. 
Then we have the following commutative diagrams:  
\begin{equation*} 
\begin{CD}
{\cal W}_n\wt{\Om}{}^i_{X_{\ul{\lam}}} @=
{\cal H}^i({\cal O}_{{\mathfrak D}}
\otimes_{{\cal O}_{{\cal P}{}^{\rm ex}}}
\Om^{\bul}_{{\cal P}{}^{{\rm ex}}_{\ul{\lam}}/{\cal W}_n(\os{\circ}{s})})\\
@| @AAA\\
{\cal W}_n\wt{\Om}{}^i_{X_{\ul{\lam}}}@=
{\cal H}^i({\cal O}_{{\mathfrak D}'}
\otimes_{{\cal O}_{{\cal P}'{}^{\rm ex}}}
\Om^{\bul}_{{\cal P}'{}^{{\rm ex}}_{\ul{\lam}}/{\cal W}_n(\os{\circ}{s})}), 
\end{CD}
\end{equation*}
\begin{equation*} 
\begin{CD}
{\cal H}^i({\cal O}_{{\mathfrak D}}
\otimes_{{\cal O}_{{\cal P}{}^{\rm ex}}}
\Om^{\bul}_{{\cal P}{}^{{\rm ex}}_{\ul{\lam}}/{\cal W}_n(\os{\circ}{s})})
\times {\cal H}^j({\cal O}_{{\mathfrak D}}
\otimes_{{\cal O}_{{\cal P}{}^{\rm ex}}}
\Om^{\bul}_{{\cal P}{}^{{\rm ex}}_{\ul{\lam}}/{\cal W}_n(\os{\circ}{s})}) @>{\wedge}>>
{\cal H}^{i+j}({\cal O}_{{\mathfrak D}}
\otimes_{{\cal O}_{{\cal P}{}^{\rm ex}}}
\Om^{\bul}_{{\cal P}{}^{{\rm ex}}_{\ul{\lam}}/{\cal W}_n(\os{\circ}{s})})
\\
@AAA @AAA\\
{\cal H}^i({\cal O}_{{\mathfrak D}'}
\otimes_{{\cal O}_{{\cal P}'{}^{\rm ex}}}
\Om^{\bul}_{{\cal P}'{}^{{\rm ex}}_{\ul{\lam}}/{\cal W}_n(\os{\circ}{s})})
\times {\cal H}^j({\cal O}_{{\mathfrak D}'}
\otimes_{{\cal O}_{{\cal P}'{}^{\rm ex}}}
\Om^{\bul}_{{\cal P}'{}^{{\rm ex}}_{\ul{\lam}}/{\cal W}_n(\os{\circ}{s})}) 
@>{\wedge}>>
{\cal H}^{i+j}({\cal O}_{{\mathfrak D}'}
\otimes_{{\cal O}_{{\cal P}'{}^{\rm ex}}}
\Om^{\bul}_{{\cal P}'{}^{{\rm ex}}_{\ul{\lam}}/{\cal W}_n(\os{\circ}{s})}). 
\end{CD}
\end{equation*}
The last commutative diagram tells us that the product structure 
(\ref{ali:oom12}) is independent of the choice of the immersion. 
It is clear that the product structure 
(\ref{ali:oom12}) induces the morphism (\ref{ali:pppw}). 
\end{proof}

\begin{prop}\label{prop:ab}
Set ${\cal W}_n({\cal O}_{X_{\ul{\lam}}}):=
{\cal W}_n\wt{\Om}{}^0_{X_{\ul{\lam}}}$. 
There exists a well-defined 
${\cal W}_n({\cal O}_{X_{\ul{\lam}}})$-module structure 
on $P_k{\cal W}_n\wt{\Om}{}^i_{X_{\ul{\lam}}}$  $(k\in {\mab Z}\cup \{{\infty}\})$
constructed in the proof of this proposition. 
This ${\cal W}_n({\cal O}_{X_{\ul{\lam}}})$-module structure induces 
a well-defined ${\cal W}_n({\cal O}_{X_{\ul{\lam}}})'$-module structure 
on $P_k{\cal W}_n\wt{\Om}{}^i_{X_{\ul{\lam}}}$ $(k\in {\mab Z}\cup \{{\infty}\})$. 
\end{prop}
\begin{proof} 
By (\ref{ali:pppw}) we have the following morphism 
\begin{align*} 
\wedge \col {\cal W}_n({\cal O}_{X_{\ul{\lam}}})\times 
P_k{\cal W}_n\wt{\Om}{}^i_{X_{\ul{\lam}}}
\lo P_k{\cal W}_n\wt{\Om}{}^i_{X_{\ul{\lam}}} \quad (k\in {\mab Z}\cup \{{\infty}\}). 
\tag{21.8.1}\label{ali:wnxul} 
\end{align*} 
This defines a ${\cal W}_n({\cal O}_{X_{\ul{\lam}}})$-module structure 
on $P_{k}{\cal W}_n\wt{\Om}{}^i_{X_{\ul{\lam}}}$. 
Now we have only to prove that there exists 
the following well-defined morphism 
\begin{align*} 
{\cal W}_n({\cal O}_{X_{\ul{\lam}}})'\lo {\cal W}_n({\cal O}_{X_{\ul{\lam}}}).
\tag{21.8.2}\label{ali:nmmm}
\end{align*}  
Let $a_i$ $(i=0, \ldots, n-1)$ be 
a local section of ${\cal O}_{X_{\ul{\lam}}}$.
Set ${\mathfrak D}_{\ul{\lam}}:=
{\mathfrak D}\times_{{\cal P}}{\cal P}_{\ul{\lam}}$. 
Let $\wt{a}_i \in {\cal O}_{{\mathfrak D}_{\ul{\lam}}}$ 
be a lift of $a_i$.
As in \cite[(4.9)]{hk} and \cite[(7.1.1)]{ndw}, 
consider the following morphism  
\begin{equation*}
s_n(0,0) \col {\cal W}_n({\cal O}_{X_{\ul{\lam}}})'\owns 
(a_0, \ldots, a_{n-1}) \lom 
\sum_{i=0}^{n-1}p^i\wt{a}_i^{p^{n-i}}
\in {\cal H}^0({\cal O}_{{\mathfrak D}}
\otimes_{{\cal O}_{{\cal P}{}^{\rm ex}}}
\Om^{\bul}_{{\cal P}{}^{{\rm ex}}_{\ul{\lam}}/{\cal W}_n(\os{\circ}{s})})
={\cal W}_n({\cal O}_{X_{\ul{\lam}}}). 
\tag{21.8.3}\label{eqn:sn00}
\end{equation*}
As in \cite[(4.9)]{hk}, we can easily check that this morphism 
is independent of the choice of the lift $\wt{a}_i$. 
As in the proof of (\ref{prop:cpp}), we can check that this 
${\cal W}_n({\cal O}_{X_{\ul{\lam}}})'$-module structure on 
${\cal H}^0({\cal O}_{{\mathfrak D}}
\otimes_{{\cal O}_{{\cal P}{}^{\rm ex}}}
\Om^{\bul}_{{\cal P}{}^{{\rm ex}}_{\ul{\lam}}/{\cal W}_n(\os{\circ}{s})})
={\cal W}_n({\cal O}_{X_{\ul{\lam}}})$ is independent of the choice of 
the immersion $X\os{\sus}{\lo} {\cal P}$. 
\end{proof}

\begin{coro} 
The wedge products {\rm (\ref{ali:pppxw})} and {\rm (\ref{ali:pppw})} 
are ${\cal W}_n({\cal O}_{X_{\ul{\lam}}})'$-bilinear and 
${\cal W}_n({\cal O}_{X_{\ul{\lam}}})$-bilinear. 
\end{coro}

\begin{rema}\label{rema:dab}
I do not check whether the morphism (\ref{ali:nmmm}) is an isomorphism 
in the case $\ul{\lam}\not= \emptyset$. 
More generally, let $(W_n\wt{\Om}_{X_{\ul{\lam}}}^i)'$ be an abelian sheaf  on 
$\os{\circ}{X}_{\ul{\lam}}$ which is a quotient of 
\begin{equation*}
\{W_n({\cal O}_{X_{\ul{\lam}}})\otimes_{\mab Z}
\bigwedge^i(M^{\rm gp}/f^{-1}({\mab Z}))\}
\oplus 
\{W_n({\cal O}_{X_{\ul{\lam}}})\otimes_{\mab Z}
\bigwedge^{i-1}(M^{\rm gp}/f^{-1}({\mab Z}))\}
\tag{21.10.1}\label{eqn:flsodw}
\end{equation*}
divided by a ${\mab Z}$-submodule ${\cal F}_n$ 
generated by the images of 
the local sections  of the following type
\begin{equation*}
(V^j(\al(a_1))
\otimes(a_1\wedge \cdots 
\wedge a_i), 0)-p^j(0, V^j(\al(a_1))
\otimes (a_2\wedge \cdots \wedge a_i)) 
\tag{21.10.2}\label{eqn:frel}
\end{equation*}
$$(a_1, \ldots, a_i \in M,~ 
0 \leq j < n),$$
where $V^j(b):= 
({\us{j~{\rm times}}
{\underbrace{0, \cdots , 0}}},
b, 0, \ldots, 0)$ for a 
local section $b\in {\cal O}_{X_{\ul{\lam}}}$ 
(\cite[(4.6)]{hk}) and a  
${\mab Z}$-submodule 
${\cal G}_n$, where 
${\cal G}_n$ is a ${\mab Z}$-submodule generated 
by the images of the local sections of 
the following type
\begin{equation*}
(0, V^j(\al(a_2))\otimes
(a_2\wedge \cdots \wedge a_i)) 
\quad 
(a_2, \ldots, a_i\in M,~ 
1 \leq j <n)
\tag{21.10.3}\label{eqn:mrel}
\end{equation*}
as an abelian sheaf on 
$X_{{\ul{\lam}}}$. 
\par
It is not difficult to check that 
the analogous morphism to $s_n$ in \cite[(4.9)]{hk} 
factors through a morphism
\begin{equation*}
s_n \col (W_{n}\wt{\Om}{}^i_{X_{\ul{\lam}}})' \lo 
W_{n}\wt{\Om}{}^i_{X_{\ul{\lam}}} 
\quad (n\in {\mab Z}_{>0}). 
\tag{21.10.4}\label{eqn:morm}
\end{equation*} 
I have not checked whether there exists an isomorphism 
\begin{align*} 
({\cal W}_n\wt{\Om}{}^i_{X_{\ul{\lam}}})' \lo {\cal W}_n\wt{\Om}{}^i_{X_{\ul{\lam}}}
\end{align*}
as abelian sheaves because we do not need this property in this book. 
\end{rema}

Next we define the derivative 
$d\col {\cal W}_n\wt{\Om}{}^i_{X_{\ul{\lam}}} 
\lo {\cal W}_n\wt{\Om}{}^{i+1}_{X_{\ul{\lam}}}$ as follows.  
This derivative has not been defined in \cite{kiha}. 
\par
Let $X\os{\sus}{\lo} {\cal P}$ be a local immersion into 
a log smooth scheme over ${\cal W}_{2n}(s)$. 
Let ${\mathfrak D}$ be the log PD-envelope of this immersion over 
$({\cal W}_{2n}(\os{\circ}{s}),p{\cal W}_{2n},[~])$. 
Set ${\cal P}_{\ul{\lam},n}:={\cal P}_{\ul{\lam}}\otimes_{{\cal W}_{2n}}{\cal W}_n$. 
By the local description of 
${\cal P}^{\rm ex}$, ${\mathfrak D}$  and ${\cal P}^{{\rm ex}}_{\ul{\lam}}$ 
((\ref{prop:fsi})), we see that 
${\cal O}_{\mathfrak D}
\otimes_{{\cal O}_{{\cal P}^{\rm ex}}}
\Om^i_{{\cal P}^{{\rm ex}}_{\ul{\lam}}/{\cal W}_{2n}(\os{\circ}{s})}$ $(i\in {\mab N})$ 
is a flat ${\cal W}_{2n}$-module. 
Consequently the following sequence is exact:  
\begin{align*} 
0\lo {\cal O}_{\mathfrak D}
\otimes_{{\cal O}_{{\cal P}^{\rm ex}}}
\Om^{\bul}_{{\cal P}^{{\rm ex}}_{\ul{\lam},n}/{\cal W}_{n}(\os{\circ}{s})}
\os{p^n}{\lo} {\cal O}_{\mathfrak D}
\otimes_{{\cal O}_{{\cal P}^{\rm ex}}}
\Om^{\bul}_{{\cal P}^{{\rm ex}}_{\ul{\lam}}/{\cal W}_{2n}(\os{\circ}{s})}
\lo {\cal O}_{\mathfrak D}
\otimes_{{\cal O}_{{\cal P}^{\rm ex}}}
\Om^{\bul}_{{\cal P}^{{\rm ex}}_{\ul{\lam},n}/{\cal W}_{n}(\os{\circ}{s})}\lo 0. 
\tag{21.10.5}\label{ali:bdd}
\end{align*} 
The derivative 
$d\col {\cal W}_n\wt{\Om}{}^i_{X_{\ul{\lam}}} 
\lo {\cal W}_n\wt{\Om}{}^{i+1}_{X_{\ul{\lam}}}$
is, by definition, the boundary morphism 
${\cal H}^i({\cal O}_{\mathfrak D}
\otimes_{{\cal O}_{{\cal P}^{\rm ex}}}
\Om^{\bul}_{{\cal P}^{{\rm ex}}_{\ul{\lam},n}/{\cal W}_{n}(\os{\circ}{s})})
\lo 
{\cal H}^{i+1}({\cal O}_{\mathfrak D}
\otimes_{{\cal O}_{{\cal P}^{\rm ex}}}
\Om^{\bul}_{{\cal P}^{{\rm ex}}_{\ul{\lam},n}/{\cal W}_{n}(\os{\circ}{s})})$
obtained by (\ref{ali:bdd}). 
It is easy to see that this derivative is independent of the choice of 
the immersion $X\os{\sus}{\lo} {\cal P}$. 
Here the reader should note that this independence immediately follows
since we allow not only the lift but also the immersion. 
If one allows only the lift, a more complicated proof is necessary 
since the product of a lift is not necessarily a lift (cf.~\cite{msemi}).

\begin{prop}\label{prop:efe}
There exists the following exact sequence$:$
\begin{align*} 
0\lo P_{k-1}{\cal W}_n\wt{\Om}{}^i_{X_{\ul{\lam}}}
\lo P_k{\cal W}_n\wt{\Om}{}^i_{X_{\ul{\lam}}}\lo 
\bigoplus_{\# \ul{\mu}=k}
a_{\ul{\lam}\cup \ul{\mu}*}
({\cal W}_n\Om^{\bul}_{\os{\circ}{X}_{\ul{\lam}\cup \ul{\mu}}}[-k]
\otimes_{\mab Z}\vp_{{\rm zar},\ul{\mu}}(\os{\circ}{X}/\os{\circ}{s}))\lo 0.
\tag{21.11.1}\label{ali:ppkkxl}
\end{align*} 
\end{prop}
\begin{proof}
Let the notations be as in the proof of (\ref{prop:cpp}). 
By the same proof as those of \cite[Lemme 1.2]{msemi} and \cite[p.~1268]{kiha}, the following sequence is locally split: 
\begin{align*}  
0 &\lo 
P_{k-1}({\cal O}_{\mathfrak D}
\otimes_{{\cal O}_{{\cal P}^{\rm ex}}}
\Om^{\bul}_{{\cal P}^{{\rm ex}}_{\ul{\lam}}/{\cal W}_{n}(\os{\circ}{s})})
\lo P_k
({\cal O}_{\mathfrak D}
\otimes_{{\cal O}_{{\cal P}^{\rm ex}}}
\Om^{\bul}_{{\cal P}^{{\rm ex}}_{\ul{\lam}}/{\cal W}_{n}(\os{\circ}{s})})\\
&\lo 
{\rm gr}_k^P
({\cal O}_{\mathfrak D}
\otimes_{{\cal O}_{{\cal P}^{\rm ex}}}
\Om^{\bul}_{{\cal P}^{{\rm ex}}_{\ul{\lam}}/{\cal W}_{n}(\os{\circ}{s})})
\lo 0. 
\end{align*} 
Hence the following morphism 
\begin{align*}
{\cal H}^i(P_{k-1}({\cal O}_{\mathfrak D}
\otimes_{{\cal O}_{{\cal P}^{\rm ex}}}
\Om^{\bul}_{{\cal P}^{{\rm ex}}_{\ul{\lam}}/{\cal W}_{n}(\os{\circ}{s})}))
\lo 
{\cal H}^i(P_{k}({\cal O}_{\mathfrak D}
\otimes_{{\cal O}_{{\cal P}^{\rm ex}}}
\Om^{\bul}_{{\cal P}^{{\rm ex}}_{\ul{\lam}}/{\cal W}_{n}(\os{\circ}{s})}))
\end{align*} 
is injective. 
This implies that the following morphism 
\begin{align*} 
P_{k-1}{\cal W}_n\wt{\Om}{}^i_{X_{\ul{\lam}}}
\lo P_k{\cal W}_n\wt{\Om}{}^i_{X_{\ul{\lam}}}
\end{align*} 
is injective. 
\par 
Because 
\begin{align*} 
&{\rm gr}_k^P
({\cal O}_{\mathfrak D}
\otimes_{{\cal O}_{{\cal P}^{\rm ex}}}
\Om^{\bul}_{{\cal P}^{{\rm ex}}_{\ul{\lam}}/{\cal W}_{n}(\os{\circ}{s})})=\\
&\bigoplus_{\# \ul{\mu}=k}
{\cal O}_{\os{\circ}{\mathfrak D}_{\ul{\lam}\cup \ul{\mu}}}
{\otimes}_{{\cal O}_{{\cal P}^{{\rm ex}}_{\ul{\lam}\cup \ul{\mu}}}}
b_{\ul{\lam}\cup \ul{\mu}*}
(\Om^{\bul}_{\os{\circ}{\cal P}{}^{{\rm ex}}_{\ul{\lam}\cup \ul{\mu}}/
{\cal W}_{n}(\os{\circ}{s})})[-k]
\otimes_{\mab Z}\vp_{{\rm zar},\ul{\mu}}({\cal P}^{\rm ex}_n/{\cal W}_{n}(\os{\circ}{s})))
\end{align*} 
((\ref{eqn:mpprrn})), 
we obtain the following equalities: 
\begin{align*} 
&{\cal H}^i({\rm gr}_k^P
({\cal O}_{\mathfrak D}
\otimes_{{\cal O}_{{\cal P}^{\rm ex}}}
\Om^{\bul}_{{\cal P}^{{\rm ex}}_{\ul{\lam}}/{\cal W}_{n}(\os{\circ}{s})}))\\
&=\bigoplus_{\# \ul{\mu}=k}{\cal H}^i
({\cal O}_{\os{\circ}{\mathfrak D}_{\ul{\lam}\cup \ul{\mu}}}
{\otimes}_{{\cal O}_{{\cal P}{}^{{\rm ex}}_{\ul{\lam}\cup \ul{\mu}}}}
b_{\ul{\lam}\cup \ul{\mu}*}
\Om^{\bul}_{\os{\circ}{\cal P}{}^{{\rm ex}}_{\ul{\lam}\cup \ul{\mu}}/{\cal W}_{n}(\os{\circ}{s})}[-k]
\otimes_{\mab Z}\vp_{{\rm zar},\ul{\mu}}({\cal P}^{\rm ex}_n/{\cal W}_{n}(\os{\circ}{s})))\\
&=\bigoplus_{\# \ul{\mu}=k}a_{\ul{\lam}\cup \ul{\mu}*}
({\cal W}_n\Om^{\bul}_{\os{\circ}{X}_{\ul{\lam}\cup \ul{\mu}}}[-k]
\otimes_{\mab Z}\vp_{{\rm zar},\ul{\mu}}(\os{\circ}{X}/\os{\circ}{s})). 
\end{align*} 
Hence we obtain the exact sequence (\ref{ali:ppkkxl}).
\end{proof}

\par 
Set 
\begin{align*} 
({\cal W}_n\wt{\Om}{}^{\bul}_{X^{(m)}},P):=
\bigoplus_{\# \ul{\lam}=m+1}
({\cal W}_n\wt{\Om}{}^{\bul}_{X_{\ul{\lam}}},P) \quad (m\in {\mab N}). 
\end{align*}
Then 
\begin{align*} 
{\rm gr}^P_k{\cal W}_n\wt{\Om}{}^{\bul}_{X^{(m)}}=
\bigoplus_{\# \ul{\lam}=m+1}
\bigoplus_{\# \ul{\mu}=k}a_{\ul{\lam}\cup \ul{\mu}*}
({\cal W}_n\Om^{\bul -k}_{\os{\circ}{X}_{\ul{\lam}\cup \ul{\mu}}}
\otimes_{\mab Z}\vp_{{\rm zar},\ul{\mu}}(\os{\circ}{X}/\os{\circ}{s})). 
\tag{21.11.2}\label{ali:pkgrm}
\end{align*}

The filtered complexes $({\cal W}_n\wt{\Om}{}^{\bul}_{X_{\ul{\lam}}},P)$ 
and $({\cal W}_n\wt{\Om}{}^{\bul}_{X^{(m)}},P)$ have well-defined product structures 
which are compatible with the filtrations $P$'s. 
Let $\theta_{\ul{\lam}}\in {\cal W}_n\wt{\Om}{}^1_{X_{\ul{\lam}}}$ 
be the cohomology class ``$[d\log t]$'' in 
${\cal H}^1({\cal O}_{\mathfrak D}
\otimes_{{\cal O}_{{\cal P}^{\rm ex}}}
\Om^{\bul}_{{\cal P}^{{\rm ex}}_{\ul{\lam}}/{\cal W}_{n}(\os{\circ}{s})})$. 
It is easy to see that $\theta_{\ul{\lam}}$ is well-defined. 
Now consider the Hirsch extension 
${\cal W}_n\wt{\Om}{}^{\bul}_{X_{\ul{\lam}}}\langle u \rangle$
of ${\cal W}_n\wt{\Om}{}^{\bul}_{X_{\ul{\lam}}}$ by the morphism
${\cal W}_nu\owns u \lom \theta_{\ul{\lam}}\in 
{\rm Ker}({\cal W}_n\wt{\Om}{}^1_{X_{\ul{\lam}}}\lo 
{\cal W}_n\wt{\Om}{}^2_{X_{\ul{\lam}}})$  
and the filtration $P$ on 
${\cal W}_n\wt{\Om}{}^{\bul}_{X_{\ul{\lam}}}\langle u \rangle$: 
\begin{align*} 
P_k{\cal W}_n\wt{\Om}{}^{\bul}_{X_{\ul{\lam}}}\langle u \rangle
:=\bigoplus_{j=0}^{\inf}{\cal W}_nu^{[j]}\otimes_{{\cal W}_n}
P_{k-2j}{\cal W}_n\wt{\Om}{}^{\bul}_{X_{\ul{\lam}}}
\tag{21.11.3}\label{ali:wpkm}
\end{align*} 
We also consider the Hirsch extension
${\cal W}_n\wt{\Om}{}^{\bul}_{X^{(m)}}\langle u \rangle$ 
of ${\cal W}_n\wt{\Om}{}^{\bul}_{X^{(m)}}$ by the morphism
${\cal W}_nu\owns u \lom \theta:=\oplus_{\vert \ul{\lam}\vert =m+1}\theta_{\ul{\lam}}
\in {\cal W}_n\wt{\Om}{}^1_{X^{(m)}}$. 
Then we obtain the single complex 
$s({\cal W}_n\wt{\Om}{}^{\bul}_{X^{(\bul)}}\langle u \rangle)$ 
of the double complex ${\cal W}_n\wt{\Om}{}^{\bul}_{X^{(\bul)}}\langle u \rangle$, 
where we consider the signs of differential morphisms of 
the double complex ${\cal W}_n\wt{\Om}{}^{\bul}_{X^{(\bul)}}\langle u \rangle$ 
as in the case of the Hirsch PD-filtered complex $H_{\rm zar}$ stated in previous sections. 
This is nothing but the cosimplicial dga 
$s(C(W\wt{\om}[u]))$ defined in \cite[Lemma 12]{kiha} 
(modulo the ambiguity of the signs of the boundary morphisms of the double complex 
in [loc.~cit.]).
\par 
We define the diagonal filtration $\del(L,P)$ on 
$s({\cal W}_n\wt{\Om}{}^{\bul}_{X^{(\bul)}}\langle u \rangle)$ as usual: 
\begin{align*} 
\del(L,P)_ks({\cal W}_n\wt{\Om}{}^{\bul}_{X^{(\bul)}}\langle u \rangle)
&
:=\bigoplus_{m=0}^{\inf}P_{k+m}({\cal W}_n\wt{\Om}{}^{\bul}_{X^{(m)}}\langle u \rangle)
\tag{21.11.4}\label{ali:pkm}\\
&:=\bigoplus_{m=0}^{\inf}\bigoplus_{j=0}^{\inf}
{\cal W}_nu^{[j]}\otimes_{{\cal W}_n}
P_{k+m-2j}{\cal W}_n\wt{\Om}{}^{\bul}_{X^{(m)}}. 
\end{align*} 
In the following we denote $\del(L,P)$ simply by $P$ by abuse of notation.
In this convention 
\begin{align*} 
&{\rm gr}_k^{P}s({\cal W}_n\wt{\Om}{}^{\bul}_{X^{(\bul)}}\langle u \rangle)
=\bigoplus_{m=0}^{\inf}\bigoplus_{j=0}^{\inf}
{\cal W}_nu^{[j]}\otimes_{{\cal W}_n}
{\rm gr}^P_{k+m-2j}{\cal W}_n\wt{\Om}{}^{\bul}_{X^{(m)}}
\tag{21.11.5}\label{ali:pkpm}\\
&=\bigoplus_{m=0}^{\inf}\bigoplus_{j=0}^{\inf}
\bigoplus_{\# \ul{\lam}=m}
\bigoplus_{\# \ul{\mu}=k+m-2j}
{\cal W}_nu^{[j]}\otimes_{{\cal W}_n}
a_{\ul{\lam}\cup \ul{\mu}*}
({\cal W}_n\Om^{\bul -k}_{\os{\circ}{X}_{\ul{\lam}\cup \ul{\mu}}}
\otimes_{\mab Z}\vp_{{\rm zar},\ul{\mu}}(\os{\circ}{X}/\os{\circ}{s})).
\end{align*} 

\par 
Next we define  the ``projection'' 
\begin{align*} 
R\col {\cal W}_n\wt{\Om}{}^{\bul}_{X_{\ul{\lam}}}\lo 
{\cal W}_{n-1}\wt{\Om}{}^{\bul}_{X_{\ul{\lam}}} \quad (n\in {\mab Z}_{\geq 1})
\tag{21.11.6}\label{ali:pkmx}
\end{align*} 
by using the methods in \cite{hyp} and \cite{nb} with additional arguments. 
This operator has not been defined in \cite{kiha} either. 
\par 
Let $n$ be a positive integer and let $i$ be a nonnegative integer.
To define the morphism (\ref{ali:pkmx}), we first assume that there exists an immersion
$X\os{\sus}{\lo} {\cal P}$ into a formally log smooth scheme over ${\cal W}(s)$ 
such that ${\cal P}$ has a lift $\varphi \col {\cal P}\lo {\cal P}$
of the Frobenius endomorphism of ${\cal P}\times_{{\cal W}(s)}s$.
Let ${\cal P}^{\rm ex}$ be the exactification of the immersion 
$X\os{\sus}{\lo} {\cal P}$. Then $\varphi \col {\cal P}\lo {\cal P}$
induces the morphism $\varphi^{\rm ex} \col {\cal P}^{\rm ex}\lo {\cal P}^{\rm ex}$. 
Let ${\cal P}^{\rm ex}_{\lam}$ be a formally smooth component of ${\cal P}^{\rm ex}$ 
which corresponds to a smooth component $X_{\lam}$ of $X$. 
Let $e_{\lam}\in M_{{\cal P}^{\rm ex}}$ be a local section corresponding to  ${\cal P}^{\rm ex}_{\lam}$. 
Then $\varphi^*(e_{\lam})=e^p_{\lam}u$ for some $u\in {\cal O}^*_{{\cal P}^{\rm ex}}$. 
Hence $\varphi^{\rm ex}$ induces a morphism 
${\cal P}^{\rm ex}_{\lam}\lo {\cal P}^{\rm ex}_{\lam}$. 
Consequently, for a subset $\ul{\lam}$ of $\Lam$, 
$\varphi^{\rm ex}$ induces a morphism 
${\cal P}^{\rm ex}_{\ul{\lam}}\lo {\cal P}^{\rm ex}_{\ul{\lam}}$. 
Let ${\mathfrak D}$ be the log PD-envelope of the immersion
$X\os{\sus}{\lo} {\cal P}^{\rm ex}$ over 
$({\cal W}(s),p{\cal W},[~])$.
Set ${\cal P}^{\rm ex}_n:={\cal P}^{\rm ex}\times_{{\cal W}(s)}{\cal W}_n(s)$, 
${\cal P}^{\rm ex}_{\ul{\lam},n}:={\cal P}^{\rm ex}_{\ul{\lam}}\times_{{\cal W}(s)}{\cal W}_n(s)$ 
 and 
${\mathfrak D}_n:={\mathfrak D}\times_{{\cal W}(s)}{\cal W}_n(s)$.
Consider the following morphism
\begin{align*}
\varphi^{{\rm ex}*}
\col & {\cal O}_{\mathfrak D}\otimes_{{\cal O}_{{\cal P}^{\rm ex}}}
{\Om}{}^i_{{\cal P}^{\rm ex}_{\ul{\lam}}/{\cal W}(\os{\circ}{s})}\lo
{\cal O}_{{\mathfrak D}}\otimes_{{\cal O}_{{\cal P}^{\rm ex}}}
{\Om}{}^i_{{\cal P}^{\rm ex}_{\ul{\lam}}/{\cal W}(\os{\circ}{s})} \quad (i\in {\mab N}).
\end{align*}
Because ${\cal O}_{{\mathfrak D}}\otimes_{{\cal O}_{{\cal P}^{\rm ex}}}
\wt{\Om}{}^i_{{\cal P}^{\rm ex}_{\ul{\lam}}/{\cal W}(\os{\circ}{s})}$ is a sheaf of flat 
${\cal O}_{{\cal W}(s)}$-modules by the proof of \cite[(1.3.5), (1.7.24)]{nb}
and because $\varphi^{{\rm ex}*}$ is divisible by $p^i$
(since $\varphi$ is a lift of the  Frobenius endomorphism),
the morphism
\begin{align*}
p^{-(i-1)}\varphi^{{\rm ex}*}
\col & {\cal O}_{{\mathfrak D}}\otimes_{{\cal O}_{{\cal P}^{\rm ex}}}
{\Om}{}^j_{{\cal P}^{\rm ex}_{\ul{\lam}}/{\cal W}(\os{\circ}{s})}\lo
{\cal O}_{{\mathfrak D}}\otimes_{{\cal O}_{{\cal P}^{\rm ex}}}
{\Om}{}^j_{{\cal P}^{\rm ex}_{\ul{\lam}}/{\cal W}(\os{\circ}{s})} \quad (j\geq i-1)
\end{align*}
is well-defined.
Because the image of the morphism above is contained in
$p{\cal O}_{{\mathfrak D}}\otimes_{{\cal O}_{{\cal P}^{\rm ex}}}
{\Om}{}^j_{{\cal P}^{\rm ex}_{\ul{\lam}}/{\cal W}(\os{\circ}{s})}$ for the case $j=i$,
we have the following well-defined morphism (cf.~\cite[Editorial comment (5)]{hyp})
\begin{align*}
p^{-(i-1)}\varphi^{{\rm ex}*} &\col {\cal W}_n\wt{\Om}{}^i_{X_{\ul{\lam}}}
={\cal H}^i({\cal O}_{{\mathfrak D}_n}\otimes_{{\cal O}_{{\cal P}^{\rm ex}_n}}
\Om^{\bul}_{{\cal P}^{\rm ex}_{\ul{\lam},n}/{\cal W}_n(\os{\circ}{s})})
\lo
\tag{21.11.7}\label{ali:enqw}\\
& 
{\cal H}^i({\cal O}_{{\mathfrak D}_{n+1}}\otimes_{{\cal O}_{{\cal P}_{n+1}}}
\Om^{\bul}_{{\cal P}^{\rm ex}_{n+1,\ul{\lam}}/{\cal W}_{n+1}(\os{\circ}{s})})
={\cal W}_{n+1}\wt{\Om}{}^i_{X_{\ul{\lam}}} \quad (i\in {\mab N}).
\end{align*}

\begin{prop}\label{prop:nqop} 
$(1)$ The morphism $p^{-(i-1)}\varphi^{{\rm ex}*}$ in {\rm (\ref{ali:enqw})} 
is independent of the choice of 
the immersion $X\os{\sus}{\lo}{\cal P}$ and the lift $\varphi \col {\cal P}\lo {\cal P}$ 
of the Frobenius endomorphism of 
${\cal P}\times_{{\cal W}(s)}s$.
We set ${\bf p}:=p^{-(i-1)}\varphi^*\col 
{\cal W}_n\wt{\Om}{}^i_{X_{\ul{\lam}}}\lo {\cal W}_{n+1}\wt{\Om}{}^i_{X_{\ul{\lam}}}$. 
\par 
$(2)$ ${\rm Ker}({\bf p}\col {\cal W}_n\wt{\Om}{}^i_{X_{\ul{\lam}}}
\lo {\cal W}_{n+1}\wt{\Om}{}^i_{X_{\ul{\lam}}})=0$.
\par 
$(3)$ ${\rm Im}({\bf p}\col {\cal W}_n\wt{\Om}{}^i_{X_{\ul{\lam}}}\lo 
{\cal W}_{n+1}\wt{\Om}{}^i_{X_{\ul{\lam}}})=
{\rm Im}(p\col {\cal W}_{n+1}\wt{\Om}{}^i_{X_{\ul{\lam}}}\lo 
{\cal W}_{n+1}\wt{\Om}{}^i_{X_{\ul{\lam}}})$. 
\end{prop}
\begin{proof}
(1): Let $X\os{\sus}{\lo}{\cal P}'$ be another immersion and another
lift $\varphi' \col {\cal P}'\lo {\cal P}'$ 
of the Frobenius endomorphism of ${\cal P}'\times_{{\cal W}(\os{\circ}{s})}\os{\circ}{s}$.
Then, by considering the products 
$({\cal P}^{\rm ex}\times_{{\cal W}_n(\os{\circ}{s})}{\cal P}'{}^{\rm ex})^{\rm ex}$
and $\varphi \times \varphi'$, we may assume that 
there exists the following commutative diagram  
\begin{equation*} 
\begin{CD} 
X@>{\subset}>>{\cal P}^{\rm ex}@>{\varphi}>>{\cal P}^{\rm ex}\\
@| @VVV @VVV \\
X@>{\subset}>>{\cal P}'{}^{\rm ex}@>{\varphi'}>>{\cal P}'{}^{\rm ex}. 
\end{CD}
\end{equation*} 
Hence we have the following commutative diagram 
\begin{equation*} 
\begin{CD} 
{\cal H}^i({\cal O}_{{\mathfrak D}_n}\otimes_{{\cal O}_{{\cal P}^{\rm ex}_n}}
\Om^{\bul}_{{\cal P}^{\rm ex}_n/{\cal W}_n(\os{\circ}{s})})@>{p^{-(i-1)}\varphi^{{\rm ex}*}}>>
{\cal H}^i({\cal O}_{{\mathfrak D}_{n+1}}
\otimes_{{\cal O}_{{\cal P}^{\rm ex}_{n+1}}}
\Om^{\bul}_{{\cal P}^{\rm ex}_{n+1}/{\cal W}_{n+1}(\os{\circ}{s})})\\
@A{\simeq}AA @AA{\simeq}A  \\
{\cal H}^i({\cal O}_{{\mathfrak D}'_{n}}\otimes_{{\cal O}_{{\cal P}'{}^{\rm ex}_n}}
\Om^{\bul}_{{\cal P}'{}^{\rm ex}_n/{\cal W}_n(\os{\circ}{s})})@>{p^{-(i-1)}\varphi'{}^{{\rm ex}*}}>>
{\cal H}^i({\cal O}_{{\mathfrak D}'_{n+1}}
\otimes_{{\cal O}_{{\cal P}'{}^{\rm ex}_{n+1}}}
\Om^{\bul}_{{\cal P}'{}^{\rm ex}_{n+1}/{\cal W}_{n+1}(\os{\circ}{s})}),
\end{CD}
\end{equation*} 
which shows the desired independence. 
\par 
(2): We may assume that ${\cal P}$ is a 
formal lift ${\cal X}/{\cal W}(s)$ of $X/s$. 
Then 
$(\Om^{\bul},\phi):=(\Om^{\bul}_{{\cal X}_{\ul{\lam}}/{\cal W}(\os{\circ}{s})},\varphi^*)$ 
satisfies the conditions $(6.0.1)\sim (6.0.5)$ in \cite{ndw}
(these conditions are only the abstract versions of 
necessary conditions for \cite[(8.8)]{bob}).  
That is, the following hold:
\medskip
\parno
(21.12.1) $\Om^i=0$ for $i<0$.
\medskip
\parno
(21.12.2)  $\Om^i$ 
$(\forall i \in {\mab N})$ are sheaves of $p$-torsion-free, $p$-adically 
complete ${\mab Z}_p$-modules.
\medskip
\parno
(21.12.3) $\phi(\Om^{i}) \subset 
\{\om \in p^i \Om^i~ \vert~ d\om \in p^{i+1} 
\Om^{i+1}\}$ $(\forall i\in {\mab N})$.
\medskip
\parno
(21.12.4) Set $\Om_1^{\bul}:=\Om^{\bul}/p\Om^{\bul}$. 
Then there exists an ${\mab F}_p$-linear isomorphism 
$$C^{-1} \col \Om^i_1\os{\sim}{\lo} {\cal H}^i(\Om^{\bul}_1) \quad 
(\forall i \in {\mab N}).$$
\medskip
\parno
(21.12.5) A composite morphism 
(${\rm mod}~p)\circ p^{-i}\phi 
\col \Om^{i} \lo \Om^i \lo \Om^i_1$ factors 
through ${\rm Ker}(d \col\Om^i_1 \lo \Om^{i+1}_{1})$ 
and the following diagram is commutative:
\begin{equation*}
\begin{CD}
\Om^i @>{\mod p}>> 
\Om^i_1\\ 
@V{p^{-i}\phi}VV  @VV{C^{-1}}V \\
\Om^i @>{\mod p}>> 
{\cal H}^i(\Om^{\bul}_1).
\end{CD}
\end{equation*}
The only nontrivial property is the existence of the inverse Cartier isomorphism 
$C^{-1}$ in (21.12.4). In the proof of \cite[Lemma 12]{kiha} 
\par  
Now set 
\begin{equation*}
Z^i_n:=\{\om \in \Om^i \vert~ 
d\om \in p^n\Om^{i+1}\}, \quad
B^i_n:=p^n\Om^i+ d\Om^{i-1}, \quad 
{\mathfrak W}_n{\Om}^i= Z^i_n/B^i_n. 
\tag{21.12.6}\label{eqn:pnzb}
\end{equation*}
Then ${\cal W}_n{\Om}^i_{X}= 
{\mathfrak W}_n{\Om}^i$ and 
the morphism 
${\bf p}\col {\cal W}_n{\Om}^i_{X}\lo 
{\cal W}_{n+1}{\Om}^i_{X}$ 
is equal to ${\bf p}\col {\mathfrak W}_n{\Om}^i\lo 
{\mathfrak W}_{n+1}{\Om}^i$ in \cite[p.~546]{ndw}. 
Because the latter morphism is injective (\cite[(6.8)]{ndw}), we obtain (2). 
\par 
(3): By \cite[p.~546]{ndw},  
${\rm Im}({\bf p}\col {\mathfrak W}_n{\Om}^i\lo {\mathfrak W}_{n+1}{\Om}^i)
={\rm Im}(p\col {\mathfrak W}_{n+1}{\Om}^i\lo {\mathfrak W}_{n+1}{\Om}^i)$. 
Hence we obtain (3). 
\end{proof}


\begin{defi}[{\bf cf.~\cite[(1.3.2)]{hyp}}]\label{defi:hdyd} 
(1) We call the morphism 
${\bf p} \col {\cal W}_n\wt{\Om}{}^i_{X_{\ul{\lam}}}\lo 
{\cal W}_{n+1}\wt{\Om}{}^i_{X_{\ul{\lam}}}$ {\it Hyodo's multiplication by} $p$ on 
${\cal W}_n\wt{\Om}{}^i_{X_{\ul{\lam}}}$. 
\par 
(2) The morphism $R\col {\cal W}_{n+1}\wt{\Om}{}^i_{X_{\ul{\lam}}}\lo 
{\cal W}_n{\Om}^i_{X_{\ul{\lam}}}$ 
is, by definition, the unique morphism fitting into the following commutative diagram: 
\begin{equation*} 
\begin{CD} 
{\cal W}_{n+1}\wt{\Om}{}^i_{X_{\ul{\lam}}}
@>{R}>>{\cal W}_n\wt{\Om}{}^i_{X_{\ul{\lam}}}\\
@V{p}VV @VV{\bf p}V\\ 
{\cal W}_{n+1}\wt{\Om}{}^i_{X_{\ul{\lam}}}
@={\cal W}_{n+1}\wt{\Om}{}^i_{X_{\ul{\lam}}},  
\end{CD}
\end{equation*} 
whose existence is assured by (\ref{prop:nqop}) (2), (3). 
\end{defi}

\parno 
By the proof of (\ref{prop:nqop}) and \cite[(6.27)]{ndw},  
the morphism $R$ is equal to the following composite morphism 
in the case where $X/s$ has a local lift ${\cal X}/{\cal W}(s)$
(cf.~\cite[(4.2)]{hk}):
\begin{align*}
{\cal W}_{n+1}\wt{\Om}^i_{X_{\ul{\lam}}}= 
Z^i_{n+1}/B^i_{n+1} 
\os{\us{\sim}{p^i}}{\lo} 
p^iZ^i_{n+1}/p^iB^i_{n+1}
& \os{{\rm proj}.}{\lo} 
p^iZ^i_{n+1}/(p^{i+n}Z^i_1+p^{i-1}dZ^{i-1}_1) 
\tag{21.13.1}\label{ali:aplqcpj}\\
{} & \os{\us{\sim}{\phi}}{\longleftarrow}
Z^i_n/B^i_n={\cal W}_n\wt{\Om}^i_{X_{\ul{\lam}}}.
\end{align*}
By \cite[(6.5)]{ndw}, 
the morphism $R$ is equal to the 
following composite morphism$:$
\begin{align*}
&{\cal W}_{n+1}\wt{\Om}^i_{X_{\ul{\lam}}}  = 
Z^i_{n+1}/B^i_{n+1}
\os{{\rm proj}.}{\lo} 
Z^i_{n+1}/(p^nZ^i_1+d\Om^{i-1}) 
\tag{21.13.2}\label{ali:alocpqj}\\
& 
\os{(p^{-i}\phi)^{-1}}{\os{\sim}{\lo}}
Z^i_n/(p^n\Om^i+pd\Om^{i-1})
\os{{\rm proj}.}{\lo}
Z^i_n/B^i_n={\cal W}_n\wt{\Om}^i_{X_{\ul{\lam}}}.
\end{align*}
In particular, $R$ is surjective. By the proof of (\ref{prop:nqop}), 
$(\Om^{\bul}_{{\cal X}_{\ul{\lam}}/{\cal W}(s)},\varphi^*)$ satisfies the axioms of 
formal de Rham-Witt complexes defined in \cite[(6.1)]{ndw}. 
Hence the following diagram is commutative by 
\cite[(6.8) (4)]{ndw}: 
\begin{equation*} 
\begin{CD} 
{\cal W}_{n+1}\wt{\Om}{}^i_{X_{\ul{\lam}}}
@>{d}>>{\cal W}_{n+1}\wt{\Om}{}^{i+1}_{X_{\ul{\lam}}}\\
@V{R}VV @VV{R}V\\ 
{\cal W}_n\wt{\Om}{}^i_{X_{\ul{\lam}}}
@>{d}>>{\cal W}_n\wt{\Om}{}^{i+1}_{X_{\ul{\lam}}}.  
\end{CD}
\end{equation*}

Set 
\begin{align*} 
{\cal W}\wt{\Om}^{\bul}_{X_{\ul{\lam}}}
:=\vpl_{n}{\cal W}_n\wt{\Om}^{\bul}_{X_{\ul{\lam}}}. 
\tag{21.13.3}\label{ali:ixl}
\end{align*} 
Here the projective system is taken by the projection 
$R\col {\cal W}_n\wt{\Om}^{\bul}_{X_{\ul{\lam}}}\lo 
{\cal W}_{n-1}\wt{\Om}^{\bul}_{X_{\ul{\lam}}}$. 
In a standard way, we also have the operators 
\begin{align*} 
F\col  {\cal W}_n\wt{\Om}^i_{X_{\ul{\lam}}}\lo 
{\cal W}_{n-1}\wt{\Om}^i_{X_{\ul{\lam}}}
\end{align*}
and 
\begin{align*} 
V\col  {\cal W}_n\wt{\Om}^i_{X_{\ul{\lam}}}\lo 
{\cal W}_{n+1}\wt{\Om}^i_{X_{\ul{\lam}}}.
\end{align*}
By \cite[(6.8) (4)]{ndw} $Fd=dF$ 
and $Vd=dV$. 
We also see that 
${\cal W}\wt{\Om}^{\bul}_{X_{\ul{\lam}}}$ 
becomes a Cartan-Dieudonn\'{e}-Raynaud algebra over 
$\Gam(s,{\cal O}_s)$.

\begin{rema}\label{rema:pst}
In the last paragraph in \cite[p.~1264]{kiha} and \cite[Lemma 12]{kiha} the sheaf 
${\cal W}\wt{\Om}^i_{X_{\ul{\lam}}}$ is defined by 
the following formula 
\begin{align*}
{\cal W}\wt{\Om}^i_{X_{\ul{\lam}}}:=\vpl_n
{\cal H}^i(\Om^{\bul}_{{\cal X}_n/{\cal W}_n(\os{\circ}{s})}
\otimes_{{\cal O}_{\cal X}}{\cal O}_{{\cal X}_{\ul{\lam}}}), 
\tag{21.14.1}\label{ali:wixl}
\end{align*} 
where ${\cal X}$ is the log formal scheme in the proof of (\ref{prop:nqop}) (2).  
The reader should note that 
the ${\cal W}\wt{\Om}^i_{X_{\ul{\lam}}}$ in (\ref{ali:ixl}) 
is different from the ${\cal W}\wt{\Om}^i_{X_{\ul{\lam}}}$ 
in (\ref{ali:wixl}): the limit  
$\vpl_n\Om^{\bul}_{{\cal X}_n/{\cal W}_n(\os{\circ}{s})}$ in (\ref{ali:wixl}) 
induces the limit taken with respect to 
the operators 
$F\col {\cal W}_n\wt{\Om}^i_{X_{\ul{\lam}}}=
{\cal H}^i(\Om^{\bul}_{{\cal X}_n/{\cal W}_n(\os{\circ}{s})})
\lo 
{\cal W}_{n-1}\wt{\Om}^i_{X_{\ul{\lam}}}
={\cal H}^i(\Om^{\bul}_{{\cal X}_{n-1}/{\cal W}_{n-1}(\os{\circ}{s})})$'s 
and not taken with respect to the operators 
$R\col {\cal W}_n\wt{\Om}^i_{X_{\ul{\lam}}}\lo 
{\cal W}_{n-1}\wt{\Om}^i_{X_{\ul{\lam}}}$'s; 
the construction of ${\cal W}\wt{\Om}^i_{X_{\ul{\lam}}}$ in [loc.~cit.] seems to be wrong.
\end{rema}

\begin{prop}\label{prop:pjc}
The projection $R$ defined in {\rm (\ref{defi:hdyd}) (2)}  
induces the following morphism 
\begin{align*} 
R\col P_k{\cal W}_{n+1}\wt{\Om}^i_{X_{\ul{\lam}}}\lo 
P_k{\cal W}_n\wt{\Om}^i_{X_{\ul{\lam}}} \quad (k\in {\mab N}). 
\end{align*}
\end{prop}
\begin{proof} 
The proof of this proposition is the same as that of \cite[(8.4) (2)]{ndw}. 
Indeed, assume that (\ref{prop:pjc}) holds for a natural number $k$. 
We proceed on descending induction on $k$. 
It is easy to prove that 
the morphism ${\bf p}\col {\cal W}_{n}\wt{\Om}^i_{X_{\ul{\lam}}}\lo 
{\cal W}_{n+1}\wt{\Om}^i_{X_{\ul{\lam}}}$ induces a morphism 
${\bf p}\col 
P_k{\cal W}_{n}\wt{\Om}^i_{X_{\ul{\lam}}}\lo 
P_k{\cal W}_{n+1}\wt{\Om}^i_{X_{\ul{\lam}}}$ and we see that 
the following diagram is commutative as in \cite[(8.4) (1)]{ndw}: 
\begin{equation*} 
\begin{CD} 
0@>>>P_{k-1}{\cal W}_{n}\wt{\Om}^i_{X_{\ul{\lam}}}
@>>>P_k{\cal W}_n\wt{\Om}^i_{X_{\ul{\lam}}} \\
@.@V{\bf p}VV @V{\bf p}VV \\ 
0@>>>P_{k-1}{\cal W}_{n+1}\wt{\Om}^i_{X_{\ul{\lam}}}
@>>>P_k{\cal W}_{n+1}\wt{\Om}^i_{X_{\ul{\lam}}} 
\end{CD}
\end{equation*} 
\begin{equation*} 
\begin{CD} 
@>>>
\bigoplus_{\# \ul{\mu}=k}a_{\ul{\lam}\cup \ul{\mu}*}
({\cal W}_n{\Om}^{i-k}_{\os{\circ}{X}_{\ul{\lam}\cup \ul{\mu}}}
\otimes_{\mab Z}\vp_{{\rm zar},\ul{\mu}}(\os{\circ}{X}/\os{\circ}{s}))@>>> 0\\
@. @V{\bf p}VV\\
@>>> \bigoplus_{\# \ul{\mu}=k}a_{\ul{\lam}\cup \ul{\mu}*}
({\cal W}_{n+1}\Om^{i-k}_{\os{\circ}{X}_{\ul{\lam}\cup \ul{\mu}}}
\otimes_{\mab Z}\vp_{{\rm zar},\ul{\mu}}(\os{\circ}{X}/\os{\circ}{s}))@>>> 0. 
\end{CD}
\end{equation*} 
Since the following diagram 
\begin{equation*} 
\begin{CD} 
P_k{\cal W}_{n+1}\wt{\Om}^i_{X_{\ul{\lam}}} @>>>
\bigoplus_{\# \ul{\mu}=k}a_{\ul{\lam}\cup \ul{\mu}*}
({\cal W}_{n+1}\Om^{i-k}_{\os{\circ}{X}_{\ul{\lam}\cup \ul{\mu}}}
\otimes_{\mab Z}\vp_{{\rm zar},\ul{\mu}}(\os{\circ}{X}/\os{\circ}{s}))\\
@V{p}VV @VV{p}V \\ 
P_k{\cal W}_{n+1}\wt{\Om}^i_{X_{\ul{\lam}}} @>>>
\bigoplus_{\# \ul{\mu}=k}a_{\ul{\lam}\cup \ul{\mu}*}
({\cal W}_{n+1}\Om^{i-k}_{\os{\circ}{X}_{\ul{\lam}\cup \ul{\mu}}}
\otimes_{\mab Z}\vp_{{\rm zar},\ul{\mu}}(\os{\circ}{X}/\os{\circ}{s})) 
\end{CD}
\end{equation*} 
is commutative, the following diagram 
\begin{equation*} 
\begin{CD} 
P_k{\cal W}_{n+1}\wt{\Om}^i_{X_{\ul{\lam}}} @>>>
\bigoplus_{\# \ul{\mu}=k}a_{\ul{\lam}\cup \ul{\mu}*}
({\cal W}_{n+1}\Om^{i-k}_{\os{\circ}{X}_{\ul{\lam}\cup \ul{\mu}}}
\otimes_{\mab Z}\vp_{{\rm zar},\ul{\mu}}(\os{\circ}{X}/\os{\circ}{s}))\\
@V{R}VV @VV{R}V \\ 
P_k{\cal W}_{n}\wt{\Om}^i_{X_{\ul{\lam}}} @>>>
\bigoplus_{\# \ul{\mu}=k}a_{\ul{\lam}\cup \ul{\mu}*}
({\cal W}_{n}\Om^{i-k}_{\os{\circ}{X}_{\ul{\lam}\cup \ul{\mu}}}
\otimes_{\mab Z}\vp_{{\rm zar},\ul{\mu}}(\os{\circ}{X}/\os{\circ}{s})) 
\end{CD}
\tag{21.15.1}\label{cd:wnpk}
\end{equation*} 
is commutative. 
By (\ref{ali:ppkkxl}) and (\ref{cd:wnpk}) we see that 
the morphism 
$R\col P_k{\cal W}_{n+1}\wt{\Om}^i_{X_{\ul{\lam}}}\lo 
P_k{\cal W}_n\wt{\Om}^i_{X_{\ul{\lam}}}$ induces 
the morphism 
$R\col P_{k-1}{\cal W}_{n+1}\wt{\Om}^i_{X_{\ul{\lam}}}\lo 
P_{k-1}{\cal W}_n\wt{\Om}^i_{X_{\ul{\lam}}}$. 
\end{proof} 

\begin{coro}\label{coro:sj}
The projection 
$R\col P_k{\cal W}_{n+1}\wt{\Om}^i_{X_{\ul{\lam}}}
\lo P_k{\cal W}_{n}\wt{\Om}^i_{X_{\ul{\lam}}}$ is surjective. 
\end{coro}
\begin{proof} 
Since the projections 
$R\col  {\cal W}_{n+1}\wt{\Om}^i_{X_{\ul{\lam}}}
\lo {\cal W}_{n}\wt{\Om}^i_{X_{\ul{\lam}}}$
and 
$R\col {\cal W}_{n+1}{\Om}^{i-k}_{\os{\circ}{X}_{\ul{\lam}\cup \ul{\mu}}}
\otimes_{\mab Z}\vp_{{\rm zar},\ul{\mu}}(\os{\circ}{X}/\os{\circ}{s})
\lo 
{\cal W}_{n}\Om^{i-k}_{\os{\circ}{X}_{\ul{\lam}\cup \ul{\mu}}}
\otimes_{\mab Z}\vp_{{\rm zar},\ul{\mu}}(\os{\circ}{X}/\os{\circ}{s})$ 
are surjective, descending induction tells us that 
the morphism $R\col P_k{\cal W}_{n+1}\wt{\Om}^i_{X_{\ul{\lam}}}
\lo P_k{\cal W}_{n}\wt{\Om}^i_{X_{\ul{\lam}}}$ is surjective. 
\end{proof} 


\par
Set 
\begin{align*} 
({\cal W}\wt{\Om}^{\bul}_{X_{\ul{\lam}}},P):=
\vpl_n({\cal W}_{n}\wt{\Om}^{\bul}_{X_{\ul{\lam}}},P) 
\end{align*}
and 
\begin{align*} 
(s({\cal W}\wt{\Om}^{\bul}_{X^{(\star)}}),P):=
\vpl_n(s({\cal W}_n\wt{\Om}^{\bul}_{X^{(\star)}}),P) \quad (m\in {\mab N}).
\end{align*}
By (\ref{ali:ppkkxl}) and (\ref{cd:wnpk}) we obtain the following isomorphisms:  
\begin{align*} 
{\rm gr}_k^P{\cal W}\wt{\Om}^{\bul}_{X_{\ul{\lam}}}=
\bigoplus_{\# \ul{\mu}=k}a_{\ul{\lam}\cup \ul{\mu}*}
({\cal W}\Om^{\bul}_{\os{\circ}{X}_{\ul{\lam}\cup \ul{\mu}}}[-k]
\otimes_{\mab Z}\vp_{{\rm zar},\ul{\mu}}(\os{\circ}{X}/\os{\circ}{s}))
\tag{21.16.1}\label{ali:grpwl}
\end{align*}
and 
\begin{align*} 
{\rm gr}^P_k{\cal W}\wt{\Om}{}^{\bul}_{X^{(m)}}=
\bigoplus_{\# \ul{\lam}=m+1}
\bigoplus_{\# \ul{\mu}=k}a_{\ul{\lam}\cup \ul{\mu}*}
({\cal W}\Om^{\bul}_{\os{\circ}{X}_{\ul{\lam}\cup \ul{\mu}}}[-k]
\otimes_{\mab Z}\vp_{{\rm zar},\ul{\mu}}(\os{\circ}{X}/\os{\circ}{s})). 
\tag{21.16.2}\label{ali:pkggrm}
\end{align*}

\par 
Let $E$ be a quasi-coherent crystal of ${\cal O}_{\os{\circ}{X}_n/{\cal W}(\os{\circ}{s})}$-modules. 
Set $E_n:=E_{{\cal W}_n(\os{\circ}{X})}$. 
Then we have the integrable connection 
\begin{align*}
\nabla \col E_n\lo 
E_n\otimes_{{\cal W}_n({\cal O}_X)}{\cal W}_n\wt{\Om}{}^1_{X_{\ul{\lam}}}. 
\tag{21.16.3}\label{eqn:lnmyy} 
\end{align*} 
Indeed, we have first obtained the log de Rham complex 
$E_n\otimes_{{\cal W}_n({\cal O}_X)}\Om^{\bul}_{{\cal W}_n(X)/{\cal W}_n(\os{\circ}{s}), [~]}$ 
and then we have the following natural composite morphism
\begin{align*} 
E_n\otimes_{{\cal W}_n({\cal O}_X)}\Om^{\bul}_{{\cal W}_n(X)/{\cal W}_n(\os{\circ}{s}), [~]}
&\lo E_n\otimes_{{\cal W}_n({\cal O}_X)}
\Om^{\bul}_{{\cal W}_n(X_{\ul{\lam}})/{\cal W}_n(\os{\circ}{s}), [~]}\\
&\lo 
E_n\otimes_{{\cal W}_n({\cal O}_X)}{\cal W}_n\wt{\Om}^{\bul}_{X_{\ul{\lam}}}. 
\end{align*} 
Here the last morphism is obtained by the following morphism 
\begin{align*} 
{\Om}^{\bul}_{{\cal W}_n(X_{\ul{\lam}})/{\cal W}_n(\os{\circ}{s}), [~]}
\lo {\cal W}_n\wt{\Om}^{\bul}_{X_{\ul{\lam}}}, 
\tag{21.16.4}\label{eqn:lnyy} 
\end{align*}  
which is, by definition, the induced morphism by 
$s_n(0,0)$ in (\ref{eqn:sn00}) and the following two morphisms
\begin{equation*}
s_n(1,0)' \col 
\Om^1_{{\cal W}_n(X_{\ul{\lam}})/{\cal W}_n(\os{\circ}{s}), [~]} \owns 
d(a_0, \ldots, a_{n-1}) \lom 
\left[\sum_{i=0}^{n-1}\wt{a}_i^{p^{n-i}-1}d\wt{a}_i\right]
\in {\cal H}^1(\Om^*_{{\mathfrak D}_{\ul{\lam}}/{\cal W}_n(\os{\circ}{s}),[~]}),  
\tag{21.16.5}\label{eqn:tise10}
\end{equation*}
\begin{equation*}
\iota  \col \Om^1_{{\cal W}_n(X_{\ul{\lam}})/{\cal W}_n(\os{\circ}{s}), [~]} \owns 
d\log b \lom \left[d\log \wt{b}\right]\in 
{\cal H}^1(\Om^*_{{\mathfrak D}_{\ul{\lam}}/{\cal W}_n(\os{\circ}{s}),[~]}),
\tag{21.16.6}\label{eqn:tdlogb}
\end{equation*}
where ${\mathfrak D}_{\ul{\lam}}$ is the log scheme over 
$({\cal W}_n(s),p{\cal W}_n,[~])$ defined in the proof of (\ref{prop:ab})  
and 
$\wt{b}\in M_{{\mathfrak D}_{\ul{\lam}}}$ are lifts of 
$a_i\in {\cal O}_{X_{\ul{\lam}}}$ and 
$b\in M_{X_{\ul{\lam}}}\subset M_{{\cal W}_n(X_{\ul{\lam}})}$, respectively.
By the same proof as that for the well-definedness of $s_n(0,0)$ again, 
$s_n(1,0)'$ and $\iota$ are well-defined (\cite[(4.9)]{hk}). 
In fact, we obtain the log de Rham complex 
$E_n\otimes_{{\cal W}_n({\cal O}_X)}
{\cal W}_n\wt{\Om}{}^{\bul}_{X_{\ul{\lam}}}$ as in \cite[(2.2.10)]{nb}.

\begin{theo}\label{theo:citt}
Assume that $E$ is locally free.  
Then the following hold$:$
\par 
$(1)$ There exists a canonical filtered isomorphism 
\begin{equation*}
(\wt{R}u_{X_{\ul{\lam}}/{\cal W}_n(\os{\circ}{s})*}
(\eps^*_{X_{\ul{\lam}}/{\cal W}_n(\os{\circ}{s})}(E_{\ul{\lam}})),P) 
\os{\sim}{\lo}  
(E_n\otimes_{{\cal W}_n({\cal O}_{X})}
{\cal W}_n\wt{\Om}^{\bul}_{X_{\ul{\lam}}},P).
\tag{21.17.1}\label{eqn:nrlowln}
\end{equation*}
The isomorphism $(\ref{eqn:nrlowln})$ is compatible with the projections.
\par 
$(2)$ The isomorphism {\rm (\ref{eqn:nrlowln})} is contravariantly functorial 
with respect to a morphism satisfying the conditions {\rm (8.1.6)} for the case $S=s$ and $T={\cal W}_n(s)$.   
\end{theo} 
\begin{proof} 
(1): Let $X_{\bul}$ be the \v{C}ech diagram obtained by an affine open covering of $X$ 
such that there exists  a simplicial immersion 
$X_{\bul}\os{\sus}{\lo} \ol{\cal P}_{\bul}$ into a log smooth 
scheme over $\ol{{\cal W}_n(s)}$. 
Let $\ol{\mathfrak D}_{\bul}$ be the log PD-envelope of this immersion 
over $({\cal W}_n(\os{\circ}{s}),p{\cal W}_n,[~])$. 
Since the immersion $X_0 \os{\sus}{\lo} {\cal W}_n(X_0)$ is nilpotent
and since $\ol{\cal P}_{\bul}$ is log smooth over $\ol{{\cal W}_n(s)}$, 
we have a morphism ${\cal W}_n(X_0)\lo \ol{\cal P}_0$ extending 
the immersion $X_0\os{\sus}{\lo} \ol{\cal P}_0$. 
Hence we have a morphism ${\cal W}_n(X_{\bul})\lo \ol{\cal P}_{\bul}$ 
extending the immersion $X_{\bul}\os{\sus}{\lo} \ol{\cal P}_{\bul}$. 
This morphism induces a morphism 
${\cal W}_n(X_{\bul})\lo \ol{\cal P}{}^{\rm ex}_{\bul}$ 
extending the immersion $X_{\bul}\os{\sus}{\lo} \ol{\cal P}{}^{\rm ex}_{\bul}$. 
Set $\ol{\cal E}{}^{\bul}:=\eps^*_{X/{\cal W}_n(\os{\circ}{s})}(E)_{\ol{\mathfrak D}_{\bul}}$. 
Set ${\cal P}{}^{\rm ex}_{\bul}:=\ol{\cal P}{}^{\rm ex}_{\bul}\times_{\ol{{\cal W}_n(s)}}{\cal W}_n(s)$.  
Set ${\mathfrak D}_{\bul}:=\ol{\mathfrak D}_{\bul}
\times_{\ol{{\cal W}_n(s)}}{\cal W}_n(s)$. 
Set ${\cal E}{}^{\bul}:=\ol{\cal E}{}^{\bul}
\otimes_{{\cal O}_{\ol{\mathfrak D}_{\bul}}}{\cal O}_{{\mathfrak D}_{\bul}}$. 
The morphism ${\cal W}_n(X_{\bul})\lo \ol{\cal P}{}^{\rm ex}_{\bul}$ 
induces the following morphism
\begin{align*}
{\cal E}^{\bul}\otimes_{{\cal O}_{{\cal P}^{\rm ex}_{\bul}}}
\Om^{\bul}_{{\cal P}^{\rm ex}_{\bul \ul{\lam}}/{\cal W}_n(\os{\circ}{s})} \lo 
{\cal E}^{\bul}\otimes_{{\cal O}_{{\cal P}^{\rm ex}_{\bul}}}
\Om^{\bul}_{{\cal W}_n(X_{\ul{\lam},\bul})/({\cal W}_n(\os{\circ}{s})), [~]}.
\tag{21.17.2}\label{ali:xwspp}
\end{align*}
Let 
$$\pi \col ((X_{\bul})_{\rm zar},f^{-1}_{\bul}({\cal W}_n))\lo (X_{\rm zar},f^{-1}({\cal W}_n))$$ 
be the natural morphism of ringed topoi.
The cosimplicial version of the morphism (\ref{eqn:lnyy}) and 
the morphism ${\cal W}_n(X_{\bul})\lo \ol{\cal P}{}^{\rm ex}_{\bul}$ 
induces the following morphism 
\begin{align*}
{\cal E}^{\bul}\otimes_{{\cal O}_{{\cal P}^{\rm ex}_{\bul}}}
\Om^{\bul}_{{\cal W}_n(X_{\ul{\lam},\bul})/({\cal W}_n(\os{\circ}{s})), [~]}
\lo E^{\bul}_n\otimes_{{\cal O}_{{\cal P}^{\rm ex}_{\bul}}}
\pi^{-1}({\cal W}_n\wt{\Om}{}^{\bul}_{X_{\ul{\lam}}}).
\tag{21.17.3}\label{ali:xwkap}
\end{align*}
By (\ref{ali:xwspp}) and (\ref{ali:xwkap}) 
we have the following composite morphism
\begin{align*}
{\cal E}^{\bul}\otimes_{{\cal O}_{{\cal P}^{\rm ex}_{\bul}}}
\Om^{\bul}_{{\cal P}^{\rm ex}_{\bul \ul{\lam}}/{\cal W}_n(\os{\circ}{s})} \lo 
{\cal E}^{\bul}\otimes_{{\cal O}_{{\cal P}^{\rm ex}_{\bul}}}
\Om^{\bul}_{{\cal W}_n(X_{\ul{\lam},\bul})/{\cal W}_n(\os{\circ}{s}), [~]}
\lo E^{\bul}_n\otimes_{{\cal O}_{{\cal P}^{\rm ex}_{\bul}}}
\pi^{-1}({\cal W}_n\wt{\Om}{}^{\bul}_{X_{\ul{\lam}}}).
\tag{21.17.4}\label{ali:xwkabpp}
\end{align*}
This composite morphism preserves $P$'s as in \cite[(2.3.18.3)]{nb}. 
Hence we have the following morphism 
\begin{align*}
({\cal E}^{\bul}\otimes_{{\cal O}_{{\cal P}^{\rm ex}_{\bul}}}
\Om^{\bul}_{{\cal P}^{\rm ex}_{\bul \ul{\lam}}/{\cal W}_n(\os{\circ}{s})},P) \lo 
(E^{\bul}_n\otimes_{{\cal O}_{{\cal P}^{\rm ex}_{\bul}}}
\pi^{-1}({\cal W}_n\wt{\Om}{}^{\bul}_{X_{\ul{\lam}}}),P).
\tag{21.17.5}\label{ali:pwkpp}
\end{align*}
By the cohomological descent for a bounded below complex, 
we have the morphism (\ref{eqn:nrlowln}). 
It is easy to check that this morphism 
is independent of the choice of 
the simplicial immersion  $X_{\bul}\os{\sus}{\lo} \ol{\cal P}_{\bul}$.  
\par
We claim that the morphism (\ref{eqn:nrlowln}) is a filtered isomorphism.  
This is a local problem on $X_{\ul{\lam},{\rm zar}}$. Hence we may assume that 
the filtrations $P$'s on the sources and the targets of (\ref{eqn:nrlowln}) 
are finite. It suffices to prove that 
the morphism 
\begin{equation*}
{\rm gr}_k^P
(\wt{R}u_{X_{\ul{\lam}}/{\cal W}_n(\os{\circ}{s})*}
(\eps^*_{X_{\ul{\lam}}/{\cal W}_n(\os{\circ}{s})}(E_{\ul{\lam}})),P) \lo 
{\rm gr}_k^P
(E_n\otimes_{{\cal W}_n({\cal O}_{X})}
{\cal W}_n\wt{\Om}^{\bul}_{X_{\ul{\lam}}}).
\tag{21.17.6}\label{eqn:nroowln}
\end{equation*}
is an isomorphism. 
The source of (\ref{eqn:nroowln}) is equal to 
\begin{align*}
\bigoplus_{\# \ul{\mu}=k} 
a_{\ul{\lam}\cup \ul{\mu}*}
Ru_{\os{\circ}{X}_{\ul{\lam}\cup \ul{\mu}}
/{\cal W}_n(\os{\circ}{s})}
(E_{\os{\circ}{X}_{\ul{\lam}\cup \ul{\mu}}
/{\cal W}_n(\os{\circ}{s})}
\otimes_{\mab Z}
\vp_{{\rm crys},\ul{\mu}}(\os{\circ}{X}/{\cal W}_n(\os{\circ}{s})))[-k]. 
\end{align*} 
by (\ref{eqn:ele}). 
On the other hand the target of (\ref{eqn:nroowln}) is equal to 
\begin{align*} 
&\bigoplus_{\# \ul{\mu}=k}
E_n\otimes_{{\cal W}_n({\cal O}_{X})}a_{\ul{\lam}\cup \ul{\mu}*}
({\cal W}_n\Om^{\bul}_{\os{\circ}{X}_{\ul{\lam}\cup \ul{\mu}}}[-k]
\otimes_{\mab Z}\vp_{{\rm zar},\ul{\mu}}(\os{\circ}{X}/\os{\circ}{s}))\\
&=
\bigoplus_{\# \ul{\mu}=k}
a_{\ul{\lam}\cup \ul{\mu}*}(E_n\vert_{\os{\circ}{X}_{\ul{\lam}\cup \ul{\mu}}}
\otimes_{{\cal W}_n({\cal O}_{\os{\circ}{X}_{\ul{\lam}\cup \ul{\mu}}})}
{\cal W}_n\Om^{\bul}_{\os{\circ}{X}_{\ul{\lam}\cup \ul{\mu}}}[-k]
\otimes_{\mab Z}\vp_{{\rm zar},\ul{\mu}}(\os{\circ}{X}/\os{\circ}{s}))
\end{align*} 
by (\ref{ali:ppkkxl}). 
Hence (\ref{eqn:nroowln}) is an isomorphism by 
\cite[III (1.5)]{ir} and Etesse's comparison theorem 
\cite[II (2.1)]{et}.  
\par 
The compatibility of the isomorphism (\ref{eqn:nrlowln})
with the projections follows from the proof of \cite[(7.18)]{ndw}.
\par 
(2): Using the argument after (\ref{prop:indu}),  we can prove the functoriality 
without difficulty. We leave the detail of the proof to the reader. 
\par 
We finish the proof.
\end{proof}

\begin{coro}\label{cobro}
$(1)$ 
Let $s(E_n\otimes_{{\cal W}_n({\cal O}_{X})}
{\cal W}_n\wt{\Om}^{\bul}_{X^{(\bul)}})$ be the single complex of the double complex 
$(\cdots \lo E_n\otimes_{{\cal W}_n({\cal O}_{X})}
{\cal W}_n\wt{\Om}^{\bul}_{X^{(m)}}\lo \cdots)_{m\in {\mab N}}$, 
where $E_n\otimes_{{\cal W}_n({\cal O}_{X})}
{\cal W}_n\wt{\Om}^{\bul}_{X^{(m)}}:=\bigoplus_{\# \ul{\lam}=m+1}
E_n\otimes_{{\cal W}_n({\cal O}_{X})}
{\cal W}_n\wt{\Om}^{\bul}_{X_{\ul{\lam}}}$. 
Then there exists a canonical filtered isomorphism 
\begin{equation*}
(\wt{R}u_{X^{(\star)}/{\cal W}_n(\os{\circ}{s})*}
(\eps^*_{X^{(\star)}/{\cal W}_n(\os{\circ}{s})}(E)),P) 
\os{\sim}{\lo}  (s(E_n\otimes_{{\cal W}_n({\cal O}_{X})}
{\cal W}_n\wt{\Om}^{\bul}_{X^{(\bul)}}),P).
\tag{21.18.1}\label{eqn:nrlsln}
\end{equation*}
The isomorphism is compatible with the projections.
\par 
$(2)$ The isomorphism {\rm (\ref{eqn:nrlsln})} is contravariantly functorial 
with respect to a morphism satisfying the conditions 
{\rm (8.1.6)} for the case $S=s$ and $T={\cal W}_n(s)$.   
\par
$(3)$ Let 
$$
\wt{R}u_{X/{\cal W}_n(\os{\circ}{s})*}
(\eps^*_{X/{\cal W}_n(\os{\circ}{s})}(E))
\os{\sim}{\lo} 
\wt{R}u_{X^{(\star)}/{\cal W}_n(\os{\circ}{s})*}
(\eps^*_{X^{(\star)}/{\cal W}_n(\os{\circ}{s})}(E^{(\star)}))
$$ be the isomorphism {\rm (\ref{eqn:e})} in the case $S=s$ and $T={\cal W}_n(s)$. 
Then the natural morphism 
\begin{equation*}
E_n\otimes_{{\cal W}_n({\cal O}_X)}{\cal W}_n\wt{\Om}{}^{\bul}_X 
\lo  
s(E_n\otimes_{{\cal W}_n({\cal O}_X)}{\cal W}_n\wt{\Om}{}^{\bul}_{X^{(\bul)}})
\tag{21.18.2}\label{eqn:exwte}
\end{equation*} 
in $C^+(f^{-1}({\cal W}_n))$ 
fits into the following commutative diagram
\begin{equation*}
\begin{CD} 
\wt{R}u_{X/{\cal W}_n(\os{\circ}{s})*}
(\eps^*_{X/{\cal W}_n(\os{\circ}{s})}(E))
@>{\sim}>> 
\wt{R}u_{X^{(\star)}/{\cal W}_n(\os{\circ}{s})*}(\eps^*_{X^{(\star)}/{\cal W}_n(\os{\circ}{s})}
(E^{(\star)}))\\
@V{\simeq}V{{\rm (\ref{eqn:nrlowln})}{\rm ~for~the~case}~\ul{\lam}=\emptyset}V 
@V{{\rm (\ref{eqn:nrlsln})}}V{\simeq}V \\
E_n\otimes_{{\cal W}_n({\cal O}_X)}{\cal W}_n\wt{\Om}{}^{\bul}_X 
@>>>
s(E_n\otimes_{{\cal W}_n({\cal O}_X)}{\cal W}_n\wt{\Om}{}^{\bul}_{X^{(\bul)}})
\end{CD} 
\tag{21.18.3}\label{eqn:eweate}
\end{equation*} 
in $D^+(f^{-1}({\cal W}_n)$.  
Here the upper isomorphism in {\rm (\ref{eqn:eweate})} is the isomorphism 
{\rm (\ref{eqn:e})} in the special case. 
Consequently the morphism {\rm (\ref{eqn:exwte})} is a quasi-isomorphism. 
\end{coro}
\begin{proof} 
We can construct the morphism (\ref{eqn:nrlsln}) as in the construction of 
the isomorphism (\ref{eqn:nrlowln}). 
Now (\ref{cobro}) immediately follows from (\ref{theo:citt}). 
\end{proof}

Let $R\col {\cal W}_{n+1}\wt{\Om}^1_{X_{\ul{\lam}}}\lo 
{\cal W}_n\wt{\Om}^1_{X_{\ul{\lam}}}$ be the projection. 
By (\ref{prop:t}),  $R(\theta_{\ul{\lam},n+1})=\theta_{\ul{\lam},n}$. 
Set 
\begin{align*} 
\theta:=\vpl_n \theta_n\in {\rm Ker}({\cal W}\wt{\Om}^1_{X_{\ul{\lam}}}\lo 
{\cal W}\wt{\Om}^2_{X_{\ul{\lam}}}).
\end{align*} 
Then we can consider the following Hirsch extensions
\begin{align*} 
({\cal W}\wt{\Om}^{\bul}_{X_{\ul{\lam}}}\langle u \rangle,P):=
\vpl_n({\cal W}_{n}\wt{\Om}^{\bul}_{X_{\ul{\lam}}}\langle u \rangle,P) 
\end{align*}
and the following filtered complex 
\begin{align*} 
({\cal W}\wt{\Om}^{\bul}_{X^{(*)}}\langle u \rangle,P):=
(s(a^{(\bul)}_*({\cal W}\wt{\Om}^{\bul}_{X^{(\bul)}}\langle u \rangle)),P):=
\vpl_n(s(a^{(\bul)}_*({\cal W}_n\wt{\Om}{}^{\bul}_{X^{(\bul)}}\langle u \rangle)),P) 
\quad (m\in {\mab N}).
\end{align*}
By (\ref{ali:grpwl}) and (\ref{ali:pkggrm}),  
\begin{align*} 
{\rm gr}_k^P
a_{\ul{\lam}*}({\cal W}\wt{\Om}^{\bul}_{X_{\ul{\lam}}}\langle u \rangle)=
\bigoplus_{j=0}^{\inf}\bigoplus_{\# \ul{\mu}=k-2j}
{\cal W}u^{[j]}\otimes_{\cal W}
a_{\ul{\lam}\cup \ul{\mu}*}
({\cal W}\Om^{\bul -k}_{\os{\circ}{X}_{\ul{\lam}\cup \ul{\mu}}}
\otimes_{\mab Z}\vp_{{\rm zar},\ul{\mu}}(\os{\circ}{X}/\os{\circ}{s}))
\tag{21.18.4}\label{ali:grpmwl}
\end{align*}
and 
\begin{align*} 
{\rm gr}^P_k{\cal W}\wt{\Om}^{\bul}_{X^{(*)}}\langle u \rangle=
\bigoplus_{j=0}^{\inf}\bigoplus_{m=0}^{\inf}
\bigoplus_{\# \ul{\lam}=m}\bigoplus_{\# \ul{\mu}=k+m-2j}
a_{\ul{\lam}\cup \ul{\mu}*}
({\cal W}\Om^{\bul -k}_{\os{\circ}{X}_{\ul{\lam}\cup \ul{\mu}}}
\otimes_{\mab Z}\vp_{{\rm zar},\ul{\mu}}(\os{\circ}{X}/\os{\circ}{s})). 
\tag{21.18.5}\label{ali:pkgmrm}
\end{align*}

\begin{defi}
Set 
\begin{align*} 
({\cal W}_nH_X(E), P):=
(s(E_n\otimes_{{\cal W}_n({\cal O}_X)}
{\cal W}_n\wt{\Om}{}^{\bul}_{X^{(\bul)}}\langle u \rangle), P)
\end{align*}  
and 
\begin{align*} 
({\cal W}H_X(E), P):=
\vpl_n({\cal W}_nH_X(E), P). 
\end{align*} 
We call $({\cal W}H_X(E), P)$ {\it Kim-Hain's filtered complex} of $E$. 
When $E$ is trivial, we call 
$({\cal W}H_X(E), P)$ {\it Kim-Hain's filtered complex} of $X/{\cal W}$ 
and we denote it by  $({\cal W}H_X, P)$. 
\end{defi}

\par 
Next we define ${\cal W}\Om^{\bul}_{X_{\ul{\lam}}}$ 
which has not been defined in \cite{kiha}.  
\par 
Let $i$ be a nonnegative integer and let $n$ be a positive integer. 
Consider a local immersion 
$X\os{\sus}{\lo} {\cal P}$ as in the proof of (\ref{prop:cpp}). 
Let us recall the complex (\ref{eqn:lbid}) in the case of the trivial coefficient and 
set 
\begin{align*} 
{\cal W}_n{\Om}{}^i_{X_{\ul{\lam}}}:=
{\cal H}^i(Ru_{X_{\ul{\lam}}/{\cal W}_n(s)*}
({\cal O}_{X_{\ul{\lam}}/{\cal W}_n(s)})).
\end{align*} 
The wedge product of $\theta_{\ul{\lam}}\in {\cal W}_n\wt{\Om}{}^1_{X_{\ul{\lam}}}$ 
induces the following morphism 
\begin{align*} 
\theta_{\ul{\lam}}\wedge \col {\cal W}_n\Om^{i-1}_{X_{\ul{\lam}}} 
\lo {\cal W}_n\wt{\Om}^{i}_{X_{\ul{\lam}}} \quad (i\in {\mab N}). 
\end{align*}

\begin{prop}\label{prop:em}
The abelian sheaf ${\cal W}_n\Om^i_{X_{\ul{\lam}}}$ $(i\in {\mab N})$ 
fits into the following exact sequence$:$ 
\begin{align*} 
0\lo {\cal W}_n\Om^{i-1}_{X_{\ul{\lam}}} 
\lo {\cal W}_n\wt{\Om}^{i}_{X_{\ul{\lam}}} \lo 
{\cal W}_n\Om^{i}_{X_{\ul{\lam}}} \lo 0. 
\tag{21.20.1}\label{ali:wnii}
\end{align*} 
\end{prop} 
\begin{proof} 
The problem is local. 
Consider the following locally split exact sequence  as in (\ref{ali:gsflaxd}): 
\begin{align*} 
0  \lo 
{\cal O}_{{\mathfrak D}}\otimes_{{\cal O}_{{\cal P}{}^{\rm ex}}}
\Om^{\bul}_{{\cal P}{}^{{\rm ex}}/{\cal W}_n(s)}[-1] 
\os{d\log t\wedge}{\lo} 
{\cal O}_{{\mathfrak D}}\otimes_{{\cal O}_{{\cal P}{}^{\rm ex}}}
\Om^{\bul}_{{\cal P}{}^{{\rm ex}}/{\cal W}_n(\os{\circ}{s})} 
\lo {\cal O}_{{\mathfrak D}}\otimes_{{\cal O}_{{\cal P}{}^{\rm ex}}}
\Om^{\bul}_{{\cal P}{}^{{\rm ex}}/{\cal W}_n(s)}
\lo 0
\tag{21.20.2}\label{ali:gsxd}
\end{align*} 
This exact sequence gives us the following sequence: 
\begin{align*} 
0  \lo 
{\cal O}_{{\mathfrak D}}\otimes_{{\cal O}_{{\cal P}{}^{\rm ex}}}
\Om^{\bul}_{{\cal P}{}^{{\rm ex}}_{\ul{\lam}}/{\cal W}_n(s)}[-1] 
\os{d\log t\wedge}{\lo} 
{\cal O}_{{\mathfrak D}}\otimes_{{\cal O}_{{\cal P}{}^{\rm ex}}}
\Om^{\bul}_{{\cal P}{}^{{\rm ex}}_{\ul{\lam}}/{\cal W}_n(\os{\circ}{s})} 
\lo {\cal O}_{{\mathfrak D}}\otimes_{{\cal O}_{{\cal P}{}^{\rm ex}}}
\Om^{\bul}_{{\cal P}{}^{{\rm ex}}_{\ul{\lam}}/{\cal W}_n(s)}
\lo 0
\tag{21.20.3}\label{ali:gsxxd}
\end{align*} 
By the same proof as that of (\ref{theo:qii}), the following sequence 
is exact: 
\begin{align*} 
0 &\lo 
{\cal H}^{i-1}({\cal O}_{{\mathfrak D}}\otimes_{{\cal O}_{{\cal P}{}^{\rm ex}}}
\Om^{\bul}_{{\cal P}{}^{{\rm ex}}_{\ul{\lam}}/{\cal W}_n(s)}) 
\os{d\log t\wedge }{\lo} 
{\cal H}^i({\cal O}_{{\mathfrak D}}\otimes_{{\cal O}_{{\cal P}{}^{\rm ex}}}
\Om^{\bul}_{{\cal P}{}^{{\rm ex}}_{\ul{\lam}}/{\cal W}_n(\os{\circ}{s})})  
\tag{21.20.4}\label{ali:gsxxhd}\\
&\lo {\cal H}^i({\cal O}_{{\mathfrak D}}\otimes_{{\cal O}_{{\cal P}{}^{\rm ex}}}
\Om^{\bul}_{{\cal P}{}^{{\rm ex}}_{\ul{\lam}}/{\cal W}_n(s)}) \lo 0
\end{align*} 
This means that the following sequence is exact: 
\begin{align*} 
0  \lo {\cal W}_n\wt{\Om}^0_{X_{\ul{\lam}}}\os{\theta_{\ul{\lam}}}{\lo} \cdots 
 \os{\theta_{\ul{\lam}}}{\lo}
{\cal W}_n\wt{\Om}^{i-1}_{X_{\ul{\lam}}}\os{\theta_{\ul{\lam}}}{\lo} 
{\cal W}_n\wt{\Om}^{i}_{X_{\ul{\lam}}}
\lo {\cal W}_n\Om^{i}_{X_{\ul{\lam}}} \lo 0. 
\tag{21.20.5}\label{ali:gsd}
\end{align*} 
Because we see that 
${\cal W}_n\Om^{i-1}_{X_{\ul{\lam}}}
={\rm Coker}(\theta_{\ul{\lam}}\col {\cal W}_n\wt{\Om}^{i-2}_{X_{\ul{\lam}}} {\lo}
{\cal W}_n\wt{\Om}^{i-1}_{X_{\ul{\lam}}})$ 
by the local calculation, we obtain the exact sequence (\ref{ali:wnii}).  
\end{proof}

\begin{coro}\label{coro:ci}
There exists an inverse Cartier isomorphism 
\begin{align*} 
C^{-1} \col \Om^{i}_{X_{\ul{\lam}}}\os{\sim}{\lo} 
{\cal W}_1\Om^{i}_{X_{\ul{\lam}}}.
\tag{21.21.1}\label{ali:cxul}
\end{align*} 
\end{coro}
\begin{proof} 
The inverse Cartier isomorphism 
$C^{-1} \col \wt{\Om}^{i}_{X_{\ul{\lam}}}\os{\sim}{\lo} 
{\cal W}_1\wt{\Om}^{i}_{X_{\ul{\lam}}}$ 
(the proof of \cite[Lemma 12]{kiha})
induces the inverse Cartier isomorphism (\ref{ali:cxul}) by 
(\ref{ali:gsd}) since $C^{-1}(d\log t)=d\log t$. 
\end{proof} 


\begin{rema}\label{rema:a}
Let the notations be as in the proof of (\ref{prop:em}). 
By (\ref{coro:ci}), $(\Om^{\bul}_{{\cal X}_{\ul{\lam}}},\phi^*)$ satisfies 
$(21.12.1)\sim (21.12.5)$. 
Hence all the results in \cite[\S6]{ndw} for formal de Rham-Witt complexes 
hold for $\{{\cal W}_n\Om^{i}_{X_{\ul{\lam}}}\}_{n=1}^{\inf}$. 
For example,  

\begin{align*} 
{\rm Fil}^r{\cal W}_{n+1}\Om^i_{X_{\ul{\lam}}}=
V^r{\cal W}_{n+1-r}\Om^i_{X_{\ul{\lam}}}+
dV^r{\cal W}_{n+1-r}\Om^{i-1}_{X_{\ul{\lam}}}. 
\tag{21.22.1}\label{eqn:funfv}
\end{align*} 

\begin{align*} 
d^{-1}(p^n {\cal W}\Om^{i+1}_{X_{\ul{\lam}}})=F^n{\cal W}\Om^i_{X_{\ul{\lam}}}, 
\tag{21.22.2}\label{eqn:fucnfv}
\end{align*} 

\begin{align*}
R_n\otimes^L_R{\cal W}\Om^{\bul}_{X_{\ul{\lam}}} 
={\cal W}_n\Om^{\bul}_{X_{\ul{\lam}}}, 
\tag{21.22.3}\label{eqn:fundfv}
\end{align*}

\begin{equation*}
{\rm Ker}(d \col 
{\cal W}_n\Om^{i}_{X_{\ul{\lam}}}
\lo {\cal W}_n\Om^{i+1}_{X_{\ul{\lam}}})
=F^n({\cal W}_{2n}\Om^{i}_{X_{\ul{\lam}}}). 
\tag{21.22.4}\label{dkfn}
\end{equation*}

\end{rema}

\par 
By using (\ref{coro:ci}) and by the same proof as that of (\ref{prop:nqop}), 
we obtain the Hyodo's multiplication 
${\bf p} \col {\cal W}_n\Om^{\bul}_{X_{\ul{\lam}}}\lo 
{\cal W}_{n+1}\Om^{\bul}_{X_{\ul{\lam}}}$ 
and the projection 
$R\col \col {\cal W}_{n+1}\Om^{\bul}_{X_{\ul{\lam}}}\lo 
{\cal W}_{n}\Om^{\bul}_{X_{\ul{\lam}}}$ 
such that ${\bf p}R=p=R{\bf p}$. 

\par 
Let $F$ be a quasi-coherent crystal of ${\cal O}_{X/{\cal W}(s)}$-modules. 
Set $F_n:=F_{{\cal W}_n(X)}$. 
Then we have the following integrable connection 
\begin{align*}
\nabla \col F_n\lo F_n\otimes_{{\cal W}_n({\cal O}_X)}
{\cal W}_n\Om^1_{X_{\ul{\lam}}}
\end{align*} 
as in (\ref{eqn:lnmyy}) 
and in fact, we obtain the log de Rham complex 
$F_n\otimes_{{\cal W}_n({\cal O}_X)}
{\cal W}_n\Om^{\bul}_{X_{\ul{\lam}}}$ as in \cite{nb}.
To obtain it, we have first obtained the log de Rham complex 
$F_n\otimes_{{\cal W}_n({\cal O}_{X_{\ul{\lam}}})}
\Om^{\bul}_{{\cal W}_n(X_{\ul{\lam}})/{\cal W}_n(s), [~]}$ 
and then we have used a natural morphism 
\begin{align*} 
{\Om}^{\bul}_{{\cal W}_n(X_{\ul{\lam}})/{\cal W}_n(s), [~]}
\lo {\cal W}_n\Om^{\bul}_{X_{\ul{\lam}}}. 
\tag{21.22.5}\label{eqn:lanyy} 
\end{align*} 
We obtain this morphism by (\ref{eqn:sn00}) and by replacing 
${\cal W}_n(\os{\circ}{s})$ with ${\cal W}_n(s)$ in 
(\ref{eqn:tise10}) and (\ref{eqn:tdlogb}). 

\begin{theo}\label{theo:crvsdw} 
Let the notations be as in {\rm (\ref{theo:citt})}. 
Then the following hold$:$
\par 
$(1)$ There exists a canonical isomorphism 
\begin{equation*}
Ru_{X_{\ul{\lam}}/{\cal W}_n(s)*}(\eps^*_{X/{\cal W}_n(s)}(E))_{\ul{\lam}})
\os{\sim}{\lo} E_n\otimes_{{\cal W}_n({\cal O}_{X})}
{\cal W}_n\Om^{\bul}_{X_{\ul{\lam}}}.
\tag{21.23.1}\label{eqn:nruowln}
\end{equation*}
This isomorphism is compatible with the projections.
\par
$(2)$ The isomorphism {\rm (\ref{eqn:nruowln})} is 
contravariantly functorial 
with respect to a morphism satisfying the conditions 
{\rm (8.1.6)} for the case $S=s$ and $T={\cal W}_n(s)$.   
\end{theo}
\begin{proof}
Let the notations be as in the proof of (\ref{theo:citt}). 
Then, by replacing ${\cal W}_n(\os{\circ}{s})$ by ${\cal W}_n(s)$, 
we have the morphism (\ref{eqn:nruowln}). 
It is easy to check that the morphism (\ref{eqn:nruowln}) 
is independent of the choice of 
the simplicial immersion $X_{\bul}\os{\sus}{\lo} \ol{\cal P}_{\bul}$.  
\par
Now the question is local on $X$; 
we may assume that there exists an immersion 
$X\os{\sus}{\lo} {\cal P}$ into a log smooth scheme 
over ${\cal W}_n(s)$. 
By (\ref{ali:gsxd}) we have the following exact sequence 
\begin{align*} 
0  \lo 
{\cal E}_n\otimes_{{\cal O}_{{\cal P}{}^{\rm ex}}}
\Om^{\bul}_{{\cal P}{}^{{\rm ex}}_{\ul{\lam}}/{\cal W}_n(s)}[-1] 
\os{d\log \tau\wedge}{\lo} 
{\cal E}_n\otimes_{{\cal O}_{{\cal P}{}^{\rm ex}}}
\Om^{\bul}_{{\cal P}{}^{{\rm ex}}_{\ul{\lam}}/{\cal W}_n(\os{\circ}{s})} 
\lo {\cal E}_n\otimes_{{\cal O}_{{\cal P}{}^{\rm ex}}}
\Om^{\bul}_{{\cal P}{}^{{\rm ex}}_{\ul{\lam}}/{\cal W}_n(s)}
\lo 0. 
\tag{21.23.2}\label{ali:gsetxd}
\end{align*}  
By (\ref{ali:gsxxhd}) we have the following exact sequence 
\begin{align*} 
0  \lo E_n\otimes_{{\cal W}_n({\cal O}_X)}
{\cal W}_n{\Om}^{\bul}_{X_{\ul{\lam}}}[-1]\os{d\log \tau\wedge }{\lo} 
E_n\otimes_{{\cal W}_n({\cal O}_X)}
{\cal W}_n\wt{\Om}^{\bul}_{X_{\ul{\lam}}}
\lo {\cal W}_n\Om^{\bul}_{X_{\ul{\lam}}} \lo 0. 
\tag{21.23.3}\label{ali:getsd}
\end{align*} 
The morphisms (\ref{eqn:nrlowln}) and (\ref{eqn:nruowln})
induce the morphism from (\ref{ali:gsetxd}) to (\ref{ali:getsd}). 
Because the morphism (\ref{eqn:nrlowln}) is an isomorphism,  
we see that the morphism (\ref{eqn:nruowln}) is an isomorphism by 
induction on the degree of the cohomological sheaves of the complexes 
$Ru_{X_{\ul{\lam}}/{\cal W}_n(s)*}(\eps^*_{X/{\cal W}_n(s)}(E))_{\ul{\lam}})$ 
and 
$E_n\otimes_{{\cal W}_n({\cal O}_{X})}
{\cal W}_n\Om^{\bul}_{X_{\ul{\lam}}}$. 
\par 
(2): Using the argument after (\ref{prop:indu}),  we can prove the functoriality 
without difficulty. We leave the detail of the proof to the reader 
(cf.~the proof of (\ref{theo:citt}) (2)). 
\par
We finish the proof.
\end{proof}

Set 
\begin{align*} 
{\cal W}_n{\Om}^{\bul}_{X^{(m)}}:=\bigoplus_{\ul{\lam}=m+1}{\cal W}_n{\Om}^{\bul}_{X_{\ul{\lam}}}. 
\end{align*} 
and 
\begin{align*} 
{\cal W}_n{\Om}^{\bul}_{X^{(*)}}:=
s({\cal W}_n{\Om}^{\bul}_{X^{(\bul)}}):=\bigoplus_{\ul{\lam}=m+1}{\cal W}_n{\Om}^{\bul}_{X_{\ul{\lam}}}. 
\end{align*} 

\begin{coro}\label{coaro}
$(1)$ 
Let $E_n\otimes_{{\cal W}_n({\cal O}_{X})}
{\cal W}_n{\Om}^{\bul}_{X^{(\star)}}$ be the single complex of 
the double complex 
$(\cdots \lo a^{(m)}_*(E_n\otimes_{{\cal W}_n({\cal O}_{X})}
{\cal W}_n{\Om}^{\bul}_{X^{(m)}})\lo \cdots)_{m\in {\mab N}}$, 
where $E_n\otimes_{{\cal W}_n({\cal O}_{X})}
{\cal W}_n{\Om}^{\bul}_{X^{(m)}}:=\bigoplus_{\# \ul{\lam}=m+1}
E_n\otimes_{{\cal W}_n({\cal O}_{X})}
{\cal W}_n{\Om}^{\bul}_{X_{\ul{\lam}}}$. 
Then there exists a canonical filtered isomorphism 
\begin{equation*}
{R}u_{X^{(\star)}/{\cal W}_n(s)*}(E) \os{\sim}{\lo}  
E_n\otimes_{{\cal W}_n({\cal O}_{X})}{\cal W}_n{\Om}^{\bul}_{X^{(\star)}}.
\tag{21.24.1}\label{eqn:nsln}
\end{equation*}
The isomorphism $(\ref{eqn:nsln})$ is compatible with the projections.
\par 
$(2)$ The isomorphism {\rm (\ref{eqn:nsln})} is contravariantly functorial 
with respect to a morphism satisfying the conditions {\rm (8.1.6)} 
for the case $S=s$ and $T={\cal W}_n(s)$.   
\par
$(3)$ Let 
$$
{R}u_{X/{\cal W}_n(s)*}
(\eps^*_{X/{\cal W}_n(s)}(E))
\os{\sim}{\lo} 
Ru_{X^{(\star)}/{\cal W}_n(s)*}(\eps^*_{X^{(\star)}/{\cal W}_n(s)}(E^{(\star)}))
$$ be the isomorphism {\rm (\ref{eqn:eetxte})} in the case $S=s$ and $T={\cal W}_n(s)$. 
Then the natural morphism 
\begin{equation*}
E_n\otimes_{{\cal W}_n({\cal O}_X)}{\cal W}_n{\Om}{}^{\bul}_X 
\lo  
E_n\otimes_{{\cal W}_n({\cal O}_X)}{\cal W}_n{\Om}{}^{\bul}_{X^{(\star)}}
\tag{21.24.2}\label{eqn:exawte}
\end{equation*} 
in $C^+(f^{-1}({\cal W}_n))$ 
fits into the following commutative diagram
\begin{equation*}
\begin{CD} 
{R}u_{X/{\cal W}_n(s)*}
(\eps^*_{X/{\cal W}_n(s)}(E))
@>{\sim}>> 
{R}u_{X^{(\star)}/{\cal W}_n(s)*}(\eps^*_{X^{(\star)}/{\cal W}_n(s)}
(E^{(\star)}))\\
@V{{\rm (\ref{eqn:nruowln})}{\rm ~for~the~case}~\ul{\lam}=\emptyset}V{\simeq}V @V{{\rm (\ref{eqn:nsln})}}V{\simeq}V \\
E_n\otimes_{{\cal W}_n({\cal O}_X)}{\cal W}_n{\Om}{}^{\bul}_X 
@>>>
E_n\otimes_{{\cal W}_n({\cal O}_X)}{\cal W}_n{\Om}{}^{\bul}_{X^{(\star)}}
\end{CD} 
\end{equation*}  
in $D^+(f^{-1}({\cal W}_n)$.  
Here the upper isomorphism in the commutative diagram above is the isomorphism 
{\rm (\ref{eqn:exawte})} in the special case. 
Consequently the morphism {\rm (\ref{eqn:exawte})} is a quasi-isomorphism. 
\end{coro}
\begin{proof} 
This immediately follows from (\ref{theo:citt}). 
\end{proof}

\par 
The following is a main result in this section

\begin{theo}[{\bf Comparison theorem between Hirsch pre-weight-filtered complexes 
and Kim-Hain's filtered complexes}]\label{theo:cap}
Let the notations be as in {\rm (\ref{theo:citt})}. 
Then there exists a canonical filtered isomorphism 
\begin{align*} 
(H_{\rm zar}(X/{\cal W}_n(\os{\circ}{s}),E), P)
\os{\sim}{\lo} ({\cal W}_nH_X(E), P) 
\tag{21.25.1}\label{ali:mcc}
\end{align*}
in ${\rm D}^+{\rm F}(f^{-1}({\cal W}_n))$ 
fitting into the following commutative diagram 
\begin{equation*} 
\begin{CD}
Ru_{X^{(\star)}/{\cal W}_n(s)*}(\eps^*_{X^{(\star)}/{\cal W}_n(s)}
(E^{(\star)}))@>{\sim}>>
E_n\otimes_{{\cal W}_n({\cal O}_X)}{\cal W}_n{\Om}{}^{\bul}_{X^{(\star)}}\\
@A{\simeq}AA @AA{\simeq}A\\
H_{\rm zar}(X/{\cal W}_n(\os{\circ}{s}),E)
@>{\sim}>> {\cal W}_nH_X(E)\\
@A{\simeq}AA @AA{\simeq}A\\
\wt{R}u_{X/{\cal W}_n(\os{\circ}{s})*}
(\eps^*_{X_{\os{\circ}{T}_0}/{\cal W}_n(\os{\circ}{s})}(E)\langle u \rangle)
@>{\sim}>>E_n\otimes_{{\cal W}_n({\cal O}_X)}
{\cal W}_n\wt{\Om}^{\bul}_{X}\langle u \rangle\\
@V{\simeq}VV @VV{\simeq}V\\
Ru_{X_{\os{\circ}{T}_0}/{\cal W}_n(s)*}
(\eps^*_{X/{\cal W}_n(s)}(E))@>{\sim}>>
E_n\otimes_{{\cal W}_n({\cal O}_X)}{\cal W}_n\Om^{\bul}_{X}
\end{CD}\tag{21.25.2}\label{ali:mcdc}
\end{equation*}
in $D^+(f^{-1}({\cal W}_n))$.  
The isomorphism {\rm (\ref{ali:mcc})} and the commutative diagram 
{\rm (\ref{ali:mcdc})} are contravariantly functorial 
for a morphism $g\col X\lo Y$ satisfying 
the condition {\rm (8.1.6)} for the case $S=s$ and $T={\cal W}_n(s)$ 
and a morphism $g_{\rm crys}^*(F)\lo E$ of ${\cal O}_{\os{\circ}{X}/{\cal W}_n(\os{\circ}{s})}$-modules for an analogous sheaf $F$ in $(\os{\circ}{Y}/{\cal W}_n(\os{\circ}{s}))_{\rm crys}$ to $E$. 
The isomorphism {\rm (\ref{ali:mcc})} and the commutative diagram 
{\rm (\ref{ali:mcdc})} are compatible with respect to the projections. 
\end{theo}
\begin{proof}
Let $X_{\bul}$ be the \v{C}ech diagram obtained by 
an affine open covering of $X$ and let $X_{\bul}\os{\sus}{\lo} \ol{\cal P}_{\bul}$ 
be the simplicial immersion into a log smooth 
scheme over $\ol{{\cal W}_n(s)}$ as before. 
Since the immersion $X_0 \os{\sus}{\lo} {\cal W}_n(X_0)$ is nilpotent 
and since $\ol{\cal P}_{\bul}$ is log smooth over $\ol{{\cal W}_n(s)}$, 
we have a morphism ${\cal W}_n(X_0)\lo \ol{\cal P}_{\bul}$ extending 
the immersion $X_{\bul}\os{\sus}{\lo} \ol{\cal P}_{\bul}$. 
Let 
$$\pi \col ((X_{\bul})_{\rm zar},f^{-1}_{\bul}({\cal W}_n)) 
\lo (X_{\rm zar},f^{-1}({\cal W}_n))$$ 
be the natural morphism of ringed topoi.
The filtered morphism (\ref{ali:pwkpp}) induces the following filtered morphism
\begin{align*}
({\cal E}^{\bul}_n\otimes_{{\cal O}_{{\cal P}_{\bul}}}
\Om^{\bul}_{{\cal P}_{\bul \ul{\lam}}/{\cal W}_n(\os{\circ}{s})}\langle u \rangle,P) 
\lo 
(E^{\bul}_n\otimes_{{\cal O}_{{\cal P}_{\bul}}}
\pi^{-1}(W_n\wt{\Om}{}^{\bul}_{X_{\ul{\lam}}})\langle u \rangle,P).
\tag{21.25.3}\label{ali:pwakpp}
\end{align*}
This morphism induces the following filtered morphism 
\begin{align*}
(H_{\rm zar}(X/{\cal W}_n(\os{\circ}{s}),E),P)
\lo ({\cal W}_nH_X(E),P).
\tag{21.25.4}\label{ali:pwbckpp}
\end{align*} 
The following induced morphism by this morphism 
\begin{align*}
{\rm gr}^P_kH_{\rm zar}(X/{\cal W}_n(\os{\circ}{s}),E)
{\lo} {\rm gr}^P_k{\cal W}_nH_X(E). 
\end{align*}
is equal to the following morphism by (\ref{eqn:ee}) and (\ref{ali:pkpm}): 
\begin{align*} 
&\bigoplus_{m\geq 0}
\bigoplus_{\# \ul{\lam}=m+1}
\bigoplus_{j\geq 0}
\bigoplus_{\# \ul{\mu}=k+m-2j}
a_{\ul{\lam}\cup \ul{\mu}*}
Ru_{\os{\circ}{X}_{\ul{\lam}\cup \ul{\mu}}
/{\cal W}_n(\os{\circ}{s})*}
(E_{\os{\circ}{X}_{\ul{\lam}\cup \ul{\mu}}
/{\cal W}_n(\os{\circ}{s})}
\otimes_{\mab Z}
\vp_{{\rm crys},\ul{\mu}}(\os{\circ}{X}/{\cal W}_n(\os{\circ}{s}))) \\
& \{-(k+m-2j)\}[-m](-(k+m-j);v)
\\
&\lo \bigoplus_{m=0}^{\inf}\bigoplus_{j=0}^{\inf}
\bigoplus_{\# \ul{\lam}=m}
\bigoplus_{\# \ul{\mu}=k+m-2j}
{\cal W}_nu^{[j]}\otimes_{{\cal W}_n}
a_{\ul{\lam}\cup \ul{\mu}*}
(E_n\otimes_{{\cal W}_n({\cal O}_{X})}
{\cal W}_n\Om^{\bul -k}_{\os{\circ}{X}_{\ul{\lam}\cup \ul{\mu}}}
\otimes_{\mab Z}\vp_{{\rm zar},\ul{\mu}}(\os{\circ}{X}/\os{\circ}{s})).
\end{align*} 
By \cite[III (1.5)]{ir} and Etesse's theorem \cite[I Theorem 2.1]{et} this is an isomorphism. 
\par 
We leave the proofs of the commutativity of the diagram 
(\ref{ali:mcdc}) and the 
contravariant functoriality and the compatibility 
with the projections to the reader. 
\end{proof}

\begin{rema}\label{rema:n}
Let $({\cal W}_nA_{X}(E),P)$ be the filtered complex defined in \cite[(2.3.21)]{nb}. 
Then we can construct a filtered morphism 
(which we should call the {\it log de Rham-Witt Fujisawa's morphism}): 
\begin{align*} 
\psi^{\rm dRW} \col ({\cal W}_nA_{X}(E),P)\lo ({\cal W}_nH_X(E),P)
\end{align*} 
whose underlying morphism 
\begin{align*} 
\psi^{\rm dRW} \col {\cal W}_nA_{X}(E)\lo {\cal W}_nH_X(E)
\end{align*}
is an isomorphism fitting into the following commutative diagram 
\begin{equation*} 
\begin{CD}
{\cal W}_nA_X(E)@>{\psi^{\rm dRW},\sim}>>{\cal W}_nH_X(E)\\
@A{\theta \wedge}AA @AA{\simeq}A\\
E\otimes_{{\cal W}_n({\cal O}_X)}{\cal W}_n\Om^{\bul}_X
@. E\otimes_{{\cal W}_n({\cal O}_X)}{\cal W}_n\wt{\Om}{}^{\bul}_X\langle u \rangle\\
@| @VV{\simeq}V\\
E\otimes_{{\cal W}_n({\cal O}_X)}{\cal W}_n\Om^{\bul}_X
@=E\otimes_{{\cal W}_n({\cal O}_X)}{\cal W}_n\Om^{\bul}_X. 
\end{CD}
\tag{21.26.1}\label{cd:ssbati}
\end{equation*} 
The filtered morphism $\psi^{\rm dRW}$ is compatible with the projections. 
\par 
As in \S\ref{sec:csc}, we can also construct the ``semi-cosimplicial version'' 
$({\cal W}_nA_{X^{(\star)}}(E),P)$ of 
this filtered complex ${\cal W}_nA_{X}(E)$ and we can prove that 
the following diagram is commutative$:$
\begin{equation*} 
\begin{CD}
(H_{\rm zar}(X/{\cal W}_n(\os{\circ}{s}),E),P)
@>{\sim}>> ({\cal W}_nH_X(E),P)\\
@VVV @VVV\\
(A_{\rm zar}(X^{(\star)}/{\cal W}_n(\os{\circ}{s}),E),P)
@>{\sim}>>({\cal W}_nA_{X^{(\star)}}(E),P)\\
@AAA @AAA\\
(A_{\rm zar}(X/{\cal W}_n(\os{\circ}{s}),E),P)
@>{\sim}>>({\cal W}_nA_{X}(E),P). 
\end{CD}\tag{21.26.2}\label{ali:mawcc}
\end{equation*}
This commutative diagram is compatible with the projections.
\end{rema}

By the same argument for the definition of the product (\ref{ali:te}), 
we obtain the following morphism 
\begin{align*} 
\cup \col ({\cal W}_nH_X(E),P)\otimes_{{\cal W}_n}({\cal W}_nH_X(E'),P)
\lo 
({\cal W}_nH_X(E\otimes_{{\cal O}_{\os{\circ}{X}/{\cal W}_n}}E'),P),
\tag{21.26.3}\label{ali:mpwcc}
\end{align*} 
where $E'$ is an analogous crystal of 
${\cal O}_{\os{\circ}{X}/{\cal W}_n(\os{\circ}{s})}$-modules to $E$.  
By the local calculation, we obtain the following: 

\begin{prop}\label{prop:pdccc}
Assume that $\os{\circ}{X}$ is quasi-compact. 
Then the following diagram is commutative$:$
\begin{equation*} 
\begin{CD}
(H_{\rm zar}(X/{\cal W}_n(\os{\circ}{s}),E),P)\otimes^L_{f^{-1}({\cal W}_n)}
(H_{\rm zar}(X/{\cal W}_n(\os{\circ}{s}),E'),P)
@>{\cup}>> (H_{\rm zar}(X/{\cal W}_n(\os{\circ}{s}),E
\otimes_{{\cal O}_{\os{\circ}{X}/{\cal W}_n(\os{\circ}{s})}}E',P)
\\
@V{\simeq}VV @VV{\simeq}V\\
({\cal W}_nH_X(E),P)\otimes_{f^{-1}({\cal W}_n)} ({\cal W}_nH_X(E'),P)
@>{\cup}>>({\cal W}_nH_X
(E\otimes_{{\cal O}_{\os{\circ}{X}/{\cal W}_n(\os{\circ}{s})}}E'),P).  
\end{CD}\tag{21.27.1}\label{ali:mwcc}
\end{equation*}
The commutative diagram  {\rm (\ref{ali:mwcc})}  is contravariantly functorial  
for a morphism $g\col X\lo Y$ satisfying 
the condition {\rm (8.1.6)} for the case $S=s$ and $T={\cal W}_n(s)$ 
and morphisms $g_{\rm crys}^*(F)\lo E$ 
and $g_{\rm crys}^*(F')\lo E'$
of ${\cal O}_{\os{\circ}{X}/{\cal W}_n(\os{\circ}{s})}$-modules 
for analogous sheaves $F$ and $F'$ in $(\os{\circ}{Y}/{\cal W}_n(\os{\circ}{s}))$ to $E$ 
and $E'$, respectively.  
The commutative diagram  {\rm (\ref{ali:mwcc})} is 
compatible with respect to the projections.
\end{prop}

If one consider the case where $E$ and $E'$ are trivial, then we obtain
the commutative diagram (\ref{cd:ixwhwctt}).

\begin{defi}
Set 
\begin{align*} 
({\cal W}H_{X,{\mab Q}}, P):=
s_{\rm TW}(a^{(\bul)}_*(({\cal W}\wt{\Om}{}^{\bul}_{X^{(\bul)}}\langle u\rangle_{\mab Q}), P)). 
\end{align*}  
We call $({\cal W}H_{X,{\mab Q}}, P)$ {\it Kim-Hain's filtered dga} of $X/{\cal W}$. 
\end{defi}

We conclude this book by giving the following: 

\begin{theo}\label{theo:wccc}
Set ${\cal W}{\Om}{}^{\bul}_{X^{(\star)},{\mab Q}}:=
s_{\rm TW}(a^{(\bul)}_*({\cal W}{\Om}{}^{\bul}_{X^{(\bul)},{\mab Q}}))$. 
There exists a canonical filtered isomorphism 
\begin{align*} 
(H_{{\rm zar},{\rm TW}}(X/{\cal W}(\os{\circ}{s}))_{\mab Q}, P)
\os{\sim}{\lo} ({\cal W}H_{X,{\mab Q}}, P) 
\tag{21.29.1}\label{ali:mcac}
\end{align*}
in $D({\rm A}^{\geq 0}{\rm F}(f^{-1}({\cal W})))$ 
fitting into the following commutative diagram 
\begin{equation*} 
\begin{CD}
R_{\rm TW}
u_{X^{(\star)}/{\cal W}(s)*}({\cal O}_{X^{(\star)}/{\cal W}(s)})_{\mab Q}
@>{\sim}>>
{\cal W}{\Om}{}^{\bul}_{X^{(\star),{\mab Q}}}\\
@A{\simeq}AA @AA{\simeq}A \\
H_{{\rm zar},{\rm TW}}(X/{\cal W}(\os{\circ}{s}))_{\mab Q}
@>{\sim}>> {\cal W}H_{X,{\mab Q}}\\
@A{\simeq}AA @AA{\simeq}A\\
\wt{R}_{\rm TW}u_{X/{\cal W}(\os{\circ}{s})*}
({\cal O}_{X_{\os{\circ}{T}_0}/{\cal W}(\os{\circ}{s})}\langle u \rangle)_{\mab Q}
@>{\sim}>> ({\cal W}\wt{\Om}^{\bul}_{X}\langle u \rangle)_{\mab Q}\\
@V{\simeq}VV @VV{\simeq}V\\
R_{\rm TW}u_{X/{\cal W}(s)*}({\cal O}_{X/{\cal W}(s)})_{\mab Q}
@>{\sim}>>
{\cal W}\Om^{\bul}_{X,{\mab Q}}. 
\end{CD}\tag{21.29.2}\label{ali:mcadc}
\end{equation*}
The isomorphism {\rm (\ref{ali:mcac})} and the commutative diagram 
{\rm (\ref{ali:mcadc})} are contravariantly functorial  
for a morphism $g\col X\lo Y$ satisfying 
the condition {\rm (8.1.6)} for the case $S=s$ and $T={\cal W}(s)$. 
The isomorphism {\rm (\ref{ali:mcac})} and the commutative diagram 
{\rm (\ref{ali:mcadc})} are compatible with respect to the projections.
\end{theo}
\begin{proof} 
The morphism (\ref{ali:xwkabpp}) 
in the case where $E$ is trivial is a filtered morphism of dga's. 
Hence we have the filtered morphism (\ref{ali:mcac}). 
\par 
We leave the proofs of the commutativity of the diagram 
(\ref{ali:mcadc}) and the 
contravariant functoriality and the compatibility 
with projections to the reader. 
\end{proof} 


\bigskip
\bigskip
\parno
Yukiyoshi Nakkajima 
\parno
Department of Mathematics,
Tokyo Denki University,
5 Asahi-cho Senju Adachi-ku,
Tokyo 120--8551, Japan. 
\parno
{\it E-mail address\/}: 
nakayuki@cck.dendai.ac.jp


\begin{thebibliography} {EGA IV-3}

\bibitem[A]{a}Achinger, P.
\newblock{\em Hodge symmetry for rigid varieties via log hard Lefshcetz}. 
Mathematical Research Letters 30 (2023), 1--31. 

\bibitem[B1]{bwl}Berthelot, P.
\newblock{\em Sur le $\ll$th\'{e}or\`{e}me de 
Lefschetz faible$\gg$ en cohomologie cristalline}. 
C.~R.~Acad.~Sci.~Paris S\'{e}r.~A-B277 (1973), A955--A958. 


\bibitem[B2]{bb}Berthelot, P.                                            
\newblock{\em Cohomologie cristalline 
des sch\'{e}mas de 
caract\'{e}ristique $p>0$}.
\newblock Lecture Notes in Math.~407, 
Springer-Verlag (1974).



\bibitem[BO1]{bob}Berthelot, P., Ogus, A.  
\newblock{\em Notes on crystalline cohomology}. 
\newblock Princeton University Press, 
University of Tokyo Press (1978).

\bibitem[BO2]{boi}Berthelot, P., Ogus, A. 
\newblock{\em $F$-isocrystals and de Rham cohomology. 
{\rm I}}. 
\newblock Invent.~Math.~72 (1983), 159--199. 



\bibitem[CL]{clpu}Chiarellotto, B., Le Stum, B.
\newblock{\em Sur la puret\'{e} de la cohomologie cristalline}. 
\newblock C.~R.~Acad.~Sci.~Paris, 
S\'{e}rie I 326 (1998), 961--963.




\bibitem[D1]{dh2}Deligne, P.                                              
\newblock{\em Th\'{e}orie  de Hodge, ${\rm II}$}.
\newblock Publ.~Math.~IH\'{E}S 40 (1971), 5--57.


\bibitem[D2]{dh3}Deligne, P.                                                  
\newblock{\em Th\'{e}orie  de Hodge, ${\rm III}$}.
\newblock Publ.~Math.~IH\'{E}S 44 (1974), 5--77.

\bibitem[E]{et}Etesse, J.-Y. 
\newblock{\em Complexe de de Rham-Witt \`{a} 
coefficients dans un cristal.} 
\newblock Compositio Math.~66 (1988), 57--120. 

\bibitem[EY]{ey}Ertl, V.,  Yamada, K.
\newblock{\em Rigid analytic reconstruction of Hyodo-Kato theory}. 
\newblock Preprint, available from https://arxiv.org/abs/1907.10964. 

\bibitem[EGA III-1]{ega3}Grothendieck, A., Dieudonn\'e, J. 
\newblock{\em \'El\'ements de g\'eom\'etrie alg\'ebrique 
${\rm III}$-$1$}.
\newblock Publ.~Math.~IH\'{E}S 11 (1961). 


\bibitem[F1]{fut}Fujisawa, T. 
\newblock{\em Mixed Hodge structures on log smooth degenerations}.
\newblock  Tohoku Math.~J.~60 (2008), 71--100.

\bibitem[F2]{fup}Fujisawa, T. 
\newblock{\em Polarizatons on limiting mixed Hodge structures}.
\newblock  J.~of Singularities 8 (2014), 146--193.

\bibitem[FN]{fn}Fujisawa, T., Nakayama, C.                                            
\newblock{\em Mixed Hodge structures on log deformations}.
\newblock  Rend.~Sem.~Mat.~Univ.~Padova 110 
(2003), 221--268.

\bibitem[GeM]{gelma}Gelfand, S.~I., Manin, Y.~I.
\newblock{\em Methods of Homological Algebra}. 
\newblock Springer-Verlag, 2nd edition (2002).

\bibitem[GK]{gkcf}Gro{\ss}e-Kl\"onne, E.
\newblock{\em The \v{C}ech filtration 
and monodromy in log crystalline cohomology}.
\newblock Trans.~Amer.~Math.~Soc.~359~(2007), 2945--2972. 

\bibitem[GN]{gn}Guill\'{e}n, F., Navarro Aznar, V.
\newblock{\em Sur le th\'{e}or\`{e}me local des cycles invariants}. 
\newblock  Duke Math.~J. 61 (1990), 133--155. 

\bibitem[Gr]{gr}Griffiths, P.
\newblock{\em Periods of integrals on algebraic manifolds.~I. 
Construction and properties of modular varieties}. 
\newblock Amer.~J.~Math.~90 (1968), 568--626

\bibitem[Ha]{hard}Hartshorne, R.:
\newblock{\em Residues and duality}. 
\newblock  Lecture Notes in Math.~20, Springer-Verlag (1966).

\bibitem[Hy]{hyp}Hyodo, O.
\newblock{\em On the de Rham-Witt 
complex attached to a semi-stable family}. 
\newblock Comp.~Math.~78 (1991), 241--260.

\bibitem[HK]{hk}Hyodo, O., Kato, K. 
\newblock{\em Semi-stable reduction and 
crystalline cohomology with logarithmic poles}. 
\newblock In:  P\'eriodes $p$-adiques, 
Seminaire de Bures, 1988. 
Ast\'{e}risque 223, Soc.~Math.~de France 
(1994),  221--268.

\bibitem[Il]{idw}Illusie, L.                                            
\newblock{\em Complexe de de Rham-Witt et cohomologie cristalline}.
\newblock  Ann.~Scient.~\'{E}c.~Norm.~Sup.~$4^e$ s\'{e}rie 12 (1979), 
501--661.

\bibitem[IR]{ir}Illusie, L., Raynaud, M.
\newblock{\em Les suites spectrales 
 associ\'ees au complexe de de Rham-Witt}.  
\newblock  Publ.~Math.~IH\'{E}S  57 (1983), 73--212. 
\bibitem[It]{iti}Ito, T. 
\newblock{\em Weight-monodromy conjecture for 
p-adically uniformized varieties}
\newblock  Invent.~Math.~159 (2005), 607--656. 


\bibitem[K1]{klog1}Kato, K.
\newblock{\em Logarithmic structures 
of Fontaine-Illusie}.
\newblock In: Algebraic analysis, 
geometry, and number theory, 
Johns Hopkins Univ. Press (1989), 191--224.

\bibitem[K2]{klog2}Kato, K.
\newblock{\em Toric singularities}.
\newblock Amer.~J.~Math.~116 (1994),  1073--1099.


\bibitem[K3]{kln}Kato, K.                                           
\newblock{\em }
\newblock A letter to Y.~Nakkajima (in Japanese) (1997).

\bibitem[KH]{kiha}Kim, M., Hain, R.~M.
\newblock{\em A de Rham-Witt approach 
to crystalline rational homotopy theory}.
\newblock  Comp.~Math.~140 (2004), 1245--1276. 


\bibitem[KM]{kme}Katz,  N., Messing, W.                                           
\newblock{\em Some consequences  of 
the Riemann hypothesis for varieties 
over finite fields}.
\newblock Invent.~Math.~23 (1974), 73--77. 

\bibitem[KtN]{kn}Kato, K., Nakayama, C.
\newblock{\em Log Betti cohomology, 
log \'{e}tale cohomology, and 
log de Rham cohomology of log schemes over ${\mab C}$}.
\newblock Kodai Math.~J.~22 (1999), 161--186.

\bibitem[KU]{ku}Kato, K., Usui, S. 
\newblock{\em Classifying spaces of degenerating polarized Hodge structures}. 
\newblock Ann.~of Math.~Stud.~169, Princeton Univ. Press (2009).



\bibitem[M]{msemi}Mokrane, A.
\newblock{\em La suite spectrale 
des poids en cohomologie de Hyodo-Kato}.
\newblock Duke Math.~J.~72 (1993), 301--336. 


\bibitem[Nak1]{ndw}Nakkajima, Y.
\newblock{\em $p$-adic weight spectral sequences of log varieties}.
\newblock  J.~Math.~Sci.~Univ.~Tokyo 12 (2005), 513--661. 

\bibitem[Nak2]{np}Nakkajima, Y.        
\newblock{\em Signs in weight spectral sequences, 
monodromy--weight conjectures, 
log Hodge symmetry and degenerations of surfaces}.
\newblock Rend.~Sem.~Mat.~Univ.~Padova 116 (2006), 71--185. 

\bibitem[Nak3]{nh3}Nakkajima, Y.
\newblock{\em Weight filtration and 
slope filtration on the rigid cohomology of 
a variety in characteristic $p>0$}.
\newblock M\'{e}m.~Soc.~Math.~France 130--131 (2012). 

\bibitem[Nak4]{nb}Nakkajima, Y.
\newblock{\em Limits of weight filtrations and limits of slope filtrations 
on infinitesimal cohomologies in mixed characteristics I}.
\newblock  Preprint: available from https://arxiv.org/abs/1902.00182. 


\bibitem[Nak5]{nf}Nakkajima, Y.
\newblock{\em An ideal proof for Fujisawa's result and its generalization}.
\newblock  Preprint: available from https://arxiv.org/abs/

\bibitem[Nak6]{ns}Nakkajima, Y.
\newblock{\em Steenbrink isomorphism and crystals on tubular neighbourhoods}. 
\newblock  Preprint: available from https://arxiv.org/abs/


\bibitem[Nav]{nav}Navarro Aznar, V.
\newblock{\em Sur la th\'{e}orie de Hodge-Deligne}. 
\newblock Invent.~Math.~90 (1987), 11--76. 

\bibitem[NS]{nh2}Nakkajima, Y., Shiho, A.
\newblock{\em Weight filtrations on log crystalline cohomologies 
of families of open smooth varieties}.
\newblock Lecture Notes in Math.~1959, 
Springer-Verlag (2008). 


\bibitem[NY]{ny}Nakkajima, Y.,  Yobuko, F. 
\newblock{\em Degenerations of log Hodge de Rham spectral sequences, 
log Kodaira vanishing theorem in characteristic $p>0$ 
and log weak Lefschetz conjecture for log crystalline cohomologies}. 
\newblock European Journal of Mathematics 7 (2021), 1537--1615.  



\bibitem[O1]{od}Ogus, A.
\newblock{\em 
$F$-isocrystals and de Rham cohomology.~${\rm II}$. 
Convergent isocrystals}.
\newblock Duke Math.~J.~51 (1984), 765--850.

\bibitem[O2]{oc}Ogus, A.
\newblock{\em $F$-crystals on schemes with constant log structure}.
\newblock Compositio Math.~97 (1995), 187--225.

\bibitem[RX]{rx}Raskind, W., Xarles X. 
\newblock{\em On the \'{e}tale cohomology of algebraic varieties with totally 
degenerate reduction over $p$-adic fields}. 
\newblock J.~Math.~Sci.~Univ.~Tokyo 14 (2007), 261--291.

\bibitem[Sa1]{sam}Saito, M.
\newblock{\em Modules de Hodge polarisables}.
\newblock Publ.~Res.~Inst.~Math.~Sci.~24 (1988), 849-995.

\bibitem[Sa2]{sap}Saito, M.
\newblock{\em Monodromy filtration and positivity}.
\newblock Preprint: available from https://arxiv.org/pdf/math/0006162.pdf

\bibitem[Sh3]{s3}Shiho, A.
\newblock{\em Relative log convergent 
cohomology and relative rigid cohomology 
${\rm I}$, ${\rm II}$, {\rm III}}.
\newblock Preprint. available from https://arxiv.org/abs/0707.1742, 
https://arxiv.org/pdf/0707.1743 and https://arxiv.org/pdf/0805.3229	

\bibitem[St1]{sti}Steenbrink, J.~H.~M.
\newblock{\em Limits of Hodge structures}. 
\newblock Invent.~Math.~31 (1976),  229--257. 

\bibitem[St2]{stm}Steenbrink, J.~H.~M.
\newblock{\em Logarithmic embeddings of varieties with 
normal crossings and mixed Hodge 
structures}.
\newblock Math.~Ann.~301 (1995), 105--118.

\bibitem[Su]{su}Sullivan, D.
\newblock{\em Infinitesimal computations in topology}. 
\newblock Publ.~Math.~IH\'{E}S 47 (1977), 269--331. 

\bibitem[SGA 7-II]{sga72}Deligne, P., Katz, N.
\newblock{\em Groupes de monodromie en 
g\'{e}om\'{e}trie alg\'{e}brique}.
\newblock Lecture Notes in Math.~340 (1973), 
Springer-Verlag.

\bibitem[T]{tsp}Tsuji, T.
\newblock{\em Poincar\'{e} duality for 
logarithmic crystalline cohomology}.
\newblock Comp.~Math.~118 (1999), 11--41.
 
 
\bibitem[U]{us}Usui, S.  
\newblock{\em Variation of mixed Hodge structure arising 
from family of logarithmic deformations II}. 
\newblock  Classifying space, Duke Math.~J. 51 (1984), 851--875.


\bibitem[W]{weib}Weibel, C.~A.
\newblock{\em An introduction to homological algebra}.
\newblock  Cambridge Studies in Advanced Mathematics 38.
Cambridge University Press, Cambridge (1994).
\end{thebibliography}
\end{document}